\documentclass[12pt]{report}
\usepackage[titles]{tocloft}

\makeatletter 
\newcommand\semiHuge{\@setfontsize\semiHuge{15}{20}}
\makeatother

\usepackage{titlesec}
\titleformat{\chapter}[display]
{\normalfont\huge\bfseries}{\chaptertitlename\ \thechapter}{0pt}{\Huge}
 \titlespacing*{\chapter}{0pt}{40pt}{40pt}

\usepackage[hang,flushmargin]{footmisc} 

\usepackage{tikz}
\usepackage{doublespace}
\usepackage{endnotes}
\usepackage{hyperref}
\usepackage{hyperendnotes}
\usepackage{float}
\usepackage{latexsym,calligra,amsmath,amssymb,stmaryrd,MnSymbol}
\usepackage[parfill]{parskip} 
\usepackage[bf, small, justification=justified,singlelinecheck=false]{caption}
\newif\ifdviwin
\usepackage{ulem}
\usepackage[all,cmtip]{xy}
\usepackage[latin1]{inputenc}
\usepackage[english]{babel}
\addto\captionsenglish{}
\usepackage{indentfirst}
\usepackage[mathscr]{eucal}
\usepackage{amssymb,amsmath,amsfonts}
\usepackage{graphicx}
\usepackage{amsmath,amscd}
\usepackage{xxcolor}

\newtheorem{theorem}{Theorem}
\newtheorem{cor}[theorem]{Corollary}
\newtheorem{lem}[theorem]{Lemma}
\newtheorem{defi}[theorem]{Definition}
\newtheorem{example}[theorem]{Example}
\numberwithin{theorem}{subsection} 
\numberwithin{equation}{subsection} 
 \newenvironment{proof}{\rm \trivlist \item[\hskip \labelsep{\it
      Proof}:]}{\par\nopagebreak \hfill $\Box$ \endtrivlist}

\DeclareMathAlphabet{\mathcalligra}{T1}{calligra}{m}{n}
\DeclareFontShape{T1}{calligra}{m}{n}{<->s*[2.2]callig15}{}

\newcommand{\mc}[1]{\mathcal{#1}}
\newcommand{\ms}[1]{\mathscr{#1}}
\newcommand{\mbf}[1]{\mathbf{#1}}
\newcommand{\mbb}[1]{\mathbb{#1}}
\newcommand{\mfr}[1]{\mathfrak{#1}}

\def\ee{\varepsilon}
\def\PP{\mathbb{P}}
\def\AA{\mathbb{A}}
\def\QQ{\mathbb{Q}}
\def\ZZ{\mathbb{Z}}

\def\RR{\mathbb{R}}
\def\HH{\mathbb{H}}
\def\CC{\mathbb{C}}
\newcommand{\tn}[1]{\textnormal{#1}}
\newcommand{\fsz}[1]{\footnotesize{#1}}
\tikzset{isometric2XYZ/.style={x={(-0.7cm,-0.3cm)}, y={(0.9cm,-0.1cm)}, z={(0cm,.8cm)}}}
\tikzset{isometricXYZ/.style={x={(-0.707cm,-0.354cm)}, y={(0.707cm,-0.354cm)}, z={(0cm,1cm)}}}
\newcommand{\comment}[1]{}
\begin{document}

\begin{singlespace}
\pagenumbering{roman} \thispagestyle{empty} \vspace*{2.5cm}
\centerline{\huge{ON THE INFINITESIMAL THEORY}} \vspace{.5cm}
\centerline{\huge{OF CHOW GROUPS}} \vspace{14.5cm}
\centerline{\large\bf Benjamin F. Dribus}  \vspace{.5cm}
\centerline{\large Louisiana State University} 
\centerline{\large bdribus@math.lsu.edu} 

\setcounter{secnumdepth}{-2}
\chapter{Preface}\label{ChapterAbstract}

The Chow groups $\tn{Ch}_X^p$ of codimension-$p$ algebraic cycles modulo rational equivalence on a smooth algebraic variety $X$ have steadfastly resisted the efforts of algebraic geometers to fathom their structure.  Except for the case $p=1$, for which $\tn{Ch}_X^p$ is an algebraic group, the Chow groups remain mysterious.   This book explores a ``linearization" approach to this problem, focusing on the {\it infinitesimal structure} of $\tn{Ch}_X^p$ near its identity element.  This method was adumbrated in recent work of Mark Green and Phillip Griffiths.  Similar topics have been explored by Bloch, Stienstra, Hesselholt, Van der Kallen, and others.  A famous formula of Bloch expresses the Chow groups $\tn{Ch}_X^p$ as Zariski sheaf cohomology groups of the algebraic $K$-theory sheaves $\ms{K}_{p,X}$ on $X$.   Bloch's formula follows from the fact that the {\it coniveau spectral sequence} for algebraic $K$-theory on $X$ induces flasque resolutions of the sheaves $\ms{K}_{p,X}$.  Algebraic cycles and algebraic $K$-theory are thereby related via the general method of {\it coniveau filtration} of a topological space.  Hence, ``linearization" of the Chow groups is related to ``linearization" of algebraic $K$-theory, which may be described in terms of negative cyclic homology.

The ``proper formal construction" arising from this approach is a ``machine" involving the coniveau spectral sequences arising from four different generalized cohomology theories on $X$, which I will call $K$-theory, augmented $K$-theory, relative $K$-theory, and relative negative cyclic homology.   The objects of principal interest are certain sheaf cohomology groups defined in terms of these theories.  In particular, the Chow group $\tn{Ch}_X^p$ is the $p$th cohomology group of the $K$-theory sheaf $\ms{K}_{p,X}$ on $X$.  Due to the critical role of the coniveau filtration, I refer to this construction as the {\it coniveau machine.}  The main theorem in this book establishes the existence of the coniveau machine for algebraic $K$-theory on a smooth algebraic variety.  This result depends on a large body of prior work of Bloch-Ogus, Thomason, Colliot-Th\'el\`ene, Hoober, and Kahn, Loday, and most recently, Corti\~nas, Haesemeyer, and Weibel.  An immediate corollary is a new formula expressing {\it generalized tangent groups} of Chow groups in terms of negative cyclic homology, which is generally more tractable than algebraic $K$-theory.  

This book is a work in progress based on my Ph.D. dissertation, defended in March of 2014.  The present version, as of late April 2014, is substantially the same as the dissertation, differing only in format.  During the last few months, a number of highly esteemed mathematicians, including Jan Stienstra, Wilberd Van der Kallen, Lars Hesselholt, Charles Weibel, Jon Rosenberg, and Marco Schlichting, have offered kind and constructive feedback that I have not yet had a chance to put to proper use.  I have also become aware of recent related work of Bloch, Esnault, and Kerz, Patel and Ravindra, and others.   Any failure to accord these efforts proper credit in the present version of this book is due only to the inevitable delays involved in the process of reception, absorption, and incorporation of new material.


\newpage
\tableofcontents
\setcounter{secnumdepth}{2}
\newpage
\pagenumbering{arabic}
\chapter{Preliminaries}\label{ChapterIntro}

The Chow groups $\tn{Ch}_X^p$ of codimension-$p$ algebraic cycles modulo rational equivalence on a smooth algebraic variety $X$ over a field $k$ containing the rational numbers remain poorly understood despite intensive study by algebraic geometers over the last half-century.   Results of Mumford \cite{MumfordZeroCycles69}, Griffiths \cite{GriffithsRationalIntegralsIandII69}, and many others have demonstrated the general intractability of $\tn{Ch}_X^p$ in the case where $p$ exceeds $1$.  

Bloch's formula \cite{BlochK2Cycles74}, extended by Quillen \cite{QuillenHigherKTheoryI72}, expresses the Chow groups of $X$ as Zariski sheaf cohomology groups of the algebraic $K$-theory\footnotemark\footnotetext{An immediate question is, ``which version of $K$-theory?"  As explained below, Bloch's formula itself holds for a variety of different versions.  However, the infinitesimal theory developed in this book requires use of the {\it nonconnective $K$-theory} of Bass \cite{BassKTheory68} and Thomason \cite{Thomason-Trobaugh90}.} sheaves $\ms{K}_{p,X}$ on $X$:
\begin{equation}\label{blochstheoremintro}
\tn{Ch}_X^p=H_{\tn{\fsz{Zar}}}^p\big(X,\ms{K}_{p,X}\big).
\end{equation}
Bloch's theorem follows from the fact that the sheafified Cousin complexes arising from the coniveau spectral sequence for algebraic $K$-theory are flasque resolutions of the sheaves $\ms{K}_{p,X}$, using the machinery of Grothendieck and Hartshorne \cite{HartshorneResiduesDuality66}.   Algebraic cycles and algebraic $K$-theory are thereby connected by means of the very general technique of filtration of a topological space by the codimensions of its points.  Bloch's theorem is a striking consequence of the close relationship between dimension and inclusions of closed subsets of algebraic schemes, arising from the coarseness of the Zariski topology.  


Mark Green and Phillip Griffiths \cite{GreenGriffithsTangentSpaces05} have recently introduced an interesting new avenue of investigation into the structure of Chow groups: the study of their tangent groups at the identity, which I will denote by $T\tn{Ch}_{X}^p$.  This approach carries the theme of linearization, familiar from elementary calculus, differential, analytic, and algebraic geometry, and Lie theory, into an unfamiliar realm populated by objects far larger and more complicated than manifolds or schemes.   This is by no means the first effort to study Chow groups in a manner analogous to Lie groups; for example, Jan Stienstra's {\it Cartier-Dieudonn\'{e} theory for Chow groups} \cite{StienstraCartierDieudonne}, \cite{StienstraCartierDieudonneCorrection}, goes beyond the tangent space to investigate higher-order deformation theory, but in a somewhat different context. 

Via Bloch's theorem, the tangent groups $T\tn{Ch}_{X}^p$ may be  identified\footnotemark\footnotetext{Here the specific version of $K$-theory used makes a crucial difference, unlike in Bloch's theorem \hyperref[blochstheoremintro]{\ref{blochstheoremintro}}.} as the sheaf cohomology groups $H_{\tn{\fsz{Zar}}}^p(X,T\ms{K}_{p,X})$, where $T\ms{K}_{p,X}$ is the tangent sheaf to algebraic $K$-theory.   Restricting consideration to Milnor $K$-theory $\ms{K}_{p,X}^{\tn{M}}$, a simpler version of algebraic $K$-theory with a straightforward presentation by generators and relations, the tangent sheaf is the $(p-1)$st sheaf of absolute K\"ahler differentials $\varOmega_{X/\QQ}^{p-1}$ on $X$.  

Based on these results, Green and Griffiths give the definition:

\begin{pgfpicture}{0cm}{0cm}{17cm}{1.25cm}
\begin{pgftranslate}{\pgfpoint{-1cm}{-.75cm}}
\pgfputat{\pgfxy(7.5,.75)}{\pgfbox[center,center]{$T\tn{Ch}_{X}^p$}}
\pgfputat{\pgfxy(8.75,.75)}{\pgfbox[center,center]{$:=$}}
\pgfputat{\pgfxy(8.75,1.2)}{\pgfbox[center,center]{\footnotesize{Griffiths}}}
\pgfputat{\pgfxy(8.75,1.55)}{\pgfbox[center,center]{\footnotesize{Green-}}}
\pgfputat{\pgfxy(10.7,.75)}{\pgfbox[center,center]{$H_{\tn{\fsz{Zar}}}^p\big(X,\varOmega_{X/\QQ}^{p-1}\big).$}}
\end{pgftranslate}
\pgfputat{\pgfxy(16,.25)}{\pgfbox[center,center]{$(1.0.0.2)$}}
\end{pgfpicture}
\label{equgreengriffithstangentintro}

This definition brings {\it arithmetic} considerations to the forefront in the study of algebraic cycles, even in the case of complex algebraic geometry.


\section{General Approach}\label{sectionimyapproach}

The purpose of this book is to offer a broad and deep generalization of Green and Griffiths' approach to the infinitesimal theory of Chow groups, expressed in terms of a a construction which I will call the {\it coniveau machine}.  The coniveau machine extends, generalizes, and formalizes the use of filtration by codimension to analyze the infinitesimal structure of objects in an abelian category arising as sheaf cohomology groups of a generalized cohomology theory with supports.  The guiding example is, of course, the Chow groups, which arise as the sheaf cohomology groups of algebraic $K$-theory, viewed as a cohomology theory with supports on an appropriate category of schemes.   ``Precursors" of the coniveau machine appear in Green and Griffiths \cite{GreenGriffithsTangentSpaces05}, and in Sen Yang's thesis, \cite{SenYangThesis}.  


\subsection{Coniveau Machine in this Book}\label{subsectionconiveathisthesis}

{\bf Abstract General Form of the Coniveau Machine.} In its most general form, the coniveau machine may be represented by the following commutative diagram of functors and natural transformations, with ``exact rows:"

\begin{pgfpicture}{0cm}{0cm}{17cm}{3cm}
\begin{pgftranslate}{\pgfpoint{3.5cm}{-.2cm}}
\begin{pgfmagnify}{1.1}{1.1}
\pgfputat{\pgfxy(1.3,2.2)}{\pgfbox[center,center]{$E_{H_{\tn{\fsz{rel}}},\mbf{S}}$}}
\pgfputat{\pgfxy(4,2.2)}{\pgfbox[center,center]{$E_{H_{\tn {aug}},\mbf{S}}$}}
\pgfputat{\pgfxy(6.8,2.2)}{\pgfbox[center,center]{$E_{H,\mbf{S}}$}}
\pgfputat{\pgfxy(2.5,2.5)}{\pgfbox[center,center]{$i$}}
\pgfputat{\pgfxy(5.5,2.5)}{\pgfbox[center,center]{$j$}}
\pgfsetendarrow{\pgfarrowlargepointed{3pt}}
\pgfxyline(2,2.2)(3.2,2.2)
\pgfxyline(4.8,2.2)(6.2,2.2)
\begin{colormixin}{50!white}
\pgfputat{\pgfxy(1.3,.2)}{\pgfbox[center,center]{$E_{H_{\tn{\fsz{rel}}}^+,\mbf{S}}$}}
\pgfputat{\pgfxy(4,.2)}{\pgfbox[center,center]{$E_{H_{\tn{\fsz{aug}}}^+,\mbf{S}}$}}
\pgfputat{\pgfxy(6.8,.2)}{\pgfbox[center,center]{$E_{H^+,\mbf{S}}$}}
\pgfputat{\pgfxy(2.5,.5)}{\pgfbox[center,center]{$i^+$}}
\pgfputat{\pgfxy(5.5,.5)}{\pgfbox[center,center]{$j^+$}}
\pgfputat{\pgfxy(.8,1.2)}{\pgfbox[center,center]{$\tn{ch}_{\tn{\fsz{rel}}}$}}
\pgfputat{\pgfxy(1.5,1.1)}{\pgfbox[center,center]{$\sim$}}
\pgfputat{\pgfxy(3.5,1.2)}{\pgfbox[center,center]{$\tn{ch}_{\tn{\fsz{aug}}}$}}
\pgfputat{\pgfxy(6.4,1.2)}{\pgfbox[center,center]{$\tn{ch}$}}
\pgfsetendarrow{\pgfarrowlargepointed{3pt}}
\pgfxyline(2,.2)(3.2,.2)
\pgfxyline(4.8,.2)(6.2,.2)
\pgfxyline(1.3,1.8)(1.3,.5)
\pgfxyline(4,1.8)(4,.5)
\pgfxyline(6.8,1.8)(6.8,.5)
\end{colormixin}
\end{pgfmagnify}
\end{pgftranslate}
\pgfputat{\pgfxy(16,1)}{\pgfbox[center,center]{$(1.1.1.3)$}}
\end{pgfpicture}
\label{equconiveaumachinefunctorintro}


Tthe top right functor $E_{H,\mbf{S}}$ assigns to a topological space $X$ belonging to a distinguished category $\mbf{S}$ of topological spaces the {\it coniveau spectral sequence} $\{E_{H,X,r}\}$ associated with an appropriate {\it cohomology theory with supports} $H$ on the {\it category of pairs} $\mbf{P}$ over $\mbf{S}$.  Here, $H_{\tn{\footnotesize{aug}}}$ and $H_{\tn{\footnotesize{rel}}}$ are ``augmented" and ``relative" versions of $H$, respectively.  The functor $E_{H_{\tn{\fsz{aug}}},\mbf{S}}$ assigns to $X$ the coniveau spectral sequence associated with the augmented theory $H_{\tn{\fsz{aug}}}$, and the functor $E_{H_{\tn{\fsz{rel}}},\mbf{S}}$ is the ``fiber" of the ``epimorphism" $j$, with $i$ the ``fiber map."  The downward arrows and the bottom row of the diagram are shaded, to indicate that they exist only under the hypothesis that $H$ possesses an appropriate ``relative additive version" $H^+$, which is related to $H$ by a ``logarithmic-type map" labeled $\tn{ch}$.  Under this hypothesis, $H_{\tn{\fsz{aug}}}^+$ and $H_{\tn{\fsz{rel}}}^+$ are ``augmented" and ``relative" versions of the additive theory $H^+$, and the functor $E_{H_{\tn{\fsz{rel}}}^+,\mbf{S}}$ is the ``fiber" of the ``epimorphism" $j^+$, with $i^+$ being the ``fiber map." 


{\bf Algebraic $K$-Theory, Negative Cyclic Homology, and Chow Groups.}  This ``vague nonsense" acquires a concrete and useful meaning when $X$ is a smooth algebraic variety, $H$ is Bass-Thomason algebraic $K$-theory $K$, and $H^+$ is negative cyclic homology $\tn{HN}$.  In this case, $\tn{ch}$ is the generalized algebraic Chern character.  The rationale for this elaborate setup is based on the fact that when $K$-theory is ``augmented" by means of a ``nilpotent thickening" of the variety $X$, then the relative generalized algebraic Chern character $\tn{ch}_{\tn{\fsz{rel}}}$ is an isomorphism of functors.  The significance of this is that cyclic homology is generally much easier to compute than algebraic $K$-theory; in particular, it can often be expressed in terms of differential forms.  Since the goal is to analyze objects arising from $K$-theory; namely, Chow groups, the additive theory $H^+$ and its augmented version are not a necessary part of the picture; hence the coniveau machine may be simplified by eliminating them: 

\begin{pgfpicture}{0cm}{0cm}{17cm}{3cm}
\begin{pgftranslate}{\pgfpoint{3.5cm}{-.2cm}}
\begin{pgfmagnify}{1.1}{1.1}
\pgfputat{\pgfxy(1.3,2.2)}{\pgfbox[center,center]{$E_{K_{\tn{rel}},\mbf{S}}$}}
\pgfputat{\pgfxy(4,2.2)}{\pgfbox[center,center]{$E_{K_{\tn {aug}},\mbf{S}}$}}
\pgfputat{\pgfxy(6.8,2.2)}{\pgfbox[center,center]{$E_{K,\mbf{S}}$}}
\pgfputat{\pgfxy(2.5,2.5)}{\pgfbox[center,center]{$i$}}
\pgfputat{\pgfxy(5.5,2.5)}{\pgfbox[center,center]{$j$}}
\pgfsetendarrow{\pgfarrowlargepointed{3pt}}
\pgfxyline(2,2.2)(3.2,2.2)
\pgfxyline(4.8,2.2)(6.2,2.2)
\begin{colormixin}{50!white}
\pgfputat{\pgfxy(1.3,.2)}{\pgfbox[center,center]{$E_{\tn{HN}_{\tn{rel}},\mbf{S}}$}}
\pgfputat{\pgfxy(.8,1.2)}{\pgfbox[center,center]{$\tn{ch}_{\tn{\fsz{rel}}}$}}
\pgfputat{\pgfxy(1.5,1.1)}{\pgfbox[center,center]{$\sim$}}
\pgfsetendarrow{\pgfarrowlargepointed{3pt}}
\pgfxyline(1.3,1.8)(1.3,.5)
\end{colormixin}
\end{pgfmagnify}
\end{pgftranslate}
\pgfputat{\pgfxy(16,1)}{\pgfbox[center,center]{$(1.1.1.4)$}}
\end{pgfpicture}
\label{equconiveaumachinefunctorintro}

{\bf Simplified ``Four-Column Version" for a Nilpotent Thickening.} Now suppose that one is interested in analyzing the infinitesimal structure of a particular Chow group $\tn{Ch}_X^p$ of a particular smooth algebraic variety $X$ over a field $k$.   The meaning of ``infinitesimal structure" in this context is that the augmented theory $K_{\tn{aug}}$ is defined by adding ``nilpotent variables," providing ``infinitesimal degrees of freedom" for algebraic cycles on $X$.  Technically, this involves multiplying $X$, in the sense of fiber products, by the prime spectrum $Y$ of a $k$-algebra generated over $k$ by nilpotent elements; e.g., a local artinian $k$-algebra.  In this case, the transformations $i$ and $j$ in equation \hyperref[equconiveaumachinefunctorintro]{1.1.0.4} split; i.e., there are arrows in the opposite directions.  

Due to Bloch's formula \hyperref[blochstheoremintro]{\ref{blochstheoremintro}}, the objects of interest in this context are the sheaf cohomology groups of the $p$th $K$-theory sheaf $\ms{K}_{p,X}$ on $X$, the augmented $K$-theory sheaf $\ms{K}_{p,X\times_kY}$, the relative $K$-theory sheaf $\ms{K}_{p,X\times_kY,Y}$ and the corresponding ``relative additive $K$-theory sheaf,"  which is the relative negative cyclic homology sheaf $\ms{HN}_{p,X\times_kY,Y}$.  The sheaf cohomology groups of these sheaves may be computed via their sheafified Cousin complexes, which are flasque resolutions in this case.  These complexes arise by sheafifying the $-p$th rows of the corresponding spectral sequences.  

Transposing these sheaf complexes into columns yields the schematic diagram appearing in figure \hyperref[figsimplifiedfourcolumnKtheoryintro]{\ref{figsimplifiedfourcolumnKtheoryintro}} below.  Note that the diagram has been rearranged so that the arrows go from left to right.

\begin{figure}[H]
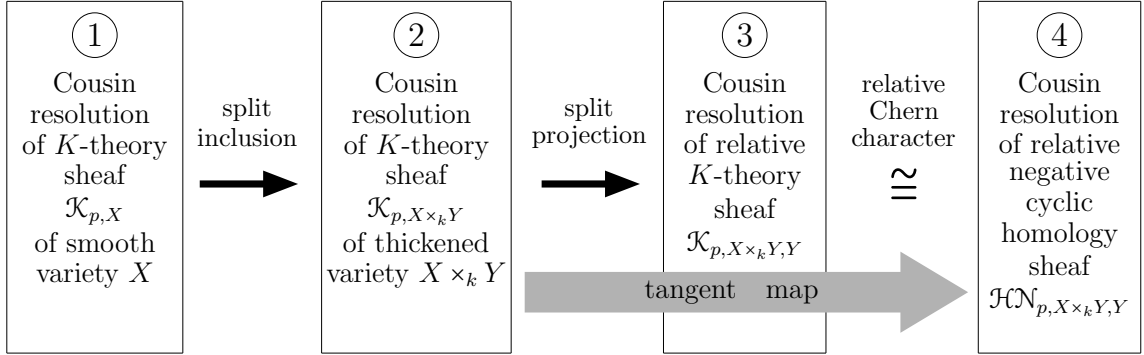

\begin{pgfpicture}{0cm}{0cm}{17cm}{4.7cm}
\begin{pgfmagnify}{.9}{.9}
\begin{pgftranslate}{\pgfpoint{.2cm}{-1.7cm}}
\begin{pgftranslate}{\pgfpoint{-.5cm}{0cm}}
\pgfxyline(.9,1.8)(.9,7)
\pgfxyline(.9,7)(3.5,7)
\pgfxyline(3.5,7)(3.5,1.8)
\pgfxyline(3.5,1.8)(.9,1.8)
\pgfputat{\pgfxy(2.2,5.8)}{\pgfbox[center,center]{Cousin}}
\pgfputat{\pgfxy(2.2,5.35)}{\pgfbox[center,center]{resolution }}
\pgfputat{\pgfxy(2.2,4.85)}{\pgfbox[center,center]{of $K$-theory}}
\pgfputat{\pgfxy(2.2,4.45)}{\pgfbox[center,center]{sheaf}}
\pgfputat{\pgfxy(2.2,3.9)}{\pgfbox[center,center]{$\ms{K}_{p,X}$}}
\pgfputat{\pgfxy(2.2,3.45)}{\pgfbox[center,center]{of smooth}}
\pgfputat{\pgfxy(2.2,2.95)}{\pgfbox[center,center]{variety $X$}}
\pgfnodecircle{Node0}[stroke]{\pgfxy(2.2,6.5)}{0.33cm}
\pgfputat{\pgfxy(2.2,6.5)}{\pgfbox[center,center]{\large{$1$}}}
\end{pgftranslate}
\begin{pgftranslate}{\pgfpoint{4.05cm}{0cm}}
\pgfxyline(1,1.8)(1,7)
\pgfxyline(1,7)(3.8,7)
\pgfxyline(3.8,7)(3.8,1.8)
\pgfxyline(3.8,1.8)(1,1.8)
\pgfputat{\pgfxy(2.4,5.8)}{\pgfbox[center,center]{Cousin}}
\pgfputat{\pgfxy(2.4,5.35)}{\pgfbox[center,center]{resolution }}
\pgfputat{\pgfxy(2.4,4.85)}{\pgfbox[center,center]{of $K$-theory}}
\pgfputat{\pgfxy(2.4,4.45)}{\pgfbox[center,center]{sheaf}}
\pgfputat{\pgfxy(2.4,3.9)}{\pgfbox[center,center]{$\ms{K}_{p,X\times_kY}$}}
\pgfputat{\pgfxy(2.4,3.45)}{\pgfbox[center,center]{of thickened}}
\pgfputat{\pgfxy(2.4,2.95)}{\pgfbox[center,center]{variety $X\times_kY$}}
 \pgfnodecircle{Node0}[stroke]{\pgfxy(2.4,6.5)}{0.33cm}
\pgfputat{\pgfxy(2.4,6.5)}{\pgfbox[center,center]{\large{$2$}}}
\pgfsetendarrow{\pgfarrowtriangle{6pt}}
\pgfsetlinewidth{3pt}
\pgfxyline(-.8,4.3)(.4,4.3)
\pgfputat{\pgfxy(-.1,5.4)}{\pgfbox[center,center]{\small{split}}}
\pgfputat{\pgfxy(-.1,5)}{\pgfbox[center,center]{\small{inclusion}}}
\end{pgftranslate}
\begin{pgftranslate}{\pgfpoint{9.1cm}{0cm}}
\pgfxyline(1,1.8)(1,7)
\pgfxyline(1,7)(3.4,7)
\pgfxyline(3.4,7)(3.4,1.8)
\pgfxyline(3.4,1.8)(1,1.8)
\pgfputat{\pgfxy(2.2,5.8)}{\pgfbox[center,center]{Cousin}}
\pgfputat{\pgfxy(2.2,5.35)}{\pgfbox[center,center]{resolution}}
\pgfputat{\pgfxy(2.2,4.9)}{\pgfbox[center,center]{of relative}}
\pgfputat{\pgfxy(2.2,4.4)}{\pgfbox[center,center]{$K$-theory}}
\pgfputat{\pgfxy(2.2,3.95)}{\pgfbox[center,center]{sheaf}}
\pgfputat{\pgfxy(2.2,3.4)}{\pgfbox[center,center]{$\ms{K}_{p,X\times_kY,Y}$}}
\pgfnodecircle{Node0}[stroke]{\pgfxy(2.2,6.5)}{0.33cm}
\pgfputat{\pgfxy(2.2,6.5)}{\pgfbox[center,center]{\large{$3$}}}
\pgfsetendarrow{\pgfarrowtriangle{6pt}}
\pgfsetlinewidth{3pt}
\pgfxyline(-.8,4.3)(.4,4.3)
\pgfputat{\pgfxy(-.1,5.4)}{\pgfbox[center,center]{\small{split}}}
\pgfputat{\pgfxy(-.1,5)}{\pgfbox[center,center]{\small{projection}}}
\end{pgftranslate}
\begin{pgftranslate}{\pgfpoint{13.75cm}{0cm}}
\pgfxyline(1,1.8)(1,7)
\pgfxyline(1,7)(3.4,7)
\pgfxyline(3.4,7)(3.4,1.8)
\pgfxyline(3.4,1.8)(1,1.8)
\pgfputat{\pgfxy(2.2,5.8)}{\pgfbox[center,center]{Cousin}}
\pgfputat{\pgfxy(2.2,5.35)}{\pgfbox[center,center]{resolution}}
\pgfputat{\pgfxy(2.2,4.9)}{\pgfbox[center,center]{of relative}}
\pgfputat{\pgfxy(2.2,4.45)}{\pgfbox[center,center]{negative}}
\pgfputat{\pgfxy(2.2,4)}{\pgfbox[center,center]{cyclic}}
\pgfputat{\pgfxy(2.2,3.55)}{\pgfbox[center,center]{homology}}
\pgfputat{\pgfxy(2.2,3.1)}{\pgfbox[center,center]{sheaf}}
\pgfputat{\pgfxy(2.2,2.55)}{\pgfbox[center,center]{$\ms{HN}_{p,X\times_kY,Y}$}}
\pgfputat{\pgfxy(-.1,5.8)}{\pgfbox[center,center]{\small{relative}}}
\pgfputat{\pgfxy(-.1,5.4)}{\pgfbox[center,center]{\small{Chern}}}
\pgfputat{\pgfxy(-.1,5)}{\pgfbox[center,center]{\small{character}}}
\pgfputat{\pgfxy(-.1,4.3)}{\pgfbox[center,center]{\huge{$\cong$}}}
\pgfnodecircle{Node0}[stroke]{\pgfxy(2.2,6.5)}{0.33cm}
\pgfputat{\pgfxy(2.2,6.5)}{\pgfbox[center,center]{\large{$4$}}}
\end{pgftranslate}
\begin{colormixin}{30!white}
\color{black}
\pgfmoveto{\pgfxy(8.05,3)}
\pgflineto{\pgfxy(13.6,3)}
\pgflineto{\pgfxy(13.6,3.3)}
\pgflineto{\pgfxy(14.6,2.7)}
\pgflineto{\pgfxy(13.6,2.1)}
\pgflineto{\pgfxy(13.6,2.4)}
\pgflineto{\pgfxy(8.05,2.4)}
\pgflineto{\pgfxy(8.05,3)}
\pgffill
\end{colormixin}
\pgfputat{\pgfxy(10.5,2.7)}{\pgfbox[center,center]{tangent}}
\pgfputat{\pgfxy(12,2.67)}{\pgfbox[center,center]{map}}
\end{pgftranslate}
\end{pgfmagnify}
\end{pgfpicture}
\caption{Simplified ``four column version" of the coniveau machine for algebraic $K$-theory on a smooth algebraic variety in the case of a nilpotent thickening.}
\label{figsimplifiedfourcolumnKtheoryintro}
\end{figure}
\vspace*{-.5cm}


{\bf Toy Version for Smooth Complex Projective Curves.}  For illustrative purposes, it is useful to examine a ``toy version" of the coniveau machine, given by restricting $X$ to be a smooth algebraic curve over the complex numbers.   This is the subject of Chapter \hyperref[chaptercurves]{\ref{chaptercurves}} of this book.  In this context, I examine the ``simplest possible" augmented version of algebraic $K$-theory; namely, the version given by taking the scheme $Y$ in the schematic diagram of figure \hyperref[figsimplifiedfourcolumnKtheoryintro]{\ref{figsimplifiedfourcolumnKtheoryintro}} above to the the prime spectrum of the algebra of dual numbers $\CC[\ee]/\ee^2$.   Since the only interesting Chow group of a smooth projective curve $X$ is $\tn{Ch}_X^1$, I set $p=1$ in figure \hyperref[figsimplifiedfourcolumnKtheoryintro]{\ref{figsimplifiedfourcolumnKtheoryintro}}.   Computing the resulting $K$-theory and negative cyclic homology sheaves in terms of ``elementary objects" then yields the diagram shown in figure \hyperref[figconiveaucurveintro]{\ref{figconiveaucurveintro}} below:  

\begin{figure}[H]
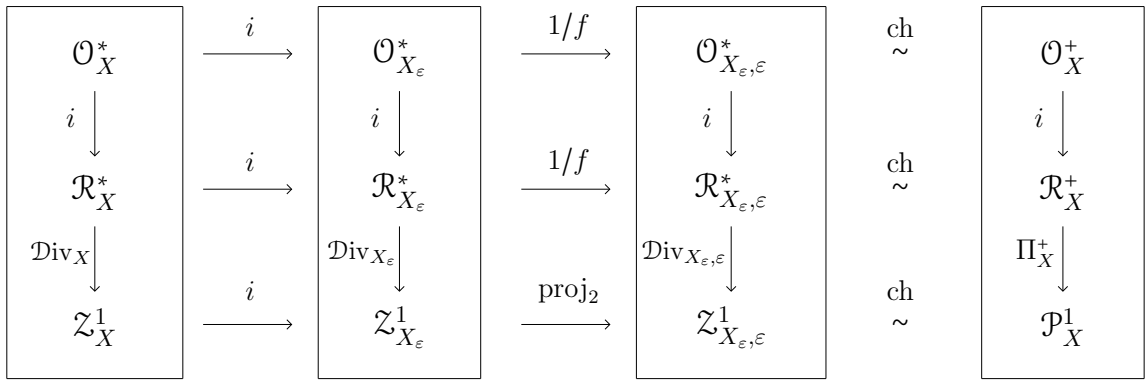

\begin{pgfpicture}{0cm}{0cm}{17cm}{5cm}
\begin{pgfmagnify}{.9}{.9}
\begin{pgftranslate}{\pgfpoint{.2cm}{-1.3cm}}
\begin{pgftranslate}{\pgfpoint{-.5cm}{0cm}}
\pgfxyline(.9,1.5)(.9,7)
\pgfxyline(.9,7)(3.5,7)
\pgfxyline(3.5,7)(3.5,1.5)
\pgfxyline(3.5,1.5)(.9,1.5)
\pgfputat{\pgfxy(2.2,6.25)}{\pgfbox[center,center]{\large{$\ms{O}_X^*$}}}
\pgfputat{\pgfxy(2.2,4.25)}{\pgfbox[center,center]{\large{$\ms{R}^*_X$}}}
\pgfputat{\pgfxy(2.2,2.25)}{\pgfbox[center,center]{\large{$\ms{Z}_X^1$}}}
\pgfputat{\pgfxy(1.85,5.35)}{\pgfbox[center,center]{$i$}}
\pgfputat{\pgfxy(1.7,3.35)}{\pgfbox[center,center]{\small{$\ms{D}\tn{iv}_X$}}}
\pgfsetendarrow{\pgfarrowpointed{3pt}}
\pgfxyline(2.2,5.75)(2.2,4.8)
\pgfxyline(2.2,3.75)(2.2,2.8)
\end{pgftranslate}
\begin{pgftranslate}{\pgfpoint{4cm}{0cm}}
\pgfxyline(1,1.5)(1,7)
\pgfxyline(1,7)(3.4,7)
\pgfxyline(3.4,7)(3.4,1.5)
\pgfxyline(3.4,1.5)(1,1.5)
\pgfputat{\pgfxy(2.2,6.25)}{\pgfbox[center,center]{\large{$\ms{O}^*_{X_\ee}$}}}
\pgfputat{\pgfxy(2.2,4.25)}{\pgfbox[center,center]{\large{$\ms{R}^*_{X_\ee}$}}}
\pgfputat{\pgfxy(1.85,5.35)}{\pgfbox[center,center]{$i$}}
\pgfsetendarrow{\pgfarrowpointed{3pt}}
\pgfxyline(-.7,6.3)(.6,6.3)
\pgfxyline(-.7,4.3)(.6,4.3)
\pgfxyline(2.2,5.75)(2.2,4.8)
\pgfputat{\pgfxy(0,6.7)}{\pgfbox[center,center]{$i$}}
\pgfputat{\pgfxy(0,4.7)}{\pgfbox[center,center]{$i$}}
\pgfputat{\pgfxy(2.2,2.25)}{\pgfbox[center,center]{\large{$\ms{Z}^1_{X_\ee}$}}}
\pgfputat{\pgfxy(1.65,3.35)}{\pgfbox[center,center]{\small{$\ms{D}\tn{iv}_{X_\ee}$}}}
\pgfputat{\pgfxy(0,2.8)}{\pgfbox[center,center]{$i$}}
\pgfsetendarrow{\pgfarrowpointed{3pt}}
\pgfxyline(-.7,2.3)(.6,2.3)
\pgfxyline(2.2,3.75)(2.2,2.8)
\end{pgftranslate}
\begin{pgftranslate}{\pgfpoint{8.9cm}{0cm}}
\pgfxyline(.8,1.5)(.8,7)
\pgfxyline(.8,7)(3.6,7)
\pgfxyline(3.6,7)(3.6,1.5)
\pgfxyline(3.6,1.5)(.8,1.5)
\pgfputat{\pgfxy(2.2,6.25)}{\pgfbox[center,center]{\large{$\ms{O}^*_{X_\ee,\ee}$}}}
\pgfputat{\pgfxy(2.2,4.25)}{\pgfbox[center,center]{\large{$\ms{R}^*_{X_\ee,\ee}$}}}
\pgfputat{\pgfxy(1.85,5.35)}{\pgfbox[center,center]{$i$}}
\pgfsetendarrow{\pgfarrowtriangle{6pt}}
\pgfsetendarrow{\pgfarrowpointed{3pt}}
\pgfxyline(-.9,6.3)(.4,6.3)
\pgfxyline(-.9,4.3)(.4,4.3)
\pgfxyline(2.2,5.75)(2.2,4.8)
\pgfputat{\pgfxy(-.2,6.7)}{\pgfbox[center,center]{$1/f$}}
\pgfputat{\pgfxy(-.2,4.7)}{\pgfbox[center,center]{$1/f$}}
\pgfputat{\pgfxy(2.2,2.25)}{\pgfbox[center,center]{\large{$\ms{Z}^1_{X_\ee,\ee}$}}}
\pgfputat{\pgfxy(1.5,3.35)}{\pgfbox[center,center]{\small{$\ms{D}\tn{iv}_{X_\ee,\ee}$}}}
\pgfputat{\pgfxy(-.2,2.8)}{\pgfbox[center,center]{$\tn{proj}_2$}}
\pgfsetendarrow{\pgfarrowpointed{3pt}}
\pgfxyline(-.9,2.3)(.4,2.3)
\pgfxyline(2.2,3.75)(2.2,2.8)
\end{pgftranslate}
\begin{pgftranslate}{\pgfpoint{13.8cm}{0cm}}
\pgfxyline(1,1.5)(1,7)
\pgfxyline(1,7)(3.4,7)
\pgfxyline(3.4,7)(3.4,1.5)
\pgfxyline(3.4,1.5)(1,1.5)
\pgfputat{\pgfxy(2.2,6.25)}{\pgfbox[center,center]{\large{$\ms{O}_X^+$}}}
\pgfputat{\pgfxy(2.2,4.25)}{\pgfbox[center,center]{\large{$\ms{R}_X^+$}}}
\pgfputat{\pgfxy(-.2,6.7)}{\pgfbox[center,center]{$\tn{ch}$}}
\pgfputat{\pgfxy(1.85,5.35)}{\pgfbox[center,center]{$i$}}
\pgfputat{\pgfxy(-.2,4.7)}{\pgfbox[center,center]{$\tn{ch}$}}
\pgfputat{\pgfxy(-.2,6.3)}{\pgfbox[center,center]{\large{$\sim$}}}
\pgfputat{\pgfxy(-.2,4.3)}{\pgfbox[center,center]{\large{$\sim$}}}
\pgfsetendarrow{\pgfarrowpointed{3pt}}
\pgfxyline(2.2,5.75)(2.2,4.8)
\pgfputat{\pgfxy(2.2,2.25)}{\pgfbox[center,center]{\large{$\ms{P}_X^1$}}}
\pgfputat{\pgfxy(1.8,3.35)}{\pgfbox[center,center]{$\Pi_X^+$}}
\pgfputat{\pgfxy(-.2,2.8)}{\pgfbox[center,center]{$\tn{ch}$}}
\pgfsetendarrow{\pgfarrowpointed{3pt}}
\pgfxyline(2.2,3.75)(2.2,2.8)
\pgfputat{\pgfxy(-.2,2.3)}{\pgfbox[center,center]{\large{$\sim$}}}
\end{pgftranslate}
\end{pgftranslate}
\end{pgfmagnify}
\end{pgfpicture}
\caption{Coniveau machine for $\tn{Ch}_X^1$ of a smooth complex projective algebraic curve $X$.}
\label{figconiveaucurveintro}
\end{figure}
\vspace*{-.5cm}

The first column on the left in figure \hyperref[figconiveaucurveintro]{\ref{figconiveaucurveintro}} is the familiar {\it sheaf divisor sequence.}  This sequence is the Cousin flasque resolution of the sheaf $\ms{O}_X^*$ of multiplicative groups of invertible regular functions on $X$, which coincides with the first sheaf of algebraic $K$-theory $\ms{K}_{1,X}$ on $X$.   The first Zariski sheaf cohomology group $H_{\tn{Zar}}^1(X,\ms{O}_X^*)$ is therefore the first Chow group $\tn{Ch}_X^1$ of $X$ by Bloch's formula \hyperref[blochstheoremintro]{\ref{blochstheoremintro}}.  The second column  is the ``thickened divisor sequence," given by tensoring the structure sheaf $\ms{O}_X$ of $X$ with the algebra of dual numbers $\CC[\ee]/\ee^2$.  This is the ``sheaf-level manifestation" of taking the fiber product $X\times_k \tn{Spec } \CC[\ee]/\ee^2$.  The third column is the ``relative sequence" defined in terms of the first two columns.  The fourth column is the ``additive version" of the relative sequence.  As noted by Green and Griffith \cite{GreenGriffithsTangentSpaces05}, the fourth column may be recognized in this case as the so-called {\it principal parts sequence.}   This viewpoint may be understood very concretely in terms of {\it Laurent tails} in the case of complex algebraic curves, but ultimately involves profound connections among algebraic $K$-theory, local cohomology, cyclic homology, and differential forms. 


\subsection{Precursors in the Literature}\label{subsectionconiveathisthesis}

{\bf Green and Griffiths' Version for an Algebraic Surface.}  Green and Griffiths \cite{GreenGriffithsTangentSpaces05} use a rudimentary version of the coniveau machine, though they do not use this terminology.  In particular, they work directly in terms of sheaf resolutions rather than beginning with the coniveau spectral sequence.  For purposes of comparison, I include a diagram in figure \hyperref[figconiveausurfaceGG]{\ref{figconiveausurfaceGG}} below representing their version of the machine for analyzing the first-order infinitesimal theory of $\tn{Ch}_X^2$ for a smooth algebraic surface $X$, using their notation.  The clearest reference to this construction appears in the proposition on page 131 of \cite{GreenGriffithsTangentSpaces05}, where Green and Griffiths state, {\it``The tangent sequence to the Bloch-Gersten-Quillen sequence \tn{[column 1]} is the Cousin flasque resolution of $\varOmega_{X/\QQ}^1$ \tn{[column 4]}."} 

\begin{figure}[H]
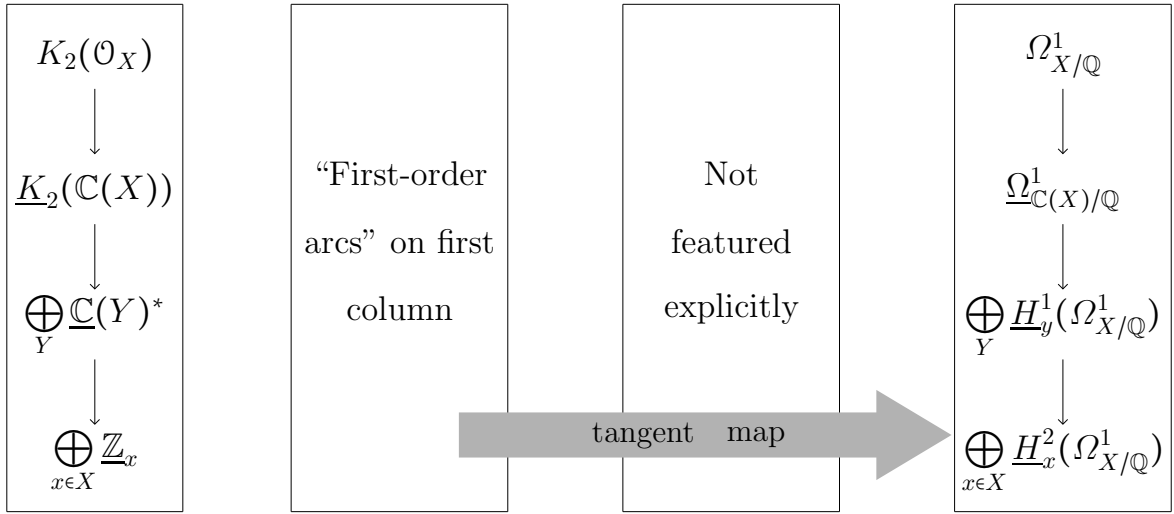

\begin{pgfpicture}{0cm}{0cm}{17cm}{7cm}
\begin{pgfmagnify}{.9}{.9}
\begin{pgftranslate}{\pgfpoint{.2cm}{-1.3cm}}
\begin{pgftranslate}{\pgfpoint{-.5cm}{0cm}}
\pgfxyline(.9,1.5)(.9,9)
\pgfxyline(.9,9)(3.5,9)
\pgfxyline(3.5,9)(3.5,1.5)
\pgfxyline(3.5,1.5)(.9,1.5)
\pgfputat{\pgfxy(2.2,8.25)}{\pgfbox[center,center]{\large{$K_2(\ms{O}_X)$}}}
\pgfputat{\pgfxy(2.2,6.25)}{\pgfbox[center,center]{\large{$\underline{K}_2(\CC(X))$}}}
\pgfputat{\pgfxy(2.2,4.25)}{\pgfbox[center,center]{\large{$\displaystyle\bigoplus_Y\underline{\CC}(Y)^*$}}}
\pgfputat{\pgfxy(2.2,2.25)}{\pgfbox[center,center]{\large{$\displaystyle\bigoplus_{x\in X}\underline{\ZZ}_x$}}}
\pgfsetendarrow{\pgfarrowpointed{3pt}}
\pgfxyline(2.2,7.75)(2.2,6.8)
\pgfxyline(2.2,5.75)(2.2,4.8)
\pgfxyline(2.2,3.75)(2.2,2.8)
\end{pgftranslate}
\begin{pgftranslate}{\pgfpoint{4cm}{0cm}}
\pgfxyline(.6,1.5)(.6,9)
\pgfxyline(.6,9)(3.8,9)
\pgfxyline(3.8,9)(3.8,1.5)
\pgfxyline(3.8,1.5)(.6,1.5)
\pgfputat{\pgfxy(2.2,6.5)}{\pgfbox[center,center]{\large{``First-order}}}
\pgfputat{\pgfxy(2.2,5.5)}{\pgfbox[center,center]{\large{arcs" on first}}}
\pgfputat{\pgfxy(2.2,4.5)}{\pgfbox[center,center]{\large{column}}}
\end{pgftranslate}
\begin{pgftranslate}{\pgfpoint{8.9cm}{0cm}}
\pgfxyline(.6,1.5)(.6,9)
\pgfxyline(.6,9)(3.8,9)
\pgfxyline(3.8,9)(3.8,1.5)
\pgfxyline(3.8,1.5)(.6,1.5)
\pgfputat{\pgfxy(2.2,6.5)}{\pgfbox[center,center]{\large{Not}}}
\pgfputat{\pgfxy(2.2,5.5)}{\pgfbox[center,center]{\large{featured}}}
\pgfputat{\pgfxy(2.2,4.5)}{\pgfbox[center,center]{\large{explicitly}}}
\end{pgftranslate}
\begin{pgftranslate}{\pgfpoint{13.8cm}{0cm}}
\pgfxyline(.6,1.5)(.6,9)
\pgfxyline(.6,9)(3.8,9)
\pgfxyline(3.8,9)(3.8,1.5)
\pgfxyline(3.8,1.5)(.6,1.5)
\pgfputat{\pgfxy(2.2,8.25)}{\pgfbox[center,center]{\large{$\varOmega_{X/\QQ}^1$}}}
\pgfputat{\pgfxy(2.2,6.25)}{\pgfbox[center,center]{\large{$\underline{\Omega}_{\CC(X)/\QQ}^1$}}}
\pgfputat{\pgfxy(2.2,4.25)}{\pgfbox[center,center]{\large{$\displaystyle\bigoplus_Y\underline{H}_y^1(\varOmega_{X/\QQ}^1)$}}}
\pgfsetendarrow{\pgfarrowpointed{3pt}}
\pgfxyline(2.2,5.75)(2.2,4.8)
\pgfputat{\pgfxy(2.2,2.25)}{\pgfbox[center,center]{\large{$\displaystyle\bigoplus_{x\in X}\underline{H}_{x}^2(\varOmega_{X/\QQ}^1)$}}}
\pgfsetendarrow{\pgfarrowpointed{3pt}}
\pgfxyline(2.2,3.75)(2.2,2.8)
\pgfxyline(2.2,7.75)(2.2,6.8)
\end{pgftranslate}
\end{pgftranslate}
\end{pgfmagnify}
\begin{pgftranslate}{\pgfpoint{-1.5cm}{-1.5cm}}
\begin{colormixin}{30!white}
\color{black}
\pgfmoveto{\pgfxy(8.05,3)}
\pgflineto{\pgfxy(13.6,3)}
\pgflineto{\pgfxy(13.6,3.3)}
\pgflineto{\pgfxy(14.6,2.7)}
\pgflineto{\pgfxy(13.6,2.1)}
\pgflineto{\pgfxy(13.6,2.4)}
\pgflineto{\pgfxy(8.05,2.4)}
\pgflineto{\pgfxy(8.05,3)}
\pgffill
\end{colormixin}
\pgfputat{\pgfxy(10.5,2.7)}{\pgfbox[center,center]{tangent}}
\pgfputat{\pgfxy(12,2.67)}{\pgfbox[center,center]{map}}
\end{pgftranslate}
\end{pgfpicture}
\caption{``Coniveau machine" of Green and Griffiths for $\tn{Ch}_X^2$ of a smooth algebraic surface $X$.}
\label{figconiveausurfaceGG}
\end{figure}
\vspace*{-.5cm}

Here, the ``Bloch-Gersten-Quillen" sequence appearing in the first column is the sheafified Cousin complex coming from the row in degree $-2$ of the coniveau spectral sequence for algebraic $K$-theory on $X$, while the ``Cousin flasque resolution of $\varOmega_{X/\QQ}^1$" appearing in the fourth column is the sheafified Cousin complex coming from the row in degree $-2$ of the coniveau spectral sequence for relative negative cyclic homology on $X$.  In this particular case, negative cyclic homology reduces to the absolute K\"{a}hler differentials $\varOmega_{X/\QQ}^1$, since the augmentation is taken with respect to the algebra of dual numbers.  The ``tangent map" takes ``first-order arcs" on the first column to their ``tangent elements" in the fourth column.   The tangent map descends to give maps on sheaf cohomology; in particular, it takes ``first-order arcs" in $\tn{Ch}_X^2$ to their ``tangents" in $H_{\tn{Zar}}^2(X,\varOmega_{X/\QQ}^1)$, which is identified as the ``tangent group at the identity" of $\tn{Ch}_X^2$.  


{\bf Sen Yang's Version.}  In his recent thesis \cite{SenYangThesis}, my colleague Sen Yang presents a version of the coniveau machine, using different terminology, which enables analysis of the first-order infinitesimal theory\footnotemark\footnotetext{I believe Sen has extended this to higher-order, but his thesis states $\ee=2$.  Also, the definition of $\Omega^\bullet_{O_X/\QQ}$ gives the correct answer only if $\ee=0$, by Hesselholt's theorem on relative $K$-theory of truncated polynomial algebras \cite{HesselholtTruncated05}.} of the $m$th Chow group $\tn{Ch}_X^m$ of a smooth $n$-dimensional algebraic variety $X$ over a field $k$ of characteristic zero.  Yang's version is represented by the commutative diagram in figure \hyperref[figSenversion]{\ref{figSenversion}} below, using his notation.\footnotemark\footnotetext{Actually, I ``correct" what I regard as a few small typos here.}  The columns are exact and the rows split exact.

\begin{figure}[H]
\begin{pgfpicture}{0cm}{0cm}{17cm}{13cm}
\begin{pgftranslate}{\pgfpoint{1cm}{-.5cm}}
\pgfputat{\pgfxy(1,13.5)}{\pgfbox[center,center]{$0$}}
\pgfputat{\pgfxy(1,12)}{\pgfbox[center,center]{$\Omega^\bullet_{O_X/\QQ}$}}
\pgfputat{\pgfxy(1,10)}{\pgfbox[center,center]{$\Omega^\bullet_{k(X)/\QQ}$}}
\pgfputat{\pgfxy(1,8)}{\pgfbox[center,center]{$\displaystyle\bigoplus_{d\in X^{(1)}}\underline{H}^1_d(\Omega^\bullet_{O_X/\QQ})$}}
\pgfputat{\pgfxy(1,6)}{\pgfbox[center,center]{$\displaystyle\bigoplus_{y\in X^{(2)}}\underline{H}^2_y(\Omega^\bullet_{O_X/\QQ})$}}
\pgfputat{\pgfxy(1,4.2)}{\pgfbox[center,center]{$\vdots$}}
\pgfputat{\pgfxy(1,2)}{\pgfbox[center,center]{$\displaystyle\bigoplus_{x\in X^{(n)}}\underline{H}^n_x(\Omega^\bullet_{O_X/\QQ})$}}
\pgfputat{\pgfxy(1,.7)}{\pgfbox[center,center]{$0$}}
\pgfputat{\pgfxy(7,13.5)}{\pgfbox[center,center]{$0$}}
\pgfputat{\pgfxy(7,12)}{\pgfbox[center,center]{$K_m(O_{X[\ee]})$}}
\pgfputat{\pgfxy(7,10)}{\pgfbox[center,center]{$K_m(k(X)[\ee])$}}
\pgfputat{\pgfxy(7,8)}{\pgfbox[center,center]{$\displaystyle\bigoplus_{d[\ee]\in X[\ee]^{(1)}}\underline{K}_{m-1}(O_{X[\ee]}\tn{ on } d[\ee])$}}
\pgfputat{\pgfxy(7,6)}{\pgfbox[center,center]{$\displaystyle\bigoplus_{y[\ee]\in X[\ee]^{(2)}}\underline{K}_{m-2}(O_{X[\ee]}\tn{ on } y[\ee])$}}
\pgfputat{\pgfxy(7,4.2)}{\pgfbox[center,center]{$\vdots$}}
\pgfputat{\pgfxy(7,2)}{\pgfbox[center,center]{$\displaystyle\bigoplus_{x[\ee]\in X[\ee]^{(n)}}\underline{K}_{m-n}(O_{X[\ee]}\tn{ on } x[\ee])$}}
\pgfputat{\pgfxy(7,.7)}{\pgfbox[center,center]{$0$}}
\pgfputat{\pgfxy(13,13.5)}{\pgfbox[center,center]{$0$}}
\pgfputat{\pgfxy(13,12)}{\pgfbox[center,center]{$K_m(O_{X})$}}
\pgfputat{\pgfxy(13,10)}{\pgfbox[center,center]{$K_m(k(X))$}}
\pgfputat{\pgfxy(13,8)}{\pgfbox[center,center]{$\displaystyle\bigoplus_{d\in X^{(1)}}\underline{K}_{m-1}(O_{X}\tn{ on } d)$}}
\pgfputat{\pgfxy(13,6)}{\pgfbox[center,center]{$\displaystyle\bigoplus_{y\in X^{(2)}}\underline{K}_{m-2}(O_{X}\tn{ on } y)$}}
\pgfputat{\pgfxy(13,4.2)}{\pgfbox[center,center]{$\vdots$}}
\pgfputat{\pgfxy(13,2)}{\pgfbox[center,center]{$\displaystyle\bigoplus_{x\in X^{(n)}}\underline{K}_{m-n}(O_{X}\tn{ on } x)$}}
\pgfputat{\pgfxy(13,.7)}{\pgfbox[center,center]{$0$}}
\pgfsetendarrow{\pgfarrowpointed{3pt}}
\pgfxyline(1,13)(1,12.5)
\pgfxyline(1,11.5)(1,10.5)
\pgfxyline(1,9.5)(1,8.5)
\pgfxyline(1,7.7)(1,6.7)
\pgfxyline(1,5.7)(1,4.7)
\pgfxyline(1,3.7)(1,2.7)
\pgfxyline(1,1.7)(1,1.2)
\pgfxyline(7,13)(7,12.5)
\pgfxyline(7,11.5)(7,10.5)
\pgfxyline(7,9.5)(7,8.5)
\pgfxyline(7,7.7)(7,6.7)
\pgfxyline(7,5.7)(7,4.7)
\pgfxyline(7,3.7)(7,2.7)
\pgfxyline(7,1.7)(7,1.2)
\pgfxyline(13,13)(13,12.5)
\pgfxyline(13,11.5)(13,10.5)
\pgfxyline(13,9.5)(13,8.5)
\pgfxyline(13,7.7)(13,6.7)
\pgfxyline(13,5.7)(13,4.7)
\pgfxyline(13,3.7)(13,2.7)
\pgfxyline(13,1.7)(13,1.2)
\pgfxyline(11.5,12)(8.5,12)
\pgfxyline(5.1,12)(2.1,12)
\pgfxyline(11.5,10)(8.5,10)
\pgfxyline(5.1,10)(2.1,10)
\pgfxyline(10.75,8.15)(10,8.15)
\pgfxyline(4.6,8.15)(2.6,8.15)
\pgfxyline(10.75,6.15)(10,6.15)
\pgfxyline(4.6,6.15)(2.6,6.15)
\pgfxyline(10.75,2.15)(10,2.15)
\pgfxyline(4.6,2.15)(2.6,2.15)
\pgfputat{\pgfxy(3.6,12.3)}{\pgfbox[center,center]{\fsz{$tan1$}}}
\pgfputat{\pgfxy(3.6,10.3)}{\pgfbox[center,center]{\fsz{$tan2$}}}
\pgfputat{\pgfxy(3.6,8.45)}{\pgfbox[center,center]{\fsz{$tan3$}}}
\pgfputat{\pgfxy(3.6,6.45)}{\pgfbox[center,center]{\fsz{$tan4$}}}
\pgfputat{\pgfxy(3.6,2.45)}{\pgfbox[center,center]{\fsz{$tan(n+2)$}}}
\end{pgftranslate}
\end{pgfpicture}
\caption{``Coniveau machine" of Green and Griffiths for $\tn{Ch}_X^2$ of a smooth algebraic surface $X$.}
\label{figSenversion}
\end{figure}

\vspace*{-.5cm}

Here, the notation $\Omega^\bullet_{O_X/\QQ}$ denotes the direct sum $\Omega^{m-1}_{O_X/\QQ}\oplus\Omega^{m-3}_{O_X/\QQ}\oplus...$, and underlines indicate skyscraper sheaves defined by the corresponding groups at the points indexing each summand.  

Yang's first column on the left corresponds to the column on the far right in figure \hyperref[figsimplifiedfourcolumnKtheoryintro]{\ref{figsimplifiedfourcolumnKtheoryintro}} above, and the arrows are reversed.  As in Green and Griffith's version, the column involving relative $K$-theory is absent.  Note that the sign of the index $m-n$ in the bottom row may be positive, zero, or negative.  This reflects the existence of nontrivial $K$-theory groups with supports in negative degrees for Bass-Thomason $K$-theory.

\section{Points of Departure}\label{sectionpointsofdeparture}

The principal foundations of this book are as follows:
\begin{enumerate}
\item The general theory of filtration of a topological space by the codimensions of its points, or {\it coniveau filtration},\footnotemark\footnotetext{The French word {\it coniveau} means ``codimension" in this context.} developed by Grothendieck and Hartshorne \cite{HartshorneResiduesDuality66}.  This theory leads to the definition of the {\it coniveau spectral sequence} underlying the construction of the coniveau machine. 
\item Bloch's formula for the Chow groups \cite{BlochK2Cycles74}, extended by Quillen \cite{QuillenHigherKTheoryI72}, first cited here in equation \hyperref[blochstheoremintro]{\ref{blochstheoremintro}} above.  This formula expresses the Chow groups as sheaf cohomology groups of algebraic $K$-theory sheaves.  It is a seminal early application of the coniveau filtration, underlying the $K$-theoretic approach to the theory of Chow groups.
\item The {\it nonconnective algebraic $K$-theory} of Bass \cite{BassKTheory68} and Thomason \cite{Thomason-Trobaugh90}.  Here, {\it nonconnective} refers to the properties of the {\it topological spectrum} $\mbf{K}$ defining this version of $K$-theory; in particular, it admits homotopy groups in negative degrees.  While Bloch's formula \hyperref[blochstheoremintro]{\ref{blochstheoremintro}} itself holds for the older version of $K$-theory due to Quillen, and even for Milnor's symbolic $K$-theory up to torsion, an adequate description of the infinitesimal theory requires Bass-Thomason $K$-theory.  Technically, this is because $\mbf{K}$ is an {\it effaceable substratum} in the sense of Colliot-Th\'{e}l\`{e}ne, Hoobler, and Kahn \cite{CHKBloch-Ogus-Gabber97}.  {\it Thomason's localization theorem} (\cite{Thomason-Trobaugh90}, theorem 7.4) captures some of the most important properties of $\mbf{K}$ in this sense. 
\item The relationships between algebraic $K$-theory and cyclic homology, explored by Goodwillie \cite{GoodWillieRelativeK86}, Loday \cite{LodayCyclicHomology98}, Keller \cite{KellerCycHomofDGAlgebras96}, \cite{KellerCyclicHomologyofExactCat96}, \cite{KellerCyclicHomologyofSchemes98}, Weibel \cite{WeibelCycliccdh-CohomNegativeK06}, \cite{WeibelCyclicSchemes91}, \cite{WeibelHodgeCyclic94}, \cite{WeibelInfCohomChernNegCyclic08}, and many others.\footnotemark\footnotetext{There is a large literature on this subject.  The few papers I mention here are some of the references of particular importance to this book.}  Roughly speaking, the {\it negative variant} of cyclic homology may be viewed as a ``linearization" of algebraic $K$-theory, playing approximately the same role as a Lie algebra does to its Lie group.  
\item The {\it Bloch-Ogus theorem} \cite{BlochOgusGersten74} and its generalizations, as explicated by Colliot-Th\'{e}l\`{e}ne, Hoobler, and Kahn \cite{CHKBloch-Ogus-Gabber97}.  Theorems of this type permit algebraic $K$-theory and cyclic homology to be recognized as {\it effaceable cohomology theories with supports.}  This, in turn, implies that the sheafified {\it Cousin complexes} corresponding to the sheaves of algebraic $K$-theory and negative cyclic homology on an appropriate scheme $X$ are {\it flasque resolutions,} which enables computation of their sheaf cohomology groups.  In the case of algebraic $K$-theory, of course, these groups include the Chow groups. 
\item Green and Griffiths' infinitesimal theory of cycle groups and Chow groups \cite{GreenGriffithsTangentSpaces05}, and related results of Van der Kallen \cite{VanderKallenRingswithManyUnits77}, \cite{VanderKallenEarlyTK271}, \cite{VanMaazenStienstra}, Bloch \cite{BlochK2Cycles74},  \cite{BlochTangentSpace72}, \cite{BlochK2Artinian75}, Stienstra \cite{StienstraCartierDieudonne}, \cite{StienstraCartierDieudonneCorrection},  \cite{MaazenStienstra77}, \cite{StienstraFormalCompletion83}, Hesselholt \cite{HesselholtTruncated05}, \cite{HesselholtBigdeRhamWitt}, myself \cite{DribusMilnorK14}, and others.  The main ambition of this book is to place at least some of these results in a rigorous broader context.  
\end{enumerate}

\section{What is New?}\label{sectionwhatisnew}

The main innovation in this book is the coniveau machine, presented as a structural organizing principle for studying the infinitesimal structure of generalized cohomology theories on appropriate categories of topological spaces.   The guiding example, and the only one examined in any detail, is the case of algebraic $K$-theory and negative cyclic homology, viewed as cohomology theories with supports on a distinguished category of algebraic schemes.  The resulting version of the coniveau machine enables analysis of the infinitesimal theory of Chow groups of smooth algebraic varieties via Bloch's formula \hyperref[blochstheoremintro]{\ref{blochstheoremintro}}. 

In more detail, this book contains the following new material:

\begin{enumerate}
\item {\bf A new definition of ``generalized tangent groups at the identity" of the Chow groups of a smooth algebraic variety.}   Here arises an important subtlety: although Bloch's formula \hyperref[blochstheoremintro]{\ref{blochstheoremintro}} is ``relatively insensitive" to the specific version of $K$-theory used, the corresponding tangent objects differ in important ways for different versions of $K$-theory.  Green and Griffiths \cite{GreenGriffithsTangentSpaces05} use Milnor $K$-theory $K^\tn{M}$ to define the tangent groups; in this book, I use Bass-Thomason $K$-theory.  

\hyperref[defigendefgroupChow]{\bf Definition 3.8.3.1(2) }{\it Let $X$ be a smooth algebraic variety over a field $k$, and let $Y$ be a separated $k$-scheme, not necessarily smooth.  Let $\mbf{K}_X$ be the spectrum of algebraic $K$-theory on $X$.  Let $\mbf{K}_{X\times_kY}$ be the ``augmented $K$-theory spectrum of $X$ with respect to $Y$," where $X\times_kY$ is shorthand for the fiber product $X\times_{\tn{Spec }k} Y$.  Define the relative $K$-theory spectrum of $X$ with respect to $Y$ to be the homotopy fiber $\mbf{K}_{X\times_kY,Y}$ of the morphism of spectra $\mbf{K}_{X\times_kY}\rightarrow \mbf{K}_X$.   Define augmented and relative $K$-theory groups of $X$, and augmented and relative $K$-theory sheaves on $X$ with respect to $Y$, in the usual way, via the corresponding spectra.   Under these conditions, the {\bf generalized tangent group at the identity} $T_Y\tn{Ch}_X^p$ of the $p$th Chow group $\tn{Ch}_X^p$ of $X$ with respect to $Y$ is the $p$th Zariski sheaf cohomology group of the relative $K$-theory sheaf $\ms{K}_{p,X\times_kY,Y}$ on $X$:
\[T_Y\tn{Ch}_X^p:=H_{\tn{\fsz{Zar}}}^p\big(X,T\ms{K}_{p,X\times_kY,Y}\big).\]
}

The assignment $X\mapsto X\times_kY$ is called ``multiplication by the fixed separated scheme $Y$."   The corresponding assignments of spectra $X\mapsto \mbf{K}_{X\times_kY,Y}$ and $X\mapsto \mbf{K}_{X\times_kY}$ define ``new cohomology theories with supports" from algebraic $K$ theory, in the sense of Colliot-Th\'{e}l\`{e}ne, Hoobler, and Kahn \cite{CHKBloch-Ogus-Gabber97}. 

The prototypical example of a generalized tangent group is the ``ordinary tangent group" $T\tn{Ch}_X^p$, for which $Y$ is the spectrum of the algebra of dual numbers $\tn{Spec }k[\ee]/\ee^2$. In this case, the tangent sheaf $T\ms{K}_{p,X}$ is just the kernel $\tn{Ker}\big[\ms{K}_{p,X_\ee}\rightarrow\ms{K}_{p,X}\big]$.   Note that even in this simple case, the definition given here differs from that of Green and Griffiths.  For example, comparing the tangent sheaves at the identity of Milnor and Bass-Thomason $K$-theory in the case $p=3$ yields:
\[T\ms{K}_{3,X}^{\tn{M}}\cong\varOmega^2_{X/\QQ}\mbox{\hspace*{.5cm} while \hspace*{.5cm}} T\ms{K}_{3,X}\cong\varOmega^2_{X/\QQ}\oplus\ms{O}_X,\]
and the ``extra factor" $\ms{O}_X$ {\it may} lead to interesting new invariants in cohomology, inaccessible to the approach of Green and Griffiths.

\item {\bf A new, rigorous definition of ``generalized deformation groups" of Chow groups.}  Logically, this definition precedes the definition of the corresponding generalized tangent groups, as reflected in definition \hyperref[defigendefgroupChow]{\ref{defigendefgroupChow}} below.  However, since the generalized tangent groups are the main focus, I reverse the order in this introductory setting.  Generalized deformation groups include groups of ``$n$th-order arcs" on the Chow groups.  Green and Griffiths \cite{GreenGriffithsTangentSpaces05} do not formally define such groups, but resort to manipulating special types of ``arcs" ad hoc.  A rigorous definition is enabled by the ``existence of the second column of the coniveau machine," which involves the procedure of ``multiplying by a fixed separated scheme" \`a la Colliot-Th\'{e}l\`{e}ne, Hoobler, and Kahn \cite{CHKBloch-Ogus-Gabber97}, as in the definition of the generalized tangent groups above.  

\hyperref[defigendefgroupChow]{\bf Definition 3.8.3.1(1) }{\it Let $X$ be a smooth algebraic variety over a field $k$, and let $Y$ be a separated $k$-scheme, not necessarily smooth. The {\bf generalized deformation group of $\tn{Ch}_X^p$ with respect to $Y$} is the group $H_{\tn{Zar}}^p(X,\ms{K}_{p, X\times_kY})$.}

The prototypical example is the group of {\bf first-order arcs} on $\tn{Ch}_X^p$, which is the group $H_{\tn{Zar}}^p(X,\ms{K}_{p, X_\ee})$, where $X_\ee$ is shorthand for the fiber product $X\times_{\tn{Spec }k} \tn{Spec }k[\ee]/\ee^2$.  

\item {\bf Existence of the Coniveau Machine for Bass-Thomason $K$-theory and negative cyclic homology.}  The meaning of this existence result is that certain very general constructions, such as the coniveau spectral sequence for a cohomology theory with supports on a category of pairs over a distinguished category of schemes, when applied to Bass-Thomason $K$-theory and negative cyclic homology, exhibit certain desirable properties, such as exactness of the corresponding sheafified Cousin complexes.  Proving this involves combining several known existence results:

\begin{enumerate}
\item Combining results of Thomason and Bloch-Ogus \cite{BlochOgusGersten74},  \`a la Colliot-Th\'{e}l\`{e}ne, Hoobler, and Kahn \cite{CHKBloch-Ogus-Gabber97}, the Cousin complexes appearing as the rows of the $E_1$-level of the coniveau spectral sequence for Bass-Thomason $K$-theory on a smooth algebraic variety $X$ sheafify to yield flasque resolutions of the $K$-theory sheaves $\ms{K}_{p,X}$ on $X$.  
\newpage
\item  ``Multiplying by a fixed separated scheme," as mentioned above, generates an ``augmented cohomology theory" satisfying the same conditions.  
\item In the case where the augmentation arises from a {\it nilpotent thickening}, a similar result holds for the relative $K$-theory sheaves with respect to the augmentation.  
\item  In this case, the relative algebraic Chern character induces an isomorphism of functors between relative algebraic $K$-theory and relative negative cyclic homology. 
\end{enumerate}

These results lead to the following ``main theorem:" 

\hyperref[theoremconiveaumachineexists]{\bf Theorem 4.4.3.2 } {\it Let $\mbf{S}_k$ be a distinguished category of schemes over a field $k$, satisfying the conditions given at the beginning of section \hyperref[subsectioncohomsupportssubstrata]{\ref{subsectioncohomsupportssubstrata}} above, and let $X$ in $\mbf{S}_k$ be smooth, equidimensional, and noetherian.  Let $X\mapsto X\times_kY$ be a nilpotent thickening of $X$.  Then the coniveau machine for Bass-Thomason $K$-theory on $X$ with respect to $Y$ exists; that is, the Cousin complexes appearing as the rows of the $E_1$-level of the coniveau spectral sequence for augmented and relative $K$-theory and negative cyclic homology on $X$ with respect to $Y$ sheafify to yield flasque resolutions of the corresponding sheaves, and the algebraic Chern character induces an isomorphism of functors between the coniveau spectral sequences of relative $K$-theory and relative negative cyclic homology.}

\item {\bf Identification of the ``generalized tangent groups" of the Chow groups $\tn{Ch}_X^p$ in terms of negative cyclic homology in the case of nilpotent thickenings of $X$.}  This follows from the existence of the coniveau machine for Bass-Thomason $K$-theory and negative cyclic homology, and the existence and properties of the algebraic Chern character.  

\hyperref[corollarytangpnegcyc]{\bf Corollary 4.4.3.4 }{\it Let $X$ be a smooth algebraic variety over a field $k$, and let $Y$ be a separated scheme over $k$.  Then for all $p$, the relative algebraic Chern character induces canonical isomorphisms
\begin{equation}\label{equtangentgroupsarenegcychomgroups}T_Y\tn{Ch}_X^p\cong H_{\tn{Zar}}^p(X,\ms{HN}_{p,X\times_kY,Y})\end{equation}
between the generalized tangent groups at the identity of the Chow groups $\tn{Ch}_X^p$ are the sheaf cohomology groups of relative negative cyclic homology on $X$ with respect to the nilpotent thickening $X\mapsto X\times_kY$.}

\item {\bf ``Almost new:" a Goodwillie-type theorem for Milnor $K$-theory; details of the proof appear in a very recent supporting paper}  \cite{DribusMilnorK14}.  This theorem is a sharp result about the relationship between cyclic homology and differential forms, based on Van der Kallen's notion of stability for a commutative ring:

\hyperref[theoremgoodwilliemilnor]{\bf Theorem 3.6.5.10 }{\it Suppose that $R$ is a split nilpotent extension of a $5$-fold stable ring $S$, with extension ideal $I$, 
whose index of nilpotency is $N$.  Suppose further that every positive integer less than or equal to $N$
 is invertible in $S$.  Then for every positive integer $n$,
\[K_{n+1,R,I} ^{\tn{M}}\cong \frac{\Omega_{R,I} ^n}{d\Omega_{R,I} ^{n-1}}.\]
}
\newpage
Jan Stienstra \cite{StienstraPrivate} has suggested re-interpreting this result in terms of his Cartier-Dieudonn\'e theory for Chow groups \cite{StienstraCartierDieudonne}, \cite{StienstraCartierDieudonneCorrection}.  Wilberd Van der Kallen \cite{VanderKallenprivate14} and Lars Hesselholt \cite{Hesselholtprivate14} have suggested generalizing the theorem in terms of de Rham-Witt cohomology.  

\end{enumerate}

\section{Summary by Chapter and Section}\label{sectionsummary}

Following is a detailed summary of the material and results appearing in chapters \hyperref[chaptercurves]{\ref{chaptercurves}},   \hyperref[ChapterTechnical]{\ref{ChapterTechnical}}, \hyperref[ChapterConiveau]{\ref{ChapterConiveau}}, and appendix  \hyperref[chapterappendix]{\ref{chapterappendix}} of this book.  The present introductory chapter is omitted. 

{\bf Chapter \hyperref[chaptercurves]{\ref{chaptercurves}}} presents a detailed exposition of the special case involving zero-cycles on a smooth projective algebraic curve $X$ over the complex numbers, and the corresponding Chow group $\tn{Ch}_X^1$ of zero-cycles on $X$ modulo rational equivalence.\footnotemark\footnotetext{Recall that {\it zero}-cycles are associated with the {\it first} Chow group is because they are {\it codimension-one} cycles, since a curve is one-dimensional.} In this case, the geometry is well-understood; indeed, many of the most interesting and important geometric aspects of the general case are reduced to trivialities in this context.  This case should be viewed as a ``toy version" of the theory, providing useful illustrations, but failing to represent all the intricacies of the general case.  Green and Griffiths \cite{GreenGriffithsTangentSpaces05} follow a similar procedure, exploring the theory of curves in their Chapter 2.  My exposition is much more detailed, but also much more focused on those aspects of structure directly contributing to the coniveau machine.  In particular, I do not discuss abelian sums, Puiseaux series, the $\ms{E}xt$-interpretation, et cetera.  The principal purpose of this chapter is to provide an avenue of easy access to the theory for graduate students and those unfamiliar with modern algebraic $K$-theory.   The more accomplished may well choose to skip this chapter entirely.  
 
Section \hyperref[SectionPrelimCurves]{\ref{SectionPrelimCurves}} introduces elementary material on smooth projective algebraic curves over the complex numbers, with the primary purpose of fixing notation.   Basic concepts such as the Zariski topology, regular and rational functions, and local rings are presented.

Section \hyperref[ZeroCycles]{\ref{ZeroCycles}} discusses zero-cycles on a curve $X$, and introduces divisors, rational equivalence, and the first Chow group $\tn{Ch}_X^1$.   A certain amount of ``notational overkill" appears in this section, with the purpose of facilitating later generalizations.   The coniveau filtration makes its first appearance here, in connection with the sheaf divisor sequence.  The first Chow group is described in terms of the Picard group $\tn{Pic}_X$ (isomorphic to $\tn{Ch}_X^1$) and the Jacobian variety $\tn{J}_X$ (isomorphic to the component of $\tn{Ch}_X^1$ containing the identity).  In particular, since $\tn{Pic}_X$ is an algebraic group, one may use algebraic Lie theory to determine independently ``what the coniveau machine ought to reveal" about the infinitesimal structure of $\tn{Ch}_X^1$. 

Section \hyperref[SectionNaive]{\ref{SectionNaive}} gives a brief intuitive overview of ``first-order infinitesimal theory" in a generic sense, using motivation from elementary differential geometry and Lie theory.  These analogies provide conceptual foreshadowing for the more complicated constructions to follow.  The theme of linearization is invoked, and the basic structure of the coniveau machine for $\tn{Ch}_X^1$ of a smooth complex projective curve $X$ is outlined, with the details to be filled in by the material in the rest of the chapter.  

Section \hyperref[SectionInfRational]{\ref{SectionInfRational}} commences the detailed exploration of first-order infinitesimal theory of zero-cycles on a smooth complex algebraic curve $X$.  Since zero-cycles on $X$ may be expressed in terms of regular functions and rational functions on $X$, it is natural to begin by studying the infinitesimal theory of these functions.  Rational functions are studied first, due to the simplicity afforded by the fact that different open subsets of $X$ have isomorphic rational function fields.  Formal ``first-order arcs" of nonzero rational functions, expressed in terms of dual numbers, are examined.  Tangent elements, sets, maps, and groups are introduced, and the constructions are sheafified.  

Section \hyperref[SectionInfRegular]{\ref{SectionInfRegular}} repeats these constructions for invertible regular functions on $X$.  The only new feature is that different open subsets of $X$ have essentially different rings of regular functions, despite having isomorphic rational function fields.   Sections \hyperref[SectionInfRational]{\ref{SectionInfRational}} and \hyperref[SectionInfRegular]{\ref{SectionInfRegular}} are not ``intrinsically geometric," in the sense that a zero-cycle may generally be described in many different ways in terms of regular and rational functions.  However, these sections do lay the groundwork for the geometric material in sections \hyperref[SectionInfZeroCycles]{\ref{SectionInfZeroCycles}} and \hyperref[SectionInfChow]{\ref{SectionInfChow}}.  

Section \hyperref[SectionThickened]{\ref{SectionThickened}} takes a brief detour to explain the notion of the ``thickened curve" $X_\ee$ corresponding to an algebraic curve $X$, which possesses the same topological space as $X$, but a different structure sheaf.  The new structure sheaf $\ms{O}_{X_\ee}$ is obtained from the structure sheaf $\ms{O}_X$ of $X$ by introducing a nilpotent element $\ee$.   Geometrically, the thickened curve $X_\ee$ is the fiber product $X\times_\CC\tn{Spec }\CC[\ee]/\ee^2$.  The thickened curve $X_\ee$ is {\it not} an algebraic variety, but an ``augmented variety;" i.e., a special type of singular scheme.  In the context of the coniveau machine, it is better to think of ``augmenting the cohomology theory," rather than ``augmenting the variety," essentially because Chow groups are intrinsically topological.  However, it is useful to keep the latter viewpoint in view, since the cohomology theories of principal interest are ``mediated by ring structure" in the sense that the corresponding functors are defined in terms of rings of functions on a topological space, in this case a scheme.   In this context, rather than ``re-proving" standard results for an augmented cohomology theory, one may sometimes use existing results for the original cohomology theory that hold for certain types of singular schemes. 

Section \hyperref[SectionInfZeroCycles]{\ref{SectionInfZeroCycles}} applies the material of sections \hyperref[SectionInfRational]{\ref{SectionInfRational}} and \hyperref[SectionInfRegular]{\ref{SectionInfRegular}} to the first-order theory of zero-cycles on $X$.  This is the first detailed discussion of infinitesimal {\it geometric} structure in this book.   In particular, the infinitesimal theory of zero-cycles may be understood by taking {\it regular structure to be the trivial part of rational structure} in an appropriate sense.  First-order arcs of zero-cycles are defined, tangent elements, sets, maps, and groups are constructed.   The constructions are sheafified, and the coniveau machine for zero-cycles on $X$, the ``toy version" shown in figure \hyperref[figconiveaucurveintro]{\ref{figconiveaucurveintro}} above, is completed.   In particular, the {\it principal parts sequence} is shown to the be the natural target of the {\it tangent map,} which carries first-order arcs of invertible regular functions, rational functions, or zero-cycles to their ``tangent elements."

Section \hyperref[SectionInfChow]{\ref{SectionInfChow}} applies the coniveau machine to analyze the infinitesimal structure of the first Chow group $\tn{Ch}_X^1$.  The results of this analysis are shown to coincide with the corresponding results from algebraic Lie theory, using the fact that the first Chow group is isomorphic to the Picard group $\tn{Pic}_X$.   The section concludes by identifying ``more sophisticated interpretations" of the objects involved, in anticipation of the general theory to follow.  In particular, the sheaf $\ms{O}_X^*$ of multiplicative groups of invertible regular functions on $X$, appearing as the first nontrivial term in the sheaf divisor sequence, is re-interpreted as the first algebraic $K$-theory sheaf $\ms{K}_{1,X}$ on $X$.  The remaining terms of the sheaf divisor sequence are interpreted in terms of $K$-theory groups with supports $K_{1,X\tn{ on } x}$ and $K_{0,X\tn{ on } x}$.  The terms in the principal parts sequence are interpreted in terms of negative cyclic homology.  The tangent map is interpreted as a rudimentary version of the relative algebraic Chern character.


{\bf Chapter \hyperref[ChapterTechnical]{\ref{ChapterTechnical}}} develops the technical building blocks of the general theory underlying the abstract expression of the coniveau machine given in equation \hyperref[equconiveaumachinefunctorintro]{1.1.1.3} above.  The material ranges from the level of standard coursework to the level of original research.  The object is to enable a graduate student or nonexpert to acquire relatively quickly the background necessary to understand and apply the succeeding constructions.  Included is basic material from algebraic geometry and the theory of algebraic cycles, followed by an overview of topics in algebraic $K$-theory, cyclic homology, and cohomology theories with supports. 

Section \hyperref[SectionBasicAG]{\ref{SectionBasicAG}} presents basic material from algebraic geometry necessary for the remainder of the book.  This section covers similar ground for algebraic varieties as section \hyperref[SectionPrelimCurves]{\ref{SectionPrelimCurves}} does for smooth complex projective curves.  In this book, an {\it algebraic variety} is a {\it noetherian integral separated algebraic scheme of finite type over an algebraically closed field.}   The technical meaning of these terms is explained in this section.  Much of the material discussed in the remainder of the book applies to more general objects such as singular schemes over arbitrary commutative rings, but the principal focus is on smooth algebraic varieties.  The more general category of separated schemes over a field $k$ is also important, since ``multiplying by a fixed separated scheme" \`a la Colliot-Th\'{e}l\`{e}ne, Hoobler, and Kahn \cite{CHKBloch-Ogus-Gabber97}, is the principal method of defining ``augmented cohomology theories" in this book.  Also included in this section is a brief discussion of dimension, codimension, and coniveau filtration.  

Section \hyperref[sectioncyclegroups]{\ref{sectioncyclegroups}} discusses background on the theory of algebraic cycles, cycle groups, and Chow groups of a smooth algebraic variety $X$.   Included is a brief discussion of equivalence relations on cycles and intersection theory, which motivates the definition of the Chow groups.  Cycle groups themselves are generally ``too large and cumbersome" to provide reasonable access to ``the most interesting questions" in this context.  Samuel's notion of {\it adequate equivalence} \cite{SamuelAdequate58} allows cycle groups to be replaced in this role by groups of cycle classes that fit together to form a graded ring under intersection.  {\it Numerical}, {\it homological}, {\it algebraic}, and {\it rational equivalence} are four adequate equivalence relations, with rational equivalence the ``finest" of the four.  The $p$th Chow group $\tn{Ch}_X^p$ of $X$ is defined to be the group of rational equivalence classes of codimension-$p$ cycles on $X$.  The section briefly describes how the codimension-one case; i.e., the case of {\it divisors}, is particularly simple, and outlines a few of the complications arising in higher codimensions, as discovered by Mumford \cite{MumfordZeroCycles69}, Griffiths \cite{GriffithsRationalIntegralsIandII69}, Clemens \cite{ClemensHommodalg83}, and others.  

Section \hyperref[sectiontangentgroupsatidentity]{\ref{sectiontangentgroupsatidentity}} defines and discusses the concept of {\it tangent groups at the identity} of Chow groups, beginning with a motivational detour from Lie theory.  The Lie-theoretic connection is more than just an analogy; for example, Jan Stienstra has carried this line of thought to great heights in his {\it Cartier-Dieudonn\'e theory for Chow groups} \cite{StienstraCartierDieudonne}, \cite{StienstraCartierDieudonneCorrection}.   The tangent group at the identity $T\tn{Ch}_X^p$ of $\tn{Ch}_X^p$ is formally defined to be the Zariski sheaf cohomology group $H_{\tn{Zar}}^p(X,T\ms{K}_{p,X})$, where $T\ms{K}_{p,X}$ is the {\it tangent sheaf at the identity} of the $p$th sheaf of Bass-Thomason algebraic $K$-theory on $X$, which may be identified in this context with kernel $\tn{Ker}[\ms{K}_{p,X_\ee}\rightarrow\ms{K}_{p,X}]$.  The section includes an explanation of why a direct definition of the tangent groups in terms of the Chow functors is inadeqaute, and compares the definition to the corresponding definition of Green and Griffiths, which uses the ``more primitive" Milnor $K$-theory instead of Bass-Thomason $K$-theory.  

Section \hyperref[sectionKtheory]{\ref{sectionKtheory}} discusses background from algebraic $K$-theory necessary for the remainder of the book.  While the subject is far too vast and technically involved to give an adequate general overview, I make an attempt to provide at least a limited quantity of historical and contextual motivation, since it is unreasonable to expect everyone interested in algebraic cycles and Chow groups to possess a deep understanding of $K$-theory.   I mention the origins of $K$-theory in Riemann-Roch-Grothendieck-Atiyah-Singer theory, and discuss how ``lower $K$-theory" was developed and understood by Milnor and his contemporaries before the first ``reasonable version of higher $K$-theory" by Quillen, and subsequent developments by Waldhausen, Bass, Thomason, and others.  I discuss ``symbolic $K$-theory," with particular emphasis on Milnor $K$-theory.  As illustrated by the work of Green and Griffiths, symbolic $K$-theory is very interesting in its own right in the study of Chow groups.  I then briefly describe the more sophisticated modern $K$-theoretic constructions in terms of homotopy theory, and particularly in terms of topological spectra.  In particular, I give a bit more detail on Waldhausen's $K$-theory, which forms the basis of Thomason's approach.  I then give the necessary technical background concerning the $K$-theory of Bass and Thomason, and particularly Thomason's nonconnective $K$-theory of perfect complexes on algebraic schemes. 

Section \hyperref[sectioncyclichomology]{\ref{sectioncyclichomology}} provides background on cyclic homology.  As mentioned above, cyclic homology may be viewed as a linearization, or ``additive version," of $K$-theory.  More precisely, relative negative cyclic homology is the ``relative additive version" of $K$-theory appearing in the fourth column of the ``simplified four-column version" of coniveau machine in figure \hyperref[figsimplifiedfourcolumnKtheoryintro]{\ref{figsimplifiedfourcolumnKtheoryintro}} above.   Alternatively, cyclic homology may be viewed as a generalization of differential forms.  The section begins from the latter viewpoint, with a discussion of the theory of {\it K\"{a}hler differentials.}  In particular, {\it absolute K\"{a}hler differentials} are the ``additive version" of algebraic $K$-theory appearing in the work of Green and Griffiths, who do not explicitly work in terms of cyclic homology.  Stienstra \cite{StienstraPrivate} suggests viewing the graded module of absolute K\"{a}hler differentials as a {\it Dieudonn\'e module} in the context of his Cartier-Dieudonn\'e theory for Chow groups.  Next, I turn to the subject of cyclic homology of an algebra over a commutative ring, discussing several different definitions.  The reason for this level of detail is that actual computations of the infinitesimal structure of Chow groups involve calculating (negative) cyclic homology groups of algebras, and their relative versions.  At the end of the section, I discuss cyclic homology of schemes, as defined by Weibel and Keller.   Particularly useful in theoretical and structural settings is Keller's general machinery of {\it localization pairs.}  

Section \hyperref[subsectionrelativeKcyclic]{\ref{subsectionrelativeKcyclic}} discusses relative $K$-theory and relative cyclic homology, which appear in the third and fourth columns of the ``simplified four-column version" of the coniveau machine, respectively.  Here, I focus mostly on the case of relative groups with respect to {\it split nilpotent extensions} of rings, since this provides the proper algebraic notion of infinitesimal structure.  In the context of analyzing the infinitesimal structure of Chow groups of a smooth algebraic variety over a field $k$, a split nilpotent extension corresponds to ``multiplying by the spectrum of a $k$-algebra generated over $k$ by nilpotent elements;" e.g., a local artinian $k$-algebra with residue field $k$.  I then turn to the discussion of {\it Goodwillie's theorem} and related results.  Goodwillie's theorem states that, under appropriate assumptions, relative algebraic $K$-theory is rationally isomorphic to relative negative cyclic homology.   In some cases, the qualifier ``rationally" can be removed.  I present a {\it Goodwillie-type theorem for Milnor $K$-theory}, proven in a separate supporting paper \cite{DribusMilnorK14}, which establishes an isomorphism between the relative Milnor $K$-theory of a split nilpotent extension of a {\it five-fold stable ring}, in the sense of Van der Kallen, and a corresponding group of absolute K\"{a}hler differentials modulo exact differentials.   Van der Kallen \cite{VanderKallenprivate14} and Hesselholt \cite{Hesselholtprivate14} suggest generalizing this result in terms of {\it de Rham-Witt cohomology.} 

Section \hyperref[sectionChern]{\ref{sectionChern}} introduces and discusses the {\it algebraic Chern character.}  The relative version of the algebraic Chern character provides the natural isomorphism between the third and fourth columns of the ``simplified four-column version" of coniveau machine in figure \hyperref[figsimplifiedfourcolumnKtheoryintro]{\ref{figsimplifiedfourcolumnKtheoryintro}} above.  In this case, the relative algebraic Chern character may be described as a composition of Goodwillie's isomorphism and a particular map between cyclic homology and negative cyclic homology.  I describe the algebraic Chern character explicitly for groups of low degree. 

Section \hyperref[sectioncohomtheoriessupports]{\ref{sectioncohomtheoriessupports}} introduces the notion of {\it cohomology theories with supports,} in the sense of Colliot-Th\'{e}l\`{e}ne, Hoobler, and Kahn \cite{CHKBloch-Ogus-Gabber97}.   I denote an arbitrary such theory by $H$; it consists of a family of contravariant functors from an appropriate category $\mbf{P}$ of pairs $(X,Z)$ of topological spaces, to an abelian category $\mbf{A}$.  Here, $X$ belongs to a distinguished category $\mbf{S}$, and $Z$ is a closed subset of $X$.  In the examples of principal interest, $\mbf{S}$ is a category of algebraic schemes.  The individual cohomology object assigned to the pair $(X,Z)$ by $H$ is denoted by $H_{X\tn{\fsz{ on }} Z}^n$, and is called ``the $n$th cohomology object of $X$ with supports in $Z$."  In particular, the cohomology objects $H_{X\tn{\fsz{ on }} X}^n$ are assigned the abbreviated notation $H_{X}^n$.   An important method of defining a cohomology theory with supports is via a special spectrum called a {\it substratum}; examples include the spectra of algebraic $K$-theory and negative cyclic homology.   These spectra can be modified via the method of ``multiplying by a fixed separated scheme" to define augmented and relative substrata, along with their corresponding cohomology groups and sheaves.  This permits the definition of {\it generalized deformation groups and tangent groups of the Chow groups} $\tn{Ch}_X^p$ of a smooth algebraic variety.  A condition on a cohomology theory with supports $H$, called {\it effaceability}, if satisfied, permits construction of a {\it coniveau spectral sequence} for $H$ with certain desirable properties.  In particular, effaceability implies a version of the Bloch-Ogus theorem \cite{BlochOgusGersten74}, which guarantees that the rows of the coniveau spectral sequence of $H$ sheafify to yield flasque resolutions of the sheaves $\ms{H}_X^n$ associated to the presheaves $U\mapsto H_{U}^n$ on $X$.  Effaceability may also be defined at the substratum level; this enables the proof of {\it universal exactness} discussed in section \hyperref[subsectionuniversalexactness]{\ref{subsectionuniversalexactness}}.


{\bf Chapter \hyperref[ChapterConiveau]{\ref{ChapterConiveau}}} presents the construction of the coniveau machine, focusing on the case of algebraic $K$-theory on algebraic schemes, and particularly smooth algebraic varieties.  Statements of lemmas and theorems are not always as general as possible in this chapter, but enough general context is provided to enable the reader to apply the same ideas to other specific theories.  

Section \hyperref[sectionintroductionconiveau]{\ref{sectionintroductionconiveau}} introduces the coniveau machine at an overall structural level.  The general definition is deliberately vague to allow flexibility.   The machine is described in terms of ``exact sequences of functors and natural transformations" involving ``absolute," ``augmented," and ``relative" versions of a cohomology theory with supports $H$ on a topological space $X$ , as well as corresponding ``additive versions" of $H$.  I focus on the specific context of algebraic $K$-theory, negative cyclic homology, and Chow groups, but also describe what questions need to be answered when applying the same ideas more generally.  The ``simplified four-column version" of the coniveau machine mentioned in section \hyperref[subsectionconiveathisthesis]{\ref{subsectionconiveathisthesis}} is also introduced.  

Section \hyperref[sectionconiveauSS]{\ref{sectionconiveauSS}}  gives an exposition of the construction of the coniveau spectral sequence in a context suitable for application to the study of cohomology theories with supports.  The section begins with a brief discussion of the general topic of filtration by codimension, and follows with the definition of the coniveau spectral sequence, and its expression in terms of the coniveau filtration.  

Section \hyperref[sectioncousinblochogus]{\ref{sectioncousinblochogus}} discusses the Cousin complexes of a cohomology theory $H$ with supports on a scheme $X$, which may be viewed as substructures of the corresponding coniveau spectral sequence.   A crucial result in this context is the Bloch-Ogus theorem \cite{BlochOgusGersten74}, which identifies conditions under which the sheafified Cousin complexes are flasque resolutions of the corresponding sheafified cohomology objects.   This permits study of the Zariski sheaf cohomology groups $H_{\tn{\fsz{Zar}}}^p(X,H_X^p)$, since sheaf cohomology may be computed in terms of flasque resolutions.  The prototypical example is when $H$ is algebraic $K$-theory, and $X$ is a smooth complex projective algebraic variety; in this case, the Cousin resolution of $H_X^n$ is the so-called Bloch-Gersten-Quillen resolution, whose $n$th sheaf cohomology group is the $n$th Chow group $\tn{Ch}_X^n$ by Bloch's theorem.  Of course, so much can also be accomplished in terms of Quillen $K$-theory, but the introduction of nilpotent elements for the purpose of studying infinitesimal geometric structure spoils the picture in this context, essentially because of the breakdown of Quillen's {\it devissage.}  Bass-Thomason $K$-theory enables nilpotent extensions, however, allowing access to the infinitesimal theory.  

Section \hyperref[sectionconiveaumachine]{\ref{sectionconiveaumachine}} presents the construction of the coniveau machine for a generalized cohomology theory with supports on an appropriate topological space, focusing on the ``simplified four-column version" case of Bass-Thomason algebraic $K$-theory for a nilpotent thickening of a smooth algebraic variety over a field $k$.  The construction is based on the fact that the spectrum-valued functor $\mbf{K}$ of Bass-Thomason $K$-theory is an effaceable substratum on a suitable category of schemes.   For smooth schemes, the first column of the machine {\it could} be constructed using Quillen $K$-theory, but this is insufficient for studying infinitesimal structure.  ``Augmenting algebraic $K$-theory" by ``multiplying by a fixed separated scheme," \`a la Colliot-Th\'{e}l\`{e}ne, Hoobler, and Kahn \cite{CHKBloch-Ogus-Gabber97}, enables the construction of the second and third columns of the machine.  The construction of the coniveau machine is completed by incorporating the Chern character to relative negative cyclic homology.  The property of {\it universal exactness} enables the construction to be further leveraged by applying additive functors at the level of Cousin complexes.


\newpage

{\bf Appendix \hyperref[chapterappendix]{\ref{chapterappendix}}} provides supporting material on algebraic $K$-theory and cyclic homology.

Section \hyperref[subsecspectra]{\ref{subsecspectra}} presents basic information on topological spectra to support the material on Bass-Thomason $K$-theory and negative cyclic homology. 

Section \hyperref[sectioncyclichomologymixedkeller]{\ref{sectioncyclichomologymixedkeller}} describes Keller's mixed complex for an algebra, not necessarily unital.  The reason for including this construction is that it underlies Keller's machinery of localization pairs. 

Section \hyperref[chaptersimplicialcyclic]{\ref{chaptersimplicialcyclic}} provides basics on simplicial and cyclic theory.  The theory of simplicial and cyclic categories and modules provides a natural language for much existing material on algebra cohomology theories including Hochschild, cyclic, and Andr\'e-Quillen cohomology, which may be viewed as generalizations of the theory of differential forms. 

Section \hyperref[sectionweibelcychomschemes]{\ref{sectionweibelcychomschemes}} describes Weibel's construction of cyclic homology for schemes.  This construction is important both because it predates Keller's more sophisticated constructions, and because it offers a convenient computational tool in the form of Cartan-Eilenberg hypercohomology.  

Section \hyperref[sectioncartaneilenberg]{\ref{sectioncartaneilenberg}} describes Cartan-Eilenberg hypercohomology, with the object of supporting Weibel's Weibel's construction of cyclic homology for schemes in section \hyperref[sectionweibelcychomschemes]{\ref{sectionweibelcychomschemes}}. 

Section \hyperref[spectralsequences]{\ref{spectralsequences}} provides background on spectral sequences.  The material is standard, but the choice of focus and careful explanation of convention should be very helpful to the reader.   In particular, the exposition is designed to remove the annoyance of worrying about details and conventions regarding spectral sequences.    In particular, this section describes the spectral sequence of an exact couple, which supports the definition of the coniveau spectral sequence in section \hyperref[sectionconiveauSS]{\ref{sectionconiveauSS}}. 

Section \hyperref[subsectionlocalcohomsheaves]{\ref{subsectionlocalcohomsheaves}} provides background on local cohomology and Cousin complexes.


\section{Notation and Syntax}\label{subsectionnotation}

In any lengthy work drawing on material from many different sources, one must make significant decisions about notation and syntax.  A few important conventions in this book are as follows:  

\begin{enumerate}
\item {\it Parentheses are often banished in favor of subscripts.}  For example, I denote the $p$th Chow group of an algebraic variety $X$ by $\tn{Ch}_X^p$, rather than $\tn{Ch}^p(X)$.  Multiple subscripts are sometimes used, separated by parentheses.  For instance, the $p$th sheaf of algebraic $K$-groups on an algebraic variety $X$ is denoted by $\ms{K}_{p,X}$.  
\item {\it Primed notation is often used for comparing pairs of elements, objects, or categories ``on an equal footing."}  For example, $X$ and $X'$ denote two arbitrary schemes belonging to the same category, but $(X,Z)$ denotes a scheme $X$ together with a distinguished closed subset $Z$. 
\item {\it Categories} are usually denoted by a fragment of a descriptive word, written in bold font.  For example, the category of perfect complexes on an algebraic scheme $X$ is denoted by $\mbf{Per}_X$.    
\item {\it Sheaves} are usually denoted by capital mathscript letters $\ms{O},\ms{K},...$, etc., but groups or rings of sections of sheaves are denoted by ordinary Roman lettering.  For example, the algebraic structure sheaf of an algebraic variety $X$ is denoted by $\ms{O}_X$, but the ring of sections of $\ms{O}_X$ over an open subset $U$ of $X$ is denoted by $O_U$, instead of the usual $\ms{O}_X(U)$.  This notation is permissible when working with one topological space at a time, which is generally the case in this book.  The objective of this convention is to make it clear at a glance ``at which ``algebraic level one is working." 
\item {\it Skyscraper sheaves} are denoted by underlining their underlying objects, using subscripts to indicate their points of support.  For example, the sheaf $\ms{R}_X$ of fields of rational functions is canonically isomorphic to the skyscraper sheaf $\underline{R_X}_x$ at the generic point $x$ of $X$.   
\item {\it Objects or morphisms whose notation derives from a proper name or descriptive word} are usually written with non-italicized lettering, rather than the usual italicized math lettering.  For example, the $p$th Chow group of a smooth algebraic variety $X$ is written as $\tn{Ch}_X^p$, not $Ch_X^p$.   The objective of this convention is to make it clear at a glance which symbols in a formula are ``individually functional," and which are merely combined to produce a single functional symbol. 
\item {\it Important technical terms and concepts} are generally written in {\bf bold font} when they are defined in the text, whether or not they appear in a formal definition.  When such a term or concept appears in the text either before or after its bold-font definition, it is usually written in {\it italic font}.  For example, important technical terms and concepts appearing in the summary in section \hyperref[sectionsummary]{\ref{sectionsummary}} are written in italic font rather than bold font.  Italics are also used for non-English words and phrases, such as {\it a priori}, and for ordinary English emphasis. 
\item A box $\oblong$ denotes the end of a proof or example. 
\end{enumerate}

\chapter{Algebraic Curves: An Illustrative Toy Version}\label{chaptercurves}

In this chapter, $X$ is a smooth projective algebraic curve over the field of complex numbers $\CC$.  In modern language, such a curve is an integral scheme of dimension $1$, proper over $\CC$, all of whose local rings are regular.   The category of such curves is equivalent to two other categories: the category of compact Riemann surfaces, and the category of function fields of transcendence degree $1$ over $\CC$, each category with its appropriate class of morphisms.  The first equivalence is a special case of the more general correspondence between analytic and algebraic geometry explicated by Serre.  The second is an immediate consequence of the fact that smooth projective curves are birationally equivalent if and only if they are isomorphic. 

The study of algebraic curves is an ancient subject, already well-developed before the introduction of modern algebraic geometry by Grothendieck and his school.   The modern machinery of algebraic schemes, algebraic $K$-theory, and cyclic homology, which has rapidly become indispensable to the general theory of algebraic cycles on algebraic varieties, may be expressed in older, more concrete, and more familiar terms in the case of algebraic curves.  In this chapter, I follow this ``low road" to the description of curves and zero-cycles.   This approach serves to fix notation in a well-understood setting and to motivate the more abstract constructions to follow, while providing a new structural perspective on the classical theory.   The main purpose of the chapter, however, is to provide an accessible avenue for graduate students and mathematician unfamiliar with modern algebraic $K$-theory to begin understanding the intuition behind the coniveau machine.  More accomplished practitioners of the art may well choose to skip this chapter entirely, at least on the first reading.

\section{Preliminaries for Smooth Projective Curves}\label{SectionPrelimCurves}

\subsection{Smooth Projective Curves; Affine Subsets}\label{subsectionsmoothprojcurves}

Let $X$ be a smooth projective algebraic curve over $\CC$.  I will sometimes shorten this to {\bf smooth projective curve}, {\bf smooth curve}, or simply {\bf curve}, when there is no danger of ambiguity.  Since $X$ is an integral scheme, it is irreducible, and therefore has a unique one-dimensional point, its generic point.   Every other point of $X$ is zero-dimensional.  By choosing a specific embedding of $X$ into complex projective space $\PP^n$ for some positive integer $n$, the set of zero-dimensional points of $X$ may be identified with the common vanishing locus of a set $\{F_1,...,F_m\}$ of homogeneous polynomials in $n+1$ variables $z_0,...,z_{n}$.  {\bf Smoothness} in this context is expressed by the condition that the matrix of partial derivatives $\{\partial F_i/\partial z_j\}$ has rank $n-1$ at every zero-dimensional point of $X$.\footnotemark\footnotetext{This matrix is often called the Jacobian (matrix) of $X$; it should not be confused with the Jacobian (variety) of $X$.}  If an algebraic curve is not smooth, it is called {\bf singular}.  

It is often useful to work with {\bf affine open subsets} of the smooth projective curve $X$.  A convenient collection $\{X_i\}$ of such subsets is given by intersecting $X$ with the open sets $V_i$ of $\PP^n$ on which $z_i\ne 0$.  Each $V_i$ is isomorphic to complex affine $n$-space $\AA^n$, with $n$ coordinates $\{z_j/z_i|j\ne i\}$.  The subsets $X_i$ are {\bf smooth affine algebraic curves}. Their sets of zero-dimensional points may be identified with the common vanishing loci of polynomials, no longer required to be homogeneous, in the $n$ coordinates of $\AA^n$, related in a straightforward way to the homogeneous polynomials defining $X$.   The {\bf coordinate ring} $A_i$ of $X_i$ is the quotient of the ring of polynomials in $n$ variables over $\CC$ by the ideal of polynomials vanishing on $X_i$.  

In figure \hyperref[figthreecurves]{\ref{figthreecurves}} below I have illustrated a smooth projective curve $X$ and two singular projective curves $Y$ and $Z$.  These curves are plane cubic curves, meaning that each may be defined in terms of a single third-degree polynomial in the homogeneous coordinates $z_1, z_2$, and $z_3$ of $\PP^2$.  Only real zero-dimensional points of the affine open subsets $X_3$, $Y_3$, and $Z_3$ are shown.  These are the subsets on which the third coordinate $z_3$ is nonzero.  Here $x$ and $y$ stand for the affine coordinates $z_1/z_3$ and $z_2/z_3$.   

\begin{figure}[H]
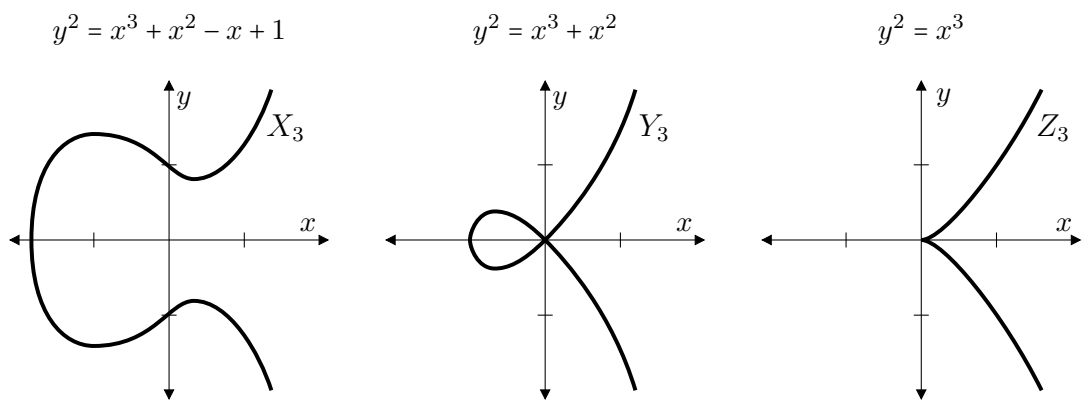

\begin{pgfpicture}{0cm}{0cm}{16cm}{5cm}
\begin{pgftranslate}{\pgfpoint{1cm}{-1cm}}
\begin{pgfscope}
\pgfxyline(2,2.9)(2,3.1)
\pgfxyline(4,2.9)(4,3.1)
\pgfxyline(2.9,2)(3.1,2)
\pgfxyline(2.9,4)(3.1,4)
\pgfsetendarrow{\pgfarrowtriangle{3pt}}
\pgfsetstartarrow{\pgfarrowtriangle{3pt}}
\pgfxyline(3,.9)(3,5.1)
\pgfxyline(.9,3)(5.1,3)
\end{pgfscope}
\begin{pgfscope}
\pgfsetlinewidth{1.5pt}
\pgfmoveto{\pgfxy(4.36,5)}
\pgfcurveto{\pgfxy(4.2,4.5)}{\pgfxy(3.8,3.81)}{\pgfxy(3.33,3.81)}
\pgfcurveto{\pgfxy(3,3.81)}{\pgfxy(2.8,4.41)}{\pgfxy(2,4.41)}
\pgfcurveto{\pgfxy(1.6,4.41)}{\pgfxy(1.17,4)}{\pgfxy(1.17,3)}
\pgfcurveto{\pgfxy(1.17,2)}{\pgfxy(1.6,1.59)}{\pgfxy(2,1.59)}
\pgfcurveto{\pgfxy(2.8,1.59)}{\pgfxy(3,2.19)}{\pgfxy(3.33,2.19)}
\pgfcurveto{\pgfxy(3.8,2.19)}{\pgfxy(4.2,1.5)}{\pgfxy(4.36,1)}
\pgfstroke
\end{pgfscope}
\pgfputat{\pgfxy(4.55,4.5)}{\pgfbox[center,center]{$X_3$}}
\pgfputat{\pgfxy(4.85,3.2)}{\pgfbox[center,center]{\small{$x$}}}
\pgfputat{\pgfxy(3.2,4.85)}{\pgfbox[center,center]{\small{$y$}}}
\pgfputat{\pgfxy(3,5.8)}{\pgfbox[center,center]{\small{$y^2=x^3+x^2-x+1$}}}
\end{pgftranslate}
\begin{pgftranslate}{\pgfpoint{6cm}{-1cm}}
\begin{pgfscope}
\pgfxyline(2,2.9)(2,3.1)
\pgfxyline(4,2.9)(4,3.1)
\pgfxyline(2.9,2)(3.1,2)
\pgfxyline(2.9,4)(3.1,4)
\pgfsetendarrow{\pgfarrowtriangle{3pt}}
\pgfsetstartarrow{\pgfarrowtriangle{3pt}}
\pgfxyline(3,.9)(3,5.1)
\pgfxyline(.9,3)(5.1,3)
\end{pgfscope}
\begin{pgfscope}
\pgfsetlinewidth{1.5pt}
\pgfmoveto{\pgfxy(4.2,5)}
\pgfcurveto{\pgfxy(3.9,4)}{\pgfxy(3.3,3.3)}{\pgfxy(3,3)}
\pgfcurveto{\pgfxy(2.7,2.7)}{\pgfxy(2.5,2.62)}{\pgfxy(2.33,2.62)}
\pgfcurveto{\pgfxy(2.1,2.62)}{\pgfxy(2,2.9)}{\pgfxy(2,3)}
\pgfcurveto{\pgfxy(2,3.1)}{\pgfxy(2.1,3.38)}{\pgfxy(2.33,3.38)}
\pgfcurveto{\pgfxy(2.5,3.38)}{\pgfxy(2.7,3.3)}{\pgfxy(3,3)}
\pgfcurveto{\pgfxy(3.3,2.7)}{\pgfxy(3.9,2)}{\pgfxy(4.2,1)}
\pgfstroke
\end{pgfscope}
\pgfputat{\pgfxy(4.45,4.5)}{\pgfbox[center,center]{$Y_3$}}
\pgfputat{\pgfxy(4.85,3.2)}{\pgfbox[center,center]{\small{$x$}}}
\pgfputat{\pgfxy(3.2,4.85)}{\pgfbox[center,center]{\small{$y$}}}
\pgfputat{\pgfxy(3,5.8)}{\pgfbox[center,center]{\small{$y^2=x^3+x^2$}}}
\end{pgftranslate}
\begin{pgftranslate}{\pgfpoint{11cm}{-1cm}}
\begin{pgfscope}
\pgfxyline(2,2.9)(2,3.1)
\pgfxyline(4,2.9)(4,3.1)
\pgfxyline(2.9,2)(3.1,2)
\pgfxyline(2.9,4)(3.1,4)
\pgfsetendarrow{\pgfarrowtriangle{3pt}}
\pgfsetstartarrow{\pgfarrowtriangle{3pt}}
\pgfxyline(3,.9)(3,5.1)
\pgfxyline(.9,3)(5.1,3)
\end{pgfscope}
\begin{pgfscope}
\pgfsetlinewidth{1.5pt}
\pgfmoveto{\pgfxy(4.6,5)}
\pgfcurveto{\pgfxy(4.1,4)}{\pgfxy(3.3,3)}{\pgfxy(3,3)}
\pgfcurveto{\pgfxy(3.3,3)}{\pgfxy(4.1,2)}{\pgfxy(4.6,1)}
\pgfstroke
\end{pgfscope}
\pgfputat{\pgfxy(4.75,4.5)}{\pgfbox[center,center]{$Z_3$}}
\pgfputat{\pgfxy(4.9,3.2)}{\pgfbox[center,center]{\small{$x$}}}
\pgfputat{\pgfxy(3.3,4.9)}{\pgfbox[center,center]{\small{$y$}}}
\pgfputat{\pgfxy(3,5.8)}{\pgfbox[center,center]{\small{$y^2=x^3$}}}
\end{pgftranslate}
\end{pgfpicture}
\caption{A smooth projective curve and two singular projective curves.}
\label{figthreecurves}
\end{figure}
\vspace*{-.5cm}

The singularity at the origin in $Y_3\subset Y$ is called a node, and the singularity at the origin in $Z_3\subset Z$ is called a cusp.  Although I will discuss only smooth curves in this chapter, singular curves such as $Y$ and $Z$ will appear in various contexts in later chapters.  

The homogeneous polynomial defining $X$ in $\PP^2$ is $z_2^2z_3=z_1^3+z_1^2z_3-z_1z_3^2+z_3^3$.  ``Dehomogenizing" this polynomial with respect to each of the three coordinates $z_1, z_2$, and $z_3$ yields affine curves $X_1, X_2$, and $X_3$, of which $X_3$ is shown above.  In the next diagram, in figure \hyperref[figaltviewscurveX]{\ref{figaltviewscurveX}} below, I have illustrated the other two affine curves $X_2$ and $X_1$.    I have also illustrated the structure of the set of all zero-dimensional points of $X$ as a smooth real surface of genus $1$, suitably deformed to fit into $3$-dimensional real space.  

\begin{figure}[H]
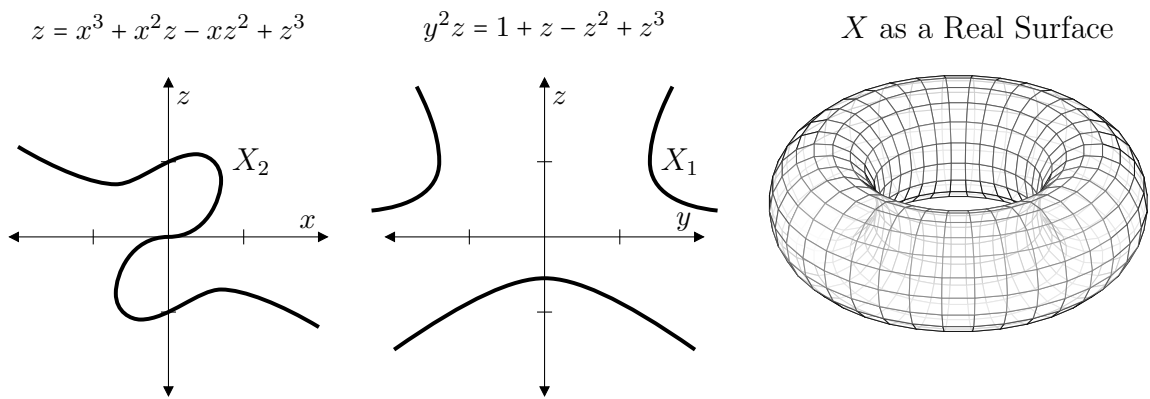

\begin{pgfpicture}{0cm}{0cm}{16cm}{5cm}
\begin{pgftranslate}{\pgfpoint{.5cm}{-1cm}}
\begin{pgfscope}
\pgfxyline(2,2.9)(2,3.1)
\pgfxyline(4,2.9)(4,3.1)
\pgfxyline(2.9,2)(3.1,2)
\pgfxyline(2.9,4)(3.1,4)
\pgfsetendarrow{\pgfarrowtriangle{3pt}}
\pgfsetstartarrow{\pgfarrowtriangle{3pt}}
\pgfxyline(3,.9)(3,5.1)
\pgfxyline(.9,3)(5.1,3)
\end{pgfscope}
\begin{pgfscope}
\pgfsetlinewidth{1.5pt}
\pgfmoveto{\pgfxy(1,4.2)}
\pgfcurveto{\pgfxy(1.5,3.9)}{\pgfxy(2,3.7)}{\pgfxy(2.3,3.7)}
\pgfcurveto{\pgfxy(2.6,3.7)}{\pgfxy(3,4.1)}{\pgfxy(3.37,4.1)}
\pgfcurveto{\pgfxy(3.5,4.1)}{\pgfxy(3.7,4)}{\pgfxy(3.7,3.75)}
\pgfcurveto{\pgfxy(3.7,3.5)}{\pgfxy(3.5,3)}{\pgfxy(3,3)}
\pgfcurveto{\pgfxy(2.5,3)}{\pgfxy(2.3,2.5)}{\pgfxy(2.3,2.25)}
\pgfcurveto{\pgfxy(2.3,2)}{\pgfxy(2.5,1.9)}{\pgfxy(2.63,1.9)}
\pgfcurveto{\pgfxy(3,1.9)}{\pgfxy(3.4,2.3)}{\pgfxy(3.7,2.3)}
\pgfcurveto{\pgfxy(4,2.3)}{\pgfxy(4.5,2.1)}{\pgfxy(5,1.8)}
\pgfstroke
\end{pgfscope}
\pgfputat{\pgfxy(4.1,4)}{\pgfbox[center,center]{$X_2$}}
\pgfputat{\pgfxy(4.85,3.2)}{\pgfbox[center,center]{\small{$x$}}}
\pgfputat{\pgfxy(3.2,4.85)}{\pgfbox[center,center]{\small{$z$}}}
\pgfputat{\pgfxy(3,5.8)}{\pgfbox[center,center]{\small{$z=x^3+x^2z-xz^2+z^3$}}}
\end{pgftranslate}
\begin{pgftranslate}{\pgfpoint{5.5cm}{-1cm}}
\begin{pgfscope}
\pgfxyline(2,2.9)(2,3.1)
\pgfxyline(4,2.9)(4,3.1)
\pgfxyline(2.9,2)(3.1,2)
\pgfxyline(2.9,4)(3.1,4)
\pgfsetendarrow{\pgfarrowtriangle{3pt}}
\pgfsetstartarrow{\pgfarrowtriangle{3pt}}
\pgfxyline(3,.9)(3,5.1)
\pgfxyline(.9,3)(5.1,3)
\end{pgfscope}
\begin{pgfscope}
\pgfsetlinewidth{1.5pt}
\pgfmoveto{\pgfxy(1.3,5)}
\pgfcurveto{\pgfxy(1.4,4.8)}{\pgfxy(1.6,4.4)}{\pgfxy(1.6,4)}
\pgfcurveto{\pgfxy(1.6,3.6)}{\pgfxy(1.2,3.4)}{\pgfxy(.7,3.35)}
\pgfstroke
\pgfmoveto{\pgfxy(4.7,5)}
\pgfcurveto{\pgfxy(4.6,4.8)}{\pgfxy(4.4,4.4)}{\pgfxy(4.4,4)}
\pgfcurveto{\pgfxy(4.4,3.6)}{\pgfxy(4.8,3.4)}{\pgfxy(5.3,3.35)}
\pgfstroke
\pgfmoveto{\pgfxy(1,1.5)}
\pgfcurveto{\pgfxy(2,2.2)}{\pgfxy(2.6,2.45)}{\pgfxy(3,2.45)}
\pgfcurveto{\pgfxy(3.4,2.45)}{\pgfxy(4,2.2)}{\pgfxy(5,1.5)}
\pgfstroke
\end{pgfscope}
\color{white}
\color{black}
\pgfputat{\pgfxy(4.85,3.2)}{\pgfbox[center,center]{\small{$y$}}}
\pgfputat{\pgfxy(3.2,4.85)}{\pgfbox[center,center]{\small{$z$}}}
\pgfputat{\pgfxy(4.8,4)}{\pgfbox[center,center]{$X_1$}}
\pgfputat{\pgfxy(3,5.8)}{\pgfbox[center,center]{\small{$y^2z=1+z-z^2+z^3$}}}
\pgfputat{\pgfxy(8.75,5.8)}{\pgfbox[center,center]{$X$ as a Real Surface}}
\end{pgftranslate}
\end{pgfpicture}
\caption{Alternate views of the smooth projective curve $X$ defined by the equation $z_2^2z_3=z_1^3+z_1^2z_3-z_1z_3^2+z_3^3$.}
\label{figaltviewscurveX}
\end{figure}
\vspace*{-.5cm}

\begin{tikzpicture} [scale=.9, isometricXYZ, line join=round,
        opacity=.8, text opacity=1.0,
        >=latex,
        inner sep=0pt,
        outer sep=2pt,
    ]
    
 \newcommand{\torus}[6]{
	 \foreach \t in {#3} \foreach \f in {#4}
            \draw [color=#5, fill=#6]
	    ({cos(\t)*(#1+#2*cos(\f))},{sin(\t)*(#1+#2*cos(\f))},{#2*sin(\f)})
	-- ({cos(\t+10)*(#1+#2*cos(\f))},{sin(\t+10)*(#1+#2*cos(\f))},{#2*sin(\f)})
	-- ({cos(\t+10)*(#1+#2*cos(\f-20))},{sin(\t+10)*(#1+#2*cos(\f-20))},{#2*sin(\f-20)})
	-- ({cos(\t)*(#1+#2*cos(\f-20))},{sin(\t)*(#1+#2*cos(\f-20))},{#2*sin(\f-20)})
	-- cycle;
	} 
\begin{pgftranslate}{\pgfpoint{14cm}{4.5cm}}
\torus{2}{.8}{220,230,...,300}{0,-20,...,-100}{black!50}{black!0};
\torus{2}{.8}{220,230,...,300}{-120,-140}{black!100}{black!0};
\torus{2}{.8}{220,230,...,300}{-160}{black!70}{black!0};
\torus{2}{.8}{210,200,...,130}{0,-20,...,-100}{black!50}{black!0};
\torus{2}{.8}{210,200,...,130}{-120,-140}{black!100}{black!0};
\torus{2}{.8}{210,200,...,130}{-160}{black!70}{black!0};
\torus{2}{.8}{120,110,...,50}{-160}{black!100}{black!0}; 
\torus{2}{.8}{120,110,...,50}{-140}{black!100}{black!0}; 
\torus{2}{.8}{120,110,...,50}{-120}{black!70}{black!0}; 
\torus{2}{.8}{120,110,...,50}{-100,-80}{black!50}{black!0}; 
\torus{2}{.8}{120,110,...,50}{-60,-40}{black!100}{black!0}; 
\torus{2}{.8}{120,110,...,50}{-20}{black!70}{black!0}; 
\torus{2}{.8}{120,110,...,50}{0}{black!50}{black!0}; 
\torus{2}{.8}{-50,-40,...,40}{-160}{black!100}{black!0};   
\torus{2}{.8}{-50,-40,...,40}{-140}{black!100}{black!0};   
\torus{2}{.8}{-50,-40,...,40}{-120}{black!70}{black!0};   
\torus{2}{.8}{-50,-40,...,40}{-100,-80}{black!50}{black!0};   
\torus{2}{.8}{-50,-40,...,40}{-60,-40}{black!100}{black!0};    
\torus{2}{.8}{-50,-40,...,40}{-20}{black!70}{black!0}; 
\torus{2}{.8}{-50,-40,...,40}{0}{black!50}{black!0};      
\torus{2}{.8}{220,230,...,300}{20,40,60}{black!100}{black!0};
\torus{2}{.8}{220,230,...,300}{80}{black!100}{black!0};
\torus{2}{.8}{220,230,...,300}{100,120,140}{black!70}{black!0};
\torus{2}{.8}{220,230,...,300}{160}{black!60}{black!0};
\torus{2}{.8}{220,230,...,300}{180}{black!70}{black!0};
\torus{2}{.8}{210,200,...,130}{20,40,60}{black!100}{black!0}; 
\torus{2}{.8}{210,200,...,130}{80}{black!100}{black!0}; 
\torus{2}{.8}{210,200,...,130}{100,120,140}{black!70}{black!0}; 
\torus{2}{.8}{210,200,...,130}{160}{black!60}{black!0}; 
\torus{2}{.8}{210,200,...,130}{180}{black!70}{black!0}; 
\torus{2}{.8}{120}{180}{black!100}{black!0}; 
\torus{2}{.8}{120}{160}{black!100}{black!0}; 
\torus{2}{.8}{120}{140}{black!80}{black!0}; 
\torus{2}{.8}{120}{120}{black!70}{black!0}; 
\torus{2}{.8}{120}{100}{black!70}{black!0}; 
\torus{2}{.8}{120}{80}{black!70}{black!0}; 
\torus{2}{.8}{120}{60}{black!70}{black!0}; 
\torus{2}{.8}{120}{40}{black!80}{black!0}; 
\torus{2}{.8}{120}{20}{black!100}{black!0}; 
\torus{2}{.8}{120}{0}{black!100}{black!0}; 
\torus{2}{.8}{110}{180}{black!100}{black!0}; 
\torus{2}{.8}{110}{160}{black!100}{black!0}; 
\torus{2}{.8}{110}{140}{black!70}{black!0}; 
\torus{2}{.8}{110}{120}{black!60}{black!0}; 
\torus{2}{.8}{110}{100}{black!60}{black!0}; 
\torus{2}{.8}{110}{80}{black!60}{black!0}; 
\torus{2}{.8}{110}{60}{black!60}{black!0}; 
\torus{2}{.8}{110}{40}{black!60}{black!0}; 
\torus{2}{.8}{110}{20}{black!70}{black!0}; 
\torus{2}{.8}{110}{0}{black!70}{black!0}; 
\torus{2}{.8}{100}{180}{black!100}{black!0}; 
\torus{2}{.8}{100}{160}{black!100}{black!0}; 
\torus{2}{.8}{100}{140}{black!70}{black!0}; 
\torus{2}{.8}{100}{120}{black!60}{black!0}; 
\torus{2}{.8}{100}{100}{black!60}{black!0}; 
\torus{2}{.8}{100}{80}{black!50}{black!0}; 
\torus{2}{.8}{100}{60}{black!50}{black!0}; 
\torus{2}{.8}{100}{40}{black!60}{black!0}; 
\torus{2}{.8}{100}{20}{black!70}{black!0}; 
\torus{2}{.8}{90}{180}{black!100}{black!0}; 
\torus{2}{.8}{90}{160}{black!100}{black!0}; 
\torus{2}{.8}{90}{140}{black!80}{black!0}; 
\torus{2}{.8}{90}{120}{black!70}{black!0}; 
\torus{2}{.8}{90}{100}{black!60}{black!0}; 
\torus{2}{.8}{90}{80}{black!50}{black!0}; 
\torus{2}{.8}{90}{60}{black!40}{black!0}; 
\torus{2}{.8}{90}{40}{black!40}{black!0}; 
\torus{2}{.8}{90}{20}{black!50}{black!0}; 
\torus{2}{.8}{80}{180}{black!100}{black!0}; 
\torus{2}{.8}{80}{160}{black!100}{black!0}; 
\torus{2}{.8}{80}{140}{black!80}{black!0}; 
\torus{2}{.8}{80}{120}{black!70}{black!0}; 
\torus{2}{.8}{80}{100}{black!60}{black!0}; 
\torus{2}{.8}{80}{80}{black!50}{black!0}; 
\torus{2}{.8}{80}{60}{black!40}{black!0}; 
\torus{2}{.8}{80}{40}{black!40}{black!0}; 
\torus{2}{.8}{80}{20}{black!50}{black!0}; 
\torus{2}{.8}{70,60,...,50}{180}{black!100}{black!0}; 
\torus{2}{.8}{70,60,...,50}{160}{black!100}{black!0}; 
\torus{2}{.8}{70,60,...,50}{140}{black!100}{black!0}; 
\torus{2}{.8}{70,60,...,50}{120}{black!70}{black!0}; 
\torus{2}{.8}{70,60,...,50}{100}{black!60}{black!0}; 
\torus{2}{.8}{70,60,...,50}{80}{black!50}{black!0}; 
\torus{2}{.8}{70,60,...,50}{60}{black!40}{black!0}; 
\torus{2}{.8}{70,60,...,50}{40}{black!40}{black!0}; 
\torus{2}{.8}{70,60,...,50}{20}{black!50}{black!0}; 
\torus{2}{.8}{-50}{180}{black!100}{black!0};
\torus{2}{.8}{-50}{160}{black!100}{black!0};
\torus{2}{.8}{-50}{140}{black!80}{black!0};
\torus{2}{.8}{-50}{120}{black!70}{black!0};
\torus{2}{.8}{-50}{100}{black!60}{black!0};
\torus{2}{.8}{-50}{80}{black!60}{black!0};
\torus{2}{.8}{-50}{60}{black!70}{black!0};
\torus{2}{.8}{-50}{40}{black!80}{black!0};
\torus{2}{.8}{-50}{20}{black!100}{black!0};
\torus{2}{.8}{-50}{0}{black!100}{black!0};
\torus{2}{.8}{-40}{180}{black!100}{black!0};
\torus{2}{.8}{-40}{160}{black!100}{black!0};
\torus{2}{.8}{-40}{140}{black!80}{black!0};
\torus{2}{.8}{-40}{120}{black!70}{black!0};
\torus{2}{.8}{-40}{100}{black!60}{black!0};
\torus{2}{.8}{-40}{80}{black!60}{black!0};
\torus{2}{.8}{-40}{60}{black!50}{black!0};
\torus{2}{.8}{-40}{40}{black!60}{black!0};
\torus{2}{.8}{-40}{20}{black!70}{black!0};
\torus{2}{.8}{-40}{0}{black!80}{black!0};
\torus{2}{.8}{-30}{180}{black!100}{black!0};
\torus{2}{.8}{-30}{160}{black!100}{black!0};
\torus{2}{.8}{-30}{140}{black!70}{black!0};
\torus{2}{.8}{-30}{120}{black!60}{black!0};
\torus{2}{.8}{-30}{100}{black!60}{black!0};
\torus{2}{.8}{-30}{80}{black!50}{black!0};
\torus{2}{.8}{-30}{60}{black!50}{black!0};
\torus{2}{.8}{-30}{40}{black!60}{black!0};
\torus{2}{.8}{-30}{20}{black!70}{black!0};
\torus{2}{.8}{-20}{180}{black!100}{black!0};
\torus{2}{.8}{-20}{160}{black!100}{black!0};
\torus{2}{.8}{-20}{140}{black!80}{black!0};
\torus{2}{.8}{-20}{120}{black!70}{black!0};
\torus{2}{.8}{-20}{100}{black!60}{black!0};
\torus{2}{.8}{-20}{80}{black!50}{black!0};
\torus{2}{.8}{-20}{60}{black!40}{black!0};
\torus{2}{.8}{-20}{40}{black!40}{black!0};
\torus{2}{.8}{-20}{20}{black!50}{black!0};
\torus{2}{.8}{-10}{180}{black!100}{black!0};
\torus{2}{.8}{-10}{160}{black!100}{black!0};
\torus{2}{.8}{-10}{140}{black!80}{black!0};
\torus{2}{.8}{-10}{120}{black!70}{black!0};
\torus{2}{.8}{-10}{100}{black!60}{black!0};
\torus{2}{.8}{-10}{80}{black!50}{black!0};
\torus{2}{.8}{-10}{60}{black!40}{black!0};
\torus{2}{.8}{-10}{40}{black!40}{black!0};
\torus{2}{.8}{-10}{20}{black!50}{black!0};
\torus{2}{.8}{-20,-10,...,40}{180}{black!100}{black!0};
\torus{2}{.8}{-20,-10,...,40}{160}{black!100}{black!0};
\torus{2}{.8}{-20,-10,...,40}{140}{black!100}{black!0};
\torus{2}{.8}{-20,-10,...,40}{120}{black!70}{black!0};
\torus{2}{.8}{-20,-10,...,40}{100}{black!60}{black!0};
\torus{2}{.8}{-20,-10,...,40}{80}{black!50}{black!0};
\torus{2}{.8}{-20,-10,...,40}{60}{black!40}{black!0};
\torus{2}{.8}{-20,-10,...,40}{40}{black!40}{black!0};
\torus{2}{.8}{-20,-10,...,40}{20}{black!50}{black!0};
\end{pgftranslate}
\end{tikzpicture}

\subsection{Zariski Topology; Regular and Rational Functions; Sheaves}\label{subsectionzariski}

The topology of greatest interest on a smooth projective curve $X$ is the {\bf Zariski topology}, which is defined by taking the closed subsets of $X$ to be its  {\bf algebraic subsets}, defined locally as vanishing loci of polynomials.  It is sometimes convenient to isolate the underlying topological space of $X$, which I will denote by $\tn{Zar}_X$ to distinguish it from $X$.  Expressions such as ``points of $X$,"  ``open subsets of $X$," and ``sheaves on $X$"  are common abbreviations for topological data associated with $\tn{Zar}_X$.   Since $X$ is one-dimensional as a complex projective curve, $\tn{Zar}_X$ may be partitioned into two subsets: the set $\tn{Zar}_X^{0}$ of codimension-zero (i.e., dimension-one) points, and the set $\tn{Zar}_X^{1}$ of codimension-one (i.e., dimension-zero) points.   $\tn{Zar}_X^{0}$ is a singleton set consisting of the unique generic point of $\tn{Zar}_X$, while $\tn{Zar}_X^{1}$ may be identified, as mentioned above, with the zero-locus in $\PP^n$ of an appropriate set of homogeneous polynomials.  The reason for using codimension is that it is more natural than dimension in the context of algebraic cycles. 

A {\bf regular function} on an open subset $U$ of $X$ is a function locally equal to a ratio of homogeneous polynomials of equal degrees in $n+1$ variables, with nowhere-vanishing denominator.  The set $O_U$ of regular functions on $U$ is a ring, called the {\bf ring of regular functions} on $U$.  The multiplicative subgroup $O_U^*$ of invertible elements of $O_U$ is called the {\bf group of invertible regular functions} on $U$.  A {\bf rational function} on $U$ is an equivalence class of regular functions, defined on open subsets of $U$, where two regular functions belong to the same equivalence class if and only if they coincide on the intersection of their sets of definition.   The set $R_X$ of rational functions on $X$ is a field called the {\bf rational function field} of $X$.   Since a rational function is determined by its values on any nonempty open subset, the set $R_U$ of rational functions on a nonempty open subset $U$ of $X$ is a field canonically isomorphic to $R_X$.   It is usually harmless to view this canonical isomorphism as equality.  In this sense, the affine open subsets $X_i$ of $X$ all share the same function field $R_X$.  This field may be realized by taking the quotient field of any of the coordinate rings $A_i$ of $X_i$.   The multiplicative subgroup $R_X^*$ of invertible elements of $R_X$ contains all nonzero elements of $R_X$, since $R_X$ is a field.  This group is called the {\bf group of nonzero rational functions} on $X$.   The groups $R_U^*$ of nonzero rational functions on nonempty open subsets $U$ of $X$ are canonically isomorphic to $R_X^*$, and these isomorphisms are usually viewed as equality. 

The {\bf sheaf of regular functions} $\ms{O}_X$ is the sheaf of rings on $X$ whose ring of sections over an open set $U$ of $X$ is $O_U$.  This sheaf is also called the  {\bf algebraic structure sheaf} of $X$.   The {\bf sheaf of invertible regular functions} $\ms{O}_X^*$ on $X$ is the sheaf of multiplicative groups whose group of sections over $U$ is the multiplicative subgroup $O_U^*$ of nowhere-vanishing elements of $O_U$.  The {\bf sheaf of rational functions} $\ms{R}_X$ on $X$ is the sheaf of fields on $X$ whose field of sections over $U$ is $R_U=R_X$.  This sheaf is canonically isomorphic to a skyscraper sheaf at the generic point of $X$ with underlying field $R_X$.  The {\bf sheaf of nonzero rational functions} $\ms{R}_X^*$ is the sheaf of multiplicative groups on $X$ whose group of sections over $U$ is the multiplicative subgroup $R_U^*=R_X^*$ of nonzero elements of $R_U=R_X$.  This sheaf is canonically isomorphic to a skyscraper sheaf at the generic point of $X$ whose underlying multiplicative group is $R_X^*$.

\subsection{Local Rings; $X$ as a Locally Ringed Space}\label{subsectionlocalrings}

The {\bf local ring} $O_x$ of the smooth projective curve $X$ at a point $x\in X$ is defined to be the {\bf stalk} of the structure sheaf $\ms{O}_X$ at $x$.  This ring consists of equivalence classes of regular functions over open neighborhoods of $x$, where two regular functions belong to the same equivalence class if and only if they coincide on some such neighborhood.  Since $O_x$ is a local ring, it has a unique maximal ideal, denoted by $m_x$.  The quotient $r_x:=O_x/m_x$ is a field called the {\bf residue field at $x$} of $X$.  The stalk $R_x$ of the sheaf $\ms{R}_X$ of rational functions on $X$ at a point $x\in X$ is canonically isomorphic to the rational function field $R_X$ of $X$, and again there is usually no harm in viewing this isomorphism as equality.  $R_x$ is obviously {\it not} isomorphic to the residue field $r_x$ in most cases; for instance, the residue field $r_x$ at a zero-dimensional point $x$ of a smooth complex projective curve $X$ is just $\CC$, while the field $R_x=R_X$ has transcendence degree one over $\CC$. 

$X$ is a {\bf locally ringed space}, which is a topological space together with a distinguished sheaf of rings, called the {\bf structure sheaf}, whose stalks are local rings.  In the present case, the topological space is $\tn{Zar}_X$, and the structure sheaf is the algebraic structure sheaf $\ms{O}_X$.  Each of the affine open subsets $X_i$ of $X$ is isomorphic as a locally ringed space to the prime spectrum $\tn{Spec }A_i$ of its coordinate ring $A_i$, and is therefore an {\bf affine algebraic scheme}.  $X$ itself is covered by the affine open subsets $X_i$, and is therefore an {\bf algebraic scheme}.

\section{Zero-Cycles on Curves; First Chow Group}\label{ZeroCycles}

\subsection{Zero-Cycles; Groups and Sheaves of Zero-Cycles}\label{subsectionzerocycles}

A {\bf zero-cycle} $z$ on a smooth projective curve $X$ is a finite formal $\ZZ$-linear combination of zero-dimensional points $x_i$ of $X$:
\[z=\sum_i n_ix_i, \hspace*{1cm} n\in\ZZ.\]
The integer $n_i$ associated with each point $x_i$ is called the {\bf multiplicity} of $x_i$ in $z$.  The sum defining $z$ may be regarded as a sum over the entire set $\tn{Zar}_X^{1}$ of codimension-one (i.e., dimension-zero) points of $X$, with the understanding that only a finite number of the multiplicities are nonzero.  The union $\bigcup_i\{x_i\}$ of the points in $z$ with nonzero multiplicities is called the {\bf support} of $z$.  The set $Z_X^1$ is an abelian group, called the {\bf group of zero-cycles} on $X$, with group operation defined by pointwise addition of multiplicities.  The identity in $Z_X^1$ is the {\bf empty cycle,} in which all zero-dimensional points of $X$ are assigned zero multiplicity.   The empty cycle is sometimes denoted by $0$ because the group operation on $Z_X^1$ is additive. 

Given an open subset $U$ of $X$, the group $Z_U^1$ of zero-cycles on $U$ may be defined in a similar way.  The {\bf sheaf of zero-cycles} on $X$ is the sheaf $\ms{Z}_X^1$ of additive groups on $X$ whose group of sections over $U$ is $Z_U^1$.  Its stalk $Z_x^1$ at a codimension-one point $x$ in $X$ is isomorphic to the underlying additive group $\ZZ^+$ of the integers, with elements representing multiplicities at $x$.  Below, in figure \hyperref[figzerocyclesX]{\ref{figzerocyclesX}}, I have illustrated three zero-cycles on a smooth projective curve $X$.  The zero-cycle on the right is the sum of the other two.  I have indicated points with positive multiplicities by filled nodes and points with negative multiplicities by empty nodes.  Observe how the multiplicities at the point indicated by the arrows cancel additively in the sum.  

\vspace*{4cm}
\begin{figure}[H]
\begin{tikzpicture} [scale=.9, isometricXYZ, line join=round,
        opacity=1, text opacity=1.0,
        >=latex,
        inner sep=0pt,
        outer sep=2pt,
    ]
 \newcommand{\torus}[6]{
	 \foreach \t in {#3} \foreach \f in {#4}
            \draw [color=#5, fill=#6]
	    ({cos(\t)*(#1+#2*cos(\f))},{sin(\t)*(#1+#2*cos(\f))},{#2*sin(\f)})
	-- ({cos(\t+10)*(#1+#2*cos(\f))},{sin(\t+10)*(#1+#2*cos(\f))},{#2*sin(\f)})
	-- ({cos(\t+10)*(#1+#2*cos(\f-20))},{sin(\t+10)*(#1+#2*cos(\f-20))},{#2*sin(\f-20)})
	-- ({cos(\t)*(#1+#2*cos(\f-20))},{sin(\t)*(#1+#2*cos(\f-20))},{#2*sin(\f-20)})
	-- cycle;
	} 	
\begin{pgftranslate}{\pgfpoint{3cm}{2cm}}
\torus{2}{.8}{220,230,...,300}{0,-20,...,-100}{black!50}{black!0};
\torus{2}{.8}{220,230,...,300}{-120,-140}{black!100}{black!0};
\torus{2}{.8}{220,230,...,300}{-160}{black!70}{black!0};
\torus{2}{.8}{210,200,...,130}{0,-20,...,-100}{black!50}{black!0};
\torus{2}{.8}{210,200,...,130}{-120,-140}{black!100}{black!0};
\torus{2}{.8}{210,200,...,130}{-160}{black!70}{black!0};
\torus{2}{.8}{120,110,...,50}{-160}{black!100}{black!0}; 
\torus{2}{.8}{120,110,...,50}{-140}{black!100}{black!0}; 
\torus{2}{.8}{120,110,...,50}{-120}{black!70}{black!0}; 
\torus{2}{.8}{120,110,...,50}{-100,-80}{black!50}{black!0}; 
\torus{2}{.8}{120,110,...,50}{-60,-40}{black!100}{black!0}; 
\torus{2}{.8}{120,110,...,50}{-20}{black!70}{black!0}; 
\torus{2}{.8}{120,110,...,50}{0}{black!50}{black!0}; 
\torus{2}{.8}{-50,-40,...,40}{-160}{black!100}{black!0};   
\torus{2}{.8}{-50,-40,...,40}{-140}{black!100}{black!0};   
\torus{2}{.8}{-50,-40,...,40}{-120}{black!70}{black!0};   
\torus{2}{.8}{-50,-40,...,40}{-100,-80}{black!50}{black!0};   
\torus{2}{.8}{-50,-40,...,40}{-60,-40}{black!100}{black!0};    
\torus{2}{.8}{-50,-40,...,40}{-20}{black!70}{black!0}; 
\torus{2}{.8}{-50,-40,...,40}{0}{black!50}{black!0};      
\torus{2}{.8}{220,230,...,300}{20,40,60}{black!100}{black!0};
\torus{2}{.8}{220,230,...,300}{80}{black!80}{black!0};
\torus{2}{.8}{220,230,...,300}{100}{black!60}{black!0};
\torus{2}{.8}{220,230,...,300}{120}{black!40}{black!0};
\torus{2}{.8}{220,230,...,300}{140}{black!20}{black!0};
\torus{2}{.8}{220,230,...,300}{160}{black!20}{black!0};
\torus{2}{.8}{220,230,...,300}{180}{black!30}{black!0};
\torus{2}{.8}{210,200,...,130}{20,40,60}{black!100}{black!0}; 
\torus{2}{.8}{210,200,...,130}{80}{black!80}{black!0}; 
\torus{2}{.8}{210,200,...,130}{100}{black!60}{black!0}; 
\torus{2}{.8}{210,200,...,130}{120}{black!40}{black!0}; 
\torus{2}{.8}{210,200,...,130}{140}{black!20}{black!0}; 
\torus{2}{.8}{210,200,...,130}{160}{black!20}{black!0}; 
\torus{2}{.8}{210,200,...,130}{180}{black!30}{black!0}; 
\torus{2}{.8}{120}{180}{black!100}{black!0}; 
\torus{2}{.8}{120}{160}{black!50}{black!0}; 
\torus{2}{.8}{120}{140}{black!30}{black!0}; 
\torus{2}{.8}{120}{120}{black!40}{black!0}; 
\torus{2}{.8}{120}{100}{black!40}{black!0}; 
\torus{2}{.8}{120}{80}{black!40}{black!0}; 
\torus{2}{.8}{120}{60}{black!40}{black!0}; 
\torus{2}{.8}{120}{40}{black!60}{black!0}; 
\torus{2}{.8}{120}{20}{black!80}{black!0}; 
\torus{2}{.8}{120}{0}{black!100}{black!0}; 
\torus{2}{.8}{110}{180}{black!100}{black!0}; 
\torus{2}{.8}{110}{160}{black!50}{black!0}; 
\torus{2}{.8}{110}{140}{black!30}{black!0}; 
\torus{2}{.8}{110}{120}{black!40}{black!0}; 
\torus{2}{.8}{110}{100}{black!40}{black!0}; 
\torus{2}{.8}{110}{80}{black!40}{black!0}; 
\torus{2}{.8}{110}{60}{black!50}{black!0}; 
\torus{2}{.8}{110}{40}{black!60}{black!0}; 
\torus{2}{.8}{110}{20}{black!60}{black!0}; 
\torus{2}{.8}{110}{0}{black!70}{black!0}; 
\torus{2}{.8}{100}{180}{black!100}{black!0}; 
\torus{2}{.8}{100}{160}{black!50}{black!0}; 
\torus{2}{.8}{100}{140}{black!40}{black!0}; 
\torus{2}{.8}{100}{120}{black!40}{black!0}; 
\torus{2}{.8}{100}{100}{black!40}{black!0}; 
\torus{2}{.8}{100}{80}{black!40}{black!0}; 
\torus{2}{.8}{100}{60}{black!40}{black!0}; 
\torus{2}{.8}{100}{40}{black!50}{black!0}; 
\torus{2}{.8}{100}{20}{black!60}{black!0}; 
\torus{2}{.8}{90}{180}{black!100}{black!0}; 
\torus{2}{.8}{90}{160}{black!100}{black!0}; 
\torus{2}{.8}{90}{140}{black!70}{black!0}; 
\torus{2}{.8}{90}{120}{black!50}{black!0}; 
\torus{2}{.8}{90}{100}{black!40}{black!0}; 
\torus{2}{.8}{90}{80}{black!30}{black!0}; 
\torus{2}{.8}{90}{60}{black!20}{black!0}; 
\torus{2}{.8}{90}{40}{black!30}{black!0}; 
\torus{2}{.8}{90}{20}{black!40}{black!0}; 
\torus{2}{.8}{80}{180}{black!100}{black!0}; 
\torus{2}{.8}{80}{160}{black!100}{black!0}; 
\torus{2}{.8}{80}{140}{black!70}{black!0}; 
\torus{2}{.8}{80}{120}{black!50}{black!0}; 
\torus{2}{.8}{80}{100}{black!40}{black!0}; 
\torus{2}{.8}{80}{80}{black!30}{black!0}; 
\torus{2}{.8}{80}{60}{black!20}{black!0}; 
\torus{2}{.8}{80}{40}{black!30}{black!0}; 
\torus{2}{.8}{80}{20}{black!40}{black!0}; 
\torus{2}{.8}{70,60,...,50}{180}{black!100}{black!0}; 
\torus{2}{.8}{70,60,...,50}{160}{black!100}{black!0}; 
\torus{2}{.8}{70,60,...,50}{140}{black!100}{black!0}; 
\torus{2}{.8}{70,60,...,50}{120}{black!80}{black!0}; 
\torus{2}{.8}{70,60,...,50}{100}{black!50}{black!0}; 
\torus{2}{.8}{70,60,...,50}{80}{black!30}{black!0}; 
\torus{2}{.8}{70,60,...,50}{60}{black!20}{black!0}; 
\torus{2}{.8}{70,60,...,50}{40}{black!30}{black!0}; 
\torus{2}{.8}{70,60,...,50}{20}{black!40}{black!0}; 
\torus{2}{.8}{-50}{180}{black!100}{black!0};
\torus{2}{.8}{-50}{160}{black!50}{black!0};
\torus{2}{.8}{-50}{140}{black!50}{black!0};
\torus{2}{.8}{-50}{120}{black!50}{black!0};
\torus{2}{.8}{-50}{100}{black!50}{black!0};
\torus{2}{.8}{-50}{80}{black!50}{black!0};
\torus{2}{.8}{-50}{60}{black!50}{black!0};
\torus{2}{.8}{-50}{40}{black!60}{black!0};
\torus{2}{.8}{-50}{20}{black!80}{black!0};
\torus{2}{.8}{-50}{0}{black!100}{black!0};
\torus{2}{.8}{-40}{180}{black!100}{black!0};
\torus{2}{.8}{-40}{160}{black!50}{black!0};
\torus{2}{.8}{-40}{140}{black!50}{black!0};
\torus{2}{.8}{-40}{120}{black!50}{black!0};
\torus{2}{.8}{-40}{100}{black!50}{black!0};
\torus{2}{.8}{-40}{80}{black!50}{black!0};
\torus{2}{.8}{-40}{60}{black!50}{black!0};
\torus{2}{.8}{-40}{40}{black!60}{black!0};
\torus{2}{.8}{-40}{20}{black!70}{black!0};
\torus{2}{.8}{-40}{0}{black!80}{black!0};
\torus{2}{.8}{-30}{180}{black!100}{black!0};
\torus{2}{.8}{-30}{160}{black!100}{black!0};
\torus{2}{.8}{-30}{140}{black!70}{black!0};
\torus{2}{.8}{-30}{120}{black!60}{black!0};
\torus{2}{.8}{-30}{100}{black!50}{black!0};
\torus{2}{.8}{-30}{80}{black!40}{black!0};
\torus{2}{.8}{-30}{60}{black!30}{black!0};
\torus{2}{.8}{-30}{40}{black!40}{black!0};
\torus{2}{.8}{-30}{20}{black!50}{black!0};
\torus{2}{.8}{-20}{180}{black!100}{black!0};
\torus{2}{.8}{-20}{160}{black!100}{black!0};
\torus{2}{.8}{-20}{140}{black!70}{black!0};
\torus{2}{.8}{-20}{120}{black!60}{black!0};
\torus{2}{.8}{-20}{100}{black!50}{black!0};
\torus{2}{.8}{-20}{80}{black!40}{black!0};
\torus{2}{.8}{-20}{60}{black!30}{black!0};
\torus{2}{.8}{-20}{40}{black!40}{black!0};
\torus{2}{.8}{-20}{20}{black!50}{black!0};
\torus{2}{.8}{-10}{180}{black!100}{black!0};
\torus{2}{.8}{-10}{160}{black!100}{black!0};
\torus{2}{.8}{-10}{140}{black!70}{black!0};
\torus{2}{.8}{-10}{120}{black!50}{black!0};
\torus{2}{.8}{-10}{100}{black!40}{black!0};
\torus{2}{.8}{-10}{80}{black!30}{black!0};
\torus{2}{.8}{-10}{60}{black!20}{black!0};
\torus{2}{.8}{-10}{40}{black!30}{black!0};
\torus{2}{.8}{-10}{20}{black!40}{black!0};
\torus{2}{.8}{-20,-10,...,40}{180}{black!100}{black!0};
\torus{2}{.8}{-20,-10,...,40}{160}{black!100}{black!0};
\torus{2}{.8}{-20,-10,...,40}{140}{black!100}{black!0};
\torus{2}{.8}{-20,-10,...,40}{120}{black!80}{black!0};
\torus{2}{.8}{-20,-10,...,40}{100}{black!50}{black!0};
\torus{2}{.8}{-20,-10,...,40}{80}{black!40}{black!0};
\torus{2}{.8}{-20,-10,...,40}{60}{black!20}{black!0};
\torus{2}{.8}{-20,-10,...,40}{40}{black!30}{black!0};
\torus{2}{.8}{-20,-10,...,40}{20}{black!40}{black!0};
\draw (0,0,-.65) node {\large{$X$}};
\node at (3,1,1) [circle,opacity=1,draw=black,fill=white,minimum size=6pt] {};
\draw (3,1.25,.8) node {$\mathbf{-1}$};
\node at (2,2,0) [circle,opacity=1,draw=black,fill=black,minimum size=6pt] {};
\draw (2,2.7,.4) node {$\mathbf{+2}$};
\node at (1,2,2.5) [circle,opacity=1,draw=black,fill=white,minimum size=6pt] {};
\draw (1,2.25,2.35) node {$\mathbf{-1}$};
\pgfsetendarrow{\pgfarrowtriangle{4pt}}
\draw [line width=2, opacity=1] (3,-.5,1) -- (3,.5,1);
\end{pgftranslate}

\begin{pgftranslate}{\pgfpoint{8.75cm}{2cm}}
\torus{2}{.8}{220,230,...,300}{0,-20,...,-100}{black!50}{black!0};
\torus{2}{.8}{220,230,...,300}{-120,-140}{black!100}{black!0};
\torus{2}{.8}{220,230,...,300}{-160}{black!70}{black!0};
\torus{2}{.8}{210,200,...,130}{0,-20,...,-100}{black!50}{black!0};
\torus{2}{.8}{210,200,...,130}{-120,-140}{black!100}{black!0};
\torus{2}{.8}{210,200,...,130}{-160}{black!70}{black!0};
\torus{2}{.8}{120,110,...,50}{-160}{black!100}{black!0}; 
\torus{2}{.8}{120,110,...,50}{-140}{black!100}{black!0}; 
\torus{2}{.8}{120,110,...,50}{-120}{black!70}{black!0}; 
\torus{2}{.8}{120,110,...,50}{-100,-80}{black!50}{black!0}; 
\torus{2}{.8}{120,110,...,50}{-60,-40}{black!100}{black!0}; 
\torus{2}{.8}{120,110,...,50}{-20}{black!70}{black!0}; 
\torus{2}{.8}{120,110,...,50}{0}{black!50}{black!0}; 
\torus{2}{.8}{-50,-40,...,40}{-160}{black!100}{black!0};   
\torus{2}{.8}{-50,-40,...,40}{-140}{black!100}{black!0};   
\torus{2}{.8}{-50,-40,...,40}{-120}{black!70}{black!0};   
\torus{2}{.8}{-50,-40,...,40}{-100,-80}{black!50}{black!0};   
\torus{2}{.8}{-50,-40,...,40}{-60,-40}{black!100}{black!0};    
\torus{2}{.8}{-50,-40,...,40}{-20}{black!70}{black!0}; 
\torus{2}{.8}{-50,-40,...,40}{0}{black!50}{black!0};      
\torus{2}{.8}{220,230,...,300}{20,40,60}{black!100}{black!0};
\torus{2}{.8}{220,230,...,300}{80}{black!80}{black!0};
\torus{2}{.8}{220,230,...,300}{100}{black!60}{black!0};
\torus{2}{.8}{220,230,...,300}{120}{black!40}{black!0};
\torus{2}{.8}{220,230,...,300}{140}{black!20}{black!0};
\torus{2}{.8}{220,230,...,300}{160}{black!20}{black!0};
\torus{2}{.8}{220,230,...,300}{180}{black!30}{black!0};
\torus{2}{.8}{210,200,...,130}{20,40,60}{black!100}{black!0}; 
\torus{2}{.8}{210,200,...,130}{80}{black!80}{black!0}; 
\torus{2}{.8}{210,200,...,130}{100}{black!60}{black!0}; 
\torus{2}{.8}{210,200,...,130}{120}{black!40}{black!0}; 
\torus{2}{.8}{210,200,...,130}{140}{black!20}{black!0}; 
\torus{2}{.8}{210,200,...,130}{160}{black!20}{black!0}; 
\torus{2}{.8}{210,200,...,130}{180}{black!30}{black!0}; 
\torus{2}{.8}{120}{180}{black!100}{black!0}; 
\torus{2}{.8}{120}{160}{black!50}{black!0}; 
\torus{2}{.8}{120}{140}{black!30}{black!0}; 
\torus{2}{.8}{120}{120}{black!40}{black!0}; 
\torus{2}{.8}{120}{100}{black!40}{black!0}; 
\torus{2}{.8}{120}{80}{black!40}{black!0}; 
\torus{2}{.8}{120}{60}{black!40}{black!0}; 
\torus{2}{.8}{120}{40}{black!60}{black!0}; 
\torus{2}{.8}{120}{20}{black!80}{black!0}; 
\torus{2}{.8}{120}{0}{black!100}{black!0}; 
\torus{2}{.8}{110}{180}{black!100}{black!0}; 
\torus{2}{.8}{110}{160}{black!50}{black!0}; 
\torus{2}{.8}{110}{140}{black!30}{black!0}; 
\torus{2}{.8}{110}{120}{black!40}{black!0}; 
\torus{2}{.8}{110}{100}{black!40}{black!0}; 
\torus{2}{.8}{110}{80}{black!40}{black!0}; 
\torus{2}{.8}{110}{60}{black!50}{black!0}; 
\torus{2}{.8}{110}{40}{black!60}{black!0}; 
\torus{2}{.8}{110}{20}{black!60}{black!0}; 
\torus{2}{.8}{110}{0}{black!70}{black!0}; 
\torus{2}{.8}{100}{180}{black!100}{black!0}; 
\torus{2}{.8}{100}{160}{black!50}{black!0}; 
\torus{2}{.8}{100}{140}{black!40}{black!0}; 
\torus{2}{.8}{100}{120}{black!40}{black!0}; 
\torus{2}{.8}{100}{100}{black!40}{black!0}; 
\torus{2}{.8}{100}{80}{black!40}{black!0}; 
\torus{2}{.8}{100}{60}{black!40}{black!0}; 
\torus{2}{.8}{100}{40}{black!50}{black!0}; 
\torus{2}{.8}{100}{20}{black!60}{black!0}; 
\torus{2}{.8}{90}{180}{black!100}{black!0}; 
\torus{2}{.8}{90}{160}{black!100}{black!0}; 
\torus{2}{.8}{90}{140}{black!70}{black!0}; 
\torus{2}{.8}{90}{120}{black!50}{black!0}; 
\torus{2}{.8}{90}{100}{black!40}{black!0}; 
\torus{2}{.8}{90}{80}{black!30}{black!0}; 
\torus{2}{.8}{90}{60}{black!20}{black!0}; 
\torus{2}{.8}{90}{40}{black!30}{black!0}; 
\torus{2}{.8}{90}{20}{black!40}{black!0}; 
\torus{2}{.8}{80}{180}{black!100}{black!0}; 
\torus{2}{.8}{80}{160}{black!100}{black!0}; 
\torus{2}{.8}{80}{140}{black!70}{black!0}; 
\torus{2}{.8}{80}{120}{black!50}{black!0}; 
\torus{2}{.8}{80}{100}{black!40}{black!0}; 
\torus{2}{.8}{80}{80}{black!30}{black!0}; 
\torus{2}{.8}{80}{60}{black!20}{black!0}; 
\torus{2}{.8}{80}{40}{black!30}{black!0}; 
\torus{2}{.8}{80}{20}{black!40}{black!0}; 
\torus{2}{.8}{70,60,...,50}{180}{black!100}{black!0}; 
\torus{2}{.8}{70,60,...,50}{160}{black!100}{black!0}; 
\torus{2}{.8}{70,60,...,50}{140}{black!100}{black!0}; 
\torus{2}{.8}{70,60,...,50}{120}{black!80}{black!0}; 
\torus{2}{.8}{70,60,...,50}{100}{black!50}{black!0}; 
\torus{2}{.8}{70,60,...,50}{80}{black!30}{black!0}; 
\torus{2}{.8}{70,60,...,50}{60}{black!20}{black!0}; 
\torus{2}{.8}{70,60,...,50}{40}{black!30}{black!0}; 
\torus{2}{.8}{70,60,...,50}{20}{black!40}{black!0}; 
\torus{2}{.8}{-50}{180}{black!100}{black!0};
\torus{2}{.8}{-50}{160}{black!50}{black!0};
\torus{2}{.8}{-50}{140}{black!50}{black!0};
\torus{2}{.8}{-50}{120}{black!50}{black!0};
\torus{2}{.8}{-50}{100}{black!50}{black!0};
\torus{2}{.8}{-50}{80}{black!50}{black!0};
\torus{2}{.8}{-50}{60}{black!50}{black!0};
\torus{2}{.8}{-50}{40}{black!60}{black!0};
\torus{2}{.8}{-50}{20}{black!80}{black!0};
\torus{2}{.8}{-50}{0}{black!100}{black!0};
\torus{2}{.8}{-40}{180}{black!100}{black!0};
\torus{2}{.8}{-40}{160}{black!50}{black!0};
\torus{2}{.8}{-40}{140}{black!50}{black!0};
\torus{2}{.8}{-40}{120}{black!50}{black!0};
\torus{2}{.8}{-40}{100}{black!50}{black!0};
\torus{2}{.8}{-40}{80}{black!50}{black!0};
\torus{2}{.8}{-40}{60}{black!50}{black!0};
\torus{2}{.8}{-40}{40}{black!60}{black!0};
\torus{2}{.8}{-40}{20}{black!70}{black!0};
\torus{2}{.8}{-40}{0}{black!80}{black!0};
\torus{2}{.8}{-30}{180}{black!100}{black!0};
\torus{2}{.8}{-30}{160}{black!100}{black!0};
\torus{2}{.8}{-30}{140}{black!70}{black!0};
\torus{2}{.8}{-30}{120}{black!60}{black!0};
\torus{2}{.8}{-30}{100}{black!50}{black!0};
\torus{2}{.8}{-30}{80}{black!40}{black!0};
\torus{2}{.8}{-30}{60}{black!30}{black!0};
\torus{2}{.8}{-30}{40}{black!40}{black!0};
\torus{2}{.8}{-30}{20}{black!50}{black!0};
\torus{2}{.8}{-20}{180}{black!100}{black!0};
\torus{2}{.8}{-20}{160}{black!100}{black!0};
\torus{2}{.8}{-20}{140}{black!70}{black!0};
\torus{2}{.8}{-20}{120}{black!60}{black!0};
\torus{2}{.8}{-20}{100}{black!50}{black!0};
\torus{2}{.8}{-20}{80}{black!40}{black!0};
\torus{2}{.8}{-20}{60}{black!30}{black!0};
\torus{2}{.8}{-20}{40}{black!40}{black!0};
\torus{2}{.8}{-20}{20}{black!50}{black!0};
\torus{2}{.8}{-10}{180}{black!100}{black!0};
\torus{2}{.8}{-10}{160}{black!100}{black!0};
\torus{2}{.8}{-10}{140}{black!70}{black!0};
\torus{2}{.8}{-10}{120}{black!50}{black!0};
\torus{2}{.8}{-10}{100}{black!40}{black!0};
\torus{2}{.8}{-10}{80}{black!30}{black!0};
\torus{2}{.8}{-10}{60}{black!20}{black!0};
\torus{2}{.8}{-10}{40}{black!30}{black!0};
\torus{2}{.8}{-10}{20}{black!40}{black!0};
\torus{2}{.8}{-20,-10,...,40}{180}{black!100}{black!0};
\torus{2}{.8}{-20,-10,...,40}{160}{black!100}{black!0};
\torus{2}{.8}{-20,-10,...,40}{140}{black!100}{black!0};
\torus{2}{.8}{-20,-10,...,40}{120}{black!80}{black!0};
\torus{2}{.8}{-20,-10,...,40}{100}{black!50}{black!0};
\torus{2}{.8}{-20,-10,...,40}{80}{black!40}{black!0};
\torus{2}{.8}{-20,-10,...,40}{60}{black!20}{black!0};
\torus{2}{.8}{-20,-10,...,40}{40}{black!30}{black!0};
\torus{2}{.8}{-20,-10,...,40}{20}{black!40}{black!0};
\node at (3,1,1) [circle,opacity=1,draw=black,fill=black,minimum size=6pt] {};
\draw (3,1.25,.8) node {$\mathbf{+1}$};
\node at (1,2,2.5) [circle,opacity=1,draw=black,fill=white,minimum size=6pt] {};
\draw (1,2.25,2.35) node {$\mathbf{-1}$};
\node at (0,3,1) [circle,opacity=1,draw=black,fill=black,minimum size=6pt] {};
\draw (0,3,.55) node {$\mathbf{+1}$};
\node at (0,-1,1) [circle,opacity=1,draw=black,fill=white,minimum size=6pt] {};
\draw (0,-1,.6) node {$\mathbf{-1}$};
\draw (0,0,-.65) node {\large{$X$}};
\pgfsetendarrow{\pgfarrowtriangle{4pt}}
\draw [line width=2, opacity=1] (3,-.5,1) -- (3,.5,1);
\end{pgftranslate}

\begin{pgftranslate}{\pgfpoint{14.5cm}{2cm}}
\torus{2}{.8}{220,230,...,300}{0,-20,...,-100}{black!50}{black!0};
\torus{2}{.8}{220,230,...,300}{-120,-140}{black!100}{black!0};
\torus{2}{.8}{220,230,...,300}{-160}{black!70}{black!0};
\torus{2}{.8}{210,200,...,130}{0,-20,...,-100}{black!50}{black!0};
\torus{2}{.8}{210,200,...,130}{-120,-140}{black!100}{black!0};
\torus{2}{.8}{210,200,...,130}{-160}{black!70}{black!0};
\torus{2}{.8}{120,110,...,50}{-160}{black!100}{black!0}; 
\torus{2}{.8}{120,110,...,50}{-140}{black!100}{black!0}; 
\torus{2}{.8}{120,110,...,50}{-120}{black!70}{black!0}; 
\torus{2}{.8}{120,110,...,50}{-100,-80}{black!50}{black!0}; 
\torus{2}{.8}{120,110,...,50}{-60,-40}{black!100}{black!0}; 
\torus{2}{.8}{120,110,...,50}{-20}{black!70}{black!0}; 
\torus{2}{.8}{120,110,...,50}{0}{black!50}{black!0}; 
\torus{2}{.8}{-50,-40,...,40}{-160}{black!100}{black!0};   
\torus{2}{.8}{-50,-40,...,40}{-140}{black!100}{black!0};   
\torus{2}{.8}{-50,-40,...,40}{-120}{black!70}{black!0};   
\torus{2}{.8}{-50,-40,...,40}{-100,-80}{black!50}{black!0};   
\torus{2}{.8}{-50,-40,...,40}{-60,-40}{black!100}{black!0};    
\torus{2}{.8}{-50,-40,...,40}{-20}{black!70}{black!0}; 
\torus{2}{.8}{-50,-40,...,40}{0}{black!50}{black!0};      
\torus{2}{.8}{220,230,...,300}{20,40,60}{black!100}{black!0};
\torus{2}{.8}{220,230,...,300}{80}{black!80}{black!0};
\torus{2}{.8}{220,230,...,300}{100}{black!60}{black!0};
\torus{2}{.8}{220,230,...,300}{120}{black!40}{black!0};
\torus{2}{.8}{220,230,...,300}{140}{black!20}{black!0};
\torus{2}{.8}{220,230,...,300}{160}{black!20}{black!0};
\torus{2}{.8}{220,230,...,300}{180}{black!30}{black!0};
\torus{2}{.8}{210,200,...,130}{20,40,60}{black!100}{black!0}; 
\torus{2}{.8}{210,200,...,130}{80}{black!80}{black!0}; 
\torus{2}{.8}{210,200,...,130}{100}{black!60}{black!0}; 
\torus{2}{.8}{210,200,...,130}{120}{black!40}{black!0}; 
\torus{2}{.8}{210,200,...,130}{140}{black!20}{black!0}; 
\torus{2}{.8}{210,200,...,130}{160}{black!20}{black!0}; 
\torus{2}{.8}{210,200,...,130}{180}{black!30}{black!0}; 
\torus{2}{.8}{120}{180}{black!100}{black!0}; 
\torus{2}{.8}{120}{160}{black!50}{black!0}; 
\torus{2}{.8}{120}{140}{black!30}{black!0}; 
\torus{2}{.8}{120}{120}{black!40}{black!0}; 
\torus{2}{.8}{120}{100}{black!40}{black!0}; 
\torus{2}{.8}{120}{80}{black!40}{black!0}; 
\torus{2}{.8}{120}{60}{black!40}{black!0}; 
\torus{2}{.8}{120}{40}{black!60}{black!0}; 
\torus{2}{.8}{120}{20}{black!80}{black!0}; 
\torus{2}{.8}{120}{0}{black!100}{black!0}; 
\torus{2}{.8}{110}{180}{black!100}{black!0}; 
\torus{2}{.8}{110}{160}{black!50}{black!0}; 
\torus{2}{.8}{110}{140}{black!30}{black!0}; 
\torus{2}{.8}{110}{120}{black!40}{black!0}; 
\torus{2}{.8}{110}{100}{black!40}{black!0}; 
\torus{2}{.8}{110}{80}{black!40}{black!0}; 
\torus{2}{.8}{110}{60}{black!50}{black!0}; 
\torus{2}{.8}{110}{40}{black!60}{black!0}; 
\torus{2}{.8}{110}{20}{black!60}{black!0}; 
\torus{2}{.8}{110}{0}{black!70}{black!0}; 
\torus{2}{.8}{100}{180}{black!100}{black!0}; 
\torus{2}{.8}{100}{160}{black!50}{black!0}; 
\torus{2}{.8}{100}{140}{black!40}{black!0}; 
\torus{2}{.8}{100}{120}{black!40}{black!0}; 
\torus{2}{.8}{100}{100}{black!40}{black!0}; 
\torus{2}{.8}{100}{80}{black!40}{black!0}; 
\torus{2}{.8}{100}{60}{black!40}{black!0}; 
\torus{2}{.8}{100}{40}{black!50}{black!0}; 
\torus{2}{.8}{100}{20}{black!60}{black!0}; 
\torus{2}{.8}{90}{180}{black!100}{black!0}; 
\torus{2}{.8}{90}{160}{black!100}{black!0}; 
\torus{2}{.8}{90}{140}{black!70}{black!0}; 
\torus{2}{.8}{90}{120}{black!50}{black!0}; 
\torus{2}{.8}{90}{100}{black!40}{black!0}; 
\torus{2}{.8}{90}{80}{black!30}{black!0}; 
\torus{2}{.8}{90}{60}{black!20}{black!0}; 
\torus{2}{.8}{90}{40}{black!30}{black!0}; 
\torus{2}{.8}{90}{20}{black!40}{black!0}; 
\torus{2}{.8}{80}{180}{black!100}{black!0}; 
\torus{2}{.8}{80}{160}{black!100}{black!0}; 
\torus{2}{.8}{80}{140}{black!70}{black!0}; 
\torus{2}{.8}{80}{120}{black!50}{black!0}; 
\torus{2}{.8}{80}{100}{black!40}{black!0}; 
\torus{2}{.8}{80}{80}{black!30}{black!0}; 
\torus{2}{.8}{80}{60}{black!20}{black!0}; 
\torus{2}{.8}{80}{40}{black!30}{black!0}; 
\torus{2}{.8}{80}{20}{black!40}{black!0}; 
\torus{2}{.8}{70,60,...,50}{180}{black!100}{black!0}; 
\torus{2}{.8}{70,60,...,50}{160}{black!100}{black!0}; 
\torus{2}{.8}{70,60,...,50}{140}{black!100}{black!0}; 
\torus{2}{.8}{70,60,...,50}{120}{black!80}{black!0}; 
\torus{2}{.8}{70,60,...,50}{100}{black!50}{black!0}; 
\torus{2}{.8}{70,60,...,50}{80}{black!30}{black!0}; 
\torus{2}{.8}{70,60,...,50}{60}{black!20}{black!0}; 
\torus{2}{.8}{70,60,...,50}{40}{black!30}{black!0}; 
\torus{2}{.8}{70,60,...,50}{20}{black!40}{black!0}; 
\torus{2}{.8}{-50}{180}{black!100}{black!0};
\torus{2}{.8}{-50}{160}{black!50}{black!0};
\torus{2}{.8}{-50}{140}{black!50}{black!0};
\torus{2}{.8}{-50}{120}{black!50}{black!0};
\torus{2}{.8}{-50}{100}{black!50}{black!0};
\torus{2}{.8}{-50}{80}{black!50}{black!0};
\torus{2}{.8}{-50}{60}{black!50}{black!0};
\torus{2}{.8}{-50}{40}{black!60}{black!0};
\torus{2}{.8}{-50}{20}{black!80}{black!0};
\torus{2}{.8}{-50}{0}{black!100}{black!0};
\torus{2}{.8}{-40}{180}{black!100}{black!0};
\torus{2}{.8}{-40}{160}{black!50}{black!0};
\torus{2}{.8}{-40}{140}{black!50}{black!0};
\torus{2}{.8}{-40}{120}{black!50}{black!0};
\torus{2}{.8}{-40}{100}{black!50}{black!0};
\torus{2}{.8}{-40}{80}{black!50}{black!0};
\torus{2}{.8}{-40}{60}{black!50}{black!0};
\torus{2}{.8}{-40}{40}{black!60}{black!0};
\torus{2}{.8}{-40}{20}{black!70}{black!0};
\torus{2}{.8}{-40}{0}{black!80}{black!0};
\torus{2}{.8}{-30}{180}{black!100}{black!0};
\torus{2}{.8}{-30}{160}{black!100}{black!0};
\torus{2}{.8}{-30}{140}{black!70}{black!0};
\torus{2}{.8}{-30}{120}{black!60}{black!0};
\torus{2}{.8}{-30}{100}{black!50}{black!0};
\torus{2}{.8}{-30}{80}{black!40}{black!0};
\torus{2}{.8}{-30}{60}{black!30}{black!0};
\torus{2}{.8}{-30}{40}{black!40}{black!0};
\torus{2}{.8}{-30}{20}{black!50}{black!0};
\torus{2}{.8}{-20}{180}{black!100}{black!0};
\torus{2}{.8}{-20}{160}{black!100}{black!0};
\torus{2}{.8}{-20}{140}{black!70}{black!0};
\torus{2}{.8}{-20}{120}{black!60}{black!0};
\torus{2}{.8}{-20}{100}{black!50}{black!0};
\torus{2}{.8}{-20}{80}{black!40}{black!0};
\torus{2}{.8}{-20}{60}{black!30}{black!0};
\torus{2}{.8}{-20}{40}{black!40}{black!0};
\torus{2}{.8}{-20}{20}{black!50}{black!0};
\torus{2}{.8}{-10}{180}{black!100}{black!0};
\torus{2}{.8}{-10}{160}{black!100}{black!0};
\torus{2}{.8}{-10}{140}{black!70}{black!0};
\torus{2}{.8}{-10}{120}{black!50}{black!0};
\torus{2}{.8}{-10}{100}{black!40}{black!0};
\torus{2}{.8}{-10}{80}{black!30}{black!0};
\torus{2}{.8}{-10}{60}{black!20}{black!0};
\torus{2}{.8}{-10}{40}{black!30}{black!0};
\torus{2}{.8}{-10}{20}{black!40}{black!0};
\torus{2}{.8}{-20,-10,...,40}{180}{black!100}{black!0};
\torus{2}{.8}{-20,-10,...,40}{160}{black!100}{black!0};
\torus{2}{.8}{-20,-10,...,40}{140}{black!100}{black!0};
\torus{2}{.8}{-20,-10,...,40}{120}{black!80}{black!0};
\torus{2}{.8}{-20,-10,...,40}{100}{black!50}{black!0};
\torus{2}{.8}{-20,-10,...,40}{80}{black!40}{black!0};
\torus{2}{.8}{-20,-10,...,40}{60}{black!20}{black!0};
\torus{2}{.8}{-20,-10,...,40}{40}{black!30}{black!0};
\torus{2}{.8}{-20,-10,...,40}{20}{black!40}{black!0};
\node at (1,2,2.5) [circle,opacity=1,draw=black,fill=white,minimum size=6pt] {};
\draw (1,2.25,2.35) node {$\mathbf{-2}$};
\node at (0,3,1) [circle,opacity=1,draw=black,fill=black,minimum size=6pt] {};
\draw (0,3,.55) node {$\mathbf{+1}$};
\node at (0,-1,1) [circle,opacity=1,draw=black,fill=white,minimum size=6pt] {};
\draw (0,-1,.6) node {$\mathbf{-1}$};
\node at (2,2,0) [circle,opacity=1,draw=black,fill=black,minimum size=6pt] {};
\draw (2,2.7,.4) node {$\mathbf{+2}$};
\draw (0,0,-.65) node {\large{$X$}};
\end{pgftranslate}
\end{tikzpicture}
\caption{zero-cycles on a smooth projective curve $X$.}
\label{figzerocyclesX}
\end{figure}
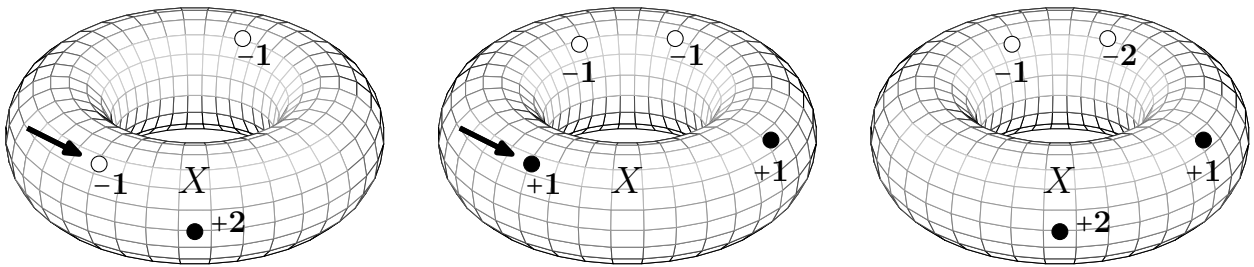

Zero-cycles on algebraic curves are special both because they are zero-dimensional and because they are of codimension-one.   The first property eliminates the possibility of complicated internal structure of individual cycles, and the second property means there is ``insufficient room" for the group of cycles as a whole to be too cumbersome.  Because of these properties, the case of zero-cycles on curves is far better understood than any other case of algebraic cycles.

\subsection{Zero-Cycles as Divisors; Principal Divisors}\label{subsectionzerocyclesdivisors}

\label{ZeroCyclesDivisors}

Codimension-one cycles on an algebraic variety $X$ are called {\bf divisors}.  In the special case where $X$ is a smooth projective curve, divisors are zero-cycles.  As discussed in section \hyperref[sectioncyclegroups]{\ref{sectioncyclegroups}} below, the case of divisors exhibits many simplifications and special relationships not present for higher-codimensional cycles.  A combination of advantages arising from working with the smallest nontrivial dimension and the smallest nontrivial codimension is what renders the case of zero-cycles on curves so simple.  Fortunately, from a suitable viewpoint, this case still presents enough structure to serve as a useful illustration.  

Associated with rational functions are special divisors called principal divisors.   If $f$ is a nonzero rational function on an open subset $U$ of $X$, then the {\bf principal divisor} $\tn{div}_U(f)$ of $f$ is defined in terms of the {\bf valuations} $\nu_x(f)$ of $f$ at the codimension-one points $x$ of $U$ as follows:
\[\tn{div}_X(f)=\sum_{x\in \tn{\footnotesize{Zar}}_U^{1}}\big(\nu_x(f)\big)x.\]
Since $X$ is a smooth complex curve, the valuations $\nu_x$ may be described in terms of complex function theory.  If $x$ is a zero of the rational function $f$, then $\nu_x(f)$ is the order of vanishing of $f$ at $x$.  If $x$ is a pole of $f$, then $\nu_x(f)$ is the order of the pole.  If $x$ is neither a zero nor a pole of $f$, then $\nu_x(f)=0$.  The sum in the definition of $\tn{div}_U(f)$ is finite, rational functions have a finite number of zeros and poles.  The set of principal divisors on $U$ is a subgroup of $Z_U^1$, called the {\bf group of principal divisors} on $U$, and denoted by $\tn{PDiv}_U$.\footnotemark\footnotetext{The sheaf associated to the presheaf $U\mapsto\tn{PDiv}_U$ is just $\ms{Z}_X^1$, since every divisor is locally principal.}   The assignment $f\mapsto \tn{div}_U(f)$ defines a map 
\[\tn{div}_U:R_U^*=R_X^*\rightarrow\tn{PDiv}_U\subset Z_U^1\]
called the {\bf divisor map}.   The divisor map is a group homomorphism, which transforms multiplication of rational functions to addition of principal divisors.  The case $U=X$ gives the {\bf global divisor map} $\tn{div}_X$.  Taking limits over open subsets containing a point $x$ of $X$ gives a {\bf pointwise divisor map} $\tn{div}_x$, which may be identified with the valuation $\nu_x$.   
The kernels of the divisor maps $\tn{div}_X$, $\tn{div}_U$, and $\tn{div}_x$ are the groups $O_X^*$, $O_U^*$,  and $O_x^*$,  respectively, since invertible regular functions have vanishing valuations. 

Not every divisor is principal.  The question of whether or not a particular divisor $z$ on an open subset $U$ of $X$ is principal involves specifying valuations at the points of $U$, given by the multiplicities of $z$, then asking whether there exists a rational function on $U$ with the specified valuations.  When $U=X$, this is an example of what is called a {\bf local-to-global problem}.\footnotemark\footnotetext{Such problems specify local conditions and ask whether a global object exists satisfying these conditions. The answer to this problem for zero-cycles on smooth projective curves is given by Abel's theorem, which involves the image of the Abel-Jacobi map in the Jacobian variety $\tn{J}_X$ of $X$.  I will revisit this important subject in greater depth in Chapter \hyperref[ChapterCyclesChow]{3}.}  Given any point $x$ of $U$, it is easy to define a rational function $f$ with specified valuations in some smaller neighborhood $U'$ of $x$.  Thus, every divisor is {\bf locally principal}.  The difficulty is that $f$ will generally have ``unwanted" zeros and poles outside $U'$.   An arbitrary zero-cycle $z$ on an open subset $U$ of $X$ is defined, in terms of rational functions, by a {\bf family of compatible pairs} $\{U_i,f_i\}$, where $\{U_i\}$ is an open cover of $U$, and $z$ equals $\tn{div}_{U_i}(f_i)$ on $U_i$.  Compatibility means that $f_i$ and $f_j$ share the same valuation at each point of the intersection $U_i\cap U_j$, for all $i$ and $j$, so that the definition of $z$ is consistent.

\subsection{Linear Equivalence; Rational Equivalence; First Chow Group $\tn{Ch}_X^1$}\label{subsectionlinratfirstchow}

Two zero-cycles $z$ and $z'$ on a smooth projective curve $X$ are called {\bf linearly equivalent} if their difference is a principal divisor.  Linear equivalence generalizes to rational equivalence for algebraic cycles of arbitrary codimension on a smooth algebraic variety, so it is useful to describe linear equivalence as rational equivalence in the context of curves.  Two zero-cycles $z$ and $z'$ on $X$ are {\bf rationally equivalent} if there exists a codimension-one cycle $Z$ on $X\times \PP^1$, flat over $C$, and two points $\lambda$ and $\lambda'$ of $\PP^1$, such that 
\[\pi_*\big(Z\cap X\times\{\lambda\}\big)-\pi_*\big(Z\cap X\times\{\lambda'\}\big)=z-z',\]
where $\pi:X\times \PP^1\rightarrow X$ is projection to $X$, and $\pi_*$ is the associated pushforward on cycles.  
Figure \hyperref[figrationalequivalence]{\ref{figrationalequivalence}} below schematically illustrates rational equivalence of two zero-cycles on a smooth projective curve $X$, represented by the horizontal line at the bottom of the diagram.  One zero-cycle is represented by the two black nodes on $X$, and the other is represented by the three white nodes.  The vertical line at the right of the diagram represents $\PP^1$, with two distinguished points $\lambda$ and $\lambda'$ indicated by the dashes. The large square in the middle of the diagram represents the surface $X\times\PP^1$, and the dark curve represents the codimension-one cycle $Z$ on $X\times\PP^1$.  

\begin{figure}[H]
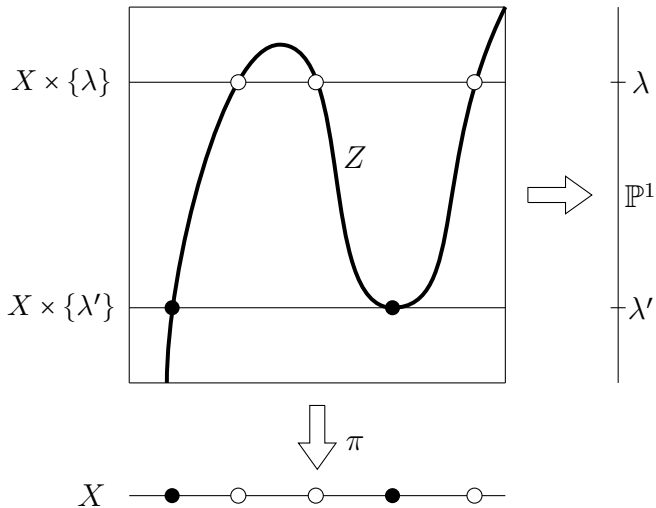

\begin{pgfpicture}{0cm}{0cm}{16cm}{7cm}
\begin{pgftranslate}{\pgfpoint{9cm}{-1cm}}
\begin{pgfrotateby}{\pgfdegree{90}}
\begin{pgfscope}
\pgfxyline(3,1.5)(3,6.5)
\pgfxyline(3,6.5)(8,6.5)
\pgfxyline(8,6.5)(8,1.5)
\pgfxyline(8,1.5)(3,1.5)
\pgfxyline(1.5,1.5)(1.5,6.5)
\pgfxyline(3,0)(8,0)
\pgfxyline(4,1.5)(4,6.5)
\pgfxyline(7,1.5)(7,6.5)
\pgfxyline(4,-.1)(4,.1)
\pgfxyline(7,-.1)(7,.1)
\end{pgfscope}
\begin{pgfscope}
\pgfsetlinewidth{1.5pt}
\pgfmoveto{\pgfxy(3,6)}
\pgfcurveto{\pgfxy(4.5,6)}{\pgfxy(7.5,5.25)}{\pgfxy(7.5,4.5)}
\pgfcurveto{\pgfxy(7.5,3.5)}{\pgfxy(4,4)}{\pgfxy(4,3)}
\pgfcurveto{\pgfxy(4,2)}{\pgfxy(6,2.5)}{\pgfxy(8,1.5)}
\pgfstroke
\end{pgfscope}
\pgfnodecircle{Node0}[fill]{\pgfxy(4,5.93)}{0.1cm}
\pgfnodecircle{Node0}[fill]{\pgfxy(1.5,5.93)}{0.1cm}
\pgfnodecircle{Node0}[fill]{\pgfxy(4,3)}{0.1cm}
\pgfnodecircle{Node0}[fill]{\pgfxy(1.5,3)}{0.1cm}
\pgfmoveto{\pgfxy(5.4,1.2)}
\pgflineto{\pgfxy(5.6,1.2)}
\pgflineto{\pgfxy(5.6,.7)}
\pgflineto{\pgfxy(5.75,.7)}
\pgflineto{\pgfxy(5.5,.35)}
\pgflineto{\pgfxy(5.25,.7)}
\pgflineto{\pgfxy(5.4,.7)}
\pgflineto{\pgfxy(5.4,1.2)}
\pgfstroke
\pgfmoveto{\pgfxy(2.7,3.9)}
\pgflineto{\pgfxy(2.7,4.1)}
\pgflineto{\pgfxy(2.2,4.1)}
\pgflineto{\pgfxy(2.2,4.25)}
\pgflineto{\pgfxy(1.85,4)}
\pgflineto{\pgfxy(2.2,3.75)}
\pgflineto{\pgfxy(2.2,3.9)}
\pgflineto{\pgfxy(2.7,3.9)}
\pgfstroke
\begin{pgfscope}
\color{white}
\pgfnodecircle{Node0}[fill]{\pgfxy(7,5.05)}{0.1cm}
\pgfnodecircle{Node0}[fill]{\pgfxy(1.5,5.05)}{0.1cm}
\pgfnodecircle{Node0}[fill]{\pgfxy(7,4.02)}{0.1cm}
\pgfnodecircle{Node0}[fill]{\pgfxy(1.5,4.02)}{0.1cm}
\pgfnodecircle{Node0}[fill]{\pgfxy(7,1.91)}{0.1cm}
\pgfnodecircle{Node0}[fill]{\pgfxy(1.5,1.91)}{0.1cm}
\end{pgfscope}
\pgfnodecircle{Node0}[stroke]{\pgfxy(7,5.05)}{0.1cm}
\pgfnodecircle{Node0}[stroke]{\pgfxy(1.5,5.05)}{0.1cm}
\pgfnodecircle{Node0}[stroke]{\pgfxy(7,4.02)}{0.1cm}
\pgfnodecircle{Node0}[stroke]{\pgfxy(1.5,4.02)}{0.1cm}
\pgfnodecircle{Node0}[stroke]{\pgfxy(7,1.91)}{0.1cm}
\pgfnodecircle{Node0}[stroke]{\pgfxy(1.5,1.91)}{0.1cm}
\end{pgfrotateby}
\end{pgftranslate}
\pgfputat{\pgfxy(5.5,5)}{\pgfbox[center,center]{$Z$}}
\pgfputat{\pgfxy(2,.5)}{\pgfbox[center,center]{$X$}}
\pgfputat{\pgfxy(5.5,1.2)}{\pgfbox[center,center]{$\pi$}}
\pgfputat{\pgfxy(1.6,3)}{\pgfbox[center,center]{\small{$X\times\{\lambda'\}$}}}
\pgfputat{\pgfxy(1.6,6)}{\pgfbox[center,center]{\small{$X\times\{\lambda\}$}}}
\pgfputat{\pgfxy(9.3,4.5)}{\pgfbox[center,center]{$\PP^1$}}
\pgfputat{\pgfxy(9.3,3)}{\pgfbox[center,center]{$\lambda'$}}
\pgfputat{\pgfxy(9.3,6)}{\pgfbox[center,center]{$\lambda$}}
\end{pgfpicture}
\caption{Rational equivalence for zero-cycles on an algebraic curve.}
\label{figrationalequivalence}
\end{figure}

The {\bf group of zero-cycles rationally equivalent to zero} on $X$ is denoted by $Z_{X,\footnotesize{\tn{rat}}}^1$.  By definition, $Z_{X,\footnotesize{\tn{rat}}}^1$ is equal to the group $\tn{PDiv}_X$ of principal divisors on $X$, although this is only true for curves.  The {\bf first Chow group} $\tn{Ch}_X^1$ of $X$ is the quotient group $Z_X^1/Z_{X,\footnotesize{\tn{rat}}}^1$, consisting of {\bf rational equivalence classes} of zero-cycles on $X$.  


\subsection{Jacobian Variety and the Picard Group}\label{subsectionJacPic}

There are many interesting ways of viewing the first Chow group $\tn{Ch}_X^1$ of a smooth complex projective curve $X$; it may be described in terms of divisors, sheaf cohomology, line bundles, invertible sheaves, the Jacobian variety, the Albanese variety, and so on.  Some of these overlapping viewpoints are valid only in the special case of curves, while others apply to the codimension-one case more generally.  

The following diagram, with exact upper row, illustrates a few of the most important relationships involving $\tn{Ch}_X^1$.  This diagram is adapted from the notes of James D. Lewis \cite{LewisLecturesonCycles00}.  Section \hyperref[subsectionzerocyclessurfaces]{\ref{subsectionzerocyclessurfaces}} below discusses this picture in the general codimension-one case.  

\begin{figure}[H]
\begin{pgfpicture}{0cm}{0cm}{17cm}{3.25cm}
\begin{pgfmagnify}{1}{1}
\begin{pgftranslate}{\pgfpoint{.5cm}{-2.25cm}}
\pgfputat{\pgfxy(1,5)}{\pgfbox[center,center]{$0$}}
\pgfputat{\pgfxy(5.4,5.3)}{\pgfbox[center,center]{$i$}}
\pgfputat{\pgfxy(8.4,5.3)}{\pgfbox[center,center]{$\delta$}}
\pgfputat{\pgfxy(4,5)}{\pgfbox[center,center]{$\mbox{Pic}_X^0$}}
\pgfputat{\pgfxy(7,5)}{\pgfbox[center,center]{$\mbox{Pic}_X$}}
\pgfputat{\pgfxy(10,5)}{\pgfbox[center,center]{$H_{X,\ZZ}^2$}}
\pgfputat{\pgfxy(13,5)}{\pgfbox[center,center]{$H_X^{0,2}$}}
\pgfputat{\pgfxy(4,3)}{\pgfbox[center,center]{$\mbox{J}_X$}}
\pgfputat{\pgfxy(7,3)}{\pgfbox[center,center]{$\mbox{Ch}_{X}^1$}}
\pgfputat{\pgfxy(10,3)}{\pgfbox[center,center]{$H_{X,\CC}^2$}}
\pgfxyline(3.95,4.5)(3.95,3.5)
\pgfxyline(4.05,4.5)(4.05,3.5)
\pgfxyline(6.95,4.5)(6.95,3.5)
\pgfxyline(7.05,4.5)(7.05,3.5)
\pgfsetendarrow{\pgfarrowpointed{3pt}}
\pgfxyline(1.3,5)(3.4,5)
\pgfxyline(4.5,5)(6.4,5)
\pgfxyline(7.6,5)(9.3,5)
\pgfxyline(10.6,5)(12.4,5)
\pgfxyline(10,4.6)(10,3.5)
\end{pgftranslate}
\end{pgfmagnify}
\end{pgfpicture}
\caption{Relationships among various groups related to a smooth complex projective curve.}
\label{figrationallewisrelationships}
\end{figure}

Besides $\mbox{Ch}_{X}^1$, the objects and maps appearing in diagram \hyperref[figrationallewisrelationships]{\ref{figrationallewisrelationships}} are identified as follows:
\begin{enumerate}
\item $J_X$ is the {\bf Jacobian variety} of $X$.  It is defined to be the complex torus
\begin{equation}\label{equcurvejacobian}J_X=\frac{H_{X,\CC}^1}{H_{X}^{1,0}+H_{X,\ZZ}^1}\cong\frac{H_{X}^{0,1}}{H_{X,\ZZ}^1},\end{equation}
where $H_{X,\CC}^1$ is the first complex cohomology group of $X$, $H_{X}^{1,0}$ and $H_{X}^{0,1}$ are the $(1,0)$ and $(0,1)$-parts of the {\it Hodge decomposition} $H_{X,\CC}^1=H_{X}^{1,0}\oplus H_{X}^{0,1}$ of $H_{X,\CC}^1$, and $H_{X,\ZZ}^1$ is the first integral cohomology group of $X$.\footnotemark\footnotetext{A good reference for this material is the first chapter of \cite{GriffithsTranscendentalAG84}.}
\item $\tn{Pic}_X$ is the {\bf Picard group} of $X$, defined to be the group of isomorphism classes of invertible sheaves, or line bundles, on $X$.  Alternatively, $\tn{Pic}_X$ may be defined to be the Zariski sheaf cohomology group $H_{\tn{\fsz{Zar}}}^1(X,\ms{O}_{X}^*)$.  Equation \hyperref[equCH1curvesheafcohom]{\ref{equCH1curvesheafcohom}} at the beginning of section \hyperref[SectionInfChow]{\ref{SectionInfChow}} below shows that the latter definition coincides with the definition $\tn{Ch}_X^1=Z_X^1/Z_{X,\footnotesize{\tn{rat}}}^1$ of the first Chow group $\tn{Ch}_X^1=Z_X^1/Z_{X,\footnotesize{\tn{rat}}}^1$ of $X$ in section \hyperref[subsectionlinratfirstchow]{\ref{subsectionlinratfirstchow}} above.  $\tn{Pic}_X$ is an example of an {\it algebraic group,} so the methods of {\it algebraic Lie theory} apply to it.  In particular, it makes sense to consider the tangent group at the identity of $\tn{Pic}_X$.  
\item $\tn{Pic}_X^0$ is the {\bf Picard variety} of $X$.  It is the connected component of the identity in the Picard group $\tn{Pic}_X$.  It is an example of a {\it group variety.}  Hence, the tangent group at the identity of the Picard group $\tn{Pic}_X$ coincides with the tangent group at the identity of the Picard variety $\tn{Pic}_X$.  
\item The map $i$ is inclusion of $\tn{Pic}_X^0$ into $\tn{Pic}_X$.
\item The map $\delta$ is the {\it first Chern class map.} 
\item The groups $H_{X,\ZZ}^2$ and $H_{X,\CC}^2$ are the second integral and complex cohomology groups of $X$, and the group $H_{X}^{0,2}$ is the $(0,2)$-part of the Hodge decomposition of $H_{X,\CC}^2$.  
\end{enumerate}

In section \hyperref[SectionInfChow]{\ref{SectionInfChow}} below, I show how these relationships may be used to identify the tangent group at the identity $T\tn{Ch}_X^1$ of $\tn{Ch}_X^1$ in the context of algebraic Lie theory.  This provides an way of checking that the coniveau machine gives the ``right answer."

\subsection{Divisor Sequence}\label{subsectiondivisorsequence}

Since every divisor on a smooth projective algebraic variety is locally principal, there exists a surjective map of sheaves $\ms{D}\tn{iv}_X:\ms{R}_X^*\longrightarrow\ms{Z}_X^1$, defined locally by sending a nonzero rational function $f$ on an open set $U$ of $X$ to its local principal divisor $\tn{div}_U(f)$ in $Z_U^1$.  This map is called the {\bf sheaf divisor map}, or simply the {\bf divisor map}, on $X$.  The maps on groups of sections induced by $\tn{div}$ are, by definition, the local divisor maps $\tn{div}_U$, and the corresponding maps on stalks are the pointwise divisor maps $\tn{div}_x=\nu_x$.  Since the kernel of the local divisor map $\tn{div}_U$ is $O_U^*$, the kernel of the sheaf divisor map is the sheaf $\ms{O}_X^*$ of multiplicative groups of invertible regular functions on $X$.  Thus, there exists a short exact sequence of sheaves
\begin{equation}\label{equsheafdivisorsequencecurves}1\rightarrow\ms{O}_X^*\overset{i}{\longrightarrow}\ms{R}_X^*\overset{\ms{D}\tn{\footnotesize{iv}}_X}{\longrightarrow}\ms{Z}_X^1\rightarrow0,\end{equation}
where $i$ is inclusion.  This sequence is called the {\bf divisor sequence} on $X$.  The divisor sequence is a {\bf hybrid exact sequence} in the sense that it mixes ``multiplicative" and ``additive" operations.\footnotemark\footnotetext{This is why the initial object in the sequence is ``$1$," which is shorthand for the trivial multiplicative group, and the terminal object is ``$0$,"  which is shorthand for the trivial additive group.} Taking global sections of the divisor sequence gives the following left-exact sequence of groups of global sections:
\begin{equation}\label{equdivisorglobalsectionscurves}0\rightarrow O_X^*\overset{i}{\longrightarrow}R_X^*\overset{\tn{\footnotesize{div}}_X}{\longrightarrow}Z_X^1,\end{equation}
where $i$ is again inclusion.   This sequence is not right-exact because divisors on $X$ are generally not principal.  Since $Z_{X,\footnotesize{\tn{lin}}}^1$ is the image of the global divisor map $\tn{div}_X$, the first Chow group $\tn{Ch}_X^1$ of $X$ may be expressed in terms of the divisor sequence by first applying the global sections functor to $\ms{Z}_X^1$, then dividing out the image of $\tn{div}_X$:
\begin{equation}\label{equCH1div}\tn{Ch}_X^1=\frac{Z_X^1}{\tn{Im \big(div}_X\big)}.\end{equation}

\subsection{$\ms{Z}_X^1$ and $\ms{R}_X^*$ as Direct Sums of Skyscraper Sheaves}\label{SubSectionRZDirectSum}

The sheaf $\ms{Z}_X^1$ of zero-cycles on $X$ is  isomorphic to the quotient sheaf $\ms{R}^*_{X}/\ms{O}^*_{X}$, but the  isomorphism is nontrivial.  In particular, the operation on the quotient sheaf $\ms{R}^*_{X}/\ms{O}^*_{X}$ is induced by multiplication of rational functions, while the operation on $\ms{Z}_X^1$ is induced by pointwise addition of multiplicities.  The ``pointwise structure" of $\ms{Z}_X^1$ may be emphasized by writing it as a direct sum of skyscraper sheaves, in terms of its stalks $Z_x^1=\ZZ^+$ at codimension-one points of $X$:
\[\ms{Z}_X^1=\bigoplus_{x\in \tn{\footnotesize{Zar}}_X^{1}}\underline{Z_x^1}=\bigoplus_{x\in \tn{\footnotesize{Zar}}_X^{1}}\underline{\ZZ}^+.\]
Each summand denotes the skyscraper sheaf at $x$ with the underlined underlying group.   The sheaf $\ms{R}_X^*$ of multiplicative groups of nonzero rational functions on $X$ may be expressed in a similar way, since this sheaf is canonically isomorphic to a skyscraper sheaf at the generic point of $X$ with underlying group $R_X^*$:
\[\ms{R}_X^*=\bigoplus_{x\in \tn{\footnotesize{Zar}}_X^{0}}\underline{R_X^*}.\]
In this case there is only one summand, since the index set $\tn{Zar}_X^{0}$ is a singleton set consisting of the unique generic point of $\tn{Zar}_X$.  

Using these expressions for $\ms{Z}_X^1$ and $\ms{R}_X^*$, the sheaf divisor map $\ms{D}\tn{iv}_X:\ms{R}_X^*\rightarrow \ms{Z}_X^1$ may also be expressed as a direct sum, in terms of the pointwise divisor maps $\tn{div}_x=\nu_x$, as
\[\ms{D}\tn{iv}_X=\bigoplus_{x\in \tn{\footnotesize{Zar}}_X^{1}}\underline{\tn{div}_x},\]
where each summand means the map on skyscraper sheaves induced by the corresponding map of groups.

\subsection{Divisor Sequence as the Cousin Resolution of $\ms{O}_X^*$; \\ Coniveau Filtration}\label{subsectioncousinconiveauOX}

The expressions for $\ms{Z}_X^1$ and $\ms{R}_X^*$ as direct sums of skyscraper sheaves are examples of an important general phenomenon: special exact sequences of sheaves on $X$, involving interesting substructures of $X$ such as algebraic cycles, may be expressed in terms of the codimensions of subvarieties of $X$.  These special exact sequences are called {\bf Cousin resolutions.}  

The divisor sequence may be expressed as the Cousin resolution of the sheaf $\ms{O}_X^*$ of invertible rational functions on $X$ by writing it in the following form:
\[1\rightarrow\ms{O}_X^*\overset{i}{\longrightarrow}\bigoplus_{x\in \tn{\footnotesize{Zar}}_X^{0}}\underline{R_X^*}\overset{\ms{D}\tn{\footnotesize{iv}}_X}{\longrightarrow}\bigoplus_{x\in \tn{\footnotesize{Zar}}_X^{1}}\underline{\ZZ}^+\rightarrow0.\]
Cousin resolutions are constructed by means of the {\bf coniveau filtration} of the underlying topological space $\tn{Zar}_X$ of $X$.  ``Coniveau" is a French word meaning ``codimension" in this context.  The coniveau filtration of $X$ is defined by the pair of subspaces $\tn{Zar}_X^{0}$ and $\tn{Zar}_X^{1}$.  The coniveau filtration provides a method of organizing topological data in terms of the codimensions of subspaces of $X$.   Appropriate generalizations of the divisor sequence give Cousin resolutions of a number of other important sheaves, both on smooth projective varieties and on certain well-behaved singular schemes, such as the thickened curve $X_\ee$ defined in section \hyperref[SectionThickened]{2.5}.

\section{First-Order Infinitesimal Theory: A Brief Look Ahead}\label{SectionNaive}

\subsection{Theme of Linearization}\label{subsectionlinearization}

At the most elementary level, the infinitesimal theory of cycle groups and Chow groups of smooth algebraic varieties may viewed in the same general spirit as the theory of Taylor series in elementary calculus, in which one simplifies the object of interest in the neighborhood of a given element by ignoring ``higher-order information."  The simplest nontrivial version of the object, embodying only its ``first-order infinitesimal structure," is its {\bf linearization}, or {\bf tangent object.}  Other familiar examples of infinitesimal structure appear in differential geometry, algebraic geometry, and Lie theory.\footnotemark\footnotetext{A sophisticated modern example is Jan Stienstra's {\it Cartier-Dieudonn\'e theory for Chow groups} \cite{StienstraCartierDieudonne}, \cite{StienstraCartierDieudonneCorrection}, in which a particular differential graded module (the ``Dieudonn\'e module") plays the role of tangent space or Lie algebra.} 

The principal motivation for studying the infinitesimal theory of cycle groups and Chow groups of smooth algebraic varieties is that the groups themselves are very complicated.  Smooth projective curves are an exception; while nontrivial and of great interest, their cycle groups and Chow groups are well-understood.   For example, using the fact that the first Chow group of a smooth projective curve is an algebraic group, its tangent object may be easily identified as a complex vector space of dimension equal to the genus of the curve.  However, this approach does not generalize to more complicated varieties.  Here I will introduce an approach that does generalize, based on the coniveau filtration of $X$.

\subsection{Coniveau Machine for Zero-Cycles on a Smooth Projective Curve}\label{subsectionconiveausmoothcurve}

The starting point for the new approach, in the case of a smooth projective curve $X$, is the divisor sequence, viewed as the Cousin resolution of $\ms{O}_X^*$.  From this sequence I will construct a diagram of sheaves on $X$ that organizes information about infinitesimal structure in terms of the coniveau filtration of $X$.  I will call this diagram the {\bf coniveau machine} for zero-cycles on a smooth projective curve.  It has the schematic form shown in figure \hyperref[figschematicconiveaucurves]{\ref{figschematicconiveaucurves}} below.

\begin{figure}[H]
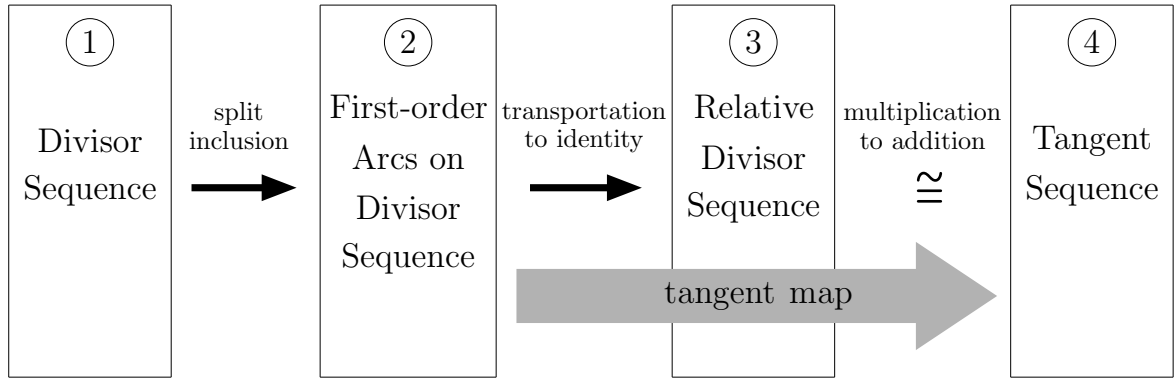

\begin{pgfpicture}{0cm}{0cm}{17cm}{5cm}
\begin{pgfmagnify}{.9}{.9}
\begin{pgftranslate}{\pgfpoint{.2cm}{-1.3cm}}
\begin{pgftranslate}{\pgfpoint{-1cm}{0cm}}
\pgfxyline(1,1.5)(1,7)
\pgfxyline(1,7)(3.4,7)
\pgfxyline(3.4,7)(3.4,1.5)
\pgfxyline(3.4,1.5)(1,1.5)
\pgfputat{\pgfxy(2.2,5)}{\pgfbox[center,center]{\large{Divisor}}}
\pgfputat{\pgfxy(2.2,4.25)}{\pgfbox[center,center]{\large{Sequence}}}
\pgfnodecircle{Node0}[stroke]{\pgfxy(2.2,6.45)}{0.35cm}
\pgfputat{\pgfxy(2.2,6.45)}{\pgfbox[center,center]{\large{$1$}}}
\end{pgftranslate}
\begin{pgftranslate}{\pgfpoint{3.6cm}{0cm}}
\pgfxyline(1,1.5)(1,7)
\pgfxyline(1,7)(3.6,7)
\pgfxyline(3.6,7)(3.6,1.5)
\pgfxyline(3.6,1.5)(1,1.5)
\pgfputat{\pgfxy(2.3,5.5)}{\pgfbox[center,center]{\large{First-order}}}
\pgfputat{\pgfxy(2.3,4.75)}{\pgfbox[center,center]{\large{Arcs on}}}
\pgfputat{\pgfxy(2.3,4)}{\pgfbox[center,center]{\large{Divisor}}}
\pgfputat{\pgfxy(2.3,3.25)}{\pgfbox[center,center]{\large{Sequence}}}
\pgfnodecircle{Node0}[stroke]{\pgfxy(2.3,6.45)}{0.35cm}
\pgfputat{\pgfxy(2.3,6.45)}{\pgfbox[center,center]{\large{$2$}}}
\pgfsetendarrow{\pgfarrowtriangle{6pt}}
\pgfsetlinewidth{3pt}
\pgfxyline(-.9,4.3)(.4,4.3)
\pgfputat{\pgfxy(-.2,5.4)}{\pgfbox[center,center]{\small{split}}}
\pgfputat{\pgfxy(-.2,5)}{\pgfbox[center,center]{\small{inclusion}}}
\end{pgftranslate}
\begin{pgftranslate}{\pgfpoint{8.8cm}{0cm}}
\pgfxyline(1,1.5)(1,7)
\pgfxyline(1,7)(3.4,7)
\pgfxyline(3.4,7)(3.4,1.5)
\pgfxyline(3.4,1.5)(1,1.5)
\pgfputat{\pgfxy(2.2,5.5)}{\pgfbox[center,center]{\large{Relative}}}
\pgfputat{\pgfxy(2.2,4.75)}{\pgfbox[center,center]{\large{Divisor}}}
\pgfputat{\pgfxy(2.2,4)}{\pgfbox[center,center]{\large{Sequence}}}
\pgfnodecircle{Node0}[stroke]{\pgfxy(2.2,6.45)}{0.35cm}
\pgfputat{\pgfxy(2.2,6.45)}{\pgfbox[center,center]{\large{$3$}}}
\pgfsetendarrow{\pgfarrowtriangle{6pt}}
\pgfsetlinewidth{3pt}
\pgfxyline(-1.1,4.3)(.3,4.3)
\pgfputat{\pgfxy(-.3,5.4)}{\pgfbox[center,center]{\small{transportation}}}
\pgfputat{\pgfxy(-.3,5)}{\pgfbox[center,center]{\small{to identity}}}
\end{pgftranslate}
\begin{pgftranslate}{\pgfpoint{13.8cm}{0cm}}
\pgfxyline(1,1.5)(1,7)
\pgfxyline(1,7)(3.4,7)
\pgfxyline(3.4,7)(3.4,1.5)
\pgfxyline(3.4,1.5)(1,1.5)
\pgfputat{\pgfxy(2.2,5)}{\pgfbox[center,center]{\large{Tangent}}}
\pgfputat{\pgfxy(2.2,4.25)}{\pgfbox[center,center]{\large{Sequence}}}
\pgfputat{\pgfxy(-.3,5.4)}{\pgfbox[center,center]{\small{multiplication}}}
\pgfputat{\pgfxy(-.3,5)}{\pgfbox[center,center]{\small{to addition}}}
\pgfputat{\pgfxy(-.2,4.3)}{\pgfbox[center,center]{\huge{$\cong$}}}
\pgfnodecircle{Node0}[stroke]{\pgfxy(2.2,6.45)}{0.35cm}
\pgfputat{\pgfxy(2.2,6.45)}{\pgfbox[center,center]{\large{$4$}}}
\end{pgftranslate}
\begin{colormixin}{30!white}
\color{black}
\pgfmoveto{\pgfxy(7.5,3.1)}
\pgflineto{\pgfxy(13.4,3.1)}
\pgflineto{\pgfxy(13.4,3.5)}
\pgflineto{\pgfxy(14.6,2.7)}
\pgflineto{\pgfxy(13.4,1.9)}
\pgflineto{\pgfxy(13.4,2.3)}
\pgflineto{\pgfxy(7.5,2.3)}
\pgflineto{\pgfxy(7.5,3.1)}
\pgffill
\end{colormixin}
\pgfputat{\pgfxy(10.5,2.7)}{\pgfbox[center,center]{\large{tangent}}}
\pgfputat{\pgfxy(12,2.65)}{\pgfbox[center,center]{\large{map}}}
\end{pgftranslate}
\end{pgfmagnify}
\end{pgfpicture}
\caption{Schematic diagram of the coniveau machine for zero-cycles on a smooth projective curve.}
\label{figschematicconiveaucurves}
\end{figure}
\vspace*{-.5cm}

The coniveau machine gives precise meaning to the intuitive notion of ``tangents of arcs of zero-cycles on $X$."  In analogy to Taylor series, such tangents encode the ``first-order infinitesimal structure" of the group $Z_X^1$ of zero-cycles on $X$.  The coniveau machine also provides a systematic method of organizing information about linear equivalence, which leads to a complete description of the first-order infinitesimal structure of $\tn{Ch}_X^1$.
 
The columns of the coniveau machine, numbered $1$ through $4$, are exact sequences of sheaves on $\tn{Zar}_X$.  The horizontal maps are chain maps among these sequences, each consisting of three maps of sheaves.  The map of principal interest is the {\bf tangent map}, which is the composition of the chain maps from columns $2$ to $3$ and $3$ to $4$.   The tangent map sends ``first-order arcs on the divisor sequence" to their ``tangent elements" in the tangent sequence.  In figure \hyperref[figdivisorsequenceconiveau]{\ref{figdivisorsequenceconiveau}} below, I have filled in the divisor sequence as the first column of the coniveau machine.   Question marks represent the nine sheaves and fifteen sheaf maps that must still be defined to complete the machine.\footnotemark\footnotetext{Although the divisor sequence should be viewed as the Cousin resolution of $\ms{O}_X^*$ in this context, I have written it in the usual form to avoid clutter.  For the same reason, I have omitted initial and terminal $1$'s and $0$'s from the exact sequences in the columns, and have suppressed the tangent map, since it is defined by composition.}

\begin{figure}[H]
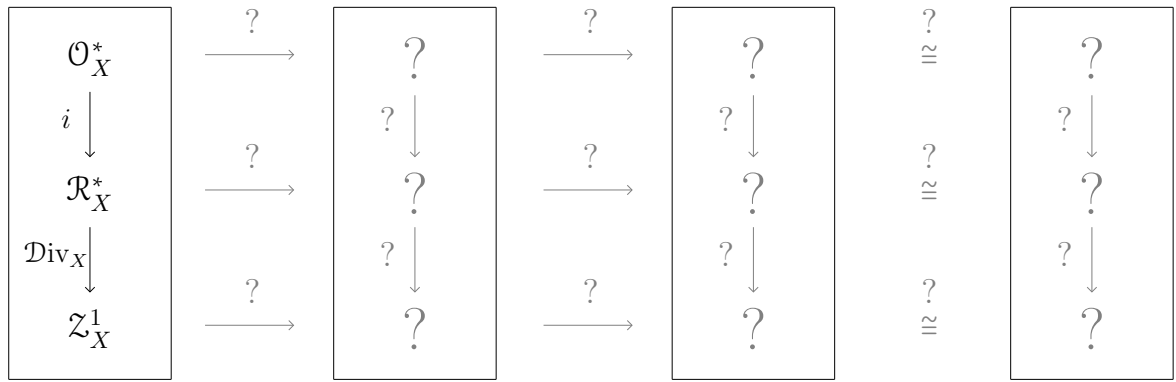

\begin{pgfpicture}{0cm}{0cm}{17cm}{5.2cm}
\begin{pgfmagnify}{.9}{.9}
\begin{pgftranslate}{\pgfpoint{.2cm}{-1.3cm}}
\begin{pgftranslate}{\pgfpoint{-1cm}{0cm}}
\pgfxyline(1,1.5)(1,7)
\pgfxyline(1,7)(3.4,7)
\pgfxyline(3.4,7)(3.4,1.5)
\pgfxyline(3.4,1.5)(1,1.5)
\pgfputat{\pgfxy(2.2,6.25)}{\pgfbox[center,center]{\large{$\ms{O}_X^*$}}}
\pgfputat{\pgfxy(2.2,4.25)}{\pgfbox[center,center]{\large{$\ms{R}^*_X$}}}
\pgfputat{\pgfxy(2.2,2.25)}{\pgfbox[center,center]{\large{$\ms{Z}_X^1$}}}
\pgfputat{\pgfxy(1.85,5.35)}{\pgfbox[center,center]{$i$}}
\pgfputat{\pgfxy(1.7,3.35)}{\pgfbox[center,center]{$\ms{D}\tn{iv}_X$}}
\pgfsetendarrow{\pgfarrowpointed{3pt}}
\pgfxyline(2.2,5.75)(2.2,4.8)
\pgfxyline(2.2,3.75)(2.2,2.8)
\end{pgftranslate}
\begin{pgftranslate}{\pgfpoint{3.8cm}{0cm}}
\pgfxyline(1,1.5)(1,7)
\pgfxyline(1,7)(3.4,7)
\pgfxyline(3.4,7)(3.4,1.5)
\pgfxyline(3.4,1.5)(1,1.5)
\begin{colormixin}{10!white}
\end{colormixin}
\begin{colormixin}{50!white}
\pgfputat{\pgfxy(2.2,6.25)}{\pgfbox[center,center]{\huge{?}}}
\pgfputat{\pgfxy(2.2,4.25)}{\pgfbox[center,center]{\huge{?}}}
\pgfputat{\pgfxy(2.2,2.25)}{\pgfbox[center,center]{\huge{?}}}
\pgfputat{\pgfxy(1.8,5.35)}{\pgfbox[center,center]{\large{?}}}
\pgfputat{\pgfxy(1.8,3.35)}{\pgfbox[center,center]{\large{?}}}
\pgfputat{\pgfxy(-.2,6.8)}{\pgfbox[center,center]{\large{?}}}
\pgfputat{\pgfxy(-.2,4.8)}{\pgfbox[center,center]{\large{?}}}
\pgfputat{\pgfxy(-.2,2.8)}{\pgfbox[center,center]{\large{?}}}
\pgfsetendarrow{\pgfarrowpointed{3pt}}
\pgfxyline(2.2,5.7)(2.2,4.8)
\pgfxyline(2.2,3.75)(2.2,2.8)
\pgfxyline(-.9,6.3)(.4,6.3)
\pgfxyline(-.9,4.3)(.4,4.3)
\pgfxyline(-.9,2.3)(.4,2.3)
\end{colormixin}
\end{pgftranslate}

\begin{pgftranslate}{\pgfpoint{8.8cm}{0cm}}
\pgfxyline(1,1.5)(1,7)
\pgfxyline(1,7)(3.4,7)
\pgfxyline(3.4,7)(3.4,1.5)
\pgfxyline(3.4,1.5)(1,1.5)
\begin{colormixin}{10!white}
\end{colormixin}
\begin{colormixin}{50!white}
\pgfputat{\pgfxy(2.2,6.25)}{\pgfbox[center,center]{\huge{?}}}
\pgfputat{\pgfxy(2.2,4.25)}{\pgfbox[center,center]{\huge{?}}}
\pgfputat{\pgfxy(2.2,2.25)}{\pgfbox[center,center]{\huge{?}}}
\pgfputat{\pgfxy(1.8,5.35)}{\pgfbox[center,center]{\large{?}}}
\pgfputat{\pgfxy(1.8,3.35)}{\pgfbox[center,center]{\large{?}}}
\pgfputat{\pgfxy(-.2,6.8)}{\pgfbox[center,center]{\large{?}}}
\pgfputat{\pgfxy(-.2,4.8)}{\pgfbox[center,center]{\large{?}}}
\pgfputat{\pgfxy(-.2,2.8)}{\pgfbox[center,center]{\large{?}}}
\pgfsetendarrow{\pgfarrowpointed{3pt}}
\pgfxyline(2.2,5.7)(2.2,4.8)
\pgfxyline(2.2,3.75)(2.2,2.8)
\pgfxyline(-.9,6.3)(.4,6.3)
\pgfxyline(-.9,4.3)(.4,4.3)
\pgfxyline(-.9,2.3)(.4,2.3)
\end{colormixin}
\end{pgftranslate}
\begin{pgftranslate}{\pgfpoint{13.8cm}{0cm}}
\pgfxyline(1,1.5)(1,7)
\pgfxyline(1,7)(3.4,7)
\pgfxyline(3.4,7)(3.4,1.5)
\pgfxyline(3.4,1.5)(1,1.5)
\begin{colormixin}{10!white}
\end{colormixin}
\begin{colormixin}{50!white}
\pgfputat{\pgfxy(2.2,6.25)}{\pgfbox[center,center]{\huge{?}}}
\pgfputat{\pgfxy(2.2,4.25)}{\pgfbox[center,center]{\huge{?}}}
\pgfputat{\pgfxy(2.2,2.25)}{\pgfbox[center,center]{\huge{?}}}
\pgfputat{\pgfxy(1.8,5.35)}{\pgfbox[center,center]{\large{?}}}
\pgfputat{\pgfxy(1.8,3.35)}{\pgfbox[center,center]{\large{?}}}
\pgfputat{\pgfxy(-.2,6.8)}{\pgfbox[center,center]{\large{?}}}
\pgfputat{\pgfxy(-.2,4.8)}{\pgfbox[center,center]{\large{?}}}
\pgfputat{\pgfxy(-.2,2.8)}{\pgfbox[center,center]{\large{?}}}
\pgfsetendarrow{\pgfarrowpointed{3pt}}
\pgfxyline(2.2,5.7)(2.2,4.8)
\pgfxyline(2.2,3.75)(2.2,2.8)
\pgfputat{\pgfxy(-.2,6.3)}{\pgfbox[center,center]{\large{$\cong$}}}
\pgfputat{\pgfxy(-.2,2.3)}{\pgfbox[center,center]{\large{$\cong$}}}
\pgfputat{\pgfxy(-.2,4.3)}{\pgfbox[center,center]{\large{$\cong$}}}
\end{colormixin}
\end{pgftranslate}
\end{pgftranslate}
\end{pgfmagnify}
\end{pgfpicture}
\caption{The divisor sequence as the first column of the coniveau machine.}
\label{figdivisorsequenceconiveau}
\end{figure}
\vspace*{-.5cm}

Suitably completed, the coniveau machine contains a wealth of information about regular functions, rational functions, and algebraic cycles on $X$, of which the infinitesimal structure of the groups $Z_X^1$ and $\tn{Ch}_X^1$ is only a small part.   In particular, the tangent map consists of three sheaf maps, but only the map between the bottom terms of columns $2$ and $4$ directly involves zero-cycles.  This map, in turn, is made up of many maps between groups of sections, but only the map on global sections directly involves $Z_X^1$ and $\tn{Ch}_X^1$.   Discarding further information, this map induces a ``tangent map for $\tn{Ch}_X^1$," which preserves information only modulo linear equivalence.

\section{First-Order Theory of Nonzero Rational Functions \\ on a Smooth Projective Curve}\label{SectionInfRational}

Arcs of nonzero rational functions on a smooth projective curve $X$ provide a natural starting point for studying the infinitesimal theory of zero-cycles on $X$.  Since every zero-cycle on $X$ is a locally principal divisor, ``arcs of zero-cycles on $X$" may be described locally on $X$ as ``divisors of arcs of nonzero rational functions."   In this section, I introduce the theory of first-order arcs of nonzero rational functions on $X$, along with their tangents.   These notions are made precise by three new sheaves $\ms{R}_{X_\ee}^*$, $\ms{R}_{X_\ee,\ee}^*$, and $\ms{R}_{X}^+$, which encode the first-order infinitesimal structure of the sheaf $\ms{R}_{X}^*$ of nonzero rational functions on $X$.   These three sheaves fill in the middle row of the coniveau machine for zero-cycles on a smooth projective curve.

\subsection{First-Order Arcs of Nonzero Rational Functions}

An arc of nonzero rational functions on a smooth projective curve $X$ may be viewed, imprecisely, as a ``function of two variables" $f(w,\ee)$, where $w$ is a local coordinate on $X$, and $\ee$ is an ``infinitesimal quantity."   In the first-order case, $f(w,\ee)$ may be written as a sum $f+ g\ee$, where $f$ and $g$ are rational functions on $X$, expressed in terms of the local variable $w$, and where $f\ne0$.   Heuristically, the arc $f+g\ee$ ``originates at $f$ and moves in the direction of $g$ as $\ee$ varies."  More precisely, $\ee$ is not a variable, but the nilpotent generator of the {\bf ring of dual numbers} $R_{X_\ee}:=R_X[\ee]/\ee^2$ over the field $R_X$.    This ring of dual numbers splits as an additive group into the direct sum $R_X\oplus R_X\ee$, so its elements are sums $f+g\ee$, invertible if and only if $f\ne0$.    

A {\bf first-order arc of nonzero rational functions} at $f$ in $R_X^*$ is defined to be an element $f+g\ee$ of the multiplicative subgroup $R_{X_\ee}^*$ of invertible elements of $R_{X_\ee}$.  Hence, $R_{X_\ee}^*$ is called the {\bf group of first-order arcs} in $R_X^*$.  Although $f\ne0$ by invertibility, $g$ may vanish; the resulting arc is called the {\bf constant arc at $f$}.   An arc $f+g\ee$  at $f$ may be transported to the identity of $R_X^*$ by dividing by $f$.  The resulting {\bf first-order arc at the identity} in $R_X^*$ is the arc $1+(g/f)\ee.$  The set $R_{X_\ee,\ee}^*$ of such arcs is a subgroup of $R_{X_\ee}^*$, called the {\bf relative group} of first-order arcs at the identity in $R_X^*$.  This group will be identified as a ``multiplicative version of the tangent group of $R_X^*$," defined below.

\subsection{Tangent Elements, Sets, Maps, and Groups}\label{subsectionrationaltangent}

The {\bf tangent element} of a first-order arc $f+g\ee$ of nonzero rational functions at $f$ in $R_X^*$ is defined to be the rational function $g$.  This definition captures the idea that a ``tangent at a point" should give the ``direction along an arc through the point."  Tangent elements are rational functions, so the {\bf tangent set} of $R_X^*$ at an element $f$ is defined to be the underlying set $R_X^{\tn{\footnotesize{set}}}$ of the function field $R_X$.  The reason for recognizing only the underlying set $R_X^{\tn{\footnotesize{set}}}$ of $R_X$ in the context of tangents to arcs in $R_X^*$ is that the operations of addition and multiplication that define the field structure on $R_X$ do not define natural pointwise operations at nonidentity elements of $R_X^*$. A different copy of this set is assigned to each element of $R_X^*$, just as a different tangent space of the same is dimension is assigned to each point in a manifold.  The  {\bf tangent map  $T_{f}$ at $f$} in $R_{X}^*$ is the map $R_{X_\ee}^*\rightarrow R_X^{\tn{\footnotesize{set}}}$ taking a first-order arc $f+g\ee$ at $f$ to its tangent element $g$.   

Since $R_{X_\ee}^*$ is a group, first-order arcs of nonzero rational functions may be multiplied.  Applying the tangent map at the identity to the product of two arcs at the identity gives the formula
\[T_1\big((1+f\ee)(1+g\ee)\big)\hspace*{.2cm}=\hspace*{.2cm}T_1(1+f\ee)+T_1(1+g\ee).\]
Thus, multiplication of first-order arcs at the identity in $R_X^*$ corresponds, under $T_1$, to addition of tangent elements in the tangent set of $R_X^*$ at the identity.  It is therefore reasonable to ascribe a group structure to this tangent set.  Hence, the {\bf tangent group} of the group $R_X^*$ of nonzero rational functions is defined to be the underlying additive group $R_X^+$ of the field $R_X$.  The tangent map at the identity $T_{1}$ is an isomorphism from the relative group $R_{X_\ee,\ee}^*$ of first-order arcs at the identity in $R_X^*$ to the tangent group $R_X^+$.  Hence, $R_{X_\ee,\ee}^*$ is a ``multiplicative version of the tangent group" of $R_X^*$.

Since transporting the arc $f+g\ee$ at $f$ to the identity yields the arc $1+(g/f)\ee$, the tangent element $g/f$ of the latter arc is called the {\bf tangent element at the identity} of the original arc $f+g\ee$.  The {\bf tangent map} $T:R_{X_\ee}^*\rightarrow R_X^+$ is defined by extending the tangent map at the identity $T_1$ to the map taking a first-order arc $f+g\ee$ to its tangent element at the identity $g/f$.  The tangent map $T$ factors through the relative multiplicative group $R_{X_\ee,\ee}^*$ via $T_{1}$.  The kernel of $T$ is isomorphic to the multiplicative group $R_X^*$ of the function field $R_X$, since an arc $f+g\ee$ has zero tangent element if and only if $g=0$.  Hence, there exists a four-term hybrid exact sequence, consisting of three multiplicative groups and one additive group:
\[1\rightarrow R_X^*\overset{i}{\longrightarrow}R_{X_\ee}^*\overset{1/f}{\longrightarrow} R_{X_\ee,\ee}^*\overset{T_1}{\rightarrow}R_X^+\rightarrow 0.\]
In the few cases in which it is useful to distinguish among the canonically isomorphic groups $R_X^*$, $R_U^*$, and $R_x^*$, I will denote the corresponding tangent groups by $R_X^+$, $R_U^+$, and $R_x^+$, the relative groups by $R_{X_\ee,\ee}^*$, $R_{U_\ee,\ee}^*$, and $R_{x_\ee,\ee}^*$, the tangent maps by $T_X$, $T_U$, and $T_x$, and the tangent maps at particular nonzero rational functions by $T_{X,f}$, $T_{U,f}$, and $T_{x,f}$.

\subsection{Sheaves of First-Order Arcs and Tangents; Sheaf Tangent Map}\label{SubsectionTangentSheafRational}

Since $R_U$ is canonically isomorphic to $R_X$ for any nonempty open subset $U$ of $X$, a separate copy of the above four-term hybrid exact sequence of groups may be assigned to each $U$, replacing $X$ with $U$ everywhere in the sequence.  This leads immediately to a four-term hybrid exact sequence of sheaves on $X$
\[1\rightarrow \ms{R}_X^*\overset{i}{\longrightarrow}\ms{R}_{X_\ee}^*\overset{1/f}{\longrightarrow}\ms{R}_{X_\ee,\ee}^* \overset{T_1}{\longrightarrow} \ms{R}_X^+\rightarrow 0,\]
where the three new sheaves $\ms{R}_{X_\ee}^*$, $\ms{R}_{X_\ee,\ee}^*$, and $\ms{R}_X^+$ are defined to be the sheaves with groups of sections $R_{U_\ee}^*$, $R_{U_\ee,\ee}^*$, and $R_U^+$, respectively, over an open subset $U$ of $X$.   The three new sheaves are canonically isomorphic to direct sums of skyscraper sheaves with single summands:   
\[\ms{R}_{X_\ee}^*=\bigoplus_{x\in \tn{\footnotesize{Zar}}_X^{0}}\underline{R_{X_\ee}^*},\hspace*{.5cm}\ms{R}_{X_\ee,\ee}^*=\bigoplus_{x\in \tn{\footnotesize{Zar}}_X^{0}}\underline{R_{X_\ee,\ee}^*},\hspace*{.5cm}\tn{and}\hspace*{.5cm}\ms{R}_X^+=\bigoplus_{x\in \tn{\footnotesize{Zar}}_X^{0}}\underline{R_X^+}.\]
The sheaf maps $i$, $1/f,$ and $T_1$ are defined by patching together copies of the corresponding maps of groups $i, 1/f,$ and $T_{U,1}$ for each open subset $U$ of $X$.  This patching is trivial because the groups of sections over any pair of nonempty open subsets of $X$ are canonically isomorphic. 

Mirroring the group-level definitions, the sheaf $\ms{R}_{X_\ee}^*$ is called the {\bf sheaf of first-order arcs} in $\ms{R}_X^*$,  and the sheaf $\ms{R}_X^+$ is called the {\bf tangent sheaf} of the sheaf $\ms{R}_X^*$.  The sheaf $\ms{R}_{X_\ee,\ee}^*$ is called the {\bf relative sheaf} of first-order arcs at the identity in $\ms{R}_X^*$.   It may be viewed as a multiplicative version of the tangent sheaf $\ms{R}_X^+$ of $\ms{R}_X^*$.  The sheaf map $T_1$ is called the {\bf sheaf tangent map at the identity} in $\ms{R}_X^*$.  The composition $T$ of the sheaf maps $T_1$ and $1/f$ is called the {\bf sheaf tangent map}, or simply the {\bf tangent map}.  The four-term hybrid exact sequence of sheaves links together with the divisor sequence to fill in the middle row of the coniveau machine for zero-cycles on a smooth projective curve, as shown in figure \hyperref[figrationalconiveau]{\ref{figrationalconiveau}} below:

\begin{figure}[H]
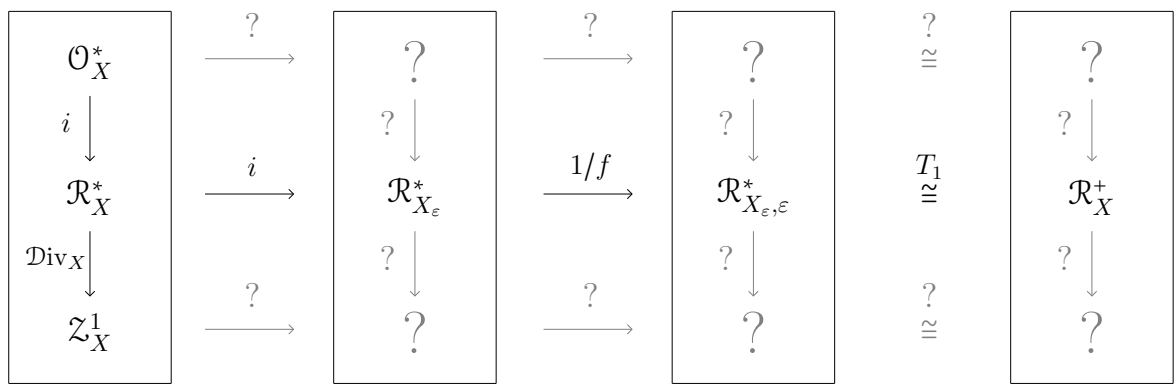

\begin{pgfpicture}{0cm}{0cm}{17cm}{5.2cm}
\begin{pgfmagnify}{.9}{.9}
\begin{pgftranslate}{\pgfpoint{.2cm}{-1.3cm}}
\begin{pgftranslate}{\pgfpoint{-1cm}{0cm}}
\pgfxyline(1,1.5)(1,7)
\pgfxyline(1,7)(3.4,7)
\pgfxyline(3.4,7)(3.4,1.5)
\pgfxyline(3.4,1.5)(1,1.5)
\pgfputat{\pgfxy(2.2,6.25)}{\pgfbox[center,center]{\large{$\ms{O}_X^*$}}}
\pgfputat{\pgfxy(2.2,4.25)}{\pgfbox[center,center]{\large{$\ms{R}^*_X$}}}
\pgfputat{\pgfxy(2.2,2.25)}{\pgfbox[center,center]{\large{$\ms{Z}_X^1$}}}
\pgfputat{\pgfxy(1.85,5.35)}{\pgfbox[center,center]{$i$}}
\pgfputat{\pgfxy(1.7,3.35)}{\pgfbox[center,center]{\small{$\ms{D}\tn{iv}_X$}}}
\pgfsetendarrow{\pgfarrowpointed{3pt}}
\pgfxyline(2.2,5.75)(2.2,4.8)
\pgfxyline(2.2,3.75)(2.2,2.8)
\end{pgftranslate}
\begin{pgftranslate}{\pgfpoint{3.8cm}{0cm}}
\pgfxyline(1,1.5)(1,7)
\pgfxyline(1,7)(3.4,7)
\pgfxyline(3.4,7)(3.4,1.5)
\pgfxyline(3.4,1.5)(1,1.5)
\pgfputat{\pgfxy(2.2,4.25)}{\pgfbox[center,center]{\large{$\ms{R}^*_{X_\ee}$}}}
\begin{colormixin}{10!white}
\end{colormixin}
\begin{colormixin}{50!white}
\pgfputat{\pgfxy(2.2,6.25)}{\pgfbox[center,center]{\huge{?}}}
\pgfputat{\pgfxy(2.2,2.25)}{\pgfbox[center,center]{\huge{?}}}
\pgfputat{\pgfxy(1.8,5.35)}{\pgfbox[center,center]{\large{?}}}
\pgfputat{\pgfxy(1.8,3.35)}{\pgfbox[center,center]{\large{?}}}
\pgfputat{\pgfxy(-.2,6.8)}{\pgfbox[center,center]{\large{?}}}
\pgfputat{\pgfxy(-.2,2.8)}{\pgfbox[center,center]{\large{?}}}
\pgfsetendarrow{\pgfarrowpointed{3pt}}
\pgfxyline(2.2,5.7)(2.2,4.8)
\pgfxyline(2.2,3.75)(2.2,2.8)
\pgfxyline(-.9,6.3)(.4,6.3)
\pgfxyline(-.9,2.3)(.4,2.3)
\end{colormixin}
\pgfsetendarrow{\pgfarrowpointed{3pt}}
\pgfxyline(-.9,4.3)(.4,4.3)
\pgfputat{\pgfxy(-.2,4.7)}{\pgfbox[center,center]{$i$}}
\end{pgftranslate}

\begin{pgftranslate}{\pgfpoint{8.8cm}{0cm}}
\pgfxyline(1,1.5)(1,7)
\pgfxyline(1,7)(3.4,7)
\pgfxyline(3.4,7)(3.4,1.5)
\pgfxyline(3.4,1.5)(1,1.5)
\pgfputat{\pgfxy(2.2,4.25)}{\pgfbox[center,center]{\large{$\ms{R}^*_{X_\ee,\ee}$}}}
\begin{colormixin}{10!white}
\end{colormixin}
\begin{colormixin}{50!white}
\pgfputat{\pgfxy(2.2,6.25)}{\pgfbox[center,center]{\huge{?}}}
\pgfputat{\pgfxy(2.2,2.25)}{\pgfbox[center,center]{\huge{?}}}
\pgfputat{\pgfxy(1.8,5.35)}{\pgfbox[center,center]{\large{?}}}
\pgfputat{\pgfxy(1.8,3.35)}{\pgfbox[center,center]{\large{?}}}
\pgfputat{\pgfxy(-.2,6.8)}{\pgfbox[center,center]{\large{?}}}
\pgfputat{\pgfxy(-.2,2.8)}{\pgfbox[center,center]{\large{?}}}
\pgfsetendarrow{\pgfarrowpointed{3pt}}
\pgfxyline(2.2,5.7)(2.2,4.8)
\pgfxyline(2.2,3.75)(2.2,2.8)
\pgfxyline(-.9,6.3)(.4,6.3)
\pgfxyline(-.9,2.3)(.4,2.3)
\end{colormixin}
\pgfsetendarrow{\pgfarrowpointed{3pt}}
\pgfxyline(-.9,4.3)(.4,4.3)
\pgfputat{\pgfxy(-.2,4.7)}{\pgfbox[center,center]{$1/f$}}
\end{pgftranslate}
\begin{pgftranslate}{\pgfpoint{13.8cm}{0cm}}
\pgfxyline(1,1.5)(1,7)
\pgfxyline(1,7)(3.4,7)
\pgfxyline(3.4,7)(3.4,1.5)
\pgfxyline(3.4,1.5)(1,1.5)
\pgfputat{\pgfxy(2.2,4.25)}{\pgfbox[center,center]{\large{$\ms{R}_X^+$}}}
\begin{colormixin}{10!white}
\end{colormixin}
\begin{colormixin}{50!white}
\pgfputat{\pgfxy(2.2,6.25)}{\pgfbox[center,center]{\huge{?}}}
\pgfputat{\pgfxy(2.2,2.25)}{\pgfbox[center,center]{\huge{?}}}
\pgfputat{\pgfxy(1.8,5.35)}{\pgfbox[center,center]{\large{?}}}
\pgfputat{\pgfxy(1.8,3.35)}{\pgfbox[center,center]{\large{?}}}
\pgfputat{\pgfxy(-.2,6.8)}{\pgfbox[center,center]{\large{?}}}
\pgfputat{\pgfxy(-.2,2.8)}{\pgfbox[center,center]{\large{?}}}
\pgfsetendarrow{\pgfarrowpointed{3pt}}
\pgfxyline(2.2,5.7)(2.2,4.8)
\pgfxyline(2.2,3.75)(2.2,2.8)
\pgfputat{\pgfxy(-.2,6.3)}{\pgfbox[center,center]{\large{$\cong$}}}
\pgfputat{\pgfxy(-.2,2.3)}{\pgfbox[center,center]{\large{$\cong$}}}
\end{colormixin}
\pgfsetendarrow{\pgfarrowpointed{3pt}}
\pgfputat{\pgfxy(-.2,4.3)}{\pgfbox[center,center]{\large{$\cong$}}}
\pgfputat{\pgfxy(-.2,4.7)}{\pgfbox[center,center]{$T_1$}}
\end{pgftranslate}
\end{pgftranslate}
\end{pgfmagnify}
\end{pgfpicture}
\caption{Rational structure as the middle row of the coniveau machine for zero-cycles on a smooth projective curve.}
\label{figrationalconiveau}
\end{figure}
\vspace*{-.5cm}

Since the four sheaves in the middle row of the coniveau machine are skyscraper sheaves at the generic point of $X$, they contain no more information than their individual underlying groups.  However, the maps in the coniveau machine involving these sheaves contain more information than can be encoded in individual group maps.  For example, the sheaf divisor map $\tn{div}:\ms{R}_X^*\rightarrow\ms{Z}_X^1$ carries all the information contained in the global, local and pointwise divisor maps $\tn{div}_X$, $\tn{div}_U$ and $\tn{div}_x$.  This information is critical in describing zero-cycles in terms of rational functions, since not every zero-cycle is a principal divisor.  As I mentioned in section \hyperref[ZeroCyclesDivisors]{2.2}, an arbitrary zero-cycle $z$ on an open subset $U$ of $X$ is defined, in terms of rational functions, by a family of compatible pairs $\{U_i,f_i\}$, where $\{U_i\}$ is an open cover of $U$, and $z$ equals $\tn{div}_{U_i}(f_i)$ on $U_i$.  Therefore, the local divisor maps $\tn{div}_{U_i}$ are all relevant.  In a similar way, an arbitrary ``first-order arc $z_\ee$ of zero-cycles on $X$," is defined in terms of a family of compatible pairs $\{U_i, f_i+g_i\ee\}$, where $\{U_i\}$ is an open cover of $X$, and where the arc $f_i+g_i\ee$ is taken to define $z_\ee$ only on the open set $U_i$.   Compatibility means that different arcs $f_i+g_i\ee$ and $f_j+g_j\ee$ of nonzero rational functions ``define the same arcs of zero-cycles" on the intersections $U_i\cap U_j$ of their sets of definition.   Compatibility may be described more precisely by considering ``rational structure modulo regular structure," as explained in section \hyperref[SectionInfRegular]{2.5} below.

\section{First-Order Theory of Invertible Regular Functions \\on a Smooth Projective Curve}\label{SectionInfRegular}

The sheaf $\ms{O}_X^*$ of invertible regular functions on a smooth projective curve $X$ is the kernel of the sheaf divisor map $\tn{div}:\ms{R}_X^*\rightarrow\ms{Z}_X^1$, and is therefore the ``trivial part" of the sheaf $\ms{R}_X^*$ of nonzero rational functions with respect to data involving zero-cycles on $X$.   In particular, the infinitesimal structure of $\ms{O}_X^*$ is the ``trivial part" of the infinitesimal structure of $\ms{R}_X^*$.   As in the case of $\ms{R}_X^*$, the first-order infinitesimal structure of $\ms{O}_X^*$ may be expressed in terms of first-order arcs and their tangents, this time involving invertible regular functions.  These notions are made precise by three new sheaves $\ms{O}_{X_\ee}^*$, $\ms{O}_{X_\ee,\ee}^*$, and $\ms{O}_{X}^+$, which complete the top row of the coniveau machine for zero-cycles on a smooth projective curve.   The first-order theory of invertible regular functions is nearly identical, in a formal sense, to the first-order theory of nonzero rational functions, with the important exception that $\ms{O}_X^*$ is not a skyscraper sheaf.  This means that groups of invertible regular functions depend nontrivially on the open sets over which they are defined.

\subsection{First-Order Arcs of Invertible Regular Functions  and \\ Invertible Germs of Regular Functions}\label{subsectionfirstorderregular}

Let $U$ be an open subset of a smooth projective curve $X$.  A {\bf first-order arc of invertible regular functions} at the element $f$ in $O_U^*$ is defined to be an element $f+g\ee$ of the multiplicative subgroup $O_{U_\ee}^*$ of invertible elements of the ring of dual numbers $O_{U_\ee}:=O_U[\ee]/\ee^2$ over the ring of regular functions $O_U$ on $U$.  Hence, $O_{U_\ee}^*$ is called the {\bf group of first-order arcs} on $O_U^*$.   This includes the case $U=X$. 

Since $\ms{O}_X^*$ is not a skyscraper sheaf, there is generally new behavior for arcs in the multiplicative subgroups $O_x^*$ of the local rings $O_x$ of $X$.  A {\bf first-order arc of invertible germs of regular functions} at the element $f$ in $O_x^*$ is defined to be an element $f+g\ee$ of the multiplicative subgroup $O_{x_\ee}^*$ of invertible elements of the ring of dual numbers $O_{x_\ee}:=O_x[\ee]/\ee^2$ over $O_x$.  Hence, $O_{x_\ee}^*$ is called the {\bf group of local first-order arcs} on $O_x^*$.  If $x$ is the generic point of $X$, then a first-order arc of germs of regular functions in $O_x$ is just an arc of nonzero rational functions, since $O_x$ is canonically isomorphic to the rational function field $R_X$ of $X$ in this case.  

A first-order arc $f+g\ee$ at an element $f$ in $O_U^*$ or $O_x^*$ may be transported to the identity of $O_U^*$ or $O_x^*$ by dividing by $f$.  The resulting {\bf first-order arc at the identity} is the arc $1+(g/f)\ee.$  The sets $O_{U_\ee,\ee}^*$ and $O_{x_\ee,\ee}^*$ of such arcs are subgroups of $O_{U_\ee}^*$ and $O_{x_\ee}^*$, called the {\bf relative groups} of arcs at the identity in $O_U^*$ and $O_x^*$.

\subsection{Tangent Elements, Sets, Maps, and Groups}\label{subsectionregulartangents}

The {\bf tangent element} of a first-order arc $f+g\ee$ of invertible regular functions or invertible germs of regular functions at an element $f$ in $O_U^*$ or $O_x^*$ is defined to be the function $g$.  The {\bf tangent sets} of $O_U^*$ and $O_x^*$ at each element are defined to be the underlying sets $O_U^{\tn{\footnotesize{set}}}$ and $O_x^{\tn{\footnotesize{set}}}$.  The  {\bf tangent maps} $T_{U,f}$ and $T_{x,f}$ at $f$ are the maps taking first-order arcs $f+g\ee$ at $f$ to their tangent elements $g$.  The {\bf tangent groups} of $O_U^*$ and $O_x^*$ are defined to be the underlying additive groups $O_U^+$ and $O_x^+$ of the rings $O_U$ and $O_x$.  The tangent maps at the identity $T_{U,1}$ and $T_{x,1}$ are isomorphisms from the relative groups $O_{U_\ee,\ee}^*$ and $O_{x_\ee,\ee}^*$ of first-order arcs at the identity to the tangent groups $O_U^+$ and $O_x^+$.  Hence, the groups $O_{U_\ee,\ee}^*$ and $O_{x_\ee,\ee}^*$ are ``multiplicative versions of the tangent groups" of $O_U^*$ and $O_x^*$.

The {\bf tangent element at the identity} of the arc $f+g\ee$ in $O_U^*$ or $O_x^*$ is defined to be the tangent element $g/f$ of the arc $1+(g/f)\ee$ at the identity given by dividing the original arc by $f$.  The {\bf tangent maps} $T_U:O_{U_\ee}^*\rightarrow O_U^+$ and $T_x:O_{x_\ee}^*\rightarrow O_x^+$ are defined by extending the tangent maps at the identity $T_{U,1}$ and $T_{x,1}$ to the maps taking first-order arcs $f+g\ee$ to their tangent elements at the identity $g/f$.  The tangent maps factor through the relative  groups via the tangent maps at the identity.  The kernels of the tangent maps $T_U$ and $T_x$ are isomorphic to the multiplicative groups $O_U^*$ and $O_x^*$.  Hence, there exist four-term ``hybrid" exact sequences, consisting of three multiplicative groups and one additive group:
\[1\rightarrow O_U^*\overset{i}{\longrightarrow} O_{U_\ee}^*\overset{1/f}{\longrightarrow}O_{U_\ee,\ee}^* \overset{T_{U,1}}{\longrightarrow} O_U^+\rightarrow 0,\]
and
\[1\rightarrow O_x^*\overset{i}{\longrightarrow} O_{x_\ee}^*\overset{1/f}{\longrightarrow}O_{x_\ee,\ee}^* \overset{T_{x,1}}{\longrightarrow} O_x^+\rightarrow 0.\]

\subsection{Sheaves of First-Order Arcs and Tangents; Sheaf Tangent Map}\label{SubsectionTangentSheafRegular}

The above four-term hybrid exact sequences involving $U$ fit together to give a four-term hybrid exact sequence of sheaves on $X$:
\[1\rightarrow \ms{O}_X^*\overset{i}{\longrightarrow}\ms{O}_{X_\ee}^*\overset{1/f}{\longrightarrow}\ms{O}_{X_\ee,\ee}^* \overset{T_1}{\longrightarrow} \ms{O}_X^+\rightarrow 0,\]
where the three new sheaves $\ms{O}_{X_\ee}^*$, $\ms{O}_{X_\ee,\ee}^*$, and $\ms{O}_X^+$ are sheaves with groups of sections $O_{U_\ee}^*$, $O_{U_\ee,\ee}^*$, and $O_U^+$, respectively, over a nonempty open set $U$ of $X$.  The maps $i$, $1/f,$ and $T_1$ at the sheaf level are defined by patching together copies of the corresponding group maps for each open set $U$.   The stalks of these sheaves at a point $x$ in $X$ are $O_{x_\ee}^*$, $O_{x_\ee,\ee}^*$, and $O_x^+$, respectively.

Mirroring the group-level definitions, the sheaf $\ms{O}_{X_\ee}^*$ of multiplicative groups is called the {\bf sheaf of first-order arcs} on $\ms{O}_X^*$,  and the sheaf $\ms{O}_X^+$ of additive groups is called the {\bf tangent sheaf} of the sheaf $\ms{O}_X^*$.  The sheaf $\ms{O}_{X_\ee,\ee}^*$ is called the {\bf relative sheaf} of arcs at the identity in $\ms{O}_X^*$.   It may be viewed as a multiplicative version of the tangent sheaf $\ms{O}_X^+$ of $\ms{O}_X^*$.  The sheaf map $T_1$ is called the {\bf sheaf tangent map at the identity} in $\ms{O}_X^*$. The composition $T$ of the sheaf maps $T_1$ and $1/f$ is called the {\bf sheaf tangent map}, or simply the {\bf tangent map}.  The sheaves $\ms{O}_{X_\ee}^*$, $\ms{O}_{X_\ee,\ee}^*$, and $\ms{O}_X^+$ are natural subsheaves of the sheaves $\ms{R}_{X_\ee}^*$, $\ms{R}_{X_\ee,\ee}^*$, and $\ms{R}_X^+$, respectively.  Calling the inclusion maps $i$, the four-term hybrid exact  sequence of sheaves involving regular functions completes the top row of the coniveau machine for zero-cycles on a smooth projective curve, as shown in the diagram in figure \hyperref[figregularconiveau]{\ref{figregularconiveau}}: 

\begin{figure}[H]
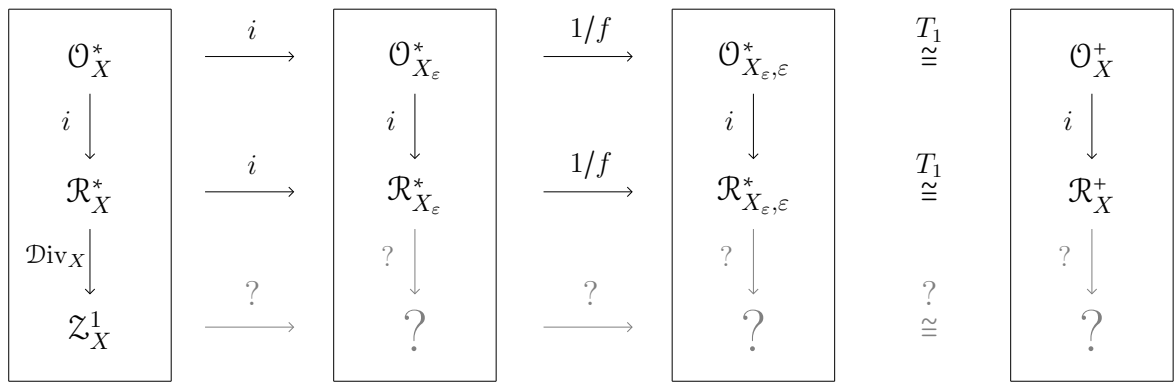

\begin{pgfpicture}{0cm}{0cm}{17cm}{5.2cm}
\begin{pgfmagnify}{.9}{.9}
\begin{pgftranslate}{\pgfpoint{.2cm}{-1.3cm}}
\begin{pgftranslate}{\pgfpoint{-1cm}{0cm}}
\pgfxyline(1,1.5)(1,7)
\pgfxyline(1,7)(3.4,7)
\pgfxyline(3.4,7)(3.4,1.5)
\pgfxyline(3.4,1.5)(1,1.5)
\pgfputat{\pgfxy(2.2,6.25)}{\pgfbox[center,center]{\large{$\ms{O}_X^*$}}}
\pgfputat{\pgfxy(2.2,4.25)}{\pgfbox[center,center]{\large{$\ms{R}^*_X$}}}
\pgfputat{\pgfxy(2.2,2.25)}{\pgfbox[center,center]{\large{$\ms{Z}_X^1$}}}
\pgfputat{\pgfxy(1.85,5.35)}{\pgfbox[center,center]{$i$}}
\pgfputat{\pgfxy(1.7,3.35)}{\pgfbox[center,center]{\small{$\ms{D}\tn{iv}_X$}}}
\pgfsetendarrow{\pgfarrowpointed{3pt}}
\pgfxyline(2.2,5.75)(2.2,4.8)
\pgfxyline(2.2,3.75)(2.2,2.8)
\end{pgftranslate}
\begin{pgftranslate}{\pgfpoint{3.8cm}{0cm}}
\pgfxyline(1,1.5)(1,7)
\pgfxyline(1,7)(3.4,7)
\pgfxyline(3.4,7)(3.4,1.5)
\pgfxyline(3.4,1.5)(1,1.5)
\pgfputat{\pgfxy(2.2,6.25)}{\pgfbox[center,center]{\large{$\ms{O}^*_{X_\ee}$}}}
\pgfputat{\pgfxy(2.2,4.25)}{\pgfbox[center,center]{\large{$\ms{R}^*_{X_\ee}$}}}
\begin{colormixin}{20!white}
\end{colormixin}
\pgfputat{\pgfxy(1.85,5.35)}{\pgfbox[center,center]{$i$}}
\begin{colormixin}{20!white}
\end{colormixin}
\pgfsetendarrow{\pgfarrowpointed{3pt}}
\pgfxyline(-.9,6.3)(.4,6.3)
\pgfxyline(-.9,4.3)(.4,4.3)
\pgfxyline(2.2,5.75)(2.2,4.8)
\pgfputat{\pgfxy(-.2,6.7)}{\pgfbox[center,center]{$i$}}
\pgfputat{\pgfxy(-.2,4.7)}{\pgfbox[center,center]{$i$}}
\begin{colormixin}{20!white}
\end{colormixin}
\begin{colormixin}{50!white}
\pgfputat{\pgfxy(2.2,2.25)}{\pgfbox[center,center]{\huge{?}}}
\pgfputat{\pgfxy(1.8,3.35)}{\pgfbox[center,center]{?}}
\pgfputat{\pgfxy(-.2,2.8)}{\pgfbox[center,center]{\large{?}}}
\pgfsetendarrow{\pgfarrowpointed{3pt}}
\pgfxyline(-.9,2.3)(.4,2.3)
\pgfxyline(2.2,3.75)(2.2,2.8)
\end{colormixin}
\end{pgftranslate}
\begin{pgftranslate}{\pgfpoint{8.8cm}{0cm}}
\pgfxyline(1,1.5)(1,7)
\pgfxyline(1,7)(3.4,7)
\pgfxyline(3.4,7)(3.4,1.5)
\pgfxyline(3.4,1.5)(1,1.5)
\pgfputat{\pgfxy(2.2,6.25)}{\pgfbox[center,center]{\large{$\ms{O}^*_{X_\ee,\ee}$}}}
\pgfputat{\pgfxy(2.2,4.25)}{\pgfbox[center,center]{\large{$\ms{R}^*_{X_\ee,\ee}$}}}
\begin{colormixin}{20!white}
\end{colormixin}
\pgfputat{\pgfxy(1.85,5.35)}{\pgfbox[center,center]{$i$}}
\begin{colormixin}{20!white}
\end{colormixin}
\begin{colormixin}{20!white}
\end{colormixin}
\pgfsetendarrow{\pgfarrowtriangle{6pt}}
\pgfsetendarrow{\pgfarrowpointed{3pt}}
\pgfxyline(-.9,6.3)(.4,6.3)
\pgfxyline(-.9,4.3)(.4,4.3)
\pgfxyline(2.2,5.75)(2.2,4.8)
\pgfputat{\pgfxy(-.2,6.7)}{\pgfbox[center,center]{$1/f$}}
\pgfputat{\pgfxy(-.2,4.7)}{\pgfbox[center,center]{$1/f$}}
\begin{colormixin}{50!white}
\pgfputat{\pgfxy(2.2,2.25)}{\pgfbox[center,center]{\huge{?}}}
\pgfputat{\pgfxy(1.8,3.35)}{\pgfbox[center,center]{?}}
\pgfputat{\pgfxy(-.2,2.8)}{\pgfbox[center,center]{\large{?}}}
\pgfsetendarrow{\pgfarrowpointed{3pt}}
\pgfxyline(-.9,2.3)(.4,2.3)
\pgfxyline(2.2,3.75)(2.2,2.8)
\end{colormixin}
\end{pgftranslate}
\begin{pgftranslate}{\pgfpoint{13.8cm}{0cm}}
\pgfxyline(1,1.5)(1,7)
\pgfxyline(1,7)(3.4,7)
\pgfxyline(3.4,7)(3.4,1.5)
\pgfxyline(3.4,1.5)(1,1.5)
\pgfputat{\pgfxy(2.2,6.25)}{\pgfbox[center,center]{\large{$\ms{O}_X^+$}}}
\pgfputat{\pgfxy(2.2,4.25)}{\pgfbox[center,center]{\large{$\ms{R}_X^+$}}}
\begin{colormixin}{20!white}
\end{colormixin}
\pgfputat{\pgfxy(-.2,6.7)}{\pgfbox[center,center]{$T_1$}}
\pgfputat{\pgfxy(1.85,5.35)}{\pgfbox[center,center]{$i$}}
\begin{colormixin}{20!white}
\end{colormixin}
\pgfputat{\pgfxy(-.2,4.7)}{\pgfbox[center,center]{$T_1$}}
\begin{colormixin}{20!white}
\end{colormixin}
\pgfputat{\pgfxy(-.2,6.3)}{\pgfbox[center,center]{\large{$\cong$}}}
\pgfputat{\pgfxy(-.2,4.3)}{\pgfbox[center,center]{\large{$\cong$}}}
\pgfsetendarrow{\pgfarrowpointed{3pt}}
\pgfxyline(2.2,5.75)(2.2,4.8)
\begin{colormixin}{50!white}
\pgfputat{\pgfxy(2.2,2.25)}{\pgfbox[center,center]{\huge{?}}}
\pgfputat{\pgfxy(1.8,3.35)}{\pgfbox[center,center]{?}}
\pgfputat{\pgfxy(-.2,2.8)}{\pgfbox[center,center]{\large{?}}}
\pgfsetendarrow{\pgfarrowpointed{3pt}}
\pgfxyline(2.2,3.75)(2.2,2.8)
\pgfputat{\pgfxy(-.2,2.3)}{\pgfbox[center,center]{\large{$\cong$}}}
\end{colormixin}
\end{pgftranslate}
\end{pgftranslate}
\end{pgfmagnify}
\end{pgfpicture}
\caption{Regular structure as the middle row of the coniveau machine for zero-cycles on a smooth projective curve.}
\label{figregularconiveau}
\end{figure}
\vspace*{-.5cm}

In section \hyperref[SubsectionTangentSheafRational]{2.4}, I described informally how a general first-order arc of zero-cycles $z_\ee$ on a smooth projective curve  $X$ may be defined in terms of a family of compatible pairs $\{U_i, f_i+g_i\ee\}$, where $\{U_i\}$ is an open cover of $X$, and where $f_i+g_i\ee$ is viewed as an arc of nonzero rational functions on $U_i$.  I discussed the meaning of compatibility in this context: that the arcs $f_i+g_i\ee$ and $f_j+g_j\ee$ of nonzero rational functions must ``define the same arcs of zero-cycles" on the intersections $U_i\cap U_j$ of their sets of definition.  Having now introduced infinitesimal regular structure, I can now be more precise about this condition: the family of pairs $\{U_i, f_i+g_i\ee\}$ is {\bf compatible} if the ratios $( f_i+g_i\ee)/( f_j+g_j\ee)$ are arcs of invertible regular functions on the intersections $U_i\cap U_j$ of their sets of definition.  These ratios make sense, because the individual arcs are elements of a multiplicative group.  Exploring the consequences of this definition is an important part of section \hyperref[SectionInfZeroCycles]{2.6} below.

\section{Thickened Curve $X_\ee$}\label{SectionThickened}

\subsection{$X_\ee$ as a Locally Ringed Space}\label{subsectionthickenedlocallyringed}

In this short section, I will pause to define a new locally ringed space $X_\ee$, called the {\bf thickened curve} corresponding to the smooth projective curve $X$.  Much of the infinitesimal theory of function groups on $X$, introduced in the previous two sections, may be described in terms of $X_\ee$.  The thickened curve $X_\ee$ has the same underlying topological space as $X$, but a different {\bf structure sheaf}, denoted by $\ms{O}_{X_\ee}$.  The ring of sections of $\ms{O}_{X_\ee}$ over an open set $U$ of $\tn{Zar}_X$ is just the ring of dual numbers $O_{U_\ee}$ over the ring of regular functions $O_U$ on $U$.  Thus, sections of $\ms{O}_{X_\ee}$ are of the form $f+g\ee$, where $f$ and $g$ are regular functions, not necessarily invertible. 

In addition to the new structure sheaf $\ms{O}_{X_\ee}$, there is another new sheaf $R_{X_\ee}$ associated with the thickened curve $X_\ee$, called the {\bf sheaf of total quotient rings} of $X_\ee$.  This new sheaf replaces the sheaf of rational functions on the smooth curve $X$, just as the new structure sheaf replaces the original algebraic structure sheaf on $X$.  The sheaf of total quotient rings $\ms{R}_{X_\ee}$ on $X_\ee$ is a skyscraper sheaf at the generic point of $\tn{Zar}_X$, whose underlying ring is the ring of dual numbers $R_{X_\ee}$ over the function field $R_X$ of $X$.  Thus, its sections are of the form $f+g\ee$, where $f$ and $g$ are rational functions, not necessarily nonzero. 

The notation I have been using for the sheaves $\ms{O}_{X_\ee}^*$ and $\ms{R}_{X_\ee}^*$ of first-order arcs is compatible with the idea that the curve $X$ itself is modified by introducing the nilpotent element $\ee$.   In particular, the sheaf $\ms{O}_{X_\ee}^*$ is  the sheaf of multiplicative groups of invertible sections of the new structure sheaf $\ms{O}_{X_\ee}$.   Its sections, which are first-order arcs of invertible regular functions, are the analogues on $X_\ee$ of invertible regular functions on $X$. Likewise, the sheaf $\ms{R}_{X_\ee}^*$ is the sheaf of multiplicative groups of invertible sections of the sheaf of total quotient rings $\ms{R}_{X_\ee}$.   Its sections, which are first-order arcs of nonzero rational functions, are the analogues on $X_\ee$ of nonzero rational functions on $X$.

It is instructive to compare the relationships among the sheaves $\ms{O}_X$, $\ms{O}_X^*$, $\ms{R}_X$, and $\ms{R}_X^*$, to the relationships among the corresponding sheaves $\ms{O}_{X_\ee}$, $\ms{O}_{X_\ee}^*$, $\ms{R}_{X_\ee}$, and $\ms{R}_{X_\ee}^*$.  Each of the sheaves $\ms{O}_X^*$, $\ms{R}_X$, and $\ms{R}_X^*$ are defined by performing a specific operation on the structure sheaf $\ms{O}_X$ of $X$.   In particular, $\ms{O}_X^*$ is defined by taking subgroups of invertible elements, $\ms{R}_X$  by taking quotient rings, and $\ms{R}_X^*$ by an appropriate composition of these operations.  Performing the same operations on the sheaf $\ms{O}_{X_\ee}$ yields the sheaves $\ms{O}_{X_\ee}^*$, $\ms{R}_{X_\ee}$, and $\ms{R}_{X_\ee}^*$.   This illustrates the naturally of considering $X_\ee=\big(\tn{Zar}_X,\ms{O}_{X_\ee}\big)$ as a locally ringed space in its own right, distinct from $X$.

\subsection{$X_\ee$ as a Singular Scheme}\label{subsectionthickenedsingularscheme}

The thickened curve $X_\ee$, unlike $X$, is not a smooth curve, because its local rings are not regular.  This means that the minimal number of generators of the maximal ideals of the local rings of $X_\ee$ are not equal to their Krull dimensions.  In particular, the local rings $O_{x_\ee}$ of $X_\ee$ at zero-dimensional points $x$ of $\tn{Zar}_X$ all have Krull dimension one, but the maximal ideals of these local rings each require two generators. 

Since $X_\ee$ is not smooth, it is called {\bf singular}.  However, the singularity of $X_\ee$ is of a rather tame and uniform type: $X_\ee$ is given by ``multiplying a nonsingular scheme by a fixed scheme" in the sense of fibered products.  Indeed, $X_\ee$ may be expressed as the fibered product
\[X_\ee:=X\times_{\tn{\footnotesize{Spec }}\CC}\tn{Spec }\CC_\ee,\]
where $\tn{Spec }\CC$ is the smooth one-point spectrum of the complex numbers, and $\tn{Spec }\CC_\ee$ is the singular one-point spectrum of the ring of dual numbers $\CC_\ee$ over $\CC$.  This expression for the thickened curve $X_\ee$ shows that its singular structure arises from the structure of the simplest possible singular scheme $\tn{Spec }\CC_\ee$, and is of essentially the same nature at every point.  Similarly uniformity in the singular structure of the thickenings of higher-dimensional smooth projective varieties will be very useful in analyzing the infinitesimal structure of Chow groups in general.

\section{First-Order Theory of Zero-Cycles \\ on a Smooth Projective Curve}\label{SectionInfZeroCycles}

In this section, I complete the construction of the coniveau machine for zero-cycles on a smooth projective curve $X$.  Three new sheaves, $\ms{Z}_{X_\ee}^1$,  $\ms{Z}_{X_\ee,\ee}^1$, and $\ms{P}_{X}^1$, fill in the bottom row of the machine. These sheaves encode the first-order infinitesimal structure of the sheaf $\ms{Z}_X^1$ of zero-cycles on $X$.  They are defined in terms of the ``rational" sheaves appearing in the middle row of the machine, by taking the ``regular" sheaves in the top row to be trivial in an appropriate sense.  The sheaf $\ms{Z}_{X_\ee}^1$, appearing at the bottom of the second column of the coniveau machine, is the {\bf sheaf of first-order arcs of zero-cycles} on $X$.  It is isomorphic, but not equal, to the quotient sheaf $\ms{R}_{X_\ee}^*/\ms{O}_{X_\ee}^*$.  The sheaf $\ms{Z}_{X_\ee,\ee}^1$, appearing at the bottom of the third column, is called the {\bf relative sheaf} of first-order arcs at the identity in $\ms{Z}_{X}^1$.  It is isomorphic, but not equal, to the quotient sheaf $\ms{R}_{X_\ee,\ee}^*/\ms{O}_{X_\ee,\ee}^*$.  Finally, the sheaf $\ms{P}_{X}^1$, appearing at the bottom of the fourth column, is the {\bf sheaf of principal parts} of rational functions on $X$. This sheaf is the {\bf tangent sheaf} of the sheaf $\ms{Z}_{X}^1$ of zero-cycles on $X$.   It is isomorphic, but not equal, to the quotient sheaf $\ms{R}_{X}^+/\ms{O}_{X}^+$.

\subsection{First-Order Arcs of Zero-Cycles}\label{subsectionfirstorderzerocycles}

At the beginning of section \hyperref[SectionInfRational]{2.4}, I mentioned that ``arcs of zero-cycles" on a smooth projective curve $X$ may be defined locally on $X$ as ``divisors of arcs of nonzero rational functions."  At the end of section \hyperref[SectionInfRational]{2.4}, I explained why families of compatible pairs $\{U_i,f_i+g_i\ee\}$ are necessary to define an arbitrary first-order arc of zero-cycles on an open subset $U$ of $X$.  At the beginning of section \hyperref[SectionInfRegular]{2.5}, I remarked that ``regular structure" should be viewed as the trivial part of ``rational structure" in the context of zero-cycles on $X$.  Finally, at the end of section \hyperref[SectionInfRegular]{2.5}, I explained the meaning of triviality of regular structure in terms of compatible pairs.   This provides enough background to precisely define first-order arcs of zero-cycles.

Let $U$ be an open subset of $X$, and let $\{U_i,f_i+g_i\ee\}$ be a family of compatible pairs on $U$.  Recall that $\{U_i\}$ is an open cover of $U$, and $f_i+g_i\ee$ is a first-order arc of nonzero rational functions on $U_i$.  Compatibility means that the ratios $(f_i+g_i\ee)/(f_j+g_j\ee)$ are first-order arcs of invertible regular functions on the intersections $U_{ij}:=U_{i}\cap U_{ j}$. The ratio $(f_i+g_i\ee)/(f_j+g_j\ee)$ may be factored as
\[\frac{f_i+g_i\ee}{f_j+g_j\ee}\hspace*{.2cm}=\hspace*{.2cm}\frac{f_i}{f_j}\Bigg(\frac{1+(g_i/f_i)\ee}{1+(g_j/f_j)\ee}\Bigg).\]
This expression defines a first-order arc of invertible regular functions on $U_{ij}$ if and only if the first factor $f_i/f_j$ is an invertible regular function on $U_{ij}$, and the second factor $\big(1+(g_i/f_i)\ee\big)\big/\big(1+(g_i/f_i)\ee\big)$ is a first-order arc of invertible regular functions at the identity in $O_{U_{ij}}^*$.    These conditions, applied to every pair of indices, may be viewed as separate conditions on the two families of pairs $\{U_i,f_i\}$ and $\big\{U_i,1+(g_i/f_i)\ee\big\}$. 

The first family of pairs $\{U_i,f_i\}$ defines a zero-cycle on $U$, whose multiplicities are given by the valuations $\nu_x(f_i)$ at the codimension-one points $x$ of $U$.   The condition that the ratios $f_i/f_j$ are invertible regular functions on the intersections $U_{ij}$ ensures that this zero-cycle is well-defined, since this means that the valuations of $f_i$ and $f_j$ coincide on $U_{ij}$.   In particular, the only information lost by $\nu_x$ is information about regular structure.  The second family of pairs $\big\{U_i,1+(g_i/f_i)\ee\big\}$ may also be interpreted in terms of information associated with codimension-one points of $U$.  For each $x\in \tn{Zar}_U^{1}$, consider the quotient map $\pi_x^*$ from $R_{X_\ee,\ee}^*$ to the quotient group $R_{X_\ee,\ee}^*/O_{x_\ee,\ee}^*$.   This map associates information from the family of pairs $\big\{U_i,1+(g_i/f_i)\ee\big\}$ with $x$.  The condition that the ratios $\big(1+(g_i/f_i)\ee\big)\big/\big(1+(g_i/f_i)\ee\big)$ are arcs of invertible regular functions ensures that this information is well-defined; again, the only information lost is information about regular structure.  

For each codimension-one point $x$ of $U$, the information from the family $\{U_i,f_i+g_i\ee\}$ associated with $x$ is the pair of values $\Big(\big(\nu_x(f_i)\big)x,\pi_x^*\big(1+(g_i/f_i)\ee\big)\Big)$.  This pair of values belongs to the group $Z_{x_\ee}^1:=Z_x^1\oplus \big(R_{X_\ee,\ee}^*/O_{x_\ee,\ee}^*\big)$, which is called the {\bf group of first-order arcs of zero-cycles} at $x$.  A general {\bf first-order arc of zero-cycles} at $x$ is any element $(n_xx,h_x)$ of this group.  Here, $n_x$ is an integer, and $h_x=[1+g_x\ee]$ is an element of the quotient group $R_{X_\ee,\ee}^*/O_{x_\ee,\ee}^*$, where the notation $[1+g_x\ee]$ means the class in the quotient group of a first-order arc of nonzero rational functions $1+g_x\ee$.   It is easy to see that every such pair $(n_xx,h_x)$ may in fact be written in the form $\Big(\big(\nu_x(f)\big)x,\pi_x^*\big(1+(g/f)\ee\big)\Big)$ for an appropriate choice of $f$ and $g$.    The ``initial zero-cycle" of the first-order arc $(n_xx,h_x)$ is the zero-cycle $n_xx$.  Since the group operation in the first summand $Z_x^1$ of $Z_{x_\ee}^1$ is addition, the arc arc $(n_xx,h_x)$ may be moved to the identity $0$ in $Z_x^1$ by subtracting $nx$.  The resulting {\bf arc of zero-cycles at the identity} in $Z_x^1$ may be identified with the element $h_x$ in the quotient group $R_{X_\ee,\ee}^*/O_{x_\ee,\ee}^*$.   I therefore relabel the quotient group as $Z_{x_\ee,\ee}^1$, and call it the {\bf relative group} of first-order arcs at the identity in $Z_x^1$.  

Now let $x$ range over all codimension-one points of $U$.  The group $Z_{U_\ee}^1:=\bigoplus_{x\in\tn{\footnotesize{Zar}}_U^{1}}Z_{x_\ee}^1$ is called the {\bf group of first-order arcs of zero-cycles} on $U$, and its elements are called {\bf first-order arcs of zero-cycles} on $U$.  In particular, the element $\sum_{x\in\tn{\footnotesize{Zar}}_U^{1}}\Big(\nu_x(f_i),\pi_x^*\big(1+(g_i/f_i)\ee\big)\Big)$ is defined to be the first-order arc of zero-cycles on $U$ associated with the family of compatible pairs $\{U_i,f_i+g_i\ee\}$.   A general element of $Z_{U_\ee}^1$ may be expressed as a formal sum $\sum_{x\in\tn{\footnotesize{Zar}}_U^{1}}(n_xx,h_x),$ where each $n_x$ is an integer, only a finite number of which are nonzero, and where each $h_x$ is an element of the relative group $Z_{x_\ee,\ee}^1$.   However, it is easy to see that every element of $Z_{U_\ee}^1$ in fact comes from a family of compatible pairs.   The ``initial zero-cycle" of the first-order arc $z_\ee:=\sum_{x\in\tn{\footnotesize{Zar}}_U^{1}}(n_xx,h_x)$ is the zero-cycle $z:=\sum_{x\in\tn{\footnotesize{Zar}}_U^{1}}n_xx$.  Therefore, the arc $z_\ee$ may be moved to the identity $0$ in $Z_U^1$ by subtracting $z$.  The resulting {\bf arc of zero-cycles at the identity} in $Z_U^1$ may be identified with the element $\sum_{x\in\tn{\footnotesize{Zar}}_U^{1}}h_x$ in the group $=\bigoplus_{x\in\tn{\footnotesize{Zar}}_U^{1}}Z_{x_\ee,\ee}^1$.   I therefore relabel this group as $Z_{U_\ee,\ee}^1$, and call it the {\bf relative group} of first-order arcs at the identity in $Z_U^1$.

\subsection{Tangent Elements, Sets, Maps, and Groups}\label{subsectionzerocyclestangents}

Since the integers $n_x$ appearing in the formula $z_\ee:=\sum_{x\in\tn{\footnotesize{Zar}}_U^{1}}(n_xx,h_x)$ for an arc of zero-cycles on an open subset $U$ of a smooth projective curve $X$ define the initial zero-cycle $z:=\sum_{x\in\tn{\footnotesize{Zar}}_U^{1}}n_xx$ of the arc, it is natural to expect that the elements $h_x$ in the same formula should define the ``tangent element of the arc."  This is the right idea, but it is useful to express tangent elements in a slightly different form.   The cases of nonzero rational functions and invertible regular functions provide guidance here: the tangent element of a first-order arc contains the same information as the corresponding arc at the identity, but expressed additively rather than multiplicatively.   This suggests  expressing the information contained in the first-order arc at the identity $\sum_{x\in\tn{\footnotesize{Zar}}_U^{1}}h_x$ in additive form. 

For this purpose, it is useful to revisit the compatibility conditions for a family of compatible pairs $\{U_i,f_i+g_i\ee\}$ above.  Another way to express the ratio $(f_i+g_i\ee)/(f_j+g_j\ee)$ is as follows:
\[\frac{f_i+g_i\ee}{f_j+g_j\ee}\hspace*{.2cm}=\hspace*{.2cm}\frac{f_i}{f_j}\Bigg(1+\Big(\frac{g_i}{f_i}-\frac{g_j}{f_j}\Big)\ee\Bigg).\]
The compatibility conditions now require that the first factor $f_i/f_j$ must be an invertible regular function on $U_{ij}$, and that the coefficient of $\ee$ in the second factor must be a regular function on $U_{ij}$.   The first condition is the same as before, but the second condition gives a new perspective, leading to the desired additive definition of tangent elements. 

The ratios $g_i/f_i$ and $g_j/f_j$, whose difference appears as the coefficient of $\ee$ in the above ratio, are the tangent vectors at the identity in $R_X$ of the first-order arcs $f_i+g_i\ee$ and $f_j+g_j\ee$ of nonzero rational functions.   Heuristically, the compatibility conditions say that ``an arc of zero-cycles on $X$ is defined by a family of arcs of nonzero rational functions whose initial functions have regular ratios and whose tangent vectors have regular differences."  The second compatibility condition implies that the family of pairs $\{U_i,g_i/f_i\}$ encode ``rational structure modulo regular structure" in an additive sense, since the differences (rather than the ratios) of $g_i/f_i$ and $g_j/f_j$ are regular. 

The family of pairs $\{U_i,g_i/f_i\}$ may now be interpreted in terms of information associated with codimension-one points of $X$.  This information is the same as the information encoded in the family $\big\{U_i,1+(g_i/f_i)\ee\big\}$, but is additive rather than multiplicative. For each $x\in \tn{Zar}_U^{1}$, consider the quotient map $\pi_x^+$ from $R_{X}^+$ to the quotient group $R_{X}^+/O_{x}^+$.   This map associates information from the family of pairs $\{U_i,g_i/f_i\}$ with $x$.   The condition that the differences $g_i/f_i-g_j/f_j$ are regular functions ensures that this information is well-defined, and the only information lost is information about regular structure. For each codimension-one point $x$ of $U$, the information from the family $\{U_i,f_i+g_i\ee\}$ associated with $x$ may be alternatively expressed as the pair of values $\Big(\big(\nu_x(f_i)\big)x,\pi_x^+(g_i/f_i)\Big)$, belonging to the group $Z_x^1\oplus \big(R_{X}^+/O_{x}^+\big)$. This group is isomorphic to $Z_{x_\ee}^1$, but the information in the second summand is expressed additively. 

For each codimension-one point $x$ of the smooth projective curve $X$, I will relabel the group $R_{X}^+/O_{x}^+$ as $P_x^1$.  Here the notation ``$P$" stands for ``principal part," as explained below.  For the pointwise zero-cycle $nx$ in the stalk $Z_x^1$, the underlying set $P_x^{1,\tn{\footnotesize{set}}}$ of $P_x^1$ is the {\bf tangent set} at $nx$ of $Z_x^1$, with one copy of $P_x^{1,\tn{\footnotesize{set}}}$ assigned to each integer multiple of $x$.   At the identity $0$ of $Z_x^1$, the additive operations group of $Z_x^1$ and $P_x^1$ are compatible, so $P_x^1$ is the {\bf tangent group} of $Z_x^1$.  The {\bf tangent map $T_{x,nx}$ at $nx$} in $Z_x^1$ carries the first-order arc of zero-cycles $\big(nx,[1+g\ee]\big)$ to its {\bf tangent element} $\pi_x^+(g)$, where as above the notation $[1+g\ee]$ means the class in the quotient group of a first-order arc of nonzero rational functions $1+g\ee$.  The tangent map transforms multiplication to addition.  In particular, the relative group $Z_{x_\ee,\ee}^1=R_{X_\ee,\ee}^*/O_{x_\ee,\ee}^*$ of first-order arcs of zero-cycles at the identity in $Z_x^1$ is isomorphic to the tangent group $P_x^1=R_{X}^+/O_{x}^+$ via the tangent map $T_{x,0}$ at the identity $0$ in $Z_x^1$.  The tangent maps $T_{x,nx}$ at the multiples $nx$ of $x$ fit together to give a single {\bf tangent map} $T_x:Z_{x_\ee}^1\rightarrow P_x^1$, which is the composition of the tangent map $T_{x,0}$ at the identity with the second projection $\tn{proj}_2:Z_{x_\ee}^1\rightarrow Z_{x_\ee,\ee}^1$.  

Allowing $x$ to range over all codimension-one points of an open subset $U$ of $X$, I will relabel the group $\bigoplus_{x\in\tn{\footnotesize{Zar}}_U^{1}}\big(R_{X}^+/O_{x}^+\big)=\bigoplus_{x\in\tn{\footnotesize{Zar}}_U^{1}}P_x^1$ as $P_U^1$.  For any zero-cycle $z=\sum_in_ix_i$ in $Z_U^1$, the underlying set $P_U^{1,\tn{\footnotesize{set}}}$ of $P_U^1$ is the {\bf tangent set} at $z$ of $Z_U^1$, one copy of $P_U^{1,\tn{\footnotesize{set}}}$ for each $z\in Z_U^1$.  The operations in $Z_U^1$, and $P_U^1$ are compatible at the identity of $Z_U^1$, so $P_U^1$ is the {\bf tangent group} of $Z_U^1$.  The {\bf tangent map} $T_{U,z}$ at $z$ in $Z_U^1$ carries the first-order arc of zero-cycles $\sum_{x\in\tn{\footnotesize{Zar}}_U^{1}}(n_ix_i,[1+g_i\ee])$ to its {\bf tangent element} $\sum_{x\in\tn{\footnotesize{Zar}}_U^{1}}\pi_x^+(g_i)$, transforming multiplication to addition.  The relative group $Z_{U_\ee,\ee}^1$ is isomorphic to the tangent group $P_U^1$ via the tangent map $T_{U,0}$ at the identity in $Z_U^1$.   The case $U=X$ gives a tangent set $P_X^{1,\tn{\footnotesize{set}}}$ and tangent maps $T_{X,z}$ at each zero-cycle $z$ on $X$, and a tangent group  $P_X^{1}$ at the identity in $Z_X^1$.  The relative group $Z_{X_\ee,\ee}^1$ is isomorphic to the tangent group $P_X^1$ via the tangent map $T_{X,0}$ at the identity in $Z_X^1$.  The tangent maps $T_{U,z}$ at the various zero-cycles $z$ in $U$ fit together to give a single {\bf tangent map} $T_U:Z_{U_\ee}^1\rightarrow P_U^1$, which is the composition of the tangent map $T_{U,0}$ at the identity with the second projection $\tn{proj}_2:Z_{U_\ee}^1\rightarrow Z_{U_\ee,\ee}^1$. 

The inclusion maps $i:Z_x^1\rightarrow Z_{x_\ee}^1$ and $i:Z_U^1\rightarrow Z_{U_\ee}^1$ fit together with the projection maps and tangent maps at the identity to give four-term additive exact sequences:
\[0\rightarrow Z_x^1\overset{i}{\longrightarrow} Z_{x_\ee}^1\overset{\tn{\footnotesize{proj}}_2}{\longrightarrow}Z_{x_\ee,\ee}^1 \overset{T_{x,0}}{\longrightarrow} P_x^1\rightarrow 0,\]
and
\[0\rightarrow Z_U^1\overset{i}{\longrightarrow} Z_{U_\ee}^1\overset{\tn{\footnotesize{proj}}_2}{\longrightarrow}Z_{U_\ee,\ee}^1 \overset{T_{x,0}}{\longrightarrow} P_U^1\rightarrow 0.\]
The reason why these are not hybrid exact sequences will have to wait until I recast these sequences in terms of algebraic $K$-theory and negative cyclic homology in Chapters 5 and 7.

\subsection{Sheaves of First-Order Arcs and Tangents; Sheaf Tangent Map}\label{subsectionzerocyclesheaves}

The above four-term additive exact sequences involving $U$ fit together to give a four-term additive exact sequence of sheaves on $X$:
\[0\rightarrow\ms{Z}_X^1\overset{i}{\longrightarrow}\ms{Z}_{X_\ee}^1\overset{\tn{\footnotesize{proj}}_2}{\longrightarrow}\ms{Z}_{X_\ee,\ee}^1\overset{T_0}{\longrightarrow}\ms{P}_X^1\rightarrow 0,\]
where the three new sheaves $\ms{Z}_{X_\ee}^1$, $\ms{Z}_{X_\ee,\ee}^1$, and $\ms{P}_X^1$ are sheaves with groups of sections $Z_{U_\ee}^1$, $Z_{U_\ee,\ee}^1$, and $P_U^1$, respectively, over an open set $U$ of $X$.  The stalks of these sheaves at a point $x$ in $X$ are $Z_{x_\ee}^1$, $Z_{x_\ee,\ee}^1$, and $P_x^1$, respectively.  The three new sheaves are canonically isomorphic to direct sums of skyscraper sheaves over the codimension-one points of $X$:
\[\ms{Z}_{X_\ee}^1=\bigoplus_{x\in \tn{\footnotesize{Zar}}_X^{1}}\underline{Z_{x_\ee}^1},\hspace*{.5cm}\ms{Z}_{X_\ee,\ee}^1=\bigoplus_{x\in \tn{\footnotesize{Zar}}_X^{1}}\underline{Z_{x_\ee,\ee}^1},\hspace*{.5cm}\tn{and}\hspace*{.5cm}\ms{P}_X^1=\bigoplus_{x\in \tn{\footnotesize{Zar}}_X^{1}}\underline{P_x^1}.\]
The maps $i$, $\tn{proj}_2,$ and $T_0$ at the sheaf level are defined by patching together copies of the corresponding group maps for each open set $U$.   Figure \hyperref[figcompletedconiveaucurve]{\ref{figcompletedconiveaucurve}} shows the three new sheaves $\ms{Z}_{X_\ee}^1$,  $\ms{Z}_{X_\ee,\ee}^1$, and $\ms{P}_{X}^1$ in their proper locations in the coniveau machine.  This completes the construction of the coniveau machine for zero-cycles on a smooth projective curve.

\begin{figure}[H]
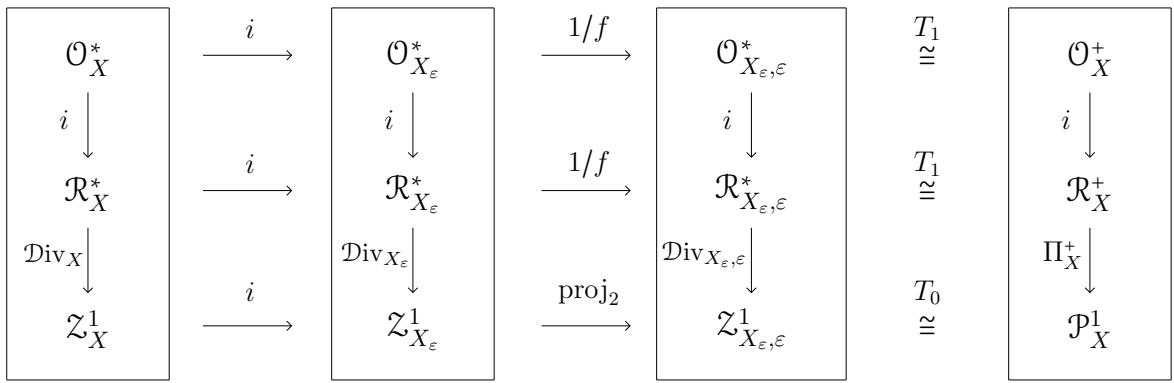

\begin{pgfpicture}{0cm}{0cm}{17cm}{5.2cm}
\begin{pgfmagnify}{.9}{.9}
\begin{pgftranslate}{\pgfpoint{.2cm}{-1.3cm}}
\begin{pgftranslate}{\pgfpoint{-1cm}{0cm}}
\pgfxyline(1,1.5)(1,7)
\pgfxyline(1,7)(3.4,7)
\pgfxyline(3.4,7)(3.4,1.5)
\pgfxyline(3.4,1.5)(1,1.5)
\pgfputat{\pgfxy(2.2,6.25)}{\pgfbox[center,center]{\large{$\ms{O}_X^*$}}}
\pgfputat{\pgfxy(2.2,4.25)}{\pgfbox[center,center]{\large{$\ms{R}^*_X$}}}
\pgfputat{\pgfxy(2.2,2.25)}{\pgfbox[center,center]{\large{$\ms{Z}_X^1$}}}
\pgfputat{\pgfxy(1.85,5.35)}{\pgfbox[center,center]{$i$}}
\pgfputat{\pgfxy(1.7,3.35)}{\pgfbox[center,center]{\small{$\ms{D}\tn{iv}_X$}}}
\pgfsetendarrow{\pgfarrowpointed{3pt}}
\pgfxyline(2.2,5.75)(2.2,4.8)
\pgfxyline(2.2,3.75)(2.2,2.8)
\end{pgftranslate}
\begin{pgftranslate}{\pgfpoint{3.8cm}{0cm}}
\pgfxyline(1,1.5)(1,7)
\pgfxyline(1,7)(3.4,7)
\pgfxyline(3.4,7)(3.4,1.5)
\pgfxyline(3.4,1.5)(1,1.5)
\pgfputat{\pgfxy(2.2,6.25)}{\pgfbox[center,center]{\large{$\ms{O}^*_{X_\ee}$}}}
\pgfputat{\pgfxy(2.2,4.25)}{\pgfbox[center,center]{\large{$\ms{R}^*_{X_\ee}$}}}
\pgfputat{\pgfxy(1.85,5.35)}{\pgfbox[center,center]{$i$}}
\pgfsetendarrow{\pgfarrowpointed{3pt}}
\pgfxyline(-.9,6.3)(.4,6.3)
\pgfxyline(-.9,4.3)(.4,4.3)
\pgfxyline(2.2,5.75)(2.2,4.8)
\pgfputat{\pgfxy(-.2,6.7)}{\pgfbox[center,center]{$i$}}
\pgfputat{\pgfxy(-.2,4.7)}{\pgfbox[center,center]{$i$}}
\pgfputat{\pgfxy(2.2,2.25)}{\pgfbox[center,center]{\large{$\ms{Z}^1_{X_\ee}$}}}
\pgfputat{\pgfxy(1.65,3.35)}{\pgfbox[center,center]{\small{$\ms{D}\tn{iv}_{X_\ee}$}}}
\pgfputat{\pgfxy(-.2,2.8)}{\pgfbox[center,center]{$i$}}
\pgfsetendarrow{\pgfarrowpointed{3pt}}
\pgfxyline(-.9,2.3)(.4,2.3)
\pgfxyline(2.2,3.75)(2.2,2.8)
\end{pgftranslate}
\begin{pgftranslate}{\pgfpoint{8.8cm}{0cm}}
\pgfxyline(.8,1.5)(.8,7)
\pgfxyline(.8,7)(3.6,7)
\pgfxyline(3.6,7)(3.6,1.5)
\pgfxyline(3.6,1.5)(.8,1.5)
\pgfputat{\pgfxy(2.2,6.25)}{\pgfbox[center,center]{\large{$\ms{O}^*_{X_\ee,\ee}$}}}
\pgfputat{\pgfxy(2.2,4.25)}{\pgfbox[center,center]{\large{$\ms{R}^*_{X_\ee,\ee}$}}}
\pgfputat{\pgfxy(1.85,5.35)}{\pgfbox[center,center]{$i$}}
\pgfsetendarrow{\pgfarrowtriangle{6pt}}
\pgfsetendarrow{\pgfarrowpointed{3pt}}
\pgfxyline(-.9,6.3)(.4,6.3)
\pgfxyline(-.9,4.3)(.4,4.3)
\pgfxyline(2.2,5.75)(2.2,4.8)
\pgfputat{\pgfxy(-.2,6.7)}{\pgfbox[center,center]{$1/f$}}
\pgfputat{\pgfxy(-.2,4.7)}{\pgfbox[center,center]{$1/f$}}
\pgfputat{\pgfxy(2.2,2.25)}{\pgfbox[center,center]{\large{$\ms{Z}^1_{X_\ee,\ee}$}}}
\pgfputat{\pgfxy(1.5,3.35)}{\pgfbox[center,center]{\small{$\ms{D}\tn{iv}_{X_\ee,\ee}$}}}
\pgfputat{\pgfxy(-.2,2.8)}{\pgfbox[center,center]{$\tn{proj}_2$}}
\pgfsetendarrow{\pgfarrowpointed{3pt}}
\pgfxyline(-.9,2.3)(.4,2.3)
\pgfxyline(2.2,3.75)(2.2,2.8)
\end{pgftranslate}
\begin{pgftranslate}{\pgfpoint{13.8cm}{0cm}}
\pgfxyline(1,1.5)(1,7)
\pgfxyline(1,7)(3.4,7)
\pgfxyline(3.4,7)(3.4,1.5)
\pgfxyline(3.4,1.5)(1,1.5)
\pgfputat{\pgfxy(2.2,6.25)}{\pgfbox[center,center]{\large{$\ms{O}_X^+$}}}
\pgfputat{\pgfxy(2.2,4.25)}{\pgfbox[center,center]{\large{$\ms{R}_X^+$}}}
\pgfputat{\pgfxy(-.2,6.7)}{\pgfbox[center,center]{$T_1$}}
\pgfputat{\pgfxy(1.85,5.35)}{\pgfbox[center,center]{$i$}}
\pgfputat{\pgfxy(-.2,4.7)}{\pgfbox[center,center]{$T_1$}}
\pgfputat{\pgfxy(-.2,6.3)}{\pgfbox[center,center]{\large{$\cong$}}}
\pgfputat{\pgfxy(-.2,4.3)}{\pgfbox[center,center]{\large{$\cong$}}}
\pgfsetendarrow{\pgfarrowpointed{3pt}}
\pgfxyline(2.2,5.75)(2.2,4.8)
\pgfputat{\pgfxy(2.2,2.25)}{\pgfbox[center,center]{\large{$\ms{P}_X^1$}}}
\pgfputat{\pgfxy(1.8,3.35)}{\pgfbox[center,center]{$\Pi_X^+$}}
\pgfputat{\pgfxy(-.2,2.8)}{\pgfbox[center,center]{$T_0$}}
\pgfsetendarrow{\pgfarrowpointed{3pt}}
\pgfxyline(2.2,3.75)(2.2,2.8)
\pgfputat{\pgfxy(-.2,2.3)}{\pgfbox[center,center]{\large{$\cong$}}}
\end{pgftranslate}
\end{pgftranslate}
\end{pgfmagnify}
\end{pgfpicture}
\caption{Completed coniveau machine for zero-cycles on a smooth complex projective curve.}
\label{figcompletedconiveaucurve}
\end{figure}

\section{First-Order Theory of the First Chow Group $\tn{Ch}_X^1$}\label{SectionInfChow}

In this section, I show how the coniveau machine for zero-cycles on a smooth complex projective curve $X$ may be used to study the first-order infinitesimal theory of the first Chow group $\tn{Ch}_X^1$ of $X$.   This case is atypical in the sense that the objects involved may be approached by well-understood traditional methods; in particular, algebraic Lie theory.  


\subsection{Tangent Group $\tn{Ch}_X^1$ according to Algebraic Lie Theory}\label{subsectionTCH1Lie}

The relationships among the Jacobian $J_X$, the Picard group $\tn{Pic}_X$, the Picard variety $\tn{Pic}_X^0$, and the first Chow group $\tn{Ch}_X^1$ of a smooth complex projective curve $X$, may be used to identify the tangent group at the identity $T\tn{Ch}_X^1$ of $\tn{Ch}_X^1$ from the perspective of algebraic Lie theory.  The same reasoning applies to the codimension-one case in general; for example, to curves on a surface.  However, it does not extend to higher codimension.  The crucial point in the codimension-one case is that $\tn{Ch}_{X}^1$ coincides with $\tn{Pic}_X$ which is an algebraic group.  Tangent groups at the identity of algebraic groups are well-defined and well-understood, being analogous to Lie algebras of Lie groups.   In higher codimensions, Chow groups are generally not algebraic groups, and even {\it defining} their tangent groups at the identity is a nontrivial issue. 

Using the fact that $\tn{Pic}_X$ is an algebraic group, the computation of the tangent group at the identity $T\tn{Ch}_X^1$ of $\tn{Ch}_X^1$ may be approached in the following way: since $J_X=\tn{Pic}_X^0$ is the connected component of the identity in $\tn{Pic}_X=\tn{Ch}_{X}^1$, the tangent group $T\tn{Ch}_X^1$ coincides with the tangent group at the identity of the torus $J_X$.  The latter tangent group coincides with $H_{X}^{0,1}$ by equation \hyperref[equcurvejacobian]{\ref{equcurvejacobian}}.\footnotemark\footnotetext{In particular, one may ignore the lattice $H_{X,\ZZ}^1$ in the quotient.  The same reasoning applies, for example, in concluding that the tangent group at the identity of the two-dimensional real torus $T=\RR^2/\ZZ^2$ is $\RR^2$.}  Applying the Dobeault isomorphism theorem,\footnotemark\footnotetext{This theorem relates the sheaf cohomology groups $H^q(X,\varOmega_{X,\tn{an}}^p)$ of the sheaves $\varOmega_{X,\tn{an}}^p$ of holomorphic $p$-forms on a complex manifold $X$ to the Dolbeault cohomology groups $H_{\overline{\partial},X}^{p,q}$, which coincide with the Hodge factors $H_X^{p,q}$ when $X$ is a K\"{a}hler manifold.  See Griffiths and Harris \cite{GriffithsHarris} chapter $0$ for details.} together with Serre's GAGA results,\footnotemark\footnotetext{``GAGA" is short for the title of Serre's paper {\it G\'{e}om\'{e}trie alg\'{e}brique et g\'{e}om\'{e}trie analytique} \cite{SerreGAGA}.  This paper demonstrates remarkable similarities, and in important special cases equivalences, between the algebraic category and the {\it a priori} much less rigid analytic category.}
\begin{equation}\label{equTCh1Dolbeault}T\tn{Ch}_X^1=H_{X}^{0,1}\overset{\tn{Dolbeault}}{\cong} H_{\tn{\fsz{an}}}^1(X,\varOmega_{X,\tn{an}}^0)=H_{\tn{\fsz{an}}}^1(X,\ms{O}_{X,\tn{an}})\overset{\tn{GAGA}}{\cong} H_{\tn{\fsz{Zar}}}^1(X,\ms{O}_{X}),\end{equation}
where the subscript ``an" refers to the analytic category.  


\subsection{Tangent Group $\tn{Ch}_X^1$ according to the Coniveau Machine}\label{subsectionTCH1Coniveau}

In section \hyperref[subsectionlinratfirstchow]{\ref{subsectionlinratfirstchow}} above, I defined the first Chow group $\tn{Ch}_X^1$ of $X$ as the quotient group $Z_X^1/Z_{X,\footnotesize{\tn{rat}}}^1$ of the group $Z_X^1$ of global zero-cycles on $X$ by the subgroup $Z_{X,\footnotesize{\tn{rat}}}^1$ of cycles rationally equivalent to zero.   In equation \hyperref[equCH1div]{\ref{equCH1div}}, I re-expressed the subgroup $Z_{X,\footnotesize{\tn{rat}}}^1$ as the image of the global divisor map $\tn{div}_X$ sending the group of rational functions $R_X^*$ on $X$ to $Z_X^1$.  This map arises when one applies the global sections functor $\Gamma$ to the sheaf divisor sequence in equation \hyperref[equsheafdivisorsequencecurves]{\ref{equsheafdivisorsequencecurves}}:
\[1\rightarrow\ms{O}_X^*\overset{i}{\longrightarrow}\ms{R}_X^*\overset{\ms{D}\tn{\footnotesize{iv}}_X}{\longrightarrow}\ms{Z}_X^1\rightarrow0,\]
to obtain the left-exact global divisor sequence of equation \hyperref[equdivisorglobalsectionscurves]{\ref{equdivisorglobalsectionscurves}}:
\[0\rightarrow O_X^*\overset{i}{\longrightarrow}R_X^*\overset{\tn{\footnotesize{div}}_X}{\longrightarrow}Z_X^1.\]
The cohomology groups of this sequence are the Zariski sheaf cohomology groups of the sheaf $\ms{O}_X^*$ of invertible regular functions on $X$ by the definition of sheaf cohomology.  In particular, $\tn{Ch}_X^1$ is the first Zariski sheaf cohomology group of $\ms{O}_X^*$:
\begin{equation}\label{equCH1curvesheafcohom}\tn{Ch}_X^1=\frac{Z_X^1}{\tn{Im \big(div}_X\big)}=\frac{\Gamma(\ms{Z}_X^1)}{\tn{Im} \big(\Gamma(\ms{D}\tn{iv}_X)\big)}=H_{\tn{\fsz{Zar}}}^1(X,\ms{O}_X^*).\end{equation}
This computation proves the statement $\tn{Ch}_X^1=\tn{Pic}_X$ stated in section \hyperref[subsectionJacPic]{\ref{subsectionJacPic}} above, using the sheaf-cohomology definition of $\tn{Pic}_X$.  Applying similar reasoning to the tangent sequence  
\[0\rightarrow\ms{O}_X^+\overset{i}{\longrightarrow}\ms{R}_X^+\overset{\ms{D}\tn{\footnotesize{iv}}_X}{\longrightarrow}\ms{P}_X^1\rightarrow0,\]
appearing in the fourth column of the coniveau machine in figure \hyperref[figcompletedconiveaucurve]{\ref{figcompletedconiveaucurve}} above, yields the computation
\begin{equation}\label{equTCH1curvesheafcohom}T\tn{Ch}_X^1=\frac{\Gamma(\ms{P}_X^1)}{\tn{Im} \big(\Gamma(\Pi_X^+)\big)}=H_{\tn{\fsz{Zar}}}^1(X,\ms{O}_X).\end{equation}
This is the same result obtained from the algebraic Lie perspective in section \hyperref[subsectionTCH1Lie]{\ref{subsectionTCH1Lie}}.   Hence, the coniveau machine gives the ``right answer" in the case of smooth complex projective curves. 



\subsection{Looking Ahead: Abstract Form of the Coniveau Machine}\label{subsectionabstractconiveauCH1}

While the computations of $\tn{Ch}_X^1$ and $T\tn{Ch}_X^1$ in terms of the Jacobian $J_X$ and the Picard group $\tn{Pic}_X$ do not generalize well to the general infinitesimal theory of Chow groups, the coniveau machine does generalize, provided it is expressed in a suitable way.  The rather complicated and unfamiliar-looking diagram in figure \hyperref[figabstractconiveaucurve]{\ref{figabstractconiveaucurve}} below gives the abstract form of the coniveau machine that generalizes to describe the infinitesimal structure of cycle groups and Chow groups of arbitrary codimension on arbitrary smooth complex varieties.  

\begin{figure}[H]
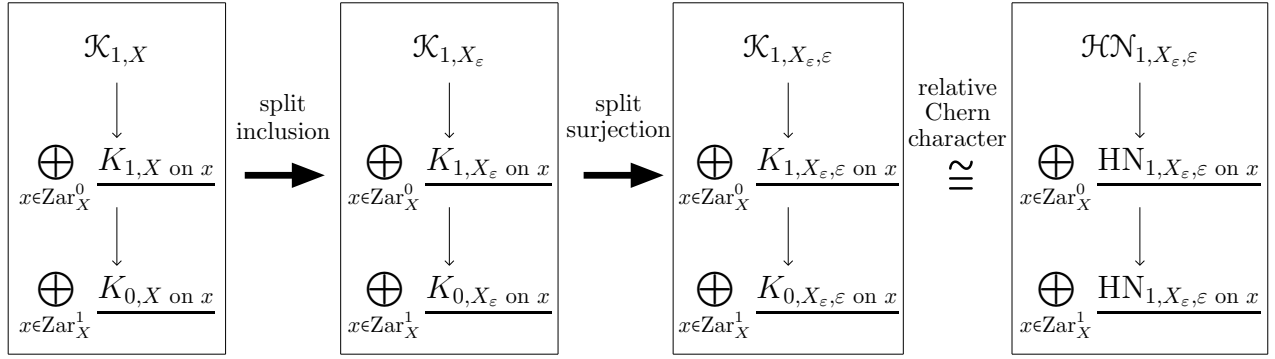

\begin{pgfpicture}{0cm}{0cm}{17cm}{5.3cm}
\begin{pgfmagnify}{.85}{.85}
\begin{pgftranslate}{\pgfpoint{.5cm}{-1.3cm}}
\begin{pgftranslate}{\pgfpoint{-1cm}{0cm}}
\pgfxyline(.5,1.5)(.5,7)
\pgfxyline(.5,7)(3.9,7)
\pgfxyline(3.9,7)(3.9,1.5)
\pgfxyline(3.9,1.5)(.5,1.5)
\pgfputat{\pgfxy(2.2,6.25)}{\pgfbox[center,center]{\large{$\ms{K}_{1,X}$}}}
\pgfputat{\pgfxy(2.2,4.25)}{\pgfbox[center,center]{\large{$\displaystyle\bigoplus_{x\in \tn{Zar}_X^0}\underline{K_{1,X\tn{ on } x}}$}}}
\pgfputat{\pgfxy(2.2,2.25)}{\pgfbox[center,center]{\large{$\displaystyle\bigoplus_{x\in \tn{Zar}_X^1}\underline{K_{0,X\tn{ on } x}}$}}}
\pgfsetendarrow{\pgfarrowpointed{3pt}}
\pgfxyline(2.2,5.75)(2.2,4.9)
\pgfxyline(2.2,3.8)(2.2,2.9)
\end{pgftranslate}
\begin{pgftranslate}{\pgfpoint{4.2cm}{0cm}}
\begin{pgfscope}
\pgfsetendarrow{\pgfarrowtriangle{6pt}}
\pgfsetlinewidth{3pt}
\pgfxyline(-1,4.3)(0,4.3)
\pgfputat{\pgfxy(-.4,5.4)}{\pgfbox[center,center]{\small{split}}}
\pgfputat{\pgfxy(-.4,5)}{\pgfbox[center,center]{\small{inclusion}}}
\end{pgfscope}
\pgfxyline(.5,1.5)(.5,7)
\pgfxyline(.5,7)(3.9,7)
\pgfxyline(3.9,7)(3.9,1.5)
\pgfxyline(3.9,1.5)(.5,1.5)
\pgfputat{\pgfxy(2.2,6.25)}{\pgfbox[center,center]{\large{$\ms{K}_{1,X_\ee}$}}}
\pgfputat{\pgfxy(2.2,4.25)}{\pgfbox[center,center]{\large{$\displaystyle\bigoplus_{x\in \tn{Zar}_X^0}\underline{K_{1,X_\ee \tn{ on } x}}$}}}
\pgfputat{\pgfxy(2.2,2.25)}{\pgfbox[center,center]{\large{$\displaystyle\bigoplus_{x\in \tn{Zar}_X^1}\underline{K_{0,X_\ee \tn{ on } x}}$}}}
\pgfsetendarrow{\pgfarrowpointed{3pt}}
\pgfxyline(2.2,5.75)(2.2,4.9)
\pgfxyline(2.2,3.8)(2.2,2.9)
\end{pgftranslate}
\begin{pgftranslate}{\pgfpoint{9.5cm}{0cm}}
\pgfxyline(.4,1.5)(.4,7)
\pgfxyline(.4,7)(4,7)
\pgfxyline(4,7)(4,1.5)
\pgfxyline(4,1.5)(.4,1.5)
\pgfputat{\pgfxy(2.2,6.25)}{\pgfbox[center,center]{\large{$\ms{K}_{1,X_\ee,\ee}$}}}
\pgfputat{\pgfxy(2.2,4.25)}{\pgfbox[center,center]{\large{$\displaystyle\bigoplus_{x\in \tn{Zar}_X^0}\underline{K_{1,X_\ee,\ee \tn{ on } x}}$}}}
\pgfputat{\pgfxy(2.2,2.25)}{\pgfbox[center,center]{\large{$\displaystyle\bigoplus_{x\in \tn{Zar}_X^1}\underline{K_{0,X_\ee,\ee \tn{ on } x}}$}}}
\pgfsetendarrow{\pgfarrowtriangle{6pt}}
\pgfsetendarrow{\pgfarrowpointed{3pt}}
\begin{pgfscope}
\pgfsetendarrow{\pgfarrowtriangle{6pt}}
\pgfsetlinewidth{3pt}
\pgfxyline(-1,4.3)(0,4.3)
\pgfputat{\pgfxy(-.45,5.4)}{\pgfbox[center,center]{\small{split}}}
\pgfputat{\pgfxy(-.45,5)}{\pgfbox[center,center]{\small{surjection}}}
\end{pgfscope}
\pgfxyline(2.2,5.75)(2.2,4.9)
\pgfxyline(2.2,3.8)(2.2,2.9)
\end{pgftranslate}
\begin{pgftranslate}{\pgfpoint{15cm}{0cm}}
\pgfxyline(.2,1.5)(.2,7)
\pgfxyline(.2,7)(4.2,7)
\pgfxyline(4.2,7)(4.2,1.5)
\pgfxyline(4.2,1.5)(.2,1.5)
\pgfputat{\pgfxy(2.2,6.25)}{\pgfbox[center,center]{\large{$\ms{HN}_{1,X_\ee,\ee}$}}}
\pgfputat{\pgfxy(2.2,4.25)}{\pgfbox[center,center]{\large{$\displaystyle\bigoplus_{x\in \tn{Zar}_X^0}\underline{\tn{HN}_{1,X_\ee,\ee \tn{ on } x}}$}}}
\pgfputat{\pgfxy(2.2,2.25)}{\pgfbox[center,center]{\large{$\displaystyle\bigoplus_{x\in \tn{Zar}_X^1}\underline{\tn{HN}_{1,X_\ee,\ee \tn{ on } x}}$}}}
\pgfputat{\pgfxy(-.65,5.7)}{\pgfbox[center,center]{\small{relative}}}
\pgfputat{\pgfxy(-.65,5.3)}{\pgfbox[center,center]{\small{Chern}}}
\pgfputat{\pgfxy(-.65,4.9)}{\pgfbox[center,center]{\small{character}}}
\pgfputat{\pgfxy(-.6,4.3)}{\pgfbox[center,center]{\huge{$\cong$}}}
\pgfsetendarrow{\pgfarrowpointed{3pt}}
\pgfxyline(2.2,5.75)(2.2,4.9)
\pgfsetendarrow{\pgfarrowpointed{3pt}}
\pgfxyline(2.2,3.8)(2.2,2.9)
\end{pgftranslate}
\end{pgftranslate}
\end{pgfmagnify}
\end{pgfpicture}
\caption{Abstract form of the coniveau machine for zero-cycles on a smooth complex projective curve.}
\label{figabstractconiveaucurve}
\end{figure}

For the simple case of curves, this expression of the coniveau machine represents vast notational and theoretical overkill, but every facet of the structure becomes functional in more general situations.  The objects appearing in this diagram are identified as follows:

\begin{enumerate}
\item $\ms{K}_{1,X}$ is the first Bass-Thomason $K$-theory sheaf on $X$.   Bass-Thomason $K$-theory is discussed in section \hyperref[subsubsectionbassthomason]{\ref{subsubsectionbassthomason}}.  In this simple case, the sheaf $\ms{K}_{1,X}$ is isomorphic to $\ms{O}_X^*$.  It is also isomorphic to the first Quillen $K$-theory sheaf on $X$, and the first Milnor $K$-theory sheaf on $X$.  However, the Bass-Thomason definition is the one that generalizes properly.  
\item The groups $K_{p,X \tn{ on } x}$ for $p=0,1$ are direct limits of Bass-Thomason groups with supports containing the point $x$.  In this simple case, the only point $x$ belonging to $\tn{Zar}_X^0$ is the generic point of $X$, and the group $K_{1,X \tn{ on } x}$ is isomorphic to the multiplicative group of nonzero rational functions $R_X^*$ on $X$.  It is also isomorphic to the first Quillen $K$-group of the function field $R_X$ of $X$, and the first Milnor $K$-theory sheaf on $X$.   For a point $x\in \tn{Zar}_X^0$, the group $K_{0,X \tn{ on } x}$ is isomorphic to the additive group $\ZZ^+$ of integers, viewed as the stalk $Z_x^1$ at $x$ of the sheaf $\ms{Z}_X^1$ of zero-cycles on $X$.  
\item $\ms{K}_{1,X_\ee}$ is the first Bass-Thomason $K$-theory sheaf on the thickened curve $X_\ee$.  In this simple case, it is isomorphic to $\ms{O}_{X_\ee}^*$; its group of sections over an open set $U\subset X$ is the group of first-order arcs $O_{U_\ee}=\big(O_U[\ee]/\ee^2\big)^*$ in $O_U^*$.   It is also isomorphic to the first Quillen $K$-theory sheaf on $X_\ee$, and the first Milnor $K$-theory sheaf on $X_\ee$. 
\item The groups $K_{p,X_\ee \tn{ on } x}$ for $p=0,1$ on the thickened curve $X_\ee$ are Bass-Thomason groups defined via direct limits in a similar manner as the corresponding groups on $X$.   The group $K_{1,X \tn{ on } x}$ at the generic point $x$ of $X$ is isomorphic to the group  $R_{X_\ee}^*=\big(R_X[\ee]/\ee^2\big)^*$ of first-order arcs on $R_X^*$.  The corresponding Quillen and Milnor $K$-groups are the same.  For a point $x\in \tn{Zar}_X^0$, the group $K_{0,X_\ee \tn{ on } x}$ is isomorphic to the group $Z_{x_\ee}^1=Z_x^1\oplus\big(R_{X_\ee,\ee}^*/O_{X_\ee,\ee}^*\big)$ of first-order arcs of zero-cycles at $x$.  Here the Bass-Thomason group differs from the Quillen and Milnor groups, both of which are isomorphic to $\ZZ^+$.  
\item $\ms{K}_{1,X_\ee,\ee}$ is the first {\it relative} Bass-Thomason $K$-theory sheaf on the thickened curve $X_\ee$ with respect to the ideal $(\ee)$.  In this simple case, it is isomorphic to the relative sheaf $\ms{O}_{X_\ee,\ee}^*$, whose group of sections over an open set $U\subset X$ is the group of first-order arcs $O_{U_\ee,\ee}=\{1+f\ee|f\in O_U\}$ at the identity in $O_U^*$.   It is also isomorphic to the corresponding relative Quillen and Milnor $K$-theory sheaves. 
\item The groups $K_{p,X_\ee,\ee \tn{ on } x}$ for $p=0,1$ on the thickened curve $X_\ee$ are punctual relative Bass-Thomason groups; i.e., groups ``at a point," defined via direct limits of relative groups with supports containing $X$.   The group $K_{1,X_\ee,\ee \tn{ on } x}$ at the generic point $x$ of $X$ is isomorphic to the group  $R_{X_\ee,\ee}^*=\{1+f\ee|f\in R_X\}$ of first-order arcs at the identity in $R_X^*$.   By contrast, Quillen $K$-theory and Milnor $K$-theory do not extend to suitably-behaved generalized cohomology theories with supports.  For a point $x\in \tn{Zar}_X^0$, the group $K_{0,X_\ee,\ee \tn{ on } x}$ is isomorphic to the relative group $Z_{x_\ee,\ee}^1=R_{X_\ee,\ee}^*/O_{X_\ee,\ee}^*$ of first-order arcs at the identity in $Z_x^1$.  
\item $\ms{HN}_{1,X_\ee,\ee}$ is the first relative negative cyclic homology sheaf on the thickened curve $X_\ee$ with respect to the ideal $(\ee)$.  In this simple case, it is isomorphic to the sheaf $\ms{O}_X^+$ of additive groups.  Referring back to section \hyperref[SubsectionTangentSheafRegular]{\ref{SubsectionTangentSheafRegular}}, $\ms{HN}_{1,X_\ee,\ee}$ is the tangent sheaf of the $K$-theory sheaf $\ms{K}_{1,X}$.  
\item The groups $\tn{HN}_{p,X_\ee,\ee \tn{ on } x}$ for $p=0,1$ on the thickened curve $X_\ee$ are relative negative cyclic homology groups defined via direct limits of relative groups with supports containing $X$.   The group $\tn{HN}_{1,X_\ee,\ee \tn{ on } x}$ at the generic point $x$ of $X$ is isomorphic to the additive group  $R_X^+$.   For a point $x\in \tn{Zar}_X^0$, the group $\tn{HN}_{0,X_\ee,\ee \tn{ on } x}$ is isomorphic to the group $P_x^1$ of principal parts of rational functions at $x$. 
\end{enumerate}

In terms of these identifications, the sheaf cohomology formula $\tn{Ch}_X^1=H_{\tn{\fsz{Zar}}}^1(X,\ms{O}_X^*)$
for the first Chow group, appearing in equation \hyperref[equCH1curvesheafcohom]{\ref{equCH1curvesheafcohom}} above, may be re-expressed as 
\begin{equation}\label{equfirstchowcurvebloch}\tn{Ch}_X^1=H_{\tn{\fsz{Zar}}}^1(X,\ms{K}_{1,X}),\end{equation}
which is the first nontrivial case of Bloch's isomorphism $\tn{Ch}_X^p=H_{\tn{\fsz{Zar}}}^1(X,\ms{K}_{p,X})$, first cited in equation \hyperref[blochstheoremintro]{\ref{blochstheoremintro}} at the beginning of Chapter \hyperref[ChapterIntro]{\ref{ChapterIntro}}.  Similarly, the formula $T\tn{Ch}_X^1=H_{\tn{\fsz{Zar}}}^1(X,\ms{O}_X)$ may be re-expressed as 
\begin{equation}\label{equfirstchowtancurvebloch}T\tn{Ch}_X^1=H_{\tn{\fsz{Zar}}}^1(X,T\ms{K}_{1,X}),\end{equation}
using the fact that $\ms{O}_X^+=\ms{HN}_{1,X_\ee,\ee}$ is the tangent sheaf $T\ms{K}_{1,X}$ of the $K$-theory sheaf $T\ms{K}_{1,X}$.   It is essentially this expression, albeit using Milnor $K$-theory, that Green and Griffiths chose to generalize to define ``formal tangent groups" to higher Chow groups of arbitrary smooth complex projective varieties.  I introduce a similar definition in section \hyperref[sectiondefitangroup]{\ref{sectiondefitangroup}}, using Bass-Thomason $K$-theory.

\chapter{Technical Background}\label{ChapterTechnical}

\section{Algebraic Geometry}\label{SectionBasicAG}

\subsection{Introduction}

In this section, I review the basic properties of smooth algebraic varieties, from which arise cycle groups and Chow groups.  Much of the terminology and theory familiar from the study of smooth complex projective curves still applies in this more general context. In particular, a smooth algebraic variety $X$ is a locally ringed space with structure sheaf denoted by $\ms{O}_X$; its underlying topological space has the Zariski topology, and is denoted by $\mbox{Zar}_X$.  Rational functions, local rings, and algebraic cycles are defined in the expected ways.   The formal definition of {\it algebraic variety} I am using here is {\it noetherian integral separated algebraic scheme of finite type over an algebraically closed field.}  For the convenience of the reader, I briefly review some elementary scheme theory, and discuss the special properties of a smooth algebraic varieties.  The material in this section is drawn principally from Hartshorne \cite{HartshorneAlgebraicGeometry77}, though with some differences in notation.  Afterwards, I discuss dimension and codimension of subsets of a smooth algebraic variety, and introduce the coniveau filtration.  This prepares the way for the discussion of algebraic cycles in section \hyperref[sectioncyclegroups]{\ref{sectioncyclegroups}}.

\subsection{Algebraic Schemes; Regular Functions}

An {\bf algebraic scheme}, or simply {\bf scheme}, is a locally ringed space $X=\big(\mbox{Zar}_X,\ms{O}_X\big)$ admitting an open covering by affine schemes.  An {\bf affine scheme} is a locally ringed space isomorphic to the {\bf prime spectrum} $\mbox{Spec}\hspace*{.05cm}A=\big(\mbox{Zar}_{\mbox{\footnotesize{Spec }}A},\ms{O}_{\mbox{\footnotesize{Spec }}A}\big)$ of a ring $A$.   The underlying topological space $\mbox{Zar}_{\mbox{\footnotesize{Spec }}A}$ of $\mbox{Spec } A$ is the set of prime ideals of $A$, equipped with the {\bf Zariski topology}, which is defined by taking a set to be closed if and only if it is the set of all prime ideals containing some ideal of $A$.   The {\bf structure sheaf} $\ms{O}_{\mbox{\footnotesize{Spec }}A}$ of $\mbox{Spec } A$ is the sheaf on $\mbox{Zar}_{\mbox{\footnotesize{Spec }A}}$ whose sections are given locally by quotients of elements of $A$, with values at each prime ideal $p$ of $A$ belonging to the local ring $A_p$.   The ring $A$ is called the {\bf affine coordinate ring} of its prime spectrum $\mbox{Spec }A$.  Returning to $X$, the topology on $X$ is called the {\bf Zariski topology}, and the structure sheaf $\ms{O}_X$ of $X$ is called the {\bf algebraic structure sheaf} of $X$, or the {\bf sheaf of regular functions} opn $X$.  The statement that $X$ admits an open covering by affine schemes means that every point $x$ of $\mbox{Zar}_X$ has an open neighborhood $U$, such that the locally ringed space given by restricting $\ms{O}_X$ to $U$ is an affine scheme. 

A {\bf regular function} on an open subset $U$ of a scheme $X$ is defined to be an element of the ring $O_U$ of sections of $\ms{O}_X$ over $U$.   Since regular functions are locally defined as quotients in the coordinate ring of an affine open subset of $X$, the sheaf $\ms{O}_X$ is the appropriate generalization of the sheaf of regular functions on a smooth projective curve. The multiplicative subgroup $O_U^*$ of invertible elements of $O_U$ is called the {\bf group of invertible regular functions} on $U$.  The {\bf sheaf of invertible regular functions} $\ms{O}_X^*$ on $X$ is the sheaf of multiplicative groups on $X$ whose group of sections over an open subset $U$ of $X$ is $O_U^*$.  The {\bf local ring} $O_x$ of $X$ at a point $x\in X$ is defined to be the stalk of the algebraic structure sheaf $\ms{O}_X$ at $x$.  It is the direct limit of the rings $O_U$ over all open subsets $U$ of $X$ containing $x$, ordered by inclusion.   This definition is equivalent to the elementary definition of the local ring at a point in a smooth projective curve, given in terms of equivalence classes of regular functions.  Since $O_x$ is a local ring, it has a unique maximal ideal, denoted by $\mfr{m}_x$.  The quotient $r_x:=O_x/\mfr{m}_x$ is a field called the {\bf residue field at $x$} of $X$.

\subsection{Noetherian Integral Schemes; Function Fields}

The condition that a smooth algebraic variety $X$ is a {\bf noetherian} scheme means that it is quasicompact and admits a covering by prime spectra of noetherian rings.  The condition that $X$ is an {\bf integral} scheme means that the rings $O_U$ of regular functions over its open subsets $U$ are integral domains.  An equivalent condition is that $X$ is both {\bf reduced}, meaning that the rings $O_U$ contain no nilpotent elements, and {\bf irreducible}, meaning that $\mbox{Zar}_X$ cannot be expressed as the union of two closed proper subsets.  Integral schemes possess direct analogues of the fields and sheaves of rational functions on smooth projective complex curves.  For non-integral schemes, these fields and sheaves  must be replaced by total quotient rings and sheaves of total quotient rings.

The {\bf rational function field} $R_X$ of an integral scheme $X$ is defined to be the local ring of $X$ at its generic point.  The elements of $R_X$ are equivalence classes of regular functions, defined on open subsets of $X$, where two regular functions belong to the same equivalence class if and only if they coincide on the intersection of their sets of definition.  Hence, $R_X$ is the appropriate generalization of the rational function field of a smooth projective curve.   Every nonempty open subset $U$ of $X$ also has a rational function field $R_U$, again defined as the local ring at its generic point, and these fields are all canonically isomorphic to $R_X$, since every nonempty open subset of $X$ contains its generic point.   It is usually harmless to view these isomorphisms as identifications.  The {\bf sheaf of rational functions} $\ms{R}_X$ on $X$ is the sheaf of fields on $X$ whose field of sections over $U$ is $R_U$.  This sheaf is canonically isomorphic to a skyscraper sheaf at the generic point of $\mbox{Zar}_X$ with underlying field $R_X$.  Its stalks $R_x$ at every point $x$ are canonically isomorphic to $R_X$.

\subsection{Separated Schemes; Smooth Algebraic Varieties}

{\bf Properties of Morphisms.}  The remaining conditions distinguishing smooth algebraic varieties depend on special properties of morphisms of schemes.   These morphisms are more general than morphisms between smooth algebraic varieties.  This level of generality is necessary in the present context, since it is sometimes important to consider morphisms between a smooth variety $X$ and singular schemes related to $X$.   A {\bf morphism} between two schemes $X$ and $X'$ consists of a continuous map $j:\mbox{Zar}_X\rightarrow\mbox{Zar}_{X'}$ together with a map of sheaves $\ms{O}_{X'}\rightarrow j_*\ms{O}_{X}$ on $X'$,  where $j_*\ms{O}_{X}$ is the direct image on $X'$ of the structure sheaf $\ms{O}_X$ of $X$.   Later, e.g., in section \hyperref[xsubsectiondistinguishedcatschemes]{\ref{subsectiondistinguishedcatschemes}}, it will be necessary to consider {\bf morphisms of pairs} $(X,Z)\rightarrow(X',Z')$, where $X$ and $X'$ are schemes and $Z$ and $Z'$ are closed subsets.   Such a morphism is defined to be a morphism $X\rightarrow X'$ such that the inverse image of $Z'$ is contained in $Z$.  

{\bf Separated Schemes.}  Separatedness is a technical condition that rules out certain ``pathological ways of gluing together affine schemes."  The special importance of separated schemes in this book is twofold. First, the category $\mbf{Sep}_k$ of separated schemes over a field $k$ satisfies the conditions at the beginning of section \hyperref[subsectioncohomsupportssubstrata]{\ref{subsectioncohomsupportssubstrata}} identifying a distinguished category of schemes suitable for defining a cohomology theory with supports in the sense of Colliot-Th\'el\`ene, Hoobler, and Kahn \cite{CHKBloch-Ogus-Gabber97}.  Second, given such a cohomology theory with supports $H$, a fixed separated scheme $Y$ may be used to define an ``augmented theory"  $H^Y$, and this construction preserves important properties of $H$.  

Given any morphism of schemes $X\rightarrow X'$, there exists an associated {\bf diagonal morphism} $\Delta$ from $X$ to the fiber product $X\times_{X'}X$, whose composition with both projection maps $X\times_{X'}X\rightarrow X$ is the identity morphism on $X$.   The original morphism is called {\bf separated} if this diagonal morphism is a {\bf closed immersion}, meaning that the map of topological spaces $\mbox{Zar}_X\rightarrow \mbox{Zar}_{X\times_{X'}X}$ induced by $\Delta$ is a homeomorphism onto its image, and the map of sheaves $\ms{O}_{X\times_{X'}X}\rightarrow\Delta_*\ms{O}_X$ induced by $\Delta$ is surjective. In this case, $X$ is called {\bf separated} over $X'$.  When $X$ is a smooth algebraic variety and $X'$ is the prime spectrum $\mbox{Spec } k$ of the ground field $k$, $X$ is called {\bf separated over $k$}.   If $R$ is a $k$-algebra, the affine scheme $X=\tn{Spec } R$ is separated over $k$, since the diagonal morphism $X\rightarrow X\times_{\tn{Spec }k}X$ corresponds to the surjective map of rings $R\otimes_kR\rightarrow R$ taking $r\otimes r'$ to the product $r$.   By essentially the same argument, any morphisms of affine schemes is separated; see Hartshorne \cite{HartshorneAlgebraicGeometry77}, Proposition 4.1, page 96. 

\begin{example}\label{examplenilpotentthickening} (Nilpotent Thickening.) \tn{Let $X$ be a smooth scheme over a field $k$, and let $Y$ be the prime spectrum of a $k$-algebra $R$ generated over $k$ by nilpotent elements; e.g., a local artinian $k$-algebra.  Then the {\bf nipotent thickening} of $X$ by $Y$ is the fiber product $X\times_kY$.   Nilpotent thickenings are generalizations of the thickened curve $X_\ee$ discussed in section \hyperref[SectionThickened]{\ref{SectionThickened}} above.  The nilpotent thickening $X\times_kY$ has the same topological space as $X$, but a different structure sheaf; namely, the ``nilpotent-augmented sheaf" $\ms{O}_X\otimes_kR$.}
\end{example}

{\bf Finite Type; Smooth Algebraic Varieties.} The condition that a smooth algebraic variety $X$ is  {\bf of finite type over $k$} reduces to the simple statement that $X$ admits a finite open covering by prime spectra of $k$-algebras.  More generally, a morphism of schemes $X\rightarrow X'$ is called {\bf of finite type} if there exists an open affine covering of $X'$ such that the inverse image of each set in the covering itself admits a finite open affine covering.   In this context, the statement that $X$ is of finite type over $k$ means that the canonical morphism $X\rightarrow\mbox{Spec }k$ is of finite type.  The condition that a smooth algebraic variety $X$ is {\bf smooth} means that all its local rings are regular.   In this context, the terms smooth, {\bf nonsingular}, and {\bf regular} are all synonymous.  Finally, an advantage of taking the ground field $k$ to be algebraically closed is that in this case there is a well-behaved functor from the ``na\"{\i}ve category of varieties over $k$," defined in terms of algebraic subsets of affine or projective spaces over $k$, to the category of schemes over $k$.

\subsection{Dimension; Codimension; Coniveau Filtration}\label{subsectiondimension}

The {\bf dimension} of an irreducible noetherian scheme $X$ is defined to be the dimension of the underlying topological space $\mbox{Zar}_X$ of $X$, which is the supremum of all integers $n$ such that there exists a chain $Z^n\subset Z^{n-1}\subset...\subset Z^0$ of distinct irreducible closed subsets of $\mbox{Zar}_X$.   Since $X$ is irreducible, the subset $Z^0$ must be $\mbox{Zar}_X$.  If $A$ is a $k$-algebra such that $X$ is locally isomorphic to $\mbox{Spec } A$, then the dimension of $X$ is equal to the {\bf Krull dimension} of $A$, which is the supremum of the heights of the prime ideals of $A$, where the {\bf height }of a prime ideal $p$ of $A$ is the supremum of all integers $n$ such that here exists a chain $p_0\subset p_1\subset...\subset p_n=p$ of distinct prime ideals of $A$. 

If $Z$ is an irreducible closed subset of $X$, then the {\bf codimension} of $Z$ in $X$ is defined to be the supremum of all integers $m$ such that there exists a chain $Z=Z^m\subset Z^{m-1}\subset...\subset Z^0$ of distinct irreducible closed subsets of $X$ containing $Z$.   Again, the subset $Z^0$ must be $\mbox{Zar}_X$.  If $x$ is a point of $X$, then the {\bf codimension} of $x$ in $X$ is defined to be the codimension of the closed subset $\overline{\{x\}}$ in $X$.  The codimension of $x$ is equal to the Krull dimension of the local ring $O_x$ at $x$ in $X$.  If $A$ is a $k$-algebra such that $X$ is locally isomorphic to $\mbox{Spec } A$, and if $x$ corresponds to a prime ideal $p$ in $X$, then the codimension of $x$ is equal to the height of $p$ in $A$.  The set of all points of codimension $p$ in a noetherian scheme $X$ is denoted by $\mbox{Zar}_X^p$.  If $X$ is $n$-dimensional, the topological space $\mbox{Zar}_X$ may be partitioned into $n+1$ subsets: the sets $\mbox{Zar}_X^{p}$ for all $p$ between $0$ and $n$.   If $X$ is irreducible, $\mbox{Zar}_X^{0}$ is a singleton set consisting of the unique generic point of $\mbox{Zar}_X$.   $\mbox{Zar}_X^{n}$ is the set of zero-dimensional points of $X$, which were the only points recognized in the classical theory of algebraic varieties.  The reason for using codimension is that it is more natural than dimension in the context of algebraic cycles. 

A {\bf decreasing filtration} on $\mbox{Zar}_X$ is a chain of inclusions $...\subset Z^p\subset Z^{p-1}\subset...$ of subsets of $\mbox{Zar}_X$.  Such a filtration is called {\bf exhaustive} if $\bigcup_pZ^p=X$, and is called {\bf separated} if $\bigcap_pZ^p=\oslash.$  The {\bf coniveau filtration} on an $n$-dimensional algebraic variety $X$ is the exhaustive, separated, decreasing filtration
\begin{equation}\label{equdefconiveau}\oslash\subset \mbox{Zar}_X^{\ge n}\subset \mbox{Zar}_X^{\ge n-1}\subset...\subset \mbox{Zar}_X^{\ge 1}\subset \mbox{Zar}_X^{\ge 0}=\mbox{Zar}_X,\end{equation}
where $\mbox{Zar}_X^{\ge p}$ is defined to be the set of all point of $X$ of codimension at least $p$.  The $p$th {\bf graded part} of the coniveau filtration on $X$ is defined to be the quotient space $\mbox{Zar}_X^{\ge p}/\mbox{Zar}_X^{\ge p+1}$, which may be identified with the set $\mbox{Zar}_X^p$ of all points of codimension $p$ in $X$.

\section{Algebraic Cycles; Cycle Groups; Chow Groups}\label{sectioncyclegroups}

\subsection{Cycles; Total Cycle Group}\label{subsectiontotalcyclegroup}

An {\bf algebraic cycle} on a smooth algebraic variety $X$ is a finite formal sum 
\begin{equation}\label{equdefcycle} z=\sum_i n_{i}z_i,\end{equation}
where each $z_i$ is a subvariety of $X$.  The coefficients $n_i$ are elements of a ring, which I will always take to be the integers unless stated otherwise.  The {\bf support} of $z$ is the union $\bigcup_iz_i$ of the subvarieties of $X$ with nonzero multiplicities in $z$.   An algebraic cycle is called {\bf effective} if all its multiplicities are nonnegative.  Every algebraic cycle is a difference of effective cycles.   The set $Z_X$ of algebraic cycles on $X$ is an abelian group, called the {\bf total cycle group} of $X$, with group operation defined by addition of multiplicities at each subvariety.   The identity in $Z_X$ is the {\bf empty cycle,} in which all subvarieties of $X$ are assigned zero multiplicity.   The empty cycle is sometimes denoted by $0$ because the group operation on $Z_X$ is additive.  The subset $Z_{X,\mbox{\footnotesize{eff}}}$ of effective cycles on $X$ is a submonoid of $Z_X$ called the {\bf monoid of effective cycles} on $X$.  The group completion of $Z_{X,\mbox{\footnotesize{eff}}}$ is just the group $Z_X$. 

Given an open subset $U$ of $X$, the group $Z_U$ of algebraic cycles and the submonoid $Z_{U,\mbox{\footnotesize{eff}}}$ of effective cycles on $U$ may be defined in a similar way.  The group completion of $Z_{U,\mbox{\footnotesize{eff}}}$ is $Z_U$.  The {\bf sheaf of algebraic cycles} on $X$ is the sheaf $\ms{Z}_X$ of additive groups on $X$ whose group of sections over $U$ is $Z_U$.   The {\bf sheaf of effective zero-cycles} on $X$ is the sheaf $\ms{Z}_{X,\mbox{\footnotesize{eff}}}$ of additive monoids on $X$ whose monoid of sections over $U$ is $Z_{U,\mbox{\footnotesize{eff}}}$.

\subsection{Codimension-$p$ Cycles; Cycle Group $Z_X^p$}\label{subsectioncodimpcycles}

The total cycle group $Z_X$ of a smooth algebraic variety $X$ has a nonnegative integer grading, induced by the codimensions of the subvarieties of $X$.  The $p$th graded piece $Z_X^p$ is a subgroup of $Z_X$ called the {\bf group of codimension-$p$ cycles} on $X$.   The total cycle group therefore decomposes as a direct sum
\begin{equation}\label{equtotalcycledirsum}Z_X=\bigoplus_{p\ge 0}Z_X^p,\end{equation}
of the groups of codimension-$p$ cycles, for all nonnegative integers $p$.    Each cycle $z=\sum_in_iz_i$ belonging to $Z_X^p$ is defined by assigning multiplicities $n_i$ to a finite number of codimension-$p$ subvarieties $z_i$ of $X$.  Each of these subvarieties has a unique generic point $x_i$, which is an element of $\mbox{Zar}_X^p$, and whose closure $\overline{x}_i$ is $z_i$.   Hence, the sum defining $z$ may be regarded as a sum over the entire set $\mbox{Zar}_X^{p}$ of codimension-$p$ points of $X$, with the understanding that only a finite number of the multiplicities are nonzero.  

The {\bf degree} of a codimension-$p$ cycle $z=\sum_in_iz_i$ is the sum $\sum_in_i$ of its multiplicities.  The subset $Z_{X,0}^p$ of codimension-$p$ cycles of degree zero on $X$ is a subgroup of $Z_X^p$. The subset $Z_{X,\mbox{\footnotesize{eff}}}^p$ of effective codimension-$p$ cycles on $X$ is a submonoid of $Z_X^p$ called the {\bf monoid of effective codimension-$p$ cycles} on $X$.  The group completion of $Z_{X,\mbox{\footnotesize{eff}}}^p$ is just the group $Z_X^p$.  Given an open subset $U$ of $X$, the group $Z_U^P$ of codimension-$p$ cycles cycles, the subgroup $Z_{U,0}^P$ of cycles of degree zero, and the submonoid $Z_{U,\mbox{\footnotesize{eff}}}^p$ of effective cycles on $U$ may be defined in a similar way.  The group completion of $Z_{U,\mbox{\footnotesize{eff}}}^p$ is $Z_U$.  The {\bf sheaf of codimension-$p$ cycles} on $X$ is the sheaf $\ms{Z}_X^p$ of additive groups on $X$ whose group of sections over $U$ is $Z_U^p$.  The {\bf sheaf of effective zero-cycles} on $X$ is the sheaf $\ms{Z}_{X,\mbox{\footnotesize{eff}}}^p$ of additive monoids on $X$ whose monoid of sections over $U$ is $Z_{U,\mbox{\footnotesize{eff}}}^p$.

\subsection{Equivalence Relations for Algebraic Cycles}\label{subsectionEquivalence}

Equivalence classes of algebraic cycles on a smooth algebraic variety $X$ are often more convenient to work with than individual cycles.  The cycle groups $Z_X^p$ are large, cumbersome, and badly behaved with respect to natural operations such as intersection.  Although the most useful and well-studied groups of cycle equivalence classes remain large and mysterious, their geometric and algebraic behavior is more reasonable than that of the cycle groups.  Suitable equivalence relations yield graded ring structures on their corresponding groups of cycle equivalence classes, with multiplication given by appropriate versions of intersection theory.   The most interesting equivalence relations induce functors from the category of smooth algebraic varieties to the category of graded rings, so that the corresponding groups of cycle equivalence classes interact in useful ways under morphisms of varieties.

\subsection{Intersection Theory; Adequate Equivalence Relations}

Since the early days of algebraic geometry, it has been recognized that the set-theoretic intersection of a codimension-$p$ subvariety  and a codimension-$q$ subvariety of a smooth algebraic variety $X$ is of codimension $p+q$ ``in the generic case."  For example, distinct plane curves meet in a finite collection of points. This suggests defining a ``multiplication map" $Z_X^p\times Z_X^q\rightarrow Z_X^{p+q}$, but this na\"{\i}ve approach fails for a number of reasons.  Codimension is not always additive under intersection, for example, when one of the two intersecting subvarieties contains the other.  Also, set-theoretic intersection does not assign the proper intersection multiplicities when intersections involve tangencies or singular points.  Two subvarieties of a smooth algebraic variety $X$, of codimensions $p$ and $q$, respectively, are said to {\bf intersect properly} if their set-theoretic intersection has dimension $p+q$.  Given a proper intersection, an appropriate intersection multiplicity may be defined in terms of the algebraic notion of length of a module.   The resulting definitions may then be extended by linearity to yield notions of proper intersection and local intersection multiplicity for pairs of algebraic cycles.  This provides a {\it partially defined} intersection product on $Z_X$, but not a ring structure. 

Pierre Samuel introduced the notion of adequate equivalence relations, in his 1960 paper {\it Relations d'\'Equivalence en G\'eom\'etrie Algebrique} \cite{SamuelAdequate58}, with the stated goal to ``remedy this state of affairs."  An equivalence relation $\sim_{equ}$ is called {\bf adequate} if it is additive, if every pair of equivalence classes of cycles possesses a corresponding pair of representatives that intersect properly, and if appropriate functorial properties involving morphisms between ambient varieties are satisfied.  An adequate equivalence relation endows the additive groups $Z_{X,\mbox{\footnotesize{equ}}}$ of equivalence classes under $\sim_{equ}$ with graded ring structures, where multiplication corresponds to addition of codimension, and generalizes the partially-defined intersection product introduced above.    

For illustrative purposes, I will discuss a few of the most familiar equivalence relations on cycles, focusing on the case of smooth complex projective varieties.  The standard reference for intersection theory is Fulton \cite{FultonIntersection}.  In this context, the coarsest adequate equivalence relation for which integer multiples of a single point remain distinct is {\bf numerical equivalence,}  which is defined in terms of the degree map $\mbox{deg}:Z_X\rightarrow\ZZ$, which counts the number of points in a zero-cycle, including multiplicities.  It is extended to all of $Z_X$ by mapping positive-dimensional subvarieties to zero.  Two cycles $\gamma,\gamma'\in Z_X$ are called numerically equivalent if $\mbox{deg}(\gamma\cdot\sigma)=\mbox{deg}(\gamma'\cdot\sigma)$ whenever both products are defined.  If $\gamma$ and $\gamma'$ are subvarieties, this means that $\gamma$ and $\gamma'$ must be of the same dimension and that the proper intersections of $\gamma$ and $\gamma'$ with any subvariety of complementary dimension must contain the same number of points, counting multiplicities.  I will denote the group of codimension-$p$ cycles on $X$ numerically equivalent to zero by $Z_{X,\tn{num}}^p$. 

For any Weil cohomology theory $H$ with cohomology groups $H_X^p$ on $X$, there exists a cycle class map $\tn{cl}_{H,p}:Z_X^p\rightarrow H_X^{2p}$.  For example, integration and Poincar\`{e}-Serre duality define a cycle class map with values in de Rham cohomology. Two cycles on $X$ are {\bf homologically equivalent} with respect to $H$ if their images in $H_X^{2p}$ coincide.  It is not known if homological equivalence depends on the choice of Weil cohomology theory, or if homological equivalence is really different than numerical equivalence.  Homological equivalence is the subject of some of the most famous open conjectures in algebraic geometry.  I will denote the group of codimension-$p$ cycles on $X$ homologically equivalent to zero by $Z_{X,\tn{hom}}^p$.  This is imprecise, since it does not specify the choice of cohomology theory, but will suffice for the limited purposes needed here. 

A third adequate equivalence relation is algebraic equivalence, which is closely related to rational equivalence, discussed in more detail in section \hyperref[subsectionratchow]{\ref{subsectionratchow}} below.  Two algebraic cycles $z$ and $z'$ on $X$ are {\bf algebraically equivalent} if there exists an algebraic curve $C$, a cycle $Z$ on $X\times C$, flat over $C$, and two points $\lambda$ and $\lambda'$ of $C$, such that 
\[\pi_*\big(Z\cap X\times\{\lambda\}\big)-\pi_*\big(Z\cap X\times\{\lambda'\}\big)=z-z',\]
where $\pi:X\times C\rightarrow X$ is projection to $X$, and $\pi_*$ is the associated pushforward on cycles.\footnotemark\footnotetext{An oft quoted, and generally incorrect, definition is that $\pi_*\big(Z\cap X\times\{\lambda\}\big)=z$ and  $\pi_*\big(Z'\cap X\times\{\lambda'\}\big)=z'$.}  As discussed in section \hyperref[subsectionzerocyclessurfaces]{\ref{subsectionzerocyclessurfaces}} below, it is known that algebraic equivalence is much finer than homological equivalence.  I will denote the group of codimension-$p$ cycles on $X$ algebraically equivalent to zero by $Z_{X,\tn{alg}}^p$.


\subsection{Rational Equivalence and Chow Groups}\label{subsectionratchow}

A fourth adequate equivalence relation, of central interest in the context of Chow groups, is rational equivalence.  Rational equivalence generalizes linear equivalence of divisors to arbitrary codimension.  It is less general than algebraic equivalence.  
Two algebraic cycles $z$ and $z'$ on $X$ are {\bf rationally equivalent} if there a cycle $Z$ on $X\times \PP^1$, flat over $C$, and two points $\lambda$ and $\lambda'$ of $\PP^1$, such that 
\[\pi_*\big(Z\cap X\times\{\lambda\}\big)-\pi_*\big(Z\cap X\times\{\lambda'\}\big)=z-z',\]
where $\pi:X\times \PP^1\rightarrow X$ is projection to $X$, and $\pi_*$ is the associated pushforward on cycles.  Hence, rational equivalence restricts the definition of algebraic equivalence to the case in which the curve $C$ is the projective line $\PP^1$.   As in the case of zero-cycles oncurves, I will denote the group of codimension-$p$ cycles on $X$ rationally equivalent to zero by $Z_{X,\tn{rat}}^p$. 

Figure \hyperref[figrationalequivalencehigher]{\ref{figrationalequivalencehigher}} illustrates the definition of rational equivalence for codimension-one cycles in the case where $X$ is an algebraic surface.  This is a higher-dimensional analogue of figure \hyperref[figrationalequivalence]{\ref{figrationalequivalence}} of section \hyperref[subsectionlinratfirstchow]{\ref{subsectionlinratfirstchow}}, which illustrates the corresponding case when $X$ is an algebraic curve.

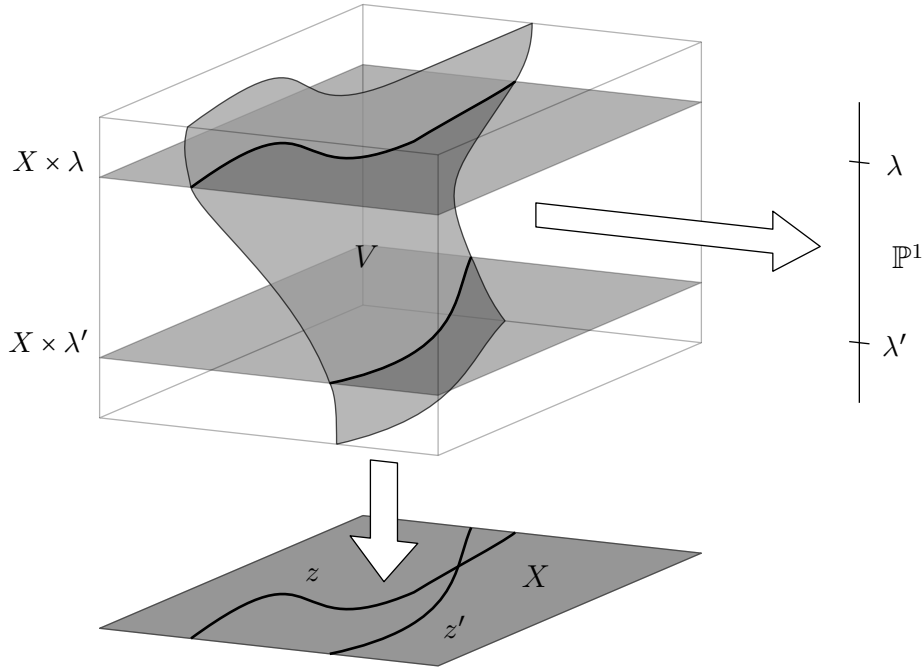
\begin{figure}[H]
\begin{tikzpicture} [scale=1, isometric2XYZ, line join=round,
        opacity=1, text opacity=1.0,
        >=latex,
        inner sep=0pt,
        outer sep=2pt,
    ]
\draw [line width=.5, opacity=.5, fill=black!100] (-2.5,-2.5,-6) -- (-2.5,2.5,-6)
--(2.5,2.5,-6)
--(2.5,-2.5,-6)
--(-2.5,-2.5,-6);
\draw [line width=.5, opacity=.3, fill=black!20] (-2.5,-2.5,-6) -- (-2.5,2.5,-6)
--(2.5,2.5,-6)
--(2.5,-2.5,-6)
--(-2.5,-2.5,-6);
\draw[line width=1pt, opacity=1, color=black] 
(2.5,.9,-6) .. controls (1,1.7,-6)  and (-1,0,-6) .. (-2.5,-.9,-6);
\draw[line width=1pt, opacity=1, color=black] 
(2.5,-1.15,-6) .. controls (-1.6,-2.3,-6)  and (2.5,.2,-6) .. (0,0.2,-6)
--(0,.2,-6) .. controls (-1,0,-6)  and (-2,-.1,-6) .. (-2.5,-.25,-6);
\draw [line width=.5, opacity=.3] (-2.5,-2.5,2.5) -- (-2.5,-2.5,-2.5);
\draw [line width=.5, opacity=.3] (-2.5,-2.5,-2.5) -- (-2.5,2.5,-2.5)
--(2.5,2.5,-2.5)
--(2.5,-2.5,-2.5)
--(-2.5,-2.5,-2.5);
\draw[line width=.5pt, opacity=.7, fill=black!40] 
(2.5,1,-2.5) .. controls (2.5,1,-2.1)  and (2.5,1,-1.7) .. (2.5,.9,-1.5)
--(2.5,.9,-1.5) .. controls (1,1.7,-1.5)  and (-1,0,-1.5) .. (-2.5,-.9,-1.5)
--(-2.5,-.9,-1.5) .. controls (-2.5,-.8,-1.7)  and (-2.5,-.7,-2.1) .. (-2.5,-.4,-2.5)
--(-2.5,-.4,-2.5).. controls (-1,.2,-2.5)  and (1,1.7,-2.5) .. (2.5,1,-2.5);
\draw [line width=.5, opacity=.3, fill=black!100] (-2.5,-2.5,-1.5) -- (-2.5,2.5,-1.5)
--(2.5,2.5,-1.5)
--(2.5,-2.5,-1.5)
--(-2.5,-2.5,-1.5);
\draw[line width=.5pt, opacity=.7, fill=black!40] 
(2.5,.9,-1.5) .. controls (2.5,.4,0)  and (2.5,.-1,1) .. (2.5,-1.15,1.5)
--(2.5,-1.15,1.5) .. controls (-1.6,-2.3,1.5)  and (2.5,.2,1.5) .. (0,0.2,1.5)
--(0,.2,1.5) .. controls (-1,0,1.5)  and (-2,-.1,1.5) .. (-2.5,-.25,1.5)
--(-2.5,-.25,1.5) .. controls (-1,0,.5)  and (-2,-1,0) .. (-2.5,-.9,-1.5)
--(-2.5,-.9,-1.5) .. controls (-1,0,-1.5)  and (1,1.7,-1.5) .. (2.5,.9,-1.5);
\draw[line width=1pt, opacity=1, color=black] 
(2.5,.9,-1.5) .. controls (1,1.7,-1.5)  and (-1,0,-1.5) .. (-2.5,-.9,-1.5);
\draw [line width=.5, opacity=.3, fill=black!100] (-2.5,-2.5,1.5) -- (-2.5,2.5,1.5)
--(2.5,2.5,1.5)
--(2.5,-2.5,1.5)
--(-2.5,-2.5,1.5);
\draw[line width=.5pt, opacity=.7, fill=black!40] 
(2.5,-1.15,1.5) .. controls (2.5,-1.2,1.8)  and (2.5,-1.3,2.1) .. (2.5,-1.2,2.5)
--(2.5,-1.2,2.5) .. controls (-2,-2.5,2.5)  and (2.5,0,2.5) .. (0,0,2.5)
--(0,0,2.5) .. controls (-1,0,2.5)  and (-2,0,2.5) .. (-2.5,0,2.5)
--(-2.5,0,2.5).. controls (-2.5,0,2.1)  and (-2.5,-.1,1.8) .. (-2.5,-.25,1.5)
--(-2.5,-.25,1.5) .. controls (-2,-.1,1.5)  and (-1,0,1.5) .. (0,.2,1.5)
--(0,.2,1.5) .. controls (2.5,.2,1.5)  and (-1.6,-2.3,1.5) .. (2.5,-1.15,1.5);
\draw[line width=1pt, opacity=1, color=black] 
(2.5,-1.15,1.5) .. controls (-1.6,-2.3,1.5)  and (2.5,.2,1.5) .. (0,0.2,1.5)
--(0,.2,1.5) .. controls (-1,0,1.5)  and (-2,-.1,1.5) .. (-2.5,-.25,1.5);
\draw [line width=.5, opacity=.3] (-2.5,-2.5,2.5) -- (-2.5,2.5,2.5)
--(2.5,2.5,2.5)
--(2.5,-2.5,2.5)
--(-2.5,-2.5,2.5);
\draw [line width=.5, opacity=.3] (2.5,2.5,2.5) -- (2.5,2.5,-2.5);
\draw [line width=.5, opacity=.3] (2.5,-2.5,2.5) -- (2.5,-2.5,-2.5);
\draw [line width=.5, opacity=.3] (-2.5,2.5,2.5) -- (-2.5,2.5,-2.5);
\draw [line width=.5, opacity=1] (-1,6,-2.5) -- (-1,6,2.5);
\draw [line width=.5, opacity=1] (-1,5.85,1.5) -- (-1,6.15,1.5);
\draw [line width=.5, opacity=1] (-1,5.85,-1.5) -- (-1,6.15,-1.5);
\draw[line width=.5pt, opacity=1, fill=white]  (2.5,1.5,-2.7) --(2.5,1.5,-4)
--(2.5,1.2,-4)  
--(2.5,1.7,-4.7)  
--(2.5,2.2,-4)  
--(2.5,1.9,-4) 
--(2.5,1.9,-2.7)   
--cycle;
\draw[line width=.5pt, opacity=1, fill=white]  (0,2,.7) --(0,5.5,.7)
--(0,5.5,1)
--(0,6.2,.5)  
--(0,5.5,0)  
--(0,5.5,.3)  
--(0,2,.3) 
--cycle;
\draw (0,2,-5.5) node {\color{black}$X$};
\draw (0,-1.3,-5.9) node {\color{black}\small{$z$}};
\draw (0,.8,-6.5) node {\color{black}\small{$z'$}};
\draw (0,7.5,.5) node {\color{black}$\PP^1$};
\draw (0,7.3,2) node {\color{black}\small{$\lambda$}};
\draw (0,7.3,-1) node {\color{black}\small{$\lambda'$}};
\draw (0,-.5,-.5) node {\color{black}$V$};
\draw (0,-5.2,.5) node {\color{black}\small{$X\times\lambda$}};
\draw (0,-5.2,-2.5) node {\color{black}\small{$X\times\lambda'$}};
\end{tikzpicture}
\caption{Rational equivalence for codimension-one cycles on an algebraic surface.}
\label{figrationalequivalencehigher}
\end{figure}

The {\bf $p$th Chow group} $\tn{Ch}_X^p$ of $X$ is the group $Z_X^p/Z_{X,\tn{rat}}^p$ of codimension-$p$ cycles on $X$ modulo rational equivalence.   Certain subgroups of $\tn{Ch}_X^p$, defined in terms of other adequate equivalence relations, play a role in the discussion below.  In particular, I will denote the group $Z_{X,\tn{num}}^p/Z_{X,\tn{rat}}^p$ of codimension-$p$ cycles on $X$ numerically equivalent to zero modulo rational equivalence by $\tn{Ch}_{X,\tn{num}}^p$.  Similarly, the groups $\tn{Ch}_{X,\tn{hom}}^p$ and $\tn{Ch}_{X,\tn{alg}}^p$ are defined by taking the quotients of $Z_{X,\tn{hom}}^p$ and $Z_{X,\tn{alg}}^p$ by $Z_{X,\tn{rat}}^p$. 


\subsection{$\tn{Ch}_X^1$: the Easy Case}\label{subsectionzerocyclessurfaces}

As mentioned in section \hyperref[subsectionzerocyclesdivisors]{\ref{subsectionzerocyclesdivisors}}, the theory of codimension-one cycles; i.e., divisors, on an algebraic variety $X$, is generally much simpler than the theory of higher-codimensional cycles on $X$.  The relationships among divisors, invertible sheaves, rational functions, etc., make up a standard part of the classical algebraic geometry curriculum.  Since the first Chow group $\tn{Ch}_X^1$ consists of rational equivalence classes of divisors, it inherits this simplicity.  For the purposes of this book, it will suffice to briefly note some of the special properties of divisors.  The main object of this exercise is to provide context for the general case, where these properties fail.   

From the viewpoint of infinitesimal structure, the most important property of $\tn{Ch}_X^1$ is that it is an algebraic group; namely, the Picard group $\tn{Pic}_X$ of $X$.  This relationship generalizes from the case of curves mentioned in section \hyperref[subsectionJacPic]{\ref{subsectionJacPic}} above.  This means that one ``already knows the answer" regarding what the infinitesimal structure of $\tn{Ch}_X^1$ ``ought to be,"  via algebraic Lie theory.  In section \hyperref[SectionInfChow]{\ref{SectionInfChow}} above, I used this viewpoint to check the answer arising from the coniveau machine for zero-cycles on curves.  In general, this means that the methods developed here cannot supply any new information regarding the infinitesimal structure of $\tn{Ch}_X^1$; standard methods are already satisfactory.  Indeed, as demonstrated by famous results of Mumford, Griffiths, and others, discussed in section \hyperref[subsectionzerocyclessurfaces]{\ref{subsectionzerocyclessurfaces}} below, $\tn{Ch}_X^1$ is generally a wholly different type of mathematical object than $\tn{Ch}_X^p$ for $p>1$.  

The following commutative diagram, with exact upper row, illustrates some of the relationships among $\tn{Ch}_X^1$ and related objects.  This diagram is adapted from the notes of James D. Lewis \cite{LewisLecturesonCycles00}.  A similar diagram appears in Fulton \cite{FultonIntersection}, chapter 19, page 385.  This diagram extends the diagram for zero-cycles on curves appearing in figure \hyperref[figrationallewisrelationships]{\ref{figrationallewisrelationships}} of section \hyperref[subsectionJacPic]{\ref{subsectionJacPic}} above. 

\begin{figure}[H]
\begin{pgfpicture}{0cm}{0cm}{17cm}{4.5cm}
\begin{pgfmagnify}{1}{1}
\begin{pgftranslate}{\pgfpoint{.5cm}{-.55cm}}
\pgfputat{\pgfxy(5.5,5.3)}{\pgfbox[center,center]{$i$}}
\pgfputat{\pgfxy(8.5,5.3)}{\pgfbox[center,center]{$\delta$}}
\pgfputat{\pgfxy(1,5)}{\pgfbox[center,center]{$0$}}
\pgfputat{\pgfxy(4,5)}{\pgfbox[center,center]{$\mbox{Pic}_X^0$}}
\pgfputat{\pgfxy(7,5)}{\pgfbox[center,center]{$\mbox{Pic}_X$}}
\pgfputat{\pgfxy(10,5)}{\pgfbox[center,center]{$H_{X,\ZZ}^2$}}
\pgfputat{\pgfxy(13,5)}{\pgfbox[center,center]{$H_X^{0,2}$}}
\pgfputat{\pgfxy(1,3)}{\pgfbox[center,center]{$\mbox{Ch}_{X,\mbox{\footnotesize{hom}}}^1$}}
\pgfputat{\pgfxy(4,3)}{\pgfbox[center,center]{$\mbox{J}_X$}}
\pgfputat{\pgfxy(7,3)}{\pgfbox[center,center]{$\mbox{Ch}_{X}^1$}}
\pgfputat{\pgfxy(10,3)}{\pgfbox[center,center]{$H_{X,\CC}^2$}}
\pgfputat{\pgfxy(1,1)}{\pgfbox[center,center]{$\mbox{Ch}_{X,\mbox{\footnotesize{alg}}}^1$}}
\pgfxyline(.95,2.5)(.95,1.5)
\pgfxyline(1.05,2.5)(1.05,1.5)
\pgfxyline(3.95,4.5)(3.95,3.5)
\pgfxyline(4.05,4.5)(4.05,3.5)
\pgfxyline(6.95,4.5)(6.95,3.5)
\pgfxyline(7.05,4.5)(7.05,3.5)
\pgfsetendarrow{\pgfarrowpointed{3pt}}
\pgfxyline(1.3,5)(3.4,5)
\pgfxyline(4.5,5)(6.4,5)
\pgfxyline(7.6,5)(9.3,5)
\pgfxyline(10.6,5)(12.4,5)
\pgfxyline(1.3,3.4)(3.5,4.6)
\pgfxyline(2,3)(3.3,3)
\pgfxyline(10,4.6)(10,3.5)
\end{pgftranslate}
\end{pgfmagnify}
\end{pgfpicture}
\caption{$\tn{Ch}_X^1$ and related objects.}
\label{figCh1relatedobjects}
\end{figure}
\vspace*{-.5cm}

The top row of this diagram arises by taking analytic sheaf cohomology of the {\it exponential sheaf sequence} and using Serre's GAGA results; see Fulton \cite{FultonIntersection}, chapter 19, for details.  The groups $\tn{Ch}_{X,\tn{hom}}^p$ and $\tn{Ch}_{X,\tn{alg}}^p$ are included in the diagram to highlight the fact that homological and algebraic equivalence coincide in codimension one.  This relationship fails spectacularly in higher codimension, as discussed in section \hyperref[subsectionzerocyclessurfaces]{\ref{subsectionzerocyclessurfaces}} below.

\subsection{First Hard Cases: $\tn{Ch}_X^2$ of Low-Dimensional Varieties}\label{subsectionzerocyclessurfaces}

The simplest case of algebraic cycles of codimension greater than one is the case of codimension-two cycles; i.e., zero-cycles, on  low-dimensional varieties, such as smooth algebraic surfaces, three-folds, and four-folds.   In the case of surfaces, since the subvarieties generating codimension-two cycles are zero-dimensional points, with trivial internal structure, one might expect this case to be reasonably simple.   One might even expect this case to be no worse than the case of curves on a surface, which is the second-simplest case of divisors.  However, the theory of zero-cycles on algebraic surfaces is vastly more complicated than the theory of divisors in any dimension, exhibiting qualitative differences that require entirely new techniques to analyze.  These difficulties only become worse if one considers the codimension-two case for three-folds or four-folds, or, more generally, the cases of higher dimensions and codimensions. 

In 1968, David Mumford \cite{MumfordZeroCycles69} proved a result on the rational equivalence of zero-cycles on algebraic surfaces that began to reveal the general intractability of algebraic cycles of higher codimensions.   What Mumford proved is that the second Chow group $\mbox{Ch}_X^2$ of an algebraic surface $X$ of positive geometric genus is ``infinite-dimensional."  The precise meaning of this statement may be understood in terms of the following three equivalent ``definitions" of finite-dimensionality of $\tn{Ch}_X^2$, paraphrasing Mumford:

\begin{enumerate}
\item Every zero-cycle of sufficiently large degree on $X$ is rationally equivalent to some effective zero-cycle. 
\item Every zero-cycle on $X$ is rationally equivalent to a difference of effective zero-cycles of some sufficiently large fixed degree. 
\item The subvariety (in an appropriate symmetric product of $X$) of zero-cycles rationally equivalent to any given effective zero-cycle on $X$ has codimension less than some sufficiently large fixed positive integer. 
\end{enumerate}

The equivalence of these three statements is a preliminary result in Mumford's paper; he then proceeds to show that these statements fail to hold in the general case.  Mumford's result was considered surprising at the time it was proven.   The Italian school of algebraic geometry, including Castlenuovo, Enriques, and Severi, had extensively studied algebraic surfaces since the 1880's, and had already considered rational equivalence in this context.  Severi, in particular, had reached a number of erroneous conclusions by assuming the opposite of Mumford's result, as detailed in Mumford's paper. Although Grothendieck's scheme-theoretic approach to algebraic geometry was already far advanced by this time, Mumford's result helped to cement the collapse of the ``intuitition-based approach" to algebraic geometry championed by the Italian school. 

In his  1969 paper {\it On the Periods of Certain Rational Integrals} \cite{GriffithsRationalIntegralsIandII69}, Phillip Griffiths demonstrated another way in which the case of codimension-two is more complicated than codimension-one, this time involving planes in the Fermat quintic fourfold.  Griffiths showed that from such planes, one may construct a cycle which is algebraically equivalent to zero, but {\it not} homologically equivalent to zero.  Thus, in general, algebraic and homological equivalence fail to coincide.  Due to Griffiths' result, the {\it Griffiths groups} $\tn{Griff}_X^p=Z_{X,\tn{hom}}^p/Z_{X,\tn{alg}}^p$ are named in his honor.  The Griffiths groups represent only part of the structure of the corresponding Chow groups. Clemens \cite{ClemensHommodalg83} subsequently showed that even the Griffiths groups can be ``huge," with the Chow groups even ``larger." 

Taken together, these examples illustrate that any reasonable effort to simplify or decompose the structure of the Chow groups $\tn{Ch}_X^p$ for $p>1$ is potentially interesting.  One famous approach to this problem is represented by the conjectured {\it filtrations} of the Chow groups by Bloch, Beilinson and others, which attempt to resolve the information in $\tn{Ch}_X^p$ via sequences of {\it higher cycle class maps.}  The linearization approach of Green and Griffiths is another potential way of gaining insight.

\section{Tangent Groups at the Identity of Chow Groups}\label{sectiontangentgroupsatidentity}

In this section, following Green and Griffiths, I begin to develop this linearization approach, defining the {\it tangent groups at the identity} $T\tn{Ch}_X^p$ of the Chow groups $T\tn{Ch}_X^p$.  These groups are analogous to the tangent group at the identity $T\tn{Ch}_X^p$ of the first Chow group $\tn{Ch}_X^1$ of a smooth complex projective curve, discussed in section \hyperref[SectionInfChow]{\ref{SectionInfChow}} above.  In the general case, the version of the tangent groups I introduce, spelled out in definition \hyperref[defitangentgroupchow]{\ref{defitangentgroupchow}} below, differs from that of Green and Griffiths, essentially because I  use a more modern version of algebraic $K$-theory.  The approach I use in this section  does not begin with elementary algebraic and geometric considerations, like those I used in  the context of algebraic curves.  Instead, it uses the general machinery of {\it tangent functors,} together with Bloch's theorem  \hyperref[blochstheoremintro]{\ref{blochstheoremintro}}, which expresses the Chow groups of $X$ as Zariski sheaf cohomology groups of the algebraic $K$-theory sheaves on $X$.  At the end of section \hyperref[subsectionabstractconiveauCH1]{\ref{subsectionabstractconiveauCH1}} above, I showed how this sheaf-cohomology perspective may be recognized {\it a posteriori} as arising from the coniveau machine.  

To motivate the somewhat abstract definition of $T\tn{Ch}_X^p$, I take a brief detour in section \hyperref[subsectionlieanalogue]{\ref{subsectionlieanalogue}} to present a well-known but illustrative analogue from a different field of mathematics; namely the relationship between Lie groups and Lie algebras.  In particular, I work out the elementary result that the tangent functor of the general linear group functor $R\mapsto \tn{GL}_{n,R}$ is the functor $R\mapsto \mfr{gl}_{n,R}$.  Afterwards, in section \hyperref[sectiondefitangroup]{\ref{sectiondefitangroup}}, I define $T\tn{Ch}_X^p$, and compare the definition to that of Green and Griffiths. 


\subsection{Lie Theory: an Illustrative Analogue}\label{subsectionlieanalogue}

{\bf Real Lie Theory.}  In real differential geometry, the objects of study are real differentiable manifolds.  The {\bf tangent space} $T_xM$ at an element $x$ of a smooth manifold $M$ is usually defined to be the real vector space of {\bf $\RR$-derivations} from the algebra $C_x^\infty$ of germs of smooth functions on $M$ at $x$ into $\RR$.  Elements of $T_xM$ are called {\bf tangent vectors} at $x$ in $M$.   Alternatively, the tangent space $T_xM$ may be defined in terms of {\bf arcs through $x$ in $M$.}  These arcs are morphisms of smooth manifolds, so the notion of arcs is native to the category.   An arc through $x$ in $M$ induces an $\RR$-derivation on $C_x^\infty$, defined by sending the germ of a smooth function to its directional derivative along the arc.     Conversely, every $\RR$-derivation on  $C_x^\infty$ arises in this way.   Thus, tangent vectors may be viewed as equivalence classes of arcs through $x$ in $M$, where two arcs are equivalent if and only if they induce the same derivation.   

In real differential Lie theory, the objects of study are real Lie groups, which possess both the structure of a differentiable manifold and an independent group structure, generally noncommutative.  The group operation on a real Lie group $G$ is required to be a smooth map $G\times G\rightarrow G$, so that all operations on $G$ are native to the category of smooth manifolds.    The {\bf Lie algebra} $\mfr{g}$ of $G$ is the tangent space $T_{e}G$ of $G$ at the identity $e$ of $G$, endowed with a generally nonassociative and noncommutative operation called the {\bf Lie bracket}, induced by the group structure of $G$.  Data such as arcs and tangent vectors may be transported between different elements of $G$ by means of the group operation.  For example, if $g(t)$ is an arc through some element $g=g(0)$ of $G$, then any arc of the form $hg(t)k$, where $hgk=e$, is an arc through the identity in $G$.   These arcs are generally distinct for different choices of $h$ and $k$, and have distinct tangent vectors.   Figure \hyperref[figCh1relatedobjects]{\ref{figCh1relatedobjects}} below depicts transportation of an arc, along with its tangent vector, from an arbitrary element $g$ of $G$ to the identity $e$ of $G$:

\vspace*{4cm}

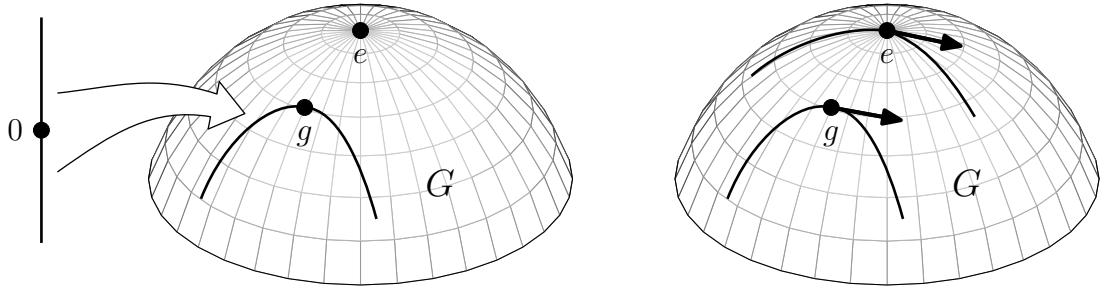
\begin{figure}[H]
\begin{tikzpicture} [scale=1.5, isometricXYZ, line join=round,
        opacity=1, text opacity=1.0,
        >=latex,
        inner sep=0pt,
        outer sep=2pt,
    ]
    \def\h{5}
\newcommand{\horcircle}[3]{
        \foreach \t in {#1} 
            \draw [color=#2]
                  ({2*sin(#3 - \h)*cos(\t - \h)}, {2*sin(#3 - \h)*sin(\t - \h)}, {2*cos(#3 - \h)})
               -- ({2*sin(#3 - \h)*cos(\t + \h)}, {2*sin(#3 - \h)*sin(\t + \h)}, {2*cos(#3 - \h)});
    }
\newcommand{\quandrantup}[4]{
        \foreach \t in {#1} \foreach \f in {#2}
            \draw [color=#3, fill=#4]
                  ({2*sin(\f - \h)*cos(\t - \h)}, {2*sin(\f - \h)*sin(\t - \h)}, {2*cos(\f - \h)})
               -- ({2*sin(\f - \h)*cos(\t + \h)}, {2*sin(\f - \h)*sin(\t + \h)}, {2*cos(\f - \h)})
               -- ({2*sin(\f + \h)*cos(\t + \h)}, {2*sin(\f + \h)*sin(\t + \h)}, {2*cos(\f + \h)})
               -- ({2*sin(\f + \h)*cos(\t - \h)}, {2*sin(\f + \h)*sin(\t - \h)}, {2*cos(\f + \h)})
               -- cycle;
    }
\begin{pgftranslate}{\pgfpoint{6cm}{.6cm}}
\quandrantup{300}{65}{black!100}{black!0}
\quandrantup{290}{65}{black!100}{black!0}
\quandrantup{280,270}{65,55,...,45}{black!100}{black!0}
\quandrantup{260,250}{65,55,...,35}{black!100}{black!0}
\quandrantup{240,230}{65,55,...,25}{black!100}{black!0}
\quandrantup{220,210}{65,55,...,25}{black!100}{black!0}
\quandrantup{200,190}{65,55,...,35}{black!100}{black!0}
\quandrantup{180,170}{65,55,...,45}{black!100}{black!0}
\quandrantup{160,150}{65}{black!100}{black!0}
\quandrantup{140,130}{65}{black!70}{black!0}
\quandrantup{120,110}{65}{black!50}{black!0}
\quandrantup{100,90,...,-10}{65}{black!40}{black!0}
\quandrantup{-20,-30}{65}{black!50}{black!0}
\quandrantup{-40,-50}{65}{black!70}{black!0}
\quandrantup{160,150,140}{55}{black!70}{black!0}
\quandrantup{130}{55}{black!50}{black!0}
\quandrantup{120,110}{55}{black!40}{black!0}
\quandrantup{100,90,...,-10}{55}{black!30}{black!0}
\quandrantup{-20,-30}{55}{black!40}{black!0}
\quandrantup{-40}{55}{black!50}{black!0}
\quandrantup{-50,-60,-70}{55}{black!70}{black!0}
\quandrantup{160,150}{45}{black!60}{black!0}
\quandrantup{140,130}{45}{black!50}{black!0}
\quandrantup{120,110}{45}{black!40}{black!0}
\quandrantup{100,90}{45}{black!30}{black!0}
\quandrantup{80,70,...,10}{45}{black!20}{black!0}
\quandrantup{0,-10}{45}{black!30}{black!0}
\quandrantup{-20,-30}{45}{black!40}{black!0}
\quandrantup{-40,-50}{45}{black!50}{black!0}
\quandrantup{-60,-70}{45}{black!60}{black!0}
\quandrantup{180,170}{35}{black!70}{black!0}
\quandrantup{160,150}{35}{black!50}{black!0}
\quandrantup{140,130}{35}{black!40}{black!0}
\quandrantup{120,110}{35}{black!30}{black!0}
\quandrantup{100,90,...,-10}{35}{black!20}{black!0}
\quandrantup{-20,-30}{35}{black!30}{black!0}
\quandrantup{-40,-50}{35}{black!40}{black!0}
\quandrantup{-60,-70}{35}{black!50}{black!0}
\quandrantup{-80,-90}{35}{black!70}{black!0}
\quandrantup{200,190}{25}{black!70}{black!0}
\quandrantup{180,170}{25}{black!50}{black!0}
\quandrantup{160,150}{25}{black!40}{black!0}
\quandrantup{140,130}{25}{black!30}{black!0}
\quandrantup{120,110}{25}{black!20}{black!0}
\quandrantup{100,90,...,-10}{25}{black!20}{black!0}
\quandrantup{-20,-30}{25}{black!20}{black!0}
\quandrantup{-40,-50}{25}{black!30}{black!0}
\quandrantup{-60,-70}{25}{black!40}{black!0}
\quandrantup{-80,-90}{25}{black!50}{black!0}
\quandrantup{-100,-110}{25}{black!70}{black!0}
\quandrantup{220,210}{15}{black!50}{black!0}
\quandrantup{200,190}{15}{black!40}{black!0}
\quandrantup{180,170}{15}{black!40}{black!0}
\quandrantup{160,150}{15}{black!30}{black!0}
\quandrantup{140,130}{15}{black!30}{black!0}
\quandrantup{120,110}{15}{black!20}{black!0}
\quandrantup{100,90,...,-10}{15}{black!20}{black!0}
\quandrantup{-20,-30}{15}{black!20}{black!0}
\quandrantup{-40,-50}{15}{black!30}{black!0}
\quandrantup{-60,-70}{15}{black!30}{black!0}
\quandrantup{-80,-90}{15}{black!40}{black!0}
\quandrantup{-100,-110}{15}{black!40}{black!0}
\quandrantup{-120,-130}{15}{black!50}{black!0}
\quandrantup{220,210}{5}{black!40}{black!0}
\quandrantup{200,190}{5}{black!30}{black!0}
\quandrantup{180,170}{5}{black!30}{black!0}
\quandrantup{160,150}{5}{black!30}{black!0}
\quandrantup{140,130}{5}{black!20}{black!0}
\quandrantup{120,110}{5}{black!20}{black!0}
\quandrantup{100,90,...,-10}{5}{black!20}{black!0}
\quandrantup{-20,-30}{5}{black!20}{black!0}
\quandrantup{-40,-50}{5}{black!20}{black!0}
\quandrantup{-60,-70}{5}{black!30}{black!0}
\quandrantup{-80,-90}{5}{black!30}{black!0}
\quandrantup{-100,-110}{5}{black!30}{black!0}
\quandrantup{-120,-130}{5}{black!40}{black!0}
\horcircle{-50,-40,...,140}{black!100}{75}
\draw [line width=1, opacity=1] (0,-4,-1.3) -- (0,-4,.7);
\node at (0,-4,-.3) [circle,opacity=1,draw=black,fill=black,minimum size=6pt,label=left:$0$] {};
\draw[line width=1pt, opacity=1] (0,-2,-.2) .. controls (0,-1.7,.45)  and (0,-1.1,1) .. (0,-.7,1.07);
\draw[line width=1pt, opacity=1] (0,-.7,1.07) .. controls (0,-.3,1.2)  and (0,0,.85) .. (0,.2,.4);
\node at (0,-.7,1.07) [circle,opacity=1,draw=black,fill=black,minimum size=6pt, label=below:$g$] {};
\node at (0,0,2) [circle,opacity=1,draw=black,fill=black,minimum size=6pt,label=below:$e$] {};
\draw[line width=.5pt, opacity=1, fill=white] 
(0,-3.8,.1) .. controls (0,-3.1,.4)  and (0,-2.6,.7) .. (0,-1.8,.8)
--(0,-1.77,.92)
--(0,-1.45,.75)
--(0,-1.91,.45)
--(0,-1.88,.57)
--(0,-1.88,.57).. controls (0,-2.6,0.4)  and (0,-3.1,0) .. (0,-3.8,-.6);
--cycle;
\draw (0,1,1) node {\large{$G$}};
\end{pgftranslate}
\begin{pgftranslate}{\pgfpoint{13cm}{.6cm}}
\quandrantup{300}{65}{black!100}{black!0}
\quandrantup{290}{65}{black!100}{black!0}
\quandrantup{280,270}{65,55,...,45}{black!100}{black!0}
\quandrantup{260,250}{65,55,...,35}{black!100}{black!0}
\quandrantup{240,230}{65,55,...,25}{black!100}{black!0}
\quandrantup{220,210}{65,55,...,25}{black!100}{black!0}
\quandrantup{200,190}{65,55,...,35}{black!100}{black!0}
\quandrantup{180,170}{65,55,...,45}{black!100}{black!0}
\quandrantup{160,150}{65}{black!100}{black!0}
\quandrantup{140,130}{65}{black!70}{black!0}
\quandrantup{120,110}{65}{black!50}{black!0}
\quandrantup{100,90,...,-10}{65}{black!40}{black!0}
\quandrantup{-20,-30}{65}{black!50}{black!0}
\quandrantup{-40,-50}{65}{black!70}{black!0}
\quandrantup{160,150,140}{55}{black!70}{black!0}
\quandrantup{130}{55}{black!50}{black!0}
\quandrantup{120,110}{55}{black!40}{black!0}
\quandrantup{100,90,...,-10}{55}{black!30}{black!0}
\quandrantup{-20,-30}{55}{black!40}{black!0}
\quandrantup{-40}{55}{black!50}{black!0}
\quandrantup{-50,-60,-70}{55}{black!70}{black!0}
\quandrantup{160,150}{45}{black!60}{black!0}
\quandrantup{140,130}{45}{black!50}{black!0}
\quandrantup{120,110}{45}{black!40}{black!0}
\quandrantup{100,90}{45}{black!30}{black!0}
\quandrantup{80,70,...,10}{45}{black!20}{black!0}
\quandrantup{0,-10}{45}{black!30}{black!0}
\quandrantup{-20,-30}{45}{black!40}{black!0}
\quandrantup{-40,-50}{45}{black!50}{black!0}
\quandrantup{-60,-70}{45}{black!60}{black!0}
\quandrantup{180,170}{35}{black!70}{black!0}
\quandrantup{160,150}{35}{black!50}{black!0}
\quandrantup{140,130}{35}{black!40}{black!0}
\quandrantup{120,110}{35}{black!30}{black!0}
\quandrantup{100,90,...,-10}{35}{black!20}{black!0}
\quandrantup{-20,-30}{35}{black!30}{black!0}
\quandrantup{-40,-50}{35}{black!40}{black!0}
\quandrantup{-60,-70}{35}{black!50}{black!0}
\quandrantup{-80,-90}{35}{black!70}{black!0}
\quandrantup{200,190}{25}{black!70}{black!0}
\quandrantup{180,170}{25}{black!50}{black!0}
\quandrantup{160,150}{25}{black!40}{black!0}
\quandrantup{140,130}{25}{black!30}{black!0}
\quandrantup{120,110}{25}{black!20}{black!0}
\quandrantup{100,90,...,-10}{25}{black!20}{black!0}
\quandrantup{-20,-30}{25}{black!20}{black!0}
\quandrantup{-40,-50}{25}{black!30}{black!0}
\quandrantup{-60,-70}{25}{black!40}{black!0}
\quandrantup{-80,-90}{25}{black!50}{black!0}
\quandrantup{-100,-110}{25}{black!70}{black!0}
\quandrantup{220,210}{15}{black!50}{black!0}
\quandrantup{200,190}{15}{black!40}{black!0}
\quandrantup{180,170}{15}{black!40}{black!0}
\quandrantup{160,150}{15}{black!30}{black!0}
\quandrantup{140,130}{15}{black!30}{black!0}
\quandrantup{120,110}{15}{black!20}{black!0}
\quandrantup{100,90,...,-10}{15}{black!20}{black!0}
\quandrantup{-20,-30}{15}{black!20}{black!0}
\quandrantup{-40,-50}{15}{black!30}{black!0}
\quandrantup{-60,-70}{15}{black!30}{black!0}
\quandrantup{-80,-90}{15}{black!40}{black!0}
\quandrantup{-100,-110}{15}{black!40}{black!0}
\quandrantup{-120,-130}{15}{black!50}{black!0}
\quandrantup{220,210}{5}{black!40}{black!0}
\quandrantup{200,190}{5}{black!30}{black!0}
\quandrantup{180,170}{5}{black!30}{black!0}
\quandrantup{160,150}{5}{black!30}{black!0}
\quandrantup{140,130}{5}{black!20}{black!0}
\quandrantup{120,110}{5}{black!20}{black!0}
\quandrantup{100,90,...,-10}{5}{black!20}{black!0}
\quandrantup{-20,-30}{5}{black!20}{black!0}
\quandrantup{-40,-50}{5}{black!20}{black!0}
\quandrantup{-60,-70}{5}{black!30}{black!0}
\quandrantup{-80,-90}{5}{black!30}{black!0}
\quandrantup{-100,-110}{5}{black!30}{black!0}
\quandrantup{-120,-130}{5}{black!40}{black!0}
\horcircle{-50,-40,...,140}{black!100}{75}
\draw[line width=1pt, opacity=1] (0,-2,-.2) .. controls (0,-1.7,.45)  and (0,-1.1,1) .. (0,-.7,1.07);
\draw[line width=1pt, opacity=1] (0,-.7,1.07) .. controls (0,-.3,1.2)  and (0,0,.85) .. (0,.2,.4);
\node at (0,-.7,1.07) [circle,opacity=1,draw=black,fill=black,minimum size=6pt, label=below:$g$] {};
\draw[line width=1pt, opacity=1] (.5,-1.2,1.35) .. controls (-.2,-1.5,1.25)  and (-.4,-1,1.55) .. (0,0,2);
\draw[line width=1pt, opacity=1] (0,0,2) .. controls (-.4,0,1.75)  and (-.5,.2,1.6) .. (-.6,.5,1.2);
\node at (0,0,2) [circle,opacity=1,draw=black,fill=black,minimum size=6pt,label=below:$e$] {};
\draw (0,1,1) node {\large{$G$}};
\pgfsetendarrow{\pgfarrowtriangle{3pt}}
\draw [line width=1.5, opacity=1] (0,-.7,1.07) -- (0,.1,1.25);
\draw [line width=1.5, opacity=1] (0,0,2) -- (-.6,.25,1.75);
\end{pgftranslate}
\end{tikzpicture}
\caption{Transportation of arcs and tangents in a Lie group.}
\label{figCh1relatedobjects}
\end{figure}
\vspace*{-.5cm}

The prototypical example of a real Lie group is the {\bf real general linear group} $GL_{n,\RR}$ of invertible $n\times n$ matrices over $\RR$, viewed as a submanifold of real Euclidean space $\RR^{n^2}$, with group operation given by matrix multiplication.  The Lie algebra $\mfr{gl}_{n,\RR}$ of $GL_{n,\RR}$ may be identified with the vector space of all $n\times n$ matrices over $\RR$, where the Lie bracket $[A,B]$ of two matrices $A$ and $B$ in $\mfr{gl}_{n,\RR}$ is defined to be their {\bf commutator} $AB-BA$.  

{\bf Generalization: General Linear Group Functor.}  Now for any integer $n\ge1$, and any commutative unital ring $R$, the {\bf general linear group} $\tn{GL}_{n,R}$ of $R$ is the group of invertible $n\times n$ matrices with entries in $R$.   If $\phi:R\rightarrow R'$ is a homomorphism of commutative unital rings, then there exists a group homorphism $\tn{GL}_{n,\phi}:\tn{GL}_{n,R}\rightarrow \tn{GL}_{n,R'}$ sending each matrix $\{r_{ij}\}$ to the matrix $\{\phi(r_{ij})\}$.  The assignment $\tn{GL}_n$ sending $R$ to $\tn{GL}_{n,R}$ and $\phi$ to $\tn{GL}_{n,\phi}$ is a functor from the category $\mc{R}$ of commutative unital rings to the category $\mc{G}$ of groups, called the {\bf general linear group functor}.   The case $n=1$ is the familiar {\bf multiplicative group functor}, which sends a ring $R$ to its multiplicative group $R^*$ of invertible elements. 

Consider the ring $R_\ee:=R[\ee]/\ee^2$ of dual numbers over $R$, along with the canonical morphism $R_\ee\rightarrow R$ in the category $\mc{R}$ sending $\ee$ to zero.   The functor $\tn{GL}_n$ may be applied to this homomorphism, and the kernel of the resulting group homomorphism is canonically isomorphic to the underlying additive group of the Lie algebra $\mfr{gl}_{n,R}$:
\begin{equation}\label{equliealgebratangent}\mfr{gl}_{n,R}^+\cong\mbox{Ker}\big[\tn{GL}_{n,R_\ee}\overset{\ee\to0}{\rightarrow}\tn{GL}_{n,R}\big].\end{equation}
To see this, consider more carefully the group homomorphism $\tn{GL}_{n,R_\ee}\overset{\ee\to0}{\rightarrow} \tn{GL}_{n,R}$.  An element of the source is an $n\times n$ invertible matrix over $R_\ee$.  Any such matrix may be split into a sum of the form $M+N\ee$, where $M$ and $N$ are $n\times n$ matrices over $R$.   Invertibility requires $M$ itself to be invertible, in which case the inverse exists and is of the form $M^{-1}-M^{-1}NM^{-1}\ee$.   The matrix $N$ may be any $n\times n$ matrix over $R$.  The set of such matrices is the underlying set of the Lie algebra $\mfr{gl}_{n,R}$.  The image of the matrix $M+N\ee$ is $M$, so this matrix belongs to the kernel if and only if $M$ is the $n\times n$ identity matrix $I_n$.  Therefore, the kernel is the multiplicative subgroup of matrices $\{I_n+N\ee\hspace*{.05cm}\}$ which is isomorphic to the underlying additive group of $\mfr{gl}_{n,R}$ via the map $I_n+N\ee\mapsto N$, which carries multiplication to addition. 

The assignment sending a ring $R$ to the additive group $\mfr{gl}_{n,R}^+$, and a morphism $\phi:R\rightarrow R'$ to the additive group homomorphism $\{r_{ij}\}\mapsto\{\phi(r_{ij})\}$, is itself a functor from $\mc{R}$ to $\mc{G}$, which I will denote by $\mfr{gl}_{n}^+$.  Since the relationship between the functors $\tn{GL}_{n,R}$ and $\mfr{gl}_{n,R}^+$, generalizes the relationship between a Lie group and its tangent space at the identity, it is reasonable to regard $\mfr{gl}_{n}^+$ as the tangent functor of $\tn{GL}_{n}$.  At a formal level, the tangent groups at the identity $T\tn{Ch}_X^p$ of the Chow groups $\tn{Ch}_X^p$ arise by application of an analogous tangent functor. 


\subsection{Definition of the Tangent Group $T\tn{Ch}_X^p$}\label{sectiondefitangroup}

The definition of the tangent group at the identity $T\tn{Ch}_X^p$ of the Chow group $\tn{Ch}_X^p$ is actually somewhat subtle, since it requires a choice of {\it extension of the Chow functor} $\tn{Ch}^p$.  The definition also involves algebraic $K$-theory, which I have thus far only mentioned parenthetically.   $K$-theory is discussed in more detail in section \hyperref[sectionKtheory]{\ref{sectionKtheory}}.  It is useful, however, to give the definition of $T\tn{Ch}_X^p$ here as a motivating guide to the subsequent $K$-theoretic material.  In this section, I also discuss why the choice of extension is not straightforward, and compares the definition used here to the definition of Green and Griffiths \cite{GreenGriffithsTangentSpaces05}. 


{\bf Definition of $T\tn{Ch}_X^p$.}  Bloch's theorem, first mentioned in equation \hyperref[blochstheoremintro]{\ref{blochstheoremintro}} in the introduction to chapter \hyperref[ChapterIntro]{\ref{ChapterIntro}}, expresses the Chow groups of $X$ as Zariski sheaf cohomology groups of the algebraic $K$-theory sheaves $\ms{K}_{p,X}$ on $X$:
\begin{equation}\label{equblochchowextension}\tn{Ch}_X^p=H_{\tn{\fsz{Zar}}}^p(X,\ms{K}_{p,X}).\end{equation}
As discussed below, the right-hand-side of Bloch's theorem \hyperref[equblochchowextension]{\ref{equblochchowextension}} is defined for more general choices of $X$ than the left-hand-side.  Hence, the right-hand-side of  \hyperref[equblochchowextension]{\ref{equblochchowextension}} represents an {\it extension of the Chow functor.}  This enables the following definition:

\begin{defi}\label{defitangentgroupchow} The tangent group at the identity $T\tn{Ch}_X^p$ of the $p$th Chow group $\tn{Ch}_X^p$ of a smooth algebraic variety $X$ is defined to be the image group of the tangent functor of Bloch's extension of the Chow functor \hyperref[equblochchowextension]{\ref{equblochchowextension}}, applied to $X$. This group may be expressed as the $p$th Zariski sheaf cohomology group of the corresponding tangent sheaf $T\ms{K}_p$ to algebraic $K$-theory:
\begin{equation}\label{equtangentgroupchow}T\tn{Ch}_X^p:=TH_{\tn{\fsz{Zar}}}^p(X,\ms{K}_{p,X})=H_{\tn{\fsz{Zar}}}^p(X,T\ms{K}_{p,X}).\end{equation}
Here, $\ms{K}_{p,X}$ is the sheaf of Bass-Thomason $K$-theory on $X$.\footnotemark\footnotetext{Quillen $K$-theory and Bass-Thomason $K$-theory are equivalent in this setting because $X$ is smooth.  Looking ahead, adding nilpotent elements requires a choice between the two theories, with Bass-Thomason $K$-theory  being the ``correct choice."}  The tangent sheaf $T\ms{K}_{p,X}$ is defined to be the kernel
\begin{equation}\label{equusualdeftangentfunctor}T\ms{K}_{p,X}:=\tn{Ker}\big[\ms{K}_{p,X_\ee}\rightarrow\ms{K}_{p,X}\big],\end{equation}
where $X_\ee$ means $X\times_{\tn{Spec}(k)}\tn{Spec } \big(k[\ee]/\ee^2\big)$.  
\end{defi}

Equation \hyperref[equusualdeftangentfunctor]{\ref{equusualdeftangentfunctor}} follows the ``usual definition of the tangent functor;" see for instance, \cite{BlochTangentSpace72}, page 205.  Moving the ``tangent operator" $T$ inside $H_{\tn{\fsz{Zar}}}^p$ in equation \hyperref[equtangentgroupchow]{\ref{equtangentgroupchow}} is justified because $H^p$ is a middle-exact functor.


{\bf Why not a direct definition in terms of $\tn{Ch}^p$?}\label{nonuniquechowextension}  Na\"{i}ve application of the ``usual definition of the tangent functor" suggests the following definition for $T\tn{Ch}^p$:
\begin{equation}\label{equnaivetangentchow}T\tn{Ch}_X^p\overset{?}{:=}\tn{Ker}\big[\tn{Ch}_{X_\ee}^p\rightarrow \tn{Ch}_X^p\big].\end{equation}
The problem with this definition is the {\it ambiguous meaning of $\tn{Ch}_{X_\ee}^p$} on the right-hand-side of the equation, since $X_\ee$ is a singular scheme.  For this definition to make sense, one must extend the definition of the Chow functor to a category including schemes like $X_\ee$.  The obvious choice, motivated again by Bloch's theorem, is to choose the extension
\begin{equation}\label{equnaiveextensionchow}\tn{Ch}_{X_{\ee}}^p\overset{?}{:=}H_{\tn{\fsz{Zar}}}^p(X,\ms{K}_{p,X_{\ee}}),\end{equation}
since the sheaf cohomology groups on the right-hand-side are well-defined.  However, there remains a uniqueness problem: there exist {\it other} functors besides $Y\mapsto H_{\tn{\fsz{Zar}}}^p(X,\mc{K}_{p,Y})$ extending $\tn{Ch}^p$ and defined on a category including $X_\ee$.  An important example is the functor $X\mapsto H_{\tn{\fsz{Zar}}}^p(X,\ms{K}_{p,X}^{\tn{M}})$, where $\mc{K}_p^{\tn{M}}$ is the $p$th {\it Milnor $K$-theory sheaf} on $X$, discussed in more detail in section \hyperref[sectionKtheory]{\ref{sectionKtheory}} below.  This is the extension used by Green and Griffiths.  


\label{comparegg}

{\bf Comparison to Green and Griffiths' definition using Milnor $K$-theory.}  I will now discuss the approach of Green and Griffiths in more detail, with further elaboration to come in section \hyperref[sectionKtheory]{\ref{sectionKtheory}}.  Green and Griffiths focus on the case of $\tn{Ch}_X^2$, where $X$ is a {\it smooth projective algebraic surface}, using the definition
\begin{equation}\label{equGGdef}T\tn{Ch}_{X,\tn{GG}}^2:=TH_{\tn{\fsz{Zar}}}^2(X,\ms{K}_{2,X}^{\tn{M}})=H_{\tn{\fsz{Zar}}}^2(X,T\ms{K}_2^{\tn{M}})=H_{\tn{\fsz{Zar}}}^2(X,\varOmega^1_{X/\QQ}),\end{equation}
where $\varOmega^1_{X/\QQ}$ is the sheaf of {\it absolute K\"{a}hler differentials on $X$}, discussed in more detail in section \hyperref[sectionKtheory]{\ref{sectionKtheory}}.  In this context, there is no difference between $K$-theory and Milnor $K$-theory, since the functors $K_2$ and $K_2^{\tn{M}}$ coincide for regular local rings.  However, in a more general context, the functors $X\mapsto H_{\tn{\fsz{Zar}}}^p(X,\ms{K}_{p,X}^{\tn{M}})$ {\it  contain less information} than the functors $X\mapsto H_{\tn{\fsz{Zar}}}^p(X,\ms{K}_{p,X})$, because Milnor $K$-theory only contributes to the highest-weight Adams eigenspace of Bass-Thomason $K$-theory.\footnotemark\footnotetext{In many cases, Milnor $K$-theory $\ms{K}_{p,X}^{\tn{M}}$ actually {\it coincides with} the highest-weight Adams eigenspace of $\ms{K}_{p,X}$, but this is not true in general.} For example,
\begin{equation}\label{equtangentk3}T\ms{K}_{3,X}^{\tn{M}}\cong\varOmega^2_{X/\QQ}\mbox{\hspace*{.5cm} but \hspace*{.5cm}} T\ms{K}_{3,X}\cong\varOmega^2_{X/\QQ}\oplus\ms{O}_X,\end{equation}
and the ``extra factor" $\ms{O}_X$ may lead to interesting invariants inaccessible to the approach of Green and Griffiths.  


\label{universal}

{\bf Desirability of a universal extension of $CH^p$.} In general, there are injections
\begin{equation}\label{injGGbassthomason}T\tn{Ch}_{X,\tn{GG}}^p\rightarrow T\tn{Ch}_X^p,\end{equation}
for every $X$ and $p$.  Hence, the functor $T\tn{Ch}^p$, as I have chosen to define it, ``captures at least as much information" as $T\tn{Ch}_{X,GG}^p$.  However, it is worth pointing out that I have not proven that my definition gives the {\it best possible} extension of $\tn{Ch}^p$ in this regard.  Ideally, one would like to have {\it universal extension functors} of $\tn{Ch}^p$ admitting injective maps from every other extension.   I am not sure how to define such functors; in particular, I see no {\it a priori} reason why they should come from sheaf cohomology.

\section{Algebraic $K$-Theory}\label{sectionKtheory}

The first three columns of the coniveau machine for codimension-$p$ cycles on a smooth $n$-dimensional algebraic variety $X$ are given by the Cousin resolutions of the algebraic $K$-theory sheaves $\ms{K}_{p,X}$ on $\tn{Zar}_X$.  For the first column, Daniel Quillen's version of $K$-theory suffices to describe the Cousin resolutions, since the local terms may be replaced by the $K$-groups of residue fields, using Quillen's {\it devissage} theorem.  For the second and third columns, which involve a singular ``thickenings" of $X$, {\it devissage} no longer applies, and it becomes necessary to use the nonconnective $K$-theory of Bass and Thomason.  One characteristic of this theory is the appearance of nontrivial $K$-groups in negative degrees.

\subsection{Preliminaries}\label{subsectionKtheoryprelim}

{\bf $K$-theory} is the study of certain integer-indexed families of abelian group-valued functors, called {\bf $K$-functors}, whose image objects are called the {\bf $K$-groups} of the objects in the source category.  These objects may be topological spaces, manifolds, rings, schemes, or special types of categories, such as exact categories or Waldhausen categories.   
$K$-theory organizes topological, geometric, or algebraic information about the objects in the source category, and provides useful invariants, which unfortunately are often difficult to compute.  Modern formulations of $K$-theory are often expressed in terms of a single functor into an appropriate category of topological spectra, with the integer-indexed, group-valued $K$-functors defined in a secondary fashion by taking homotopy groups.  In {\bf algebraic $K$-theory}, the objects of the source category are rings or schemes.  Important variants of algebraic $K$-theory include the symbolic $K$-theories introduced by Milnor, Dennis, Stein, Loday, and Beilinson; Quillen's $K$-theory and $G$-theory, defined in terms of topological classifying spaces; and the nonconnective $K$-theory of Bass and Thomason, based on Waldhausen's general construction.  Modern algebraic $K$-theory involves the intermediate step of associating a special category to each ring or scheme, such as the category of finitely-generated projective modules over a ring, or an appropriate category of perfect complexes of sheaves on a scheme.  General category-theoretic versions of $K$-theory are then applied, and the resulting $K$-groups are defined to be the $K$-groups of the ring or scheme. 

The foundations of $K$-theory were laid by Alexander Grothendieck in the mid 1950's, in the process of formulating his generalization of the Riemann-Roch theorem.  Characteristically, Grothendieck was thinking in very general terms, but being unsatisfied at the time with this particular facet of his work, he allowed some of his principal ideas about $K$-theory and associated category-theoretic notions to be published by others in more specific forms.   In particular, the {\bf Grothendieck group}, later recognized as the zeroth $K$-group, was introduced by Borel and Serre in their 1958 paper {\it Le th\'{e}or\`{e}me de Riemann-Roch} \cite{BorelSerreRR58}, which details Grothendieck's generalization of the Riemann-Roch theorem in the context of algebraic geometry.  This paper presents the Grothendieck group  as a particular quotient of the free abelian group generated by isomorphism classes of coherent algebraic sheaves on an algebraic variety.  Despite its algebraic beginnings, $K$-theory found many of its earliest applications in algebraic topology and differential geometry.  Atiyah, Hirzebruch,  Singer, and others developed the theory during the 1960's in the context of vector bundles on topological spaces.   The Atiyah-Singer index theorem, a further generalization of the Riemann-Roch theorem, relating analytic and topological data on a compact smooth manifold in differential geometry, is one of the major results related to this work.\footnotemark\footnotetext{On a side note, this theorem has become so prominent in superstring theory in physics that at least one string theorist has proposed renaming it the ``Atiyah-Singer string index theorem."}    

The early algebraic $K$-theory of rings, also developed in the 1960's, involves only the functors $K_0$ and $K_1$.  The theory of $K_0$ and $K_1$ is often called {\bf lower $K$-theory}.  $K_0$ sends a ring to the Grothendieck group of the set of isomorphism classes of its finitely generated projective modules.  $K_1$, introduced by Hyman Bass, sends a ring to the abelianization of its infinite general linear group.   The $K_1$-group of a ring is a generalization of its multiplicative group of invertible elements, and in many important cases the two groups are identical.  In 1970, John Milnor introduced a new functor $K_2$.  This functor sends a ring to the center of its Steinberg group.  In certain special cases, such as the case of local rings, $K_2$ admits a simple symbolic description.  {\bf Milnor $K$-theory} generalizes this special case, yielding Milnor $K$-functors $K_p^{\tn{M}}$, for all $p\ge0$.   Using the most na\"{i}ve definition,\footnotemark\footnotetext{Unfortunately, there is no consensus in the literature about how these functors should be defined for general rings.  I discuss this further in section \hyperref[subsectionsymbolic]{\ref{subsectionsymbolic}} below.} these functors assign to each ring the graded parts of a particular quotient of the tensor algebra over its multiplicative group of invertible elements.   Milnor $K$-theory is perhaps the most prominent version of {\bf symbolic $K$-theory}, which is a loose term I use here to refer to versions of $K$-theory whose $K$-groups may be conveniently expressed in terms of generators and relations, rather than requiring a homotopy-theoretic definition.  Milnor $K$-theory has far-reaching, and indeed surprising, importance, given its elementary definition.  For example, Milnor $K$-theory has deep connections to motivic cohomology.  However, Milnor $K$-theory suffers from some technical inadequacies as an extension of the lower $K$-functors $K_0$ and $K_1$.  

Daniel Quillen's 1972 paper {\it Higher algebraic $K$-theory} \cite{QuillenHigherKTheoryI72} is generally acknowledged as the beginning of modern algebraic $K$-theory.   In this paper, Quillen defined the Quillen $K$-functors $K_p^{\tn{\fsz{Q}}}$, for all $p\ge0$, which assign to each ring or scheme the homotopy groups of a particular infinite-dimensional topological classifying space.\footnotemark\footnotetext{I use the nonstandard notation $K_p^{\tn{\fsz{Q}}}$ for the Quillen functors, which are usually by just $K_p$, because the $K$-theory used in this book is usually {\it not} Quillen $K$-theory.}  In fact, Quillen gave two different $K$-theoretic constructions, called the {\it plus-construction} and the {\it $Q$-construction,} which give the same results for rings.   The plus-construction is defined directly in terms of the infinite general linear group of a ring.  The $Q$-construction is defined more generally for exact categories, and may therefore be applied to either rings or schemes, using appropriate exact categories as intermediary constructions.   For a ring, the exact category used is the category of finitely generated projective modules over the ring, and for a scheme, it is the category of locally free coherent sheaves on the scheme.   Quillen's approach produces a theory that is structurally superior to the earlier symbolic $K$-theories, but deep and difficult to analyze.  A variant of Quillen $K$-theory, called $G$-theory, may be defined for noetherian rings and schemes, by applying the $Q$-construction to different exact categories.   For a noetherian ring, the exact category used is the abelian category of all finitely-generated modules over the ring.   For a scheme, it is the abelian category of all coherent sheaves on the scheme.   For regular noetherian rings and schemes, the Quillen $K$-theory and $G$-theory functors produce the same groups, but this is not true in general.  Powerful theorems, which Quillen called the {\it devissage} and {\it localization} theorems, apply to Quillen $G$-theory, but not to Quillen $K$-theory.

Friedhelm Walhausen introduced a new version of $K$-theory in his 1985 paper {\it Algebraic $K$-Theory of Spaces} \cite{WaldhausenKTheoryofSpaces85}, based on a construction called the {\it $S$-construction.}  The $S$-construction defines a functor from a general type of category, called a {\it Waldhausen category}, into an appropriate category of topological spectra.   Waldhausen categories include exact categories as a special case, so the $S$-construction may be applied to rings and schemes as an alternative to the $Q$-construction.  In 1990, Robert Thomason defined an improved version of algebraic $K$-theory for algebraic schemes in his paper {\it Higher Algebraic $K$-Theory of Schemes and of Derived Categories} \cite{Thomason-Trobaugh90},\footnotemark\footnotetext{Thomason listed his deceased colleague Thomas Trobaugh as co-author, due to his appearance in a dream of Thomason's.} by applying Waldhausen's version of $K$-theory to a particular category of perfect complexes of sheaves on a scheme.  One characteristic of Thomason's approach is that it admits nontrivial $K$-groups in negative degrees; i.e., the associated spectrum is {\it nonconnective.}  Thomason draws heavily on similar earlier ideas of Hyman Bass \cite{BassKTheory68}.  The Bass-Thomason version of $K$-theory is the structurally ``correct" version for constructing the coniveau machine.  It is described in more detail in section \hyperref[subsubsectionbassthomason]{\ref{subsubsectionbassthomason}} below.

\subsection{Symbolic $K$-Theory}\label{subsectionsymbolic}

Milnor $K$-theory and Dennis-Stein $K$-theory are {\it symbolic $K$-theories;} an informal term which means, in this context, $K$-theories whose $K$-groups admit simple presentations via generators, called {\it symbols,} and relations.  More sophisticated ``modern" $K$-theories, such as Quillen's $K$-theory, Waldhausen's $K$-theory, and the amplification of Bass and Thomason in the context of algebraic schemes, have homotopy-theoretic definitions.  Symbolic $K$-theories have the advantage of being relatively elementary, but tend to lack certain desirable formal properties.  In this sense, they represent one extreme of the seemingly unavoidable tradeoff between accessibility and formal integrity in algebraic $K$-theory.   

In this section, I discuss Milnor $K$-theory, and mention a theorem expressing relative Milnor $K$-theory in terms of absolute K\"{a}hler differentials for a large class of commutative rings.  The primary purpose of this section is to make close contact with the approach of Green and Griffiths, who work exclusively in terms of Milnor $K$-theory.  Dennis-Stein theory plays an auxiliary role in the proof of some results about Milnor $K$-theory.  The material in this section closely follows my paper \cite{DribusMilnorK14}.


\label{subsubsectionMilnorK}

{\bf Milnor $K$-Theory in Terms of Tensor Algebras.} Milnor $K$-theory first appeared in John Milnor's 1970 paper {\it Algebraic $K$-Theory and Quadratic Forms} \cite{MilnorAlgebraicKTheoryQforms70}, in the context of fields.   Around the same time, Milnor, Steinberg, Matsumoto, Dennis, Stein, and others were studying the second $K$-group $K_{2,R}$ of a general ring $R$, defined by Milnor in 1967 as the center of the Steinberg group of $R$.  $K_{2,R}$ is often called ``Milnor's $K_2$" in honor of its discoverer, but is in fact the ``full $K_2$-group."  In particular, it is much more complicated in general than the ``second Milnor $K$-group" $K_{2,R}^{\tn{M}}$, defined below in terms of tensor algebras.  Adding further to the confusion of terminology, the two groups $K_{2,R}$ and $K_{2,R}^{\tn{M}}$ {\it are} equal in many important cases; in particular, when $R$ is a field, division ring, a local ring, or even a semilocal ring.\footnotemark\footnotetext{See \cite{WeibelKBook}, Chapter III, Theorem 5.10.5, page 43, for details.}  This result is usually called Matsumoto's theorem, since its original version was proved by Matsumoto, for fields, in an arithmetic setting.  Matsumoto's theorem was subsequently extended to division rings by Milnor, and finally to semilocal rings by Dennis and Stein. 

The following definition introduces the ``na\"{i}vest" version of Milnor $K$-theory:

\begin{defi}\label{defiMilnorK} Let $R$ be a commutative ring with identity, and let $R^*$ be its multiplicative group of invertible elements, viewed as a $\ZZ$-module. 
\begin{enumerate}
\item The {\bf Milnor $K$-ring} $K_R^{\tn{M}}$ of $R$ is the quotient
\[K_R^{\tn{M}}:=\frac{T_{R^*/\mathbb{Z}}}{I_{\tn{\fsz{St}}}}\]
of the tensor algebra $T_{R^*/\mathbb{Z}}$ by the ideal $I_{\tn{\fsz{St}}}$ generated by elements of the form $r\otimes(1-r)$.  
\item The $n$th {\bf Milnor $K$-group} $K_{R,n} ^{\tn{M}}$ of $R$, defined for $n\ge0$, is the $n$th graded piece of $K_R^{\tn{M}}$. 
\end{enumerate}
\end{defi}

For a general ring $k$, the tensor algebra $T_{R^*/\mathbb{Z}}$ of $R$ over $k$ is by definition the graded $k$-algebra whose zeroth graded piece is $k$, whose $n$th graded piece is the $n$-fold tensor product $R\otimes_k...\otimes_kR$ for $n\ge1$, and whose multiplicative operation is induced by the tensor product.  The subscript ``$\tn{St}$" assigned to the ideal $I_{\tn{\fsz{St}}}$ stands for ``Steinberg," since the defining relations $r\otimes(1-r)\sim0$ of $K_R^{\tn{M}}$ are called {\it Steinberg relations}.  The ring $K_R^{\tn{M}}$ is noncommutative, since concatenation of tensor products is noncommutative; more specifically, it is {\it anticommutative} if $R$ has ``enough units," in a sense made precise below.  The $n$th Milnor $K$-group $K_{R,n} ^{\tn{M}}$ is generated, under {\it addition} in $K_R^{\tn{M}}$, by equivalence classes of $n$-fold tensors $r_1\otimes...\otimes r_n$.   Such equivalence classes are denoted by symbols $\{r_1,...,r_n\}$, called {\it Steinberg symbols}.   When working with individual Milnor $K$-groups, the operation is usually viewed {\it multiplicatively,} and the identity element is usually denoted by $1$.   

There is no consensus in the literature about how the Milnor $K$-groups $K_n^{\tn{M}}(R)$ should be defined for general $n$ and $R$.   The definition I use here is the most na\"{i}ve one.  Its claim to relevance relies on foundational work by Steinberg, Milnor, Matsumoto, and others.  Historically, the Steinberg symbol arose as a map $R^*\times R^*\rightarrow K_{2,R}$, defined in terms of special matrices.   The properties of this map, including the relations satisfied by its images, may be analyzed concretely in terms of matrix properties.\footnotemark\footnotetext{See Weibel \cite{WeibelKBook} Chapter III, or Rosenberg \cite{RosenbergK94} Chapter 4 for details.}  In the case where $R$ is a field, Matsumoto's theorem states that the image of the Steinberg symbol map generates $K_{2,R}$, and that all the relations satisfied by elements of the image follow from the relations of the tensor product and the Steinberg relations.   This allows a simple re-definition of $K_{2,R}$ in terms of a tensor algebras when $R$ is a field, with the generators {\it renamed} Steinberg symbols.  Abstracting this result to general $n$ and $R$ leads to definition \hyperref[defiMilnorK]{\ref{defiMilnorK}} above.   However, it has been understood from the beginning that the resulting ``Milnor $K$-theory" is seriously deficient in many respects.   Quillen, Waldhausen, Bass, Thomason, and others have since addressed many of these deficiencies by defining more elaborate versions of $K$-theory, but there still remain many reasons why symbolic $K$-theories are of interest.  For example, they are closely connected to motivic cohomology, provide interesting approaches to the study of Chow groups and higher Chow groups, and arise in physical settings in superstring theory and elsewhere.  The viewpoint of the present paper, involving $\lambda$-decompositions, cyclic homology, and differential forms, is partly motivated by these considerations, particularly the theory of Chow groups. 

It is instructive to briefly examine a few different treatments of Milnor $K$-theory in the literature.  Weibel \cite{WeibelKBook} chooses to confine his definition of Milnor $K$-theory to the original context of fields (Chapter III, section 7), while defining Steinberg symbols more generally (Chapter IV, example 1.10.1, page 8), and also discussing many other types of symbols, including Dennis-Stein symbols (Chapter III, defnition 5.11, page 43), and Loday symbols (Chapter IV, exercise 1.22, page 122).\footnotemark\footnotetext{Interestingly, the Loday symbols project nontrivially into a range of different pieces of the $\lambda$-decomposition of Quillen $K$-theory.  See Weibel \cite{WeibelKBook} Chapter IV, example 5.11.1, page 52, for details.} Elbaz-Vincent and M\"{u}ller-Stach \cite{ElbazVincentMilnor02} define Milnor $K$-theory for general rings (Definition 1.1, page 180) in terms of generators and relations, but take the additive inverse relation of lemma \hyperref[lemrelationsstable]{\ref{lemrelationsstable}} as part of the definition.   The result is a generally nontrivial quotient of the Milnor $K$-theory of definition \hyperref[defiMilnorK]{\ref{defiMilnorK}} above.\\

\begin{example}\tn{Let $R=\ZZ_2[x]/x^2$.   The multiplicative group $R^*$ is isomorphic to $\ZZ_2$, generated by the element $r:=1+x$.  The Steinberg ideal is empty since $1-r$ is not a unit.  Hence, $K_{2,R}^{\tn{M}}$ is just $R^*\otimes_\ZZ R^*\cong\ZZ_2$, generated by the symbol $\{r,r\}=\{r,-r\}$, while Elbaz-Vincent and M\"{u}ller-Stach's corresponding group is trivial.   By contrast, the additive inverse relation $\{r,-r\}=1$ always holds if one uses the original definition of Steinberg symbols in terms of matrices; see Weibel \cite{WeibelKBook} Chapter III, remark 5.10.4, page 43.  This may be interpreted as an indication that this relation is a desirable property for ``enhanced" versions of  Milnor $K$-theory.}
\end{example}
\hspace{16.3cm} $\oblong$

Moritz Kerz \cite{KerzMilnorLocal} has suggested an ``improved version of Milnor $K$-theory," motivated by a desire to correct certain formal shortcomings of the ``na\"{i}ve" version defined in terms of the tensor product.  For example, this na\"{i}ve version fails to satisfy the {\it Gersten conjecture}.   Thomason \cite{ThomasonNoMilnor92} has shown that Milnor $K$-theory does not extend to a theory of smooth algebraic varieties with desirable properties such as $\mbb{A}^1$-homotopy invariance and functorial homomorphisms to more complete version of $K$-theory. Hence, the proper choice of definition depends on what properties and applications one wishes to study.


\label{subsubsectionMilnorKgenerators}

{\bf Milnor $K$-Theory in Terms Generators and Relations.}  It is sometimes convenient to forget about the tensor algebra definition \hyperref[defiMilnorK]{\ref{defiMilnorK}} of Milnor $K$-theory, and describe the groups $K_{n,R} ^{\tn{M}}$ abstractly in terms of generators and relations.  The generators are the Steinberg symbols discussed in section \hyperref[subsubsectionMilnorK]{\ref{subsubsectionMilnorK}} above, and the relations are those arising from the tensor algebra and the Steinberg ideal $I_{\tn{\fsz{St}}}$.  

\begin{lem}\label{lemMilnorrelations}As an abstract multiplicative group, $K_{n,R} ^{\tn{M}}$ is generated by the Steinberg symbols $\{r_1,...,r_n\}$, where $r_j\in R^*$ for all $j$, subject to the relations
\begin{enumerate}
\addtocounter{enumi}{-1}
\item $K_{n,R} ^{\tn{M}}$ is abelian.
\item Multiplicative relation: $\{...,r_jr_j',...\}\{...,r_j,...\}^{-1}\{...,r_j',...\}^{-1}=1$.
\item Steinberg relation: $\{...,r,1-r,...\}=1$.
\end{enumerate}
\end{lem}
\begin{proof} This follows directly from definition \hyperref[defiMilnorK]{\ref{defiMilnorK}} and the properties of the tensor algebra. 
\end{proof}
In the Steinberg relation, the elements $r$ and $1-r$ may appear in any pair of {\it consecutive} entries of the Steinberg symbol $\{r_1,...,r_n\}$.

\label{subsubsectionMilnorKfirstfew}

\begin{example}\label{examplefirstfewmilnor} (The First Few Milnor $K$-Groups.) \tn{For future reference, I explicitly describe the first few Milnor $K$-groups of a commutative ring:}

\tn{The {\bf zeroth Milnor $K$-group} $K_{0,R}^{\tn{M}}$ of $R$ is equal to the underlying additive group $\ZZ^+$ of the integers, since $\ZZ$ is the ring over which the tensor algebra $T_{R^*/\mathbb{Z}}$ is defined.   If $R$ is a local ring, the zeroth Milnor $K$-group of $R$ is isomorphic to the Grothendieck group $K_{0,R}$ of $R$.   This is because every finitely generated projective module over a local ring is free, and is therefore entirely specified by its rank.  Therefore, the rank map $K_{0,R}\rightarrow\ZZ=K_{0,R}^{\tn{M}}$ is an isomorphism.   This relationship breaks down for nonlocal rings.  For instance, suppose $R$ is the ring of algebraic integers in a number field $F$.   This means that $F$ is a finite algebraic extension of the rational numbers, and $R$ is the integral closure of $\ZZ$ in $F$.  Then the Grothendieck group $K_{0,R}$ of $R$ is the direct sum $\ZZ\oplus C_R$, where $C_R$ is the ideal class group of $R$.   Since it is the zeroth graded piece of the graded ring $K_{R}^{\tn{M}}$, the zeroth Milnor $K$-group $K_{0}^{\tn{M}}$ is itself a ring, with multiplication inherited from $K_{R}^{\tn{M}}$.  This multiplication is the usual multiplication in $\ZZ$.}

\tn{The {\bf first Milnor $K$-group} $K_{1,R}^{\tn{M}}$ of $R$ is the multiplicative group $R^*$ of invertible elements of $R$, with group operation given by the usual multiplication in $R^*$.  This means that {\it addition} in the Milnor $K$-ring $K_{R}^{\tn{M}}$, restricted to the first graded piece $K_{1,R}^{\tn{M}}$, corresponds to {\it multiplication} in the original ring $R$; i.e., $\{r\}+\{r'\}=\{rr'\}$.   Multiplication in the Milnor $K$-ring arises from the tensor product, and has nothing to do with multiplication in $R$; for example, $\{r\}\times\{r'\}=\{r,r'\}$.   If $R$ is a commutative local ring, the first Milnor $K$-group of $R$ is isomorphic to the first $K$-group $K_{1,R}$ of $R$, as defined by Bass.   The isomorphism comes from the determinant map $GL_R\rightarrow R^*$, which descends to the abelianization $K_{1,R}$ of $GL_R$.   This relationship breaks down for nonlocal rings.   In general, $K_{1,R}$ is the direct sum $R^*\oplus SK_{1,R}$, where $SK_{1,R}$ is the abelianization of the special linear group $SL_R$.   For instance, if $R$ is the ring of algebraic integers in a number field $F$, the elements of $SK_{1,R}$ are given by {\it Mennicke} symbols, which are classes of $2\times 2$ matrices in $SK_{1,R}$.} 

\tn{The {\bf second Milnor $K$-group} $K_{2,R}^{\tn{M}}$ of $R$ was discussed in section \hyperref[subsubsectionMilnorK]{\ref{subsubsectionMilnorK}} above.  It is generated by Steinberg symbols of the form $\{r,r'\}$, where $r$ and $r'$ are invertible elements of $R$.   Since the ideal $I_{\tn{\footnotesize{St}}}$, appearing in the definition $K_R ^{\tn{M}}=T_{R^*/\mathbb{Z}}/{I_{\tn{\footnotesize{St}}}}$ of the Milnor $K$-ring, is generated by $2$-tensors of the form $r\otimes(1-r)$, the ``lowest-degree effects" of taking the quotient of the tensor algebra by $I_{\tn{\footnotesize{St}}}$ appear in the second graded piece of $K_{2,R}^{\tn{M}}$ of $K_R ^{\tn{M}}$.}

\end{example}

\label{subsubsectionDSBL}

{\bf Dennis-Stein-Beilinson-Loday $K$-Theory.} An alternative version of symbolic $K$-theory, which I will call Dennis-Stein-Beilinson-Loday $K$-theory, plays a minor role in section \hyperref[subsectionGoodwillieMilnor]{\ref{subsectionGoodwillieMilnor}} below.  It will suffice to define only the second Dennis-Stein-Beilinson-Loday $K$-group, and the corresponding relative groups.  Dennis and Stein initially considered a symbolic version of $K_2$ in their 1973 paper {\it $K_2$ of radical ideals and semi-local rings revisited} \cite{DennisStein}.  Their definition was later generalized to higher $K$-theory by Loday and Beilinson.\footnotemark\footnotetext{Loday \cite{LodayCyclicHomology98}, section 11.2, exercise E.11.1.1 and section 11.2 exercise E.11.2.3.}
\begin{defi}\label{defidennisstein} Let $R$ be a commutative ring with identity, and let $R^*$ be its multiplicative group of invertible elements.  The second {\bf Dennis-Stein-Beilinson-Loday $K$-group} $D_{2,R}$ of $R$ is the multiplicative abelian group whose generators are symbols $\langle a,b\rangle$ for each pair of elements $a$ and $b$ in $R$ such that $1+ab\in R^*$, subject to the additional relations
\begin{enumerate}
\item $\langle a,b\rangle\langle -b,-a\rangle=1.$
\item $\langle a,b\rangle\langle a,c\rangle=\langle a,b+c+abc\rangle.$
\item $\langle a,bc\rangle=\langle ab,c\rangle\langle ac,b\rangle.$
\end{enumerate}
\end{defi}

This definition appears in Maazen and Stienstra \cite{MaazenStienstra77} definition 2.2, page 275, and Van der Kallen \cite{VanderKallenRingswithManyUnits77}, page 488.

\subsection{Nonconnective $K$-Theory of Bass and Thomason}\label{subsectionconnectivenonconnective}

{\bf Topological Spectra.}  Modern $K$-theory is often described in terms of {\it topological spectra,} which are sequences of pointed spaces, equipped with {\it structure maps} sending the reduced suspension of one space to the next space.  In this context, the familiar $K$-theory groups are the homotopy groups of $K$-theory spectra, computed via {\it stable homotopy theory.}  Section \hyperref[subsecspectra]{\ref{subsecspectra}} of the appendix presents the necessary background on spectra. 


{\bf Waldhausen's Version of $K$-Theory.}  The details of Waldhausen's construction of $K$-theory appear in Waldhausen and also in the first section of Thomason \cite{Thomason-Trobaugh90}.  I include here a thumbnail sketch for the purposes of comparison to Keller's {\it mixed negative cyclic homology}, defined in terms of {\it localization pairs}.   One begins with a complicial biWaldhausen category $\mbf{A}$, which is a special subcategory of the category of chain complexes of some abelian category $\mc{A}$.  The category $\mbf{A}$ is equipped with {\it cofibrations}, which are morphisms analogous to injections or inflations, and {\it weak equivalences}, which are morphisms analogous to quasi-isomorphisms.  See Thomason \cite{Thomason-Trobaugh90} pages 253-254 for precise statements.  

Next, one performs the ``$S$-dot construction" to obtain a simplicial category $\mbf{w}S_\bullet \mbf{A}$, whose objects in degree $n$ are functors from special partially ordered sets to $\mbf{A}$, ``roughly equivalent'' to chains $A_1\rightarrow A_2\rightarrow...\rightarrow A_n$ of length $n$ of cofibrations in $\mbf{A}$, and whose morphisms are special natural transformations.  See Thomason \cite{Thomason-Trobaugh90} pages 259-260 for precise statements.  Taking the nerve of $\mbf{w}S_\bullet \mbf{A}$ in each degree; i.e. considering separately the subcategories $\mbf{w}S_n \mbf{A}$ of $\mbf{w}S_\bullet \mbf{A}$ for each $n$ and forming the simplicial set whose $m$-simplices are chains of length $m$ of functors and natural transformations as described above yields a bisimplicial set $N_\bullet\mbf{w}S_\bullet \mbf{A}$, where the ``N" stands for ``nerve."   The loop space $\Omega|N_\bullet\mbf{w}S_\bullet \mbf{A}|$ on the geometric realization $|N_\bullet\mbf{w}S_\bullet \mbf{A}|$ of this bisimplicial set is the zeroth space of a spectrum $\mbf{K}_{\mbf{A}}$ called the $K$-theory spectrum of $\mbf{A}$.  


\label{subsubsectionperfect}

{\bf Perfect Complexes.}  The category of perfect complexes on a quasi-compact separated scheme is central both to Bass and Thomason's version of $K$-theory, and to Keller's version of negative cyclic homology for schemes.  Here I briefly introduce the definitions, along with a few guiding remarks about some of the major references, for the convenience of the reader.  In this section, $X$ is a scheme over a field $k$ of characteristic zero, with structure sheaf $\ms{O}_X$.   Most of the theory applies much more generally, but it is convenient to follow some of the references in adopting a rather specific view at present. 

 A {\bf strictly perfect complex} of $\ms{O}_X$-modules on $X$ is a bounded complex of algebraic vector bundles on $X$.  An {\bf algebraic vector bundle} is a locally free sheaf of $\ms{O}_X$-modules, of finite rank.  The definition of strictly perfect complexes I have given here is the one used by Thomason \cite{Thomason-Trobaugh90}, 2.2.2, page 285, which is drawn directly from SGA, \cite{SGA671} I, 2.1.  A {\bf bounded} complex, in this context, is simply a complex with a finite number of nonzero terms.  Thomason calls this condition ``strict boundedness," to distinguish it from homological boundedness.  A {\bf perfect complex} of $\ms{O}_X$-modules on $X$ is any complex of $\ms{O}_X$-modules locally quasi-isomorphic to a strictly perfect complex.  Here, I am using the definition used by Thomason \cite{Thomason-Trobaugh90}, 2.2.10, page 290, which is drawn, with a slight simplification, from SGA \cite{SGA671}, I, 4.7.\footnotemark\footnotetext{Note that Thomason, \cite{Thomason-Trobaugh90}, has the numbering wrong here, listing the source as \cite{SGA671} I, 4.2.} Schlichting \cite{SchlichtingKTheory11}, page 202, and Keller \cite{KellerCycHomofDGAlgebras96}, page 5, also use the same definition.    I will denote the category of perfect complexes of $\ms{O}_X$-modules on $X$ by $\mbf{Per}_X$, and the corresponding category of strictly perfect complexes by $\mbf{SPer}_X$.   

Due to certain cardinality issues involving the category $\mbf{Per}_X$, Weibel et al., \cite{WeibelCycliccdh-CohomNegativeK06}, example 2.7, page 8, work with a distinguished subcategory, which I will denote by $\mbf{Par}_X$.   This subcategory is chosen to be an exact differential graded category over the ground field $k$.  Hence, it admits Keller's machinery of localization pairs.\footnotemark\footnotetext{See Keller's papers \cite{KellerCycHomofDGAlgebras96}, \cite{KellerCyclicHomologyofExactCat96}, and \cite{KellerCyclicHomologyofSchemes98}.}  More precisely, $\mbf{Par}_X$ is defined to be the category of perfect, bounded-above complexes of flat $\ms{O}_X$-modules, whose stalks have cardinality at most equal to the cardinality of $F$, where $F$ is a distinguished extension of the ground field $k$, and where $X$ belongs to the category $\mbf{Fin}_F$ of schemes essentially of finite type over $F$.\footnotemark\footnotetext{Weibel et al., \cite{WeibelCycliccdh-CohomNegativeK06}, denote this category by $\mbf{Ch}_{\tn{parf}}(X)$,  where ``parf" stands for ``parfait;" i.e., ``perfect," a la SGA \cite{SGA671}.} Thomason \cite{Thomason-Trobaugh90}, 1.4, page 259, by contrast, makes a blanket assumption of smallness for every Waldhausen category he uses, without specifying details.  He justifies this by the fact, cited from SGA 4, that the associated $K$-theory spectra are independent of the choice of Grothedieck universe.  Thomason also restricts the category of perfect complexes to those of finite Tor amplitude\footnotemark\footnotetext{See Thomason \cite{Thomason-Trobaugh90}, definition 3.1 page 312} in his definition of the $K$-theory of schemes.  


\label{subsubsectionbassthomason}

{\bf The Nonconnective $K$-Theory of Bass and Thomason.}  For a pair $(X,Z)$, where $X$ is a quasi-compact, quasi-separated scheme of finite Krull dimension and $Z$ is a closed subspace of $X$ such that $X-Z$ is also quasi-compact, Thomason's nonconnective $K$-theory spectrum $\mathbf{K}_{X \tn{ \footnotesize{on} } Z}$ is defined (\cite{Thomason-Trobaugh90}, definition 6.4, page 360) as the homotopy colimit of a diagram 
\[F^0\rightarrow F^{-1}\rightarrow F^{-2}\rightarrow...,\]
where $F^0$ is the connective $K$-theory spectrum $\mathbf{K}^{\tn{con}}_{X \tn{ \footnotesize{on} } Z}$, which in turn is defined (\cite{Thomason-Trobaugh90}, definition 3.1, page 313) to be the $K$-theory spectrum of the complicial biWaldhausen category of those perfect complexes on $X$ which are acyclic on $X-Y$.  Several different notations for these spectra appear in the literature.\footnotemark\footnotetext{$\mathbf{K}^{con}(X \tn{ on } Z)$ is my nonstandard notation.  Thomason denotes $\mathbf{K}^{\tn{con}}(X \tn{ on } Z)$ by $K(X \tn{ on } Z)$ (\cite{Thomason-Trobaugh90} definition 3.1 page 313) and $\mathbf{K}(X \tn{ on } Z)$ by $K^{B}(X \tn{ on } Z)$ (\cite{Thomason-Trobaugh90} definition 6.4 page 360), where the ``$B$'' stands for ``Bass."  Colliot-Th\'el\`ene, Hoobler, and Kahn \cite{CHKBloch-Ogus-Gabber97}, also uses the notation $K^B$ (\cite{Thomason-Trobaugh90}, 7.4 (6), page 43), but only in the context of the cohomology (i.e. group) axioms {\bf COH1} and {\bf COH5}, not the substratum (i.e. spectrum) axioms {\bf SUB1} and {\bf SUB5}.}  In my notation, the ``con" in the superscript of $\mathbf{K}^{\tn{con}}_{X \tn{ \footnotesize{on} } Z}$ stands for ``connective." 

The nonconnective $K$-theory spectrum $\mathbf{K}_{X \tn{ \footnotesize{on} } Z}$ satisfies {\bf Thomason's localization theorem} (\cite{Thomason-Trobaugh90},Theorem 7.4, page 365), which states that there is a homotopy fiber sequence
\[\mathbf{K}_{X \tn{ \footnotesize{on} } Z}\rightarrow\mathbf{K}_{X}\rightarrow \mathbf{K}_{X-Z}.\]
This sequence makes the nonconnective $K$-theory spectrum into a {\bf substratum functor}, in the sense of definition \hyperref[defisubstratum]{\ref{defisubstratum}}, for the nonconnective $K$-theory groups $K_{p,\hspace*{.05cm}X\tn{ \footnotesize{on} } Z}$, viewed collectively as an abstract cohomology theory with supports.  The corresponding statement for the connective $K$-theory spectrum $\mathbf{K}^{\tn{con}}_{X \tn{ \footnotesize{on} } Z}$ is false because the corresponding localization theorem fails to hold.   Colliot-Th\'el\`ene, Hoobler, and Kahn remark on this in \cite{CHKBloch-Ogus-Gabber97}, 7.4 (6), page 43. 


{\bf ``Augmented $K$-Theory;" Multiplying by a Fixed Separated Scheme.} There exist a number of interesting ways of modifying algebraic $K$-theory to produce useful related functors.  One method of particular importance is to multiply each scheme by a fixed scheme before forming $K$-theory spectra.  This is made precise in the following definition:

\begin{defi}\label{augmentedKspectrum} Let $Y$ be a fixed scheme over a field $k$, and let $X$ vary over the category of $k$-schemes.  The {\bf augmented $K$-theory spectrum of $X$ with respect to $Y$} is the spectrum $\mbf{K}_{X\times_kY}$, where $X\times_kY$ is shorthand for the fiber product $X\times_{\tn{Spec }k} Y$.  The {\bf augmented $K$-theory groups of $X$} with respect to $Y$ are the homotopy groups of $\mbf{K}_{X\times_kY}$.  
\end{defi}

Of principal interest in this book is the case in which $X$ is a smooth algebraic variety and $Y$ is a separated scheme, not necessarily smooth.  In this case, multiplying by $Y$ yields the augmented version of $K$-theory used to define generalized deformation groups and generalized tangent groups of the Chow groups $\tn{Ch}_X^p$ of $X$ in definition \hyperref[defigendefgroupChow]{\ref{defigendefgroupChow}} below.


\section{Cyclic Homology}\label{sectioncyclichomology}

The fourth and final column of the coniveau machine for algebraic $K$-theory on a smooth algebraic variety $X$ over a field $k$ containing $\QQ$ is the Cousin resolution of the relative negative cyclic homology sheaf $\ms{HN}_{p,X\times_k Y,Y}$ on $X$, where $Y$ is the prime spectrum of a $k$-algebra $A$ generated over $k$ by nilpotent elements.  When $A$ is the algebra of dual numbers $k[\ee]/\ee^2$, a theorem of Hesselholt  \cite{HesselholtTruncated05} demonstrates that this relative sheaf may be expressed in terms of sheaves of absolute K\"{a}hler differentials, as follows:
\[\ms{HN}_{p,X_\ee,\ee}\cong \varOmega^{p-1}_{X/\QQ}\oplus \varOmega^{p-3}_{X/\QQ}\oplus...\]
However, I will retain the more general expression because of its conceptual advantages.  In particular, negative cyclic homology is the proper receptacle for the relative Chern character, as described in section \hyperref[sectionChern]{\ref{sectionChern}} below.    

Cyclic homology may be viewed as a linearization of $K$-theory, and is sometimes called ``additive $K$-theory" in the literature.  This view has considerable heuristic utility in the context of commutative algebraic geometry.  Another useful perspective is to view cyclic homology as a generalization of differential forms.  The idea that cyclic homology is a good place to look for infinitesimal invariants of Chow groups is partly motivated by the fact that the cyclic homology of an algebra corresponds to its algebraic $K$-theory in approximately the same way that its Lie algebra of matrices corresponds to its general linear group.  For details, see Loday \cite{LodayCyclicHomology98}, 10.2.19 and 11.2.12, and Weibel \cite{WeibelInfCohomChernNegCyclic08}, theorem A.15. 

Most actual calculations of negative cyclic homology pertaining to the infinitesimal structure of Chow groups can be performed locally, and for this, it suffices to develop the theory for local $k$-algebras.   Sections \hyperref[subsectionkahlerderham]{\ref{subsectionkahlerderham}} through \hyperref[subsectionnegcyccomring]{\ref{subsectionnegcyccomring}} below provide the standard tools for such computations.  However, the formal construction of the coniveau machine requires a more general scheme-theoretic definition.   Perhaps the most convenient way of computing negative cyclic homology of schemes is via Weibel's definition \cite{WeibelCyclicSchemes91} in terms of Cartan-Eilenberg hypercohomology.  This approach is described in section \hyperref[sectionweibelcychomschemes]{\ref{sectionweibelcychomschemes}} of the appendix.  For theoretical purposes, it is better to use Bernhard Keller's machinery of localization pairs, using the modified category $\mbf{Par}_X$ of perfect complexes of $\ms{O}_X$-modules, described in section \hyperref[subsectionconnectivenonconnective]{\ref{subsectionconnectivenonconnective}} above, together with its acyclic subcategory.  This approach is briefly described in section \hyperref[subsectioncycschemes]{\ref{subsectioncycschemes}} below.  


\subsection{Absolute K\"{a}hler Differentials; Algebraic de Rham Complex}\label{subsectionkahlerderham}

A major source of motivation for the development of cyclic homology was a need to generalize differential forms to the context of noncommutative algebra and geometry.  {\it A posteriori}, differential forms may be viewed as primitive versions of generalized cohomology theories such as Hochschild and cyclic homology.   In particular, a number of important results about cyclic homology may be expressed in terms of differential forms.  In the algebraic setting, the appropriate notion of differential forms is given by K\"{a}hler differentials. 
\begin{defi}\label{defikahler}Let $R$ be a commutative algebra over a commutative ring $k$ with identity. The $k$-module of {\bf K\"{a}hler differentials} $\Omega^1_{R/k}$ of $R$ with respect to $k$ is the module generated over $k$ by symbols of the form $rdr'$, subject to the relations 
\begin{enumerate}
\item $rd(\alpha r'+\beta r'')=\alpha rdr'+\beta rdr''$ for $\alpha,\beta\in k$ and  $r,r',r''\in M$ ($k$-linearity).
\item $rd(r'r'')=rr'dr''+rr''dr'$ (Leibniz rule).
\end{enumerate}
\end{defi}

Higher-degree modules of K\"{a}hler differentials $\Omega_{R/k}^n$ are defined by taking exterior powers of $\Omega^1_{R/k}$.  The modules $\Omega_{R/k}^n$ are the graded pieces of the exterior algebra $\Omega_{R/k}^\bullet:=\bigwedge\Omega^1_{R/k}$ over $\Omega^1_{R/k}$.  By definition, it is the quotient of the tensor algebra by the ideal generated by elements of the form $dr\otimes dr$.   The ring multiplication in $\Omega_{R/k}^\bullet$ is called the {\bf wedge product}, and is denoted by $\wedge$.   The map 
\[d:\Omega^n_{R/k}\rightarrow \Omega^{n+1}_{R/k}\]
\[r_0dr_1\wedge...\wedge dr_n\mapsto dr_0\wedge dr_1\wedge...\wedge dr_n\]
makes $\Omega_{R/k}^\bullet$ into a {\it differential graded ring.}  It may be viewed as a complex, called the {\bf algebraic de Rham complex.}  If the ground ring $k$ is the ring of integers $\ZZ$, then the modules $\Omega^n_{R/\ZZ}$ are abelian groups; i.e., $\ZZ$-modules, called the groups of {\bf absolute K\"{a}hler differentials.}  


\subsection{Cyclic Homology of Algebras over Commutative Rings}\label{subsectioncychomalgcommring}

Let $k$ be a commutative ring.   For a $k$-algebra $R$, not necessarily commutative, there several approaches to defining cyclic homology.  For example, Loday \cite{LodayCyclicHomology98}, 2.1.11, page 59 gives a summary of four different definitions that coincide when the ground ring $k$ contains $\QQ$. The two definitions that are most useful for the subject of this book are those expressed via the {\it cyclic bicomplex} $\tn{CC}_R$, and the {\it $\tn{B}$-bicomplex} $\tn{B}_R$.   The latter definition makes sense only if $R$ is unital.  The first step in introducing these definitions is to discuss the Hochschild complex of $R$.  It is also useful to define Hochschild homology, since this theory is often useful in computing cyclic homology.  

{\bf Hochschild Complex and Hochschild Homology.}  The {\bf Hochschild complex} $\tn{C}_R$ of $R$ is the complex
\begin{equation}\label{equhochschildcomplex}\xymatrix{...\ar[r]^b  &R^{\otimes 3}\ar[r]^b&R^{\otimes 2}\ar[r]^b&R,}\end{equation}
where $R^{\otimes n}$ means $R\otimes_k...\otimes_k R$ ($n$ factors).   Here, $R$ is taken to occupy degree zero, so the $n$th term $C_{n,R}$ in the complex $\tn{C}_R$ is $R^{\otimes n+1}$.  Each boundary map $b=b_n:R^{\otimes n+1}\rightarrow R^{\otimes n}$ is given by a sum of {\it face operators} $d_n^i$:
\begin{equation}\label{equhochschildsum}b=\sum_{i=0}^n(-1)^i d_n^i,\end{equation}
defined by the formulas
\begin{equation}\label{equfaceoperators}d_n^i(r_0,..., r_n) = \left\{ \begin{array}{ll}
         (r_0,..., r_ir_{i+1},..., r_n) & \mbox{ if } i<n,\\
       (r_nr_0, r_1, ..., r_{n-1}) &  \mbox{ if } i=n.
        \end{array} \right.\end{equation}
The notation $(r_0,..., r_n)$ is shorthand for $r_0\otimes...\otimes r_n$, and is used to avoid proliferation of tensor symbols. The face operators $d_n^i$ make $\tn{C}_R$ into a {\it presimplicial module}; i.e., a {\it presimplicial object} in the category of $k$-modules.  If $R$ is unital, the complex $\tn{C}_R$ also has {\it degeneracy operators} $s_n^i:R^{\otimes n+1}\rightarrow R^{\otimes n+2}$ for $0\le i\le n$ defined by the formula
\begin{equation}\label{equdegeneracy}s_n^i(r_0,...,r_n)=(r_0,...,r_i,1,r_{i+1},...,r_n),\end{equation}
which make $\tn{C}_R$ into a {\it simplicial module.} 

\begin{defi} The homology groups $\tn{HH}_{n,R}$ of the Hochschild complex $\tn{C}_R$ are called the {\bf Hochschild homology groups} of $R$. 
\end{defi}


\label{hochschildkonstantrosenberg}

{\bf Smooth Algebras and the Hochschild-Konstant-Rosenberg Theorem.}  An algebra $R$ over a commutative ring $k$ is {\bf smooth} in the sense of Loday\footnotemark\footnotetext{Loday \cite{LodayCyclicHomology98}, appendix E, compares five different notions of smoothness and proves equivalences among them under appropriate assumptions.  The criterion used by Weibel \cite{WeibelHomologicalAlgebra94}, chapter 9, page 314, is the fifth of these, though the notation is different.  Andr'{e}-Quillen cohomology is the ``proper" cohomology theory for studying smoothness.  See Loday \cite{LodayCyclicHomology98}, section 3.4 and appendix E, and Iyengar \cite{IyengarAndreQuillen99}.} if it is flat over $k$ and if, for any maximal ideal $\mfr{m}$ of $R$, the kernel $J$ of the localized multiplication map
\[\mu_\mfr{m}:\big(R\otimes_k R\big)_{\mu^{-1}(\mfr{m})}\rightarrow R_\mfr{m}\]
is generated by a regular sequence in $\big(R\otimes_k R\big)_{\mu^{-1}(\mfr{m})}$.  In particular, the ring of regular functions on a smooth algebraic variety over an algebraically closed field is smooth, and localizations of smooth algebras are smooth.\footnotemark\footnotetext{See Weibel \cite{WeibelHomologicalAlgebra94}, Exercise 9.3.2, page 314, Loday \cite{LodayCyclicHomology98}, example 3.4.3, page 103, and also Zainoulline \cite{ZainoullineSmooth99}.}  

When $R$ is a smooth algebra over a commutative ring $k$ containing $\QQ$, the {\bf Hochschild-Konstant-Rosenberg Theorem}  states that there is an isomorphism of graded algebras between the algebra of K\"{a}hler differentials $\Omega_{R/k}^\bullet$ and Hochschild homology:
\begin{equation}\label{equhochschildkonstantrosen}\Omega_{R/k} ^\bullet \cong\tn{HH}_{\bullet,R}.
\end{equation}
See Loday \cite{LodayCyclicHomology98} theorem 3.4.4, page 103, for details.


\label{CHviacyclicbcomplex}

{\bf Cyclic Homology via the Cyclic Bicomplex.}  The {\bf cyclic bicomplex} $\tn{CC}_R$ of a $k$-algebra $R$ is constructed from copies of the Hochschild complex; figure \ref{figcyclicbicomplex} below illustrates this construction.  $\tn{CC}_R$ is a {\it first-quadrant bicomplex} with a $2$-periodic structure on its columns.  The $R$ in the lower left corner occupies bidegree $(0,0)$ by convention.  Each term in the $n$th row is $(n+1)$-fold tensor product $R^{\otimes n+1}$ of $R$.  The vertical maps in even-degree columns are the Hochschild boundary maps $b$, and the vertical maps $-b'$ in odd-degree columns are given by deleting the last face map $d_n^n$ from each Hochschild boundary map.   The horizontal maps from odd to even degrees are $1-t$, where for the $n$th row, $t=t_n$ is an endomorphism on $R^{\otimes n+1}$ called the {\bf cyclic operator,} given by moving the final factor to the front and adding a sign:
\begin{equation}\label{equcyclicoperator}t(r_0,...,r_n)=(-1)^n(r_n,r_0,...,r_{n-1}).\end{equation}
The sign $(-1)^n$ is the sign of the elementary cyclic permutation on $(n+1)$ symbols; it is included to streamline certain formulas.  The horizontal maps from even to odd degrees are given by the {\bf norm operators} $N:=1+t+...+t^n$, where $n$ again denotes the row.   

\begin{figure}[H]
\begin{pgfpicture}{0cm}{0cm}{17cm}{6.5cm}
\begin{pgftranslate}{\pgfpoint{.5cm}{0cm}}
\pgfputat{\pgfxy(2,.5)}{\pgfbox[center,center]{$R$}}
\pgfputat{\pgfxy(2,2.5)}{\pgfbox[center,center]{$R^{\otimes 2}$}}
\pgfputat{\pgfxy(2,4.5)}{\pgfbox[center,center]{$R^{\otimes 3}$}}
\pgfnodecircle{Node1}[fill]{\pgfxy(2,6.4)}{0.025cm}
\pgfnodecircle{Node1}[fill]{\pgfxy(2,6.55)}{0.025cm}
\pgfnodecircle{Node1}[fill]{\pgfxy(2,6.7)}{0.025cm}
\pgfputat{\pgfxy(4.5,.5)}{\pgfbox[center,center]{$R$}}
\pgfputat{\pgfxy(4.5,2.5)}{\pgfbox[center,center]{$R^{\otimes 2}$}}
\pgfputat{\pgfxy(4.5,4.5)}{\pgfbox[center,center]{$R^{\otimes 3}$}}
\pgfnodecircle{Node1}[fill]{\pgfxy(4.5,6.4)}{0.025cm}
\pgfnodecircle{Node1}[fill]{\pgfxy(4.5,6.55)}{0.025cm}
\pgfnodecircle{Node1}[fill]{\pgfxy(4.5,6.7)}{0.025cm}
\pgfputat{\pgfxy(1.75,1.5)}{\pgfbox[center,center]{\footnotesize{$b$}}}
\pgfputat{\pgfxy(1.75,3.5)}{\pgfbox[center,center]{\footnotesize{$b$}}}
\pgfputat{\pgfxy(1.75,5.5)}{\pgfbox[center,center]{\footnotesize{$b$}}}
\pgfputat{\pgfxy(4.15,1.5)}{\pgfbox[center,center]{\footnotesize{$-b'$}}}
\pgfputat{\pgfxy(4.15,3.5)}{\pgfbox[center,center]{\footnotesize{$-b'$}}}
\pgfputat{\pgfxy(4.15,5.55)}{\pgfbox[center,center]{\footnotesize{$-b'$}}}
\pgfputat{\pgfxy(3.25,.8)}{\pgfbox[center,center]{\footnotesize{$1-t$}}}
\pgfputat{\pgfxy(3.25,2.8)}{\pgfbox[center,center]{\footnotesize{$1-t$}}}
\pgfputat{\pgfxy(3.25,4.8)}{\pgfbox[center,center]{\footnotesize{$1-t$}}}
\pgfputat{\pgfxy(5.75,.8)}{\pgfbox[center,center]{\footnotesize{$N$}}}
\pgfputat{\pgfxy(5.75,2.8)}{\pgfbox[center,center]{\footnotesize{$N$}}}
\pgfputat{\pgfxy(5.75,4.8)}{\pgfbox[center,center]{\footnotesize{$N$}}}
\pgfsetendarrow{\pgfarrowlargepointed{3pt}}
\pgfxyline(2,2.2)(2,0.8)
\pgfxyline(2,4.2)(2,2.8)
\pgfxyline(2,6.2)(2,4.8)
\pgfxyline(4.5,2.2)(4.5,0.8)
\pgfxyline(4.5,4.2)(4.5,2.8)
\pgfxyline(4.5,6.2)(4.5,4.8)
\pgfxyline(4,.5)(2.5,.5)
\pgfxyline(4,2.5)(2.5,2.5)
\pgfxyline(4,4.5)(2.5,4.5)
\pgfxyline(6.5,.5)(5,.5)
\pgfxyline(6.5,2.5)(5,2.5)
\pgfxyline(6.5,4.5)(5,4.5)
\end{pgftranslate}
\begin{pgftranslate}{\pgfpoint{5.5cm}{0cm}}
\pgfputat{\pgfxy(2,.5)}{\pgfbox[center,center]{$R$}}
\pgfputat{\pgfxy(2,2.5)}{\pgfbox[center,center]{$R^{\otimes 2}$}}
\pgfputat{\pgfxy(2,4.5)}{\pgfbox[center,center]{$R^{\otimes 3}$}}
\pgfnodecircle{Node1}[fill]{\pgfxy(2,6.4)}{0.025cm}
\pgfnodecircle{Node1}[fill]{\pgfxy(2,6.55)}{0.025cm}
\pgfnodecircle{Node1}[fill]{\pgfxy(2,6.7)}{0.025cm}
\pgfputat{\pgfxy(4.5,.5)}{\pgfbox[center,center]{$R$}}
\pgfputat{\pgfxy(4.5,2.5)}{\pgfbox[center,center]{$R^{\otimes 2}$}}
\pgfputat{\pgfxy(4.5,4.5)}{\pgfbox[center,center]{$R^{\otimes 3}$}}
\pgfnodecircle{Node1}[fill]{\pgfxy(4.5,6.4)}{0.025cm}
\pgfnodecircle{Node1}[fill]{\pgfxy(4.5,6.55)}{0.025cm}
\pgfnodecircle{Node1}[fill]{\pgfxy(4.5,6.7)}{0.025cm}
\pgfputat{\pgfxy(1.75,1.5)}{\pgfbox[center,center]{\footnotesize{$b$}}}
\pgfputat{\pgfxy(1.75,3.5)}{\pgfbox[center,center]{\footnotesize{$b$}}}
\pgfputat{\pgfxy(1.75,5.5)}{\pgfbox[center,center]{\footnotesize{$b$}}}
\pgfputat{\pgfxy(4.15,1.5)}{\pgfbox[center,center]{\footnotesize{$-b'$}}}
\pgfputat{\pgfxy(4.15,3.5)}{\pgfbox[center,center]{\footnotesize{$-b'$}}}
\pgfputat{\pgfxy(4.15,5.55)}{\pgfbox[center,center]{\footnotesize{$-b'$}}}
\pgfputat{\pgfxy(3.25,.8)}{\pgfbox[center,center]{\footnotesize{$1-t$}}}
\pgfputat{\pgfxy(3.25,2.8)}{\pgfbox[center,center]{\footnotesize{$1-t$}}}
\pgfputat{\pgfxy(3.25,4.8)}{\pgfbox[center,center]{\footnotesize{$1-t$}}}
\pgfputat{\pgfxy(5.75,.8)}{\pgfbox[center,center]{\footnotesize{$N$}}}
\pgfputat{\pgfxy(5.75,2.8)}{\pgfbox[center,center]{\footnotesize{$N$}}}
\pgfputat{\pgfxy(5.75,4.8)}{\pgfbox[center,center]{\footnotesize{$N$}}}
\pgfnodecircle{Node1}[fill]{\pgfxy(6.7,.5)}{0.025cm}
\pgfnodecircle{Node1}[fill]{\pgfxy(6.85,.5)}{0.025cm}
\pgfnodecircle{Node1}[fill]{\pgfxy(7,.5)}{0.025cm}
\pgfnodecircle{Node1}[fill]{\pgfxy(6.7,2.5)}{0.025cm}
\pgfnodecircle{Node1}[fill]{\pgfxy(6.85,2.5)}{0.025cm}
\pgfnodecircle{Node1}[fill]{\pgfxy(7,2.5)}{0.025cm}
\pgfnodecircle{Node1}[fill]{\pgfxy(6.7,4.5)}{0.025cm}
\pgfnodecircle{Node1}[fill]{\pgfxy(6.85,4.5)}{0.025cm}
\pgfnodecircle{Node1}[fill]{\pgfxy(7,4.5)}{0.025cm}
\pgfsetendarrow{\pgfarrowlargepointed{3pt}}
\pgfxyline(2,2.2)(2,0.8)
\pgfxyline(2,4.2)(2,2.8)
\pgfxyline(2,6.2)(2,4.8)
\pgfxyline(4.5,2.2)(4.5,0.8)
\pgfxyline(4.5,4.2)(4.5,2.8)
\pgfxyline(4.5,6.2)(4.5,4.8)
\pgfxyline(4,.5)(2.5,.5)
\pgfxyline(4,2.5)(2.5,2.5)
\pgfxyline(4,4.5)(2.5,4.5)
\pgfxyline(6.5,.5)(5,.5)
\pgfxyline(6.5,2.5)(5,2.5)
\pgfxyline(6.5,4.5)(5,4.5)
\end{pgftranslate}
\end{pgfpicture}
\caption{The Cyclic Bicomplex $\tn{CC}_R$.}
\label{figcyclicbicomplex}
\end{figure}
\vspace*{-.5cm}

\begin{defi}\label{defiHCofR} Let $R$ be an algebra over a commutative ring $k$.  The {\bf cyclic homology} groups $\tn{HC}_{n,R}$ of $R$ are the homology groups of the total complex $\tn{Tot } \tn{CC}_R$ of the cyclic bicomplex $\tn{CC}_R$.
\end{defi}

The groups $\tn{HC}_{n,R}$ depend on the underlying ring $k$, so it would more precise to write $\tn{HC}_{n,R/k}$.  However, I will follow tradition and suppress $k$.  Note also that there is no need to choose between sums and products in forming the total complex in this case, because the number of terms in each total degree is finite.  This is true because $\tn{CC}_R$ is concentrated in the first quadrant.  Also observe that the cyclic homology groups in negative degrees automatically vanish.   Finally, Loday  \cite{LodayCyclicHomology98}, 2.1.16, page 60 has results about the behavior of cyclic homology under change of ground ring and localization
 that are particularly useful in the context of algebraic varieties and schemes, where one often wishes to work locally. 
 
\begin{example} \tn{For any $k$-algebra $R$, 
\begin{equation}\label{equHC0loday}\tn{HC}_{0,R}=R/[R,R],\end{equation}
where $[R,R]$ is the center of $R$.  In particular, if $R$ is commutative, then $\tn{HC}_{0,R}=R$.  See Loday  \cite{LodayCyclicHomology98}, 2.1.12, page 59, for details.}
\end{example}
\hspace{16.3cm} $\oblong$

\begin{example}\label{exHC1Omega/dOmega} \tn{If $R$ is a commutative unital $k$-algebra, then
\begin{equation}\label{equHC1loday}\tn{HC}_{1,R}\cong\frac{\Omega_{R/k}^1}{dR}.\end{equation}
See Loday  \cite{LodayCyclicHomology98}, proposition 2.1.14, page 59, for details.}
\end{example}
\hspace{16.3cm} $\oblong$

\begin{example} \tn{If $R$ is a smooth algebra over a commutative ring $k$ containing $\QQ$, the cyclic homology $HC_{n,R}$ may be expressed in terms of the de Rham cohomology groups $H_{dR} ^m(R)$, which are the cohomology groups of the algebraic de Rham complex:\footnotemark\footnotetext{See Loday \cite{LodayCyclicHomology98}, page 69.  As with the cyclic homology groups themselves, the de Rham cohomology groups depend on the underlying ring $k$.} 
\begin{equation}\label{equlambdacyclic}
\tn{HC}_{n,R}\cong \frac{\Omega_{R/k} ^n}{d\Omega_{R/k} ^{n-1}}\oplus H_{dR} ^{n-2}(R)\oplus H_{dR} ^{n-4}(R)\oplus...,
\end{equation}
This direct sum decomposition in fact coincides with the {\it lambda decomposition} of cyclic homology.  See \cite{LodayCyclicHomology98} theorems 4.6.7 and 4.6.10, pages 151-153, for details.} 
\end{example}
\hspace{16.3cm} $\oblong$


\label{HCviaBBicomplex}

{\bf Cyclic Homology via the $\tn{B}$-Bicomplex.}  The {\bf $\tn{B}$-bicomplex} $\tn{B}_R$ of a {\it unital} $k$-algebra $R$ is constructed by removing redundant information from the cyclic bicomplex $\tn{CC}_R$ of $R$.  The construction is illustrated in figure \ref{figBcomplex} below. The reason for the unital hypothesis is because in this case there exists a {\it contracting homotopy} of the $b'$-complex defining the odd-degree columns of $\tn{CC}_R$.  This allows for the removal of these columns and the rearrangement of the even-degree columns in a manner described in detail by Loday \cite{LodayCyclicHomology98} .\footnotemark\footnotetext{See Loday \cite{LodayCyclicHomology98}, page 12, for a description of the contracting homotopy.  See pages 56-57 of the same reference for a description of how the bicomplex $\tn{B}_R$ is derived from the cyclic bicomplex.}  Note that the columns are rearranged so that the $n$th term $\big(\tn{Tot } \tn{B}_R\big)_n$ of the total complex $\tn{Tot } \tn{B}_R$ is the direct sum of the terms $R^{\otimes n}, R^{\otimes(n-1)},...,R^{\otimes2}, R$ going {\it across the $n$th row} of $\tn{B}_R$. 

\begin{figure}[H]
\begin{pgfpicture}{0cm}{0cm}{17cm}{6.25cm}
\begin{pgftranslate}{\pgfpoint{2.5cm}{-.25cm}}
\pgfputat{\pgfxy(2,.5)}{\pgfbox[center,center]{$R$}}
\pgfputat{\pgfxy(2,2.5)}{\pgfbox[center,center]{$R^{\otimes 2}$}}
\pgfputat{\pgfxy(2,4.5)}{\pgfbox[center,center]{$R^{\otimes 3}$}}
\pgfnodecircle{Node1}[fill]{\pgfxy(2,6.4)}{0.025cm}
\pgfnodecircle{Node1}[fill]{\pgfxy(2,6.55)}{0.025cm}
\pgfnodecircle{Node1}[fill]{\pgfxy(2,6.7)}{0.025cm}
\pgfputat{\pgfxy(4.5,2.5)}{\pgfbox[center,center]{$R$}}
\pgfputat{\pgfxy(4.5,4.5)}{\pgfbox[center,center]{$R^{\otimes 2}$}}
\pgfnodecircle{Node1}[fill]{\pgfxy(4.5,6.4)}{0.025cm}
\pgfnodecircle{Node1}[fill]{\pgfxy(4.5,6.55)}{0.025cm}
\pgfnodecircle{Node1}[fill]{\pgfxy(4.5,6.7)}{0.025cm}
\pgfputat{\pgfxy(7,4.5)}{\pgfbox[center,center]{$R$}}
\pgfnodecircle{Node1}[fill]{\pgfxy(7,6.4)}{0.025cm}
\pgfnodecircle{Node1}[fill]{\pgfxy(7,6.55)}{0.025cm}
\pgfnodecircle{Node1}[fill]{\pgfxy(7,6.7)}{0.025cm}
\pgfputat{\pgfxy(1.75,1.5)}{\pgfbox[center,center]{\footnotesize{$b$}}}
\pgfputat{\pgfxy(1.75,3.5)}{\pgfbox[center,center]{\footnotesize{$b$}}}
\pgfputat{\pgfxy(1.75,5.5)}{\pgfbox[center,center]{\footnotesize{$b$}}}
\pgfputat{\pgfxy(4.25,3.5)}{\pgfbox[center,center]{\footnotesize{$b$}}}
\pgfputat{\pgfxy(4.25,5.55)}{\pgfbox[center,center]{\footnotesize{$b$}}}
\pgfputat{\pgfxy(6.75,5.5)}{\pgfbox[center,center]{\footnotesize{$b$}}}
\pgfputat{\pgfxy(3.25,2.8)}{\pgfbox[center,center]{\footnotesize{$B$}}}
\pgfputat{\pgfxy(3.25,4.8)}{\pgfbox[center,center]{\footnotesize{$B$}}}
\pgfputat{\pgfxy(5.75,4.8)}{\pgfbox[center,center]{\footnotesize{$B$}}}
\pgfsetendarrow{\pgfarrowlargepointed{3pt}}
\pgfxyline(2,2.2)(2,0.8)
\pgfxyline(2,4.2)(2,2.8)
\pgfxyline(2,6.2)(2,4.8)
\pgfxyline(4.5,4.2)(4.5,2.8)
\pgfxyline(4.5,6.2)(4.5,4.8)
\pgfxyline(7,6.2)(7,4.8)
\pgfxyline(4,2.5)(2.5,2.5)
\pgfxyline(4,4.5)(2.5,4.5)
\pgfxyline(6.5,4.5)(5,4.5)
\end{pgftranslate}
\end{pgfpicture}
\caption{The bicomplex $\tn{B}_R$.}
\label{figBcomplex}
\end{figure}
\vspace*{-.5cm}

The maps $B$ in the $\tn{B}$-bicomplex $B_R$ are called the {\bf Connes boundary maps}.  It is useful to give an explicit description of these maps.  They are defined by the formula
\begin{equation}\label{equconnesboundary}B_n:=(1-t_{n+1})\circ s_{n}\circ N_n,\end{equation}
where 
\begin{enumerate}
\item $1$ is the identity map on $R^{\otimes n+2}$. 
\item $t_{n+1}:R^{\otimes n+2}\rightarrow R^{\otimes n+2}$ is the cyclic operator in degree $n+1$.
\item $s_n=s_n^{n+1}:R^{\otimes n+1}\rightarrow R^{\otimes n+2}$ is the {\bf extra degeneracy}\footnotemark\footnotetext{Note that the subscript $n$ denotes the {\it source degree}, while the superscript $n+1$ denotes the {\it extra} in ``extra degeneracy," since there exist are $n+1$ ``ordinary degeneracy maps" $s_n^0,...,s_n^n$ between degrees $n$ and $n+1$.  The extra degeneracy may also be defined in terms of the last ``ordinary" degeneracy $s_n^n$ and the cyclic operator $t_{n+1}$ as $s_n^{n+1}=(-1)^{n+1}t_{n+1}s_n^n$.  The sign is added here to cancel the sign accompanying the cyclic operator $t_{n+1}$.} between degrees $n$ and $n+1$, defined by the formula
\begin{equation}\label{equextradegeneracy}s_{n}^{n+1}(r_0,...,r_{n})=(1,r_0,...,r_{n}).\end{equation}
\item $N_n$ is the norm operator in degree $n$.
\end{enumerate}
With these conventions, the Hochschild boundary map $b$ and the Connes boundary map $B$ satisfy the equations $B^2=bB+Bb=0$.\footnotemark\footnotetext{The definition of $B$ is sometimes given with a sign: $B_n=(-1)^{n+1}(1-t_{n+1})s_{n}N_n$ (note the discrepancy between Loday \cite{LodayCyclicHomology98} pages 57 and 78, for instance).  This difference corresponds to whether or not the cyclic operator $t_{n+1}$ is defined with a sign.}

\begin{example}\tn{The zeroth Connes boundary map $B_0$ is given by the formula:
\begin{equation}\label{equzerothconnesboundary}B_0(r)= (1,r)+(r,1),\end{equation}
where the plus sign comes from the sign on the cyclic operator $t_1$.  The first  Connes boundary map $B_1$ is given by the formula:
\begin{equation}\label{equfirstthconnesboundary}B_1(r_1,r_2)=(1,r_1,r_2)-(1,r_2,r_1)-(r_2,1,r_1)+(r_1,1,r_2).\end{equation}}
\end{example}
\hspace{16.3cm} $\oblong$

\begin{defi}\label{defiBHC} Let $R$ be an algebra over a commutative ring.  The {\bf $\tn{B}$-groups} $H_{n,R}^{\tn{B}}$ of $R$ are the homology groups of the total complex $\tn{Tot } \tn{B}_R$.  
\end{defi}

If $R$ is unital, there is a natural inclusion of complexes $\mbox{Tot } \tn{B}_R\rightarrow \tn{Tot } \tn{CC}_R$, which induces an isomorphism in homology:

\begin{lem}\label{lemHCBgroupsisom} Let $R$ be a unital algebra over a commutative ring.  Then for all $n$, there are canonical isomorphisms:
\begin{equation}\label{equHCisomorphisms}\tn{HC}_{n,R}\cong H_{n,R} ^{\tn{B}}.\end{equation}
\end{lem} 
\begin{proof}
See Loday \cite{LodayCyclicHomology98}, theorem 2.1.8. 
\end{proof}

\begin{example} \tn{Via the $\tn{B}$-bicomplex $\tn{B}_R$, it is clear that the group $\tn{HC}_{1,R}$ is the quotient of the group $R\otimes R$ by the relations $r\otimes r'r''=rr'\otimes r''+rr''\otimes r'$ (Leibniz rule) and $1\otimes r=r\otimes 1$.   The isomorphism $\tn{HC}_{1,R}\cong \Omega_{R/k}^1/dR$ cited in example \hyperref[exHC1Omega/dOmega]{\ref{exHC1Omega/dOmega}} above is then induced by the map $r\otimes r'\mapsto rdr'$. See Loday \cite{LodayCyclicHomology98}, proposition 2.1.14, page 59, for details.}
\end{example}
\hspace{16.3cm} $\oblong$


{\bf Connes' Periodicity Exact Sequence.}  Let $R$ be a unital $k$-algebra, with $\tn{B}$-bicomplex $\tn{B}_R$.  
Peeling off the factors $R^{\otimes n+1}$ from the terms $\big(\tn{Tot } \tn{B}_R\big)_n=R^{\otimes n+1}\oplus R^{\otimes n}\oplus...\oplus R$ yields a short exact sequence of complexes:
\begin{equation}\label{equSESConnesperiodicity}\xymatrix{0\ar[r]&\tn{C}_R\ar[r]^I&\tn{Tot } \tn{B}_R\ar[r]^S&\tn{Tot } \tn{B}_R[2]\ar[r]&0,}\end{equation}
where $I$ is the inclusion map and the notation $\mbox{Tot } \tn{B}_R[2]$ means the total complex shifted down by two degrees.\footnotemark\footnotetext{Note that due to the convention regarding ``rearrangement of the columns," a two-degree shift is accomplished by shifting over {\it one column} in figure \hyperref[figBcomplexmixed]{\ref{figBcomplexmixed}}.}  The map $S$ in equation \hyperref[equSESConnesperiodicity]{\ref{equSESConnesperiodicity}} is called the {\it periodicity map.}  It is analogous to the Bott periodicity map in topological $K$-theory.  The  short exact sequence appearing in equation \hyperref[equSESConnesperiodicity]{\ref{equSESConnesperiodicity}} induces a long exact sequence in homology:
\begin{equation}\label{equLEShochcyc}\xymatrix{...\ar[r]&\tn{HC}_{n+1,R}\ar[r]^S&\tn{HC}_{n-1,R}\ar[r]^B&\tn{HH}_{n,R}\ar[r]^I&\tn{HC}_{n,R}\ar[r]^S&\tn{HC}_{n-2,R}\ar[r]&...,}\end{equation}
where $I$ and $S$ are induced by the eponymous maps at the level of complexes, and $B$ is induced by Connes' boundary map.  This sequence is called Connes' periodicity exact sequence.  As mentioned by Loday, it is analogous to the Gysin sequence for the homology of an $S^1$-space in the context of algebraic topology.  Connes' periodicity exact sequence is an important tool in practical calculations of cyclic homology.  In particular, it often allows the computation of cyclic homology from Hochschild homology via induction arguments.  A basic application is the fact that a $k$-algebra map induces an isomorphism in cyclic homology if and only if it induces an isomorphism in Hochschild homology (Loday \cite{LodayCyclicHomology98}, Corollary 2.2.3, page 62). 

\begin{example} \tn{Let $R$ be a smooth unital algebra over a commutative ring $k$.  The ``tail" of Connes periodicity exact sequence 
\begin{equation}\label{equLESconneslowdegrees}\xymatrix{...\ar[r]&\tn{HC}_{2,R}\ar[r]^S&\tn{HC}_{0,R}\ar[r]^B&\tn{HH}_{1,R}\ar[r]^I&\tn{HC}_{1,R}\ar[r]^S&0,}\end{equation}
has the form
\begin{equation}\label{equLESconneslowdegrees}\xymatrix{...\ar[r]&\displaystyle\frac{\Omega_{R/k}^2}{d\Omega_{R/k}^1}\times k\ar[r]^-S&R\ar[r]^-B&\Omega_{R/k}^1\ar[r]^-I&\displaystyle\frac{\Omega_{R/k}^1}{dR}\ar[r]^S&0,}\end{equation}
where $S$ is inclusion of the second factor, $B$ is the differential map $d$, and $I$ is the quotient map. 
}
\end{example}
\hspace{16.3cm} $\oblong$

\subsection{Negative Cyclic Homology of Algebras over Commutative Rings}\label{subsectionnegcyccomring}

In section \hyperref[subsectioncychomalgcommring]{\ref{subsectioncychomalgcommring}}, I recounted how modifying the cyclic bicomplex $\tn{CC}_R$ to yield the $\tn{B}$-bicomplex $\tn{B}_R$ for a $k$-algebra $R$, gives alternative cyclic homology groups $H_{n,R}^{\tn{B}}$, which are isomorphic to the ordinary cyclic homology groups $\tn{HC}_{n,R}$ contains $\mathbb{Q}$ and $R$ is unital.  In this section, I discuss a different modification of $\tn{CC}_R$ that produces the variant of cyclic homology most appropriate for the coniveau machine.  This variant is called negative cyclic homology.  


\label{subsectionnegcycbicomplex}

{\bf Negative Cyclic Homology via the Negative Cyclic Complex $\tn{CN}_R$.}  The periodicity of the cyclic bicomplex $\tn{CC}_R$ allows this complex to be extended to an {\it upper-half-plane complex} called the {\bf periodic cyclic bicomplex}, denoted by $\tn{CP}_R$, with nontrivial columns in all integer degrees.  The cyclic bicomplex $\tn{CC}_R$ may be regarded as the sub-bicomplex of $\tn{CP}_R$ consisting of all columns of non-negative degree.  The {\bf negative cyclic bicomplex} $\tn{CN}_R$ is defined by truncating from $\tn{CP}_R$ all columns of degree greater than $1$.  $\tn{CN}_R$ is a second-quadrant bicomplex up to a shift; the column in degree $1$ is the right-most nonzero column.   The bicomplexes $\tn{CC}_R$ and $\tn{CN}_R$ share their zeroth and first columns, but are otherwise disjoint.   The three complexes $\tn{CC}_R$, $\tn{CP}_R$, and $\tn{CN}_R$ are illustrated in figure \ref{figcyclicbicomplexes} below.

\begin{figure}[H]
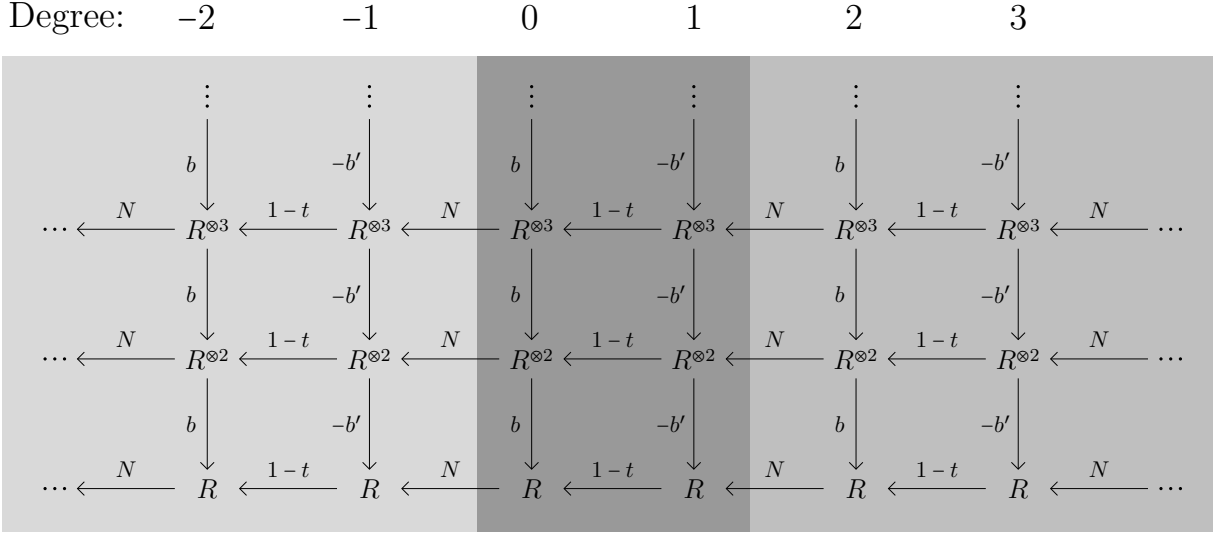

\begin{pgfpicture}{0cm}{0cm}{17cm}{7.5cm}
\begin{pgfmagnify}{1.15}{1.15}
\pgfputat{\pgfxy(1,6.2)}{\pgfbox[center,center]{Degree:}}
\pgfputat{\pgfxy(2.5,6.2)}{\pgfbox[center,center]{$-2$}}
\pgfputat{\pgfxy(4.4,6.2)}{\pgfbox[center,center]{$-1$}}
\pgfputat{\pgfxy(6.35,6.2)}{\pgfbox[center,center]{$0$}}
\pgfputat{\pgfxy(8.25,6.2)}{\pgfbox[center,center]{$1$}}
\pgfputat{\pgfxy(10.1,6.2)}{\pgfbox[center,center]{$2$}}
\pgfputat{\pgfxy(12,6.2)}{\pgfbox[center,center]{$3$}}
\begin{colormixin}{15!white}
\color{black}
\pgfmoveto{\pgfxy(.25,.25)}
\pgflineto{\pgfxy(5.75,.25)}
\pgflineto{\pgfxy(5.75,5.75)}
\pgflineto{\pgfxy(.25,5.75)}
\pgflineto{\pgfxy(.25,.25)}
\pgffill
\end{colormixin}
\begin{colormixin}{40!white}
\color{black}
\pgfmoveto{\pgfxy(5.75,.25)}
\pgflineto{\pgfxy(8.9,.25)}
\pgflineto{\pgfxy(8.9,5.75)}
\pgflineto{\pgfxy(5.75,5.75)}
\pgflineto{\pgfxy(5.75,.25)}
\pgffill
\end{colormixin}
\begin{colormixin}{25!white}
\color{black}
\pgfmoveto{\pgfxy(8.9,.25)}
\pgflineto{\pgfxy(14.25,.25)}
\pgflineto{\pgfxy(14.25,5.75)}
\pgflineto{\pgfxy(8.9,5.75)}
\pgflineto{\pgfxy(8.9,.25)}
\pgffill
\end{colormixin}
\begin{pgfmagnify}{.75}{.75}
\begin{pgftranslate}{\pgfpoint{1.5cm}{.5cm}}
\pgfnodecircle{Node1}[fill]{\pgfxy(-.2,.5)}{0.025cm}
\pgfnodecircle{Node1}[fill]{\pgfxy(-.35,.5)}{0.025cm}
\pgfnodecircle{Node1}[fill]{\pgfxy(-.5,.5)}{0.025cm}
\pgfnodecircle{Node1}[fill]{\pgfxy(-.2,2.5)}{0.025cm}
\pgfnodecircle{Node1}[fill]{\pgfxy(-.35,2.5)}{0.025cm}
\pgfnodecircle{Node1}[fill]{\pgfxy(-.5,2.5)}{0.025cm}
\pgfnodecircle{Node1}[fill]{\pgfxy(-.2,4.5)}{0.025cm}
\pgfnodecircle{Node1}[fill]{\pgfxy(-.35,4.5)}{0.025cm}
\pgfnodecircle{Node1}[fill]{\pgfxy(-.5,4.5)}{0.025cm}
\pgfputat{\pgfxy(2,.5)}{\pgfbox[center,center]{$R$}}
\pgfputat{\pgfxy(2,2.5)}{\pgfbox[center,center]{$R^{\otimes 2}$}}
\pgfputat{\pgfxy(2,4.5)}{\pgfbox[center,center]{$R^{\otimes 3}$}}
\pgfnodecircle{Node1}[fill]{\pgfxy(2,6.4)}{0.025cm}
\pgfnodecircle{Node1}[fill]{\pgfxy(2,6.55)}{0.025cm}
\pgfnodecircle{Node1}[fill]{\pgfxy(2,6.7)}{0.025cm}
\pgfputat{\pgfxy(4.5,.5)}{\pgfbox[center,center]{$R$}}
\pgfputat{\pgfxy(4.5,2.5)}{\pgfbox[center,center]{$R^{\otimes 2}$}}
\pgfputat{\pgfxy(4.5,4.5)}{\pgfbox[center,center]{$R^{\otimes 3}$}}
\pgfnodecircle{Node1}[fill]{\pgfxy(4.5,6.4)}{0.025cm}
\pgfnodecircle{Node1}[fill]{\pgfxy(4.5,6.55)}{0.025cm}
\pgfnodecircle{Node1}[fill]{\pgfxy(4.5,6.7)}{0.025cm}
\pgfputat{\pgfxy(.75,.8)}{\pgfbox[center,center]{\footnotesize{$N$}}}
\pgfputat{\pgfxy(.75,2.8)}{\pgfbox[center,center]{\footnotesize{$N$}}}
\pgfputat{\pgfxy(.75,4.8)}{\pgfbox[center,center]{\footnotesize{$N$}}}
\pgfputat{\pgfxy(1.75,1.5)}{\pgfbox[center,center]{\footnotesize{$b$}}}
\pgfputat{\pgfxy(1.75,3.5)}{\pgfbox[center,center]{\footnotesize{$b$}}}
\pgfputat{\pgfxy(1.75,5.5)}{\pgfbox[center,center]{\footnotesize{$b$}}}
\pgfputat{\pgfxy(4.15,1.5)}{\pgfbox[center,center]{\footnotesize{$-b'$}}}
\pgfputat{\pgfxy(4.15,3.5)}{\pgfbox[center,center]{\footnotesize{$-b'$}}}
\pgfputat{\pgfxy(4.15,5.55)}{\pgfbox[center,center]{\footnotesize{$-b'$}}}
\pgfputat{\pgfxy(3.25,.8)}{\pgfbox[center,center]{\footnotesize{$1-t$}}}
\pgfputat{\pgfxy(3.25,2.8)}{\pgfbox[center,center]{\footnotesize{$1-t$}}}
\pgfputat{\pgfxy(3.25,4.8)}{\pgfbox[center,center]{\footnotesize{$1-t$}}}
\pgfputat{\pgfxy(5.75,.8)}{\pgfbox[center,center]{\footnotesize{$N$}}}
\pgfputat{\pgfxy(5.75,2.8)}{\pgfbox[center,center]{\footnotesize{$N$}}}
\pgfputat{\pgfxy(5.75,4.8)}{\pgfbox[center,center]{\footnotesize{$N$}}}
\pgfsetendarrow{\pgfarrowlargepointed{3pt}}
\pgfxyline(1.5,.5)(0,.5)
\pgfxyline(1.5,2.5)(0,2.5)
\pgfxyline(1.5,4.5)(0,4.5)
\pgfxyline(2,2.2)(2,0.8)
\pgfxyline(2,4.2)(2,2.8)
\pgfxyline(2,6.2)(2,4.8)
\pgfxyline(4.5,2.2)(4.5,0.8)
\pgfxyline(4.5,4.2)(4.5,2.8)
\pgfxyline(4.5,6.2)(4.5,4.8)
\pgfxyline(4,.5)(2.5,.5)
\pgfxyline(4,2.5)(2.5,2.5)
\pgfxyline(4,4.5)(2.5,4.5)
\pgfxyline(6.5,.5)(5,.5)
\pgfxyline(6.5,2.5)(5,2.5)
\pgfxyline(6.5,4.5)(5,4.5)
\end{pgftranslate}
\begin{pgftranslate}{\pgfpoint{6.5cm}{.5cm}}
\pgfputat{\pgfxy(2,.5)}{\pgfbox[center,center]{$R$}}
\pgfputat{\pgfxy(2,2.5)}{\pgfbox[center,center]{$R^{\otimes 2}$}}
\pgfputat{\pgfxy(2,4.5)}{\pgfbox[center,center]{$R^{\otimes 3}$}}
\pgfnodecircle{Node1}[fill]{\pgfxy(2,6.4)}{0.025cm}
\pgfnodecircle{Node1}[fill]{\pgfxy(2,6.55)}{0.025cm}
\pgfnodecircle{Node1}[fill]{\pgfxy(2,6.7)}{0.025cm}
\pgfputat{\pgfxy(4.5,.5)}{\pgfbox[center,center]{$R$}}
\pgfputat{\pgfxy(4.5,2.5)}{\pgfbox[center,center]{$R^{\otimes 2}$}}
\pgfputat{\pgfxy(4.5,4.5)}{\pgfbox[center,center]{$R^{\otimes 3}$}}
\pgfnodecircle{Node1}[fill]{\pgfxy(4.5,6.4)}{0.025cm}
\pgfnodecircle{Node1}[fill]{\pgfxy(4.5,6.55)}{0.025cm}
\pgfnodecircle{Node1}[fill]{\pgfxy(4.5,6.7)}{0.025cm}
\pgfputat{\pgfxy(1.75,1.5)}{\pgfbox[center,center]{\footnotesize{$b$}}}
\pgfputat{\pgfxy(1.75,3.5)}{\pgfbox[center,center]{\footnotesize{$b$}}}
\pgfputat{\pgfxy(1.75,5.5)}{\pgfbox[center,center]{\footnotesize{$b$}}}
\pgfputat{\pgfxy(4.15,1.5)}{\pgfbox[center,center]{\footnotesize{$-b'$}}}
\pgfputat{\pgfxy(4.15,3.5)}{\pgfbox[center,center]{\footnotesize{$-b'$}}}
\pgfputat{\pgfxy(4.15,5.55)}{\pgfbox[center,center]{\footnotesize{$-b'$}}}
\pgfputat{\pgfxy(3.25,.8)}{\pgfbox[center,center]{\footnotesize{$1-t$}}}
\pgfputat{\pgfxy(3.25,2.8)}{\pgfbox[center,center]{\footnotesize{$1-t$}}}
\pgfputat{\pgfxy(3.25,4.8)}{\pgfbox[center,center]{\footnotesize{$1-t$}}}
\pgfputat{\pgfxy(5.75,.8)}{\pgfbox[center,center]{\footnotesize{$N$}}}
\pgfputat{\pgfxy(5.75,2.8)}{\pgfbox[center,center]{\footnotesize{$N$}}}
\pgfputat{\pgfxy(5.75,4.8)}{\pgfbox[center,center]{\footnotesize{$N$}}}
\pgfsetendarrow{\pgfarrowlargepointed{3pt}}
\pgfxyline(2,2.2)(2,0.8)
\pgfxyline(2,4.2)(2,2.8)
\pgfxyline(2,6.2)(2,4.8)
\pgfxyline(4.5,2.2)(4.5,0.8)
\pgfxyline(4.5,4.2)(4.5,2.8)
\pgfxyline(4.5,6.2)(4.5,4.8)
\pgfxyline(4,.5)(2.5,.5)
\pgfxyline(4,2.5)(2.5,2.5)
\pgfxyline(4,4.5)(2.5,4.5)
\pgfxyline(6.5,.5)(5,.5)
\pgfxyline(6.5,2.5)(5,2.5)
\pgfxyline(6.5,4.5)(5,4.5)
\end{pgftranslate}
\begin{pgftranslate}{\pgfpoint{11.5cm}{.5cm}}
\pgfputat{\pgfxy(2,.5)}{\pgfbox[center,center]{$R$}}
\pgfputat{\pgfxy(2,2.5)}{\pgfbox[center,center]{$R^{\otimes 2}$}}
\pgfputat{\pgfxy(2,4.5)}{\pgfbox[center,center]{$R^{\otimes 3}$}}
\pgfnodecircle{Node1}[fill]{\pgfxy(2,6.4)}{0.025cm}
\pgfnodecircle{Node1}[fill]{\pgfxy(2,6.55)}{0.025cm}
\pgfnodecircle{Node1}[fill]{\pgfxy(2,6.7)}{0.025cm}
\pgfputat{\pgfxy(4.5,.5)}{\pgfbox[center,center]{$R$}}
\pgfputat{\pgfxy(4.5,2.5)}{\pgfbox[center,center]{$R^{\otimes 2}$}}
\pgfputat{\pgfxy(4.5,4.5)}{\pgfbox[center,center]{$R^{\otimes 3}$}}
\pgfnodecircle{Node1}[fill]{\pgfxy(4.5,6.4)}{0.025cm}
\pgfnodecircle{Node1}[fill]{\pgfxy(4.5,6.55)}{0.025cm}
\pgfnodecircle{Node1}[fill]{\pgfxy(4.5,6.7)}{0.025cm}
\pgfputat{\pgfxy(1.75,1.5)}{\pgfbox[center,center]{\footnotesize{$b$}}}
\pgfputat{\pgfxy(1.75,3.5)}{\pgfbox[center,center]{\footnotesize{$b$}}}
\pgfputat{\pgfxy(1.75,5.5)}{\pgfbox[center,center]{\footnotesize{$b$}}}
\pgfputat{\pgfxy(4.15,1.5)}{\pgfbox[center,center]{\footnotesize{$-b'$}}}
\pgfputat{\pgfxy(4.15,3.5)}{\pgfbox[center,center]{\footnotesize{$-b'$}}}
\pgfputat{\pgfxy(4.15,5.55)}{\pgfbox[center,center]{\footnotesize{$-b'$}}}
\pgfputat{\pgfxy(3.25,.8)}{\pgfbox[center,center]{\footnotesize{$1-t$}}}
\pgfputat{\pgfxy(3.25,2.8)}{\pgfbox[center,center]{\footnotesize{$1-t$}}}
\pgfputat{\pgfxy(3.25,4.8)}{\pgfbox[center,center]{\footnotesize{$1-t$}}}
\pgfputat{\pgfxy(5.75,.8)}{\pgfbox[center,center]{\footnotesize{$N$}}}
\pgfputat{\pgfxy(5.75,2.8)}{\pgfbox[center,center]{\footnotesize{$N$}}}
\pgfputat{\pgfxy(5.75,4.8)}{\pgfbox[center,center]{\footnotesize{$N$}}}
\pgfnodecircle{Node1}[fill]{\pgfxy(6.7,.5)}{0.025cm}
\pgfnodecircle{Node1}[fill]{\pgfxy(6.85,.5)}{0.025cm}
\pgfnodecircle{Node1}[fill]{\pgfxy(7,.5)}{0.025cm}
\pgfnodecircle{Node1}[fill]{\pgfxy(6.7,2.5)}{0.025cm}
\pgfnodecircle{Node1}[fill]{\pgfxy(6.85,2.5)}{0.025cm}
\pgfnodecircle{Node1}[fill]{\pgfxy(7,2.5)}{0.025cm}
\pgfnodecircle{Node1}[fill]{\pgfxy(6.7,4.5)}{0.025cm}
\pgfnodecircle{Node1}[fill]{\pgfxy(6.85,4.5)}{0.025cm}
\pgfnodecircle{Node1}[fill]{\pgfxy(7,4.5)}{0.025cm}
\pgfsetendarrow{\pgfarrowlargepointed{3pt}}
\pgfxyline(2,2.2)(2,0.8)
\pgfxyline(2,4.2)(2,2.8)
\pgfxyline(2,6.2)(2,4.8)
\pgfxyline(4.5,2.2)(4.5,0.8)
\pgfxyline(4.5,4.2)(4.5,2.8)
\pgfxyline(4.5,6.2)(4.5,4.8)
\pgfxyline(4,.5)(2.5,.5)
\pgfxyline(4,2.5)(2.5,2.5)
\pgfxyline(4,4.5)(2.5,4.5)
\pgfxyline(6.5,.5)(5,.5)
\pgfxyline(6.5,2.5)(5,2.5)
\pgfxyline(6.5,4.5)(5,4.5)
\end{pgftranslate}
\end{pgfmagnify}
\end{pgfmagnify}
\end{pgfpicture}
\caption{Cyclic, Periodic Cyclic, and  Negative Cyclic Bicomplexes.}
\label{figcyclicbicomplexes}
\end{figure}

The bicomplexes $\tn{CP}_R$ and $\tn{CN}_R$ may be used to define periodic and negative cyclic homology, respectively.  In the context of this book, periodic cyclic homology plays a specific and limited role; namely, the vanishing of {\it relative} periodic cyclic homology in the relative SBI sequence for a nilpotent ideal, defined below, gives a degree-shifting isomorphism between relative cyclic homology and relative negative cyclic homology in this case. 

The following definition follows Loday  \cite{LodayCyclicHomology98}:

\begin{defi}\label{definegativecycalg} Let $R$ be an algebra over a commutative ring.  
\begin{enumerate}
\item The {\bf periodic cyclic homology} groups $\tn{HP}_{n,R}$ of $R$ are the homology groups of the {\it product} total complex $\tn{ToT } \tn{CP}_R$ of the periodic cyclic bicomplex $\tn{CP}_R$. 
\item The {\bf negative cyclic homology} groups $\tn{HN}_{n,R}$ of $R$ are the homology groups of the {\it product} total complex $\tn{ToT } \tn{CN}_R$ of the negative cyclic bicomplex $\tn{CN}_R$. 
\end{enumerate}
\end{defi}

Note that $\tn{HN}_{n,R}$ is generally nontrivial in negative degrees. 

Comparing definition \hyperref[definegativecycalg]{\ref{definegativecycalg}} with definitions \hyperref[defiHCofR]{\ref{defiHCofR}} and \hyperref[defiBHC]{\ref{defiBHC}} above, one may observe that in the present case one must choose between sums and products in forming the total complex, since each total degree has an infinite number of second-quadrant terms.  Weibel et al. choose to use modified total complexes $\tn{ToT}' \tn{B}_C^{\tn{per}}$ and $\tn{ToT}' \tn{B}_C^-$ to define periodic and negative cyclic homology,\footnotemark\footnotetext{The complexes $\tn{ToT}' \tn{B}_C^{\tn{per}}$ and $\tn{ToT}' \tn{B}_C^-$ are precisely the complexes $\tn{HP}(C,b,B)$ and $\tn{HN}(C,b,B)$ appearing in Weibel et al. \cite{WeibelCycliccdh-CohomNegativeK06}, page 6.} where, for a general bicomplex $\tn{C}=\{\tn{C}_{p,q}\}$,
\begin{equation}\label{equweibelmodifiedprodcomplex} \big(\tn{ToT}' \tn{C}\big)^n:=(x_{p,q})\in\prod_{p+q=n}\tn{C}_{p,q}.
\end{equation}
 This difference of definitions does not affect the {\it relative} groups of principal interest in this book.\footnotemark\footnotetext{In particular, in the smooth case, it affects only de Rham cohomology groups of high degree appearing as factors of periodic cyclic and negative cyclic homology.}


\label{subsectionnegcycBcomplex}

{\bf Negative Cyclic Homology via the Negative $\tn{B}$-bicomplex $\tn{B}_R^-$.} If the $k$-algebra $R$ is unital, then the $\tn{B}$-bicomplex $\tn{B}_R$ described in section \hyperref[HCviaBBicomplex]{\ref{HCviaBBicomplex}} above is defined, and this bicomplex may be extended leftward to include nontrivial columns in negative degrees, in a manner similar to the construction of the periodic cyclic bicomplex.   The result is the {\bf periodic $\tn{B}$-bicomplex $\tn{B}_R^{\tn{per}}$}, an {\it upper triangular bicomplex} whose terms are trivial whenever the column degree exceeds the row degree.   In the general case, $\tn{B}_R^{\tn{per}}$ has nontrivial rows and columns in all integer degrees.   The $\tn{B}$-bicomplex $\tn{B}_R$ may be regarded as the sub-bicomplex of $\tn{B}_R^{\tn{per}}$ consisting of all columns of non-negative degree.  The corresponding {\bf negative $\tn{B}$-bicomplex} $\tn{B}_R^-$ is obtained from by truncating all columns of positive degree from the periodic $\tn{B}$-bicomplex $\tn{B}_R^{\tn{per}}$.   The bicomplexes $\tn{B}_R$ and $\tn{B}_R^-$ share the single column in degree zero.\footnotemark\footnotetext{Note the difference from the case of $\tn{CC}_R$ and $\tn{CN}_R$, which share two columns.}  The three complexes $\tn{B}_R$, $\tn{B}_R^{\tn{per}}$, and $\tn{B}_R^-$ are illustrated in figure \ref{figBcomplexes} below.

The following definition follows Loday  \cite{LodayCyclicHomology98}:

\begin{defi}\label{definegativecycalg} Let $R$ be an algebra over a commutative ring.  
\begin{enumerate}
\item The {\bf periodic cyclic $\tn{B}$-homology groups} $\tn{HN}_{n,R}^{\tn{B}}$ of $R$ are the homology groups of the product total complex $\tn{ToT } \tn{B}_R^{\tn{per}}$ of the periodic cyclic $\tn{B}$-bicomplex $\tn{B}_R^{\tn{per}}$. 
\item The {\bf negative cyclic $\tn{B}$-homology groups} $\tn{HN}_{n,R}^{\tn{B}}$ of $R$ are the homology groups of the product total complex $\tn{ToT } \tn{B}_R^-$ of the negative cyclic $\tn{B}$-bicomplex $\tn{B}_R^-$. 
\end{enumerate}
\end{defi}

\begin{lem}\label{lemnegcycisomB} Let $R$ be a unital algebra over a commutative ring.
Then for all $n$, there exist canonical isomorphisms:
\begin{equation}\label{equnegcycisomB}\tn{HN}_{n,R}^{\tn{B}}\cong \tn{HN}_{n,R}.\end{equation} 
\end{lem}
\begin{proof}
See Loday \cite{LodayCyclicHomology98}, 5.1.7, page 161. 
\end{proof}

\begin{figure}[H]
\begin{pgfpicture}{0cm}{0cm}{17cm}{11.5cm}
\begin{pgfmagnify}{1.15}{1.15}
\pgfputat{\pgfxy(1,10.2)}{\pgfbox[center,center]{\small{Degree:}}}
\pgfputat{\pgfxy(3.625,10.2)}{\pgfbox[center,center]{\small{$-2$}}}
\pgfputat{\pgfxy(5.475,10.2)}{\pgfbox[center,center]{\small{$-1$}}}
\pgfputat{\pgfxy(7.475,10.2)}{\pgfbox[center,center]{\small{$0$}}}
\pgfputat{\pgfxy(9.375,10.2)}{\pgfbox[center,center]{\small{$1$}}}
\pgfputat{\pgfxy(11.225,10.2)}{\pgfbox[center,center]{\small{$2$}}}
\pgfputat{\pgfxy(1,8.2)}{\pgfbox[center,center]{\small{$2$}}}
\pgfputat{\pgfxy(1,6.7)}{\pgfbox[center,center]{\small{$1$}}}
\pgfputat{\pgfxy(1,5.2)}{\pgfbox[center,center]{\small{$0$}}}
\pgfputat{\pgfxy(1,3.7)}{\pgfbox[center,center]{\small{$-1$}}}
\pgfputat{\pgfxy(1,2.2)}{\pgfbox[center,center]{\small{$-2$}}}
\begin{pgftranslate}{\pgfpoint{.375cm}{1.5cm}}
\begin{colormixin}{15!white}
\color{black}
\pgfmoveto{\pgfxy(1.5,-1.35)}
\pgflineto{\pgfxy(6,2.25)}
\pgflineto{\pgfxy(6,8.3)}
\pgflineto{\pgfxy(1.5,8.3)}
\pgflineto{\pgfxy(1.5,-1.35)}
\pgffill
\end{colormixin}
\begin{colormixin}{40!white}
\color{black}
\pgfmoveto{\pgfxy(6,2.25)}
\pgflineto{\pgfxy(8.25,4.05)}
\pgflineto{\pgfxy(8.25,8.3)}
\pgflineto{\pgfxy(6,8.3)}
\pgflineto{\pgfxy(6,2.25)}
\pgffill
\end{colormixin}
\begin{colormixin}{25!white}
\color{black}
\pgfmoveto{\pgfxy(8.25,4.05)}
\pgflineto{\pgfxy(13.5625,8.3)}
\pgflineto{\pgfxy(8.25,8.3)}
\pgflineto{\pgfxy(8.25,4.05)}
\pgffill
\end{colormixin}
\end{pgftranslate}
\begin{pgfmagnify}{.75}{.75}
\begin{pgftranslate}{\pgfpoint{3cm}{2.5cm}}
\pgfputat{\pgfxy(2,.5)}{\pgfbox[center,center]{$R$}}
\pgfputat{\pgfxy(2,2.5)}{\pgfbox[center,center]{$R^{\otimes 2}$}}
\pgfputat{\pgfxy(2,4.5)}{\pgfbox[center,center]{$R^{\otimes 3}$}}
\pgfputat{\pgfxy(2,6.5)}{\pgfbox[center,center]{$R^{\otimes 4}$}}
\pgfputat{\pgfxy(2,8.5)}{\pgfbox[center,center]{$R^{\otimes 5}$}}
\pgfnodecircle{Node1}[fill]{\pgfxy(2,9.9)}{0.025cm}
\pgfnodecircle{Node1}[fill]{\pgfxy(2,10.05)}{0.025cm}
\pgfnodecircle{Node1}[fill]{\pgfxy(2,10.2)}{0.025cm}
\pgfputat{\pgfxy(4.5,2.5)}{\pgfbox[center,center]{$R$}}
\pgfputat{\pgfxy(4.5,4.5)}{\pgfbox[center,center]{$R^{\otimes 2}$}}
\pgfputat{\pgfxy(4.5,6.5)}{\pgfbox[center,center]{$R^{\otimes 3}$}}
\pgfputat{\pgfxy(4.5,8.5)}{\pgfbox[center,center]{$R^{\otimes 4}$}}
\pgfnodecircle{Node1}[fill]{\pgfxy(4.5,9.9)}{0.025cm}
\pgfnodecircle{Node1}[fill]{\pgfxy(4.5,10.05)}{0.025cm}
\pgfnodecircle{Node1}[fill]{\pgfxy(4.5,10.2)}{0.025cm}
\pgfputat{\pgfxy(7,4.5)}{\pgfbox[center,center]{$R$}}
\pgfputat{\pgfxy(7,6.5)}{\pgfbox[center,center]{$R^{\otimes 2}$}}
\pgfputat{\pgfxy(7,8.5)}{\pgfbox[center,center]{$R^{\otimes 3}$}}
\pgfnodecircle{Node1}[fill]{\pgfxy(7,9.9)}{0.025cm}
\pgfnodecircle{Node1}[fill]{\pgfxy(7,10.05)}{0.025cm}
\pgfnodecircle{Node1}[fill]{\pgfxy(7,10.2)}{0.025cm}
\pgfputat{\pgfxy(9.5,6.5)}{\pgfbox[center,center]{$R$}}
\pgfputat{\pgfxy(9.5,8.5)}{\pgfbox[center,center]{$R^{\otimes 2}$}}
\pgfnodecircle{Node1}[fill]{\pgfxy(9.5,9.9)}{0.025cm}
\pgfnodecircle{Node1}[fill]{\pgfxy(9.5,10.05)}{0.025cm}
\pgfnodecircle{Node1}[fill]{\pgfxy(9.5,10.2)}{0.025cm}
\pgfputat{\pgfxy(12,8.5)}{\pgfbox[center,center]{$R$}}
\pgfnodecircle{Node1}[fill]{\pgfxy(12,9.9)}{0.025cm}
\pgfnodecircle{Node1}[fill]{\pgfxy(12,10.05)}{0.025cm}
\pgfnodecircle{Node1}[fill]{\pgfxy(12,10.2)}{0.025cm}
\pgfputat{\pgfxy(1.75,1.5)}{\pgfbox[center,center]{\footnotesize{$b$}}}
\pgfputat{\pgfxy(1.75,3.5)}{\pgfbox[center,center]{\footnotesize{$b$}}}
\pgfputat{\pgfxy(1.75,5.5)}{\pgfbox[center,center]{\footnotesize{$b$}}}
\pgfputat{\pgfxy(4.25,3.5)}{\pgfbox[center,center]{\footnotesize{$b$}}}
\pgfputat{\pgfxy(4.25,5.5)}{\pgfbox[center,center]{\footnotesize{$b$}}}
\pgfputat{\pgfxy(6.75,5.5)}{\pgfbox[center,center]{\footnotesize{$b$}}}
\pgfputat{\pgfxy(1.75,7.5)}{\pgfbox[center,center]{\footnotesize{$b$}}}
\pgfputat{\pgfxy(4.25,7.5)}{\pgfbox[center,center]{\footnotesize{$b$}}}
\pgfputat{\pgfxy(6.75,7.5)}{\pgfbox[center,center]{\footnotesize{$b$}}}
\pgfputat{\pgfxy(9.25,7.5)}{\pgfbox[center,center]{\footnotesize{$b$}}}
\pgfputat{\pgfxy(3.25,2.8)}{\pgfbox[center,center]{\footnotesize{$B$}}}
\pgfputat{\pgfxy(3.25,4.8)}{\pgfbox[center,center]{\footnotesize{$B$}}}
\pgfputat{\pgfxy(5.75,4.8)}{\pgfbox[center,center]{\footnotesize{$B$}}}
\pgfputat{\pgfxy(3.25,6.8)}{\pgfbox[center,center]{\footnotesize{$B$}}}
\pgfputat{\pgfxy(3.25,8.8)}{\pgfbox[center,center]{\footnotesize{$B$}}}
\pgfputat{\pgfxy(5.75,6.8)}{\pgfbox[center,center]{\footnotesize{$B$}}}
\pgfputat{\pgfxy(5.75,8.8)}{\pgfbox[center,center]{\footnotesize{$B$}}}
\pgfputat{\pgfxy(8.25,6.8)}{\pgfbox[center,center]{\footnotesize{$B$}}}
\pgfputat{\pgfxy(8.25,8.8)}{\pgfbox[center,center]{\footnotesize{$B$}}}
\pgfputat{\pgfxy(10.75,8.8)}{\pgfbox[center,center]{\footnotesize{$B$}}}
\begin{pgftranslate}{\pgfpoint{.5cm}{0cm}}
\pgfnodecircle{Node1}[fill]{\pgfxy(-.2,.5)}{0.025cm}
\pgfnodecircle{Node1}[fill]{\pgfxy(-.35,.5)}{0.025cm}
\pgfnodecircle{Node1}[fill]{\pgfxy(-.5,.5)}{0.025cm}
\pgfnodecircle{Node1}[fill]{\pgfxy(-.2,2.5)}{0.025cm}
\pgfnodecircle{Node1}[fill]{\pgfxy(-.35,2.5)}{0.025cm}
\pgfnodecircle{Node1}[fill]{\pgfxy(-.5,2.5)}{0.025cm}
\pgfnodecircle{Node1}[fill]{\pgfxy(-.2,4.5)}{0.025cm}
\pgfnodecircle{Node1}[fill]{\pgfxy(-.35,4.5)}{0.025cm}
\pgfnodecircle{Node1}[fill]{\pgfxy(-.5,4.5)}{0.025cm}
\pgfnodecircle{Node1}[fill]{\pgfxy(-.2,6.5)}{0.025cm}
\pgfnodecircle{Node1}[fill]{\pgfxy(-.35,6.5)}{0.025cm}
\pgfnodecircle{Node1}[fill]{\pgfxy(-.5,6.5)}{0.025cm}
\pgfnodecircle{Node1}[fill]{\pgfxy(-.2,8.5)}{0.025cm}
\pgfnodecircle{Node1}[fill]{\pgfxy(-.35,8.5)}{0.025cm}
\pgfnodecircle{Node1}[fill]{\pgfxy(-.5,8.5)}{0.025cm}
\end{pgftranslate}
\pgfsetendarrow{\pgfarrowlargepointed{3pt}}
\pgfxyline(2,2.2)(2,0.8)
\pgfxyline(2,4.2)(2,2.8)
\pgfxyline(2,6.2)(2,4.8)
\pgfxyline(2,8.2)(2,6.8)
\pgfxyline(2,9.7)(2,8.8)
\pgfxyline(4.5,4.2)(4.5,2.8)
\pgfxyline(4.5,6.2)(4.5,4.8)
\pgfxyline(4.5,8.2)(4.5,6.8)
\pgfxyline(4.5,9.7)(4.5,8.8)
\pgfxyline(7,6.2)(7,4.8)
\pgfxyline(7,8.2)(7,6.8)
\pgfxyline(7,9.7)(7,8.8)
\pgfxyline(9.5,8.2)(9.5,6.8)
\pgfxyline(9.5,9.7)(9.5,8.8)
\pgfxyline(12,9.7)(12,8.8)
\pgfxyline(1.5,.5)(.5,.5)
\pgfxyline(1.5,2.5)(.5,2.5)
\pgfxyline(1.5,4.5)(.5,4.5)
\pgfxyline(1.5,6.5)(.5,6.5)
\pgfxyline(1.5,8.5)(.5,8.5)
\pgfxyline(4,6.5)(2.5,6.5)
\pgfxyline(4,8.5)(2.5,8.5)
\pgfxyline(6.5,6.5)(5,6.5)
\pgfxyline(6.5,8.5)(5,8.5)
\pgfxyline(9,6.5)(7.5,6.5)
\pgfxyline(9,8.5)(7.5,8.5)
\pgfxyline(11.5,8.5)(10,8.5)
\pgfxyline(4,2.5)(2.5,2.5)
\pgfxyline(4,4.5)(2.5,4.5)
\pgfxyline(6.5,4.5)(5,4.5)
\end{pgftranslate}
\end{pgfmagnify}
\end{pgfmagnify}
\end{pgfpicture}
\caption{The bicomplexes $\tn{B}_R$, $\tn{B}^{\tn{per}}_R$, and $\tn{B}_R^-$.}
\label{figBcomplexes}
\end{figure}

\begin{example}\label{exnegcycsmoothchar0}\tn{If $R$ is a smooth algebra over a commutative ring containing $\QQ$, then
\[\tn{HN}_{n,R}\cong Z_R^n\times\prod_{i>0}H_{\tn{dR}}^{n+2i}(R),\]
where $Z_R^n:=\tn{Ker}[d:\Omega_{R/k}^n\rightarrow\Omega_{R/k}^{n+1}]$.  
See Loday \cite{LodayCyclicHomology98}, 5.1.12, page 163, for details.}
\end{example}
\hspace{16.3cm} $\oblong$ 

\begin{example}\label{exnegcycsmoothchar0}\tn{Let $R$ be a commutative unital $k$-algebra. Using the isomorphism $\tn{HC}_{0,R}\cong R$, there is a map $\tn{HC}_{0,R}\rightarrow \tn{HN}_{1,R}$ sending $r$ to $dr$ in the first factor $Z_R^1$ of $\tn{HN}_{1,R}$, since exact forms are closed.  Similarly, there is a map $\tn{HC}_{1,R}\rightarrow \tn{HN}_{2,R}$ sending $rdr'$ to $dr\wedge dr'$.} 
\end{example}
\hspace{16.3cm} $\oblong$ 

{\bf The $SBI$ Sequence.} From figure \hyperref[figcyclicbicomplexes]{\ref{figcyclicbicomplexes}} above, it is clear that there exists a short exact sequence of bicomplexes
\begin{equation}\label{equSEScycperneg}\xymatrix{0\ar[r]&\tn{CN}_R\ar[r]^I&\tn{CP}_R\ar[r]^-S&\tn{CC}_R[0,2]\ar[r]&0,}\end{equation}
where $I$ is inclusion of the negative cyclic bicomplex $\tn{CN}_R$ into the periodic cyclic bicomplex $\tn{CP}_R$, and $S$ is the projection of $\tn{CP}_R$ onto its subbicomplex in column degrees at least two, which is equal to the shifted cyclic bicomplex $\tn{CC}_R[0,2]$.  

This short exact sequence leads to a long exact sequence in homology:
\begin{equation}\label{equSBI}\xymatrix{...\ar[r]&\tn{HP}_{n+1,R}\ar[r]^S&\tn{HC}_{n-1,R}\ar[r]^B&\tn{HN}_{n,R}\ar[r]^I&\tn{HP}_{n,R}\ar[r]^S&\tn{HC}_{n-2,R}\ar[r]&...}\end{equation}

\begin{example} \tn{Let $R$ be a smooth unital algebra over a commutative ring $k$ containing $\QQ$.  The ``tail" of the SBI sequence 
\begin{equation}\label{equLESconneslowdegrees}\xymatrix{...\ar[r]&\tn{HC}_{2,R}\ar[r]^S&\tn{HC}_{0,R}\ar[r]^B&\tn{HN}_{1,R}\ar[r]^I&\tn{HC}_{1,R}\ar[r]&0,}\end{equation}
has the form
\begin{equation}\label{equLESconneslowdegrees}\xymatrix{...\ar[r]&\displaystyle\prod_{i\ge0}H_{\tn{dR}}^{2i}(R)\ar[r]^-S&R\ar[r]^-B&\displaystyle Z_R^1\times\prod_{i>0}H_{\tn{dR}}^{2i+1}(R)\ar[r]^-I&\displaystyle\prod_{i\ge0}H_{\tn{dR}}^{2i+1}(R)\ar[r]&0,}\end{equation}
where $S$ is inclusion of the ``bottom" factor $H_{\tn{dR}}^0(R)=k$, $B$ is the differential map $d$ into the first factor $Z_R^1=\tn{Ker}[d:\Omega_{R/k}^1\rightarrow\Omega_{R/k}^2]$, and $I$ is the quotient map. 
}
\end{example}
\hspace{16.3cm} $\oblong$

\subsection{Cyclic Homology and Negative Cyclic Homology of Schemes}\label{subsectioncycschemes}

As mentioned at the beginning of section \hyperref[sectioncyclichomology]{\ref{sectioncyclichomology}} above, the formal definition of the coniveau machine requires a suitable theory of negative cyclic homology for schemes.  Weibel showed how to construct such a theory in 1991, but Keller's approach involving localization pairs provides some theoretical advantages.  This approach involves a very general intermediate construct called a {\it mixed complex.}   Cyclic homology and negative cyclic homology of $k$-algebras may be defined using mixed complexes; section \hyperref[sectioncyclichomologymixedkeller]{\ref{sectioncyclichomologymixedkeller}} of the appendix describes this approach in detail.  In the present section, I describe the generalities of mixed complexes, before introducing Keller's approach to negative cyclic homology for schemes. 

A {\bf mixed complex} $(C,b,B)$ is a complex $(C,b)$ consisting of $k$-modules $C_n$ and boundary maps $b_n$, together with a family of maps $B:C_n\rightarrow C_{n+1}$ such that $b^2=B^2=bB+Bb=0$.  See Loday \cite{LodayCyclicHomology98}, page 79, or Keller \cite{KellerCyclicHomologyofExactCat96}, page 4, for more details.\footnotemark\footnotetext{Note that Keller does not explicitly include the condition that $B^2=0$.}   The notation $(C,b,B)$ is often abbreviated to $C$.  A mixed complex $C$ determines a first-quadrant bicomplex $\tn{B}_C$, illustrated in figure \ref{figBcomplexmixed} below.

The prototypical mixed complex is the complex $(\tn{C}_R,b,B)$ where $(\tn{C}_R,b)$ is the Hochschild complex of a unital $k$-algebra $R$, introduced in section \hyperref[subsectioncychomalgcommring]{\ref{subsectioncychomalgcommring}} above, and $B$ is Connes' boundary map $B$.   The corresponding bicomplex is the familiar $\tn{B}$-bicomplex $\tn{B}_R$.   Note that Bernhard Keller  \cite{KellerCyclicHomologyofExactCat96} has defined a more general mixed complex associated with $R$, that makes sense in the nonunital case.  The construction of this complex is described in section \hyperref[sectioncyclichomologymixedkeller]{\ref{sectioncyclichomologymixedkeller}} of the appendix.

\begin{figure}[H]
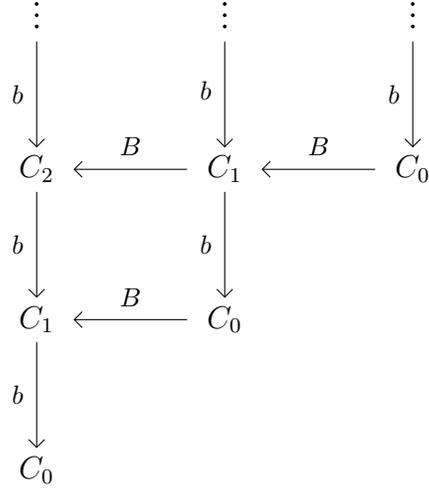

\begin{pgfpicture}{0cm}{0cm}{17cm}{6cm}
\begin{pgftranslate}{\pgfpoint{2.5cm}{-.25cm}}
\pgfputat{\pgfxy(2,.5)}{\pgfbox[center,center]{$C_0$}}
\pgfputat{\pgfxy(2,2.5)}{\pgfbox[center,center]{$C_1$}}
\pgfputat{\pgfxy(2,4.5)}{\pgfbox[center,center]{$C_2$}}
\pgfnodecircle{Node1}[fill]{\pgfxy(2,6.4)}{0.025cm}
\pgfnodecircle{Node1}[fill]{\pgfxy(2,6.55)}{0.025cm}
\pgfnodecircle{Node1}[fill]{\pgfxy(2,6.7)}{0.025cm}
\pgfputat{\pgfxy(4.5,2.5)}{\pgfbox[center,center]{$C_0$}}
\pgfputat{\pgfxy(4.5,4.5)}{\pgfbox[center,center]{$C_1$}}
\pgfnodecircle{Node1}[fill]{\pgfxy(4.5,6.4)}{0.025cm}
\pgfnodecircle{Node1}[fill]{\pgfxy(4.5,6.55)}{0.025cm}
\pgfnodecircle{Node1}[fill]{\pgfxy(4.5,6.7)}{0.025cm}
\pgfputat{\pgfxy(7,4.5)}{\pgfbox[center,center]{$C_0$}}
\pgfnodecircle{Node1}[fill]{\pgfxy(7,6.4)}{0.025cm}
\pgfnodecircle{Node1}[fill]{\pgfxy(7,6.55)}{0.025cm}
\pgfnodecircle{Node1}[fill]{\pgfxy(7,6.7)}{0.025cm}
\pgfputat{\pgfxy(1.75,1.5)}{\pgfbox[center,center]{\footnotesize{$b$}}}
\pgfputat{\pgfxy(1.75,3.5)}{\pgfbox[center,center]{\footnotesize{$b$}}}
\pgfputat{\pgfxy(1.75,5.5)}{\pgfbox[center,center]{\footnotesize{$b$}}}
\pgfputat{\pgfxy(4.25,3.5)}{\pgfbox[center,center]{\footnotesize{$b$}}}
\pgfputat{\pgfxy(4.25,5.55)}{\pgfbox[center,center]{\footnotesize{$b$}}}
\pgfputat{\pgfxy(6.75,5.5)}{\pgfbox[center,center]{\footnotesize{$b$}}}
\pgfputat{\pgfxy(3.25,2.8)}{\pgfbox[center,center]{\footnotesize{$B$}}}
\pgfputat{\pgfxy(3.25,4.8)}{\pgfbox[center,center]{\footnotesize{$B$}}}
\pgfputat{\pgfxy(5.75,4.8)}{\pgfbox[center,center]{\footnotesize{$B$}}}
\pgfsetendarrow{\pgfarrowlargepointed{3pt}}
\pgfxyline(2,2.2)(2,0.8)
\pgfxyline(2,4.2)(2,2.8)
\pgfxyline(2,6.2)(2,4.8)
\pgfxyline(4.5,4.2)(4.5,2.8)
\pgfxyline(4.5,6.2)(4.5,4.8)
\pgfxyline(7,6.2)(7,4.8)
\pgfxyline(4,2.5)(2.5,2.5)
\pgfxyline(4,4.5)(2.5,4.5)
\pgfxyline(6.5,4.5)(5,4.5)
\end{pgftranslate}
\end{pgfpicture}
\caption{Bicomplex $\tn{B}_C$ associated with a mixed complex $C$.}
\label{figBcomplexmixed}
\end{figure}
\vspace*{-.5cm}
 
{\bf Hochschild Homology of a Mixed Complex.} The {\bf Hochschild homology groups} $\tn{HH}_{n,C}$ of a mixed complex $C$ are the homology groups of the first column of $C$; i.e., the homology groups of the complex $(C,b)$.   For the mixed complex $(\tn{C}_R,b,B)$, the groups $\tn{HH}_{n,C}$ are just the Hochschild homology groups $\tn{HH}_{n,R}$ of $R$, defined in section \hyperref[subsectioncychomalgcommring]{\ref{subsectioncychomalgcommring}} above.  See Loday \cite{LodayCyclicHomology98}, page 79, for more details. 


\label{cyclichomologymixed}

{\bf Cyclic Homology of a Mixed Complex.} Since the bicomplex $\tn{B}_C$ associated with a mixed complex $C=(C,b,B)$ is an abstraction of the familiar $\tn{B}$-bicomplex $\tn{B}_R$, it is no surprise that the associated cyclic homology groups are defined in an analogous way.

\begin{defi}\label{deficycmixed} Let $C=(C,b,B)$ be a mixed complex.  The {\bf cyclic homology groups} $\tn{HC}_{n,C}$ of $C$ are the homology groups of the total complex $\tn{Tot } \tn{B}_C$ of $C$.
\end{defi}
There is a Connes periodicity exact sequence for a mixed complex, generalizing the sequence appearing in equation \hyperref[equLEShochcyc]{\ref{equLEShochcyc}} above.  


\label{negcychommixed}

{\bf Negative Cyclic Homology of a Mixed Complex.}  Like the familiar $\tn{B}$-bicomplex $\tn{B}_R$ associated with a unital $k$-algebra $R$, the bicomplex $\tn{B}_C$ associated with a mixed complex $C=(C,b,B)$ may be extended to the left to yield a periodic bicomplex $\tn{B}_C^{\tn{per}}$ and a negative bicomplex $\tn{B}_C^-$.  The bicomplexes $\tn{B}_C$ and $\tn{B}_C^-$ are natural sub-bicomplexes of $\tn{B}_C^{\tn{per}}$, which share the column in degree zero and are otherwise disjoint.   These bicomplexes fit into a diagram analogous to the one shown in figure \ref{figBcomplexes} above, where each term $R^{\otimes n+1}$ is generalized to the abstract term $C_n$.   

As in the case of algebras, the product total complexes $\tn{ToT } \tn{B}_C^-$ and $\tn{ToT } \tn{B}_C^{\tn{per}}$ generally contain infinite numbers of factors.  I follow Weibel et al. \cite{WeibelCycliccdh-CohomNegativeK06} and use the modified total complexes and $\tn{ToT}' \tn{B}_C^{\tn{per}}$ and $\tn{ToT}' \tn{B}_C^-$ to define periodic cyclic and negative cyclic homology. 

\begin{defi}\label{definegcycmixed} Let $C=(C,b,B)$ be a mixed complex.  
\begin{enumerate}
\item The {\bf periodic cyclic homology groups} $\tn{HP}_{n,C}$ of $C$ are the homology groups of the modified total complex $\tn{ToT}' \tn{B}_C^{\tn{per}}$.
\item The {\bf negative cyclic homology groups} $\tn{HN}_{n,C}$ of $C$ are the homology groups of the modified total complex $\tn{ToT}' \tn{B}_C^-$.
\end{enumerate}
\end{defi}

Note that as in the special case of $k$-algebras, the negative cyclic homology groups $\tn{HN}_{n,C}$ are generally nontrivial in negative degrees. 


{\bf Keller's Localization Pairs.}  Bernhard Keller \cite{KellerCyclicHomologyofExactCat96} defines Hochschild, cyclic, and negative cyclic homology theories for special pairs of categories called {\it localization pairs.}\footnotemark\footnotetext{{\it ``An exact category gives actually rise to a pair of exact DG categories: the category of complexes and its full subcategory of acyclic subcomplexes. This pair is an example of a ``localization pair."}  Keller \cite{KellerCyclicHomologyofExactCat96}.}   This is accomplished by first associating a mixed complex with a localization pair, then proceeding to define homology as described above.  The importance of Keller's approach in the context of the coniveau machine is that it facilitates flexible and general definitions of cyclic homology and negative cyclic homology for an algebraic scheme $X$, using the localization pair consisting of the modified category $\mbf{Par}_X$ of perfect complexes of $\ms{O}_X$-modules on a scheme $X$, described in section \hyperref[subsectionconnectivenonconnective]{\ref{subsectionconnectivenonconnective}} above, and its acyclic subcategory.   Keller's approach allows cyclic homology and negative cyclic homology to be treated on a similar footing as Bass-Thomason $K$-theory, since all three theories rely on the same underlying category of perfect complexes. 

Using Keller's approach, one may construct a functor $\mbf{HN}$ from an appropriate category of schemes to an appropriate category of spectra via the following general steps:

\begin{enumerate}
\item Begin with the localization pair $(\mbf{Par}_X,\mbf{Ac})$, where $\mbf{Par}_X$ is the modified category of perfect complexes of $\ms{O}_X$-modules on $X$ described in section \hyperref[subsectionconnectivenonconnective]{\ref{subsectionconnectivenonconnective}} above, and $\mbf{Ac}$ is its full subcategory of acyclic complexes.
\item Keller's machinery of localization pairs (\cite{KellerCyclicHomologyofExactCat96}) produces a mixed complex $C_X=(C,b,B)$.  
\item The complex $\tn{HN}_X$ is then defined to be the complex   
\[\tn{HN}_X:=\tn{ToT}'(...\rightarrow0\rightarrow C_X\overset{B}{\rightarrow} C_X[-1]\overset{B}{\rightarrow} C_X[-2]\overset{B}{\rightarrow} ...),\]
where $\tn{ToT}'$ is the modified total complex defined in equation \hyperref[equweibelmodifiedprodcomplex]{\ref{equweibelmodifiedprodcomplex}} above, and the bicomplex appearing the parentheses is constructed by sewing together copies of the mixed complex $C_X$ by means of the ``Connes boundary map" $B$.  
\item There exists an Eilenberg-MacLane functor from complexes to spectra; see \cite{WeibelHomologicalAlgebra94} for details.  This functor converts the complex $\tn{HN}_X$ to a spectrum $\mbf{HN}_X$.  
\end{enumerate}

From this viewpoint, the assignment $X\mapsto \mbf{HN}_X$ defines a presheaf of spectra on the Zariski site.  In section \hyperref[sectioncohomtheoriessupports]{\ref{sectioncohomtheoriessupports}} below, I show that this presheaf is a {\it substratum} of spectra satisfying the \'etale Mayer-Vietoris and projective bundle properties, and may thus be treated on the same footing as Bass-Thomason $K$-theory.

\section{Relative $K$-Theory and Relative Cyclic Homology}\label{subsectionrelativeKcyclic}

The third and fourth columns of the coniveau machine relate the relative versions of two different generalized cohomology theories.  A good general heuristic is the think of the theory appearing in the third column as a relative multiplicative theory, and the theory in the fourth column as a relative additive theory.  In the simple case of codimension-one cycles on an algebraic curve, considered in chapter \hyperref[chaptercurves]{\ref{chaptercurves}} above, this interpretation may be taken literally, since the first object in the third column is the sheaf of multiplicative groups $\ms{O}_{X_\ee,\ee}^*$, while the first object in the fourth column is the sheaf of additive groups $\ms{O}_X^+$.  In the more general case of codimension-$p$ cycles on an $n$-dimensional algebraic variety or scheme, the theory in the third column is the relative version of Bass-Thomason $K$-theory, and the theory in the fourth column is the relative version of negative cyclic homology.  In this section, I will discuss appropriate cases of these theories.  

In general, relative versions of generalized cohomology theories are defined so as to possess convenient functorial properties, and this does not always lead to simple descriptions in terms of absolute objects.  For example, in modern algebraic $K$-theory, relative $K$-groups are usually defined via homotopy fibers, which guarantees the existence of long exact sequences relating absolute and relative groups.\footnotemark\footnotetext{For example, see Weibel \cite{WeibelKBook} Chapter IV, page 8.}  Even for Milnor $K$-theory, ``sophisticated" definitions are sometimes appropriate.\footnotemark\footnotetext{For example, Marc Levine defines relative Milnor $K$-theory of a commutative semilocal ring $R$ with respect to an entire list of ideals $(I_1,...,I_s)$ of $R$, in terms of hypercubes of elements of $R$.}   Similarly, for cyclic homology, relative constructions are generally carried out at the level of bicomplexes, leading to long exact sequences in the total homology.\footnotemark\footnotetext{For example, see Loday \cite{LodayCyclicHomology98}, 2.1.15, page 60.} Here, however, the relative groups of principal interest arise from split nilpotent extensions of commutative rings, and may therefore be defined as kernels of maps between absolute groups.  The prototypical example of a split nilpotent extension of a commutative ring $S$ is the ring $R=S[\ee]/\ee^2\rightarrow S$ of dual numbers over $S$.  This is the extension used repeated in Chapter \hyperref[chaptercurves]{\ref{chaptercurves}}, with $S$ usually being a ring of regular or rational functions associated to an algebraic curve.  Prototypical examples of relative objects in this context are the sheaves $\ms{O}_{X_\ee,\ee}^*$ and $\ms{O}_X^+$.  The first sheaf is the kernel of the map $\ms{O}_{X_\ee}^*\rightarrow\ms{O}_X^*$, given by sending $\ee$ to zero, and the second sheaf is {\it isomorphic}, as an additive group, to the kernel of the corresponding map $\ms{O}_{X_\ee}\rightarrow\ms{O}_X$. 


\subsection{Split Nilpotent Extensions; Split Nilpotent Pairs}\label{subsectionnilpotent}

{\bf Split Nilpotent Extensions.}  A split nilpotent extension of a ring $S$ is the algebraic counterpart of an infinitesimal thickening of the affine scheme $\tn{Spec}(S)$.   Heuristically, a split nilpotent extension augments $S$ by the addition of elements ``sufficiently small" that products of sufficiently many such elements vanish.  

\begin{defi}\label{defisplitnilppairs} Let $S$ be a commutative ring with identity.  A {\bf split nilpotent extension} of $S$ is a split surjection $R\rightarrow S$ whose kernel $I$, called the extension ideal, is nilpotent.  
\end{defi}
The {\bf index of nilpotency} of $I$ is the smallest integer $N$ such that $I^N=0$.  A nilpotent extension ideal $I$ is contained in any maximal ideal $J$ of $R$, since $R/J$ is a field, and hence belongs to the Jacobson radical of $R$.  Hence, a split nilpotent extension is a special case of what Maazen and Stienstra \cite{MaazenStienstra77} call a {\it split radical extension.}\\

\begin{example}\tn{As mentioned above, the simplest nontrivial split nilpotent extension of $S$ is the extension to the ring of dual numbers $R=S[\ee]/\ee^2\rightarrow S$ over $S$.  This is the extension underlying the results of Chapter \hyperref[chaptercurves]{\ref{chaptercurves}}.  It is also the extension involved in Van der Kallen's early computation \cite{VanderKallenEarlyTK271} of relative $K_2$, described in lemma \hyperref[vanderkallenearly]{\ref{vanderkallenearly}} below, which plays a prominent role in the approach of Green and Griffiths \cite{GreenGriffithsTangentSpaces05} to studying the tangent groups at the identity of $Z_X^2$ and $\tn{Ch}_X^2$, where $X$ is an algebraic surface.  More generally, if $k$ is a field and $S$ is a $k$-algebra, then tensoring $S$ with any local artinian $k$-algebra $A$ induces a split nilpotent extension $S\otimes_k A\rightarrow S$.  These are the extensions considered by Stienstra \cite{StienstraFormalCompletion83} in his study of the {\it formal completion} $\widehat{\tn{Ch}}_{X,A,\mfr{m}}^2$ of $\tn{Ch}_X^2$ for a smooth projective surface over a field containing the rational numbers.}
\end{example}
\hspace{16.3cm} $\oblong$

\vspace*{.2cm}


{\bf Category of Split Nilpotent Pairs.}  It is sometimes useful to think of a split nilpotent extension as a {\it pair} $(R,I)$, where $R$ is a commutative ring with identity and $I$ is a nilpotent ideal of $R$ such that the quotient homomorphism $R\rightarrow R/I:=S$ is split.  This is the language used by Maazen and Stienstra \cite{MaazenStienstra77} for split {\it radical} pairs; split nilpotent pairs are a special case.  The family of such pairs forms a category as follows:

\vspace*{.2cm}

\begin{defi}\label{defisplitnilpotentpairs} The {\bf category of split nilpotent pairs} $\mbf{Nil}$ is the category whose objects are pairs $(R,I)$, and whose morphisms $(R,I)\rightarrow(R',I')$ are ring homomorphisms $R\rightarrow R'$ such that $I$ maps into $I'$.  
\end{defi}

\vspace*{.2cm}

Sometimes it is appropriate to restrict attention to a subcategory; for example, $R$ might be required to be a $k$-algebra for some ground ring $k$, and the homomorphisms $R\rightarrow R'$ defining morphisms of pairs might be required to be $k$-algebra homomorphisms.  Note that the category of split nilpotent pairs has nothing to do with the category of pairs $\mbf{P}_k$ over a distinguished category of schemes $\mbf{S}_k$, introduced in section \hyperref[sectioncohomtheoriessupports]{\ref{sectioncohomtheoriessupports}} below. 

\vspace*{.2cm}

The principal reason for introducing the category of split nilpotent pairs $\mbf{Nil}$ is to provide a ``global description" of the relative generalized algebraic Chern character introduced in section \hyperref[subsectionalgchern]{\ref{subsectionalgchern}} below.  Relative algebraic $K$-theory and relative negative cyclic homology, suitably restricted, may be viewed as functors from $\mbf{Nil}$ to the category of abelian groups.  In this context, the relative generalized algebraic Chern character is an {\it isomorphism of functors} between these two theories. 

\vspace*{.2cm}


\subsection{Relative Groups with Respect to Split Nilpotent Extensions}\label{subsectionnilpotent}

\vspace*{.2cm}

The following definition expresses relative $K$-groups and relative cyclic homology groups with respect to split nilpotent extensions as kernels of homomorphisms of the corresponding absolute groups. 

\newpage

\begin{defi}\label{defirelativeKcyclic} Let $S$ be a commutative ring with identity, and let $R$ be a split nilpotent extension of $S$ with extension ideal $I$.   For Milnor $K$-theory, Bass-Thomason $K$-theory, absolute K\"{a}hler differentials, cyclic homology, periodic cyclic homology, and negative cyclic homology, the $n$th groups of $R$ relative to $I$ are defined to be the kernels of the homomorphisms induced by the canonical split surjection $R\rightarrow S$ sending elements of $I$ to zero:
\begin{equation}\begin{array}{lcl}
       K_{n,R,I}^{\tn{M}}&=&\tn{Ker}[K_{n,R}^{\tn{M}}\rightarrow K_{n,S}^{\tn{M}}],\\
       K_{n,R,I}&=&\tn{Ker}[K_{n,R}\rightarrow K_{n,S}],\\
      \Omega^n_{R,I}&=&\tn{Ker}[\Omega_{R/\ZZ}^n\rightarrow \Omega_{S/\ZZ}^n],\\
        \tn{HC}_{n,R,I}&=&\tn{Ker}[\tn{HC}_{n,R}\rightarrow \tn{HC}_{n,S}],\\
        \tn{HP}_{n,R,I}&=&\tn{Ker}[\tn{HC}_{n,R}\rightarrow \tn{HP}_{n,S}],\\
        \tn{HN}_{n,R,I}&=&\tn{Ker}[\tn{HC}_{n,R}\rightarrow \tn{HN}_{n,S}].\end{array}
\end{equation}
\end{defi}

Similar definitions apply to other generalized cohomology theories with respect to split nilpotent ideals.  These include relative Dennis-Stein-Beilinson-Loday $K$-theory, relative algebraic de Rham theory, relative de Rham-Witt theory, and ``ordinary" relative cyclic homology, rather than the negative variant.  It is sometimes convenient to refer to the groups in definition \hyperref[defirelativeKcyclic]{\ref{defirelativeKcyclic}} as the relative groups of the split nilpotent extension $R\rightarrow S$, rather than relative groups with respect to the ideal $I$.   

\begin{example}\label{extangentgroups} (Tangent groups). \tn{When $R$ is the ring of dual numbers $R=S[\ee]/\ee^2\rightarrow S$ over $S$, the relative groups $K_{n,R,I}^{\tn{M}}$, $K_{n,R,I}$, etc., are called the {\bf tangent groups at the identity} of the absolute $K_{n,S}^{\tn{M}}$, $K_{n,S}$, etc., and are denoted by $TK_{n,S}^{\tn{M}}$, $TK_{n,S}$, etc.  This generalizes the definitions of tangent groups appearing in chapter \hyperref[chaptercurves]{\ref{chaptercurves}} in the context of smooth complex algebraic curves} 
\end{example}
\hspace{16.3cm} $\oblong$

As the terminology of tangent groups suggests, relative groups may be viewed as an abstraction of {\it linearization}.   In particular, relative groups are often simpler than the corresponding absolute groups.  This can be very useful; for example, Weibel et al. \cite{WeibelInfCohomChernNegCyclic08}, use the relative Chern character $\tn{ch}_{n,R,I}: K_{n,R,I}\rightarrow \tn{HN}_{n,R,I}$ to show that the corresponding {\it absolute} Chern character respects lambda and Adams operations in characteristic zero. 


{\bf Relative $SBI$ Sequence.} Let $R$ be a commutative $k$-algebra and $I\subset R$ an ideal.  The relative construction of cyclic, periodic cyclic, and negative cyclic homology at the level of bicomplexes leads to a short exact sequence of bicomplexes
\begin{equation}\label{equSESrelcycperneg}\xymatrix{0\ar[r]&\tn{CN}_{R,I}\ar[r]&\tn{CP}_{R,I}\ar[r]&\tn{CC}_{R,I}[0,2]\ar[r]&0,}\end{equation}
which leads to a ``relative $SBI$ sequence:"\footnotemark\footnotetext{I have suppressed the maps $S$, $B$, and $I$ here to avoid multiple uses of the symbol $I$.}
\begin{equation}\label{equSBI}\xymatrix{...\ar[r]&\tn{HP}_{n+1,R,I}\ar[r]&\tn{HC}_{n-1,R,I}\ar[r]&\tn{HN}_{n,R,I}\ar[r]&\tn{HP}_{n,R,I}\ar[r]&\tn{HC}_{n-2,R,I}\ar[r]&....}\end{equation}
In the case where $R$ is a split nilpotent extension of $S$ and the ground ring $k$ contains $\QQ$, the relative periodic cyclic homology groups $\tn{HP}_{n+1,R,I}$ vanish.  Using this result in the relative SBI sequence yields canonical isomorphisms between relative cyclic homology and relative negative cyclic homology up to a shift in degree.  This result is summarized in the following lemma:

\vspace*{.2cm}

\begin{lem}\label{HCisomHN} Let $S$ be a commutative ring with identity, and let $R$ be a split nilpotent extension of $S$ with extension ideal $I$.  Then the relative periodic cyclic homology groups $\tn{HP}_{n,R,I}$ vanish.  Hence, the relative version of the $SBI$ sequence takes the following form:
\begin{equation}\label{equSBI}\xymatrix{...\ar[r]&0\ar[r]^-S&\tn{HC}_{n-1,R,I}\ar[r]^B&\tn{HN}_{n,R,I}\ar[r]^I&0\ar[r]^S&...,}\end{equation}
The Connes boundary maps $B$ are therefore isomorphisms $B_n:\tn{HC}_{n-1,R,I}\rightarrow \tn{HN}_{n,R,I}$ in the split nilpotent case. 
\end{lem}

\vspace*{.2cm}

\begin{proof} See Loday \cite{LodayCyclicHomology98}, Corollary 4.1.15, page 121, and also Proposition 11.4.10, page 374.
\end{proof}

\vspace*{.2cm}

For future reference, I mention here that the  composition of the Connes map $B_n:\tn{HC}_{n-1,R,I}\rightarrow \tn{HN}_{n,R,I}$ with Goodwillie's isomorphism $\rho_n:K_{n,R,I}\rightarrow \tn{HC}_{n-1,R,I}$ is the relative Chern character $\tn{ch}:K_{n,R,I}\rightarrow \tn{HN}_{n,R,I}$ connecting the third and fourth columns of the coniveau machine.  This is discussed further in sections \hyperref[subsectionrelativeGoodwillie]{\ref{subsectionrelativeGoodwillie}} and \hyperref[subsectionalgchern]{\ref{subsectionalgchern}} below. 


\vspace*{.2cm}

\subsection{Relative Symbolic $K$-Theory in the Split Nilpotent Case}\label{subsubsectionfirstfewMilnorKrel}

In this section, I discuss in more detail relative Milnor $K$-groups in the split nilpotent case.   I also briefly mention the corresponding case of relative Dennis-Stein-Beilinson-Loday $K$-groups.  The purpose of this section is to facilitate computations of the type performed by Green and Griffiths. Let $S$ be a commutative ring with identity, and let $R$ be a split nilpotent extension of $S$ with extension ideal $I$.  By definition \hyperref[defiMilnorK]{\ref{defiMilnorK}} above, every entry $r_j$ in a Steinberg symbol $\{r_1,...,r_n\}\in K_{n,R}^{\tn{M}}$ belongs to the multiplicative group $R^*$ of invertible elements of $R$.  A subgroup of $R^*$ of particular importance in the context of relative Milnor $K$-theory is the subgroup $(1+I)^*$ of invertible elements of the form $1+i$, where $i$ belongs to the extension ideal $I$.   Since $I$ is assumed to be nilpotent in the present context, every such element is invertible, with inverse given by
\[(1+i)^{-1}=1-i+i^2-...,\]
where the series terminates because $i$ is nilpotent. However, I will continue to use the notation $(1+I)^*$ to emphasize the group operation.  

The following lemma characterizes the relative Milnor $K$-groups $K_{n,R,I} ^{\tn{M}}$ in terms of such elements.

\newpage

\begin{lem}\label{lemrelativegenerators} Let $S$ be a commutative ring with identity, and let $R$ be a split nilpotent extension of $S$ with extension ideal $I$. Then the relative Milnor $K$-group $K_{n,R,I} ^{\tn{M}}$ is generated by Steinberg symbols $\{r_1,...,r_n\}$ with at least one entry $r_j$ belonging to $(1+I)^*$. 
\end{lem}
\begin{proof}  See \cite{DribusMilnorK14}, lemma 3.2.3.
\end{proof}
Lemma \hyperref[lemrelativegenerators]{\ref{lemrelativegenerators}} facilitates explicit description relative Milnor $K$-groups in the split nilpotent case.   Let $R$, $S$, and $I$ be as above.

\begin{example}(Relative Milnor $K$-groups of low degree). \tn{The {\bf zeroth relative Milnor $K$-group} $K_{0,R,I}^{\tn{M}}$ is trivial, because both absolute groups $K_{0,R}^{\tn{M}}$ and $K_{0,S}^{\tn{M}}$ are, by definition, equal to the underlying additive group $\ZZ^+$ of the integers, and the induced map in Milnor $K$-theory $K_{0,R}^{\tn{M}}\rightarrow K_{0,S}^{\tn{M}}$ is by definition the identity. The {\bf first relative Milnor $K$-group} $K_{1,R,I}^{\tn{M}}$ is by definition equal to the multiplicative subgroup $(1+I)^*$ of $R^*$, since it is generated under multiplication in $R^*$ by elements of the form $1+i$ for $i\in I$. The {\bf second relative Milnor $K$-group} $K_{2,R,I}^{\tn{M}}$ is much more subtle and interesting.  A famous theorem of Spencer Bloch shows that this group is isomorphic to the group $\Omega_{R,I} ^1/dI$ of absolute K\"{a}hler differentials relative to $I$, modulo exact differentials.\footnotemark\footnotetext{These differentials are {\it absolute} in the sense that they are differentials with respect to $\ZZ$; they are {\it relative} to $I$ in the same sense that the corresponding Milnor $K$-groups are relative to $I$.}  I discuss K\"{a}hler differentials and Bloch's theorem in much more detail in section \hyperref[subsectionkahlerbloch]{\ref{subsectionkahlerbloch}} below.  In particular, theorem \hyperref[theorembloch]{\ref{theorembloch}} is Bloch's theorem.}
\end{example}
\hspace{16.3cm} $\oblong$

\begin{example}\label{extangentgroupsmilnor}(Tangent groups). \tn{As mentioned in example \hyperref[extangentgroups]{\ref{extangentgroups}} above, in the special case where $R$ is the ring of dual numbers $S[\ee]/\ee^2$ over $S$ and $I$ is the nilpotent ideal $(\ee)$, the relative groups $K_{n,R,I}^{\tn{M}}$ are by definition the tangent groups at the identity $TK_{n,S}^{\tn{M}}$ of the absolute groups $K_{n,S}^{\tn{M}}$.   In this case, there is a degree-shifting isomorphism to the absolute absolute K\"{a}hler differentials over $S$:
\begin{equation}\label{equtangentgroupmilnorK}TK_{n,R,I}^{\tn{M}}=\Omega_{S/\ZZ}^{n-1}.\end{equation}
The case $n=2$ is due to Van der Kallen \cite{VanderKallenEarlyTK271}.   The isomorphisms appearing in equation \hyperref[equtangentgroupmilnorK]{\ref{equtangentgroupmilnorK}} may be viewed as primitive Goodwillie-type isomorphisms in the sense of section \hyperref[subsectionrelativeGoodwillie]{\ref{subsectionrelativeGoodwillie}} below.  Alternatively, they may be viewed as rudimentary versions of the relative Chern character.}
\end{example}
\hspace{16.3cm} $\oblong$

\begin{example}\label{preliminaryGoodwillietype}\tn{Under ``mild" hypotheses, the {\bf higher relative Milnor $K$-groups} $K_{n,R,I}^{\tn{M}}$ of a general split nilpotent extension admit description in terms of the algebraic de Rham complex, or the more general {\it de Rham-Witt complexes.}  Hence, these groups are generally more accessible than the corresponding relative groups of modern homotopy-theoretic versions of $K$-theory, which require cyclic homology and topological cyclic homology for their description.  In particular, theorem \hyperref[theoremgoodwilliemilnor]{\ref{theoremgoodwilliemilnor}} below, proven in my paper \cite{DribusMilnorK14}, states that under a few mild hypotheses regarding invertibility and stability\footnotemark\footnotetext{This notion goes back to Van der Kallen.  It is introduced in definition \hyperref[defiVanderKallenStability]{\ref{defiVanderKallenStability}} below.} of $R$, the $(n+1)$st relative Milnor $K$-group of $R$ with respect to $I$ is isomorphic to the $n$th group of absolute K\"{a}hler differentials relative to $I$, modulo exact differentials:
\begin{equation}\label{equexamplegoodwillietype}K_{n+1,R,I}^{\tn{M}}\cong\frac{\Omega_{R,I} ^n}{d\Omega_{R,I} ^{n-1}},\end{equation}
 where the isomorphism is given by the formula:
\[\{r_0,r_1,...,r_n\}\mapsto\log(r_0)\frac{dr_1}{r_1}\wedge...\wedge\frac{dr_n}{r_n},\]
where $r_0$ is assumed to belong to $(1+I)^*$.\footnotemark\footnotetext{This assumption is permitted by the invertibility and stability hypotheses in the statement of theorem \hyperref[theoremgoodwilliemilnor]{\ref{theoremgoodwilliemilnor}}.} This isomorphism is a ``less rudimentary" version of  the relative Chern character.  It is discussed in more detail in section \hyperref[subsectionrelativeGoodwillie]{\ref{subsectionrelativeGoodwillie}} below.}
\end{example}
\hspace{16.3cm} $\oblong$

\begin{example}\label{DSBL2}\tn{In the split nilpotent case, relative Dennis-Stein-Beilinson-Loday $K$-groups also admit a description analogous to that of lemma \hyperref[lemrelativegenerators]{\ref{lemrelativegenerators}}.  In this case, the generators are Dennis-Stein-Beilinson-Loday symbols $\langle r_1,r_2\rangle$ for which either $r_1$ or $r_2$ belongs to $I$.  However, this particular result is not needed in what follows.}
\end{example}
\hspace{16.3cm} $\oblong$

\subsection{Goodwillie-Type Theorems}\label{subsectionrelativeGoodwillie}

Green and Griffiths \cite{GreenGriffithsTangentSpaces05} define the tangent sheaves $T\ms{K}_{n,X}^{\tn{M}}$ of the Milnor $K$-theory sheaves $\ms{K}_{n,X}^{\tn{M}}$ on a smooth $n$-dimensional algebraic variety $X$ to be the sheaves $\varOmega_{X/\ZZ}^{n-1}$ of absolute K\"{a}hler differentials on $X$ in the next lowest degree.  In particular, they define $T\ms{K}_{1,X}^{\tn{M}}=T\ms{O}_{X}^*=\ms{O}_X$, and $T\ms{K}_{1,X}^{\tn{M}}=T\ms{O}_{X}^*=\varOmega_{X/\ZZ}^{1}$.  The first of these two definitions is ``trivial," but the second, originally proven by Van der Kallen, plays a major role in their study.  These results may be viewed as primitive versions of what I will call {\it Goodwillie-type theorems,} in honor of Goodwillie's prototypical result, described in section \hyperref[GoodwilliesIsom]{\ref{GoodwilliesIsom}} below.  Such theorems relate relative algebraic $K$-theory and relative cyclic homology with respect to nilpotent ideals. 


\label{subsectionVanderKallen}

{\bf Van der Kallen: an Early Computation of $K_{2,S[\ee]/\ee^2,(\ee)}$.}  Consider again the simplest nontrivial split nilpotent extension of a commutative ring with identity $S$; namely, extension to the dual numbers $R=S[\ee]/\ee^2$ over $S$.  In 1971, Wilberd Van der Kallen \cite{VanderKallenEarlyTK271} considered an important special case of this extension, and proved the following result:

\begin{lem}\label{vanderkallenearly} Let $S$ be a field containing the rational numbers, or a local ring with residue field containing the rational numbers.  Then the second relative $K$-group of $S[\ee]/\ee^2$ with respect to the ideal $(\ee)$ is the first group of absolute K\"{a}hler differentials:
\begin{equation}\label{equvanderkallenearly}K_{2,S[\ee]/\ee^2,(\ee)}\cong \Omega_{S/\ZZ}^1.\end{equation}
\end{lem}
\begin{proof} See Van der Kallen \cite{VanderKallenEarlyTK271}.
\end{proof}

Van der Kallen's result features prominently in the work of Green and Griffiths.\footnotemark\footnotetext{They also include a messy but elementary symbolic proof of this result in \cite{GreenGriffithsTangentSpaces05}, appendix 6.3.1, pages 77-81.}  To understand this, recall Bloch's formula \hyperref[blochstheoremintro]{\ref{blochstheoremintro}} for the Chow groups in terms of sheaf cohomology:
\[\tn{Ch}_X^p=H_{\tn{\fsz{Zar}}}^p(X,\ms{K}_{p,X}).\]
Taking for granted the fact that ``linearization commutes with taking sheaf cohomology" in this case gives the following linearized form of Bloch's formula:
\begin{equation}\label{linearblochstheorem}
T\tn{Ch}_X^p=H_{\tn{\fsz{Zar}}}^p(X,T\ms{K}_{p,X}).
\end{equation}
Sheafifying equation \hyperref[equvanderkallenearly]{\ref{equvanderkallenearly}} and substituting it into equation \hyperref[linearblochstheorem]{\ref{linearblochstheorem}} in the case $p=2$ yields the expression
\begin{equation}\label{linearblochstheoremkahler}
T\tn{Ch}_X^2=H_{\tn{\fsz{Zar}}}^2(X,\varOmega_{X/\ZZ}^1),
\end{equation}
where $\varOmega_{X/\ZZ}^1$ is the sheaf of absolute K\"{a}hler differentials on $X$.   Working primarily from the viewpoint of complex algebraic geometry, Green and Griffiths were struck by the ``mysterious" appearance of absolute differentials in this context, and much of their study \cite{GreenGriffithsTangentSpaces05} is an effort to explain the ``geometric origins" of such differentials.  The right-hand-side of equation \hyperref[linearblochstheoremkahler]{\ref{linearblochstheoremkahler}} is what Green and Griffiths call the ``formal tangent space to $\tn{Ch}_X^2$."

Green and Griffiths generalize equation \hyperref[linearblochstheoremkahler]{\ref{linearblochstheoremkahler}} to give a definition\footnotemark\footnotetext{See \cite{GreenGriffithsTangentSpaces05} equation 8.53, page 145} of the tangent space $T\tn{Ch}_X^p$ of the $p$th Chow group of a $p$-dimensional smooth projective variety:
\begin{equation}\label{linearblochstheoremkahlerhigher}
T\tn{Ch}_X^p=H_{\tn{\fsz{Zar}}}^p(X,\varOmega_{X/\ZZ}^{p-1}).
\end{equation}
Equation \hyperref[linearblochstheoremkahlerhigher]{\ref{linearblochstheoremkahlerhigher}} is given by substituting the expression \begin{equation}\label{tangenttoMilnorK}T\ms{K}_{p,X}^{\tn{M}}=\varOmega_{X/\ZZ}^{p-1}\end{equation}
 into the linearized Bloch formula \hyperref[linearblochstheoremkahlerhigher]{\ref{linearblochstheoremkahlerhigher}}.  Equation \hyperref[tangenttoMilnorK]{\ref{tangenttoMilnorK}} is the definition Green and Griffiths use for the ``tangent sheaf" of the Milnor $K$-theory sheaf $\ms{K}_{p,X}^{\tn{M}}$.  It is the sheafification of the group-level equation \hyperref[equtangentgroupmilnorK]{\ref{equtangentgroupmilnorK}}.


\label{subsectionkahlerbloch}

{\bf A Result of Bloch.}  The following relationship between the second relative $K$-group and the first group of absolute K\"{a}hler differentials of a split nilpotent extension was first pointed out by Spencer Bloch. 

\begin{theorem}\label{theorembloch} Suppose $R$ is a split nilpotent extension of a ring $S$, with extension ideal $I$ 
whose index of nilpotency is $N$.  Suppose further that every positive integer less than or equal to $N$
 is invertible in $S$.  Then 
\begin{equation}\label{equtheoremblochk2}K_{2}(R,I)\cong \frac{\Omega_{R,I} ^1}{dI}.\end{equation}
\end{theorem}
\begin{proof} See Bloch \cite{BlochK2Artinian75}, or Maazen and Stienstra \cite{MaazenStienstra77}, Example 3.12 page 287.  
\end{proof}

\begin{example}\tn{Van der Kallen's isomorphism $K_{2,S[\ee]/\ee^2,(\ee)}\cong \Omega_{S/\ZZ}^1$ is a special case of Bloch's result.  To see this, note that in this case, $R=S[\ee]/\ee^2$, $I=(\ee)$, $S=R/I$ is assumed to be local, $I$ is nilpotent, and the underlying field $k$ contains $\QQ$.   Under these conditions, it is easy to show that the group $\Omega_{R,I} ^1/dI$ is isomorphic to the group $\Omega_S^1=\Omega_{S/\ZZ}^1$ of absolute K\"{a}hler differentials over $S$.  Indeed, by lemma \hyperref[lemrelativekahlergenerators]{\ref{lemrelativekahlergenerators}}, the relative group $\Omega_{R,I} ^1$ is generated by differentials of the form $\ee adb+c d\ee$ for some $a,b,c\in S$.  Hence, in the quotient $\Omega_{R,I} ^1/dI$, 
\begin{equation}\label{equvanderkallencomputation}d(c\ee)=cd\ee+\ee dc=0,\hspace*{.5cm}\tn{so}\hspace*{.5cm}cd\ee=-\ee dc,\end{equation}
 by the Leibniz rule and exactness.  This shows that $\Omega_{R,I} ^1/dI$ is generated by differentials of the form $\ee adb$, and it is easy to see that all the remaining relations come from $\Omega_{S/\ZZ}^1$.  Identifying $\ee adb$ with $adb$ then gives the isomorphism $\Omega_{R,I}^1\cong\Omega_{S/\ZZ}^1$.\footnotemark\footnotetext{In many instances, it is better to ``carry along the $\ee$," and to think of $\Omega_{R,I}^1$ as $\Omega_{S/\ZZ}^1\otimes_k(\ee)$, since the latter form generalizes in important ways.  For example, $\Omega_{S/\ZZ}^1\otimes_k(\ee)$ is replaced by $\Omega_{S/\ZZ}^1\otimes_k\mfr{m}$ for an appropriate local artinian $k$-algebra with maximal idea $\mfr{m}$ and residue field $k$ in the context of Stienstra's paper \cite{StienstraFormalCompletion83} on the formal completion of $\tn{Ch}_X^2$.}}
 \end{example}
\hspace{16.3cm} $\oblong$

Bloch's isomorphism \hyperref[equtheoremblochk2]{\ref{equtheoremblochk2}} helps establish the base case of the induction proof of theorem \hyperref[theoremgoodwilliemilnor]{\ref{theoremgoodwilliemilnor}} below, the details of which appear in my paper \cite{DribusMilnorK14}. 


\label{StienstraFormalCompletion}

{\bf Stienstra: Formal Completion of $\tn{Ch}_X^2$; Cartier-Dieudonn\'e Theory.}  A sheafified version of Bloch's isomorphism \hyperref[equtheoremblochk2]{\ref{equtheoremblochk2}} was used by Jan Stienstra in his study \cite{StienstraFormalCompletion83} of the {\it formal completion at the identity}  $\widehat{\tn{Ch}}_{X,A,\mfr{m}}^2$ of the Chow group  $\tn{Ch}_X^2$ of a smooth projective surface defined over a field $k$ containing the rational numbers.   Stienstra's formal completion generalizes the tangent group at the identity $T\tn{Ch}_X^2=H_{\tn{\fsz{Zar}}}^2(X,\varOmega_{X/\ZZ}^1)$ appearing in equation \hyperref[linearblochstheoremkahler]{\ref{linearblochstheoremkahler}} above.  As discussed in section \hyperref[subsectionVanderKallen]{\ref{subsectionVanderKallen}}, this expression comes from sheafifying Van der Kallen's isomorphism $K_{2,S[\ee]/\ee^2,(\ee)}\cong \Omega_{S/\ZZ}^1$, appearing in equation \hyperref[equvanderkallenearly]{\ref{equvanderkallenearly}}.  Bloch's more general result immediately gives much more.  Exploring this thread, Stienstra identifies the formal completion $\widehat{\tn{Ch}}_{X,A,\mfr{m}}^2$ as the Zariski sheaf cohomology group 
\begin{equation}\label{equstienstra}\widehat{\tn{Ch}}_{X,A,\mfr{m}}^2=H_{\tn{\fsz{Zar}}}^2\Bigg(X,\frac{\varOmega_{X\otimes_k A,X\otimes_k\mfr{m}}^1}{d(\ms{O}_X\otimes_k\mfr{m})}\Bigg),\end{equation}
where $A$ is a local artinian $k$-algebra with maximal ideal $\mfr{m}$ and residue field $k$, and where $\widehat{\tn{Ch}}_X^2$ is viewed as a functor from the category of local artinian $k$-algebras with residue field $k$ to the category of abelian groups.\footnotemark\footnotetext{See Stienstra \cite{StienstraFormalCompletion83} page 366. Stienstra writes, {\it ``We may forget about $K$-theory.  Our problem has become [analyzing $H_{\tn{Zar}}^n\big(X,\varOmega_{X\otimes A,X\otimes m}/d(\ms{O}_X\otimes_k\mfr{m})\big)$, as a functor of $(A,\mfr{m})$.]"}  Much of Stienstra's paper consists of expressing these cohomology groups in terms of the simpler objects $H_{\tn{Zar}}^n(X,\ms{O}_X)$, $\Omega_{A,\mfr{m}}^1$, and $H_{\tn{Zar}}^n(X,\varOmega_{X/\ZZ}^1)$.}  This formidable-looking expression arises from sheafification of Bloch's isomorphism $K_{2,R,I}\cong \Omega_{R,I} ^1/dI$ in theorem \hyperref[theorembloch]{\ref{theorembloch}}, substituted into the linearized version of Bloch's expression for the Chow groups in equation \hyperref[linearblochstheorem]{\ref{linearblochstheorem}}.   

\begin{example}\tn{If $A$ is the ring of dual numbers $k[\ee]/\ee^2$ over $k$, and $\mfr{m}$ is the ideal $(\ee)$, then one recovers the case studied by Green and Griffiths:
\[\widehat{\tn{Ch}}_{X,k[\ee]/\ee^2, (\ee)}^2=T\tn{Ch}_X^2=H_{\tn{\fsz{Zar}}}^2(X,\varOmega_{X/\ZZ}^1).\]}
\end{example}
\hspace{16.3cm} $\oblong$

Stienstra has carried this theory much further in his paper {\it Cartier-Dieudonn\'e Theory for Chow Groups} \cite{StienstraCartierDieudonne}, and the subsequent correction \cite{StienstraCartierDieudonneCorrection} which are beyond the scope of this book to adequately describe.  


\label{hesselholttruncated}

{\bf Hesselholt: Relative $K$-Theory of Truncated Polynomial Algebras.}  Lars Hesselholt has done a body of substantial work on the relative $K$-theory of rings with respect to nilpotent ideals, and corresponding Goodwillie-type theorems.\footnotemark\footnotetext{In a formal sense, this story may be considered complete: as Hesselholt points out (\cite{HesselholtTruncated05}, page 72) {\it ``If the ideal... ... is nilpotent, the relative $K$-theory can be expressed completely in terms of the cyclic homology of Connes and the topological cyclic homology of B\"{o}kstedt-Hsiang-Madsen."}  However, one is still concerned with unwinding the latter theories in cases of particular interest.  This is the object of Hesselholt's paper \cite{HesselholtTruncated05}, and of my paper \cite{DribusMilnorK14}.}  Here, I cite only one of his many results that is of particular relevance to the subject of this book.   In his paper {\it $K$-theory of truncated polynomial algebras}, \cite{HesselholtTruncated05}, Hesselholt computes the relative $K$-theory of polynomial algebras of the form $S[\ee]/\ee^N$ for some integer $N\ge1$, where $S$ is commutative regular noetherian algebra over a field.  The case $N=1$ is trivial, yielding back $S$, while the case $N=2$ gives the familiar extension of $S$ to the ring of dual numbers over $S$.  Hesselholt's result may be expressed as follows:
\begin{equation}\label{equhesselholt}K_{n+1,S[\ee]/\ee^N,(\ee)}\cong\bigoplus_{m\ge0}\big(\Omega_{S/\ZZ} ^{n-2m}\big)^{N-1}.\end{equation}
The case $N=2$ gives the expression
\begin{equation}\label{hesselholtlambdadecomp}K_{n+1,S[\ee]/\ee^2,(\ee)}\cong\Omega_{S/\ZZ} ^{n}\oplus \Omega_{S/\ZZ} ^{n-2}\oplus \Omega_{S/\ZZ} ^{n-4}\oplus...\end{equation}
The first summand on the right-hand side of equation \hyperref[hesselholtlambdadecomp]{\ref{hesselholtlambdadecomp}} is immediately recognizable as the tangent group at the identity of Milnor $K$-theory, identified in equation \hyperref[equtangentgroupmilnorK]{\ref{equtangentgroupmilnorK}}.  The remaining summands may be viewed, roughly speaking, as representing ``tangents to the non-symbolic part of $K$-theory."  As in equation \hyperref[equlambdacyclic]{\ref{equlambdacyclic}} above, this direct sum decomposition is a lambda decomposition, as explained below.  


\label{GoodwilliesIsom}

{\bf Goodwillie's Isomorphism.}  Goodwillie's isomorphism \cite{GoodWillieRelativeK86} is more general than the foregoing results, but is less precise and harder to describe.  In particular, it is only a {\it rational} isomorphism in general, meaning that the generalized cohomological objects involved must be tensored with $\QQ$ for the isomorphism to hold.  Further, the right-hand side is expressed in terms of cyclic homology, which is generally harder to compute than differentials.  In stating Goodwillie's isomorphism, I use notation and terminology similar to that of Loday \cite{LodayCyclicHomology98}. 

\begin{theorem}\label{theoremgoodwillie}Let $R$ be a ring and let $I$ be a two-sided nilpotent ideal in $R$.  Then for any positive integer $n$, there is a canonical isomorphism
\begin{equation}\label{equgoodwillie}\rho:K_{n,R,I}\otimes\QQ\cong\tn{HC}_{n-1,R,I}\otimes\QQ.\end{equation}
\end{theorem}
\begin{proof}See Goodwillie \cite{GoodWillieRelativeK86}.
\end{proof}

As discussed in section \hyperref[subsectionrelchernisom]{\ref{subsectionrelchernisom}} below, Goodwillie's isomorphism composes with a particular map from cyclic homology to negative cyclic homology, to give the {\it relative Chern character,} which connects the third and fourth columns of the coniveau machine. 

Hesselholt's result \hyperref[equhesselholt]{\ref{equhesselholt}} is an ``exact case" of Goodwillie's theorem, in which tensoring with $\QQ$ is not necessary.  In this case, the direct sum $\bigoplus_{m\ge0}\big(\Omega_{S/\ZZ} ^{n-2m}\big)^{N-1}$ is the lambda decomposition of the relative cyclic homology group $HC_{n-1,S[\ee]/\ee^N,(\ee)}$

\subsection{A Goodwillie-Type Theorem for Milnor $K$-Theory of Rings \\ with Many Units}\label{subsectionGoodwillieMilnor}

In this section I describe a concrete generalization of Green and Griffiths' approach to studying $T\tn{Ch}_X^2$.  Much of the material discussed here closely follows my paper \cite{DribusMilnorK14}, in which more details may be found.  Jan Stienstra has suggested \cite{StienstraPrivate} that the  theorem proven in \cite{DribusMilnorK14}, which I include here as theorem \hyperref[theoremgoodwilliemilnor]{\ref{theoremgoodwilliemilnor}} below, may be combined with the main theorem in his paper \cite{StienstraCartierDieudonne} to generalize his results in \cite{StienstraFormalCompletion83}. 


\label{vanderkallenstability}

{\bf Rings with Many Units; Van der Kallen Stability.}  Certain convenient properties of algebraic $K$-theory, particularly symbolic $K$-theories, rely on an assumption that the ring under consideration has ``enough units," or that its units are ``organized in a convenient way."   One way to make this idea precise is via Van der Kallen's \cite{VanderKallenRingswithManyUnits77} notion of stability.\footnotemark\footnotetext{This terminology seems to have first appeared in Van der Kallen, Maazen, and Stienstra's 1975 paper {\it A Presentation for Some $K_2(R,n)$} \cite{VanMaazenStienstra}.}  This notion is closely related to the {\it stable range conditions} of Hyman Bass, introduced in the early 1960's.
\begin{defi}\label{defiVanderKallenStability}Let $S$ be a  commutative ring with identity, and let $m$ be a positive integer.  
\begin{enumerate}
\item A pair $(s,s')$ of elements of $S$ is called {\bf unimodular} if $sS+s'S=S$. 
\item  $S$ is called $\mbf{m}$-{\bf fold stable} if, given any family $\{(s_j,s_j')\}_{j=1}^m$ of unimodular pairs in $S$, there exists an element $s\in S$ such that $s_j+s_j's$ is a unit in $S$ for each $j$. 
\end{enumerate}
\end{defi}
\vspace*{.3cm}
\begin{example}\label{examplesemilocalstable}\tn{A semilocal ring\footnotemark\footnotetext{A {\it commutative} ring $R$ is semilocal if and and only if it has a finite number of maximal ideals.  The general definition is that $R/J_R$ is semisimple, where $J_R$ is the Jacobson radical of $R$.} is $m$-fold stable if and only if all its residue fields contain at least $m+1$ elements.  See Van der Kallen, Maazen and Stienstra \cite{VanMaazenStienstra}, page 935, or Van der Kallen \cite{VanderKallenRingswithManyUnits77}, page 489.  In particular, for any $m$, the class of $m$-fold stable rings is much larger than the class of local rings of smooth algebraic varieties over a field containing the rational numbers, which are the rings of principal interest in the context of Green and Griffiths \cite{GreenGriffithsTangentSpaces05}.  Due to the relationship between stability and the size of residue fields, theorem \hyperref[theoremgoodwilliemilnor]{\ref{theoremgoodwilliemilnor}} below allows some of the computations of Green and Griffiths to be repeated in positive characteristic.}
\end{example}
\hspace{16.3cm} $\oblong$

The following easy lemma establishes two useful consequences of stability necessary for the proof of the theorem  \hyperref[theoremgoodwilliemilnor]{\ref{theoremgoodwilliemilnor}} below:

\begin{lem}\label{lemstability} Suppose $R$ is a split nilpotent extension of a commutative ring $S$ with identity.  Let $I$ be the extension ideal. 
\begin{enumerate} 
\item If $S$ is $m$-fold stable, then $R$ is also $m$-fold stable.
\item If $S$ is $2$-fold stable and $2$ is invertible in $S$, then every element of $R$ is the sum of 
two units.
\end{enumerate}
\end{lem}
\begin{proof} See \cite{DribusMilnorK14}, lemma 2.3.3.
\end{proof}

Lemma \hyperref[lemMilnorrelations]{\ref{lemMilnorrelations}} of section \hyperref[subsubsectionMilnorKgenerators]{\ref{subsubsectionMilnorKgenerators}} above translates the na\"{i}ve tensor algebra definition of Milnor $K$-theory into a definition in terms of Steinberg symbols and relations.  The following lemma gathers together additional relations satisfied by Steinberg symbols in the $5$-fold stable case.  The first of these, the idempotent relation, actually requires no stability assumption, but is included here, rather than in lemma \hyperref[lemMilnorrelations]{\ref{lemMilnorrelations}}, because it is information-theoretically redundant.  The other two relations, however, require the ring to have ``enough units." 

\begin{lem}\label{lemrelationsstable} Let $R$ be a $5$-fold stable ring.  Then the Steinberg symbols $\{r_0,...,r_n\}$ generating $K_{n+1} ^{\tn{M}}(R)$ satisfy the following additional relations:
\begin{enumerate}
\addtocounter{enumi}{2}
\item Idempotent relation: if any entry of the Steinberg symbol $\{r_1,...,r_n\}$ is idempotent, then $\{r_1,...,r_n\}=1$ in $K_{n,R}^{\tn{M}}$.
\item Additive inverse relation: $\{...,r,-r,...\}=1$ in $K_{n,R}^{\tn{M}}$.
\item Anticommutativity: $\{..., r',r,...\}=\{...,r,r',...\}^{-1}$ in $K_{n,R}^{\tn{M}}$
\end{enumerate}
\end{lem}
\begin{proof}  See \cite{DribusMilnorK14}, lemma 3.2.2.
\end{proof}


\label{subsectionsymbolicstability}

{\bf Symbolic $K$-Theory and Stability.}  If a ring $R$ is sufficiently stable in the sense of definition \hyperref[defiVanderKallenStability]{\ref{defiVanderKallenStability}} above, different versions of symbolic $K$-theory tend to produce isomorphic $K$-groups. An important example of such an isomorphism involves the second Dennis-Stein $K$-group and the second Milnor $K$-group.
\begin{theorem}\label{theoremvanderkallen} (Van der Kallen) Let $S$ be a commutative ring with identity, and suppose that $S$ is $5$-fold stable.  Then 
\begin{equation}K_{2,S} ^{\tn{M}}\cong D_{2,S}.\end{equation}
\end{theorem}
\begin{proof} See Van der Kallen \cite{VanderKallenRingswithManyUnits77}, theorem 8.4, page 509. 
\end{proof}
Whether or not theorem \hyperref[theoremvanderkallen]{\ref{theoremvanderkallen}} remains true if one weakens the stability hypothesis to $4$-fold stability apparently remains unknown.\footnotemark\footnotetext{Van der Kallen  \cite{VanderKallenRingswithManyUnits77} writes {\it ``We do not know if $4$-fold stability suffices for theorem 8.4."} More recently, Van der Kallen tells me \cite{VanderKallenprivate14} that the answer to this question is still apparently unknown.}

The following result, involving the ``full" relative $K$-groups, rather than merely the relative Milnor $K$-groups, does not require any stability hypothesis:

\begin{theorem}\label{theoremmaazen} (Maazen and Stienstra) If $R$ is a split radical extension of $S$ with extension ideal $I$, then 
\begin{equation}K_{2,R,I}\cong D_{2,R,I}\end{equation}
\end{theorem}
\begin{proof}Maazen and Stienstra \cite{MaazenStienstra77}, theorem 3.1, page 279.
\end{proof}


\label{subsectionkahlerbloch}

{\bf Algebraic de Rham Complex; Relative Version.}  The ``higher-degree analogues" of the group $\Omega_{R,I} ^1/dI$ appearing on the right-hand side of Bloch's isomorphism in theorem \hyperref[theorembloch]{\ref{theorembloch}} are the groups $\Omega_{R,I}^{n}/d\Omega_{R,I}^{n-1}$ for $n\ge 2$.  To understand these groups, it is useful to first examine the corresponding ``absolute" groups $\Omega_{R/\ZZ}^{n}/d\Omega_{R/\ZZ}^{n}$.  The following lemma describes these groups in terms of generators and relations:

\begin{lem}\label{lemKahlerrelations}As an abstract additive group, $\Omega_{R/\ZZ}^{n}/d\Omega_{R/\ZZ}^{n-1}$ is generated by  differentials $r_0dr_1\wedge...\wedge dr_n$, where $r_j\in R$ for all $j$, subject to the relations
\begin{enumerate}
\addtocounter{enumi}{-1}
\item $\Omega_{R/\ZZ}^{n+1}/d\Omega_{R/\ZZ}^{n}$ is abelian.
\item Additive relation: $(r_0+r_0')dr_1\wedge...\wedge dr_n=r_0dr_1\wedge...\wedge dr_n+r_0'dr_1\wedge...\wedge dr_n$.  
\item Leibniz rule: $r_0d(r_1r_1')\wedge dr_2\wedge...\wedge dr_n=r_0r_1dr_1'\wedge dr_2\wedge...\wedge dr_n+r_0r_1'dr_1\wedge dr_2\wedge...\wedge dr_n$.  
\item Alternating relation: $r_0dr_1\wedge...\wedge dr_j\wedge dr_{j+1}\wedge...\wedge dr_n=-r_0dr_1\wedge...\wedge dr_{j+1}\wedge dr_j\wedge...\wedge dr_n$.
\item Exactness: $dr_1\wedge...\wedge dr_n=0$. 
\end{enumerate}
\end{lem}
\begin{proof} This follows directly from definition \hyperref[defikahler]{\ref{defikahler}} and the properties of the exterior algebra.   
\end{proof}
These relations, of course, imply other familiar relations.  For example, the Leibniz rule and exactness together imply that
\[r_1dr_0\wedge dr_2\wedge...\wedge dr_n= -r_0dr_1\wedge dr_2\wedge...\wedge dr_n,\]
 so the alternating property ``extends to coefficients."   This, in turn, implies that additivity is not ``confined to coefficients:" 
\[r_0dr_1\wedge...\wedge d(r_j+r_j')\wedge...\wedge dr_n=r_0dr_1\wedge...\wedge dr_j\wedge...\wedge dr_n+r_0dr_1\wedge...\wedge dr_j'\wedge...\wedge dr_n.\]
Similarly, repeated use of the alternating relation implies that applying a permutation to the elements $r_0,...,r_n$ appearing in the differential $r_0dr_1\wedge...\wedge dr_n$ yields the same differential, multiplied by the sign of the permutation. 

The following lemma establishes properties of K\"{a}hler differentials analogous to the properties of Milnor $K$-groups established in lemma \hyperref[lemrelativegenerators]{\ref{lemrelativegenerators}}.  Minor stability and invertibility assumptions are necessary to yield the desired results.

\begin{lem}\label{lemrelativekahlergenerators} Let $R$ be split nilpotent extension of a $2$-fold stable ring $S$, in which $2$ is invertible.  Let $I$ be the extension ideal.  
\begin{enumerate}
\item The group $\Omega_{R,I}^n$ of absolute K\"{a}hler differentials of degree $n$ relative to $I$ is generated by differentials of the form $rd\bar{r}\wedge d\bar{r}'$, where $r$ is either $1$ or belongs to $I$, where $\bar{r}\in I$, and where $\bar{r}'\in R^*$.  
\item The group $\Omega_{R,I}^{n+1}/d\Omega_{R,I}^{n}$ is a subgroup of the group $\Omega_{R}^{n+1}/d\Omega_{R}^{n}$.  In particular, $d\Omega_{R,I}^{n}=d\Omega_{R}^{n}\cap \Omega_{R,I}^{n+1}$.  
\end{enumerate}
\end{lem}
\begin{proof}  See \cite{DribusMilnorK14}, lemma 3.3.2.
\end{proof}


\label{dlog}

{\bf The $d\log$ Map; the de Rham-Witt Viewpoint.}  The following lemma establishes the existence of the ``canonical $d\log$ map" from Milnor $K$-theory to the absolute K\"{a}hler differentials.\\

\begin{lem}\label{lemdlog} Let $R$ be a commutative ring.  The map $R^*\rightarrow\Omega_{R/\ZZ}^1$ sending $r$ to $d\log(r)=dr/r$ extends to a homomorphism $d\log:T_{R^*/\mathbb{Z}}\rightarrow \Omega_{R/\ZZ}^\bullet$ of graded rings, by sending sums to sums and tensor products to exterior products.   This homomorphism induces a homomorphism of graded rings:
\[d\log:K_R^{\tn{M}}\longrightarrow \Omega_{R/\ZZ}^\bullet\]
\begin{equation}\label{equdlog}\{r_0,...,r_n\}\mapsto\displaystyle\frac{dr_0}{r_0}\wedge...\wedge\frac{dr_n}{r_n}.\end{equation}
\end{lem}
\begin{proof} The map $d\log:T_{R^*/\mathbb{Z}}\rightarrow \Omega_{R/\ZZ}^*$ is a graded ring homomorphism by construction, since its definition stipulates that sums are sent to sums and tensor products to exterior products.  Elements of the form $r\otimes(1-r)$ in $R*\otimes_\ZZ R^*$ map to zero in $\Omega_{R/\ZZ}^2$ by the alternating property of the exterior product:
\[d\log\big(r\otimes(1-r)\big)=\frac{dr}{r}\wedge\frac{d(1-r)}{1-r}=-\frac{1}{r(1-r)}dr\wedge dr=0,\]
so $d\log$ descends to a homomorphism $K_R^{\tn{M}}\longrightarrow \Omega_{R/\ZZ}^\bullet$. 
\end{proof}

Lars Hesselholt \cite{HesselholtBigdeRhamWitt} provides a more sophisticated viewpoint regarding the $d\log$ map and the closely related map $\phi_{n+1}$ in the main theorem, expressed in equation \hyperref[equationphi]{\ref{equationphi}} above.  This viewpoint is expressed in terms of pro-abelian groups, de Rham Witt complexes, and Frobenius endomorphisms.   In describing it, I closely paraphrase a private communication \cite{Hesselholtprivate14} from Hesselholt.  For every commutative ring $R$, there exists a map of pro-abelian groups
\[d\log: K_{n,R}^{\tn{M}}\rightarrow W\Omega_{R/\ZZ}^n\]
from Milnor $K$-theory to an appropriate de Rham-Witt theory $W\Omega_{R/\ZZ}^n$, taking the Steinberg symbol $\{r_1,...,r_n\}$ to the element $d\log[r_1]...d\log[r_n]$, where $[r]$ is the Teichm\"{u}ller representative of $r$ in an appropriate ring of Witt vectors of $R$.   Here, $W\Omega_{R/\ZZ}^n$ may represent either the $p$-typical de Rham-Witt groups or the big Rham-Witt groups.  On the $p$-typical de Rham-Witt complex, there is a (divided) Frobenius endomorphism $F=F_p$; on the big de Rham-Witt complex, there is a (divided) Frobenius endomorphism $F_n$ for every positive integer $n$.  The map $d\log$ maps into the sub-pro-abelian group 
\[(W\Omega_{R/\ZZ}^n)^{F=\tn{Id}}\subset W\Omega_{R/\ZZ}^n,\]
fixed by the appropriate Frobenius endomorphism or endomorphisms.  Using the big de Rham complex, one may conjecture\footnotemark\footnotetext{This is Hesselholt's idea.} that for every commutative ring $R$ and every nilpotent ideal $I\subset R$, the induced map of relative groups
\[K_{n,R,I}^{\tn{M}}\rightarrow(W\Omega_{R,I}^n)^{F=\tn{Id}},\]
is an isomorphism of pro-abelian groups.  Expressing the right-hand-side in terms of differentials, as in the main theorem \hyperref[equmaintheorem]{\ref{equmaintheorem}}, likely requires some additional hypotheses.\footnotemark\footnotetext{Hesselholt  \cite{Hesselholtprivate14} writes, {\it ``If every prime number $l$ different from a fixed prime number $p$ is invertible in $R$, then one should be able to use the $p$-typical de Rham-Witt groups instead of the big de Rham-Witt groups. In this context, [the main theorem] can be seen as a calculation of this Frobenius fixed set... ... In order to be able to express the Frobenius fixed set in terms of differentials (as opposed to de Rham-Witt differentials), I would think that it is necessary to invert $N$... ...I do not think that inverting 2 is enough."}}


\label{goodwillietheorem}

{\bf The Theorem.}  The following ``main theorem" from my paper \cite{DribusMilnorK14} establishes the existence of a rather general Goodwillie-type isomorphism for Milnor $K$-theory:

\begin{theorem}\label{theoremgoodwilliemilnor} Suppose that $R$ is a split nilpotent extension of a $5$-fold stable ring $S$, with extension ideal $I$, 
whose index of nilpotency is $N$.  Suppose further that every positive integer less than or equal to $N$
 is invertible in $S$.  Then for every positive integer $n$,
\begin{equation}\label{equmaintheorem}K_{n+1,R,I} ^{\tn{M}}\cong \frac{\Omega_{R,I} ^n}{d\Omega_{R,I} ^{n-1}}.\end{equation}
\end{theorem}
\begin{proof} See \cite{DribusMilnorK14}, section 4. 
\end{proof}

The isomorphism \hyperref[equmaintheorem]{\ref{equmaintheorem}} is given by the map $\displaystyle \phi_{n+1}:K_{n+1,R,I} ^{\tn{M}}\longrightarrow\frac{\Omega_{R,I} ^n}{d\Omega_{R,I} ^{n-1}}$ sending
\begin{equation}\label{equationphi}
\{r_0,r_1,...,r_n\}\mapsto\log(r_0)\frac{dr_1}{r_1}\wedge...\wedge\frac{dr_n}{r_n},
\end{equation}
where $r_0$ belongs to the subgroup $(1+I)^*$ of the multiplicative group $R^*$ of $R$, and $r_1,...,r_n$ belong to $R^*$.  The inverse isomorphism is the map
\[\psi_{n+1}:\frac{\Omega_{R,I} ^n}{d\Omega_{R,I} ^{n-1}}\longrightarrow K_{n+1,R,I} ^{\tn{M}}\]
\begin{equation}\label{equationpsi}
r_0dr_1\wedge...\wedge dr_m\wedge dr_{m+1}\wedge...\wedge dr_n\mapsto\{e^{r_0r_{m+1}...r_{n}},e^{r_1},...,e^{r_m},r_{m+1},...,r_{n}\},
\end{equation}
where the elements $r_0,r_1,...,r_m$ belong to the ideal $I$, and the elements $r_{m+1},...,r_{n}$ belong to the multiplicative group $R^*$ of $R$.   Part of the proof of the theorem involves verifying that the formulae \hyperref[equationphi]{\ref{equationphi}} and \hyperref[equationpsi]{\ref{equationpsi}}, defined in terms of special group elements, extend to homomorphisms. 

In particular, Bloch's isomorphism takes the Steinberg symbol $\{1+i,r\}$ to 
\[\log(1+i)\frac{dr}{r}=\Big(i-\frac{i^2}{2}+\frac{i^3}{3}-...\Big)\frac{dr}{r},\]
where the series in the parentheses terminates because $i$ is nilpotent.

\section{Algebraic Chern Character}\label{sectionChern}

The algebraic Chern character is a natural transformation of functors from algebraic $K$-theory to negative cyclic homology.    The relative version of the algebraic Chern character, which is closely related to Goodwillie's isomorphism \hyperref[equgoodwillie]{\ref{equgoodwillie}}, induces an isomorphism of complexes between the Cousin resolution of relative $K$-theory and the Cousin resolution of relative negative cyclic homology.  Loday \cite{LodayCyclicHomology98} gives a good treatment of the Chern character in the modern algebraic setting, and compares this to the more familiar classical version.   Classically, the Chern character computes an invariant of topological $K$-theory with values in de Rham cohomology.  In the modern algebraic setting, an appropriate variant of cyclic homology replaces de Rham cohomology.  This replacement is particularly essential in noncommutative geometry, where the classical definitions of differential forms, and hence de Rham cohomology, do not apply.  In the commutative case, both theories are available, but fortunately there is no need to choose between them: the Chern character to cyclic homology lifts the classical Chern character, so nothing is lost by changing the target to the more general theory of cyclic homology. 


\subsection{Preliminaries}\label{subsectionchernprelim}

{\bf Geometric Background.}  The version of the Chern character that appears in the coniveau machine represents the last of a sequence of generalizations and abstractions arising from the theory of characteristic classes in differential and algebraic geometry, beginning in the 1940's and 1950's.  The Chern character first appeared in the theory of complex vector bundles on manifolds.   Let $\ms{E}$ be a complex vector bundle on a manifold $X$.  Important invariants of $\ms{E}$ are its Chern classes $c_i(\ms{E})$, which are elements of the de Rham cohomology of $X$.\footnotemark\footnotetext{If one takes $X$ to be a smooth projective variety, the Chern classes $c_i(\ms{E})$ may be viewed as elements of the Chow ring $\tn{Ch}_X$, defined in terms of their intersection properties.  These descend to cohomology via a cycle class map.}  These may be expressed in terms of the curvature form $\Omega$ associated with a connection $\nabla$ on $X$.  In this context, the Chern character $\tn{ch}(\ms{E})$ may be defined as a certain universal polynomial in the Chern classes $c_i(\ms{E})$.\footnotemark\footnotetext{Explicitly $\tn{ch}(\ms{E})=r+c_1+\frac{1}{2}(c_1^2-2c_2)+\frac{1}{6}(c_1^3-3c_1c_2+3c_3)+...$, where $r$ is the rank of $\ms{E}$ and where I have abbreviated $c_i(\ms{E})$ to $c_i$.  See Hartshorne \cite{HartshorneAlgebraicGeometry77}, appendix $A$, sections 3 and 4 for the case in which $X$ is a smooth projective variety.}  It may be expressed by exponentiating the curvature form $\Omega$ in an appropriate sense.  

{\bf Algebraic Version for $K_0$.}  Moving to the algebraic setting, finitely generated projective modules over a ring $R$ assume the role of complex vector bundles.  In this context also, the Chern character may be defined via a connection.  In particular, if $M$ is a finitely generated projective module over a $k$-algebra $R$,  and $\nabla:M\otimes_R\Omega_{R/k}^n\rightarrow M\otimes_R\Omega_{R/k}^{n+1}$ is a connection on $R$ with curvature form $\Omega$, then the Chern character $\tn{ch}(M)$ is the (algebraic) de Rham cohomology class of $\tn{tr}(e^{\Omega})$.  The connection $\nabla$ is omitted in the expression $\tn{ch}(M)$ due to a theorem stating that the cohomology class of $\tn{tr}(e^{\Omega})$ is independent of the choice of connection.\footnotemark\footnotetext{See Loday \cite{LodayCyclicHomology98} Theorem-Definition 8.1.7, page 260.} Thus defined, the Chern character descends to a map on the zeroth-degree part of algebraic $K$-theory:
\[\tn{ch}_0:K_{0,R}\rightarrow H_{\tn{dR}}(R;k)\]
\vspace*{-.2cm}
\begin{equation}\label{equchernzeroth}[M]\mapsto \tn{ch}(M),\end{equation}
where $[M]$ is the class of $M$ in the Grothendieck group $K_{0,R}$.\footnotemark\footnotetext{See Loday \cite{LodayCyclicHomology98} Theorem 8.2.4, page 262.}


\subsection{Relative Generalized Algebraic Chern Character}\label{subsectionalgchern}


{\bf Outline of the General Construction.}  Loday \cite{LodayCyclicHomology98} describes how the Chern character introduced in section \hyperref[subsectionchernprelim]{\ref{subsectionchernprelim}} may be progressively generalized, first to a map from $K_{0,R}$ to cyclic homology (chaper 8, section 8.3), then to a map from $K_{n,R}$ to negative cyclic homology (chapter 8, section 8.4), and finally to a relative version
\begin{equation}\label{equrelChern}\tn{ch}_n^{-}:K_{n,R,I}\rightarrow \tn{HN}_{n,R,I},\end{equation}
(chapter 11, section 11.4).  The following description, paraphrasing Loday, is intended to apprise the reader of where the details may be found and how they are used, not to provide a self-contained exposition.  A very good recent discussion of this topic appears in Corti\~nas and Weibel's 2009 paper {\it Relative Chern Characters for Nilpotent Ideals} \cite{WeibelRelativeChernNilp}.

Loday uses the {\it relative Volodin construction} to describe $\tn{ch}_n^{-}$.   There is a map of complexes
\begin{equation}\label{equloday11.4.6}C_\bullet\big(X_{R,I}\big)\rightarrow\tn{Ker}\big[\tn{ToT }\tn{CN}_R\rightarrow \tn{ToT }\tn{CN}_{R,I}\big],\end{equation}
(\cite{LodayCyclicHomology98} lemma 11.4.6, page 372), where $X_{R,I}$ is a relative Volodin-type space (11.3.2), $C_\bullet$ is  the Eilenberg-MacLane complex (11.4.5; also section \hyperref[sectioncartaneilenberg]{\ref{sectioncartaneilenberg}} of the appendix), $\tn{CN}_R$ is the negative cyclic bicomplex of $R$ (5.1), $\tn{CN}_{R,I}$ is the relative negative cyclic bicomplex defined as the kernel of the map $\tn{CN}_R\rightarrow \tn{CN}_{R/I}$ (2.1.15), and $\tn{ToT}$ is the ``total complex" whose degree-$n$ term is $\prod_{p+q=n}\tn{CN}_{p,q}$ (5.1.2).   Taking homology gives a map
\begin{equation}\label{equloday11.4.7}H_\bullet\big(X_{R,I}\big)\rightarrow \tn{HN}_{R,I},\end{equation}
(\cite{LodayCyclicHomology98} 11.4.7, page 373).  Composing on the left with the Hurewicz map 
\begin{equation}\label{equlodayhurewicz}K_{n,R,I}=\pi_n\big(X_{R,I}^+\big)\rightarrow H_n\big(X_{R,I}\big),\end{equation}
yields the desired relative Chern character. 

{\bf Ladder Diagram.}  The absolute and relative Chern characters fit into the following commutative ``ladder diagram" with exact rows.

\begin{pgfpicture}{0cm}{0cm}{17cm}{3cm}
\begin{pgftranslate}{\pgfpoint{1cm}{0cm}}
\pgfputat{\pgfxy(0,2.5)}{\pgfbox[center,center]{$...$}}
\pgfputat{\pgfxy(2,2.5)}{\pgfbox[center,center]{$K_{n+1,R}$}}
\pgfputat{\pgfxy(4.5,2.5)}{\pgfbox[center,center]{$K_{n+1,R/I}$}}
\pgfputat{\pgfxy(7,2.5)}{\pgfbox[center,center]{$K_{n,R,I}$}}
\pgfputat{\pgfxy(9.5,2.5)}{\pgfbox[center,center]{$K_{n,R}$}}
\pgfputat{\pgfxy(12,2.5)}{\pgfbox[center,center]{$K_{n,R/I}$}}
\pgfputat{\pgfxy(14,2.5)}{\pgfbox[center,center]{$...$}}
\pgfputat{\pgfxy(0,.5)}{\pgfbox[center,center]{$...$}}
\pgfputat{\pgfxy(2,.5)}{\pgfbox[center,center]{$\tn{HN}_{n+1,R}$}}
\pgfputat{\pgfxy(4.5,.5)}{\pgfbox[center,center]{$\tn{HN}_{n+1,R/I}$}}
\pgfputat{\pgfxy(7,.5)}{\pgfbox[center,center]{$\tn{HN}_{n,R,I}$}}
\pgfputat{\pgfxy(9.5,.5)}{\pgfbox[center,center]{$\tn{HN}_{n,R}$}}
\pgfputat{\pgfxy(12,.5)}{\pgfbox[center,center]{$\tn{HN}_{n,R/I}$}}
\pgfputat{\pgfxy(14,.5)}{\pgfbox[center,center]{$...$}}
\pgfputat{\pgfxy(14,2.5)}{\pgfbox[center,center]{$...$}}
\pgfputat{\pgfxy(2,1.6)}{\pgfbox[center,center]{\small{$\tn{ch}_{n+1,R}$}}}
\pgfputat{\pgfxy(4.5,1.6)}{\pgfbox[center,center]{\small{$\tn{ch}_{n+1,R/I}$}}}
\pgfputat{\pgfxy(7,1.6)}{\pgfbox[center,center]{\small{$\tn{ch}_{n,R,I}$}}}
\pgfputat{\pgfxy(9.5,1.6)}{\pgfbox[center,center]{\small{$\tn{ch}_{n,R}$}}}
\pgfputat{\pgfxy(12,1.6)}{\pgfbox[center,center]{\small{$\tn{ch}_{n,R/I}$}}}
\pgfxyline(2,2.2)(2,1.7)
\pgfxyline(4.5,2.2)(4.5,1.7)
\pgfxyline(7,2.2)(7,1.7)
\pgfxyline(9.5,2.2)(9.5,1.7)
\pgfxyline(12,2.2)(12,1.7)
\pgfsetendarrow{\pgfarrowlargepointed{3pt}}
\pgfxyline(.3,2.5)(1.3,2.5)
\pgfxyline(2.8,2.5)(3.6,2.5)
\pgfxyline(5.4,2.5)(6.3,2.5)
\pgfxyline(7.7,2.5)(8.9,2.5)
\pgfxyline(10.1,2.5)(11.3,2.5)
\pgfxyline(12.7,2.5)(13.6,2.5)
\pgfxyline(.3,.5)(1.2,.5)
\pgfxyline(2.9,.5)(3.5,.5)
\pgfxyline(5.5,.5)(6.2,.5)
\pgfxyline(7.8,.5)(8.8,.5)
\pgfxyline(10.2,.5)(11.2,.5)
\pgfxyline(12.8,.5)(13.5,.5)
\pgfxyline(2,1.3)(2,.8)
\pgfxyline(4.5,1.3)(4.5,.8)
\pgfxyline(7,1.3)(7,.8)
\pgfxyline(9.5,1.3)(9.5,.8)
\pgfxyline(12,1.3)(12,.8)
\end{pgftranslate}
\end{pgfpicture}

Loday gives a proof of this result in his book (\cite{LodayCyclicHomology98} Proposition 11.4.8, page 373).\footnotemark\footnotetext{Loday does not say that the rows are exact in the statement of the proposition, but this is what he means.}

\begin{example}\tn{Recall from example \hyperref[exnegcycsmoothchar0]{\ref{exnegcycsmoothchar0}} above that if $S$ is a smooth algebra over a commutative ring $k$ containing $\QQ$, then 
\[\tn{HN}_{2,S}\cong Z_S^2\times \prod_{i>0}H_{\tn{dR}}^{2i+2}(R).\]
The quotient of $Z_R^2$ by exact forms is the second de Rham cohomology group $H_{\tn{dR}}^{2}(R;k)$, so $\tn{HN}_{2,R}$ projects into $H_{\tn{dR}}^{2}(R;k)$.  The composite map
\[K_{2,R}^{\tn{M}}\rightarrow K_{2,R}\overset{\tn{ch}_2^-}{\longrightarrow}\tn{HN}_{2,R}\rightarrow H_{\tn{dR}}^{2}(R)\]
is the ``canonical  $d\tn{log}$ map" sending the Steinberg symbol $\{r,r'\}$ to $(dr/r)\wedge (dr'/r)$ in $H_{\tn{dR}}^{2}(R)$.  See Loday \cite{LodayCyclicHomology98}, page 275, for details.} 
\end{example}
\hspace{16.3cm} $\oblong$

\newpage

\begin{example}\tn{Suppose $S$ is a local ring at a point on a smooth complex algebraic variety, and let $R=S_\ee$ be the ring of dual numbers over $S$.  The above ladder diagram splits into short exact ladder diagrams:}
\end{example}

\begin{pgfpicture}{0cm}{0cm}{17cm}{3cm}
\begin{pgftranslate}{\pgfpoint{.5cm}{0cm}}
\pgfputat{\pgfxy(2,2.5)}{\pgfbox[center,center]{$0$}}
\pgfputat{\pgfxy(5,2.5)}{\pgfbox[center,center]{$K_{n,S_\ee,(\ee)}$}}
\pgfputat{\pgfxy(8,2.5)}{\pgfbox[center,center]{$K_{n,S_\ee}$}}
\pgfputat{\pgfxy(11,2.5)}{\pgfbox[center,center]{$K_{n,S}$}}
\pgfputat{\pgfxy(14,2.5)}{\pgfbox[center,center]{$0$}}
\pgfputat{\pgfxy(2,.5)}{\pgfbox[center,center]{$0$}}
\pgfputat{\pgfxy(5,.5)}{\pgfbox[center,center]{$\tn{HN}_{n,S_\ee,(\ee)}$}}
\pgfputat{\pgfxy(8,.5)}{\pgfbox[center,center]{$\tn{HN}_{n,S_\ee}$}}
\pgfputat{\pgfxy(11,.5)}{\pgfbox[center,center]{$\tn{HN}_{n,S}$}}
\pgfputat{\pgfxy(14,.5)}{\pgfbox[center,center]{$0$}}
\pgfputat{\pgfxy(5,1.6)}{\pgfbox[center,center]{\small{$\tn{ch}_{n,S_\ee,(\ee)}$}}}
\pgfputat{\pgfxy(8,1.6)}{\pgfbox[center,center]{\small{$\tn{ch}_{n,S_\ee}$}}}
\pgfputat{\pgfxy(11,1.6)}{\pgfbox[center,center]{\small{$\tn{ch}_{n,S}$}}}
\pgfxyline(5,2.2)(5,1.7)
\pgfxyline(8,2.2)(8,1.7)
\pgfxyline(11,2.2)(11,1.7)
\pgfsetendarrow{\pgfarrowlargepointed{3pt}}
\pgfxyline(5,1.3)(5,.8)
\pgfxyline(8,1.3)(8,.8)
\pgfxyline(11,1.3)(11,.8)
\pgfxyline(2.4,2.5)(3.9,2.5)
\pgfxyline(5.9,2.5)(7.1,2.5)
\pgfxyline(8.9,2.5)(10.1,2.5)
\pgfxyline(11.9,2.5)(13.1,2.5)
\pgfxyline(2.4,.5)(3.9,.5)
\pgfxyline(5.9,.5)(7.1,.5)
\pgfxyline(8.9,.5)(10.1,.5)
\pgfxyline(11.9,.5)(13.1,.5)
\end{pgftranslate}
\end{pgfpicture}

For $n=0$, the relative groups vanish, and the diagram is not very interesting.  For $n=1$, the diagram reduces to 

\begin{pgfpicture}{0cm}{0cm}{17cm}{3cm}
\begin{pgftranslate}{\pgfpoint{.5cm}{0cm}}
\pgfputat{\pgfxy(6.5,2.8)}{\pgfbox[center,center]{$\tn{incl}$}}
\pgfputat{\pgfxy(9.5,2.8)}{\pgfbox[center,center]{$\ee\mapsto0$}}
\pgfputat{\pgfxy(6.3,.8)}{\pgfbox[center,center]{$d$}}
\pgfputat{\pgfxy(9.5,.8)}{\pgfbox[center,center]{$\ee\mapsto0$}}

\pgfputat{\pgfxy(5.6,1.5)}{\pgfbox[center,center]{$\tn{log}$}}
\pgfputat{\pgfxy(8.6,1.5)}{\pgfbox[center,center]{$d\tn{log}$}}
\pgfputat{\pgfxy(11.6,1.5)}{\pgfbox[center,center]{$d\tn{log}$}}

\pgfputat{\pgfxy(2,2.5)}{\pgfbox[center,center]{$0$}}
\pgfputat{\pgfxy(5,2.5)}{\pgfbox[center,center]{$1+\ee S$}}
\pgfputat{\pgfxy(8,2.5)}{\pgfbox[center,center]{$S_\ee^*$}}
\pgfputat{\pgfxy(11,2.5)}{\pgfbox[center,center]{$S^*$}}
\pgfputat{\pgfxy(14,2.5)}{\pgfbox[center,center]{$0$}}
\pgfputat{\pgfxy(2,.5)}{\pgfbox[center,center]{$0$}}
\pgfputat{\pgfxy(5,.5)}{\pgfbox[center,center]{$\ee S$}}
\pgfputat{\pgfxy(8,.5)}{\pgfbox[center,center]{$Z_{S_\ee}^1\times H_{\tn{dR}}$}}
\pgfputat{\pgfxy(11,.5)}{\pgfbox[center,center]{$Z_S^1\times H_{\tn{dR}}$}}
\pgfputat{\pgfxy(14,.5)}{\pgfbox[center,center]{$0$}}

\pgfsetendarrow{\pgfarrowlargepointed{3pt}}
\pgfxyline(5,2.2)(5,.8)
\pgfxyline(8,2.2)(8,.8)
\pgfxyline(11,2.2)(11,.8)
\pgfxyline(2.4,2.5)(3.9,2.5)
\pgfxyline(5.9,2.5)(7.1,2.5)
\pgfxyline(8.9,2.5)(10.1,2.5)
\pgfxyline(11.9,2.5)(13.1,2.5)
\pgfxyline(2.4,.5)(3.9,.5)
\pgfxyline(5.6,.5)(6.9,.5)
\pgfxyline(9.1,.5)(10,.5)
\pgfxyline(11.9,.5)(13.1,.5)
\end{pgftranslate}
\end{pgfpicture}

For $n=2$, the diagram is 

\begin{pgfpicture}{0cm}{0cm}{17cm}{3cm}
\begin{pgftranslate}{\pgfpoint{.5cm}{0cm}}
\pgfputat{\pgfxy(6.5,2.8)}{\pgfbox[center,center]{$\tn{incl}$}}
\pgfputat{\pgfxy(9.5,2.8)}{\pgfbox[center,center]{$\ee\mapsto0$}}
\pgfputat{\pgfxy(6.3,.8)}{\pgfbox[center,center]{$d$}}
\pgfputat{\pgfxy(9.5,.8)}{\pgfbox[center,center]{$\ee\mapsto0$}}

\pgfputat{\pgfxy(5.9,1.5)}{\pgfbox[center,center]{$\tn{log}\times d\tn{log}$}}
\pgfputat{\pgfxy(8.6,1.5)}{\pgfbox[center,center]{$d\tn{log}$}}
\pgfputat{\pgfxy(11.6,1.5)}{\pgfbox[center,center]{$d\tn{log}$}}

\pgfputat{\pgfxy(2,2.5)}{\pgfbox[center,center]{$0$}}
\pgfputat{\pgfxy(5,2.5)}{\pgfbox[center,center]{$K_{2,S_\ee,(\ee)}^{\tn{M}}$}}
\pgfputat{\pgfxy(8,2.5)}{\pgfbox[center,center]{$K_{2,S_\ee}^{\tn{M}}$}}
\pgfputat{\pgfxy(11,2.5)}{\pgfbox[center,center]{$K_{2,S}^{\tn{M}}$}}
\pgfputat{\pgfxy(14,2.5)}{\pgfbox[center,center]{$0$}}
\pgfputat{\pgfxy(2,.5)}{\pgfbox[center,center]{$0$}}
\pgfputat{\pgfxy(5,.5)}{\pgfbox[center,center]{$\ee\Omega^1_{S}$}}
\pgfputat{\pgfxy(8,.5)}{\pgfbox[center,center]{$Z_{S_\ee}^2\times H_{\tn{dR}}$}}
\pgfputat{\pgfxy(11,.5)}{\pgfbox[center,center]{$Z_S^2\times H_{\tn{dR}}$}}
\pgfputat{\pgfxy(14,.5)}{\pgfbox[center,center]{$0$}}

\pgfsetendarrow{\pgfarrowlargepointed{3pt}}
\pgfxyline(5,2.2)(5,.8)
\pgfxyline(8,2.2)(8,.8)
\pgfxyline(11,2.2)(11,.8)
\pgfxyline(2.4,2.5)(3.9,2.5)
\pgfxyline(5.9,2.5)(7.1,2.5)
\pgfxyline(8.9,2.5)(10.1,2.5)
\pgfxyline(11.9,2.5)(13.1,2.5)
\pgfxyline(2.4,.5)(3.9,.5)
\pgfxyline(5.6,.5)(6.9,.5)
\pgfxyline(9.1,.5)(10,.5)
\pgfxyline(11.9,.5)(13.1,.5)
\end{pgftranslate}
\end{pgfpicture}

\hspace{16.3cm} $\oblong$


\subsection{Relative Chern Character as an Isomorphism of Functors}\label{subsectionrelchernisom}

The importance of the relative (generalized algebraic) Chern character $\tn{ch}$ in the context of the coniveau machine is that it induces the isomorphism between the third and fourth columns.  This construction depends on the functorial properties of $\tn{ch}$ and of the coniveau spectral sequence, introduced in section \hyperref[sectionconiveauSS]{\ref{sectionconiveauSS}} below.  For $\tn{ch}$, the necessary properties are captured by the following lemma:

\begin{lem}\label{lemrelchernisomfunctors} The (generalized algebraic) Chern character $\tn{ch}_{n,R}:K_{n,R}\rightarrow \tn{HN}_{n,R}$ extends to a natural transformation of functors from algebraic $K$-theory to negative cyclic homology.  Its relative version $\tn{ch}_{n,R,I}:K_{n,R,I}\rightarrow \tn{HN}_{n,R,I}$ extends to a natural transformation of functors from relative algebraic $K$-theory to relative negative cyclic homology, viewed as functors from an appropriate category of pairs $(R,I)$ to the category of abelian groups.  When the chosen category of pairs is the category $\mbf{Nil}$ of split nilpotent pairs defined in definition \hyperref[defisplitnilppairs]{\ref{defisplitnilppairs}} above, $\tn{ch}$ is an isomorphism of functors.
\end{lem}
\begin{proof} See Corti\~nas and Weibel \cite{WeibelRelativeChernNilp}, section 6.
\end{proof}


Goodwillie's isomorphism, cited in theorem \hyperref[theoremgoodwillie]{\ref{theoremgoodwillie}} above, may be written in the form $\rho: K_{n,R,I}\rightarrow\tn{HC}_{n-1,R,I}$, when $R$ is a $\QQ$ algebra.   In this case, $\rho$ ``coincides with" $\tn{ch}$ in the following sense:

\begin{lem}\label{lemgoodwilliechern} Let $R$ be a $\QQ$-algebra, and $I\subset R$ a nilpotent ideal.  Then the composite map
\[K_{n,R,I}\overset{\rho}{\longrightarrow}\tn{HC}_{n-1,R,I}\overset{B}{\longrightarrow}\tn{HN}_{n,R,I}\]
is the relative Chern character, where $B$ is the map from Connes' cyclic bicomplex.
\end{lem}
\begin{proof} See Loday \cite{LodayCyclicHomology98} Theorem 11.4.11, page 374.
\end{proof}

\section{Cohomology Theories with Supports}\label{sectioncohomtheoriessupports}

The top row of the coniveau machine for codimension-$p$ cycles on a smooth algebraic variety $X$ is defined in terms of the sheaves of algebraic $K$-theory $\ms{K}_{p,X}$, ``augmented $K$-theory" $\ms{K}_{p,X_\ee}$, relative $K$-theory $\ms{K}_{p,X_\ee,\ee}$, and relative negative cyclic homology $\ms{HN}_{p,X_\ee,\ee}$.   The remaining rows may be constructed in a functorial manner from the top row, but this requires extending algebraic $K$-theory and negative cyclic homology to {\it cohomology theories with supports.}  This is essentially because the terms appearing in the lower rows depend on information associated with closed subsets $Z$ of $X$.  From the perspective of studying Chow groups, this is not very surprising, since the Chow groups are intimately related to closed subsets of $X$; namely the supports of algebraic cycles.  

In this section, I introduce the theory of cohomology theories with supports, following closely the development of Colliot-Th\'el\`ene, Hoobler, and Kahn \cite{CHKBloch-Ogus-Gabber97}.  A cohomology theory with supports is a special family of contravariant functors from a category of pairs $(X,Z)$, where $X$ is a scheme and $Z$ is a closed subset,  to an abelian category $\mbf{A}$.  In many important cases, a cohomology theory with supports may be defined in terms of a {\it substratum}, which is a special contravariant functor from an appropriate category of pairs to the category $\mbf{Ch}_{\mbf{A}}$ of chain complexes of objects of an abelian category $\mbf{A}$, or to an appropriate category $\mbf{E}$ of topological spectra.  The substrata of principal interest for studying the infinitesimal theory of Chow groups are Bass and Thomason's nonconnective $K$-theory spectrum $\mathbf{K}$, and the negative cyclic homology spectrum $\mathbf{HN}$.  The corresponding cohomology theories with supports yield group-valued nonconnective $K$-theory functors $K_p$, and negative cyclic homology functors $\tn{HN}_p$.  


\subsection{Cohomology Theories with Supports}\label{subsectioncohomsupportssubstrata}

\label{subsectiondistinguishedcatschemes}

{\bf Some Distinguished Categories of Schemes.}  Let $\mbf{S}_k$ be a full subcategory of the category of all schemes over a field $k$, stable under \'etale extensions.    Assume that $\mbf{S}_k$ includes the prime spectrum $\tn{Spec }k$ of $k$, and that whenever a scheme $X$ belongs to  $\mbf{S}_k$, the projective space $\mbb{P}^1 _X:=\mbb{P}^1 _{\ZZ}\times_{\tn{Spec } \ZZ}X$ also belongs to $\mbf{S}_k$.  Examples of categories satisfying these assumptions include the category $\mbf{Sep}_k$ of separated schemes over $k$, and the category $\mbf{Sm}_k$ of smooth schemes over $k$.  Given the category $\mbf{S}_k$, define $\mbf{P}_k$ to be the category of pairs $(X, Z)$, where $X$ belongs to $\mbf{S}_k$, and where $Z$ is a closed subset of $X$.   A morphism of pairs $f: (X', Z') \to (X, Z)$ is a morphism $f: X' \to X$ in $\mbf{S}_k$ such that $f^{-1}(Z) \subset Z'$.    


{\bf Cohomology Theories with Supports.}  The defining property of a cohomology theory with supports over a category of pairs $\mbf{P}_k$ is the existence of a certain long exact sequence of cohomology groups with supports for every {\it triple} $(X,Y,Z)$, where $X$ belongs to the distinguished category of schemes $\mbf{S}_k$, and $Y$ and $Z$ are closed subsets of $X$ such that $Z \subseteq Y \subseteq X$.  The following definition makes this precise:

\begin{defi}\label{deficohomsupports} A {\bf cohomology theory with supports} on the category of pairs $\mbf{P}_k$ over a distinguished category of schemes $\mbf{S}_k$ is a family $H:=\{H^n\}_{n\in\mathbb{Z}}$ of contravariant functors
 \[H^n:(X, Z) \mapsto H^n _{X\tn{ \footnotesize{on} } Z}\]
 from $\mbf{P}_k$ to an abelian category $\mbf{A}$, satisfying the following condition: for any triple $Z \subseteq Y \subseteq X$, where $Y, Z$ are closed in $X$, there exists a long exact sequence:
 \[...\longrightarrow H^n_{X\tn{ \footnotesize{on} } Z}\overset{i^n}{\longrightarrow}H^n_{X\tn{ \footnotesize{on} } Y}\overset{j^n}{\longrightarrow}H^n_{X-Z\tn{ \footnotesize{on} } Y-Z}\overset{d^n}{\longrightarrow}H^{n+1}_{X\tn{ \footnotesize{on}  } Z}\longrightarrow...\]
where the maps $i^n$ and $j^n$ are induced by the morphisms of pairs $(X,Z)\leftarrow(X,Y)$ and $(X,Y)\leftarrow(X-Z,Y-Z)$, and where $d^n$ is the $n$th {\bf connecting morphism}.
\end{defi}

The most important examples of cohomology theories with supports for the purposes of this book are Bass-Thomason $K$-theory and negative cyclic homology.  

The family $\mbf{Co}_{\mbf{P}_k}$ of all cohomology theories with supports on $\mbf{P}_k$, together with their natural transformations, is a contravariant {\bf functor category} on $\mbf{P}_k$, where natural transformations between cohomology theories with supports are the morphisms in $\mbf{Co}_{\mbf{P}_k}$.  An important example of such a morphism is the Chern character between Bass-Thomason $K$-theory and negative cyclic homology.  



{\bf ``New Theories out of Old;" Multiplying by a Fixed Separated Scheme.} Colliot-Th\'el\`ene, Hoobler, and Kahn \cite{CHKBloch-Ogus-Gabber97} devote considerable attention to the notion of defining a new cohomology theory with supports $H'$ by modifying the definition of an existing theory $H$ in some useful way.  The point of such an exercise is that important properties of $H$ can sometimes be easily established for $H'$ by modifying their proofs for $H$.  The example of cardinal importance in this book is the property of {\it effaceability,} discussed in section \hyperref[effaceability]{\ref{effaceability}} below.  

In the context of this book, the ``new" cohomology theory with supports $H'$ will usually be an ``augmented" or ``relative" version of the ``old" theory $H$.  A case of particular interest is where $H^Y$ is given by ``multiplying by a fixed separated scheme."  This notion is made precise in the following definition:

\newpage

\begin{defi}\label{lemnewoutofoldcohom} Let $\mbf{S}_k$ a category of schemes over a field $k$ satisfying the conditions given at the beginning of section \hyperref[subsectioncohomsupportssubstrata]{\ref{subsectioncohomsupportssubstrata}}.   Let $Y$ be a separated scheme over $k$.  Let $H$ be a cohomology theory with supports on the category of pairs over $\mbf{S}_k$, with values in an abelian category $\mbf{A}$.   Define a family of functors $H^{Y}=\{H_n^Y\}_{n\in\ZZ}$ on the category of pairs over $\mbf{S}_k$ as follows:
\begin{equation}\label{equnewoutofoldcohom}H_{n,X\tn{ on } Z}^Y:=H_{n,X\times_k Y\tn{ on } Z\times_k Y }.\end{equation}
As usual, the notation $\times_k$ for the fiber product means $\times_{\tn{Spec }k}$.  I will call the family $H^{Y}$ the {\bf augmented version of $H$ with respect to $Y$.} 
\end{defi}

It is an easy exercise to show that $H^Y$ is indeed a cohomology theory with supports; see Colliot-Th\'el\`ene, Hoobler, and Kahn \cite{CHKBloch-Ogus-Gabber97}, 5.5(1). 


\subsection{Substrata}\label{subsectionsubstrata}

{\bf Substrata.}  A substratum is a complex-valued or spectrum-valued functor that serves as a ``precursor" to a cohomology theory with supports.  The corresponding cohomology functors are given by taking cohomology groups of complexes or homotopy groups of spectra, as explained in definition \hyperref[defisubstratumtocohom]{\ref{defisubstratumtocohom}} below. 

\begin{defi}\label{defisubstratum}  A {\bf substratum} on a distinguished category of schemes $\mbf{S}_k$ is a contravariant functor $C:X \to C_X$ from $\mbf{S}_k$ to the category $\mbf{Ch}_{\mbf{A}}$ of chain complexes of objects of an abelian category $\mbf{A}$, or to an appropriate category $\mbf{E}$ of topological spectra.   
\end{defi}

For each pair $(X,Z)$ in $\mbf{P}_k$, one may then define a ``complex or spectrum with supports" $C_{X\tn{ \footnotesize{on} } Z}$ by taking the homotopy fiber of the map of complexes or spectra $C_X \to C_{X-Z}$.   These fit together to give a short exact sequence of complexes or spectra: 
\begin{lem}\label{lemsubstrattocohom} For any triple $(X, Y, Z)$ as above, there exists a sequence of complexes, or spectra, exact up to homotopy:
\begin{equation}\label{SEScomplexorspectra}
 0 \longrightarrow C_{X\tn{ \footnotesize{on} } Z} \longrightarrow C_{X\tn{ \footnotesize{on}  } Y} \longrightarrow
C_{X-Z\tn{ \footnotesize{on}  } Y-Z}\longrightarrow 0.
\end{equation}
\end{lem}
A cohomology theory with supports may then be defined by taking cohomology groups of complexes or homotopy groups of spectra:
\begin{defi}\label{defisubstratumtocohom} Let $C$ be a substratum on the category of pairs $\mbf{P}_k$.  If the target category of $C$ is a category of complexes $\mbf{Ch}_{\mbf{A}}$, define functors $H^n$ from $\mbf{P}_k$ to $\mbf{A}$ by taking cohomology of complexes:
\[H^n_{X\tn{ on } Z} := H^q\big(C_{X\tn{ \footnotesize{on} } Z}\big).\]
If the target category of $C$ is a category of spectra $\mbf{E}$, define functors functors $H^n$ from $\mc{P}_k$ by taking homotopy groups of spectra:
\[H^n_{X\tn{ \footnotesize{on} } Z} := \pi_{-q}\big(C_{X\tn{ \footnotesize{on} } Z}\big).\]
\end{defi}

\newpage

By basic homological algebra, the short exact sequence of complexes or spectra in equation \hyperref[SEScomplexorspectra]{\ref{SEScomplexorspectra}} induces a long exact sequence of the groups $H^n_{X\tn{ on } Z}$: 
\[...\longrightarrow H^n_{X\tn{ \footnotesize{on} } Z}\longrightarrow H^n_{X\tn{ \footnotesize{on} } Y}\longrightarrow H^n_{X-Z\tn{ \footnotesize{on} } Y-Z}\longrightarrow H^{n+1}_{X\tn{ \footnotesize{on}  } Z}\longrightarrow...\]
Hence, the functors $H^n$ define a cohomology theory with supports as defined in definition \hyperref[deficohomsupports]{\ref{deficohomsupports}} above.  

\vspace*{.2cm}

The most important examples of substrata for the purposes of this book are Bass-Thomason $K$-theory $\mbf{K}$ and negative cyclic homology $\mbf{HN}$. 

\vspace*{.2cm}

The family $\mbf{Sub}_{\mbf{P}_k}$ of all substrata on $\mbf{P}_k$, together with their natural transformations, is a contravariant functor category on the category of pairs $\mbf{P}_k$.   The Chern character between Bass-Thomason $K$-theory and negative cyclic homology may be understood at the level of substrata by working with the spectra $\mbf{K}$ and $\mbf{HN}$. 

\vspace*{.2cm}


{\bf ``New Theories out of Old" at the Substratum Level.}  As in the case of cohomology theories with supports, it is often useful to modify a substratum $C$ to obtain a new substratum $C'$.  The case of principal interest here is again given by multiplying by a fixed separated scheme.

\vspace*{.2cm}

\begin{defi}\label{lemnewoutofoldsubstrat} Let $\mbf{S}_k$ a category of schemes over a field $k$ satisfying the conditions given at the beginning of section \hyperref[subsectioncohomsupportssubstrata]{\ref{subsectioncohomsupportssubstrata}}.   Let $Y$ be a separated scheme over $k$.  Let $C$ be a substratum on $\mbf{S}_k$, with values in an abelian category $\mbf{A}$.   Define a functor $C^{Y}$ on $\mbf{S}_k$ as follows:
\begin{equation}\label{equnewoutofoldsubstrat}C_X^Y:=C_{Y\times_k T}.\end{equation}
\end{defi}

It is an easy exercise to show that $C^Y$ is indeed a substratum; see Colliot-Th\'el\`ene, Hoobler, and Kahn \cite{CHKBloch-Ogus-Gabber97}, 5.5(1). 


\subsection{Generalized Deformation Groups and Generalized Tangent Groups of Chow Groups}\label{subsectiongeneralizeddeftan}

\vspace*{.2cm}

Now let $X$ be a smooth algebraic variety over a field $k$, and let $Y$ be a fixed separated scheme over $k$.  Let $C$ be the substratum $\mbf{K}$ of Bass-Thomason $K$-theory, and let $C^Y$ be the modified substratum whose value on $X$ is the augmented $K$-theory spectrum $\mbf{K}_{X\times_kY}$ defined in section \hyperref[subsectionconnectivenonconnective]{\ref{subsectionconnectivenonconnective}} above.  Recall that the relative $K$-theory spectrum $\mbf{K}_{X\times_kY,Y}$ is the homotopy fiber of the morphism of spectra $\mbf{K}_{X\times_kY}\rightarrow\mbf{K}_X$. Define $K$-theory groups and sheaves on $X$ from these spectra in the usual way.  Via Bloch's formula for the Chow groups, the augmented and relative $K$-theory spectra $\mbf{K}_{X\times_kY}$ and $\mbf{K}_{X\times_kY,Y}$ enable definition of generalized deformation groups and generalized tangent groups of the Chow groups.   The following definition makes this precise: 

\newpage

\begin{defi}\label{defigendefgroupChow} Let $X$ be a smooth algebraic variety over a field $k$, and let $Y$ be a separated $k$-scheme, not necessarily smooth.  
\begin{enumerate}
\item The {\bf generalized deformation group} $D_Y\tn{Ch}_X^p$ of the $p$th Chow group $\tn{Ch}_X^p$ of $X$ with respect to $Y$ is the $p$th Zariski sheaf cohomology group of the augmented $K$-theory sheaf $\ms{K}_{p,X\times_kY}$ on $X$:
\begin{equation}\label{equgendefchow}D_Y\tn{Ch}_X^p:=H_{\tn{\fsz{Zar}}}^p(X,\ms{K}_{p,X\times_kY}).\end{equation}
\item The {\bf generalized tangent group at the identity} $T_Y\tn{Ch}_X^p$ of the $p$th Chow group $\tn{Ch}_X^p$ of $X$ with respect to $Y$ is the $p$th Zariski sheaf cohomology group of the relative $K$-theory sheaf $\ms{K}_{p,X\times_kY,Y}$ on $X$:
\begin{equation}\label{equgentanchow}T_Y\tn{Ch}_X^p:=H_{\tn{\fsz{Zar}}}^p(X,\ms{K}_{p,X\times_kY,Y}).\end{equation}
\end{enumerate}
\end{defi}


\subsection{Effaceability}\label{effaceability}

Effaceability is a technical condition involving the behavior of a cohomology theory with supports, or a substratum, with respect to open neighborhoods of finite collections of points in smooth affine schemes in $\mbf{S}_k$.  This condition is useful for the construction of the coniveau machine because it guarantees the exactness of certain {\it Cousin complexes,} as described below.   In particular, this provides a method of computing the generalized deformation groups and generalized tangent groups of Chow groups. 

\begin{defi}\label{defieffaceability} Let $X\in\mbf{S}_k$ be an affine scheme, and let $\{t_1,...,t_r\}$ be a finite set of points of $X$.  
\begin{enumerate}
\item A cohomology theory with supports $H=\{H^n\}_{n\in\mathbb{Z}}$ on the category of pairs $\mbf{P}_k$ over $\mbf{S}_k$ is called {\bf effaceable}\footnotemark\footnotetext{Colliot-Th\'el\`ene, Hoobler, and Kahn \cite{CHKBloch-Ogus-Gabber97} call this condition {\bf strict effaceability} (definition 5.1.8, page 28) but I drop the adjective ``strict" since this is the only such property used here. } {\bf at} $\{t_1,...,t_r\}$ if, given any integer $p\ge0$, any open neighborhood $W$ of $\{t_1,...,t_r\}$ in $X$, and any closed subset $Z\subset W$ of codimension at least $p+1$, there exists a smaller open neighborhood $U\subseteq W$ of $\{t_1,...,t_r\}$ in $X$ and a closed subset $Z'\subseteq W$, containing $Z$, with $\tn{codim}_{W}(Z')\ge p$, such that the map 
\[H^n_{U\tn{ \footnotesize{on} } Z\cap U}\rightarrow H^n_{U\tn{ \footnotesize{on} } Z'\cap U}\]
is zero for all $n\in\mathbb{Z}$. $H$ is called {\bf effaceable} if this condition is satisfied for any smooth $X$ and any $\{t_1,...,t_r\}$.  
\item A substratum $C$ on $\mbf{S}_k$ is called {\bf effaceable at} $\{t_1,...,t_r\}$ if, given $p, W,$ and $Z$ as above, there exists $U$ and $Z'$ as above such that the map of substrata
\[C_{U\tn{ \footnotesize{on} } Z\cap U}\rightarrow C_{U\tn{ \footnotesize{on} } Z'\cap U}\]
is nullhomotopic.  A substratum is called {\bf effaceable} if this condition is satisfied for any smooth $X$ and any $\{t_1,...,t_r\}$.  
\end{enumerate}
\end{defi}

It is natural to try to draw schematic diagrams illustrating the effaceability condition.  Such diagrams are not necessarily very enlightening, partly because it is hard to represent the conditions involving codimension.  Figure \hyperref[figeffaceability]{\ref{figeffaceability}} below makes no attempt to illustrate codimension accurately, merely showing the inclusion properties of the subsets $W, U, Z,$ and $Z'$ of $X$.  The diagram at least illustrates the heuristic idea that effaceability ``has something to do with interpolation."  

\begin{figure}[H]
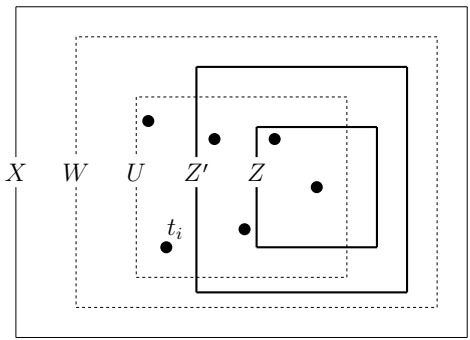

\begin{pgfpicture}{0cm}{0cm}{17cm}{5cm}
\begin{pgfmagnify}{.8}{.8}
\begin{pgftranslate}{\pgfpoint{5cm}{.5cm}}
\pgfxyline(-1,0)(6.5,0)
\pgfxyline(6.5,0)(6.5,5.5)
\pgfxyline(6.5,5.5)(-1,5.5)
\pgfxyline(-1,5.5)(-1,3)
\pgfxyline(-1,2.5)(-1,0)
\begin{pgfscope}
\pgfsetdash{{0.05cm}{0.05cm}}{0cm}
\pgfxyline(0,.5)(6,.5)
\pgfxyline(6,.5)(6,5)
\pgfxyline(6,5)(0,5)
\pgfxyline(0,5)(0,3)
\pgfxyline(0,2.5)(0,.5)
\pgfxyline(1,1)(4.5,1)
\pgfxyline(4.5,1)(4.5,4)
\pgfxyline(4.5,4)(1,4)
\pgfxyline(1,4)(1,3)
\pgfxyline(1,2.5)(1,1)
\end{pgfscope}
\pgfsetlinewidth{1pt}
\pgfxyline(2,.75)(5.5,.75)
\pgfxyline(5.5,.75)(5.5,4.5)
\pgfxyline(5.5,4.5)(2,4.5)
\pgfxyline(2,4.5)(2,3)
\pgfxyline(2,2.5)(2,.75)
\pgfxyline(3,1.5)(5,1.5)
\pgfxyline(5,1.5)(5,3.5)
\pgfxyline(5,3.5)(3,3.5)
\pgfxyline(3,3.5)(3,3)
\pgfxyline(3,2.5)(3,1.5)
\pgfputat{\pgfxy(-1,2.75)}{\pgfbox[center,center]{$X$}}
\pgfputat{\pgfxy(0,2.75)}{\pgfbox[center,center]{$W$}}
\pgfputat{\pgfxy(1,2.75)}{\pgfbox[center,center]{$U$}}
\pgfputat{\pgfxy(2,2.75)}{\pgfbox[center,center]{$Z'$}}
\pgfputat{\pgfxy(3,2.75)}{\pgfbox[center,center]{$Z$}}
\pgfputat{\pgfxy(1.65,1.8)}{\pgfbox[center,center]{$t_i$}}
\pgfnodecircle{Node0}[fill]{\pgfxy(2.3,3.3)}{0.1cm}
\pgfnodecircle{Node0}[fill]{\pgfxy(4,2.5)}{0.1cm}
\pgfnodecircle{Node0}[fill]{\pgfxy(2.8,1.8)}{0.1cm}
\pgfnodecircle{Node0}[fill]{\pgfxy(3.3,3.3)}{0.1cm}
\pgfnodecircle{Node0}[fill]{\pgfxy(1.5,1.5)}{0.1cm}
\pgfnodecircle{Node0}[fill]{\pgfxy(1.2,3.6)}{0.1cm}
\end{pgftranslate}
\end{pgfmagnify}
\end{pgfpicture}
\caption{Effaceability: given $\{t_1,...,t_r\}$, $W$, and $Z$, there exist $U$ and $Z'$.}
\label{figeffaceability}
\end{figure}

{\bf Effacement Theorems.}  Colliot-Th\'el\`ene, Hoobler, and Kahn \cite{CHKBloch-Ogus-Gabber97} prove several different versions of an {\it effacement theorem} giving conditions under which cohomology theories with supports, or substrata, are effaceable.   These different versions offer a balance between generality and ease of application.  Algebraic $K$-theory and negative cyclic homology satisfy somewhat stronger hypotheses than those required for the most general effacement theorems, so easier versions may be used for these theories.   The effacement conditions I will use for cohomology theories with supports are called {\it \'etale excision} and the {\it cohomological projective bundle condition}.  The effacement conditions I will use for substrata are called the {\it \'etale Mayer-Veitoris condition} and the {\it projective bundle condition for substrata}.\footnotemark\footnotetext{Colliot-Th\'el\`ene, Hoobler, and Kahn \cite{CHKBloch-Ogus-Gabber97} offer {\it five} different conditions on cohomology theories with supports, labeled {\bf COH1} through {\bf COH5}, and five conditions on substrata, labeled {\bf SUB1} through {\bf SUB5}.  Appropriate pairs of these conditions ensure effaceability.  Here, {\bf COH1} and {\bf SUB1} are \'etale excision and the \'etale Mayer-Veitoris condition, respectively, while {\bf COH5} and {\bf SUB5} are the projective bundle conditions for cohomology theories with supports and substrata, respectively.  The conditions {\bf COH2} and {\bf SUB2} are called ``key lemmas" for the corresponding theories.  The most general effacement theorem appearing in the paper \cite{CHKBloch-Ogus-Gabber97}, namely Theorem 5.1.10 on page 28, states that a cohomology theory with supports satisfying {\bf COH1} and {\bf COH2}, or a substratum satisfying {\bf SUB1} and {\bf SUB2}, is effaceable.  Other versions of the effacement theorem appearing in the paper replace {\bf COH2} for cohomology theories with supports, and {\bf SUB@} for substrata, with less general alternative conditions that are easier to apply.   The alternative conditions I will use here are {\bf COH5} and {\bf SUB5}.}


\label{effaceabilityconditionscohom}

{\bf Effaceability for Cohomology Theories with Supports.}  Let $H$ and $H'$ be cohomology theories with supports on the category of pairs $\mbf{P}_k$ over the distinguished category of schemes $\mbf{S}_k$.  Following Colliot-Th\'el\`ene, Hoobler, and Kahn \cite{CHKBloch-Ogus-Gabber97}, I will specify a condition {\bf COH1}, involving the behavior of $H$ with respect to \'etale covers of $X$, and a condition {\bf COH5}, involving the behavior of $H$ and $H'$ with respect to projective bundles over subsets of $X$.   

\begin{defi}\label{defiCOH1} The cohomology theory with supports $H$ satisfies the condition {\bf COH1}, called {\bf \'etale excision}, if $H$ is additive, and if for any diagram of the form shown below, where $f$ is \'etale and $f^{-1}(Z)\rightarrow Z$ is an isomorphism, the induced map $f^*:H_{X\tn{\footnotesize{ on }} Z}^q \rightarrow H_{X'\tn{\footnotesize{ on }} Z}^q$ is an isomorphism for all $q$:

\begin{pgfpicture}{0cm}{0cm}{17cm}{2cm}
\begin{pgftranslate}{\pgfpoint{5.5cm}{-.75cm}}

\pgfputat{\pgfxy(2.75,2.5)}{\pgfbox[center,center]{$X'$}}
\pgfputat{\pgfxy(1.5,1)}{\pgfbox[center,center]{$Z$}}
\pgfputat{\pgfxy(2.75,1)}{\pgfbox[center,center]{$X$}}
\pgfputat{\pgfxy(2.95,1.8)}{\pgfbox[center,center]{$f$}}
\pgfsetendarrow{\pgfarrowlargepointed{3pt}}
\pgfxyline(2.75,2.2)(2.75,1.3)
\pgfxyline(1.8,1)(2.5,1)
\pgfxyline(1.6,1.3)(2.4,2.2)
\end{pgftranslate}

\end{pgfpicture}

\end{defi}

Some preliminary work is required before stating the condition {\bf COH5}.  Let $V$ be an open subset of the $n$-dimensional affine space $\AA_k^n$ over $k$ for some $n$, and let the diagram

\begin{pgfpicture}{0cm}{0cm}{17cm}{2.25cm}
\begin{pgftranslate}{\pgfpoint{5cm}{-.75cm}}
\pgfputat{\pgfxy(1.5,2.5)}{\pgfbox[center,center]{$\AA_V^1$}}
\pgfputat{\pgfxy(3.25,2.5)}{\pgfbox[center,center]{$\PP_V^1$}}
\pgfputat{\pgfxy(5,2.5)}{\pgfbox[center,center]{$V$}}
\pgfputat{\pgfxy(3.25,1)}{\pgfbox[center,center]{$V$}}
\pgfputat{\pgfxy(2.3,2.75)}{\pgfbox[center,center]{$j$}}
\pgfputat{\pgfxy(4.25,2.7)}{\pgfbox[center,center]{$s_\infty$}}
\pgfputat{\pgfxy(2.1,1.65)}{\pgfbox[center,center]{$\pi$}}
\pgfputat{\pgfxy(3.5,1.9)}{\pgfbox[center,center]{$\tilde{\pi}$}}
\pgfputat{\pgfxy(4.45,1.65)}{\pgfbox[center,center]{$=$}}
\pgfsetendarrow{\pgfarrowlargepointed{3pt}}
\pgfxyline(1.85,2.5)(2.85,2.5)
\pgfxyline(4.75,2.5)(3.65,2.5)
\pgfxyline(1.7,2.2)(3,1.2)
\pgfxyline(4.7,2.2)(3.5,1.2)
\pgfxyline(3.25,2.2)(3.25,1.3)
\end{pgftranslate}
\end{pgfpicture}

represent the inclusion of $\AA_k^1$ and the section at infinity into the projective space $\PP_V^1$ over $V$.  Let $F$ be a closed subset of $V$. Let $H$ and $H'$ be cohomology theories with supports on the category $\mbf{P}_k$ of pairs over the distinguished category $\mc{S}_k$ of schemes, with values in a common abelian category $\mbf{A}$, such that for any pair $(X,Z)$ there exists a map
\begin{equation}\label{equprojbundlecohom1}\xymatrixcolsep{2pc}\xymatrix{\tn{Pic}_X\ar[r]&\tn{Hom}_{\mbf{A}}(H_{X\tn{\footnotesize{ on }} Z}^{'*},H_{X\tn{\footnotesize{ on }} Z}^*),}\end{equation}
which is functorial for pairs $(X,Z)$.  Taking $X=\PP_V^1$, $Z=\PP_F^1$, there is a homomorphism
\begin{equation}\label{equprojbundlecohom2}\xymatrixcolsep{4.5pc}\xymatrix{H_{\PP_V^1\tn{\footnotesize{ on }} \PP_F^1}^{'*} \ar[r]^{[\mc{O}(1)]-[\mc{O}]}&H_{\PP_V^1\tn{\footnotesize{ on }} \PP_F^1}^*}.\end{equation}
Composing with $\tilde{\pi}^*$, there is a homomorphism 
\begin{equation}\label{equprojbundlecohom3}\xymatrixcolsep{4pc}\xymatrix{ H_{V\tn{\footnotesize{ on }} F}^{'*}  \ar[r]^{\alpha_{V,F}} &H_{\PP_V^1\tn{\footnotesize{ on }} \PP_F^1}^*},\end{equation}
which is functorial for pairs $(V,F)$.

\begin{defi}\label{defiCOH5} The pair of cohomology theories with supports $(H,H')$ satisfy the condition {\bf COH5}, called the {\bf cohomological projective bundle formula}, if for $V$, $F$, $\tilde{\pi}$ as given above, the natural map
\begin{equation}\label{equprojbundlecohom4}\xymatrixcolsep{4.5pc}\xymatrix{H_{V\tn{\footnotesize{ on }} F}^q\oplus H_{V\tn{\footnotesize{ on }} F}^{'q} \ar[r]^-{\tilde{\pi}^*,\hspace*{.05cm}\alpha_{V,F}} &H_{\PP_V^1\tn{\footnotesize{ on }} \PP_F^1}^q}\end{equation}
is an isomorphism for all $q$.  In particular, if the pair $(H,H)$, satisfies {\bf COH5}, one says that $H$ satisfies {\bf COH5}. 
\end{defi}


\label{effaceabilityconditionssubstrat}

{\bf Effaceability for Substrata.} Let $C$ and $C'$ be substrata on $\mbf{P}_k$.  Again following \cite{CHKBloch-Ogus-Gabber97}, I specify a condition {\bf SUB1}, involving the behavior of $C$ with respect to \'etale covers of $X$, and a condition {\bf SUB5}, involving the behavior of $C$ and $C'$ with respect to projective bundles over subsets of $X$.   

\begin{defi}\label{defiSUB1} The substratum $C$ satisfies the condition {\bf SUB1}, called the {\bf \'etale Mayer-Vietoris condition}, if $C$ is additive, and if and for any diagram of the form shown below on the left, where $f$ is \'etale and $f^{-1}(Z)\rightarrow Z$ is an isomorphism, the commutative square shown below on the right is homotopy cartesian:

\begin{pgfpicture}{0cm}{0cm}{17cm}{2.25cm}
\begin{pgftranslate}{\pgfpoint{2.5cm}{-.75cm}}
\pgfputat{\pgfxy(2.75,2.5)}{\pgfbox[center,center]{$X'$}}
\pgfputat{\pgfxy(1.5,1)}{\pgfbox[center,center]{$Z$}}
\pgfputat{\pgfxy(2.75,1)}{\pgfbox[center,center]{$X$}}
\pgfputat{\pgfxy(2.95,1.8)}{\pgfbox[center,center]{$f$}}
\pgfsetendarrow{\pgfarrowlargepointed{3pt}}
\pgfxyline(2.75,2.2)(2.75,1.3)
\pgfxyline(1.8,1)(2.5,1)
\pgfxyline(1.6,1.3)(2.4,2.2)
\end{pgftranslate}
\begin{pgftranslate}{\pgfpoint{6.5cm}{-.75cm}}
\pgfputat{\pgfxy(1.5,2.5)}{\pgfbox[center,center]{$C_{X'}$}}
\pgfputat{\pgfxy(4,2.5)}{\pgfbox[center,center]{$C_{X'-Z}$}}
\pgfputat{\pgfxy(1.5,1)}{\pgfbox[center,center]{$C_{X}$}}
\pgfputat{\pgfxy(4,1)}{\pgfbox[center,center]{$C_{X-Z}$}}
\pgfputat{\pgfxy(2.55,2.7)}{\pgfbox[center,center]{$v$}}
\pgfputat{\pgfxy(2.55,1.2)}{\pgfbox[center,center]{$u$}}
\pgfsetendarrow{\pgfarrowlargepointed{3pt}}
\pgfxyline(1.5,1.3)(1.5,2.2)
\pgfxyline(4,1.3)(4,2.2)
\pgfxyline(2.1,2.5)(2.95,2.5)
\pgfxyline(2.1,1)(2.95,1)
\end{pgftranslate}
\end{pgfpicture}

\end{defi}

A useful fact cited in \cite{CHKBloch-Ogus-Gabber97}, lemma 5.1.2, page 25, is that the above square is homotopy cartesian if and only if the induced map
\[\xymatrix{C_{X\tn{ \footnotesize{on} } Z} \ar[r]^{f} &C_{X'\tn{ \footnotesize{on} } Z}}\]
is a homotopy equivalence. 

Some preliminary work is required before stating condition {\bf SUB5}. As in the case of cohomology theories with supports, let $V$ be an open subset of $\AA_k^n$ for some $n$, and let the diagram

\begin{pgfpicture}{0cm}{0cm}{17cm}{2.25cm}
\begin{pgftranslate}{\pgfpoint{5cm}{-.75cm}}
\pgfputat{\pgfxy(1.5,2.5)}{\pgfbox[center,center]{$\AA_V^1$}}
\pgfputat{\pgfxy(3.25,2.5)}{\pgfbox[center,center]{$\PP_V^1$}}
\pgfputat{\pgfxy(5,2.5)}{\pgfbox[center,center]{$V$}}
\pgfputat{\pgfxy(3.25,1)}{\pgfbox[center,center]{$V$}}
\pgfputat{\pgfxy(2.3,2.75)}{\pgfbox[center,center]{$j$}}
\pgfputat{\pgfxy(4.25,2.7)}{\pgfbox[center,center]{$s_\infty$}}
\pgfputat{\pgfxy(2.1,1.65)}{\pgfbox[center,center]{$\pi$}}
\pgfputat{\pgfxy(3.5,1.9)}{\pgfbox[center,center]{$\tilde{\pi}$}}
\pgfputat{\pgfxy(4.45,1.65)}{\pgfbox[center,center]{$=$}}
\pgfsetendarrow{\pgfarrowlargepointed{3pt}}
\pgfxyline(1.85,2.5)(2.85,2.5)
\pgfxyline(4.75,2.5)(3.65,2.5)
\pgfxyline(1.7,2.2)(3,1.2)
\pgfxyline(4.7,2.2)(3.5,1.2)
\pgfxyline(3.25,2.2)(3.25,1.3)
\end{pgftranslate}
\end{pgfpicture}

represent the inclusion of $\AA_k^1$ and the section at infinity into $\PP_V^1$.  Let $F$ be a closed subset of $V$. Let $C$ and $C'$ be substrata on the category $\mbf{P}_k$ of pairs over the distinguished category $\mbf{S}_k$ of schemes, with values in an appropriate category $\mbf{E}$ of complexes or spectra over a common abelian category $\mbf{A}$, such that for any $X$ in $\mbf{S}_k$ there exists a map
\begin{equation}\label{equprojbundlesub1}\xymatrixcolsep{2pc}\xymatrix{\tn{Pic}_X\ar[r]&\tn{Hom}_{\mbf{E}}(C_X',C_X),}\end{equation}
which is functorial for $X$.  Taking $X=\PP_V^1$, there exists a map 
\begin{equation}\label{equprojbundlesub2}\xymatrixcolsep{5pc}\xymatrix{ C_{\PP_V^1}' \ar[r]^{[\mc{O}(1)]-[\mc{O}]} &C_{\PP_V^1}}.\end{equation}
Hence, composing with $\tilde{\pi}^*$, there exists a map 
\begin{equation}\label{equprojbundlesub3}\xymatrixcolsep{4pc}\xymatrix{ C_V' \ar[r]^{\alpha_{V}} &C_{\PP_V^1}},\end{equation}
functorial for $V$.  For spectra, these maps are in the {\it stable homotopy category}. 

\newpage

\begin{defi}\label{defiSUB5} The pair of substrata $(C,C')$ satisfy the condition {\bf COH5}, called the {\bf projective bundle formula for substrata}, if for $V$, $\tilde{\pi}$ as given above, the natural maps
\begin{equation}\label{equprojbundlesub3}\xymatrixcolsep{4.5pc}\xymatrix{C_V\oplus C_V' \ar[r]^-{\tilde{\pi}^*,\alpha_{V}} &C_{\PP_V^1}\hspace*{1cm}\tn{for complexes}}\end{equation}
\begin{equation}\label{equprojbundlesub4}\xymatrixcolsep{4.5pc}\xymatrix{C_V\vee C_V' \ar[r]^-{\tilde{\pi}^*,\alpha_{V}} &C_{\PP_V^1}\hspace*{1cm}\tn{for substrata}}\end{equation}
are homotopy equivalences.  In particular, if the pair $(C,C)$, satisfies {\bf SUB5}, one says that $C$ satisfies {\bf SUB5}. 
\end{defi}

\subsection{Effaceability for Bass-Thomason $K$-Theory and \\ Negative Cyclic Homology}\label{subsectioneffaceabilitynonconnectiveK}

The effaceability conditions for Bass-Thomason $K$-theory follow directly from some of the principal results of Thomason's seminal paper \cite{Thomason-Trobaugh90}, particularly {\it Thomason's localization theorem} (\cite{Thomason-Trobaugh90}, theorem 7.4).  The corresponding results for negative cyclic homology are easier, and follow from parenthetical results in Weibel et al.  \cite{WeibelCycliccdh-CohomNegativeK06}. 


{\bf Effaceability for Bass-Thomason $K$-Theory.}   Let $\mathbf{K}$ be the substratum assigning a scheme $X$ to the Bass-Thomason nonconnective $K$-theory spectrum $\mbf{K}_X$. 

The following theorem is part of Thomason's localization theorem:

\begin{theorem} Let $X$ be a quasi-compact and quasi-separated scheme.  Let $Z$ be a closed subspace of $X$ such that $X-Z$ is quasi-compact.  Then there is a  homotopy fiber sequence
\begin{equation}\label{equthomasonlocal1}\mathbf{K}_{X\tn{ on } Z}\rightarrow \mathbf{K}_X\rightarrow \mathbf{K}_{X-Z}.\end{equation}
\end{theorem}
\begin{proof} See \cite{Thomason-Trobaugh90}, Theorem 7.4.
\end{proof}

The following theorem is a preliminary result supporting Thomason's localization theorem:

\begin{theorem} Let $f:X'\rightarrow X$ be a map of quasi-compact and quasi-separated schemes which is \'etale and induces an isomorphism $f^{-1}(Z)\rightarrow Z$.  Let $Z$ be a closed subspace of $X$ such that $X-Z$ is quasi-compact.  Then the map of spectra 
\begin{equation}\label{equthomasonetal2}f^*:\mathbf{K}_{X\tn{ on } Z}\rightarrow \mathbf{K}_{X'\tn{ on } Z}\end{equation}
is a homotopy equivalence. 
\end{theorem}
\begin{proof} See Thomason \cite{Thomason-Trobaugh90}, theorem 7.1. 
\end{proof}

Together, these results imply that the substratum functor $\mathbf{K}$ satisfies  \'etale Mayer-Vietoris condition {\bf SUB1}. 

The following theorem is a special case of another supporting result for Thomason's localization theorem: 

\begin{theorem} Let $V$ be an open subset of $\AA_k^n$ for some $n$.  Let $\tilde{V}$ be the trivial bundle of rank $2$ over $V$, and let $\PP_V^1:=\PP(\tilde{V})$ be the corresponding projective space bundle.  Then there is a homotopy equivalence
\[\mathbf{K}_V\oplus\mathbf{K}_V\rightarrow \mathbf{K}_{\PP_V^1}.\]
\end{theorem}
\begin{proof} See Thomason \cite{Thomason-Trobaugh90} theorem 7.3.
\end{proof}

This result implies that the substratum functor $\mathbf{K}$ satisfies projective bundle condition for substrata {\bf SUB5}. 


\label{subsectioneffaceabilitynegativecyclic}

{\bf Effaceability for Negative Cyclic Homology} The \'etale Mayer-Vietoris condition {\bf SUB1} is called {\it Nisnevich descent} by Weibel et al. \cite{WeibelCycliccdh-CohomNegativeK06}, page 3.  In more detail, following the terminology of \cite{WeibelCycliccdh-CohomNegativeK06}, an {\bf elementary Nisnevich square} is a cartesian square of schemes 
\[\xymatrix{X'\ar[d]_f  &Y'\ar[l]\ar[d] \\
X&Y\ar[l]_i}\]
for which $X\leftarrow Y:i$ is an open embedding, $f:X'\rightarrow X$ is \'etale, and $(X'-Y')\rightarrow(X-Y)$ is an isomorphism.   
A functor $\mbf{E}$ from the category $\mbf{S}_{\tn{fin}/k}$ of schemes essentially of finite type over a field $k$ with values in a suitable category of spectra satisfies {\bf Nisnevich descent},\footnotemark\footnotetext{Weibel et al. remark that the term ``\'etale descent" used in Weibel and Geller is equivalent to Nisnevich descent for presheaves of $\mathbb{Q}$-modules.} as defined in \cite{WeibelCycliccdh-CohomNegativeK06}, if, for any such cartesian square, the square of spectra given by applying $\mc{E}$ is homotopy cartesian.  Letting $Z$ be a closed subset of $X$ and $Y=X-Z$ its complement, there is an open embedding $X\leftarrow Y:i$, which may be completed to a cartesian square.   If $\mbf{E}$ is contravariant, then the Nisnevich descent condition is precisely the \'etale Mayer-Vietoris condition given by Colliot-Th\'el\`ene, Hoobler and Kahn \cite{CHKBloch-Ogus-Gabber97}.  

\begin{lem}\label{lemHNSUB1}The substratum $\mbf{HN}$ satisfies the \'etale Mayer-Vietoris condition {\bf SUB1}.
\end{lem}
\begin{proof}See \cite{WeibelCycliccdh-CohomNegativeK06} theorem 2.9, page 9.
\end{proof}

\begin{lem}\label{lemHNSUB5}The substratum $\mbf{HN}$ satisfies the projective bundle condition for substrata {\bf SUB5}.
\end{lem}
\begin{proof}See \cite{WeibelCycliccdh-CohomNegativeK06} remark 2.11, page 9.
\end{proof}

\chapter{Coniveau Machine}\label{ChapterConiveau}

\section{Introduction}\label{sectionintroductionconiveau}

In this chapter, I construct the coniveau machine for a suitable cohomology theory with supports $H$ on the category of pairs $\mbf{P}$ over an appropriate category of topological spaces $\mbf{S}$.  In the example of principal interest, $H$ is algebraic $K$-theory, and $\mbf{S}$ is a distinguished category $\mbf{S}_k$ of schemes over a field $k$, satisfying the conditions given at the beginning of section \hyperref[subsectioncohomsupportssubstrata]{\ref{subsectioncohomsupportssubstrata}}, and including smooth algebraic varieties as a subcategory.  Some of the steps involved in the construction are carried out only for this particular case, even though many of these steps are more general.  The approach of entertaining a very broad viewpoint, despite focusing almost exclusively on a particular special case, is chosen in response to two facts: first, many aspects of the construction are very general, without any exclusive restriction to smooth algebraic varieties, or even to the field of algebraic geometry; second, I do not yet know how useful or applicable the theory is in other contexts.  


\subsection{Generalities}\label{subsectiongeneralities}

The coniveau machine produces ``short exact sequences of spectral sequences," represented by the rows of the following ``commutative diagram:" 

\begin{pgfpicture}{0cm}{0cm}{17cm}{2.6cm}
\begin{pgftranslate}{\pgfpoint{3.5cm}{-.5cm}}
\begin{pgfmagnify}{1.1}{1.1}
\pgfputat{\pgfxy(1.3,2.2)}{\pgfbox[center,center]{$E_{H_{\tn{\fsz{rel}}},\mbf{S}}$}}
\pgfputat{\pgfxy(4,2.2)}{\pgfbox[center,center]{$E_{H_{\tn{\fsz{aug}}},\mbf{S}}$}}
\pgfputat{\pgfxy(6.8,2.2)}{\pgfbox[center,center]{$E_{H,\mbf{S}}$}}
\pgfputat{\pgfxy(2.5,2.5)}{\pgfbox[center,center]{$i$}}
\pgfputat{\pgfxy(5.5,2.5)}{\pgfbox[center,center]{$j$}}
\pgfsetendarrow{\pgfarrowlargepointed{3pt}}
\pgfxyline(2,2.2)(3.2,2.2)
\pgfxyline(4.8,2.2)(6.2,2.2)
\begin{colormixin}{50!white}
\pgfputat{\pgfxy(1.3,.2)}{\pgfbox[center,center]{$E_{H_{\tn{\fsz{rel}}}^+,\mbf{S}}$}}
\pgfputat{\pgfxy(4,.2)}{\pgfbox[center,center]{$E_{H_{\tn{\fsz{aug}}}^+,\mbf{S}}$}}
\pgfputat{\pgfxy(6.8,.2)}{\pgfbox[center,center]{$E_{H^+,\mbf{S}}$}}
\pgfputat{\pgfxy(2.5,.5)}{\pgfbox[center,center]{$i^+$}}
\pgfputat{\pgfxy(5.5,.5)}{\pgfbox[center,center]{$j^+$}}
\pgfputat{\pgfxy(.8,1.2)}{\pgfbox[center,center]{$\tn{ch}_{\tn{\fsz{rel}}}$}}
\pgfputat{\pgfxy(3.5,1.2)}{\pgfbox[center,center]{$\tn{ch}_{\tn{\fsz{aug}}}$}}
\pgfputat{\pgfxy(6.4,1.2)}{\pgfbox[center,center]{$\tn{ch}$}}
\pgfsetendarrow{\pgfarrowlargepointed{3pt}}
\pgfxyline(2,.2)(3.2,.2)
\pgfxyline(4.8,.2)(6.2,.2)
\pgfxyline(1.3,1.8)(1.3,.5)
\pgfxyline(4,1.8)(4,.5)
\pgfxyline(6.8,1.8)(6.8,.5)
\end{colormixin}
\end{pgfmagnify}
\end{pgftranslate}
\pgfputat{\pgfxy(16,1)}{\pgfbox[center,center]{$(4.1.1.1)$}}
\end{pgfpicture}
\label{equconiveaumachinefunctorch4}

Here, the term $E_{H,\mbf{S}}$ at the top right is the functor assigning to a topological space $X$ belonging to $\mbf{S}$ the {\it coniveau spectral sequence} $\{E_{H,X,r}\}$ associated with an appropriate cohomology theory with supports $H$ on the category of pairs $\mbf{P}$ over $\mbf{S}$.  The functors $E_{H_{\tn{\fsz{aug}}}}$ and $E_{H_{\tn{\fsz{rel}}}}$ assign to $X$ the coniveau spectral sequences associated with ``augmented" and ``relative" versions of $H$, respectively.   The downward arrows, and the bottom row of the diagram in equation \hyperref[equconiveaumachinefunctorch4]{4.1.1.1}, are shaded, to indicate that they exist only under the hypothesis that $H$ possesses an appropriate ``additive version" $H^+$, related to $H$ by a ``logarithmic-type map" labeled $\tn{ch}$.  Under this hypothesis, $H_{\tn{\fsz{aug}}}^+$ and $H_{\tn{\fsz{rel}}}^+$ are ``augmented" and ``relative" versions of the additive theory $H^+$.   

The foregoing discussion is encapsulated in the following definition, with some deliberate vagueness:

\begin{defi}\label{deficoniveauabstract} Let $\mbf{S}$ be a category of topological spaces, and $H$ a cohomology theory with supports on the category of pairs $\mbf{P}$ over $\mbf{S}$, with values in an abelian category $\mbf{A}$.  Let $H_{\tn{\fsz{aug}}}$ be an appropriate ``augmented version" of $H$, and let $H_{\tn{\fsz{rel}}}$ be the corresponding relative theory.  Let $H^+$ be an ``additive version of $H$," if it exists, and let $H_{\tn{\fsz{aug}}}^+$ and $H_{\tn{\fsz{rel}}}^+$ be the corresponding augmented and relative additive theories.   Then the {\bf coniveau machine} for $H$ on $\mbf{S}$ with respect to the augmentation $H\mapsto H_{\tn{\fsz{aug}}}$, if it exists, is the commutative diagram of functors and natural transformations with ``exact rows," depicted in equation \hyperref[equconiveaumachinefunctorch4]{4.1.0.1} above, where each functor assigns to a space $X$ in $\mbf{S}$ the coniveau spectral sequence for the appropriate cohomology theory with supports. 
\end {defi}

This preliminary definition leaves many unanswered questions, including the following:
\begin{enumerate}
\item Given a cohomology theory with supports $H$ on the category of pairs over $\mbf{S}$, how is the ``augmented theory" $H_{\tn{\fsz{aug}}}$ defined?   In the example of principal interest, $\mbf{S}$ is a category of schemes $\mbf{S}_k$ over a field $k$ satisfying the conditions given at the beginning of section \hyperref[subsectioncohomsupportssubstrata]{\ref{subsectioncohomsupportssubstrata}}, $H$ is algebraic $K$-theory, and the augmented theory $H_{\tn{\fsz{aug}}}$ is defined by ``multiplying by a fixed separated scheme $Y$," as explained in section \hyperref[sectioncohomtheoriessupports]{\ref{sectioncohomtheoriessupports}} above.  Note that to guarantee that the relative Chern character $\tn{ch}_{\tn{\fsz{rel}}}$ induces an isomorphism of spectral sequences in this context, one must assume the further condition that the augmentation is given by nilpotent thickening; i.e., by taking $Y$ to be the prime spectrum of a $k$-algebra generated over $k$ by nilpotent elements.   This suggests that a minimal requirement for the ``entire construction to work" in general is that ``some type of underlying ring structure must be present."   Another way to express this is by making the observation that the most of the cohomology theories with supports to which one can imagine applying the construction are ``mediated by ring structure," in the sense that the corresponding functors are defined in terms of rings of functions on a topological space.  This is true, for example, of the list of cohomology theories with supports given by Colliot-Th\'el\`ene, Hoobler, and Kahn \cite{CHKBloch-Ogus-Gabber97}, sections 7 and 8. 

In some contexts, one may choose to ``ascribe the augmentation to objects of $\mbf{S}$ rather than to $H$."  For example, when ``multiplying by a fixed separated scheme $Y$," sending each scheme $X$ in $\mbf{S}$ to $X\times_kY$, one may choose to think of applying the same cohomology theory $H$ to two different geometric objects $X$ and $X\times_kY$, rather than augmenting $H$.  While this may be a ``better" point of view from the perspective of trying to extend results about particular types of varieties or schemes to a more general context, it is unnatural in the cases of principal interest in this book, in which only nilpotent thickenings are involved.  Not only are the underlying topological spaces the same, in this context, in the absolute, augmented, and relative cases, but the entire method of analysis is based upon a purely topological construction; namely, filtration by codimension.  It is most natural in this context to consider ``different cohomology theories on the same object." 

\item How is the relative theory $H_{\tn{\fsz{rel}}}$ defined?  In this book, I focus mostly on relative $K$-theory and relative cyclic homology for nilpotent thickenings of a smooth algebraic variety $X$ over a field $k$, which reduce to split nilpotent extensions at the level of $k$-algebras. In this case, relative cohomology theories admit easy description at the group level in terms of kernels, as described in definition \hyperref[subsectionnilpotent]{\ref{subsectionnilpotent}}.  More generally, cohomology theories involve intermediary constructs such as complexes, bicomplexes, spectra, spectral sequences, etc., and relative constructions must generally be performed at this level to ensure good functorial properties.  

Of course, the procedure of ``multiplying by a fixed separated scheme $Y$" to define the augmented theory $H_{\tn{\fsz{aug}}}$ is much more general than nilpotent thickening, and the resulting relative theory is generally {\it not} defined in terms of kernels at the group level.  In particular, if $H$ is algebraic $K$-theory, defined in terms of spectrum-valued functors $\mbf{K}_{X\tn{ \fsz{on} } Z}$ of Bass and Thomason, then the augmented theory is defined via the spectrum $\mbf{K}_{X\times_kY}$, and the relative theory is defined via the spectrum $\mbf{K}_{X\times_kY,Y}$ which is by definition the homotopy fiber of the morphism of spectra $\mbf{K}_{X\times_kY}\rightarrow\mbf{K}_{X}$. 

\item What does ``additive version of $H$" mean in general, and what is the map $\tn{ch}$?  When $H$ is algebraic $K$-theory, additive version of $H$ is negative cyclic homology, and the map $\tn{ch}$ is the algebraic Chern character, which may be viewed as a generalized logarithmic map (or derivative thereof), sending ``multiplicative" structure in $K$-theory to ``additive structure" in negative cyclic homology.  More generally, exponential and logarithmic-type maps are ubiquitous in mathematics; e.g., in Lie theory, and this ubiquity is due to the very general defining characteristics of such maps.  Hence, while the term ``additive" is best viewed as an analogy in the general context, it may signify much more for many choices of $\mbf{S}$ and $H$ besides those examined here. 

\item What is the functorial nature of $E_{H,\mbf{S}}$, $E_{H_{\tn{\fsz{aug}}},\mbf{S}}$, etc., in general?  This is somewhat involved.  Returning to the case where $\mbf{S}$ is a category of schemes $\mbf{S}_k$ over a field $k$, and $H$ is algebraic $K$-theory, a morphism of schemes $X\leftarrow X'$ in $\mbf{S}$ leads to a functorial morphism of coniveau spectral sequences $\{E_{H,X,r}\}\rightarrow \{E_{H,X',r}\}$.  Since this is merely a consequence of the fact that $K$ is a contravariant functor, similar statements apply to other cohomology theories for which the coniveau spectral sequence is defined.  At a higher level, one may also consider natural transformations $H\rightarrow H'$ of cohomology theories with supports, and study the existence and properties of induced transformations $E_{H,\mbf{S}}\rightarrow E_{H',\mbf{S}}$ in the covariant direction.  The transformations $i$, $j$, and $\tn{ch}$ (if it exists) are examples of such covariant transformations.  Of course, the prototypical example of $\tn{ch}$ is the algebraic Chern character from algebraic $K$-theory to negative cyclic homology. 
\end{enumerate}


\subsection{Main Focus: Algebraic $K$-Theory for Nilpotent Thickenings of Smooth Algebraic Varieties}

As already described in the introduction, analyzing the infinitesimal structure of the Chow groups $\tn{Ch}_X^p$ of a smooth algebraic variety $X$ over a field $k$ involves nilpotent thickenings $X\mapsto X\times_kY$, which may be described locally in terms of split nilpotent extensions of $k$-algebras.   The ``nilpotent variables" added via these extensions supply ``infinitesimal degrees of freedom," or ``directions of motion" for algebraic cycles on $X$.   Under these circumstances, the natural transformations $i$ and $j$ in equation \hyperref[equconiveaumachinefunctorintro]{4.1.1.1} split, and the relative Chern character $\tn{ch}_{\tn{\fsz{rel}}}$ is an isomorphism of functors.  The generalized tangent group of the Chow groups $\tn{Ch}_X^p$ with respect to $Y$ is the $p$th Zariski sheaf cohomology group of the relative sheaf $\ms{K}_{p,X\times_kY,Y}$, as described in section \hyperref[subsectiongeneralizeddeftan]{\ref{subsectiongeneralizeddeftan}} above.  This cohomology group may be computed via a flasque resolution of the sheaf $\ms{K}_{p,X\times_kY,Y}$, which arises by sheafifying the $-p$th row of the coniveau spectral sequence for relative $K$-theory on $X$, as described in section \hyperref[sectionconiveaumachine]{\ref{sectionconiveaumachine}} below. 

Since the relative Chern character $\tn{ch}_{\tn{\fsz{rel}}}$ is an isomorphism of functors in this context, it suffices to compute the $p$th Zariski sheaf cohomology group of the relative negative cyclic homology sheaf $\ms{HN}_{p,X\times_kY,Y}$ on $X$.  This cohomology group, in turn, may be computed via the sheafified $-p$th row of the coniveau spectral sequence for $\tn{HN}_{\tn{\fsz{rel}}}$.  The absolute and augmented versions of negative cyclic homology are ``irrelevant" in this context, as are the remaining rows of the coniveau spectral sequences.  Hence, one may simplify the picture by isolating and sheafifying the $-p$th rows of the coniveau spectral sequences for absolute, augmented, and relative $K$-theory, and for relative negative cyclic homology.   It is convenient to transpose these rows to form four columns, which are the Cousin flasque resolutions of the sheaves $\ms{K}_{p,X}$, $\ms{K}_{p,X\times_kY}$, $\ms{K}_{p,X\times_kY,Y}$, and $\ms{HN}_{p,X\times_kY,Y}$, respectively.  The first and second columns are connected by the map of complexes arising from the splitting of the natural transformation $j$ in equation \hyperref[equconiveaumachinefunctorintro]{4.1.1.1}, and the second and third columns are connected by the map of complexes arising from the splitting of the natural transformation $i$.   The third and fourth columns are connected by the map of complexes induced by the relative algebraic Chern character.  Putting all this together yields the schematic diagram of complexes shown in figure \hyperref[figsimplifiedfourcolumnKtheorych4]{\ref{figsimplifiedfourcolumnKtheorych4}} below, arranged so that the arrows go from left to right:

\begin{figure}[H]
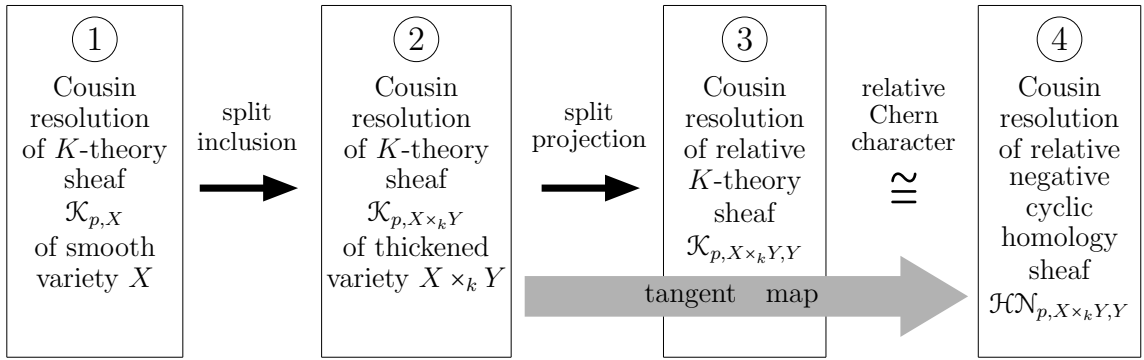

\begin{pgfpicture}{0cm}{0cm}{17cm}{4.8cm}
\begin{pgfmagnify}{.9}{.9}
\begin{pgftranslate}{\pgfpoint{.2cm}{-1.7cm}}
\begin{pgftranslate}{\pgfpoint{-.5cm}{0cm}}
\pgfxyline(.9,1.8)(.9,7)
\pgfxyline(.9,7)(3.5,7)
\pgfxyline(3.5,7)(3.5,1.8)
\pgfxyline(3.5,1.8)(.9,1.8)
\pgfputat{\pgfxy(2.2,5.8)}{\pgfbox[center,center]{Cousin}}
\pgfputat{\pgfxy(2.2,5.35)}{\pgfbox[center,center]{resolution }}
\pgfputat{\pgfxy(2.2,4.85)}{\pgfbox[center,center]{of $K$-theory}}
\pgfputat{\pgfxy(2.2,4.45)}{\pgfbox[center,center]{sheaf}}
\pgfputat{\pgfxy(2.2,3.9)}{\pgfbox[center,center]{$\ms{K}_{p,X}$}}
\pgfputat{\pgfxy(2.2,3.45)}{\pgfbox[center,center]{of smooth}}
\pgfputat{\pgfxy(2.2,2.95)}{\pgfbox[center,center]{variety $X$}}
\pgfnodecircle{Node0}[stroke]{\pgfxy(2.2,6.5)}{0.33cm}
\pgfputat{\pgfxy(2.2,6.5)}{\pgfbox[center,center]{\large{$1$}}}
\end{pgftranslate}
\begin{pgftranslate}{\pgfpoint{4.05cm}{0cm}}
\pgfxyline(1,1.8)(1,7)
\pgfxyline(1,7)(3.8,7)
\pgfxyline(3.8,7)(3.8,1.8)
\pgfxyline(3.8,1.8)(1,1.8)
\pgfputat{\pgfxy(2.4,5.8)}{\pgfbox[center,center]{Cousin}}
\pgfputat{\pgfxy(2.4,5.35)}{\pgfbox[center,center]{resolution }}
\pgfputat{\pgfxy(2.4,4.85)}{\pgfbox[center,center]{of $K$-theory}}
\pgfputat{\pgfxy(2.4,4.45)}{\pgfbox[center,center]{sheaf}}
\pgfputat{\pgfxy(2.4,3.9)}{\pgfbox[center,center]{$\ms{K}_{p,X\times_kY}$}}
\pgfputat{\pgfxy(2.4,3.45)}{\pgfbox[center,center]{of thickened}}
\pgfputat{\pgfxy(2.4,2.95)}{\pgfbox[center,center]{variety $X\times_kY$}}
 \pgfnodecircle{Node0}[stroke]{\pgfxy(2.4,6.5)}{0.33cm}
\pgfputat{\pgfxy(2.4,6.5)}{\pgfbox[center,center]{\large{$2$}}}
\pgfsetendarrow{\pgfarrowtriangle{6pt}}
\pgfsetlinewidth{3pt}
\pgfxyline(-.8,4.3)(.4,4.3)
\pgfputat{\pgfxy(-.1,5.4)}{\pgfbox[center,center]{\small{split}}}
\pgfputat{\pgfxy(-.1,5)}{\pgfbox[center,center]{\small{inclusion}}}
\end{pgftranslate}
\begin{pgftranslate}{\pgfpoint{9.1cm}{0cm}}
\pgfxyline(1,1.8)(1,7)
\pgfxyline(1,7)(3.4,7)
\pgfxyline(3.4,7)(3.4,1.8)
\pgfxyline(3.4,1.8)(1,1.8)
\pgfputat{\pgfxy(2.2,5.8)}{\pgfbox[center,center]{Cousin}}
\pgfputat{\pgfxy(2.2,5.35)}{\pgfbox[center,center]{resolution}}
\pgfputat{\pgfxy(2.2,4.9)}{\pgfbox[center,center]{of relative}}
\pgfputat{\pgfxy(2.2,4.4)}{\pgfbox[center,center]{$K$-theory}}
\pgfputat{\pgfxy(2.2,3.95)}{\pgfbox[center,center]{sheaf}}
\pgfputat{\pgfxy(2.2,3.4)}{\pgfbox[center,center]{$\ms{K}_{p,X\times_kY,Y}$}}
\pgfnodecircle{Node0}[stroke]{\pgfxy(2.2,6.5)}{0.33cm}
\pgfputat{\pgfxy(2.2,6.5)}{\pgfbox[center,center]{\large{$3$}}}
\pgfsetendarrow{\pgfarrowtriangle{6pt}}
\pgfsetlinewidth{3pt}
\pgfxyline(-.8,4.3)(.4,4.3)
\pgfputat{\pgfxy(-.1,5.4)}{\pgfbox[center,center]{\small{split}}}
\pgfputat{\pgfxy(-.1,5)}{\pgfbox[center,center]{\small{projection}}}
\end{pgftranslate}
\begin{pgftranslate}{\pgfpoint{13.75cm}{0cm}}
\pgfxyline(1,1.8)(1,7)
\pgfxyline(1,7)(3.4,7)
\pgfxyline(3.4,7)(3.4,1.8)
\pgfxyline(3.4,1.8)(1,1.8)
\pgfputat{\pgfxy(2.2,5.8)}{\pgfbox[center,center]{Cousin}}
\pgfputat{\pgfxy(2.2,5.35)}{\pgfbox[center,center]{resolution}}
\pgfputat{\pgfxy(2.2,4.9)}{\pgfbox[center,center]{of relative}}
\pgfputat{\pgfxy(2.2,4.45)}{\pgfbox[center,center]{negative}}
\pgfputat{\pgfxy(2.2,4)}{\pgfbox[center,center]{cyclic}}
\pgfputat{\pgfxy(2.2,3.55)}{\pgfbox[center,center]{homology}}
\pgfputat{\pgfxy(2.2,3.1)}{\pgfbox[center,center]{sheaf}}
\pgfputat{\pgfxy(2.2,2.55)}{\pgfbox[center,center]{$\ms{HN}_{p,X\times_kY,Y}$}}
\pgfputat{\pgfxy(-.1,5.8)}{\pgfbox[center,center]{\small{relative}}}
\pgfputat{\pgfxy(-.1,5.4)}{\pgfbox[center,center]{\small{Chern}}}
\pgfputat{\pgfxy(-.1,5)}{\pgfbox[center,center]{\small{character}}}
\pgfputat{\pgfxy(-.1,4.3)}{\pgfbox[center,center]{\huge{$\cong$}}}
\pgfnodecircle{Node0}[stroke]{\pgfxy(2.2,6.5)}{0.33cm}
\pgfputat{\pgfxy(2.2,6.5)}{\pgfbox[center,center]{\large{$4$}}}
\end{pgftranslate}
\begin{colormixin}{30!white}
\color{black}
\pgfmoveto{\pgfxy(8.05,3)}
\pgflineto{\pgfxy(13.6,3)}
\pgflineto{\pgfxy(13.6,3.3)}
\pgflineto{\pgfxy(14.6,2.7)}
\pgflineto{\pgfxy(13.6,2.1)}
\pgflineto{\pgfxy(13.6,2.4)}
\pgflineto{\pgfxy(8.05,2.4)}
\pgflineto{\pgfxy(8.05,3)}
\pgffill
\end{colormixin}
\pgfputat{\pgfxy(10.5,2.7)}{\pgfbox[center,center]{tangent}}
\pgfputat{\pgfxy(12,2.67)}{\pgfbox[center,center]{map}}
\end{pgftranslate}
\end{pgfmagnify}
\end{pgfpicture}
\caption{Simplified ``four column version" of the coniveau machine for algebraic $K$-theory on a smooth algebraic variety in the case of a nilpotent thickening.}
\label{figsimplifiedfourcolumnKtheorych4}
\end{figure}

For the remainder of this chapter, I will focus mostly on this ``simplified four-column version" of the coniveau machine.  For example, when speaking of the ``first column of the machine," I mean the Cousin resolution of $\ms{K}_{p,X}$, {\it not} the column involving the generalized relative Chern character in equation  \hyperref[equconiveaumachinefunctorintro]{4.1.1.1}.

\section{Coniveau Spectral Sequence}\label{sectionconiveauSS}

Relatively early in the development of modern algebraic geometry, Hartshorne and Grothendieck \cite{HartshorneResiduesDuality66} introduced a very general method for constructing a spectral sequence associated with a complex of sheaves on a filtered topological space.  In the case of principal interest for this book, the topological space is the Zariski space $\tn{Zar}_X$ of an algebraic scheme $X$, filtered by codimension, and the complex of sheaves is defined in terms of a cohomology theory with supports $H$.   In this case, the resulting spectral sequence is called the {\it coniveau spectral sequence} for $H$ on $X$.  The functors appearing in the abstract form of the coniveau machine in definition \hyperref[deficoniveauabstract]{\ref{deficoniveauabstract}} assign such coniveau spectral sequences to objects of a distinguished category of schemes $\mbf{S}$.

Even for experienced practitioners, the use of spectral sequences can present some annoying features, due to the large amount of bookkeeping involved and the presence of certain inconsistencies in the literature.  For this reason, I have included detailed material on spectral sequences in section \hyperref[spectralsequences]{\ref{spectralsequences}} of the appendix.  This material serves both to fix notation and to make the book more self-contained.  For the remainder of this section, I will use the material in \hyperref[spectralsequences]{\ref{spectralsequences}} as needed.  Of particular importance is the general construction of the spectral sequence associated with an {\it exact couple,} 


\subsection{Filtration by Codimension}\label{subsectionfiltrationbycodim}

In equation \hyperref[equdefconiveau]{\ref{equdefconiveau}} of section \hyperref[subsectiondimension]{\ref{subsectiondimension}} above, I introduced the coniveau filtration on an $n$-dimensional noetherian scheme $X$ as a decreasing filtration   
\begin{equation}\label{coniveaufiltrchapter4} \oslash\subset \tn{Zar}_X^{\ge n}\subset \tn{Zar}_X^{\ge n-1}\subset...\subset \tn{Zar}_X^{\ge 1}\subset \tn{Zar}_X^{\ge 0}=\tn{Zar}_X,\end{equation}
where $\tn{Zar}_X^{\ge p}$ is defined to be the set of all point of $X$ of codimension at least $p$.  The $p$th associated graded piece of the coniveau filtration may be identified with the set $\tn{Zar}_X^p$ of all points of codimension $p$ in $X$.  

It is sometimes useful to consider the more general notion of filtration with respect to an abstract codimension function.  Let $T$ be a topological space, and let $\tn{Irred}_T$ be the set of irreducible closed subsets of $T$, partially ordered by strict inclusion.  

\begin{defi}\label{defidimensionfunction} An abstract {\bf codimension function} is a contravariant order morphism $\tn{codim}_T:\tn{Irred}_T\rightarrow L$, where $L$ is a totally ordered set.
\end{defi}

Usually $L$ is taken to be the nonnegative integers, sometimes augmented by an absorbing element $\infty$ to accommodate the empty set.\footnotemark\footnotetext{There exist many important examples of noninteger dimension and codimension in other areas of mathematics, though these often involve additional structures such as a metric.  Familiar examples include Hausdorff dimension and Minkowski (i.e. ``box-counting") dimension.}   The meaning of a contravariant order morphism is that for any strict inclusion of irreducible closed subsets $U\subset V$ of $T$, $\tn{codim}_X U>\tn{codim}_X V$ in $L$.  

The filtration of the topological space $T$ with respect to an abstract codimension function $\tn{codim}_T$ is the family of subsets $\{T^{\ge l}\}_{l\in L}$,  where $T^{\ge l}:=\tn{codim}_T^{-1}(\{l'\in L|l'\ge l\})\subset \tn{Irred}_T$.  The family $\{T^{\ge l}\}$ is totally ordered by inclusion.  The map from $L$ to the underlying ordered set of $\{T^{\ge l}\}$ sending $l$ to the element corresponding to $T^{\le l}$ is a contravariant order isomorphism.  In this sense, the filtration of $T$ with respect to $\tn{codim}_T$ is a {\it contravariant categorification} of $L$, sending elements to subsets of $\tn{Irred}_T$ and relations to inclusions in the opposite order.   In the case where $L$ is the augmented integers $\ZZ\cup\{\infty\}$, and where $\tn{codim}_T(T)$ is taken to be zero, the filtration of $T$ with respect to $\tn{codim}_T$ may be written as a sequence
\begin{equation}\label{codimsequencech4}\oslash\subset...\subset T^{\ge l}\subset T^{\ge l-1}\subset...\subset T^{\ge 1}\subset T^{\ge 0}=T.\end{equation}
$T$ is called {\bf finite-dimensional} with respect to $\tn{codim}_T$ when this sequence is bounded below.  This places the coniveau filtration on an $n$-dimensional noetherian scheme $X$ in a wider context.  

Filtration with respect to a codimension function is not the only useful way to filter a topological space $T$ by means of its closed subsets.  In particular, given additional structure on $T$, one may design a filtration to organize this information in a convenient way.  An important example of additional structure in this sense is a sheaf, or complex of sheaves, on $T$, with values in an abelian category.  Hartshorne discusses a number of different sheaf-related filtrations on topological spaces in \cite{HartshorneResiduesDuality66}, which is based on ideas of Grothendieck concerning duality of coherent sheaves.  As Balmer observes in \cite{BalmerNiveauSS00}, page 2, the construction of the coniveau spectral sequence does not depend on the details of Zariski codimension in algebraic geometry, but applies to any suitable codimension function.  Balmer calls his corresponding functions dimension functions, and his spectral sequence the {\bf niveau spectral sequence.}  In the terminology of this book, an abstract {\bf dimension function} is a {\it covariant} order morphism $\tn{dim}_X:\tn{Irred}_T\rightarrow L$.\footnotemark\footnotetext{Balmer makes everything covariant by inverting the totally ordered target set, which in his case is the augmented integers.  Hence, his ``dimension function" measuring the ``codimensions" of irreducible closed subsets of schemes is the {\it negative} of the Krull codimension function.  I prefer to maintain the order-theoretic distinction between dimension and codimension, and to adhere to the cohomological/contravariant/codimensional viewpoint familiar in the study of cycle and Chow groups.}

\subsection{Coniveau Spectral Sequence}\label{coniveauspectralsequence}

The material in this section is standard, but much more specialized and targeted to the subject of this book than the supporting material in section \hyperref[spectralsequences]{\ref{spectralsequences}} of the appendix.   The main references are Hartshorne \cite{HartshorneResiduesDuality66} Chapter IV, and Colliot-Th\'el\`ene, Hoobler, and Kahn \cite{CHKBloch-Ogus-Gabber97}.  The notation, as usual, is mine, and differs somewhat from that in the references.  For much of this section, I will focus on the case where the topological space $X$ under consideration is an equidimensional and noetherian scheme over a field $k$.

Let $\mbf{S}_k$ be a distinguished category of schemes over a field $k$, and let $X$ be an object of $\mbf{S}_k$.  Let $H:=\{H^n\}_{n\in\ZZ}$ be a cohomology theory with supports on the category of pairs over $\mbf{S}_k$.  Consider a chain $\overline{Z}$ of closed subsets of $X$ of the form $\oslash\subset Z^d\subset Z^{d-1}\subset...\subset Z^0=X$, where the inclusions are proper.  By the definition of a cohomology theory with supports, every inclusion $Z^{p+1}\subset Z^p$ in the chain $\overline{Z}$ induces a long exact cohomology sequence of cohomology groups, of the following form:
\begin{equation}\label{leschainofsubsets}...\longrightarrow H^{p+q}_{X\tn{ \footnotesize{on} } Z^{p+1}}\overset{i_{\overline{Z}}^{p+1,q-1}}{\longrightarrow} H^{p+q}_{X\tn{ \footnotesize{on} } Z^{p}}\overset{j_{\overline{Z}}^{p,q}}{\longrightarrow} H^{p+q}_{X-Z^{p+1}\tn{ \footnotesize{on} } Z^{p}-Z^{p+1}}\overset{k_{\overline{Z}}^{p,q}}{\longrightarrow} H^{p+q+1}_{X\tn{ \footnotesize{on} } Z^{p+1}}\longrightarrow...\end{equation}
These sequences are of the same form as the long exact sequences 
\[...\longrightarrow D^{p+1,q-1}\overset{i^{p+1,q-1}}{\longrightarrow} D^{p,q}\overset{j^{p,q}}{\longrightarrow} E^{p,q}\overset{k^{p,q}}{\longrightarrow} D^{p+1,q}\longrightarrow...\]
associated with an exact couple appearing in equation \hyperref[equlesexactcouple]{\ref{equlesexactcouple}} of section \hyperref[spectralsequences]{\ref{spectralsequences}} of the appendix.  These sequences may be used to {\it define} an exact couple $C_{\overline{Z}}:=(D_{\overline{Z}},E_{\overline{Z}},i_{\overline{Z}},j_{\overline{Z}},k_{\overline{Z}})$ of bigraded $k$-vector spaces, and from this a spectral sequence $\{E_{\overline{Z},r}\}:=\{E_{\overline{Z},r}^{p,q},d_{\overline{Z},r}^{p,q}\}$.   In particular, $D_{\overline{Z}}^{p,q}$ is defined to be the cohomology group $H^{p+q}_{X\tn{ \footnotesize{on} } Z^p}$, and $E_{\overline{Z}}^{p,q}$ is defined to be the cohomology group $H^{p+q}_{X-Z^{p+1}\tn{ \footnotesize{on} } Z^{p}-Z^{p+1}}$.  The maps $i_{\overline{Z}},j_{\overline{Z}},$ and $k_{\overline{Z}}$ are defined to be the maps whose graded pieces appear in the long exact cohomology sequence above.  The exact couple $C_{\overline{Z}}$ yields a spectral sequence $\{E_{\overline{Z},r}\}$, whose $E_{\overline{Z},1}$-terms are $E_{\overline{Z}}^{p,q}$.  This spectral sequences converges to $\{H_X^n\}$ with respect to the filtration $F^pH_X^n=\tn{Im}\big(D_{\overline{Z}}^{p,n-p}\rightarrow H_X^n\big)$, where the map  $D_{\overline{Z}}^{p,n-p}\rightarrow H_X^n$ is the composition
\begin{equation}\label{equCSSequ2}D_{\overline{Z}}^{p,n-p}=H^{n}_{X\tn{ \footnotesize{on} } Z^p}\overset{i_{\overline{Z}}^{p,q}}{\longrightarrow} H^{n}_{X\tn{ \footnotesize{on} } Z^{p-1}}\overset{i_{\overline{Z}}^{p-1,q+1}}{\longrightarrow}... \overset{i_{\overline{Z}}^{2,n-2}}{\longrightarrow}H^{n}_{X\tn{ \footnotesize{on} } Z^{1}}\overset{i_{\overline{Z}}^{1,n-1}}{\longrightarrow} H^{n}_{X}.\end{equation}
Convergence follows from the fact that $C_{\overline{Z}}$ is bounded above; indeed, for $p>d$, the set $Z^p$ is empty, so the cohomology group $D_{\overline{Z}}^{p,n-p}=H^{n}_{X\tn{ \footnotesize{on} } Z^p}$ is zero in this case.  

Now assume that $X$ is equidimensional and noetherian of dimension $d$, and that for all $p$, the codimension of the closed subset $Z^p$ in $X$ is at least $p$.  The family of all chains of closed subsets $\overline{Z}$ of the form $\oslash\subset Z^d\subset Z^{d-1}\subset...\subset Z^0=X$ is partially ordered by the relation $\prec$ defined by setting $\overline{Z}\prec \overline{Z}'$ if and only if $Z^p\subseteq  {Z^p}'$ for all $p$.   The relations $\overline{Z}\prec\overline{Z}'$ may be viewed as the morphisms of an appropriate category of chains of closed subsets.  The construction of the exact couple is covariant with respect to these relations, meaning that the construction is a functor taking each morphism $\overline{Z}\prec\overline{Z}'$ to a morphism of exact couples $C_{\overline{Z}}\rightarrow C_{\overline{Z}'}$.  Hence, there is an induced partial order on the family of exact couples $\{C_{\overline{Z}}\}$.  Passing to the limit in this partially ordered family, one may define a define a new exact couple $C:=(D,E,i,j,k)$, where 
\begin{equation}\label{equlimitDE}D^{p,q}:=\lim_{\{\overline{Z}\}} H^{p+q}_{X\tn{ \footnotesize{on} } Z^p}\tn{\hspace*{.5cm} and \hspace*{.5cm}} E^{p,q}:=\lim_{\{\overline{Z}\}} H^{p+q}_{X-Z^{p+1}\tn{ \footnotesize{on} } Z^{p}-Z^{p+1}}.\end{equation}

\begin{defi}\label{deficoniveauSS} The spectral sequence $\{E_r^{p,q},d_r^{p,q}\}$ associated to the exact couple $C=(D,E,i,j,k)$ defined above is called the {\bf coniveau spectral sequence} for  $H$ on $X$.
\end{defi}

When it is important to emphasize the cohomology theory $H$ and the space $X$, the coniveau spectral sequence may be denoted by $\{E_{H,X,r}\}$, and its terms by $E_{H,X,r}^{p,q}$.   Usually, $H$ and $X$ are clear from context, and the coniveau spectral sequence may be denoted simply by $\{E_{r}\}$ and its terms by $E_{r}^{p,q}$. 


\subsection{Coniveau Spectral Sequence via Coniveau Filtration}\label{coniveauspectralsequence}

Let $\mbf{S}_k$ be a distinguished category of schemes over a field $k$, and let $X$ be an object of $\mbf{S}_k$.  Assume that $X$ is equidimensional and noetherian of dimension $d$.  Let $H:=\{H^n\}_{n\in\ZZ}$ be a cohomology theory with supports on the category of pairs over $\mbf{S}_k$.  For a point $x$ in $\tn{Zar}_X^p$, define the {\bf cohomology group} $H_{X\tn{ \footnotesize{on} } x}^{p+q}$ {\bf supported at $x$} to be the limit group $\displaystyle \lim_{\{U|x\in U\}} H_{U\tn{ \footnotesize{on} } \overline{x}\cap U}^{p+q}$ over all open sets $U$ containing $x$.  The following lemma, adapted from Colliot-Th\'{e}l\`{e}ne, Hoobler, and Kahn \cite{CHKBloch-Ogus-Gabber97}, Lemma 1.2.1, page 6, is crucial for expressing the coniveau spectral sequence for  $H=\{H^n\}_{n\in\ZZ}$ on $X$ in terms of the coniveau filtration of $X$.  

\begin{lem}\label{lemCHKlemma121}The $E_1$-terms $E_{X,1}^{p,q}$ of the coniveau spectral sequence for $H$ on $X$ are of the form 
\begin{equation}\label{equE1terms}E_{X,1}^{p,q}\cong\coprod_{x\in \tn{\footnotesize{Zar}}_X^{p}} H_{X\tn{ \footnotesize{on} } x}^{p+q}.\end{equation}
\end{lem}
\begin{proof}  The proof can be found in Colliot-Th\'{e}l\`{e}ne, Hoobler, and Kahn \cite{CHKBloch-Ogus-Gabber97}, but the version I give here gives more details and is more consistent with the choice of notation and terminology in this book.   The first step is to prove the following intermediate statement: {\it if $T_1,...,T_r$ are pairwise disjoint closed subsets of $X$, then for any $p\ge 0$,}
\begin{equation}\label{equlem4231intermediate}\bigoplus_i H_{X\tn{ \footnotesize{on} } T_i}^p\cong H_{X\tn{ \footnotesize{on} } \cup T_i}^p.\end{equation}
To prove this, first assume, by induction, that $r=2$.  Now consider the following diagram:

\begin{pgfpicture}{0cm}{0cm}{17cm}{3.5cm}
\begin{pgftranslate}{\pgfpoint{0cm}{-.25cm}}
\pgfputat{\pgfxy(8.5,.5)}{\pgfbox[center,center]{$H_{X-T_2 \tn{ \footnotesize{on} }T_1}^p$}}
\pgfputat{\pgfxy(5,2)}{\pgfbox[center,center]{$H_{X \tn{ \footnotesize{on} }T_1}^p$}}
\pgfputat{\pgfxy(8.5,2)}{\pgfbox[center,center]{$H_{X \tn{ \footnotesize{on} }T_1\cup T_2}^p$}}
\pgfputat{\pgfxy(12,2)}{\pgfbox[center,center]{$H_{X-T_1 \tn{ \footnotesize{on} }T_2}^p$}}
\pgfputat{\pgfxy(8.5,3.5)}{\pgfbox[center,center]{$H_{X \tn{ \footnotesize{on} }T_2}^p$}}
\pgfsetendarrow{\pgfarrowlargepointed{3pt}}
\pgfxyline(8.5,3.1)(8.5,2.3)
\pgfxyline(8.5,1.6)(8.5,.8)
\pgfxyline(5.9,2)(7.2,2)
\pgfxyline(9.7,2)(10.8,2)
\pgfxyline(5.4,1.5)(7.3,.7)
\pgfxyline(9.4,3.3)(11.3,2.5)
\pgfputat{\pgfxy(6.6,2.3)}{\pgfbox[center,center]{\small{$i$}}}
\pgfputat{\pgfxy(10.2,2.3)}{\pgfbox[center,center]{\small{$j$}}}
\pgfputat{\pgfxy(8.25,2.75)}{\pgfbox[center,center]{\small{$i'$}}}
\pgfputat{\pgfxy(8.75,1.25)}{\pgfbox[center,center]{\small{$j'$}}}
\pgfputat{\pgfxy(6.3,.9)}{\pgfbox[center,center]{\small{$\phi$}}}
\pgfputat{\pgfxy(10.5,3.1)}{\pgfbox[center,center]{\small{$\psi$}}}
\end{pgftranslate}
\end{pgfpicture}

This diagram is commutative with exact row and column, and the maps $\phi:=j'\circ i$ and $\psi:=j\circ i'$ are isomorphisms.  The row and column are exact because they are parts of the long exact sequences in cohomology corresponding to the triple inclusions $T_1\subset T_1\cup T_2\subset X$ and $T_2\subset T_1\cup T_2\subset X$, respectively.   To show that $\phi$ is an isomorphism, consider the morphisms of pairs $(X-T_2,T_1)\rightarrow(X,T_1\cup T_2)$ and $(X,T_1\cup T_2)\rightarrow(X,T_1)$ given by the inclusion $X-T_2\rightarrow X$ and the identity $X\rightarrow X$.  These morphisms induce the maps $j'$ and $i$, respectively, and their composition is the morphism of pairs $(X-T_2,T_1)\rightarrow(X,T_1)$ inducing the map $\phi$.  This is the same morphism of pairs appearing in the following \'etale excision square, where all maps are inclusions:

\begin{pgfpicture}{0cm}{0cm}{17cm}{2cm}
\begin{pgftranslate}{\pgfpoint{5cm}{-.25cm}}
\pgfputat{\pgfxy(2,.5)}{\pgfbox[center,center]{$T_1$}}
\pgfputat{\pgfxy(4.5,.5)}{\pgfbox[center,center]{$X$}}
\pgfputat{\pgfxy(2,2)}{\pgfbox[center,center]{$T_1$}}
\pgfputat{\pgfxy(4.5,2)}{\pgfbox[center,center]{$X-T_2$}}
\pgfsetendarrow{\pgfarrowlargepointed{3pt}}
\pgfxyline(2,1.7)(2,.8)
\pgfxyline(2.5,.5)(4,.5)
\pgfxyline(2.5,2)(3.8,2)
\pgfxyline(4.5,1.7)(4.5,.8)
\end{pgftranslate}
\end{pgfpicture}

Thus, by \'etale excision, $\phi$ is an isomorphism.  A symmetrical argument shows that $\psi$ is an isomorphism.  The inverses of $\phi$ and $\psi$ split the row and column, so $H_{X \tn{ \footnotesize{on} }T_1}^p\oplus H_{X \tn{ \footnotesize{on} }T_2}^p=H_{X \tn{ \footnotesize{on} }T_1\cup T_2}^p$.  By induction, this proves the intermediate statement. 

To prove the lemma, consider again the directed system $\{\overline{Z},\prec\}$ of chains of closed subsets of $X$.   For a particular chain $\overline{Z}:=\oslash\subset Z^d\subset Z^{d-1}\subset...\subset Z^0=X$, I have assumed that the codimension of $Z^p$ is at least $p$.   Label the irreducible components of $Z^p$ of minimal codimension $p$ by $Y_1,...,Y_r$.   Now consider the set difference $Z^p-Z^{p+1}$.   The idea of the proof is that $Z^p-Z^{p+1}$ is ``approximately" equal to the disjoint union $\displaystyle\coprod_{i=1}^r(Y_i-Z^{p+1})$, and that equality is attained at some finite stage of the limiting process.   In general, the $Y_i$ intersect each other, so $Z^{p+1}$ must include all the intersections (which are of strictly lower dimension) for the union to be disjoint.  Also, $Z^p$ itself may include some components of higher codimension, so $Z^{p+1}$ must include these components as well for the difference $Z^p-Z^{p+1}$ to equal the above disjoint union.   By the definition of the partial order $\prec$, given any chain $\overline{Z}$, there exists another chain $\overline{Z}'$, such that $\overline{Z}\prec\overline{Z}'$ and the above conditions are satisfied in $\overline{Z}'$.   For instance, one could simply let ${Z^p}'=Z^p$, and choose $Z_{p+1}'$ to contain the intersections and lower-dimensional components mentioned above.   Once equality has been attained, 
\begin{equation}\label{equlem4231lim}H^{p+q}_{X-Z^{p+1}\tn{ \footnotesize{on} } Z^{p}-Z^{p+1}}=H^{p+q}_{X-Z^{p+1}\tn{ \footnotesize{on} } \coprod_{i=1}^r(Y_i-Z^{p+1})}=\bigoplus_{i=1}^r H^{p+q}_{X-Z^{p+1}\tn{ \footnotesize{on} }Y_i-Z^{p+1}}.\end{equation}
In the limit, $Z^p$ eventually includes every point of codimension at least $p$, and $Z_{p+1}$ eventually includes every point of codimension at least $p+1$.  The difference $Z^p-Z^{p+1}$ eventually includes every point of $\tn{Zar}_X^p$ and eventually excludes every other point.  Thus, the limit is $\displaystyle E^{p,q}\cong\coprod_{x\in \tn{\footnotesize{Zar}}_X^{p}} H_{X\tn{ \footnotesize{on} } x}^{p+q},$ as claimed.
\end{proof}

The $E_1$-terms, and the convergence property, of the coniveau spectral sequence for $\{H^n\}_{n\in\ZZ}$ on $X$, may be expressed by the concise statement
\begin{equation}\label{equE1convergence}E_1^{p,q}\cong\coprod_{x\in \tn{\footnotesize{Zar}}_X^{p}} H_{X\tn{ \footnotesize{on} } x}^{p+q}\Rightarrow H_X^{p+q}.\end{equation}

\section{Cousin Complexes; Bloch-Ogus Theorem}\label{sectioncousinblochogus}

Figure \hyperref[figconiveauintroreprise]{\ref{figconiveauintroreprise}} below shows the $E_1$-level of the coniveau spectral sequence for a cohomology theory with supports $H$ on an equidimensional noetherian scheme $X$ over a field $k$.  The shading serves to divide the ``plane" of bidegrees into quadrants.\footnotemark\footnotetext{This device is particularly useful in future diagrams in which the bidegree of an object may not be obvious from the identity of the object itself.}  Note that all terms $E_1^{p,q}$ with $p\le0$ vanish, since $X$ has no points of negative codimension.   Also note that $H$ may admit nontrivial cohomology groups with negative degrees.  

\begin{figure}[H]
\begin{pgfpicture}{0cm}{0cm}{17cm}{6cm}
\pgfputat{\pgfxy(3.5,5.5)}{\pgfbox[center,center]{$E_1$-level, abstract form}}
\pgfputat{\pgfxy(12,5.5)}{\pgfbox[center,center]{$E_1$-level, concrete form}}
\begin{pgftranslate}{\pgfpoint{.5cm}{0cm}}
\begin{pgfmagnify}{.85}{.85}
\begin{colormixin}{20!white}
\color{black}
\pgfmoveto{\pgfxy(3,3)}
\pgflineto{\pgfxy(7.5,3)}
\pgflineto{\pgfxy(7.5,6)}
\pgflineto{\pgfxy(3,6)}
\pgflineto{\pgfxy(3,3)}
\pgffill
\pgfmoveto{\pgfxy(3,3)}
\pgflineto{\pgfxy(3,0)}
\pgflineto{\pgfxy(-.5,0)}
\pgflineto{\pgfxy(-.5,3)}
\pgflineto{\pgfxy(3,3)}
\pgffill
\end{colormixin}
\pgfputat{\pgfxy(.5,1)}{\pgfbox[center,center]{\large{$E_1^{-1,-1}$}}}
\pgfputat{\pgfxy(3,1)}{\pgfbox[center,center]{\large{$E_1^{0,-1}$}}}
\pgfputat{\pgfxy(5.5,1)}{\pgfbox[center,center]{\large{$E_1^{1,-1}$}}}
\pgfputat{\pgfxy(.5,3)}{\pgfbox[center,center]{\large{$E_1^{-1,0}$}}}
\pgfputat{\pgfxy(3,3)}{\pgfbox[center,center]{\large{$E_1^{0,0}$}}}
\pgfputat{\pgfxy(5.5,3)}{\pgfbox[center,center]{\large{$E_1^{1,0}$}}}
\pgfputat{\pgfxy(.5,5)}{\pgfbox[center,center]{\large{$E_1^{-1,1}$}}}
\pgfputat{\pgfxy(3,5)}{\pgfbox[center,center]{\large{$E_1^{0,1}$}}}
\pgfputat{\pgfxy(5.5,5)}{\pgfbox[center,center]{\large{$E_1^{1,1}$}}}
\pgfputat{\pgfxy(1.8,1.3)}{\pgfbox[center,center]{\small{$d_1^{-1,-1}$}}}
\pgfputat{\pgfxy(1.8,3.3)}{\pgfbox[center,center]{\small{$d_1^{-1,0}$}}}
\pgfputat{\pgfxy(1.8,5.3)}{\pgfbox[center,center]{\small{$d_1^{-1,1}$}}}
\pgfputat{\pgfxy(4.4,1.3)}{\pgfbox[center,center]{\small{$d_1^{0,-1}$}}}
\pgfputat{\pgfxy(6.9,1.3)}{\pgfbox[center,center]{\small{$d_1^{1,-1}$}}}
\pgfputat{\pgfxy(4.4,3.3)}{\pgfbox[center,center]{\small{$d_1^{0,0}$}}}
\pgfputat{\pgfxy(6.9,3.3)}{\pgfbox[center,center]{\small{$d_1^{1,0}$}}}
\pgfputat{\pgfxy(4.4,5.3)}{\pgfbox[center,center]{\small{$d_1^{0,1}$}}}
\pgfputat{\pgfxy(6.9,5.3)}{\pgfbox[center,center]{\small{$d_1^{1,1}$}}}
\pgfsetendarrow{\pgfarrowlargepointed{3pt}}
\pgfxyline(1.2,1)(2.2,1)
\pgfxyline(3.7,1)(4.7,1)
\pgfxyline(1.2,3)(2.2,3)
\pgfxyline(3.7,3)(4.7,3)
\pgfxyline(1.2,5)(2.2,5)
\pgfxyline(3.7,5)(4.7,5)
\pgfxyline(6.2,1)(7.2,1)
\pgfxyline(6.2,3)(7.2,3)
\pgfxyline(6.2,5)(7.2,5)
\end{pgfmagnify}
\end{pgftranslate}
\begin{pgftranslate}{\pgfpoint{8.5cm}{0cm}}
\begin{pgfmagnify}{.85}{.85}
\begin{colormixin}{20!white}
\color{black}
\pgfmoveto{\pgfxy(3,3)}
\pgflineto{\pgfxy(9,3)}
\pgflineto{\pgfxy(9,6)}
\pgflineto{\pgfxy(3,6)}
\pgflineto{\pgfxy(3,3)}
\pgffill
\pgfmoveto{\pgfxy(3,3)}
\pgflineto{\pgfxy(3,0)}
\pgflineto{\pgfxy(0,0)}
\pgflineto{\pgfxy(0,3)}
\pgflineto{\pgfxy(3,3)}
\pgffill
\end{colormixin}
\pgfputat{\pgfxy(.5,1)}{\pgfbox[center,center]{\large{$0$}}}
\pgfputat{\pgfxy(3,.9)}{\pgfbox[center,center]{$\displaystyle\coprod_{x\in \tn{\footnotesize{Zar}}_X^{0}} H_{X\tn{ \footnotesize{on} } x}^{-1}$}}
\pgfputat{\pgfxy(6.5,.9)}{\pgfbox[center,center]{$\displaystyle\coprod_{x\in \tn{\footnotesize{Zar}}_X^{1}} H_{X\tn{ \footnotesize{on} } x}^{0}$}}
\pgfputat{\pgfxy(.5,3)}{\pgfbox[center,center]{\large{$0$}}}
\pgfputat{\pgfxy(3,2.9)}{\pgfbox[center,center]{$\displaystyle\coprod_{x\in \tn{\footnotesize{Zar}}_X^{0}} H_{X\tn{ \footnotesize{on} } x}^{0}$}}
\pgfputat{\pgfxy(6.5,2.9)}{\pgfbox[center,center]{$\displaystyle\coprod_{x\in \tn{\footnotesize{Zar}}_X^{1}} H_{X\tn{ \footnotesize{on} } x}^{1}$}}
\pgfputat{\pgfxy(.5,5)}{\pgfbox[center,center]{\large{$0$}}}
\pgfputat{\pgfxy(3,4.9)}{\pgfbox[center,center]{$\displaystyle\coprod_{x\in \tn{\footnotesize{Zar}}_X^{0}} H_{X\tn{ \footnotesize{on} } x}^{1}$}}
\pgfputat{\pgfxy(6.5,4.9)}{\pgfbox[center,center]{$\displaystyle\coprod_{x\in \tn{\footnotesize{Zar}}_X^{1}} H_{X\tn{ \footnotesize{on} } x}^{2}$}}
\pgfputat{\pgfxy(4.9,1.3)}{\pgfbox[center,center]{\small{$d_1^{0,-1}$}}}
\pgfputat{\pgfxy(8.4,1.3)}{\pgfbox[center,center]{\small{$d_1^{1,-1}$}}}
\pgfputat{\pgfxy(4.9,3.3)}{\pgfbox[center,center]{\small{$d_1^{0,0}$}}}
\pgfputat{\pgfxy(8.4,3.3)}{\pgfbox[center,center]{\small{$d_1^{1,0}$}}}
\pgfputat{\pgfxy(4.9,5.3)}{\pgfbox[center,center]{\small{$d_1^{0,1}$}}}
\pgfputat{\pgfxy(8.4,5.3)}{\pgfbox[center,center]{\small{$d_1^{1,1}$}}}
\pgfsetendarrow{\pgfarrowlargepointed{3pt}}
\pgfxyline(.8,1)(1.7,1)
\pgfxyline(4.4,1)(5.2,1)
\pgfxyline(.8,3)(1.7,3)
\pgfxyline(4.4,3)(5.2,3)
\pgfxyline(.8,5)(1.7,5)
\pgfxyline(4.4,5)(5.2,5)
\pgfxyline(7.9,1)(8.7,1)
\pgfxyline(7.9,3)(8.7,3)
\pgfxyline(7.9,5)(8.7,5)
\end{pgfmagnify}
\end{pgftranslate}
\end{pgfpicture}
\caption{$E_1$ level of the coniveau spectral sequence, expressed abstractly and in terms of filtration by codimension.}
\label{figconiveauintroreprise}
\end{figure}

\subsection{Cousin Complexes}\label{subsectioncousincomplexes}

The rows of the $E_1$ level of the coniveau spectral sequence $\{E_r^{p,q},d_r^{p,q}\}$ for $H$ on $X$ are important on an individual basis, since they arise in very general settings and possess important universal properties.  In particular, they are {\it Cousin complexes} at the group level.  The general theory of Cousin complexes is outlined in section \hyperref[subsectionlocalcohomsheaves]{\ref{subsectionlocalcohomsheaves}} of the appendix.   As described below, sheafified versions of these Cousin complexes are {\it flasque resolutions} of the sheaves $\ms{H}_X^q$ induced by $H$ on $X$, and may therefore be used to compute the sheaf cohomology groups of $\ms{H}_X^q$.  In particular, when $H$ is algebraic $K$-theory, the sheafified Cousin complexes may be used to compute the Chow groups $\tn{Ch}_X^q=H^q(\ms{K}_{q,X})$.  

\begin{defi}\label{defigroupCousincomplexes} The $q$th row of the $E_1$-level of the coniveau spectral sequence $\{E_r^{p,q},d_r^{p,q}\}$ for $H$ on $X$ is called the $q$th {\bf Cousin complex} for $H$ on $X$.  
\end{defi}

If the dimension of $X$ is $d$, the $q$th Cousin complex has at most $d+1$ nontrivial terms, since the only possible codimensions for points of $X$ are $0,...,d$.  Hence, the $q$th Cousin complex for a cohomology theory with supports $H$ on a $d$-dimensional scheme $X$ is of the form

\begin{pgfpicture}{0cm}{0cm}{17cm}{2cm}
\begin{pgftranslate}{\pgfpoint{-.4cm}{0cm}}
\pgfputat{\pgfxy(16.3,1)}{\pgfbox[center,center]{(4.3.1.0)}}
\pgfputat{\pgfxy(.7,1)}{\pgfbox[center,center]{$0$}}
\pgfputat{\pgfxy(3,.9)}{\pgfbox[center,center]{$\displaystyle\coprod_{x\in \tn{\footnotesize{Zar}}_X^{0}} H_{X\tn{ \footnotesize{on} } x}^{q}$}}
\pgfputat{\pgfxy(6.5,.9)}{\pgfbox[center,center]{$\displaystyle\coprod_{x\in \tn{\footnotesize{Zar}}_X^{1}} H_{X\tn{ \footnotesize{on} } x}^{q+1}$}}
\pgfputat{\pgfxy(12.5,.9)}{\pgfbox[center,center]{$\displaystyle\coprod_{x\in \tn{\footnotesize{Zar}}_X^{d}} H_{X\tn{ \footnotesize{on} } x}^{q+d}$}}
\pgfputat{\pgfxy(15,1)}{\pgfbox[center,center]{$0$}}
\pgfputat{\pgfxy(4.9,1.3)}{\pgfbox[center,center]{\small{$d_1^{0,q}$}}}
\pgfputat{\pgfxy(8.4,1.3)}{\pgfbox[center,center]{\small{$d_1^{1,q}$}}}
\pgfputat{\pgfxy(10.9,1.3)}{\pgfbox[center,center]{\small{$d_1^{d-1,q}$}}}
\pgfsetendarrow{\pgfarrowlargepointed{3pt}}
\pgfxyline(.9,1)(1.7,1)
\pgfxyline(4.4,1)(5.2,1)
\pgfxyline(7.9,1)(8.7,1)
\pgfnodecircle{Node0}[fill]{\pgfxy(8.9,1)}{0.02cm}
\pgfnodecircle{Node0}[fill]{\pgfxy(9,1)}{0.02cm}
\pgfnodecircle{Node0}[fill]{\pgfxy(9.1,1)}{0.02cm}
\pgfnodecircle{Node0}[fill]{\pgfxy(10,1)}{0.02cm}
\pgfnodecircle{Node0}[fill]{\pgfxy(10.1,1)}{0.02cm}
\pgfnodecircle{Node0}[fill]{\pgfxy(10.2,1)}{0.02cm}
\pgfxyline(10.4,1)(11.2,1)
\pgfxyline(13.9,1)(14.7,1)
\end{pgftranslate}
\end{pgfpicture}

The next lemma, adapted from Colliot-Th\'el\`ene, Hoobler, and Kahn \cite{CHKBloch-Ogus-Gabber97}, Lemma 1.2.2, page 7, says that the terms in the Cousin complexes sheafify to give flasque sheaves on $X$.  A sheaf is called {\bf flasque} if its restriction maps are surjective.  In the present case, flasqueness is a reflection of reductionism; all the information in every sheaf appearing in each Cousin complex belongs to individual points of $X$, so the information assigned by every such a sheaf to an open subset of $X$ is neither more nor less than the sum of the pointwise information over the relevant points in the subset. 

\begin{lem}\label{lemflasque} For all choices of $n$ and $p$, the presheaf $\ms{H}_X^n$ on $X$ 
\begin{equation}\label{equpresheafHn}U\mapsto H_U^n=\coprod_{x\in \tn{\footnotesize{Zar}}_U^p} H^{n}_{X\tn{ \footnotesize{on} } x}\end{equation}
is a flasque sheaf, which may be identified with the coproduct of skyscraper sheaves
\begin{equation}\label{equpresheafHnissheaf}\coprod_{x\in \tn{\footnotesize{Zar}}_X^p}  \underline{H^{n}_{X\tn{ \footnotesize{on} } x}}.\end{equation}
\end{lem}

\begin{proof} The equality of the specified presheaf and sheaf follows immediately from the direct limit definition of the cohomology group $H^{n}_{X\tn{ \footnotesize{on} } x}$, and the definition of a skyscraper sheaf.  Indeed, the summand $\underline{H^{n}_{X\tn{ \footnotesize{on} } x}}$ of the specified sheaf contributes a factor of $H^{n}_{X\tn{ \footnotesize{on} } x}$ to its group of sections over $U$ if and only if $x\in U$, which gives precisely the same groups of sections as the specified presheaf.  Flasqueness follows from the definition of inclusion and the fact that all the information in the sheaf belongs to individual points of $X$.      
\end{proof}

The sheafified Cousin complexes are of the form 

\begin{pgfpicture}{0cm}{0cm}{17cm}{2cm}
\begin{pgftranslate}{\pgfpoint{-.4cm}{0cm}}
\pgfputat{\pgfxy(16.3,1)}{\pgfbox[center,center]{(4.3.1.3)}}
\pgfputat{\pgfxy(.7,1)}{\pgfbox[center,center]{$0$}}
\pgfputat{\pgfxy(3,.9)}{\pgfbox[center,center]{$\displaystyle\coprod_{x\in \tn{\footnotesize{Zar}}_X^{0}} \underline{H_{X\tn{ \footnotesize{on} } x}^{q}}$}}
\pgfputat{\pgfxy(6.5,.9)}{\pgfbox[center,center]{$\displaystyle\coprod_{x\in \tn{\footnotesize{Zar}}_X^{1}} \underline{H_{X\tn{ \footnotesize{on} } x}^{q+1}}$}}
\pgfputat{\pgfxy(12.5,.9)}{\pgfbox[center,center]{$\displaystyle\coprod_{x\in \tn{\footnotesize{Zar}}_X^{d}} \underline{H_{X\tn{ \footnotesize{on} } x}^{q+d}}$}}
\pgfputat{\pgfxy(15,1)}{\pgfbox[center,center]{$0$,}}
\pgfputat{\pgfxy(4.9,1.3)}{\pgfbox[center,center]{\small{$d_1^{0,q}$}}}
\pgfputat{\pgfxy(8.4,1.3)}{\pgfbox[center,center]{\small{$d_1^{1,q}$}}}
\pgfputat{\pgfxy(10.9,1.3)}{\pgfbox[center,center]{\small{$d_1^{d-1,q}$}}}
\pgfsetendarrow{\pgfarrowlargepointed{3pt}}
\pgfxyline(.9,1)(1.7,1)
\pgfxyline(4.4,1)(5.2,1)
\pgfxyline(7.9,1)(8.7,1)
\pgfnodecircle{Node0}[fill]{\pgfxy(8.9,1)}{0.02cm}
\pgfnodecircle{Node0}[fill]{\pgfxy(9,1)}{0.02cm}
\pgfnodecircle{Node0}[fill]{\pgfxy(9.1,1)}{0.02cm}
\pgfnodecircle{Node0}[fill]{\pgfxy(10,1)}{0.02cm}
\pgfnodecircle{Node0}[fill]{\pgfxy(10.1,1)}{0.02cm}
\pgfnodecircle{Node0}[fill]{\pgfxy(10.2,1)}{0.02cm}
\pgfxyline(10.4,1)(11.2,1)
\pgfxyline(13.9,1)(14.7,1)
\end{pgftranslate}
\end{pgfpicture}
\label{equsheafifiedcousin}

where I have used the same notation for the sheaf maps as for the maps at the level of groups. 


\subsection{Bloch-Ogus Theorem}\label{BlochOgus}

As shown in the last few sections, the construction of the coniveau spectral sequence and Cousin complexes for a cohomology theory with supports $H$ on a scheme $X$ over a field $k$ requires some assumptions on $X$, but none on $H$.   In particular, though the construction of the coniveau spectral sequence requires $X$ to be equidimensional and noetherian, the only property of $H$ used in the construction is the defining property that the sequences of inclusions $Z^{p+1}\subset Z^p\subset X$ induce long exact sequence of cohomology groups with supports.   Indeed, the Cousin complexes exist even under much more general conditions.   The desired use of the sheafified Cousin complexes in the present context, however, is to compute the cohomology groups of the sheaves $\ms{H}_X^n$ on $X$ associated to the cohomology theory with supports $H$.  For this to succeed, the sheafified Cousin complexes must be acyclic resolutions.   Each sheaf in the sheafified Cousin complexes is automatically flasque, and hence acyclic, but the complexes must also be exact, which fails for general cohomology theories with supports, even for equidimensional noetherian schemes.  Hence, it is necessary to introduce some further assumptions on $X$ and $H$.   In this section, I will assume that $X$ is smooth and $H$ is effaceable.  The smoothness criterion can be relaxed to a certain extent. 

The following lemma is adapted from Colliot-Th\'el\`ene, Hoobler, and Kahn \cite{CHKBloch-Ogus-Gabber97}, Proposition 2.1.2, page 8.  The proof I present is easier than the proof given there, because the effaceability condition I use is simpler than the one used in \cite{CHKBloch-Ogus-Gabber97}.\footnotemark\footnotetext{\cite{CHKBloch-Ogus-Gabber97}, Proposition 2.1.2, page 8, involves the weaker effaceability condition which assumes only that the composition $H_{W\tn{ \footnotesize{on} } Z}\rightarrow H_{U\tn{ \footnotesize{on} } Z\cap U}\rightarrow H_{U\tn{ \footnotesize{on} } Z'\cap U}$ vanishes.  The effaceability condition used here, which  \cite{CHKBloch-Ogus-Gabber97} calls {\it strict effaceability}, assumes that the second arrow in this composition itself vanishes.}  Note that the group-level Cousin complexes involving the semilocal ring spectrum $Y$ appearing in this proposition are {\it not} the group-level Cousin complexes for $X$, and their exactness does {\it not} imply exactness of the group-level Cousin complexes for $X$.  However, they sheafify to give the same sheaf-level Cousin complexes, and their exactness implies the exactness of these sheaf complexes. 

\begin{lem}\label{semilocaleffacement} Let $\mbf{S}_k$ be a distinguished category of schemes over a field $k$, and let $X$ be an object of $\mbf{S}_k$.  Assume that $X$ is equidimensional and noetherian of dimension $d$.  Let $H:=\{H^n\}_{n\in\ZZ}$ be a cohomology theory with supports on the category of pairs over $\mbf{S}_k$.   Let $R=O_{t_1,...,t_r}$ be the semilocal ring of $X$ at $(t_1,...,t_r)$, and let $Y=\tn{Spec }R$.  Suppose that $H$ is effaceable at $(t_1,...,t_r)$.  Then in the exact couple defining the coniveau spectral sequence for $H$ on $Y$, the maps $i^{p,q}$ are identically zero for all $p>0$.  Hence, computing the cohomology of the $E_1$-level, 
\[E_2^{p,q}=\begin{cases} H_{\tn{\footnotesize{Zar}}}^p(X,H_X^q) &\tn{if } p = 0 \\
0 & \tn{if } p > 0. \end{cases}\]
Further, the Cousin complexes yield exact sequences

\begin{pgfpicture}{0cm}{0cm}{17cm}{2cm}
\begin{pgftranslate}{\pgfpoint{1.5cm}{0cm}}

\pgfputat{\pgfxy(-1,1)}{\pgfbox[center,center]{$0$}}
\pgfputat{\pgfxy(.5,1)}{\pgfbox[center,center]{$H_Y^n$}}
\pgfputat{\pgfxy(3,.9)}{\pgfbox[center,center]{$\displaystyle\coprod_{x\in \tn{\footnotesize{Zar}}_Y^{0}} H_{Y\tn{ \footnotesize{on} } x}^{n}$}}
\pgfputat{\pgfxy(6.5,.9)}{\pgfbox[center,center]{$\displaystyle\coprod_{x\in \tn{\footnotesize{Zar}}_Y^{1}} H_{Y\tn{ \footnotesize{on} } x}^{n+1}$}}
\pgfputat{\pgfxy(12.5,.9)}{\pgfbox[center,center]{$\displaystyle\coprod_{x\in \tn{\footnotesize{Zar}}_Y^{d}} H_{Y\tn{ \footnotesize{on} } x}^{n+d}$}}
\pgfputat{\pgfxy(15,1)}{\pgfbox[center,center]{$0$}}
\pgfputat{\pgfxy(4.9,1.3)}{\pgfbox[center,center]{\small{$d_1^{0,n}$}}}
\pgfputat{\pgfxy(8.4,1.3)}{\pgfbox[center,center]{\small{$d_1^{1,n}$}}}
\pgfputat{\pgfxy(10.9,1.3)}{\pgfbox[center,center]{\small{$d_1^{d-1,n}$}}}
\pgfsetendarrow{\pgfarrowlargepointed{3pt}}
\pgfxyline(-.8,1)(0,1)
\pgfxyline(.9,1)(1.7,1)
\pgfxyline(4.4,1)(5.2,1)
\pgfxyline(7.9,1)(8.7,1)
\pgfnodecircle{Node0}[fill]{\pgfxy(8.9,1)}{0.02cm}
\pgfnodecircle{Node0}[fill]{\pgfxy(9,1)}{0.02cm}
\pgfnodecircle{Node0}[fill]{\pgfxy(9.1,1)}{0.02cm}
\pgfnodecircle{Node0}[fill]{\pgfxy(10,1)}{0.02cm}
\pgfnodecircle{Node0}[fill]{\pgfxy(10.1,1)}{0.02cm}
\pgfnodecircle{Node0}[fill]{\pgfxy(10.2,1)}{0.02cm}
\pgfxyline(10.4,1)(11.2,1)
\pgfxyline(13.9,1)(14.7,1)
\end{pgftranslate}
\end{pgfpicture}

\end{lem}

\begin{proof} Let $X, U, W, Z,$ and $Z$ be defined as in the effaceability condition, as illustrated in the schematic diagram on the right below, and let $(t_1,...,t_r)$ and $Y$ be defined as in the statement of the proposition.  For each $p>0$, Consider the diagram on the left in figure \hyperref[figeffaceabilityprop]{\ref{figeffaceabilityprop}} below:

\begin{figure}[H]
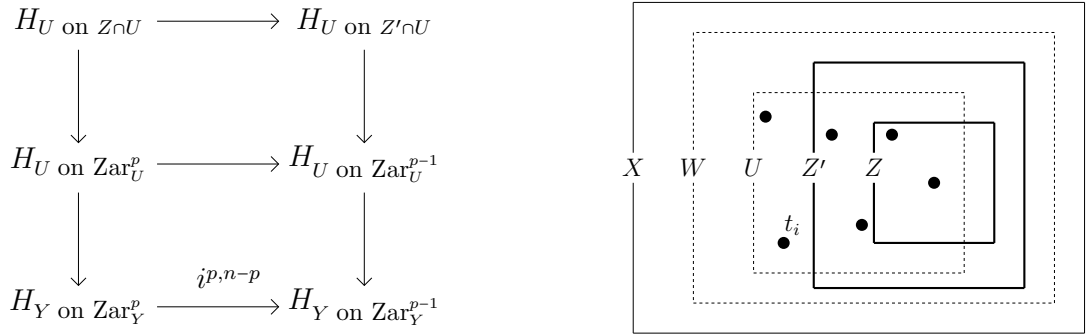

\begin{pgfpicture}{0cm}{0cm}{17cm}{5.5cm}
\begin{pgftranslate}{\pgfpoint{-1.5cm}{0cm}}
\begin{pgfmagnify}{.95}{.95}
\pgfputat{\pgfxy(5.6,1.4)}{\pgfbox[center,center]{$i^{p,n-p}$}}
\pgfputat{\pgfxy(3.5,5)}{\pgfbox[center,center]{$H_{U\tn{ \footnotesize{on} } Z\cap U}$}}
\pgfputat{\pgfxy(7.5,5)}{\pgfbox[center,center]{$H_{U\tn{ \footnotesize{on} } Z'\cap U}$}}
\pgfputat{\pgfxy(3.5,3)}{\pgfbox[center,center]{$H_{U\tn{ \footnotesize{on} } \tn{\footnotesize{Zar}}_U^{p}}$}}
\pgfputat{\pgfxy(7.5,3)}{\pgfbox[center,center]{$H_{U\tn{ \footnotesize{on} } \tn{\footnotesize{Zar}}_U^{p-1}}$}}
\pgfputat{\pgfxy(3.5,1)}{\pgfbox[center,center]{$H_{Y\tn{ \footnotesize{on} } \tn{\footnotesize{Zar}}_Y^{p}}$}}
\pgfputat{\pgfxy(7.5,1)}{\pgfbox[center,center]{$H_{Y\tn{ \footnotesize{on} } \tn{\footnotesize{Zar}}_Y^{p-1}}$}}
\pgfsetendarrow{\pgfarrowlargepointed{3pt}}
\pgfxyline(4.6,5)(6.3,5)
\pgfxyline(4.6,3)(6.3,3)
\pgfxyline(4.6,1)(6.3,1)
\pgfxyline(3.5,4.6)(3.5,3.3)
\pgfxyline(7.5,4.6)(7.5,3.3)
\pgfxyline(3.5,2.6)(3.5,1.3)
\pgfxyline(7.5,2.6)(7.5,1.3)
\end{pgfmagnify}
\end{pgftranslate}
\begin{pgftranslate}{\pgfpoint{10cm}{.6cm}}
\begin{pgfmagnify}{.8}{.8}
\pgfxyline(-1,0)(6.5,0)
\pgfxyline(6.5,0)(6.5,5.5)
\pgfxyline(6.5,5.5)(-1,5.5)
\pgfxyline(-1,5.5)(-1,3)
\pgfxyline(-1,2.5)(-1,0)
\begin{pgfscope}
\pgfsetdash{{0.05cm}{0.05cm}}{0cm}
\pgfxyline(0,.5)(6,.5)
\pgfxyline(6,.5)(6,5)
\pgfxyline(6,5)(0,5)
\pgfxyline(0,5)(0,3)
\pgfxyline(0,2.5)(0,.5)
\pgfxyline(1,1)(4.5,1)
\pgfxyline(4.5,1)(4.5,4)
\pgfxyline(4.5,4)(1,4)
\pgfxyline(1,4)(1,3)
\pgfxyline(1,2.5)(1,1)
\end{pgfscope}
\pgfsetlinewidth{1pt}
\pgfxyline(2,.75)(5.5,.75)
\pgfxyline(5.5,.75)(5.5,4.5)
\pgfxyline(5.5,4.5)(2,4.5)
\pgfxyline(2,4.5)(2,3)
\pgfxyline(2,2.5)(2,.75)
\pgfxyline(3,1.5)(5,1.5)
\pgfxyline(5,1.5)(5,3.5)
\pgfxyline(5,3.5)(3,3.5)
\pgfxyline(3,3.5)(3,3)
\pgfxyline(3,2.5)(3,1.5)
\pgfputat{\pgfxy(-1,2.75)}{\pgfbox[center,center]{$X$}}
\pgfputat{\pgfxy(0,2.75)}{\pgfbox[center,center]{$W$}}
\pgfputat{\pgfxy(1,2.75)}{\pgfbox[center,center]{$U$}}
\pgfputat{\pgfxy(2,2.75)}{\pgfbox[center,center]{$Z'$}}
\pgfputat{\pgfxy(3,2.75)}{\pgfbox[center,center]{$Z$}}
\pgfputat{\pgfxy(1.65,1.8)}{\pgfbox[center,center]{$t_i$}}
\pgfnodecircle{Node0}[fill]{\pgfxy(2.3,3.3)}{0.1cm}
\pgfnodecircle{Node0}[fill]{\pgfxy(4,2.5)}{0.1cm}
\pgfnodecircle{Node0}[fill]{\pgfxy(2.8,1.8)}{0.1cm}
\pgfnodecircle{Node0}[fill]{\pgfxy(3.3,3.3)}{0.1cm}
\pgfnodecircle{Node0}[fill]{\pgfxy(1.5,1.5)}{0.1cm}
\pgfnodecircle{Node0}[fill]{\pgfxy(1.2,3.6)}{0.1cm}
\end{pgfmagnify}
\end{pgftranslate}
\end{pgfpicture}
\caption{Effaceability, the $E_2$-page of the coniveau spectral sequence, and exactness of the corresponding Cousin complexes.}
\label{figeffaceabilityprop}
\end{figure}

By effaceability, the top horizontal map is identically zero.  Passing to the limit over $Z$, the middle horizontal map is also zero.  Passing to the limit over $W$, which contains $U$ and $(t_1,...,t_r)$, the bottom horizontal map $i^{p,n-p}$ is zero.   Recall that the differential $d_1^{p,n-p}:E_1^{p,n-p}\rightarrow E_1^{p+1,n-p}$ is the composition $j_1^{p+1,n-p}\circ k_1^{p,n-p}$, since $k$ follows $j$ in the diagram for the exact couple.  By exactness, $\tn{Ker}(k_1^{p,n-p})=\tn{Im}(i_1^{p,n-p})=0$, so $d_1^{p,n-p}$ vanishes for every $p>0$.   This leaves the first differential $d_1^{0,n}$ to analyze.  This is the composition 
\end{proof}

The following corollary of lemma \hyperref[semilocaleffacement]{\ref{semilocaleffacement}} is a version of the {\it Bloch-Ogus theorem.}  It is adapted from \cite{CHKBloch-Ogus-Gabber97} corollary 5.1.11, page 29, and proposition 5.4.3, page 33.   

\begin{cor}\label{corblochogus}(Bloch-Ogus Theorem). Let $k$ be an infinite field, and $\mbf{S}_k$ a category of schemes over $k$ satisfying the conditions given at the beginning of section \hyperref[subsectioncohomsupportssubstrata]{\ref{subsectioncohomsupportssubstrata}}.  Let $H$ be a cohomology theory with supports on the category of pairs over $\mbf{S}_k$, satisfying \'etale excision \tn{({\bf COH1})} and the projective bundle condition \tn{({\bf COH5})}.  Then, for any smooth scheme $X$ belonging to $\mbf{S}_k$, the sheafified Cousin complexes appearing in equation \hyperref[equsheafifiedcousin]{4.3.1.3} above are flasque resolutions of the sheaves $\ms{H}_X^n$ associated to the presheaves $U\mapsto H_U^n$, and the $E_2$-terms of the coniveau spectral sequence for $H$ on $X$ are
\begin{equation}\label{equblochogus}E_2^{p,q}=H_{\tn{\footnotesize{Zar}}}^p(X,\ms{H}_X^q).\end{equation}

\end{cor}
\begin{proof} Lemma  \hyperref[semilocaleffacement]{\ref{semilocaleffacement}} implies that the sheafified Cousin complexes are flasque resolutions of $\ms{H}^q_X$, with groups of global sections given by the $E_2$-terms in equation \hyperref[equblochogus]{\ref{equblochogus}}. 
\end{proof}

\section{Coniveau Machine for Algebraic $K$-Theory in the Nilpotent Case}\label{sectionconiveaumachine}

In this section, I construct in detail the coniveau machine for Bass-Thomason algebraic $K$-theory, focusing on the ``simplified four-column version" for a smooth algebraic variety over a field $k$, augmented via a nilpotent thickening.  Here, $K$-theory is viewed as a cohomology theory with supports on the category $\mbf{P}_k$ of pairs over a distinguished category $\mbf{S}_k$ of schemes over a field $k$, satisfying the conditions given at the beginning of section \hyperref[subsectioncohomsupportssubstrata]{\ref{subsectioncohomsupportssubstrata}} above.  For the convenience of the reader, I reproduce in figure \hyperref[figsimplifiedfourcolumnKtheorych4b]{\ref{figsimplifiedfourcolumnKtheorych4b}} below the schematic diagram of the four-column version already appearing in figures \hyperref[figsimplifiedfourcolumnKtheoryintro]{\ref{figsimplifiedfourcolumnKtheoryintro}} and \hyperref[figsimplifiedfourcolumnKtheorych4]{\ref{figsimplifiedfourcolumnKtheorych4}} above.  

The arguments presented in this section, together with previous results from chapters \hyperref[ChapterTechnical]{\ref{ChapterTechnical}} and \hyperref[ChapterConiveau]{\ref{ChapterConiveau}}, actually serve to establish the more general version of the coniveau machine appearing in equation \hyperref[equconiveaumachinefunctorch4]{4.1.1.1} above, but I choose to concentrate mostly on the simpler four-column construction in order to focus attention on the infinitesimal theory of Chow groups.  Throughout this section, schemes over a field $k$ are assumed to be ``as general as is convenient;" for example, in section \hyperref[subsectionfirstcolumn]{\ref{subsectionfirstcolumn}} below, I sometimes make the assumption that a $k$-scheme $X$ is equidimensional and noetherian.  However, these assumptions always include the case of smooth algebraic varieties over $k$.

\begin{figure}[H]
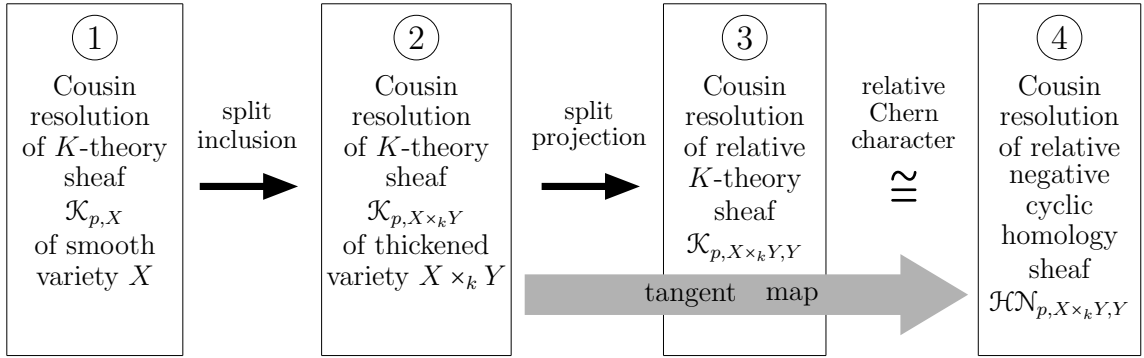

\begin{pgfpicture}{0cm}{0cm}{17cm}{4.7cm}
\begin{pgfmagnify}{.9}{.9}
\begin{pgftranslate}{\pgfpoint{.2cm}{-1.7cm}}
\begin{pgftranslate}{\pgfpoint{-.5cm}{0cm}}
\pgfxyline(.9,1.8)(.9,7)
\pgfxyline(.9,7)(3.5,7)
\pgfxyline(3.5,7)(3.5,1.8)
\pgfxyline(3.5,1.8)(.9,1.8)
\pgfputat{\pgfxy(2.2,5.8)}{\pgfbox[center,center]{Cousin}}
\pgfputat{\pgfxy(2.2,5.35)}{\pgfbox[center,center]{resolution }}
\pgfputat{\pgfxy(2.2,4.85)}{\pgfbox[center,center]{of $K$-theory}}
\pgfputat{\pgfxy(2.2,4.45)}{\pgfbox[center,center]{sheaf}}
\pgfputat{\pgfxy(2.2,3.9)}{\pgfbox[center,center]{$\ms{K}_{p,X}$}}
\pgfputat{\pgfxy(2.2,3.45)}{\pgfbox[center,center]{of smooth}}
\pgfputat{\pgfxy(2.2,2.95)}{\pgfbox[center,center]{variety $X$}}
\pgfnodecircle{Node0}[stroke]{\pgfxy(2.2,6.5)}{0.33cm}
\pgfputat{\pgfxy(2.2,6.5)}{\pgfbox[center,center]{\large{$1$}}}
\end{pgftranslate}
\begin{pgftranslate}{\pgfpoint{4.05cm}{0cm}}
\pgfxyline(1,1.8)(1,7)
\pgfxyline(1,7)(3.8,7)
\pgfxyline(3.8,7)(3.8,1.8)
\pgfxyline(3.8,1.8)(1,1.8)
\pgfputat{\pgfxy(2.4,5.8)}{\pgfbox[center,center]{Cousin}}
\pgfputat{\pgfxy(2.4,5.35)}{\pgfbox[center,center]{resolution }}
\pgfputat{\pgfxy(2.4,4.85)}{\pgfbox[center,center]{of $K$-theory}}
\pgfputat{\pgfxy(2.4,4.45)}{\pgfbox[center,center]{sheaf}}
\pgfputat{\pgfxy(2.4,3.9)}{\pgfbox[center,center]{$\ms{K}_{p,X\times_kY}$}}
\pgfputat{\pgfxy(2.4,3.45)}{\pgfbox[center,center]{of thickened}}
\pgfputat{\pgfxy(2.4,2.95)}{\pgfbox[center,center]{variety $X\times_kY$}}
 \pgfnodecircle{Node0}[stroke]{\pgfxy(2.4,6.5)}{0.33cm}
\pgfputat{\pgfxy(2.4,6.5)}{\pgfbox[center,center]{\large{$2$}}}
\pgfsetendarrow{\pgfarrowtriangle{6pt}}
\pgfsetlinewidth{3pt}
\pgfxyline(-.8,4.3)(.4,4.3)
\pgfputat{\pgfxy(-.1,5.4)}{\pgfbox[center,center]{\small{split}}}
\pgfputat{\pgfxy(-.1,5)}{\pgfbox[center,center]{\small{inclusion}}}
\end{pgftranslate}
\begin{pgftranslate}{\pgfpoint{9.1cm}{0cm}}
\pgfxyline(1,1.8)(1,7)
\pgfxyline(1,7)(3.4,7)
\pgfxyline(3.4,7)(3.4,1.8)
\pgfxyline(3.4,1.8)(1,1.8)
\pgfputat{\pgfxy(2.2,5.8)}{\pgfbox[center,center]{Cousin}}
\pgfputat{\pgfxy(2.2,5.35)}{\pgfbox[center,center]{resolution}}
\pgfputat{\pgfxy(2.2,4.9)}{\pgfbox[center,center]{of relative}}
\pgfputat{\pgfxy(2.2,4.4)}{\pgfbox[center,center]{$K$-theory}}
\pgfputat{\pgfxy(2.2,3.95)}{\pgfbox[center,center]{sheaf}}
\pgfputat{\pgfxy(2.2,3.4)}{\pgfbox[center,center]{$\ms{K}_{p,X\times_kY,Y}$}}
\pgfnodecircle{Node0}[stroke]{\pgfxy(2.2,6.5)}{0.33cm}
\pgfputat{\pgfxy(2.2,6.5)}{\pgfbox[center,center]{\large{$3$}}}
\pgfsetendarrow{\pgfarrowtriangle{6pt}}
\pgfsetlinewidth{3pt}
\pgfxyline(-.8,4.3)(.4,4.3)
\pgfputat{\pgfxy(-.1,5.4)}{\pgfbox[center,center]{\small{split}}}
\pgfputat{\pgfxy(-.1,5)}{\pgfbox[center,center]{\small{projection}}}
\end{pgftranslate}
\begin{pgftranslate}{\pgfpoint{13.75cm}{0cm}}
\pgfxyline(1,1.8)(1,7)
\pgfxyline(1,7)(3.4,7)
\pgfxyline(3.4,7)(3.4,1.8)
\pgfxyline(3.4,1.8)(1,1.8)
\pgfputat{\pgfxy(2.2,5.8)}{\pgfbox[center,center]{Cousin}}
\pgfputat{\pgfxy(2.2,5.35)}{\pgfbox[center,center]{resolution}}
\pgfputat{\pgfxy(2.2,4.9)}{\pgfbox[center,center]{of relative}}
\pgfputat{\pgfxy(2.2,4.45)}{\pgfbox[center,center]{negative}}
\pgfputat{\pgfxy(2.2,4)}{\pgfbox[center,center]{cyclic}}
\pgfputat{\pgfxy(2.2,3.55)}{\pgfbox[center,center]{homology}}
\pgfputat{\pgfxy(2.2,3.1)}{\pgfbox[center,center]{sheaf}}
\pgfputat{\pgfxy(2.2,2.55)}{\pgfbox[center,center]{$\ms{HN}_{p,X\times_kY,Y}$}}
\pgfputat{\pgfxy(-.1,5.8)}{\pgfbox[center,center]{\small{relative}}}
\pgfputat{\pgfxy(-.1,5.4)}{\pgfbox[center,center]{\small{Chern}}}
\pgfputat{\pgfxy(-.1,5)}{\pgfbox[center,center]{\small{character}}}
\pgfputat{\pgfxy(-.1,4.3)}{\pgfbox[center,center]{\huge{$\cong$}}}
\pgfnodecircle{Node0}[stroke]{\pgfxy(2.2,6.5)}{0.33cm}
\pgfputat{\pgfxy(2.2,6.5)}{\pgfbox[center,center]{\large{$4$}}}
\end{pgftranslate}
\begin{colormixin}{30!white}
\color{black}
\pgfmoveto{\pgfxy(8.05,3)}
\pgflineto{\pgfxy(13.6,3)}
\pgflineto{\pgfxy(13.6,3.3)}
\pgflineto{\pgfxy(14.6,2.7)}
\pgflineto{\pgfxy(13.6,2.1)}
\pgflineto{\pgfxy(13.6,2.4)}
\pgflineto{\pgfxy(8.05,2.4)}
\pgflineto{\pgfxy(8.05,3)}
\pgffill
\end{colormixin}
\pgfputat{\pgfxy(10.5,2.7)}{\pgfbox[center,center]{tangent}}
\pgfputat{\pgfxy(12,2.67)}{\pgfbox[center,center]{map}}
\end{pgftranslate}
\end{pgfmagnify}
\end{pgfpicture}
\caption{Simplified ``four column version" of the coniveau machine for algebraic $K$-theory on a smooth algebraic variety in the case of a nilpotent thickening.}
\label{figsimplifiedfourcolumnKtheorych4b}
\end{figure}

\subsection{First Column of the Machine via the Bloch-Ogus Theorem}\label{subsectionfirstcolumn}

Let $\mbf{S}_k$ be a distinguished category of schemes over a field $k$, satisfying the conditions given at the beginning of section \hyperref[subsectioncohomsupportssubstrata]{\ref{subsectioncohomsupportssubstrata}} above, and let $X$ in $\mbf{S}_k$ be smooth.  If $H$ is an effaceable cohomology theory with supports on the category of pairs over $\mbf{S}_k$, then the Bloch-Ogus theorem \hyperref[corblochogus]{\ref{corblochogus}} guarantees ``good enough behavior" to construct the first column of the coniveau machine for $H$ on $X$.  Here, of course, $H$ is Bass-Thomason algebraic $K$-theory.  


{\bf Homological versus Cohomological $K$-Theory.}  Let $Z$ be a closed subset of $X$, and let $\mbf{K}_X$, $\mbf{K}_{X\tn{ \footnotesize{on} } Z},$ and $\mbf{K}_{X-Z},$ be the $K$-theory spectra of Bass and Thomason.  Let $K_{n, X}$, $K_{n, X\tn{ \footnotesize{on} } Z}$, and $K_{n, X-Z}$ be the corresponding $K$-theory groups, and $\ms{K}_{n, X}$, $\ms{K}_{n, X\tn{ \footnotesize{on} } Z}$, and $\ms{K}_{n, X-Z}$ the corresponding $K$-theory sheaves.    In section \hyperref[subsectioneffaceabilitynonconnectiveK]{\ref{subsectioneffaceabilitynonconnectiveK}} above, I showed how results of Thomason \cite{Thomason-Trobaugh90} imply that $\mbf{K}$ is an effaceable substratum on the category $\mbf{P}_k$ of pairs over $\mbf{S}_k$.  Hence, Bass-Thomason $K$-theory is an effaceable cohomology theory with supports.  In particular, by Thomason's localization theorem ( \cite{Thomason-Trobaugh90}, theorem 7.4), there exists a long exact sequence
\begin{equation}\label{equlesthomasonsec44}...\longleftarrow K_{n, X\tn{ \footnotesize{on} } Z}\longleftarrow K_{n+1, X}\longleftarrow K_{n+1, X-Z}\longleftarrow K_{n+1, X\tn{ \footnotesize{on} } Z}\longleftarrow...\end{equation}
The reason for writing the arrows going to the left is to emphasize that $K$-theory is traditionally expressed in homologically rather than cohomologically.  To view $K$-theory as a cohomology theory with supports, I follow the usual practice of defining cohomology groups by exchanging subscripts for superscripts and additively inverting degrees:
\begin{equation}\label{equKhomtocohom}K_X^n:=K_{-n,X},\hspace*{.5cm} K_{X\tn{ \footnotesize{on} } Z}^n:=K_{-n, X\tn{ \footnotesize{on} } Z},\hspace*{.5cm}\tn{and} \hspace*{.5cm}K_{X-Z}^n:=K_{-n, X-Z}.\end{equation}
Let $K=\{K^n\}_{n\in\ZZ}$ be the resulting cohomology theory with supports, and let $X$ in $\mbf{S}_k$ be a smooth noetherian scheme, equidimensional of dimension $n$ over $k$.  Then the $E_1$-level of the coniveau spectral sequence $\{E_{K,X,r}\}$ for $K$ on $X$ is of the form shown in figure \hyperref[figconiveaucohomologicalK]{\ref{figconiveaucohomologicalK}} below:

\begin{figure}[H]
\begin{pgfpicture}{0cm}{0cm}{17cm}{5.5cm}
\begin{pgftranslate}{\pgfpoint{2cm}{0cm}}
\begin{pgfmagnify}{.85}{.85}
\begin{colormixin}{20!white}
\color{black}
\pgfmoveto{\pgfxy(3,3)}
\pgflineto{\pgfxy(15.5,3)}
\pgflineto{\pgfxy(15.5,6)}
\pgflineto{\pgfxy(3,6)}
\pgflineto{\pgfxy(3,3)}
\pgffill
\pgfmoveto{\pgfxy(3,3)}
\pgflineto{\pgfxy(3,0)}
\pgflineto{\pgfxy(0,0)}
\pgflineto{\pgfxy(0,3)}
\pgflineto{\pgfxy(3,3)}
\pgffill
\end{colormixin}
\pgfputat{\pgfxy(.5,1)}{\pgfbox[center,center]{\large{$0$}}}
\pgfputat{\pgfxy(3,.9)}{\pgfbox[center,center]{$\displaystyle\coprod_{x\in \tn{\footnotesize{Zar}}_X^{0}} K_{X\tn{ \footnotesize{on} } x}^{-1}$}}
\pgfputat{\pgfxy(6.5,.9)}{\pgfbox[center,center]{$\displaystyle\coprod_{x\in \tn{\footnotesize{Zar}}_X^{1}} K_{X\tn{ \footnotesize{on} } x}^{0}$}}
\pgfputat{\pgfxy(12.4,.9)}{\pgfbox[center,center]{$\displaystyle\coprod_{x\in \tn{\footnotesize{Zar}}_X^{n}} K_{X\tn{ \footnotesize{on} } x}^{n-1}$}}
\pgfputat{\pgfxy(15,1)}{\pgfbox[center,center]{\large{$0$}}}
\pgfputat{\pgfxy(.5,3)}{\pgfbox[center,center]{\large{$0$}}}
\pgfputat{\pgfxy(3,2.9)}{\pgfbox[center,center]{$\displaystyle\coprod_{x\in \tn{\footnotesize{Zar}}_X^{0}} K_{X\tn{ \footnotesize{on} } x}^{0}$}}
\pgfputat{\pgfxy(6.5,2.9)}{\pgfbox[center,center]{$\displaystyle\coprod_{x\in \tn{\footnotesize{Zar}}_X^{1}} K_{X\tn{ \footnotesize{on} } x}^{1}$}}
\pgfputat{\pgfxy(12.4,2.9)}{\pgfbox[center,center]{$\displaystyle\coprod_{x\in \tn{\footnotesize{Zar}}_X^{n}} K_{X\tn{ \footnotesize{on} } x}^{n}$}}
\pgfputat{\pgfxy(15,3)}{\pgfbox[center,center]{\large{$0$}}}
\pgfputat{\pgfxy(.5,5)}{\pgfbox[center,center]{\large{$0$}}}
\pgfputat{\pgfxy(3,4.9)}{\pgfbox[center,center]{$\displaystyle\coprod_{x\in \tn{\footnotesize{Zar}}_X^{0}} K_{X\tn{ \footnotesize{on} } x}^{1}$}}
\pgfputat{\pgfxy(6.5,4.9)}{\pgfbox[center,center]{$\displaystyle\coprod_{x\in \tn{\footnotesize{Zar}}_X^{1}} K_{X\tn{ \footnotesize{on} } x}^{2}$}}
\pgfputat{\pgfxy(12.4,4.9)}{\pgfbox[center,center]{$\displaystyle\coprod_{x\in \tn{\footnotesize{Zar}}_X^{n}} K_{X\tn{ \footnotesize{on} } x}^{n+1}$}}
\pgfputat{\pgfxy(15,5)}{\pgfbox[center,center]{\large{$0$}}}
\pgfputat{\pgfxy(4.9,1.3)}{\pgfbox[center,center]{\small{$d_1^{0,-1}$}}}
\pgfputat{\pgfxy(8.4,1.3)}{\pgfbox[center,center]{\small{$d_1^{1,-1}$}}}
\pgfputat{\pgfxy(4.9,3.3)}{\pgfbox[center,center]{\small{$d_1^{0,0}$}}}
\pgfputat{\pgfxy(8.4,3.3)}{\pgfbox[center,center]{\small{$d_1^{1,0}$}}}
\pgfputat{\pgfxy(4.9,5.3)}{\pgfbox[center,center]{\small{$d_1^{0,1}$}}}
\pgfputat{\pgfxy(8.4,5.3)}{\pgfbox[center,center]{\small{$d_1^{1,1}$}}}
\pgfsetendarrow{\pgfarrowlargepointed{3pt}}
\pgfxyline(.8,1)(1.7,1)
\pgfxyline(4.4,1)(5.2,1)
\pgfxyline(.8,3)(1.7,3)
\pgfxyline(4.4,3)(5.2,3)
\pgfxyline(.8,5)(1.7,5)
\pgfxyline(4.4,5)(5.2,5)
\pgfxyline(7.9,1)(8.7,1)
\pgfxyline(7.9,3)(8.7,3)
\pgfxyline(7.9,5)(8.7,5)
\pgfxyline(10.3,1)(11.1,1)
\pgfxyline(10.3,3)(11.1,3)
\pgfxyline(10.3,5)(11.1,5)
\pgfxyline(13.8,1)(14.7,1)
\pgfxyline(13.8,3)(14.7,3)
\pgfxyline(13.8,5)(14.7,5)
\pgfnodecircle{Node0}[fill]{\pgfxy(8.9,1)}{0.02cm}
\pgfnodecircle{Node0}[fill]{\pgfxy(9.05,1)}{0.02cm}
\pgfnodecircle{Node0}[fill]{\pgfxy(9.2,1)}{0.02cm}
\pgfnodecircle{Node0}[fill]{\pgfxy(9.8,1)}{0.02cm}
\pgfnodecircle{Node0}[fill]{\pgfxy(9.95,1)}{0.02cm}
\pgfnodecircle{Node0}[fill]{\pgfxy(10.1,1)}{0.02cm}
\pgfnodecircle{Node0}[fill]{\pgfxy(8.9,3)}{0.02cm}
\pgfnodecircle{Node0}[fill]{\pgfxy(9.05,3)}{0.02cm}
\pgfnodecircle{Node0}[fill]{\pgfxy(9.2,3)}{0.02cm}
\pgfnodecircle{Node0}[fill]{\pgfxy(9.8,3)}{0.02cm}
\pgfnodecircle{Node0}[fill]{\pgfxy(9.95,3)}{0.02cm}
\pgfnodecircle{Node0}[fill]{\pgfxy(10.1,3)}{0.02cm}
\pgfnodecircle{Node0}[fill]{\pgfxy(8.9,5)}{0.02cm}
\pgfnodecircle{Node0}[fill]{\pgfxy(9.05,5)}{0.02cm}
\pgfnodecircle{Node0}[fill]{\pgfxy(9.2,5)}{0.02cm}
\pgfnodecircle{Node0}[fill]{\pgfxy(9.8,5)}{0.02cm}
\pgfnodecircle{Node0}[fill]{\pgfxy(9.95,5)}{0.02cm}
\pgfnodecircle{Node0}[fill]{\pgfxy(10.1,5)}{0.02cm}
\end{pgfmagnify}
\end{pgftranslate}
\end{pgfpicture}
\caption{$E_1$-level of the coniveau spectral sequence for cohomological algebraic $K$-theory on a smooth noetherian scheme of dimension $n$ over a field $k$.}
\label{figconiveaucohomologicalK}
\end{figure}

Converting this back to homological $K$-theory, and using the fact that negative $K$-groups vanish on smooth schemes, yields the fourth-quadrant ``trapezoidal diagram" depicted in figure \hyperref[figconiveauhomologicalK]{\ref{figconiveauhomologicalK}} below:

\begin{figure}[H]
\begin{pgfpicture}{0cm}{0cm}{17cm}{11.5cm}
\begin{pgftranslate}{\pgfpoint{1.5cm}{6.5cm}}
\begin{pgfmagnify}{.85}{.85}
\begin{colormixin}{20!white}
\color{black}
\pgfmoveto{\pgfxy(3,3)}
\pgflineto{\pgfxy(17.5,3)}
\pgflineto{\pgfxy(17.5,6)}
\pgflineto{\pgfxy(3,6)}
\pgflineto{\pgfxy(3,3)}
\pgffill
\pgfmoveto{\pgfxy(3,3)}
\pgflineto{\pgfxy(3,-7.5)}
\pgflineto{\pgfxy(-1,-7.5)}
\pgflineto{\pgfxy(-1,3)}
\pgflineto{\pgfxy(3,3)}
\pgffill
\end{colormixin}
\pgfputat{\pgfxy(-.5,-6)}{\pgfbox[center,center]{$0$}}
\pgfputat{\pgfxy(3,-6.1)}{\pgfbox[center,center]{$\displaystyle\coprod_{x\in \tn{\footnotesize{Zar}}_X^{0}} K_{d+1,X\tn{ \footnotesize{on} } x}$}}
\pgfputat{\pgfxy(6.5,-6.1)}{\pgfbox[center,center]{$\displaystyle\coprod_{x\in \tn{\footnotesize{Zar}}_X^{1}} K_{d,X\tn{ \footnotesize{on} } x}$}}
\pgfputat{\pgfxy(13.5,-6.1)}{\pgfbox[center,center]{$\displaystyle\coprod_{x\in \tn{\footnotesize{Zar}}_X^{n}} K_{1,X\tn{ \footnotesize{on} } x}$}}
\pgfputat{\pgfxy(17,-6)}{\pgfbox[center,center]{$0$}}
\pgfputat{\pgfxy(-.5,-4)}{\pgfbox[center,center]{$0$}}
\pgfputat{\pgfxy(3,-4.1)}{\pgfbox[center,center]{$\displaystyle\coprod_{x\in \tn{\footnotesize{Zar}}_X^{0}} K_{d,X\tn{ \footnotesize{on} } x}$}}
\pgfputat{\pgfxy(6.5,-4.1)}{\pgfbox[center,center]{$\displaystyle\coprod_{x\in \tn{\footnotesize{Zar}}_X^{1}} K_{d-1,X\tn{ \footnotesize{on} } x}$}}
\pgfputat{\pgfxy(13.5,-4.1)}{\pgfbox[center,center]{$\displaystyle\coprod_{x\in \tn{\footnotesize{Zar}}_X^{n}} K_{0,X\tn{ \footnotesize{on} } x}$}}
\pgfputat{\pgfxy(17,-4)}{\pgfbox[center,center]{$0$}}
\pgfputat{\pgfxy(-.5,-1)}{\pgfbox[center,center]{$0$}}
\pgfputat{\pgfxy(3,-1.1)}{\pgfbox[center,center]{$\displaystyle\coprod_{x\in \tn{\footnotesize{Zar}}_X^{0}} K_{2,X\tn{ \footnotesize{on} } x}$}}
\pgfputat{\pgfxy(6.5,-1.1)}{\pgfbox[center,center]{$\displaystyle\coprod_{x\in \tn{\footnotesize{Zar}}_X^{1}} K_{1,X\tn{ \footnotesize{on} } x}$}}
\pgfputat{\pgfxy(10,-1.1)}{\pgfbox[center,center]{$\displaystyle\coprod_{x\in \tn{\footnotesize{Zar}}_X^{2}} K_{0,X\tn{ \footnotesize{on} } x}$}}
\pgfputat{\pgfxy(13.5,-1)}{\pgfbox[center,center]{$0$}}
\pgfputat{\pgfxy(17,-1)}{\pgfbox[center,center]{$0$}}
\pgfputat{\pgfxy(-.5,1)}{\pgfbox[center,center]{$0$}}
\pgfputat{\pgfxy(3,.9)}{\pgfbox[center,center]{$\displaystyle\coprod_{x\in \tn{\footnotesize{Zar}}_X^{0}} K_{1,X\tn{ \footnotesize{on} } x}$}}
\pgfputat{\pgfxy(6.5,.9)}{\pgfbox[center,center]{$\displaystyle\coprod_{x\in \tn{\footnotesize{Zar}}_X^{1}} K_{0,X\tn{ \footnotesize{on} } x}$}}
\pgfputat{\pgfxy(13.5,.9)}{\pgfbox[center,center]{$0$}}
\pgfputat{\pgfxy(17,1)}{\pgfbox[center,center]{$0$}}
\pgfputat{\pgfxy(-.5,3)}{\pgfbox[center,center]{$0$}}
\pgfputat{\pgfxy(3,2.9)}{\pgfbox[center,center]{$\displaystyle\coprod_{x\in \tn{\footnotesize{Zar}}_X^{0}} K_{0,X\tn{ \footnotesize{on} } x}$}}
\pgfputat{\pgfxy(6.5,2.9)}{\pgfbox[center,center]{$0$}}
\pgfputat{\pgfxy(13.5,2.9)}{\pgfbox[center,center]{$0$}}
\pgfputat{\pgfxy(17,3)}{\pgfbox[center,center]{$0$}}
\pgfputat{\pgfxy(-.5,5)}{\pgfbox[center,center]{$0$}}
\pgfputat{\pgfxy(3,4.9)}{\pgfbox[center,center]{$0$}}
\pgfputat{\pgfxy(6.5,4.9)}{\pgfbox[center,center]{$0$}}
\pgfputat{\pgfxy(13.5,4.9)}{\pgfbox[center,center]{$0$}}
\pgfputat{\pgfxy(17,5)}{\pgfbox[center,center]{$0$}}
\pgfputat{\pgfxy(10,1)}{\pgfbox[center,center]{$0$}}
\pgfputat{\pgfxy(10,3)}{\pgfbox[center,center]{$0$}}
\pgfputat{\pgfxy(10,5)}{\pgfbox[center,center]{$0$}}
\pgfputat{\pgfxy(4.9,1.3)}{\pgfbox[center,center]{\small{$d_1^{0,-1}$}}}
\pgfputat{\pgfxy(4.9,-.7)}{\pgfbox[center,center]{\small{$d_1^{0,-2}$}}}
\pgfputat{\pgfxy(4.9,-3.7)}{\pgfbox[center,center]{\footnotesize{$d_1^{0,-n}$}}}
\pgfputat{\pgfxy(4.9,-5.7)}{\pgfbox[center,center]{\footnotesize{$d_1^{0,-1-n}$}}}
\pgfputat{\pgfxy(8.4,-.7)}{\pgfbox[center,center]{\small{$d_1^{1,-2}$}}}
\pgfputat{\pgfxy(8.4,-3.7)}{\pgfbox[center,center]{\footnotesize{$d_1^{1,-n}$}}}
\pgfputat{\pgfxy(8.4,-5.7)}{\pgfbox[center,center]{\footnotesize{$d_1^{1,-n-1}$}}}
\pgfputat{\pgfxy(11.8,-3.7)}{\pgfbox[center,center]{\footnotesize{$d_1^{n-1,-n}$}}}
\pgfputat{\pgfxy(11.8,-5.7)}{\pgfbox[center,center]{\footnotesize{$d_1^{n-1,-n-1}$}}}
\pgfsetendarrow{\pgfarrowlargepointed{3pt}}
\pgfxyline(.8,-6)(1.5,-6)
\pgfxyline(4.6,-6)(5.2,-6)
\pgfxyline(8.1,-6)(8.7,-6)
\pgfxyline(11.5,-6)(12.2,-6)
\pgfxyline(15,-6)(15.7,-6)
\pgfxyline(.8,-4)(1.5,-4)
\pgfxyline(4.6,-4)(5.2,-4)
\pgfxyline(8.1,-4)(8.7,-4)
\pgfxyline(11.5,-4)(12.2,-4)
\pgfxyline(15,-4)(15.7,-4)
\pgfxyline(.8,-1)(1.5,-1)
\pgfxyline(4.6,-1)(5.2,-1)
\pgfxyline(8.1,-1)(8.7,-1)
\pgfxyline(11.5,-1)(12.2,-1)
\pgfxyline(15,-1)(15.7,-1)
\pgfxyline(.8,1)(1.5,1)
\pgfxyline(4.6,1)(5.2,1)
\pgfxyline(8.1,1)(8.7,1)
\pgfxyline(11.5,1)(12.2,1)
\pgfxyline(15,1)(15.7,1)
\pgfxyline(.8,3)(1.5,3)
\pgfxyline(4.6,3)(5.2,3)
\pgfxyline(8.1,3)(8.7,3)
\pgfxyline(11.5,3)(12.2,3)
\pgfxyline(15,3)(15.7,3)
\pgfxyline(.8,5)(1.5,5)
\pgfxyline(4.6,5)(5.2,5)
\pgfxyline(8.1,5)(8.7,5)
\pgfxyline(11.5,5)(12.2,5)
\pgfxyline(15,5)(15.7,5)
\pgfnodecircle{Node0}[fill]{\pgfxy(-.5,-2.1)}{0.02cm}
\pgfnodecircle{Node0}[fill]{\pgfxy(-.5,-2)}{0.02cm}
\pgfnodecircle{Node0}[fill]{\pgfxy(-.5,-1.9)}{0.02cm}
\pgfnodecircle{Node0}[fill]{\pgfxy(3,-2.1)}{0.02cm}
\pgfnodecircle{Node0}[fill]{\pgfxy(3,-2)}{0.02cm}
\pgfnodecircle{Node0}[fill]{\pgfxy(3,-1.9)}{0.02cm}
\pgfnodecircle{Node0}[fill]{\pgfxy(6.5,-2.1)}{0.02cm}
\pgfnodecircle{Node0}[fill]{\pgfxy(6.5,-2)}{0.02cm}
\pgfnodecircle{Node0}[fill]{\pgfxy(6.5,-1.9)}{0.02cm}
\pgfnodecircle{Node0}[fill]{\pgfxy(10,-2.1)}{0.02cm}
\pgfnodecircle{Node0}[fill]{\pgfxy(10,-2)}{0.02cm}
\pgfnodecircle{Node0}[fill]{\pgfxy(10,-1.9)}{0.02cm}
\pgfnodecircle{Node0}[fill]{\pgfxy(13.5,-2.1)}{0.02cm}
\pgfnodecircle{Node0}[fill]{\pgfxy(13.5,-2)}{0.02cm}
\pgfnodecircle{Node0}[fill]{\pgfxy(13.5,-1.9)}{0.02cm}
\pgfnodecircle{Node0}[fill]{\pgfxy(17,-2.1)}{0.02cm}
\pgfnodecircle{Node0}[fill]{\pgfxy(17,-2)}{0.02cm}
\pgfnodecircle{Node0}[fill]{\pgfxy(17,-1.9)}{0.02cm}
\pgfnodecircle{Node0}[fill]{\pgfxy(-.5,-3.1)}{0.02cm}
\pgfnodecircle{Node0}[fill]{\pgfxy(-.5,-3)}{0.02cm}
\pgfnodecircle{Node0}[fill]{\pgfxy(-.5,-2.9)}{0.02cm}
\pgfnodecircle{Node0}[fill]{\pgfxy(3,-3.1)}{0.02cm}
\pgfnodecircle{Node0}[fill]{\pgfxy(3,-3)}{0.02cm}
\pgfnodecircle{Node0}[fill]{\pgfxy(3,-2.9)}{0.02cm}
\pgfnodecircle{Node0}[fill]{\pgfxy(6.5,-3.1)}{0.02cm}
\pgfnodecircle{Node0}[fill]{\pgfxy(6.5,-3)}{0.02cm}
\pgfnodecircle{Node0}[fill]{\pgfxy(6.5,-2.9)}{0.02cm}
\pgfnodecircle{Node0}[fill]{\pgfxy(10,-3.1)}{0.02cm}
\pgfnodecircle{Node0}[fill]{\pgfxy(10,-3)}{0.02cm}
\pgfnodecircle{Node0}[fill]{\pgfxy(10,-2.9)}{0.02cm}
\pgfnodecircle{Node0}[fill]{\pgfxy(13.5,-3.1)}{0.02cm}
\pgfnodecircle{Node0}[fill]{\pgfxy(13.5,-3)}{0.02cm}
\pgfnodecircle{Node0}[fill]{\pgfxy(13.5,-2.9)}{0.02cm}
\pgfnodecircle{Node0}[fill]{\pgfxy(17,-3.1)}{0.02cm}
\pgfnodecircle{Node0}[fill]{\pgfxy(17,-3)}{0.02cm}
\pgfnodecircle{Node0}[fill]{\pgfxy(17,-2.9)}{0.02cm}
\pgfnodecircle{Node0}[fill]{\pgfxy(-.5,-7.1)}{0.02cm}
\pgfnodecircle{Node0}[fill]{\pgfxy(-.5,-7)}{0.02cm}
\pgfnodecircle{Node0}[fill]{\pgfxy(-.5,-6.9)}{0.02cm}
\pgfnodecircle{Node0}[fill]{\pgfxy(3,-7.1)}{0.02cm}
\pgfnodecircle{Node0}[fill]{\pgfxy(3,-7)}{0.02cm}
\pgfnodecircle{Node0}[fill]{\pgfxy(3,-6.9)}{0.02cm}
\pgfnodecircle{Node0}[fill]{\pgfxy(6.5,-7.1)}{0.02cm}
\pgfnodecircle{Node0}[fill]{\pgfxy(6.5,-7)}{0.02cm}
\pgfnodecircle{Node0}[fill]{\pgfxy(6.5,-6.9)}{0.02cm}
\pgfnodecircle{Node0}[fill]{\pgfxy(10,-7.1)}{0.02cm}
\pgfnodecircle{Node0}[fill]{\pgfxy(10,-7)}{0.02cm}
\pgfnodecircle{Node0}[fill]{\pgfxy(10,-6.9)}{0.02cm}
\pgfnodecircle{Node0}[fill]{\pgfxy(13.5,-7.1)}{0.02cm}
\pgfnodecircle{Node0}[fill]{\pgfxy(13.5,-7)}{0.02cm}
\pgfnodecircle{Node0}[fill]{\pgfxy(13.5,-6.9)}{0.02cm}
\pgfnodecircle{Node0}[fill]{\pgfxy(17,-7.1)}{0.02cm}
\pgfnodecircle{Node0}[fill]{\pgfxy(17,-7)}{0.02cm}
\pgfnodecircle{Node0}[fill]{\pgfxy(17,-6.9)}{0.02cm}
\pgfnodecircle{Node0}[fill]{\pgfxy(8.9,-4)}{0.02cm}
\pgfnodecircle{Node0}[fill]{\pgfxy(9,-4)}{0.02cm}
\pgfnodecircle{Node0}[fill]{\pgfxy(9.1,-4)}{0.02cm}
\pgfnodecircle{Node0}[fill]{\pgfxy(11.1,-4)}{0.02cm}
\pgfnodecircle{Node0}[fill]{\pgfxy(11.2,-4)}{0.02cm}
\pgfnodecircle{Node0}[fill]{\pgfxy(11.3,-4)}{0.02cm}
\pgfnodecircle{Node0}[fill]{\pgfxy(8.9,-6)}{0.02cm}
\pgfnodecircle{Node0}[fill]{\pgfxy(9,-6)}{0.02cm}
\pgfnodecircle{Node0}[fill]{\pgfxy(9.1,-6)}{0.02cm}
\pgfnodecircle{Node0}[fill]{\pgfxy(11.1,-6)}{0.02cm}
\pgfnodecircle{Node0}[fill]{\pgfxy(11.2,-6)}{0.02cm}
\pgfnodecircle{Node0}[fill]{\pgfxy(11.3,-6)}{0.02cm}
\end{pgfmagnify}
\end{pgftranslate}
\end{pgfpicture}
\caption{$E_1$-level of the coniveau spectral sequence for homological algebraic $K$-theory on a smooth algebraic scheme, incorporating vanishing of negative $K$-groups.}
\label{figconiveauhomologicalK}
\end{figure}

{\bf Existence of the First Column of the Coniveau Machine via Bloch-Ogus.}  By the Bloch-Ogus Theorem \hyperref[corblochogus]{\ref{corblochogus}}, the corresponding sheafified Cousin complexes are flasque resolutions of the algebraic $K$-theory sheaves $\ms{K}_{0,X},\hspace*{.2cm}\ms{K}_{1,X},\hspace*{.2cm}...\hspace*{.2cm},\ms{K}_{n,X},...,$ on $X$.  In the case where $X$ is a smooth algebraic variety, Bloch's formula expresses the Chow groups $\tn{Ch}_X^0$, $\tn{Ch}_X^1$, ..., $\tn{Ch}_X^n$,... as the sheaf cohomology groups
\begin{equation}\label{equsheafcohomsec44}H^0(X,\ms{K}_{0,X}),\hspace*{.2cm}H^1(X,\ms{K}_{1,X}),\hspace*{.2cm}...\hspace*{.2cm},H^n(X,\ms{K}_{n,X}),...\end{equation}
By the definition of sheaf cohomology, these are the groups

\begin{pgfpicture}{0cm}{0cm}{17cm}{2.25cm}
\begin{pgftranslate}{\pgfpoint{0cm}{0cm}}
\pgfputat{\pgfxy(16,1)}{\pgfbox[center,center]{(4.4.1.4)}}
\pgfputat{\pgfxy(2,1)}{\pgfbox[center,center]{$\displaystyle\Gamma\Bigg(\coprod_{x\in \tn{\footnotesize{Zar}}_X^{0}} \underline{K_{0,X\tn{ \footnotesize{on} } x}}\Bigg)$}}
\pgfputat{\pgfxy(3.9,.8)}{\pgfbox[center,center]{,}}
\pgfputat{\pgfxy(8.25,.8)}{\pgfbox[center,center]{,}}
\pgfputat{\pgfxy(8.95,.8)}{\pgfbox[center,center]{...}}
\pgfputat{\pgfxy(9.45,.8)}{\pgfbox[center,center]{,}}
\pgfputat{\pgfxy(13.5,.8)}{\pgfbox[center,center]{,}}
\pgfputat{\pgfxy(14.2,.8)}{\pgfbox[center,center]{...}}
\pgfputat{\pgfxy(6.25,1)}{\pgfbox[center,center]{$\displaystyle\frac{\Gamma\Bigg(\displaystyle\coprod_{x\in \tn{\footnotesize{Zar}}_X^{1}} \underline{K_{0,X\tn{ \footnotesize{on} } x}}\Bigg)}{\tn{Im}\Big(\Gamma\big(\underline{d_1^{0,-1}}\big)\Big)}$}}
\pgfputat{\pgfxy(11.5,1)}{\pgfbox[center,center]{$\displaystyle\frac{\Gamma\Bigg(\displaystyle\coprod_{x\in \tn{\footnotesize{Zar}}_X^{n}} \underline{K_{0,X\tn{ \footnotesize{on} } x}}\Bigg)}{\tn{Im}\Big(\Gamma\big(\underline{d_1^{n-1,-n}}\big)\Big)}$}}
\end{pgftranslate}
\end{pgfpicture}
\label{equ4415}

where $\underline{d_1^{0,-1}}$ is the sheaf map corresponding to $d_1^{0,-1}$ in the sheafified Cousin complex, and $\Gamma\big(\underline{d_1^{0,-1}}\big)$ is the corresponding map on global sections, and similarly for $\underline{d_1^{n-1,-n}}$.  

These results may be rephrased in the context of the coniveau machine as follows:

\begin{lem}\label{lemfirstcolumn}Let $\mbf{S}_k$ be a distinguished category of schemes over a field $k$, satisfying the conditions given at the beginning of section \hyperref[subsectioncohomsupportssubstrata]{\ref{subsectioncohomsupportssubstrata}} above, and let $X$ in $\mbf{S}_k$ be  smooth, equidimensional, and noetherian.  Then the first column of the coniveau machine for Bass-Thomason $K$-theory on $X$ exists; that is, the Cousin complex appearing as the $p$th row of the $E_1$-level of the coniveau spectral sequence for $K$ on $X$ sheafifies to yield a flasque resolution of the $K$-theory sheaf $\ms{K}_{p,X}$ on $X$ for all $p$.  
\end{lem}
\begin{proof}This is established by the foregoing discussion; the crucial ingredient is the Bloch-Ogus theorem.
\end{proof}


{\bf ``Column" Terminology.}  In both the introduction to this book \hyperref[subsectionconiveathisthesis]{\ref{subsectionconiveathisthesis}} and the beginning of the present chapter \hyperref[sectionintroductionconiveau]{\ref{sectionintroductionconiveau}}, I described verbally how the ``simplified four-column version" of the coniveau machine for Bass-Thomason algebraic $K$-theory $K$ on a smooth algebraic variety over a field $k$, augmented via a nilpotent thickening, is constructed by isolating and sheafifying certain rows of the coniveau spectral sequence for the absolute, augmented, relative, and relative additive versions of $K$ on $X$, then transposing them to form the four columns of the machine.  Here, I spell this procedure out in more detail for the first column of the machine. 

As indicated in equations \hyperref[equsheafcohomsec44]{\ref{equsheafcohomsec44}} and \hyperref[equ4415]{4.4.1.4} above, the Chow group $\tn{Ch}_X^p$ is the $p$th sheaf cohomology group of the sheaf $\ms{K}_{p,X}$, computed by applying the global sections functor $\Gamma$ to the $p$th sheafified Cousin complex, i.e., the sheaf complex obtained by sheafifying the $-p$th row of the coniveau spectral sequence for $K$ on $X$:

\begin{pgfpicture}{0cm}{0cm}{17cm}{2cm}
\begin{pgftranslate}{\pgfpoint{2.5cm}{0cm}}
\pgfputat{\pgfxy(13.4,1)}{\pgfbox[center,center]{(4.4.1.6)}}
\pgfputat{\pgfxy(.7,1)}{\pgfbox[center,center]{$0$}}
\pgfputat{\pgfxy(3,.9)}{\pgfbox[center,center]{$\displaystyle\coprod_{x\in \tn{\footnotesize{Zar}}_X^{0}} \underline{K_{p,X\tn{ \footnotesize{on} } x}}$}}
\pgfputat{\pgfxy(7,.9)}{\pgfbox[center,center]{$\displaystyle\coprod_{x\in \tn{\footnotesize{Zar}}_X^{1}} \underline{K_{p-1,X\tn{ \footnotesize{on} } x}}$}}
\pgfputat{\pgfxy(5.1,1.3)}{\pgfbox[center,center]{\small{$d_1^{0,-p}$}}}
\pgfputat{\pgfxy(9.3,1.3)}{\pgfbox[center,center]{\small{$d_1^{1,-p}$}}}
\pgfsetendarrow{\pgfarrowlargepointed{3pt}}
\pgfxyline(.9,1)(1.7,1)
\pgfxyline(4.5,1)(5.4,1)
\pgfxyline(8.9,1)(9.7,1)
\pgfnodecircle{Node0}[fill]{\pgfxy(9.9,1)}{0.02cm}
\pgfnodecircle{Node0}[fill]{\pgfxy(10,1)}{0.02cm}
\pgfnodecircle{Node0}[fill]{\pgfxy(10.1,1)}{0.02cm}
\end{pgftranslate}
\end{pgfpicture}

In particular, when $p=n$, the dimension of $X$, the sheafified Cousin complex is 

\begin{pgfpicture}{0cm}{0cm}{17cm}{2cm}
\begin{pgftranslate}{\pgfpoint{0cm}{0cm}}
\pgfputat{\pgfxy(15.9,1)}{\pgfbox[center,center]{(4.4.1.7)}}
\pgfputat{\pgfxy(.5,1)}{\pgfbox[center,center]{$0$}}
\pgfputat{\pgfxy(2.8,.9)}{\pgfbox[center,center]{$\displaystyle\coprod_{x\in \tn{\footnotesize{Zar}}_X^{0}} \underline{K_{n,X\tn{ \footnotesize{on} } x}}$}}
\pgfputat{\pgfxy(6.5,.9)}{\pgfbox[center,center]{$\displaystyle\coprod_{x\in \tn{\footnotesize{Zar}}_X^{1}} \underline{K_{n-1,X\tn{ \footnotesize{on} } x}}$}}
\pgfputat{\pgfxy(12,.9)}{\pgfbox[center,center]{$\displaystyle\coprod_{x\in \tn{\footnotesize{Zar}}_X^{d}} \underline{K_{0,X\tn{ \footnotesize{on} } x}}$}}
\pgfputat{\pgfxy(14.5,1)}{\pgfbox[center,center]{$0$,}}
\pgfputat{\pgfxy(4.7,1.3)}{\pgfbox[center,center]{\footnotesize{$d_1^{0,-n}$}}}
\pgfputat{\pgfxy(8.6,1.3)}{\pgfbox[center,center]{\footnotesize{$d_1^{1,-n}$}}}
\pgfputat{\pgfxy(10.4,1.3)}{\pgfbox[center,center]{\footnotesize{$d_1^{n-1,-n}$}}}
\pgfsetendarrow{\pgfarrowlargepointed{3pt}}
\pgfxyline(.8,1)(1.5,1)
\pgfxyline(4.3,1)(5,1)
\pgfxyline(8.2,1)(8.8,1)
\pgfxyline(9.9,1)(10.7,1)
\pgfxyline(13.6,1)(14.2,1)
\pgfnodecircle{Node0}[fill]{\pgfxy(9,1)}{0.02cm}
\pgfnodecircle{Node0}[fill]{\pgfxy(9.1,1)}{0.02cm}
\pgfnodecircle{Node0}[fill]{\pgfxy(9.2,1)}{0.02cm}
\pgfnodecircle{Node0}[fill]{\pgfxy(9.5,1)}{0.02cm}
\pgfnodecircle{Node0}[fill]{\pgfxy(9.6,1)}{0.02cm}
\pgfnodecircle{Node0}[fill]{\pgfxy(9.7,1)}{0.02cm}
\end{pgftranslate}
\end{pgfpicture}

so that the $n$th Chow group involves $K_0$.  

\begin{example}\label{exampletoytheorycurves} \tn{In chapter \hyperref[chaptercurves]{\ref{chaptercurves}}, in the ``toy version of the theory" for smooth complex projective algebraic curves, this particular complex is just the sheaf divisor sequence appearing in equation \hyperref[equsheafdivisorsequencecurves]{\ref{equsheafdivisorsequencecurves}}:
\[1\rightarrow\ms{O}_X^*\overset{i}{\longrightarrow}\ms{R}_X^*\overset{\ms{D}\tn{\footnotesize{iv}}_X}{\longrightarrow}\ms{Z}_X^1\rightarrow0,\]
where the leftmost ``$1$" in the sequence reflects the fact that ``addition" in the first graded piece $\ms{K}_{1,X}\cong \ms{O}_X^*$ of the sheaf of rings $\ms{K}_X$ corresponds to ``multiplication" in $\ms{O}_X^*$.  This sequence forms the ``first column of the coniveau machine for curves" appearing in figure \hyperref[figcompletedconiveaucurve]{\ref{figcompletedconiveaucurve}}.   This construction ignores the other Cousin complexes appearing in the coniveau spectral sequence for $K$ on $X$, so the development of Chapter  \hyperref[chaptercurves]{\ref{chaptercurves}} reveals only part of the structural picture, even in the simple case of curves.}
\end{example}
\hspace{16.3cm} $\oblong$

Figure \hyperref[figsheafifiedtrapezoid]{\ref{figsheafifiedtrapezoid}} below illustrates how the first column of the ``simplified four-column version" of the coniveau machine fits into the trapezoidal diagram arising from the coniveau spectral sequence for $K$ on $X$, first shown in figure \hyperref[figconiveauhomologicalK]{\ref{figconiveauhomologicalK}} above.  Here, I have sheafified the entire trapezoidal diagram.  The boxed row, when transposed, is the first column of the coniveau machine.

\begin{figure}[H]
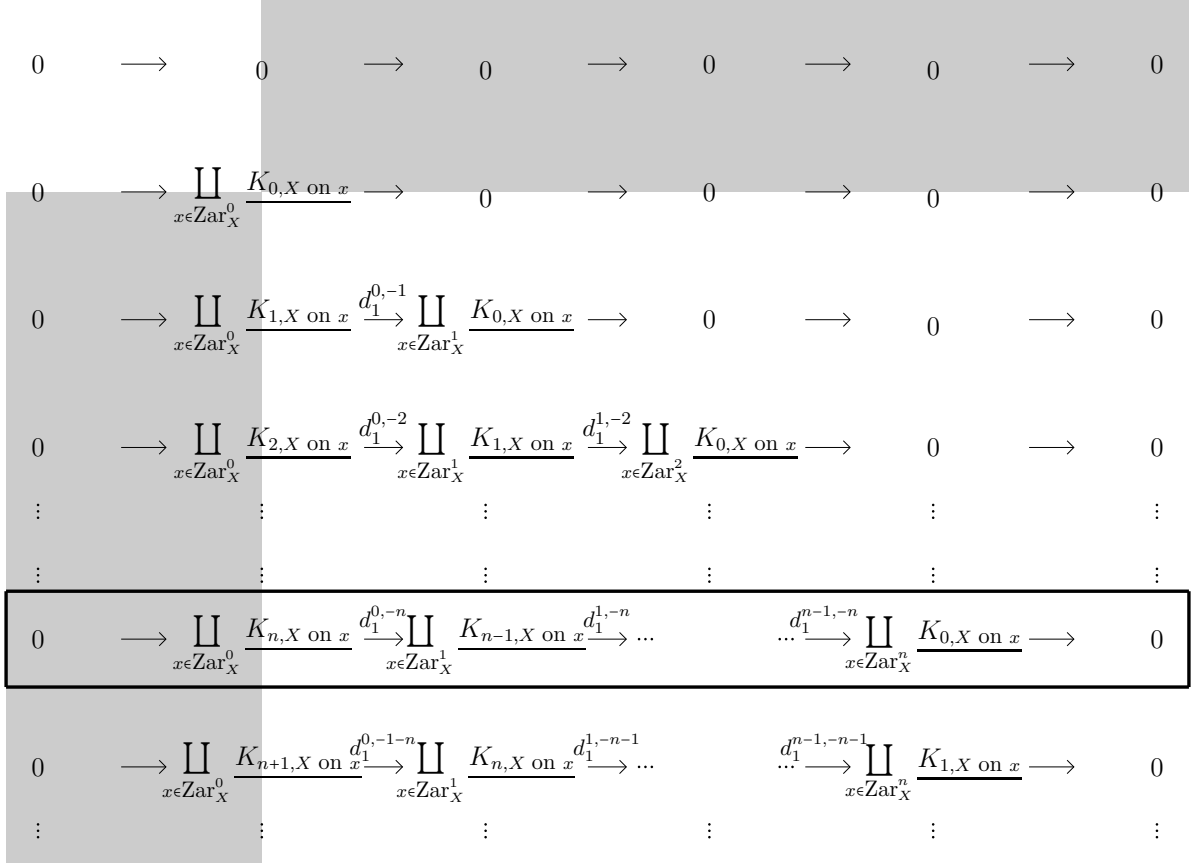

\begin{pgfpicture}{0cm}{0cm}{17cm}{11.5cm}
\begin{pgftranslate}{\pgfpoint{1.5cm}{6.5cm}}
\begin{pgfmagnify}{.85}{.85}
\begin{colormixin}{20!white}
\color{black}
\pgfmoveto{\pgfxy(3,3)}
\pgflineto{\pgfxy(17.5,3)}
\pgflineto{\pgfxy(17.5,6)}
\pgflineto{\pgfxy(3,6)}
\pgflineto{\pgfxy(3,3)}
\pgffill
\pgfmoveto{\pgfxy(3,3)}
\pgflineto{\pgfxy(3,-7.5)}
\pgflineto{\pgfxy(-1,-7.5)}
\pgflineto{\pgfxy(-1,3)}
\pgflineto{\pgfxy(3,3)}
\pgffill
\end{colormixin}
\pgfputat{\pgfxy(-.5,-6)}{\pgfbox[center,center]{$0$}}
\pgfputat{\pgfxy(3,-6.1)}{\pgfbox[center,center]{$\displaystyle\coprod_{x\in \tn{\footnotesize{Zar}}_X^{0}} \underline{K_{n+1,X\tn{ \footnotesize{on} } x}}$}}
\pgfputat{\pgfxy(6.5,-6.1)}{\pgfbox[center,center]{$\displaystyle\coprod_{x\in \tn{\footnotesize{Zar}}_X^{1}} \underline{K_{n,X\tn{ \footnotesize{on} } x}}$}}
\pgfputat{\pgfxy(13.5,-6.1)}{\pgfbox[center,center]{$\displaystyle\coprod_{x\in \tn{\footnotesize{Zar}}_X^{n}} \underline{K_{1,X\tn{ \footnotesize{on} } x}}$}}
\pgfputat{\pgfxy(17,-6)}{\pgfbox[center,center]{$0$}}
\pgfputat{\pgfxy(-.5,-4)}{\pgfbox[center,center]{$0$}}
\pgfputat{\pgfxy(3,-4.1)}{\pgfbox[center,center]{$\displaystyle\coprod_{x\in \tn{\footnotesize{Zar}}_X^{0}} \underline{K_{n,X\tn{ \footnotesize{on} } x}}$}}
\pgfputat{\pgfxy(6.5,-4.1)}{\pgfbox[center,center]{$\displaystyle\coprod_{x\in \tn{\footnotesize{Zar}}_X^{1}} \underline{K_{n-1,X\tn{ \footnotesize{on} } x}}$}}
\pgfputat{\pgfxy(13.5,-4.1)}{\pgfbox[center,center]{$\displaystyle\coprod_{x\in \tn{\footnotesize{Zar}}_X^{n}} \underline{K_{0,X\tn{ \footnotesize{on} } x}}$}}
\pgfputat{\pgfxy(17,-4)}{\pgfbox[center,center]{$0$}}
\pgfputat{\pgfxy(-.5,-1)}{\pgfbox[center,center]{$0$}}
\pgfputat{\pgfxy(3,-1.1)}{\pgfbox[center,center]{$\displaystyle\coprod_{x\in \tn{\footnotesize{Zar}}_X^{0}} \underline{K_{2,X\tn{ \footnotesize{on} } x}}$}}
\pgfputat{\pgfxy(6.5,-1.1)}{\pgfbox[center,center]{$\displaystyle\coprod_{x\in \tn{\footnotesize{Zar}}_X^{1}} \underline{K_{1,X\tn{ \footnotesize{on} } x}}$}}
\pgfputat{\pgfxy(10,-1.1)}{\pgfbox[center,center]{$\displaystyle\coprod_{x\in \tn{\footnotesize{Zar}}_X^{2}} \underline{K_{0,X\tn{ \footnotesize{on} } x}}$}}
\pgfputat{\pgfxy(13.5,-1)}{\pgfbox[center,center]{$0$}}
\pgfputat{\pgfxy(17,-1)}{\pgfbox[center,center]{$0$}}
\pgfputat{\pgfxy(-.5,1)}{\pgfbox[center,center]{$0$}}
\pgfputat{\pgfxy(3,.9)}{\pgfbox[center,center]{$\displaystyle\coprod_{x\in \tn{\footnotesize{Zar}}_X^{0}} \underline{K_{1,X\tn{ \footnotesize{on} } x}}$}}
\pgfputat{\pgfxy(6.5,.9)}{\pgfbox[center,center]{$\displaystyle\coprod_{x\in \tn{\footnotesize{Zar}}_X^{1}} \underline{K_{0,X\tn{ \footnotesize{on} } x}}$}}
\pgfputat{\pgfxy(13.5,.9)}{\pgfbox[center,center]{$0$}}
\pgfputat{\pgfxy(17,1)}{\pgfbox[center,center]{$0$}}
\pgfputat{\pgfxy(-.5,3)}{\pgfbox[center,center]{$0$}}
\pgfputat{\pgfxy(3,2.9)}{\pgfbox[center,center]{$\displaystyle\coprod_{x\in \tn{\footnotesize{Zar}}_X^{0}} \underline{K_{0,X\tn{ \footnotesize{on} } x}}$}}
\pgfputat{\pgfxy(6.5,2.9)}{\pgfbox[center,center]{$0$}}
\pgfputat{\pgfxy(13.5,2.9)}{\pgfbox[center,center]{$0$}}
\pgfputat{\pgfxy(17,3)}{\pgfbox[center,center]{$0$}}
\pgfputat{\pgfxy(-.5,5)}{\pgfbox[center,center]{$0$}}
\pgfputat{\pgfxy(3,4.9)}{\pgfbox[center,center]{$0$}}
\pgfputat{\pgfxy(6.5,4.9)}{\pgfbox[center,center]{$0$}}
\pgfputat{\pgfxy(13.5,4.9)}{\pgfbox[center,center]{$0$}}
\pgfputat{\pgfxy(17,5)}{\pgfbox[center,center]{$0$}}
\pgfputat{\pgfxy(10,1)}{\pgfbox[center,center]{$0$}}
\pgfputat{\pgfxy(10,3)}{\pgfbox[center,center]{$0$}}
\pgfputat{\pgfxy(10,5)}{\pgfbox[center,center]{$0$}}
\pgfputat{\pgfxy(4.9,1.3)}{\pgfbox[center,center]{\small{$d_1^{0,-1}$}}}
\pgfputat{\pgfxy(4.9,-.7)}{\pgfbox[center,center]{\small{$d_1^{0,-2}$}}}
\pgfputat{\pgfxy(4.9,-3.7)}{\pgfbox[center,center]{\footnotesize{$d_1^{0,-n}$}}}
\pgfputat{\pgfxy(4.9,-5.7)}{\pgfbox[center,center]{\footnotesize{$d_1^{0,-1-n}$}}}
\pgfputat{\pgfxy(8.4,-.7)}{\pgfbox[center,center]{\small{$d_1^{1,-2}$}}}
\pgfputat{\pgfxy(8.4,-3.7)}{\pgfbox[center,center]{\footnotesize{$d_1^{1,-n}$}}}
\pgfputat{\pgfxy(8.4,-5.7)}{\pgfbox[center,center]{\footnotesize{$d_1^{1,-n-1}$}}}
\pgfputat{\pgfxy(11.8,-3.7)}{\pgfbox[center,center]{\footnotesize{$d_1^{n-1,-n}$}}}
\pgfputat{\pgfxy(11.8,-5.7)}{\pgfbox[center,center]{\footnotesize{$d_1^{n-1,-n-1}$}}}
\pgfsetendarrow{\pgfarrowlargepointed{3pt}}
\pgfxyline(.8,-6)(1.5,-6)
\pgfxyline(4.6,-6)(5.2,-6)
\pgfxyline(8.1,-6)(8.7,-6)
\pgfxyline(11.5,-6)(12.2,-6)
\pgfxyline(15,-6)(15.7,-6)
\pgfxyline(.8,-4)(1.5,-4)
\pgfxyline(4.6,-4)(5.2,-4)
\pgfxyline(8.1,-4)(8.7,-4)
\pgfxyline(11.5,-4)(12.2,-4)
\pgfxyline(15,-4)(15.7,-4)
\pgfxyline(.8,-1)(1.5,-1)
\pgfxyline(4.6,-1)(5.2,-1)
\pgfxyline(8.1,-1)(8.7,-1)
\pgfxyline(11.5,-1)(12.2,-1)
\pgfxyline(15,-1)(15.7,-1)
\pgfxyline(.8,1)(1.5,1)
\pgfxyline(4.6,1)(5.2,1)
\pgfxyline(8.1,1)(8.7,1)
\pgfxyline(11.5,1)(12.2,1)
\pgfxyline(15,1)(15.7,1)
\pgfxyline(.8,3)(1.5,3)
\pgfxyline(4.6,3)(5.2,3)
\pgfxyline(8.1,3)(8.7,3)
\pgfxyline(11.5,3)(12.2,3)
\pgfxyline(15,3)(15.7,3)
\pgfxyline(.8,5)(1.5,5)
\pgfxyline(4.6,5)(5.2,5)
\pgfxyline(8.1,5)(8.7,5)
\pgfxyline(11.5,5)(12.2,5)
\pgfxyline(15,5)(15.7,5)
\pgfnodecircle{Node0}[fill]{\pgfxy(-.5,-2.1)}{0.02cm}
\pgfnodecircle{Node0}[fill]{\pgfxy(-.5,-2)}{0.02cm}
\pgfnodecircle{Node0}[fill]{\pgfxy(-.5,-1.9)}{0.02cm}
\pgfnodecircle{Node0}[fill]{\pgfxy(3,-2.1)}{0.02cm}
\pgfnodecircle{Node0}[fill]{\pgfxy(3,-2)}{0.02cm}
\pgfnodecircle{Node0}[fill]{\pgfxy(3,-1.9)}{0.02cm}
\pgfnodecircle{Node0}[fill]{\pgfxy(6.5,-2.1)}{0.02cm}
\pgfnodecircle{Node0}[fill]{\pgfxy(6.5,-2)}{0.02cm}
\pgfnodecircle{Node0}[fill]{\pgfxy(6.5,-1.9)}{0.02cm}
\pgfnodecircle{Node0}[fill]{\pgfxy(10,-2.1)}{0.02cm}
\pgfnodecircle{Node0}[fill]{\pgfxy(10,-2)}{0.02cm}
\pgfnodecircle{Node0}[fill]{\pgfxy(10,-1.9)}{0.02cm}
\pgfnodecircle{Node0}[fill]{\pgfxy(13.5,-2.1)}{0.02cm}
\pgfnodecircle{Node0}[fill]{\pgfxy(13.5,-2)}{0.02cm}
\pgfnodecircle{Node0}[fill]{\pgfxy(13.5,-1.9)}{0.02cm}
\pgfnodecircle{Node0}[fill]{\pgfxy(17,-2.1)}{0.02cm}
\pgfnodecircle{Node0}[fill]{\pgfxy(17,-2)}{0.02cm}
\pgfnodecircle{Node0}[fill]{\pgfxy(17,-1.9)}{0.02cm}
\pgfnodecircle{Node0}[fill]{\pgfxy(-.5,-3.1)}{0.02cm}
\pgfnodecircle{Node0}[fill]{\pgfxy(-.5,-3)}{0.02cm}
\pgfnodecircle{Node0}[fill]{\pgfxy(-.5,-2.9)}{0.02cm}
\pgfnodecircle{Node0}[fill]{\pgfxy(3,-3.1)}{0.02cm}
\pgfnodecircle{Node0}[fill]{\pgfxy(3,-3)}{0.02cm}
\pgfnodecircle{Node0}[fill]{\pgfxy(3,-2.9)}{0.02cm}
\pgfnodecircle{Node0}[fill]{\pgfxy(6.5,-3.1)}{0.02cm}
\pgfnodecircle{Node0}[fill]{\pgfxy(6.5,-3)}{0.02cm}
\pgfnodecircle{Node0}[fill]{\pgfxy(6.5,-2.9)}{0.02cm}
\pgfnodecircle{Node0}[fill]{\pgfxy(10,-3.1)}{0.02cm}
\pgfnodecircle{Node0}[fill]{\pgfxy(10,-3)}{0.02cm}
\pgfnodecircle{Node0}[fill]{\pgfxy(10,-2.9)}{0.02cm}
\pgfnodecircle{Node0}[fill]{\pgfxy(13.5,-3.1)}{0.02cm}
\pgfnodecircle{Node0}[fill]{\pgfxy(13.5,-3)}{0.02cm}
\pgfnodecircle{Node0}[fill]{\pgfxy(13.5,-2.9)}{0.02cm}
\pgfnodecircle{Node0}[fill]{\pgfxy(17,-3.1)}{0.02cm}
\pgfnodecircle{Node0}[fill]{\pgfxy(17,-3)}{0.02cm}
\pgfnodecircle{Node0}[fill]{\pgfxy(17,-2.9)}{0.02cm}
\pgfnodecircle{Node0}[fill]{\pgfxy(-.5,-7.1)}{0.02cm}
\pgfnodecircle{Node0}[fill]{\pgfxy(-.5,-7)}{0.02cm}
\pgfnodecircle{Node0}[fill]{\pgfxy(-.5,-6.9)}{0.02cm}
\pgfnodecircle{Node0}[fill]{\pgfxy(3,-7.1)}{0.02cm}
\pgfnodecircle{Node0}[fill]{\pgfxy(3,-7)}{0.02cm}
\pgfnodecircle{Node0}[fill]{\pgfxy(3,-6.9)}{0.02cm}
\pgfnodecircle{Node0}[fill]{\pgfxy(6.5,-7.1)}{0.02cm}
\pgfnodecircle{Node0}[fill]{\pgfxy(6.5,-7)}{0.02cm}
\pgfnodecircle{Node0}[fill]{\pgfxy(6.5,-6.9)}{0.02cm}
\pgfnodecircle{Node0}[fill]{\pgfxy(10,-7.1)}{0.02cm}
\pgfnodecircle{Node0}[fill]{\pgfxy(10,-7)}{0.02cm}
\pgfnodecircle{Node0}[fill]{\pgfxy(10,-6.9)}{0.02cm}
\pgfnodecircle{Node0}[fill]{\pgfxy(13.5,-7.1)}{0.02cm}
\pgfnodecircle{Node0}[fill]{\pgfxy(13.5,-7)}{0.02cm}
\pgfnodecircle{Node0}[fill]{\pgfxy(13.5,-6.9)}{0.02cm}
\pgfnodecircle{Node0}[fill]{\pgfxy(17,-7.1)}{0.02cm}
\pgfnodecircle{Node0}[fill]{\pgfxy(17,-7)}{0.02cm}
\pgfnodecircle{Node0}[fill]{\pgfxy(17,-6.9)}{0.02cm}
\pgfnodecircle{Node0}[fill]{\pgfxy(8.9,-4)}{0.02cm}
\pgfnodecircle{Node0}[fill]{\pgfxy(9,-4)}{0.02cm}
\pgfnodecircle{Node0}[fill]{\pgfxy(9.1,-4)}{0.02cm}
\pgfnodecircle{Node0}[fill]{\pgfxy(11.1,-4)}{0.02cm}
\pgfnodecircle{Node0}[fill]{\pgfxy(11.2,-4)}{0.02cm}
\pgfnodecircle{Node0}[fill]{\pgfxy(11.3,-4)}{0.02cm}
\pgfnodecircle{Node0}[fill]{\pgfxy(8.9,-6)}{0.02cm}
\pgfnodecircle{Node0}[fill]{\pgfxy(9,-6)}{0.02cm}
\pgfnodecircle{Node0}[fill]{\pgfxy(9.1,-6)}{0.02cm}
\pgfnodecircle{Node0}[fill]{\pgfxy(11.1,-6)}{0.02cm}
\pgfnodecircle{Node0}[fill]{\pgfxy(11.2,-6)}{0.02cm}
\pgfnodecircle{Node0}[fill]{\pgfxy(11.3,-6)}{0.02cm}
\pgfclearendarrow
\pgfsetlinewidth{1.5pt}
\pgfxyline(-1,-4.75)(17.5,-4.75)
\pgfxyline(-1,-4.75)(-1,-3.25)
\pgfxyline(17.5,-4.75)(17.5,-3.25)
\pgfxyline(-1,-3.25)(17.5,-3.25)
\end{pgfmagnify}
\end{pgftranslate}
\end{pgfpicture}
\caption{Sheafified version of the trapezoidal diagram in figure \hyperref[figconiveauhomologicalK]{\ref{figconiveauhomologicalK}} above.  Box indicates first column of the coniveau machine.}
\label{figsheafifiedtrapezoid}
\end{figure}

{\bf Operations along the Diagonal.}  Note that the intersection of cycles on open subsets of $X$ induces operations along the diagonal of total degree zero in the trapezoidal diagrams appearing in figures \hyperref[figconiveauhomologicalK]{\ref{figconiveauhomologicalK}} and \hyperref[figsheafifiedtrapezoid]{\ref{figsheafifiedtrapezoid}}.  These operations correspond to multiplication in the Chow ring $\tn{Ch}_X$ of $X$.  It would be interesting to look for similar operations among the sheafified $E_1^{p,-p}$-terms of the coniveau spectral sequences for other cohomology theories with supports on suitable categories of topological spaces. 

\subsection{``New Theories Out of Old;" Second and Third Columns}\label{subsectionnewoutofold}

{\bf Viewpoints Concerning Augmentation.}  The existence of the first column of the coniveau machine for algebraic $K$-theory on a smooth algebraic variety over a field $k$, established in lemma \hyperref[lemfirstcolumn]{\ref{lemfirstcolumn}} above, requires a smoothness hypotheses for the scheme $X$.  This hypothesis derives from the smoothness hypothesis in the Bloch-Ogus theorem \hyperref[corblochogus]{\ref{corblochogus}}.   Turning to the second column of the machine, as depicted in figure \hyperref[figsimplifiedfourcolumnKtheorych4b]{\ref{figsimplifiedfourcolumnKtheorych4b}}, this smoothness hypothesis can no longer be invoked, since the second column involves the thickened scheme $X\times_kY$, which is singular.  

As mentioned in section \hyperref[subsectiongeneralities]{\ref{subsectiongeneralities}} at the beginning of this chapter, the ``augmentation" involved in constructing the second column of the coniveau machine may be viewed in at least two different ways: one may either think of ``augmenting the cohomology theory," or ``augmenting the scheme."  Technically, there is no difference between the two choices, since the ``augmented cohomology theory" is defined by applying the original cohomology theory to the ``augmented scheme" $X\times_kY$.  However, ``augmenting the cohomology theory" is the preferred viewpoint in the present context, since the objects of principal interest in this book are the cycle groups and Chow groups of $X$ in the case where $X$ is a smooth algebraic variety.  Stated differently, although the thickened scheme $X\times_kY$ is singular, its singularity is of a very ``uniform" type.  Hence, attempting a wholesale extension of the theory to the case of singular schemes is not really relevant to the construction of the second column.

{\bf Augmented Cohomology Theories with Supports and Substrata.}  Hence, I will describe the second column of the coniveau machine in terms of ``augmented cohomology," viewpoint, following Colliot-Th\'el\`ene, Hoobler, and Kahn's program \cite{CHKBloch-Ogus-Gabber97} of building ``new theories out of old."  In this case, the ``new theory" is the algebraic $K$-theory of $X\times _kY$, {\it viewed as a cohomology theory with supports on the original scheme} $X$, as explained at the end of section \hyperref[subsectionconnectivenonconnective]{\ref{subsectionconnectivenonconnective}} above.  More generally, I use the following definition, adapted from Colliot-Th\'el\`ene, Hoobler, and Kahn \cite{CHKBloch-Ogus-Gabber97} 5.5:

\begin{defi}\label{definewoutofold} Let $\mbf{S}_k$ a category of schemes over a field $k$ satisfying the conditions given at the beginning of section \hyperref[subsectioncohomsupportssubstrata]{\ref{subsectioncohomsupportssubstrata}}.   Let $Y$ be a separated scheme over $k$. 
\begin{enumerate}
\item Let $H$ be a cohomology theory with supports on the category of pairs over $\mbf{S}_k$, with values in an abelian category $\mbf{A}$.    Define a family of functors $H^{Y}={H_n^Y}_{n\in\ZZ}$ on the category of pairs over $\mbf{S}_k$ as follows:
\begin{equation}\label{equnewoutofoldcohom}H_{n,X\tn{ on } Z}^Y:=H_{n,X\times_k Y\tn{ on } Z\times_k Y }.\end{equation}
\item Let $C$ be a substratum with supports on the category of pairs over $\mbf{S}_k$, with values in an abelian category $\mbf{A}$.   Let $Y$ be a separated scheme over $k$.  Define a functor $C^{Y}$ on $\mbf{S}_k$ as follows:
\begin{equation}\label{equnewoutofoldsubstrat}C_X^Y:=C_{X\times_k Y}.\end{equation}
\end{enumerate}
\end{defi}

For this construction to work, it is necessary to verify that the ``augmented" functors  $H_n^Y$ and $C^Y$ satisfy the conditions to define  cohomology theories with supports and substrata, respectively, and that the construction preserves effaceability.  This is verified by the following lemma, also adapted from Colliot-Th\'el\`ene, Hoobler, and Kahn \cite{CHKBloch-Ogus-Gabber97} 5.5:

\begin{lem}\label{lemnewoutofold} Let $\mbf{S}_k$ a category of schemes over a field $k$ satisfying the conditions given at the beginning of section \hyperref[subsectioncohomsupportssubstrata]{\ref{subsectioncohomsupportssubstrata}}.   Let $Y$ be a separated scheme over $k$.
\begin{enumerate}
\item  $H^Y$ is a cohomology theory with supports on  the category of pairs over $\mbf{S}_k$.  If $H$ satisfies \'etale excision \tn{({\bf COH1})}, and the cohomological projective bundle formula \tn{({\bf COH5})}, then so does $H^Y$.  
\item $C^Y$ is a substratum on $\mbf{S}_k$.  If $C$ satisfies the \'etale Mayer-Vietoris property \tn{({\bf SUB1})}, and the projective bundle formula for substrata \tn{({\bf SUB5})}, then so does $C^Y$.  
\end{enumerate}
\end{lem}
\begin{proof} See Colliot-Th\'el\`ene, Hoobler, and Kahn \cite{CHKBloch-Ogus-Gabber97} 5.5. 
\end{proof}


{\bf Second and Third Columns of the Machine.}  Applying lemma \hyperref[lemnewoutofold]{\ref{lemnewoutofold}} to Bass-Thomason $K$-theory leads to the following result:

 \begin{lem}\label{lemsecondcolumn} Let $\mbf{S}_k$ be a distinguished category of schemes over a field $k$, satisfying the conditions given at the beginning of section \hyperref[subsectioncohomsupportssubstrata]{\ref{subsectioncohomsupportssubstrata}} above, and let $X$ in $\mbf{S}_k$ be  smooth, equidimensional, and noetherian. Let $Y$ be a separated scheme over $k$.  Then the second column of the coniveau machine for Bass-Thomason $K$-theory on $X$ exists; that is, the Cousin complexes appearing as the rows of the $E_1$-level of the coniveau spectral sequence for augmented $K$-theory on $X$ with respect to $Y$ sheafify to yield flasque resolutions of the sheaves $\ms{K}_{p,X\times_kY}$ on $X$.
\end{lem}
\begin{proof}This is established by the foregoing discussion. 
\end{proof}

Recall that the sheaves $\ms{K}_{p,X\times_kY}$ are defined in terms of the augmented $K$-theory spectrum $\mbf{K}_{X\times_kY}$ of $X$ with respect to $Y$, defined in definition \hyperref[augmentedKspectrum]{\ref{augmentedKspectrum}} above.  

When $X$ is a smooth algebraic variety, lemma \hyperref[lemsecondcolumn]{\ref{lemsecondcolumn}} enables computation of the generalized deformation groups $D_Y\tn{Ch}_X^p$, defined in definition \hyperref[defigendefgroupChow]{\ref{defigendefgroupChow}} above, in terms of the sheafified Cousin complexes arising from the spectral sequence for augmented $K$-theory.   Note that the augmented $K$-theory groups in negative degrees are generally nontrivial here, since they are ultimately $K$-groups of singular schemes. 

In the case where $X\mapsto X\times_kY$ is a nilpotent thickening, the image of $X$ under the functor $E_{H_{\tn{\fsz{rel}}},\mbf{S}}$ appearing in the generalized form of the coniveau machine is just the kernel of the canonical split surjection of spectral sequences induced by the splitting of the natural transformation $i$:

\begin{pgfpicture}{0cm}{0cm}{17cm}{2.6cm}
\begin{pgftranslate}{\pgfpoint{3.5cm}{-.5cm}}
\begin{pgfmagnify}{1.1}{1.1}
\pgfputat{\pgfxy(1.3,2.2)}{\pgfbox[center,center]{$E_{H_{\tn{\fsz{rel}}},\mbf{S}}$}}
\pgfputat{\pgfxy(4,2.2)}{\pgfbox[center,center]{$E_{H_{\tn{\fsz{aug}}},\mbf{S}}$}}
\pgfputat{\pgfxy(6.8,2.2)}{\pgfbox[center,center]{$E_{H,\mbf{S}}$}}
\pgfputat{\pgfxy(2.5,2.5)}{\pgfbox[center,center]{$i$}}
\pgfputat{\pgfxy(5.5,2.5)}{\pgfbox[center,center]{$j$}}
\pgfsetendarrow{\pgfarrowlargepointed{3pt}}
\pgfxyline(2,2.2)(3.2,2.2)
\pgfxyline(4.8,2.2)(6.2,2.2)
\begin{colormixin}{50!white}
\pgfputat{\pgfxy(1.3,.2)}{\pgfbox[center,center]{$E_{H_{\tn{\fsz{rel}}}^+,\mbf{S}}$}}
\pgfputat{\pgfxy(4,.2)}{\pgfbox[center,center]{$E_{H_{\tn{\fsz{aug}}}^+,\mbf{S}}$}}
\pgfputat{\pgfxy(6.8,.2)}{\pgfbox[center,center]{$E_{H^+,\mbf{S}}$}}
\pgfputat{\pgfxy(2.5,.5)}{\pgfbox[center,center]{$i^+$}}
\pgfputat{\pgfxy(5.5,.5)}{\pgfbox[center,center]{$j^+$}}
\pgfputat{\pgfxy(.8,1.2)}{\pgfbox[center,center]{$\tn{ch}_{\tn{\fsz{rel}}}$}}
\pgfputat{\pgfxy(3.5,1.2)}{\pgfbox[center,center]{$\tn{ch}_{\tn{\fsz{aug}}}$}}
\pgfputat{\pgfxy(6.4,1.2)}{\pgfbox[center,center]{$\tn{ch}$}}
\pgfsetendarrow{\pgfarrowlargepointed{3pt}}
\pgfxyline(2,.2)(3.2,.2)
\pgfxyline(4.8,.2)(6.2,.2)
\pgfxyline(1.3,1.8)(1.3,.5)
\pgfxyline(4,1.8)(4,.5)
\pgfxyline(6.8,1.8)(6.8,.5)
\end{colormixin}
\end{pgfmagnify}
\end{pgftranslate}
\pgfputat{\pgfxy(16,1)}{\pgfbox[center,center]{$(4.4.2.3)$}}
\end{pgfpicture}
\label{equconiveaumachinefunctorch4}

 Focusing on the level of individual Cousin complexes and sheafifying yields the existence of the third column of the coniveau machine: 

\begin{lem}\label{lemthirdcolumn} Let $\mbf{S}_k$ be a distinguished category of schemes over a field $k$, satisfying the conditions given at the beginning of section \hyperref[subsectioncohomsupportssubstrata]{\ref{subsectioncohomsupportssubstrata}} above, and let $X$ in $\mbf{S}_k$ be  smooth, equidimensional, and noetherian. Let $Y$ be the prime spectrum of a $k$-algebra generated over $k$ by nilpotent elements.  Then the third column of the coniveau machine for Bass-Thomason $K$-theory on $X$ exists; that is, the Cousin complexes appearing as the rows of the $E_1$-level of the coniveau spectral sequence for relative  $K$-theory on $X$ with respect to $Y$ sheafify to yield flasque resolutions of the sheaves $\ms{K}_{p,X\times_kY,Y}$ on $X$.
\end{lem}
\begin{proof}This is established by the foregoing discussion. 
\end{proof}

When $X$ is a smooth algebraic variety, lemma \hyperref[lemthirdcolumn]{\ref{lemthirdcolumn}} enables computation of the generalized tangent groups at the identity $T_Y\tn{Ch}_X^p$, defined in definition \hyperref[defigendefgroupChow]{\ref{defigendefgroupChow}} above, in terms of the sheafified Cousin complexes arising from the spectral sequence for relative $K$-theory.  As in the case of augmented $K$-theory, the relative $K$-theory groups in negative degrees are generally nontrivial here.   A better method of computation, and the whole reason for constructing the coniveau machine in the first place, is to use the relative additive theory; in this case, negative cyclic homology.

\subsection{Relative Chern Character; Fourth Column}\label{subsectionconiveaucompleted}

As stated in lemma  \hyperref[lemrelchernisomfunctors]{\ref{lemrelchernisomfunctors}}, the relative generalized algebraic Chern character $\tn{ch}$ is an isomorphism between the relative algebraic $K$-theory functor and the relative negative cyclic homology functor, considered as functors from the category $\mbf{Nil}$ of split nilpotent pairs $(R,I)$ defined in definition \hyperref[defisplitnilppairs]{\ref{defisplitnilppairs}} to the category of abelian groups.  Hence, for each $p$, $\tn{ch}$ induces functorial isomorphisms 
\begin{equation}\label{equchernsheafisom}K_{p,X\times_kY,Y}\rightarrow \tn{HN}_{p,X\times_kY,Y}.\end{equation}
Due to the functorial properties of the coniveau spectral sequence, there therefore exists an isomorphism of functors between the coniveau spectral sequence for relative $K$-theory and the coniveau spectral sequence for relative negative cyclic homology.  This immediately yields the following lemma: 

\begin{lem}\label{lemfourthcolumn} Let $\mbf{S}_k$ be a distinguished category of schemes over a field $k$, satisfying the conditions given at the beginning of section \hyperref[subsectioncohomsupportssubstrata]{\ref{subsectioncohomsupportssubstrata}} above, and let $X$ in $\mbf{S}_k$ be  smooth, equidimensional, and noetherian. Let $Y$ be the prime spectrum of a $k$-algebra generated over $k$ by nilpotent elements.  Then the fourth column of the coniveau machine for Bass-Thomason $K$-theory on $X$ exists; that is, the algebraic Chern character induces an isomorphism of functors between the coniveau spectral sequences of relative $K$-theory and relative negative cyclic homology with respect to the nilpotent thickening $X\mapsto X\times_kY$.
\end{lem}
\begin{proof}This is established by the foregoing discussion. 
\end{proof}

Combining lemmas \hyperref[lemfirstcolumn]{\ref{lemfirstcolumn}}, \hyperref[lemsecondcolumn]{\ref{lemsecondcolumn}}, \hyperref[lemthirdcolumn]{\ref{lemthirdcolumn}}, and \hyperref[lemfourthcolumn]{\ref{lemfourthcolumn}} leads to the following theorem, which is the main result of this book:

\begin{theorem}\label{theoremconiveaumachineexists} Let $\mbf{S}_k$ be a distinguished category of schemes over a field $k$, satisfying the conditions given at the beginning of section \hyperref[subsectioncohomsupportssubstrata]{\ref{subsectioncohomsupportssubstrata}} above, and let $X$ in $\mbf{S}_k$ be  smooth, equidimensional, and noetherian.  Let $X\mapsto X\times_kY$ be a nilpotent thickening of $X$.  Then the coniveau machine for Bass-Thomason $K$-theory on $X$ with respect to $Y$ exists; that is, the Cousin complexes appearing as the rows of the $E_1$-level of the coniveau spectral sequence for augmented and relative $K$-theory and negative cyclic homology on $X$ with respect to $Y$ sheafify to yield flasque resolutions of the corresponding sheaves, and the algebraic Chern character induces an isomorphism of functors between the coniveau spectral sequences of relative $K$-theory and relative negative cyclic homology. 
\end{theorem}
\begin{proof}The statement follows from a straightforward combination of lemmas \hyperref[lemfirstcolumn]{\ref{lemfirstcolumn}}, \hyperref[lemsecondcolumn]{\ref{lemsecondcolumn}}, \hyperref[lemthirdcolumn]{\ref{lemthirdcolumn}}, and \hyperref[lemfourthcolumn]{\ref{lemfourthcolumn}}. 
\end{proof}

\begin{example}\label{exampletangentconiveau} \tn{For the $p$th Chow group $\tn{Ch}_X^p$, using augmentation by the prime spectrum of the algebra of dual numbers, this leads to the version of the coniveau machine shown in figure \hyperref[figsimplifiedfourcolumnKtheorych4c]{\ref{figsimplifiedfourcolumnKtheorych4c}} below:}
\end{example}

\begin{figure}[H]
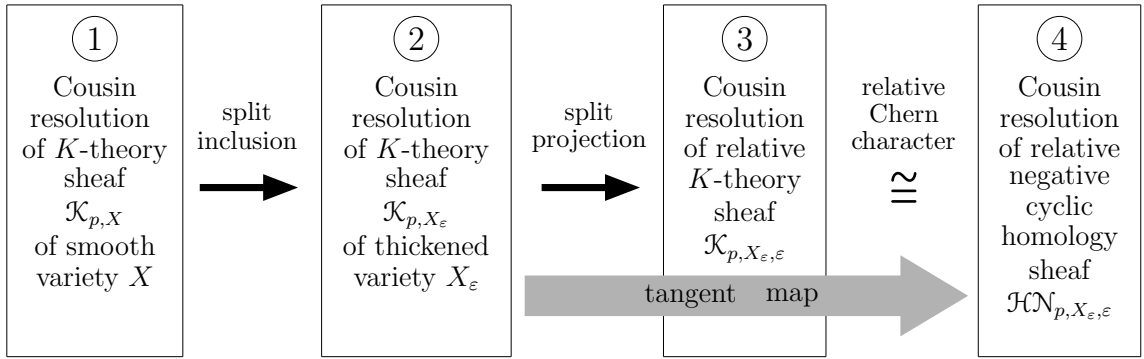

\begin{pgfpicture}{0cm}{0cm}{17cm}{4.7cm}
\begin{pgfmagnify}{.9}{.9}
\begin{pgftranslate}{\pgfpoint{.2cm}{-1.9cm}}
\begin{pgftranslate}{\pgfpoint{-.5cm}{0cm}}
\pgfxyline(.9,1.8)(.9,7)
\pgfxyline(.9,7)(3.5,7)
\pgfxyline(3.5,7)(3.5,1.8)
\pgfxyline(3.5,1.8)(.9,1.8)
\pgfputat{\pgfxy(2.2,5.8)}{\pgfbox[center,center]{Cousin}}
\pgfputat{\pgfxy(2.2,5.35)}{\pgfbox[center,center]{resolution }}
\pgfputat{\pgfxy(2.2,4.85)}{\pgfbox[center,center]{of $K$-theory}}
\pgfputat{\pgfxy(2.2,4.45)}{\pgfbox[center,center]{sheaf}}
\pgfputat{\pgfxy(2.2,3.9)}{\pgfbox[center,center]{$\ms{K}_{p,X}$}}
\pgfputat{\pgfxy(2.2,3.45)}{\pgfbox[center,center]{of smooth}}
\pgfputat{\pgfxy(2.2,2.95)}{\pgfbox[center,center]{variety $X$}}
\pgfnodecircle{Node0}[stroke]{\pgfxy(2.2,6.5)}{0.33cm}
\pgfputat{\pgfxy(2.2,6.5)}{\pgfbox[center,center]{\large{$1$}}}
\end{pgftranslate}
\begin{pgftranslate}{\pgfpoint{4.05cm}{0cm}}
\pgfxyline(1,1.8)(1,7)
\pgfxyline(1,7)(3.8,7)
\pgfxyline(3.8,7)(3.8,1.8)
\pgfxyline(3.8,1.8)(1,1.8)
\pgfputat{\pgfxy(2.4,5.8)}{\pgfbox[center,center]{Cousin}}
\pgfputat{\pgfxy(2.4,5.35)}{\pgfbox[center,center]{resolution }}
\pgfputat{\pgfxy(2.4,4.85)}{\pgfbox[center,center]{of $K$-theory}}
\pgfputat{\pgfxy(2.4,4.45)}{\pgfbox[center,center]{sheaf}}
\pgfputat{\pgfxy(2.4,3.9)}{\pgfbox[center,center]{$\ms{K}_{p,X_\ee}$}}
\pgfputat{\pgfxy(2.4,3.45)}{\pgfbox[center,center]{of thickened}}
\pgfputat{\pgfxy(2.4,2.95)}{\pgfbox[center,center]{variety $X_\ee$}}
 \pgfnodecircle{Node0}[stroke]{\pgfxy(2.4,6.5)}{0.33cm}
\pgfputat{\pgfxy(2.4,6.5)}{\pgfbox[center,center]{\large{$2$}}}
\pgfsetendarrow{\pgfarrowtriangle{6pt}}
\pgfsetlinewidth{3pt}
\pgfxyline(-.8,4.3)(.4,4.3)
\pgfputat{\pgfxy(-.1,5.4)}{\pgfbox[center,center]{\small{split}}}
\pgfputat{\pgfxy(-.1,5)}{\pgfbox[center,center]{\small{inclusion}}}
\end{pgftranslate}
\begin{pgftranslate}{\pgfpoint{9.1cm}{0cm}}
\pgfxyline(1,1.8)(1,7)
\pgfxyline(1,7)(3.4,7)
\pgfxyline(3.4,7)(3.4,1.8)
\pgfxyline(3.4,1.8)(1,1.8)
\pgfputat{\pgfxy(2.2,5.8)}{\pgfbox[center,center]{Cousin}}
\pgfputat{\pgfxy(2.2,5.35)}{\pgfbox[center,center]{resolution}}
\pgfputat{\pgfxy(2.2,4.9)}{\pgfbox[center,center]{of relative}}
\pgfputat{\pgfxy(2.2,4.4)}{\pgfbox[center,center]{$K$-theory}}
\pgfputat{\pgfxy(2.2,3.95)}{\pgfbox[center,center]{sheaf}}
\pgfputat{\pgfxy(2.2,3.4)}{\pgfbox[center,center]{$\ms{K}_{p,X_\ee,\ee}$}}
\pgfnodecircle{Node0}[stroke]{\pgfxy(2.2,6.5)}{0.33cm}
\pgfputat{\pgfxy(2.2,6.5)}{\pgfbox[center,center]{\large{$3$}}}
\pgfsetendarrow{\pgfarrowtriangle{6pt}}
\pgfsetlinewidth{3pt}
\pgfxyline(-.8,4.3)(.4,4.3)
\pgfputat{\pgfxy(-.1,5.4)}{\pgfbox[center,center]{\small{split}}}
\pgfputat{\pgfxy(-.1,5)}{\pgfbox[center,center]{\small{projection}}}
\end{pgftranslate}
\begin{pgftranslate}{\pgfpoint{13.75cm}{0cm}}
\pgfxyline(1,1.8)(1,7)
\pgfxyline(1,7)(3.4,7)
\pgfxyline(3.4,7)(3.4,1.8)
\pgfxyline(3.4,1.8)(1,1.8)
\pgfputat{\pgfxy(2.2,5.8)}{\pgfbox[center,center]{Cousin}}
\pgfputat{\pgfxy(2.2,5.35)}{\pgfbox[center,center]{resolution}}
\pgfputat{\pgfxy(2.2,4.9)}{\pgfbox[center,center]{of relative}}
\pgfputat{\pgfxy(2.2,4.45)}{\pgfbox[center,center]{negative}}
\pgfputat{\pgfxy(2.2,4)}{\pgfbox[center,center]{cyclic}}
\pgfputat{\pgfxy(2.2,3.55)}{\pgfbox[center,center]{homology}}
\pgfputat{\pgfxy(2.2,3.1)}{\pgfbox[center,center]{sheaf}}
\pgfputat{\pgfxy(2.2,2.55)}{\pgfbox[center,center]{$\ms{HN}_{p,X_\ee,\ee}$}}
\pgfputat{\pgfxy(-.1,5.8)}{\pgfbox[center,center]{\small{relative}}}
\pgfputat{\pgfxy(-.1,5.4)}{\pgfbox[center,center]{\small{Chern}}}
\pgfputat{\pgfxy(-.1,5)}{\pgfbox[center,center]{\small{character}}}
\pgfputat{\pgfxy(-.1,4.3)}{\pgfbox[center,center]{\huge{$\cong$}}}
\pgfnodecircle{Node0}[stroke]{\pgfxy(2.2,6.5)}{0.33cm}
\pgfputat{\pgfxy(2.2,6.5)}{\pgfbox[center,center]{\large{$4$}}}
\end{pgftranslate}
\begin{colormixin}{30!white}
\color{black}
\pgfmoveto{\pgfxy(8.05,3)}
\pgflineto{\pgfxy(13.6,3)}
\pgflineto{\pgfxy(13.6,3.3)}
\pgflineto{\pgfxy(14.6,2.7)}
\pgflineto{\pgfxy(13.6,2.1)}
\pgflineto{\pgfxy(13.6,2.4)}
\pgflineto{\pgfxy(8.05,2.4)}
\pgflineto{\pgfxy(8.05,3)}
\pgffill
\end{colormixin}
\pgfputat{\pgfxy(10.5,2.7)}{\pgfbox[center,center]{tangent}}
\pgfputat{\pgfxy(12,2.67)}{\pgfbox[center,center]{map}}
\end{pgftranslate}
\end{pgfmagnify}
\end{pgfpicture}
\caption{Simplified ``four column version" of the coniveau machine for algebraic $K$-theory on a smooth algebraic variety $X$ over a field $k$, thickened by the prime spectrum of the algebra of dual numbers over $k$.}
\label{figsimplifiedfourcolumnKtheorych4c}
\end{figure}

In particular, the cases where $X$ is a smooth curve and $p=1$, and where $X$ is a smooth surface and $p=2$, are the ``toy version" of the coniveau machine constructed in Chapter \hyperref[chaptercurves]{\ref{chaptercurves}}, and the version of Green and Griffiths discussed in the introduction, respectively.

\hspace{16.3cm} $\oblong$

The existence of the coniveau machine enables the computation of the generalized tangent groups at the identity of the Chow groups of a smooth algebraic variety $X$ via negative cyclic homology.  This is made precise by the following corollary:

\newpage

\begin{cor}\label{corollarytangpnegcyc} Let $X$ be a smooth algebraic variety over a field $k$, and let $Y$ be a separated scheme over $k$.  Then for all $p$, the relative algebraic Chern character induces canonical isomorphisms
\begin{equation}\label{equtangentgroupsarenegcychomgroups}T_Y\tn{Ch}_X^p\cong H_{\tn{Zar}}^p(X,\ms{HN}_{p,X\times_kY,Y})\end{equation}
between the generalized tangent groups at the identity of the Chow groups $\tn{Ch}_X^p$ are the sheaf cohomology groups of relative negative cyclic homology on $X$ with respect to the nilpotent thickening $X\mapsto X\times_kY$.
\end{cor}
\begin{proof} By definition \hyperref[defigendefgroupChow]{\ref{defigendefgroupChow}}, the generalized tangent groups at the identity $T_Y\tn{Ch}_X^p$ of the Chow groups $\tn{Ch}_X^p$ are the Zariski sheaf cohomology groups $H_{\tn{\fsz{Zar}}}^p(X,\ms{K}_{p,X\times_kY})$; the result then follows via the relative Chern character.  
\end{proof}

\begin{example}\label{examplesen} \tn{When $k$ is of characteristic zero and $Y$ is the prime spectrum of the algebra of dual numbers, corollary \hyperref[corollarytangpnegcyc]{\ref{corollarytangpnegcyc}} gives the same answer as Sen Yang's version of the coniveau machine \cite{SenYangThesis}, illustrated in figure \hyperref[figSenversion]{\ref{figSenversion}} above:
\begin{equation}\label{equSentangent}T\tn{Ch}_X^p\cong H_{\tn{Zar}}^p(X,\varOmega_{X/\ZZ}^{p-1}\oplus\varOmega_{X/\ZZ}^{p-3}\oplus...)\end{equation}
}
\end{example}

\hspace{16.3cm} $\oblong$

\subsection{Universal Exactness}\label{subsectionuniversalexactness}

In section \hyperref[subsectionnewoutofold]{\ref{subsectionnewoutofold}} above, I discussed the choice of viewpoint between ``augmenting the scheme" and ``augmenting the cohomology theory" in constructing the second and third columns of the coniveau machine.  The latter alternative proved more appropriate in this context.  A third possibility is to directly ``augment the Cousin complexes" arising from the coniveau spectral sequences, and show that the resulting complexes remain exact under appropriate assumptions.  This is the approach of {\it universal exactness}, which provides some additional ways of extending the theory.  


{\bf Universal Exactness.}  An exact sequence in an abelian category is called {\it universally exact} if it remains exact following the application of certain additive functors.  The concept of universal exactness was introduced by Daniel Grayson in his short 1985 paper {\it Universal Exactness in Algebraic $K$-Theory} \cite{GraysonUniversalExactness85}, in the special case of abelian groups.  Grayson proved (\cite{GraysonUniversalExactness85}, Corollary 6, page 141) that the sheafified Cousin resolutions for algebraic $K$-theory on smooth algebraic varieties, discussed in a more general context in section \hyperref[sectioncousinblochogus]{\ref{sectioncousinblochogus}} above, are universally exact.  In particular, Grayson observed that applying universal exactness to the Cousin complex for the $p$th $K$-theory sheaf $\ms{K}_{p,X}$ on a smooth algebraic variety $X$, and taking sheaf cohomology, leads to the identity
\begin{equation}\label{equgraysontorsion}\tn{Ch}_X^p\otimes\ZZ_m=H_{\tn{Zar}}^p(X,\ms{K}_{p,X}\otimes\ZZ_m),\end{equation}
via Bloch's theorem.  Here, I outline a more general approach to universal exactness for cohomology theories with supports, based on Colliot-Th\'el\`ene, Hoobler, and Kahn \cite{CHKBloch-Ogus-Gabber97}.  

The following definition of universal exactness is adapted from Colliot-Th\'el\`ene, Hoobler and Kahn \cite{CHKBloch-Ogus-Gabber97}, definition 6.1.1, page 34:

\begin{defi}\label{defiuniversalexactness} Let $\mbf{A}$ be an abelian category.   A complex $C$ of objects of $\mbf{A}$ is called {\bf universally exact} if for any abelian category $\mbf{B}$ and any additive functor $\ms{F}:\mbf{A}\rightarrow\mbf{B}$ commuting with colimits of filtered diagrams, the complex $\ms{F}_{C}$ in $\mbf{B}$ is exact.
\end{defi}

{\bf Universal Exactness for Effaceable Substrata.} The complexes of principal interest in the present context are, of course, the Cousin complexes arising from the coniveau spectral sequence of an effaceable cohomology theory with supports $H$, applied to a smooth scheme $X$.  As discussed below, it turns out that if such a cohomology theory arises from an effaceable substratum functor, then universal exactness holds. 

The following lemma is adapted from Colliot-Th\'el\`ene, Hoobler and Kahn \cite{CHKBloch-Ogus-Gabber97}, Theorem 6.2.1, page 35.  

\begin{lem}\label{lemuniversalexactness} Let $\mbf{S}_k$ be a category of schemes over a field $k$ satisfying the conditions given at the beginning of section \hyperref[subsectioncohomsupportssubstrata]{\ref{subsectioncohomsupportssubstrata}} above, and let $X$ be an affine variety in $\mbf{S}_k$.  Let $H$ a cohomology theory with supports on the category of pairs over $\mbf{S}_k$, and assume that $H$ is defined via a substratum $C$, effaceable at a family of points $(t_1,...,t_r)$.  Let $R=O_{t_1,...,t_r}$ be the semilocal ring of $X$ at $(t_1,...,t_r)$, and let $Y=\tn{Spec }R$.  Then the Cousin complexes 

\begin{pgfpicture}{0cm}{0cm}{17cm}{2cm}
\begin{pgftranslate}{\pgfpoint{1.5cm}{0cm}}
\pgfputat{\pgfxy(-1,1)}{\pgfbox[center,center]{$0$}}
\pgfputat{\pgfxy(.5,1)}{\pgfbox[center,center]{$H_Y^n$}}
\pgfputat{\pgfxy(3,.9)}{\pgfbox[center,center]{$\displaystyle\coprod_{x\in \tn{\footnotesize{Zar}}_Y^{0}} H_{Y\tn{ \footnotesize{on} } x}^{n}$}}
\pgfputat{\pgfxy(6.5,.9)}{\pgfbox[center,center]{$\displaystyle\coprod_{x\in \tn{\footnotesize{Zar}}_Y^{1}} H_{Y\tn{ \footnotesize{on} } x}^{n+1}$}}
\pgfputat{\pgfxy(12.5,.9)}{\pgfbox[center,center]{$\displaystyle\coprod_{x\in \tn{\footnotesize{Zar}}_Y^{d}} H_{Y\tn{ \footnotesize{on} } x}^{n+d}$}}
\pgfputat{\pgfxy(15,1)}{\pgfbox[center,center]{$0$}}
\pgfputat{\pgfxy(4.9,1.3)}{\pgfbox[center,center]{\small{$d_1^{0,n}$}}}
\pgfputat{\pgfxy(8.4,1.3)}{\pgfbox[center,center]{\small{$d_1^{1,n}$}}}
\pgfputat{\pgfxy(10.9,1.3)}{\pgfbox[center,center]{\small{$d_1^{d-1,n}$}}}
\pgfsetendarrow{\pgfarrowlargepointed{3pt}}
\pgfxyline(-.8,1)(0,1)
\pgfxyline(.9,1)(1.7,1)
\pgfxyline(4.4,1)(5.2,1)
\pgfxyline(7.9,1)(8.7,1)
\pgfnodecircle{Node0}[fill]{\pgfxy(8.9,1)}{0.02cm}
\pgfnodecircle{Node0}[fill]{\pgfxy(9,1)}{0.02cm}
\pgfnodecircle{Node0}[fill]{\pgfxy(9.1,1)}{0.02cm}
\pgfnodecircle{Node0}[fill]{\pgfxy(10,1)}{0.02cm}
\pgfnodecircle{Node0}[fill]{\pgfxy(10.1,1)}{0.02cm}
\pgfnodecircle{Node0}[fill]{\pgfxy(10.2,1)}{0.02cm}
\pgfxyline(10.4,1)(11.2,1)
\pgfxyline(13.9,1)(14.7,1)
\end{pgftranslate}
\end{pgfpicture}

are universally exact. 
\end{lem}
\begin{proof} See Colliot-Th\'el\`ene, Hoobler and Kahn \cite{CHKBloch-Ogus-Gabber97}, Theorem 6.2.1, page 35.
\end{proof}

Since Bass-Thomason $K$-theory and negative cyclic homology are defined in terms of effaceable substrata, this result allows the coniveau machine for to be generalized further to enable analysis of ``modified generalized deformation groups and tangent groups" analogous to the groups studied by Grayson \cite{GraysonUniversalExactness85}.

\appendix
\renewcommand{\chaptername}{Appendix}
\addtocontents{toc}{\protect\renewcommand\protect\chaptername{Appendix}}



\chapter{  Supporting Material on $K$-Theory and Cyclic Homology}\label{chapterappendix}

\section{Topological Spectra}\label{subsecspectra}

Modern algebraic $K$-theory may be described in terms of the theory of topological spectra.   In this section, I briefly review some standard material from the theory of spectra for the convenience of the reader. The four references \cite{AdamsStableHomotopy74}, \cite{ThomasonKTheoryEtaleCohom85}, \cite{BousfieldFriedlander77}, and \cite{ThomasonLoopSpaceMachines78} represent the spectrum-theoretic background of the papers of Thomason \cite{Thomason-Trobaugh90} and Colliot-Th\'el\`ene, Hoobler, and Khan \cite{CHKBloch-Ogus-Gabber97} in which spectra are crucial.  See \cite{Thomason-Trobaugh90}, Introduction, page 249, and \cite{CHKBloch-Ogus-Gabber97}, section 5.2, page 29.  Schwede's book project \cite{SchwedeSymmetricSpectra2007} seems to be a helptul modern source on the subject.

A topological spectrum $\mbf{E}:=\{E_n,\ee_n\}$ is a sequence of pointed spaces $E_n$ and {\it structure maps} $\ee_n$ mapping the reduced suspension $\Sigma_{E_n}$ of $E_n$ to $E_{n+1}$.  The indices $n$ are usually chosen to run over the integers or the nonnegative integers.\footnotemark\footnotetext{There is also a notion of {\bf coordinate free spectra,} for which there are defined structure maps for every injective map of finite sets; Waldhausen's $K$-theory spectrum may be viewed as coordinate free in this sense (see \cite{SchwedeSymmetricSpectra2007} page 23).}  Since the reduced suspension functor $\Sigma$ is left-adjoint to the loop space functor $\Omega$, the structure maps are equivalent to maps $E_n\rightarrow\Omega E_{n+1}$.  If these maps are isomorphic, then $\mbf{E}$ is called an $\Omega$-spectrum.  Isomorphism in this context may mean homeomorphism, homotopy equivalence, or weak homotopy equivalence, in descending order of strength.

The homotopy groups $\pi_{r,\mbf{E}}$ of a spectrum $\mbf{E}$ are {\bf stable homotopy groups} in the following sense.  The reduced suspension map $E_n\rightarrow \Sigma_{E_n}$ and the structure map $\Sigma_{E_n}\rightarrow E_{n+1}$ together induce homomorphisms on the homotopy groups of spaces:
\[\pi_{n+r,E_n}\rightarrow \pi_{n+r+1,\Sigma_{E_{n}}}\rightarrow \pi_{n+r+1,E_{n+1}},\]
for each $r\in\mbb{Z}$.  Calling the composite maps $s_n$, one has a directed system
\[\xymatrix{ ...\ar[r]^-{s_{n-1}}&\pi_{n+r,E_n}\ar[r]^-{s_n} & \pi_{n+r+1,E_{n+1}}\ar[r]^-{s_{n+1}}&...,}\]
and one defines the homotopy group $\pi_{r,\mbf{E}}$ of $\mbf{E}$ to be the direct limit $\varinjlim_n\pi_{n+r,E_n}$ of this directed system.  In general, the term {\bf stability} in homotopy theory refers to phenomena that do not depend on behavior at any particular dimension or over any finite range of dimensions, but can be detected as the dimension grows arbitrarily large.  Note in particular that the limit $\pi_{r,\mbf{E}}=\varinjlim_n\pi_{n+r,E_n}$ may or may not be attained at some finite $n$, depending on the spectrum $\mbf{E}$.   However, if $\mbf{E}$ is an $\Omega$-spectrum, then the homomorphism $s_n:\pi_{n+r,E_{n}}\rightarrow \pi_{n+r+1,E_{n+1}}$ is an isomorphism for all $n\ge 1-r$, so in this case 
\[\pi_{r,\mbf{E}}=\pi_{n+r,E_{n}}\tn{\hspace*{1cm} for any \hspace*{1cm}} n\ge 1-r.\]
Note also that the homotopy groups $\pi_{r,\mbf{E}}$ are defined, and may be nonzero, for negative $r$, whether or not the index set for the spaces $E_n$ includes negative integers.   In many important cases, the homotopy groups vanish in negative degrees, just as do the homotopy groups of an individual topological space.  For instance, the Waldhausen $K$-theory spectrum (\cite{Thomason-Trobaugh90} page 260) has vanishing negative homotopy groups.  If $\pi_{r,\mbf{E}}=0$ for $r<0$, then $\mbf{E}$ is called a {\bf connective} spectrum.   Otherwise it is called {\bf nonconnective}. 

Just as the connectivity of a spectrum $\mbf{E}$ does not depend on the range of indices of the individual topological spaces $E_n$, it does not depend on the connectivity properties of the $E_n$, although the two types of properties are analogous.  More precisely, a topological space is called {\it $r$-connected} if it is path-connected and its first $r$ homotopy groups (beginning with the fundamental group) vanish.  Hence, a connective spectrum is analogous in this sense to a topological space which is $r$-connected for some nonnegative $r$.  However, one may change the connectivity of a spectrum trivially by simply shifting the indices of the spaces $E_n$; if $\pi_{r}(\mbf{E})\ne0$ for some $r$, then the spectrum $\mbf{E}'$ with $E'_n:=E_{n+r+1}$ is nonconnective since $\pi_{-1}(\mbf{E}')=\pi_{r}(\mbf{E})$.   Usually, however, the indices are fixed by topological considerations specific to the problem at hand.   For instance, if one begins with a pointed topological space $E$ and forms the {\it suspension spectrum} $\mbf{E}:=\Sigma^\infty E$ by repeated application of the reduced suspension functor,  it is natural to index the spaces $E_n$ so that $E_0=E$.

\section{Keller's Mixed Complex}\label{sectioncyclichomologymixedkeller}

For any $k$-algebra $R$, unital or nonunital, Bernhard Keller \cite{KellerCyclicHomologyofExactCat96} defines a mixed complex $(M_R,d,B)$, as I will now describe.  This approach is more general than the approaches to the cyclic homology of algebras outlined in sections \hyperref[subsectioncychomalgcommring]{\ref{subsectioncychomalgcommring}} and \hyperref[subsectionnegcyccomring]{\ref{subsectionnegcyccomring}}, and generalizes better.  

Recall that the map $(1-t)$ links the columns $(\tn{C}_R,-b')$ and $(\tn{C}_R,b)$ of degrees $1$ and $0$ in the cyclic bicomplex $\tn{CC}_R$, where $t$ is the cyclic operator, whose index depends on the row, and where $1$ denotes the identity operator.  Also, the norm operator $N$ links the columns of degrees $2$ and $1$.  Since $\tn{CC}_R$ is a bicomplex, the maps $b,b',(1-t),$ and $N$ satisfy the following identities:
\begin{equation}\label{identitieskeller}b^2=b'^2=b(1-t)-(1-t)b'=Nb-b'N=(1-t)N=N(1-t)=0.\end{equation}
In particular, the map $(1-t)$ is a chain map between the complexes $(C_R,b')$ and $(C_R,b)$.\footnotemark\footnotetext{Note how the minus sign has disappeared on the boundary map $b'$; this is because of the standard convention for bicomplexes; the maps are chosen so that each square {\it anticommutes}, rather than commuting.  Thus, the map $(1-t)$ as it appears in $\tn{CC}_R$ is not quite a chain map; the minus sign must be removed to give one.} The {\bf mapping cone} $(M_R,d)$ over the chain map $(1-t):(C_R,b')\rightarrow (C_R,b)$ is by definition the chain complex whose $n$th term is 
\begin{equation}\label{equmappingcone}M_n:=C_{n,R}\oplus C_{n-1,R}=R^{\otimes n+1}\oplus R^{\otimes n}.\end{equation}
The differential $d$ in $(M_R,d)$ is defined by the matrix 
\begin{equation}\label{equkellerconediff}d = \begin{pmatrix} b & 1-t\\ 0 & -b' \end{pmatrix}.\end{equation}
This is illustrated in figure \hyperref[figkellermixed]{\ref{figkellermixed}} below.  See Keller \cite{KellerCyclicHomologyofExactCat96}, page 5, for details.\footnotemark\footnotetext{Note that the minus sign has reappeared, but this is by definition of the mapping cone. Some other authors define the mapping cone differently; for instance, see Weibel \cite{WeibelHomologicalAlgebra94}, section 1.5, page 18.} 

\begin{figure}[H]
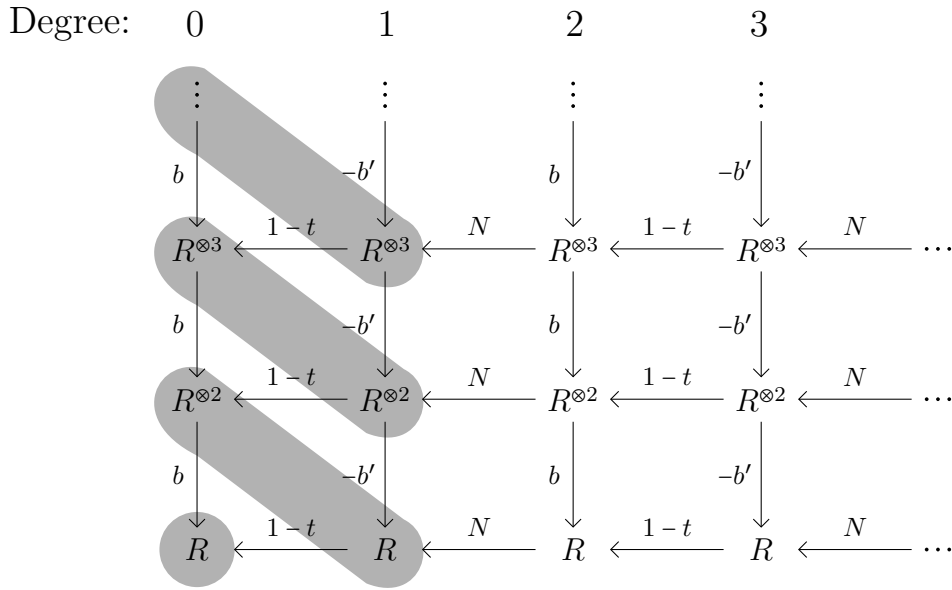

\begin{pgfpicture}{0cm}{0cm}{17cm}{8cm}
\begin{pgftranslate}{\pgfpoint{1.5cm}{-1cm}}
\begin{pgftranslate}{\pgfpoint{-.25cm}{1cm}}
\begin{colormixin}{100!white}
\begin{colormixin}{30!white}
\pgfnodecircle{Node1}[fill]{\pgfxy(1.5,.75)}{0.5cm}
\end{colormixin}
\end{colormixin}
\begin{colormixin}{30!white}
\color{black}
\pgfmoveto{\pgfxy(3.75,.3)}
\pgflineto{\pgfxy(1.5,2)}
\pgfcurveto{\pgfxy(.5,2.5)}{\pgfxy(1,3.35)}{\pgfxy(1.6,3.15)}
\pgflineto{\pgfxy(4.25,1.15)}
\pgfcurveto{\pgfxy(4.75,.9)}{\pgfxy(4.5,0)}{\pgfxy(3.75,.3)}
\pgffill
\end{colormixin}
\end{pgftranslate}
\begin{pgftranslate}{\pgfpoint{-.25cm}{3cm}}
\begin{colormixin}{30!white}
\color{black}
\pgfmoveto{\pgfxy(3.75,.3)}
\pgflineto{\pgfxy(1.5,2)}
\pgfcurveto{\pgfxy(.5,2.5)}{\pgfxy(1,3.35)}{\pgfxy(1.6,3.15)}
\pgflineto{\pgfxy(4.25,1.15)}
\pgfcurveto{\pgfxy(4.75,.9)}{\pgfxy(4.5,0)}{\pgfxy(3.75,.3)}
\pgffill
\end{colormixin}
\end{pgftranslate}
\begin{pgftranslate}{\pgfpoint{-.25cm}{5cm}}
\begin{colormixin}{30!white}
\color{black}
\pgfmoveto{\pgfxy(3.75,.3)}
\pgflineto{\pgfxy(1.5,2)}
\pgfcurveto{\pgfxy(.5,2.5)}{\pgfxy(1,3.35)}{\pgfxy(1.6,3.15)}
\pgflineto{\pgfxy(4.25,1.15)}
\pgfcurveto{\pgfxy(4.75,.9)}{\pgfxy(4.5,0)}{\pgfxy(3.75,.3)}
\pgffill
\end{colormixin}
\end{pgftranslate}
\begin{pgftranslate}{\pgfpoint{-.75cm}{1.25cm}}
\pgfputat{\pgfxy(-.5,7.5)}{\pgfbox[left,center]{\large{Degree:}}}
\pgfputat{\pgfxy(1.85,7.5)}{\pgfbox[left,center]{\large{0}}}
\pgfputat{\pgfxy(4.4,7.5)}{\pgfbox[left,center]{\large{1}}}
\pgfputat{\pgfxy(6.9,7.5)}{\pgfbox[left,center]{\large{2}}}
\pgfputat{\pgfxy(9.35,7.5)}{\pgfbox[left,center]{\large{3}}}
\pgfputat{\pgfxy(2,.5)}{\pgfbox[center,center]{$R$}}
\pgfputat{\pgfxy(2,2.5)}{\pgfbox[center,center]{$R^{\otimes 2}$}}
\pgfputat{\pgfxy(2,4.5)}{\pgfbox[center,center]{$R^{\otimes 3}$}}
\pgfnodecircle{Node1}[fill]{\pgfxy(2,6.4)}{0.025cm}
\pgfnodecircle{Node1}[fill]{\pgfxy(2,6.55)}{0.025cm}
\pgfnodecircle{Node1}[fill]{\pgfxy(2,6.7)}{0.025cm}
\pgfputat{\pgfxy(4.5,.5)}{\pgfbox[center,center]{$R$}}
\pgfputat{\pgfxy(4.5,2.5)}{\pgfbox[center,center]{$R^{\otimes 2}$}}
\pgfputat{\pgfxy(4.5,4.5)}{\pgfbox[center,center]{$R^{\otimes 3}$}}
\pgfnodecircle{Node1}[fill]{\pgfxy(4.5,6.4)}{0.025cm}
\pgfnodecircle{Node1}[fill]{\pgfxy(4.5,6.55)}{0.025cm}
\pgfnodecircle{Node1}[fill]{\pgfxy(4.5,6.7)}{0.025cm}
\pgfputat{\pgfxy(7,.5)}{\pgfbox[center,center]{$R$}}
\pgfputat{\pgfxy(7,2.5)}{\pgfbox[center,center]{$R^{\otimes 2}$}}
\pgfputat{\pgfxy(7,4.5)}{\pgfbox[center,center]{$R^{\otimes 3}$}}
\pgfnodecircle{Node1}[fill]{\pgfxy(7,6.4)}{0.025cm}
\pgfnodecircle{Node1}[fill]{\pgfxy(7,6.55)}{0.025cm}
\pgfnodecircle{Node1}[fill]{\pgfxy(7,6.7)}{0.025cm}
\pgfputat{\pgfxy(1.75,1.5)}{\pgfbox[center,center]{\footnotesize{$b$}}}
\pgfputat{\pgfxy(1.75,3.5)}{\pgfbox[center,center]{\footnotesize{$b$}}}
\pgfputat{\pgfxy(1.75,5.5)}{\pgfbox[center,center]{\footnotesize{$b$}}}
\pgfputat{\pgfxy(4.15,1.5)}{\pgfbox[center,center]{\footnotesize{$-b'$}}}
\pgfputat{\pgfxy(4.15,3.5)}{\pgfbox[center,center]{\footnotesize{$-b'$}}}
\pgfputat{\pgfxy(4.15,5.55)}{\pgfbox[center,center]{\footnotesize{$-b'$}}}
\pgfputat{\pgfxy(6.75,1.5)}{\pgfbox[center,center]{\footnotesize{$b$}}}
\pgfputat{\pgfxy(6.75,3.5)}{\pgfbox[center,center]{\footnotesize{$b$}}}
\pgfputat{\pgfxy(6.75,5.5)}{\pgfbox[center,center]{\footnotesize{$b$}}}
\pgfputat{\pgfxy(3.25,.8)}{\pgfbox[center,center]{\footnotesize{$1-t$}}}
\pgfputat{\pgfxy(3.25,2.8)}{\pgfbox[center,center]{\footnotesize{$1-t$}}}
\pgfputat{\pgfxy(3.25,4.8)}{\pgfbox[center,center]{\footnotesize{$1-t$}}}
\pgfputat{\pgfxy(5.75,.8)}{\pgfbox[center,center]{\footnotesize{$N$}}}
\pgfputat{\pgfxy(5.75,2.8)}{\pgfbox[center,center]{\footnotesize{$N$}}}
\pgfputat{\pgfxy(5.75,4.8)}{\pgfbox[center,center]{\footnotesize{$N$}}}
\pgfsetendarrow{\pgfarrowlargepointed{3pt}}
\pgfxyline(2,2.2)(2,0.8)
\pgfxyline(2,4.2)(2,2.8)
\pgfxyline(2,6.2)(2,4.8)
\pgfxyline(4.5,2.2)(4.5,0.8)
\pgfxyline(4.5,4.2)(4.5,2.8)
\pgfxyline(4.5,6.2)(4.5,4.8)
\pgfxyline(7,2.2)(7,0.8)
\pgfxyline(7,4.2)(7,2.8)
\pgfxyline(7,6.2)(7,4.8)
\pgfxyline(4,.5)(2.5,.5)
\pgfxyline(4,2.5)(2.5,2.5)
\pgfxyline(4,4.5)(2.5,4.5)
\pgfxyline(6.5,.5)(5,.5)
\pgfxyline(6.5,2.5)(5,2.5)
\pgfxyline(6.5,4.5)(5,4.5)
\end{pgftranslate}
\begin{pgftranslate}{\pgfpoint{4.25cm}{1.25cm}}
\pgfputat{\pgfxy(4.5,.5)}{\pgfbox[center,center]{$R$}}
\pgfputat{\pgfxy(4.5,2.5)}{\pgfbox[center,center]{$R^{\otimes 2}$}}
\pgfputat{\pgfxy(4.5,4.5)}{\pgfbox[center,center]{$R^{\otimes 3}$}}
\pgfnodecircle{Node1}[fill]{\pgfxy(4.5,6.4)}{0.025cm}
\pgfnodecircle{Node1}[fill]{\pgfxy(4.5,6.55)}{0.025cm}
\pgfnodecircle{Node1}[fill]{\pgfxy(4.5,6.7)}{0.025cm}
\pgfputat{\pgfxy(4.15,1.5)}{\pgfbox[center,center]{\footnotesize{$-b'$}}}
\pgfputat{\pgfxy(4.15,3.5)}{\pgfbox[center,center]{\footnotesize{$-b'$}}}
\pgfputat{\pgfxy(4.15,5.55)}{\pgfbox[center,center]{\footnotesize{$-b'$}}}
\pgfputat{\pgfxy(3.25,.8)}{\pgfbox[center,center]{\footnotesize{$1-t$}}}
\pgfputat{\pgfxy(3.25,2.8)}{\pgfbox[center,center]{\footnotesize{$1-t$}}}
\pgfputat{\pgfxy(3.25,4.8)}{\pgfbox[center,center]{\footnotesize{$1-t$}}}
\pgfputat{\pgfxy(5.75,.8)}{\pgfbox[center,center]{\footnotesize{$N$}}}
\pgfputat{\pgfxy(5.75,2.8)}{\pgfbox[center,center]{\footnotesize{$N$}}}
\pgfputat{\pgfxy(5.75,4.8)}{\pgfbox[center,center]{\footnotesize{$N$}}}
\pgfnodecircle{Node1}[fill]{\pgfxy(6.7,.5)}{0.025cm}
\pgfnodecircle{Node1}[fill]{\pgfxy(6.85,.5)}{0.025cm}
\pgfnodecircle{Node1}[fill]{\pgfxy(7,.5)}{0.025cm}
\pgfnodecircle{Node1}[fill]{\pgfxy(6.7,2.5)}{0.025cm}
\pgfnodecircle{Node1}[fill]{\pgfxy(6.85,2.5)}{0.025cm}
\pgfnodecircle{Node1}[fill]{\pgfxy(7,2.5)}{0.025cm}
\pgfnodecircle{Node1}[fill]{\pgfxy(6.7,4.5)}{0.025cm}
\pgfnodecircle{Node1}[fill]{\pgfxy(6.85,4.5)}{0.025cm}
\pgfnodecircle{Node1}[fill]{\pgfxy(7,4.5)}{0.025cm}
\pgfsetendarrow{\pgfarrowlargepointed{3pt}}
\pgfxyline(4.5,2.2)(4.5,0.8)
\pgfxyline(4.5,4.2)(4.5,2.8)
\pgfxyline(4.5,6.2)(4.5,4.8)
\pgfxyline(4,.5)(2.5,.5)
\pgfxyline(4,2.5)(2.5,2.5)
\pgfxyline(4,4.5)(2.5,4.5)
\pgfxyline(6.5,.5)(5,.5)
\pgfxyline(6.5,2.5)(5,2.5)
\pgfxyline(6.5,4.5)(5,4.5)
\end{pgftranslate}
\end{pgftranslate}
\end{pgfpicture}
\caption{Mapping cone over the chain map $(1-t)$.}
\label{figkellermixed}
\end{figure}

The matrix definition means that if one views an element $m=m_n\in M_n$ as a column vector of the form 
\[m_n = \begin{pmatrix} r_{n} \\ r_{n-1} \end{pmatrix}, r_{n}\in R^{\otimes n+1}, r_{n-1}\in R^{\otimes n},\] 
then
\[d(m)=d_n(m_n)=\begin{pmatrix} b_n & 1-t_{n-1}\\ 0 & -b_{n-1}' \end{pmatrix} \begin{pmatrix} r_{n} \\ r_{n-1} \end{pmatrix}=\begin{pmatrix} b_n(r_{n})+(1-t_{n-1})(r_{n-1}) \\ -b'_{n-1}(r_{n-1}) \end{pmatrix}.\]
Matrix multiplication shows that $d^2=0$, so that $(M_R,d)$ is indeed a complex:
\[\begin{pmatrix} b & 1-t\\ 0 & -b' \end{pmatrix}\begin{pmatrix} b & 1-t\\ 0 & -b' \end{pmatrix}=\begin{pmatrix} b^2 & b(1-t)-(1-t)b'\\ 0 & b'^2 \end{pmatrix}=0.\]
Next, define a cochain map 
\[B^M= \begin{pmatrix} 0 & 0\\ N & 0\end{pmatrix},\]
on $M_R$.\footnotemark\footnotetext{The ``$M$'' stands for ``mixed.''  This is my nonstandard notation; Keller simply uses the notation $B$, but this is extremely confusing since $B$ is used in several different ways in very similar contexts.  In particular, $B^M$ is {\it not} the Connes boundary map $B$.  More on this in the remarks below.} $B^M_n$ maps $M_n$ to $M_{n+1}$. For an element $m_n$ as above, 
\[B^M_n(m_n)=\begin{pmatrix} 0 & 0\\ N_n & 0\end{pmatrix} \begin{pmatrix} r_{n} \\ r_{n-1} \end{pmatrix}=\begin{pmatrix} 0 \\ N_n(r_n) \end{pmatrix}.\]

Figure \hyperref[figkellermixed]{\ref{figkellermixed}} illustrates both the abstract and concrete forms of Keller's mixed complex. $(M_R,d,B^M)$.  

\begin{figure}[H]
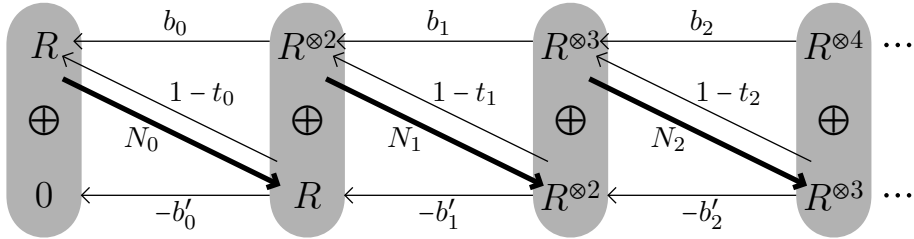

\begin{pgfpicture}{0cm}{0cm}{17cm}{7cm}

\begin{pgftranslate}{\pgfpoint{0cm}{1.25cm}}
\begin{colormixin}{100!white}
\end{colormixin}
\pgfputat{\pgfxy(.5,5.25)}{\pgfbox[left,center]{Abstract form:}}
\pgfputat{\pgfxy(1.5,4)}{\pgfbox[center,center]{\large{$M_0$}}}
\pgfputat{\pgfxy(3.25,4.4)}{\pgfbox[center,center]{\small{$d$}}}
\pgfputat{\pgfxy(3.25,3.6)}{\pgfbox[center,center]{\small{$B^M$}}}
\pgfsetendarrow{\pgfarrowlargepointed{3pt}}
\pgfsetlinewidth{2pt}
\pgfxyline(2,3.85)(4.5,3.85)
\pgfsetlinewidth{.5pt}
\pgfxyline(4.5,4.15)(2,4.15)
\color{black} 
\begin{pgftranslate}{\pgfpoint{3.5cm}{0cm}}
\pgfputat{\pgfxy(1.5,4)}{\pgfbox[center,center]{\large{$M_1$}}}
\pgfputat{\pgfxy(3.25,4.4)}{\pgfbox[center,center]{\small{$d$}}}
\pgfputat{\pgfxy(3.25,3.6)}{\pgfbox[center,center]{\small{$B^M$}}}
\pgfsetendarrow{\pgfarrowlargepointed{3pt}}
\pgfsetlinewidth{2pt}
\pgfxyline(2,3.85)(4.5,3.85)
\pgfsetlinewidth{.5pt}
\pgfxyline(4.5,4.15)(2,4.15)
\color{black} 
\end{pgftranslate}
\begin{pgftranslate}{\pgfpoint{7cm}{0cm}}
\pgfputat{\pgfxy(1.5,4)}{\pgfbox[center,center]{\large{$M_2$}}}
\pgfputat{\pgfxy(3.25,4.4)}{\pgfbox[center,center]{\small{$d$}}}
\pgfputat{\pgfxy(3.25,3.6)}{\pgfbox[center,center]{\small{$B^M$}}}
\pgfsetendarrow{\pgfarrowlargepointed{3pt}}
\pgfsetlinewidth{2pt}
\pgfxyline(2,3.85)(4.5,3.85)
\pgfsetlinewidth{.5pt}
\pgfxyline(4.5,4.15)(2,4.15)
\color{black} 
\end{pgftranslate}
\begin{pgftranslate}{\pgfpoint{10.5cm}{0cm}}
\pgfputat{\pgfxy(1.5,4)}{\pgfbox[center,center]{\large{$M_3$}}}
\pgfnodecircle{Node1}[fill]{\pgfxy(2.2,4)}{0.03cm}
\pgfnodecircle{Node1}[fill]{\pgfxy(2.35,4)}{0.03cm}
\pgfnodecircle{Node1}[fill]{\pgfxy(2.5,4)}{0.03cm}
\end{pgftranslate}
\end{pgftranslate}
\begin{pgftranslate}{\pgfpoint{3.5cm}{0cm}}
\pgfputat{\pgfxy(-3,4)}{\pgfbox[left,center]{Concrete form:}}
\begin{pgftranslate}{\pgfpoint{-3.5cm}{0cm}}
\begin{colormixin}{30!white}
\color{black}
\pgfmoveto{\pgfxy(1,.75)}
\pgflineto{\pgfxy(1,2.75)}
\pgfcurveto{\pgfxy(1,3.5)}{\pgfxy(2,3.5)}{\pgfxy(2,2.75)}
\pgflineto{\pgfxy(2,.75)}
\pgfcurveto{\pgfxy(2,0)}{\pgfxy(1,0)}{\pgfxy(1,.75)}
\pgffill
\end{colormixin}
\begin{colormixin}{30!white}
\color{black}
\pgfmoveto{\pgfxy(4.5,.75)}
\pgflineto{\pgfxy(4.5,2.75)}
\pgfcurveto{\pgfxy(4.5,3.5)}{\pgfxy(5.5,3.5)}{\pgfxy(5.5,2.75)}
\pgflineto{\pgfxy(5.5,.75)}
\pgfcurveto{\pgfxy(5.5,0)}{\pgfxy(4.5,0)}{\pgfxy(4.5,.75)}
\pgffill
\end{colormixin}
\begin{colormixin}{30!white}
\color{black}
\pgfmoveto{\pgfxy(8,.75)}
\pgflineto{\pgfxy(8,2.75)}
\pgfcurveto{\pgfxy(8,3.5)}{\pgfxy(9,3.5)}{\pgfxy(9,2.75)}
\pgflineto{\pgfxy(9,.75)}
\pgfcurveto{\pgfxy(9,0)}{\pgfxy(8,0)}{\pgfxy(8,.75)}
\pgffill
\end{colormixin}
\begin{colormixin}{30!white}
\color{black}
\pgfmoveto{\pgfxy(11.5,.75)}
\pgflineto{\pgfxy(11.5,2.75)}
\pgfcurveto{\pgfxy(11.5,3.5)}{\pgfxy(12.5,3.5)}{\pgfxy(12.5,2.75)}
\pgflineto{\pgfxy(12.5,.75)}
\pgfcurveto{\pgfxy(12.5,0)}{\pgfxy(11.5,0)}{\pgfxy(11.5,.75)}
\pgffill
\end{colormixin}
\pgfputat{\pgfxy(1.5,2.75)}{\pgfbox[center,center]{\large{$R$}}}
\pgfputat{\pgfxy(1.5,1.75)}{\pgfbox[center,center]{\large{$\bigoplus$}}}
\pgfputat{\pgfxy(1.5,.75)}{\pgfbox[center,center]{\large{$0$}}}
\pgfputat{\pgfxy(2.8,1.5)}{\pgfbox[center,center]{\small{$N_0$}}}
\pgfputat{\pgfxy(3.25,3.05)}{\pgfbox[center,center]{\small{$b_0$}}}
\pgfputat{\pgfxy(3.6,2.1)}{\pgfbox[center,center]{\small{$1-t_0$}}}
\pgfputat{\pgfxy(3.25,.5)}{\pgfbox[center,center]{\small{$-b'_0$}}}
\pgfsetendarrow{\pgfarrowlargepointed{3pt}}
\pgfsetlinewidth{2pt}
\pgfxyline(1.75,2.3)(4.6,.9)
\pgfsetlinewidth{.5pt}
\pgfxyline(4.5,2.8)(1.9,2.8)
\pgfxyline(4.6,1.2)(1.75,2.6)
\pgfxyline(4.5,.75)(2,.75)
\end{pgftranslate}
\begin{pgftranslate}{\pgfpoint{0cm}{0cm}}
\pgfputat{\pgfxy(1.5,.75)}{\pgfbox[center,center]{\large{$R$}}}
\pgfputat{\pgfxy(1.5,1.75)}{\pgfbox[center,center]{\large{$\bigoplus$}}}
\pgfputat{\pgfxy(1.5,2.75)}{\pgfbox[center,center]{\large{$R^{\otimes 2}$}}}
\pgfputat{\pgfxy(3.25,3.05)}{\pgfbox[center,center]{\small{$b_1$}}}
\pgfputat{\pgfxy(3.6,2.1)}{\pgfbox[center,center]{\small{$1-t_1$}}}
\pgfputat{\pgfxy(3.25,.5)}{\pgfbox[center,center]{\small{$-b'_1$}}}
\pgfputat{\pgfxy(2.8,1.5)}{\pgfbox[center,center]{\small{$N_1$}}}
\pgfsetendarrow{\pgfarrowlargepointed{3pt}}
\pgfsetlinewidth{2pt}
\pgfxyline(1.75,2.3)(4.6,.9)
\pgfsetlinewidth{.5pt}
\pgfxyline(4.5,2.8)(1.9,2.8)
\pgfxyline(4.7,1.2)(1.85,2.6)
\pgfxyline(4.5,.75)(2,.75)
\color{black} 
\end{pgftranslate}
\begin{pgftranslate}{\pgfpoint{3.5cm}{0cm}}
\pgfputat{\pgfxy(1.5,.75)}{\pgfbox[center,center]{\large{$R^{\otimes 2}$}}}
\pgfputat{\pgfxy(1.5,1.75)}{\pgfbox[center,center]{\large{$\bigoplus$}}}
\pgfputat{\pgfxy(1.5,2.75)}{\pgfbox[center,center]{\large{$R^{\otimes 3}$}}}
\pgfputat{\pgfxy(3.25,3.05)}{\pgfbox[center,center]{\small{$b_2$}}}
\pgfputat{\pgfxy(3.6,2.1)}{\pgfbox[center,center]{\small{$1-t_2$}}}
\pgfputat{\pgfxy(3.25,.5)}{\pgfbox[center,center]{\small{$-b'_2$}}}
\pgfputat{\pgfxy(2.8,1.5)}{\pgfbox[center,center]{\small{$N_2$}}}
\pgfsetendarrow{\pgfarrowlargepointed{3pt}}
\pgfsetlinewidth{2pt}
\pgfxyline(1.75,2.3)(4.6,.9)
\pgfsetlinewidth{.5pt}
\pgfxyline(4.5,2.8)(1.9,2.8)
\pgfxyline(4.7,1.2)(1.85,2.6)
\pgfxyline(4.5,.75)(2,.75)
\color{black} 
\end{pgftranslate}
\begin{pgftranslate}{\pgfpoint{7cm}{0cm}}
\pgfputat{\pgfxy(1.5,.75)}{\pgfbox[center,center]{\large{$R^{\otimes 3}$}}}
\pgfputat{\pgfxy(1.5,1.75)}{\pgfbox[center,center]{\large{$\bigoplus$}}}
\pgfputat{\pgfxy(1.5,2.75)}{\pgfbox[center,center]{\large{$R^{\otimes 4}$}}}
\end{pgftranslate}
\pgfnodecircle{Node1}[fill]{\pgfxy(9.2,.75)}{0.03cm}
\pgfnodecircle{Node1}[fill]{\pgfxy(9.35,.75)}{0.03cm}
\pgfnodecircle{Node1}[fill]{\pgfxy(9.5,.75)}{0.03cm}
\pgfnodecircle{Node1}[fill]{\pgfxy(9.2,2.75)}{0.03cm}
\pgfnodecircle{Node1}[fill]{\pgfxy(9.35,2.75)}{0.03cm}
\pgfnodecircle{Node1}[fill]{\pgfxy(9.5,2.75)}{0.03cm}
\end{pgftranslate}
\end{pgfpicture}
\caption{Abstract and Concrete Forms of Keller's Mixed Complex $(M_R,d,B)$.}
\label{figkellermixed}
\end{figure}

Clearly $(B^M)^2=0$.  Also, $dB^M+B^Md=0$, as can be shown by matrix multiplication:
\[dB^M+B^Md=\begin{pmatrix} b & 1-t\\ 0 & -b' \end{pmatrix}\begin{pmatrix} 0 & 0\\ N & 0\end{pmatrix}+\begin{pmatrix} 0 & 0\\ N & 0\end{pmatrix} \begin{pmatrix} b & 1-t\\ 0 & -b' \end{pmatrix}=\begin{pmatrix} (1-t)N & 0\\ -b'N & 0\end{pmatrix}+\begin{pmatrix} 0 & 0\\ Nb & N(1-t)\end{pmatrix},\]
which is the zero matrix because 
\[(1-t)N=N(1-t)=Nb-b'N\]
by the identities above.   This shows that $(M_R,d,B^M)$ is a mixed complex, as claimed.

\section{Simplicial and Cyclic Theory}\label{chaptersimplicialcyclic}

In this section, I briefly describe the simplicial and cyclic categories, simplicial and cyclic objects, and cyclic modules.  This material supports a general approach toward cyclic homology useful for dealing with broad structural issues at a high level of algebraic hierarchy.  In particular, Keller's machinery of localization pairs may be understood in this context. 

\subsection{Simplicial Basics} 

The simplicial category $\Delta$ is the category whose objects are finite linearly ordered sets $[n]=\{0,...,n\}$, called $n$-simplices, and whose morphisms are order-preserving (i.e. nondecreasing) maps.   The morphisms in $\Delta$ are conveniently described in terms of special morphisms called {\it face operators} and {\it degeneracy operators}.   The face operator $\delta^n_i:[n]\rightarrow[n+1]$ is the unique injective morphism that omits the element $i$ in the image, and the degeneracy operator $\sigma^{n}_i:[n]\rightarrow[n-1]$ is the unique surjective morphism that has two elements mapping to $i$.   More concretely,
\[\delta^n_i(j) = \left\{ \begin{array}{ll}
         j & \mbox{ if } j<i,\\
        j+1 &  \mbox{ if } j\ge i,
        \end{array} \right. \]    
and
\[\sigma^n_i(j) = \left\{ \begin{array}{ll}
         j & \mbox{ if } j\le i,\\
        j-1 &  \mbox{ if } j> i.
        \end{array} \right. \] 
The face and degeneracy operators satisfy the following identities:
\begin{equation*}
\delta^{n+1}_j\delta^n_i =\delta^{n+1}_i\delta^n_{j-1} \mbox{ if } i<j,
\end{equation*}
\begin{equation*}
\sigma^{n-1}_j\sigma^n_i =\sigma^{n-1}_i\sigma^n_{j+1} \mbox{ if } i\le j,
\end{equation*}
and 
\[\sigma^{n+1}_j\delta^n_i = \left\{ \begin{array}{ll}
         \delta^{n-1}_i\sigma^n_j & \mbox{ if } i<j,\\
         \mbox{Id}_{[n]} & \mbox{ if } i=j \mbox{ or } i=j+1,\\
       \delta^{n-1}_{i-1}\sigma^n_j & \mbox{ if } i>j+1.
        \end{array} \right. \] 
Figure \ref{simplicialcategory} illustrates the face and degeneracy operators in the simplicial category.
        
\begin{figure}[H]
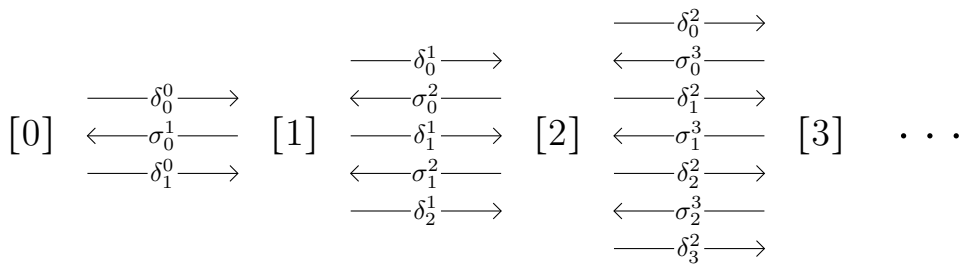

\begin{pgfpicture}{0cm}{0cm}{17cm}{4.6cm}
\begin{pgftranslate}{\pgfpoint{0cm}{1.25cm}}
\pgfputat{\pgfxy(1.5,1)}{\pgfbox[center,center]{\large{$[0]$}}}
\pgfputat{\pgfxy(5,1)}{\pgfbox[center,center]{\large{$[1]$}}}
\pgfputat{\pgfxy(8.5,1)}{\pgfbox[center,center]{\large{$[2]$}}}
\pgfputat{\pgfxy(12,1)}{\pgfbox[center,center]{\large{$[3]$}}}
\begin{pgftranslate}{\pgfpoint{.25cm}{0cm}}
\pgfputat{\pgfxy(3,1.5)}{\pgfbox[center,center]{\footnotesize{$\delta^0_0$}}}
\pgfputat{\pgfxy(3,1)}{\pgfbox[center,center]{\footnotesize{$\sigma^1_0$}}}
\pgfputat{\pgfxy(3,.5)}{\pgfbox[center,center]{\footnotesize{$\delta^0_1$}}}
\pgfxyline(2,1.5)(2.8,1.5)
\pgfxyline(4,1)(3.2,1)
\pgfxyline(2,.5)(2.8,.5)
\pgfsetendarrow{\pgfarrowlargepointed{3pt}}
\pgfxyline(3.2,1.5)(4,1.5)
\pgfxyline(2.8,1)(2,1)
\pgfxyline(3.2,.5)(4,.5)
\pgfclearendarrow
\end{pgftranslate}
\begin{pgftranslate}{\pgfpoint{3.75cm}{0cm}}
\pgfputat{\pgfxy(3,2)}{\pgfbox[center,center]{\footnotesize{$\delta^1_0$}}}
\pgfputat{\pgfxy(3,1.5)}{\pgfbox[center,center]{\footnotesize{$\sigma^2_0$}}}
\pgfputat{\pgfxy(3,1)}{\pgfbox[center,center]{\footnotesize{$\delta^1_1$}}}
\pgfputat{\pgfxy(3,.5)}{\pgfbox[center,center]{\footnotesize{$\sigma^2_1$}}}
\pgfputat{\pgfxy(3,0)}{\pgfbox[center,center]{\footnotesize{$\delta^1_2$}}}
\pgfxyline(2,2)(2.8,2)
\pgfxyline(4,1.5)(3.2,1.5)
\pgfxyline(2,1)(2.8,1)
\pgfxyline(4,.5)(3.2,.5)
\pgfxyline(2,0)(2.8,0)
\pgfsetendarrow{\pgfarrowlargepointed{3pt}}
\pgfxyline(3.2,2)(4,2)
\pgfxyline(2.8,1.5)(2,1.5)
\pgfxyline(3.2,1)(4,1)
\pgfxyline(2.8,.5)(2,.5)
\pgfxyline(3.2,0)(4,0)
\pgfclearendarrow
\end{pgftranslate}
\begin{pgftranslate}{\pgfpoint{7.25cm}{0cm}}
\pgfputat{\pgfxy(3,2.5)}{\pgfbox[center,center]{\footnotesize{$\delta^2_0$}}}
\pgfputat{\pgfxy(3,2)}{\pgfbox[center,center]{\footnotesize{$\sigma^3_0$}}}
\pgfputat{\pgfxy(3,1.5)}{\pgfbox[center,center]{\footnotesize{$\delta^2_1$}}}
\pgfputat{\pgfxy(3,1)}{\pgfbox[center,center]{\footnotesize{$\sigma^3_1$}}}
\pgfputat{\pgfxy(3,.5)}{\pgfbox[center,center]{\footnotesize{$\delta^2_2$}}}
\pgfputat{\pgfxy(3,0)}{\pgfbox[center,center]{\footnotesize{$\sigma^3_2$}}}
\pgfputat{\pgfxy(3,-.5)}{\pgfbox[center,center]{\footnotesize{$\delta^2_3$}}}
\pgfxyline(2,2.5)(2.8,2.5)
\pgfxyline(4,2)(3.2,2)
\pgfxyline(2,1.5)(2.8,1.5)
\pgfxyline(4,1)(3.2,1)
\pgfxyline(2,.5)(2.8,.5)
\pgfxyline(4,0)(3.2,0)
\pgfxyline(2,-.5)(2.8,-.5)
\pgfsetendarrow{\pgfarrowlargepointed{3pt}}
\pgfxyline(3.2,2.5)(4,2.5)
\pgfxyline(2.8,2)(2,2)
\pgfxyline(3.2,1.5)(4,1.5)
\pgfxyline(2.8,1)(2,1)
\pgfxyline(3.2,.5)(4,.5)
\pgfxyline(2.8,0)(2,0)
\pgfxyline(3.2,-.5)(4,-.5)
\pgfclearendarrow
\end{pgftranslate}
\pgfnodecircle{Node1}[fill]{\pgfxy(13.1,1)}{0.035cm}
\pgfnodecircle{Node1}[fill]{\pgfxy(13.45,1)}{0.035cm}
\pgfnodecircle{Node1}[fill]{\pgfxy(13.8,1)}{0.035cm}
\end{pgftranslate}
\end{pgfpicture}
\caption{The Simplicial Category.}
\label{simplicialcategory}
\end{figure}

Every morphism $\alpha:[n]\rightarrow[m]$ in the simplicial category has a unique factorization $\alpha=\delta\sigma$, where $\sigma$ is an epimorphism consisting of a composition of degeneracy operators, and $\delta$ is a monomorphism consisting of a composition of face operators. 


{\bf Simplicial Objects.}  Let $\mathcal{C}$ be a category.  Simplicial and cosimplicial objects in $\mathcal{C}$ are families of objects and morphisms in $\mathcal{C}$ that either reproduce or reverse the abstract structure of the face and degeneracy morphisms in the simplicial category $\Delta$.  Unfortunately, the objects that reverse the structure are called ``simplicial," while the objects that reproduce the structure are called ``cosimplicial."   This terminological stumbling block is a historical artifact that is too deeply ingrained to contravene. 

Formally, a {\it simplicial object} in $\mathcal{C}$ is a contravariant functor from $\Delta$ to $\mathcal{C}$.  More concretely, a simplicial object in $\mathcal{C}$ is a sequence $C_0,C_1,...$ of objects in $\mathcal{C}$, together with {\it face operators} $d_n^i:C_n\rightarrow C_{n-1}$ for $0\le i\le n$ and {\it degeneracy operators} $s^i_n:C_{n}\rightarrow C_{n+1}$ for $0\le i\le n$, satisfying the following identites:
\begin{equation*}
d_{n+1}^id_n^j =d_{n+1}^{j-1}d_n^{i} \mbox{ if } i<j,
\end{equation*}
\begin{equation*}
s_{n-1}^is_n^j =s_{n-1}^{j+1}s_n^{i} \mbox{ if } i\le j,
\end{equation*}
and 
\[d_{n+1}^is_n^j = \left\{ \begin{array}{ll}
         s_{n-1}^{j-1}d_n^i & \mbox{ if } i<j,\\
         \mbox{Id}_{C_n} & \mbox{ if } i=j \mbox{ or } i=j+1,\\
       s_{n-1}^{j}d_n^{i-1}& \mbox{ if } i>j+1.
        \end{array} \right. \] 
Figure \ref{simplicialobject} shows a simplicial object in an arbitrary category $\mc{C}$.  Pay special attention to the fact that the face and degeneracy operators of a simplicial object go in the opposite directions from the face and degeneracy operators in the simplicial category $\Delta$.  They really ought to be called coface and codegeneracy operators, but instead these names are assigned to the corresponding operators in cosimplicial objects (which ought to be called simplicial objects).

\begin{figure}[H]
\begin{pgfpicture}{0cm}{0cm}{17cm}{4.6cm}
\begin{pgftranslate}{\pgfpoint{0cm}{1.25cm}}
\pgfputat{\pgfxy(1.5,1)}{\pgfbox[center,center]{\large{$C_0$}}}
\pgfputat{\pgfxy(5,1)}{\pgfbox[center,center]{\large{$C_1$}}}
\pgfputat{\pgfxy(8.5,1)}{\pgfbox[center,center]{\large{$C_2$}}}
\pgfputat{\pgfxy(12,1)}{\pgfbox[center,center]{\large{$C_3$}}}
\begin{pgftranslate}{\pgfpoint{.25cm}{0cm}}
\pgfputat{\pgfxy(3,1.5)}{\pgfbox[center,center]{\footnotesize{$d_1^0$}}}
\pgfputat{\pgfxy(3,1)}{\pgfbox[center,center]{\footnotesize{$s_0^0$}}}
\pgfputat{\pgfxy(3,.5)}{\pgfbox[center,center]{\footnotesize{$d_1^1$}}}
\pgfxyline(4,1.5)(3.2,1.5)
\pgfxyline(2,1)(2.8,1)
\pgfxyline(4,.5)(3.2,.5)
\pgfsetendarrow{\pgfarrowlargepointed{3pt}}
\pgfxyline(2.8,1.5)(2,1.5)
\pgfxyline(3.2,1)(4,1)
\pgfxyline(2.8,.5)(2,.5)
\pgfclearendarrow
\end{pgftranslate}
\begin{pgftranslate}{\pgfpoint{3.75cm}{0cm}}
\pgfputat{\pgfxy(3,2)}{\pgfbox[center,center]{\footnotesize{$d_2^0$}}}
\pgfputat{\pgfxy(3,1.5)}{\pgfbox[center,center]{\footnotesize{$s_1^0$}}}
\pgfputat{\pgfxy(3,1)}{\pgfbox[center,center]{\footnotesize{$d_2^1$}}}
\pgfputat{\pgfxy(3,.5)}{\pgfbox[center,center]{\footnotesize{$s_1^1$}}}
\pgfputat{\pgfxy(3,0)}{\pgfbox[center,center]{\footnotesize{$d_2^2$}}}
\pgfxyline(4,2)(3.2,2)
\pgfxyline(2,1.5)(2.8,1.5)
\pgfxyline(4,1)(3.2,1)
\pgfxyline(2,.5)(2.8,.5)
\pgfxyline(4,0)(3.2,0)
\pgfsetendarrow{\pgfarrowlargepointed{3pt}}
\pgfxyline(2.8,2)(2,2)
\pgfxyline(3.2,1.5)(4,1.5)
\pgfxyline(2.8,1)(2,1)
\pgfxyline(3.2,.5)(4,.5)
\pgfxyline(2.8,0)(2,0)
\pgfclearendarrow
\end{pgftranslate}
\begin{pgftranslate}{\pgfpoint{7.25cm}{0cm}}
\pgfputat{\pgfxy(3,2.5)}{\pgfbox[center,center]{\footnotesize{$d_3^0$}}}
\pgfputat{\pgfxy(3,2)}{\pgfbox[center,center]{\footnotesize{$s_2^0$}}}
\pgfputat{\pgfxy(3,1.5)}{\pgfbox[center,center]{\footnotesize{$d_3^1$}}}
\pgfputat{\pgfxy(3,1)}{\pgfbox[center,center]{\footnotesize{$s_2^1$}}}
\pgfputat{\pgfxy(3,.5)}{\pgfbox[center,center]{\footnotesize{$d_3^2$}}}
\pgfputat{\pgfxy(3,0)}{\pgfbox[center,center]{\footnotesize{$s_2^2$}}}
\pgfputat{\pgfxy(3,-.5)}{\pgfbox[center,center]{\footnotesize{$d_3^3$}}}
\pgfxyline(4,2.5)(3.2,2.5)
\pgfxyline(2,2)(2.8,2)
\pgfxyline(4,1.5)(3.2,1.5)
\pgfxyline(2,1)(2.8,1)
\pgfxyline(4,.5)(3.2,.5)
\pgfxyline(2,0)(2.8,0)
\pgfxyline(4,-.5)(3.2,-.5)
\pgfsetendarrow{\pgfarrowlargepointed{3pt}}
\pgfxyline(2.8,2.5)(2,2.5)
\pgfxyline(3.2,2)(4,2)
\pgfxyline(2.8,1.5)(2,1.5)
\pgfxyline(3.2,1)(4,1)
\pgfxyline(2.8,.5)(2,.5)
\pgfxyline(3.2,0)(4,0)
\pgfxyline(2.8,-.5)(2,-.5)
\pgfclearendarrow
\end{pgftranslate}
\pgfnodecircle{Node1}[fill]{\pgfxy(13.1,1)}{0.035cm}
\pgfnodecircle{Node1}[fill]{\pgfxy(13.45,1)}{0.035cm}
\pgfnodecircle{Node1}[fill]{\pgfxy(13.8,1)}{0.035cm}
\end{pgftranslate}
\end{pgfpicture}
\caption{A simplicial object.}
\label{simplicialobject}
\end{figure}

\subsection{Cyclic Basics} 

The cyclic category $\Delta C$ has the same objects as the simplicial category, but has extra morphisms generated by the cyclic operators $\tau^n:[n]\rightarrow[n]$ defined by 
\[\tau^n(i) = \left\{ \begin{array}{ll}
         i-1 & \mbox{ if } i>0,\\
        n &  \mbox{ if } i=0.
        \end{array} \right. \]    
Note that $\tau^n$ acts in the opposite way from the cyclic operators $t_n$ defined above.  This is because cyclic objects are defined to be contravariant functors, just as simplicial objects are defined to be contravariant functors.  

In addition to the identities satisfied by the face and degeneracy operators, the following identities are also satisfied in the cyclic category:
\[\tau^n\delta^n_i = \left\{ \begin{array}{ll}
         \delta^n_{i-1}\tau^n & \mbox{ if } i>0,\\
          \delta^n_n& \mbox{ if } i=0.
        \end{array} \right. \]    
\[\tau^n\sigma^{n+1}_i = \left\{ \begin{array}{ll}
         \sigma^{n+1}_{i-1}\tau^n & \mbox{ if } i>0,\\
          \sigma^{n+1}_n(\tau^{n+1})^2& \mbox{ if } i=0.
        \end{array} \right. \]    
\[(\tau^n)^{n+1}=\mbox{Id}_{[n]}\]
For example, the curious-looking identity $\tau^n\sigma^{n+1}_0=\sigma^{n+1}_n(\tau^{n+1})^2$ may be verified graphically as illustrated in figure \ref{identity}. 

\begin{figure}[H]
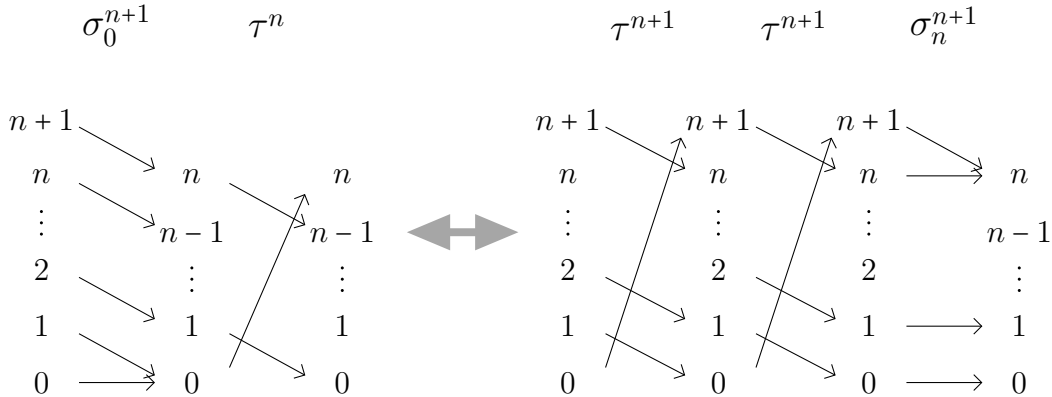

\begin{pgfpicture}{0cm}{0cm}{17cm}{6.1cm}
\begin{pgftranslate}{\pgfpoint{0cm}{0cm}}
\pgfputat{\pgfxy(1,4)}{\pgfbox[center,center]{$n+1$}}
\pgfputat{\pgfxy(1,3.25)}{\pgfbox[center,center]{$n$}}
\pgfnodecircle{Node1}[fill]{\pgfxy(1,2.775)}{0.02cm}
\pgfnodecircle{Node1}[fill]{\pgfxy(1,2.625)}{0.02cm}
\pgfnodecircle{Node1}[fill]{\pgfxy(1,2.475)}{0.02cm}
\pgfputat{\pgfxy(1,2)}{\pgfbox[center,center]{$2$}}
\pgfputat{\pgfxy(1,1.25)}{\pgfbox[center,center]{$1$}}
\pgfputat{\pgfxy(1,.5)}{\pgfbox[center,center]{$0$}}
\pgfsetendarrow{\pgfarrowlargepointed{3pt}}
\pgfxyline(1.5,3.9)(2.5,3.35)
\pgfxyline(1.5,3.15)(2.5,2.6)
\pgfxyline(1.5,1.9)(2.5,1.35)
\pgfxyline(1.5,1.15)(2.5,0.6)
\pgfxyline(1.5,.5)(2.5,.5)
\end{pgftranslate}
\begin{pgftranslate}{\pgfpoint{2cm}{0cm}}
\pgfputat{\pgfxy(1,3.25)}{\pgfbox[center,center]{$n$}}
\pgfputat{\pgfxy(1,2.5)}{\pgfbox[center,center]{$n-1$}}
\pgfnodecircle{Node1}[fill]{\pgfxy(1,2.025)}{0.02cm}
\pgfnodecircle{Node1}[fill]{\pgfxy(1,1.875)}{0.02cm}
\pgfnodecircle{Node1}[fill]{\pgfxy(1,1.725)}{0.02cm}
\pgfputat{\pgfxy(1,1.25)}{\pgfbox[center,center]{$1$}}
\pgfputat{\pgfxy(1,.5)}{\pgfbox[center,center]{$0$}}
\pgfsetendarrow{\pgfarrowlargepointed{3pt}}
\pgfxyline(1.5,3.15)(2.5,2.6)
\pgfxyline(1.5,1.15)(2.5,0.6)
\pgfxyline(1.5,.7)(2.5,3)
\end{pgftranslate}
\begin{pgftranslate}{\pgfpoint{4cm}{0cm}}
\pgfputat{\pgfxy(1,3.25)}{\pgfbox[center,center]{$n$}}
\pgfputat{\pgfxy(1,2.5)}{\pgfbox[center,center]{$n-1$}}
\pgfnodecircle{Node1}[fill]{\pgfxy(1,2.025)}{0.02cm}
\pgfnodecircle{Node1}[fill]{\pgfxy(1,1.875)}{0.02cm}
\pgfnodecircle{Node1}[fill]{\pgfxy(1,1.725)}{0.02cm}
\pgfputat{\pgfxy(1,1.25)}{\pgfbox[center,center]{$1$}}
\pgfputat{\pgfxy(1,.5)}{\pgfbox[center,center]{$0$}}
\end{pgftranslate}
\begin{pgftranslate}{\pgfpoint{7cm}{0cm}}
\pgfputat{\pgfxy(1,4)}{\pgfbox[center,center]{$n+1$}}
\pgfputat{\pgfxy(1,3.25)}{\pgfbox[center,center]{$n$}}
\pgfnodecircle{Node1}[fill]{\pgfxy(1,2.775)}{0.02cm}
\pgfnodecircle{Node1}[fill]{\pgfxy(1,2.625)}{0.02cm}
\pgfnodecircle{Node1}[fill]{\pgfxy(1,2.475)}{0.02cm}
\pgfputat{\pgfxy(1,2)}{\pgfbox[center,center]{$2$}}
\pgfputat{\pgfxy(1,1.25)}{\pgfbox[center,center]{$1$}}
\pgfputat{\pgfxy(1,.5)}{\pgfbox[center,center]{$0$}}
\pgfsetendarrow{\pgfarrowlargepointed{3pt}}
\pgfxyline(1.5,3.9)(2.55,3.35)
\pgfxyline(1.5,1.9)(2.55,1.35)
\pgfxyline(1.5,1.15)(2.55,0.6)
\pgfxyline(1.5,.7)(2.5,3.75)
\end{pgftranslate}
\begin{pgftranslate}{\pgfpoint{9cm}{0cm}}
\pgfputat{\pgfxy(1,4)}{\pgfbox[center,center]{$n+1$}}
\pgfputat{\pgfxy(1,3.25)}{\pgfbox[center,center]{$n$}}
\pgfnodecircle{Node1}[fill]{\pgfxy(1,2.775)}{0.02cm}
\pgfnodecircle{Node1}[fill]{\pgfxy(1,2.625)}{0.02cm}
\pgfnodecircle{Node1}[fill]{\pgfxy(1,2.475)}{0.02cm}
\pgfputat{\pgfxy(1,2)}{\pgfbox[center,center]{$2$}}
\pgfputat{\pgfxy(1,1.25)}{\pgfbox[center,center]{$1$}}
\pgfputat{\pgfxy(1,.5)}{\pgfbox[center,center]{$0$}}
\pgfsetendarrow{\pgfarrowlargepointed{3pt}}
\pgfxyline(1.5,3.9)(2.55,3.35)
\pgfxyline(1.5,1.9)(2.55,1.35)
\pgfxyline(1.5,1.15)(2.55,0.6)
\pgfxyline(1.5,.7)(2.5,3.75)
\end{pgftranslate}
\begin{pgftranslate}{\pgfpoint{11cm}{0cm}}
\pgfputat{\pgfxy(1,4)}{\pgfbox[center,center]{$n+1$}}
\pgfputat{\pgfxy(1,3.25)}{\pgfbox[center,center]{$n$}}
\pgfnodecircle{Node1}[fill]{\pgfxy(1,2.775)}{0.02cm}
\pgfnodecircle{Node1}[fill]{\pgfxy(1,2.625)}{0.02cm}
\pgfnodecircle{Node1}[fill]{\pgfxy(1,2.475)}{0.02cm}
\pgfputat{\pgfxy(1,2)}{\pgfbox[center,center]{$2$}}
\pgfputat{\pgfxy(1,1.25)}{\pgfbox[center,center]{$1$}}
\pgfputat{\pgfxy(1,.5)}{\pgfbox[center,center]{$0$}}
\pgfsetendarrow{\pgfarrowlargepointed{3pt}}
\pgfxyline(1.5,3.9)(2.5,3.35)
\pgfxyline(1.5,3.25)(2.5,3.25)
\pgfxyline(1.5,1.25)(2.5,1.25)
\pgfxyline(1.5,.5)(2.5,.5)
\end{pgftranslate}
\begin{pgftranslate}{\pgfpoint{13cm}{0cm}}
\pgfputat{\pgfxy(1,3.25)}{\pgfbox[center,center]{$n$}}
\pgfputat{\pgfxy(1,2.5)}{\pgfbox[center,center]{$n-1$}}
\pgfnodecircle{Node1}[fill]{\pgfxy(1,2.025)}{0.02cm}
\pgfnodecircle{Node1}[fill]{\pgfxy(1,1.875)}{0.02cm}
\pgfnodecircle{Node1}[fill]{\pgfxy(1,1.725)}{0.02cm}
\pgfputat{\pgfxy(1,1.25)}{\pgfbox[center,center]{$1$}}
\pgfputat{\pgfxy(1,.5)}{\pgfbox[center,center]{$0$}}
\end{pgftranslate}
\begin{colormixin}{35!white}
\pgfsetendarrow{\pgfarrowtriangle{8pt}}
\pgfsetstartarrow{\pgfarrowtriangle{8pt}}
\pgfsetlinewidth{3pt}
\pgfxyline(6.1,2.5)(7.1,2.5)
\end{colormixin}
\pgfputat{\pgfxy(2,5.25)}{\pgfbox[center,center]{\large{$\sigma^{n+1}_0$}}}
\pgfputat{\pgfxy(4,5.25)}{\pgfbox[center,center]{\large{$\tau^{n}$}}}
\pgfputat{\pgfxy(9,5.25)}{\pgfbox[center,center]{\large{$\tau^{n+1}$}}}
\pgfputat{\pgfxy(11,5.25)}{\pgfbox[center,center]{\large{$\tau^{n+1}$}}}
\pgfputat{\pgfxy(13,5.25)}{\pgfbox[center,center]{\large{$\sigma^{n+1}_n$}}}
\end{pgfpicture}
\caption{Example: the identity $\tau^n\sigma^{n+1}_0=\sigma^{n+1}_n(\tau^{n+1})^2.$}
\label{identity}
\end{figure}

Figure \ref{cycliccategory} illustrates the face, degeneracy, and cyclic operators in the cyclic category.

\begin{figure}[H]
\begin{pgfpicture}{0cm}{0cm}{17cm}{4.6cm}
\begin{pgftranslate}{\pgfpoint{0cm}{1.25cm}}
\pgfputat{\pgfxy(1.5,1)}{\pgfbox[center,center]{\large{$[0]$}}}
\pgfputat{\pgfxy(5,1)}{\pgfbox[center,center]{\large{$[1]$}}}
\pgfputat{\pgfxy(8.5,1)}{\pgfbox[center,center]{\large{$[2]$}}}
\pgfputat{\pgfxy(12,1)}{\pgfbox[center,center]{\large{$[3]$}}}
\begin{pgftranslate}{\pgfpoint{.25cm}{0cm}}
\pgfputat{\pgfxy(3,1.5)}{\pgfbox[center,center]{\footnotesize{$\delta^0_0$}}}
\pgfputat{\pgfxy(3,1)}{\pgfbox[center,center]{\footnotesize{$\sigma^1_0$}}}
\pgfputat{\pgfxy(3,.5)}{\pgfbox[center,center]{\footnotesize{$\delta^0_1$}}}
\pgfputat{\pgfxy(1.25,-.3)}{\pgfbox[center,center]{\footnotesize{$\tau^0$}}}
\pgfxyline(2,1.5)(2.8,1.5)
\pgfxyline(4,1)(3.2,1)
\pgfxyline(2,.5)(2.8,.5)
\pgfsetendarrow{\pgfarrowlargepointed{3pt}}
\pgfxyline(3.2,1.5)(4,1.5)
\pgfxyline(2.8,1)(2,1)
\pgfxyline(3.2,.5)(4,.5)
\pgfmoveto{\pgfxy(1,.6)}
\pgfcurveto{\pgfxy(.5,-.3)}{\pgfxy(2,-.3)}{\pgfxy(1.5,.6)}
\pgfstroke
\pgfclearendarrow
\end{pgftranslate}
\begin{pgftranslate}{\pgfpoint{3.75cm}{0cm}}
\pgfputat{\pgfxy(3,2)}{\pgfbox[center,center]{\footnotesize{$\delta^1_0$}}}
\pgfputat{\pgfxy(3,1.5)}{\pgfbox[center,center]{\footnotesize{$\sigma^2_0$}}}
\pgfputat{\pgfxy(3,1)}{\pgfbox[center,center]{\footnotesize{$\delta^1_1$}}}
\pgfputat{\pgfxy(3,.5)}{\pgfbox[center,center]{\footnotesize{$\sigma^2_1$}}}
\pgfputat{\pgfxy(3,0)}{\pgfbox[center,center]{\footnotesize{$\delta^1_2$}}}
\pgfputat{\pgfxy(1.25,-.3)}{\pgfbox[center,center]{\footnotesize{$\tau^1$}}}
\pgfxyline(2,2)(2.8,2)
\pgfxyline(4,1.5)(3.2,1.5)
\pgfxyline(2,1)(2.8,1)
\pgfxyline(4,.5)(3.2,.5)
\pgfxyline(2,0)(2.8,0)
\pgfsetendarrow{\pgfarrowlargepointed{3pt}}
\pgfxyline(3.2,2)(4,2)
\pgfxyline(2.8,1.5)(2,1.5)
\pgfxyline(3.2,1)(4,1)
\pgfxyline(2.8,.5)(2,.5)
\pgfxyline(3.2,0)(4,0)
\pgfmoveto{\pgfxy(1,.6)}
\pgfcurveto{\pgfxy(.5,-.3)}{\pgfxy(2,-.3)}{\pgfxy(1.5,.6)}
\pgfstroke
\pgfclearendarrow
\end{pgftranslate}
\begin{pgftranslate}{\pgfpoint{7.25cm}{0cm}}
\pgfputat{\pgfxy(3,2.5)}{\pgfbox[center,center]{\footnotesize{$\delta^2_0$}}}
\pgfputat{\pgfxy(3,2)}{\pgfbox[center,center]{\footnotesize{$\sigma^3_0$}}}
\pgfputat{\pgfxy(3,1.5)}{\pgfbox[center,center]{\footnotesize{$\delta^2_1$}}}
\pgfputat{\pgfxy(3,1)}{\pgfbox[center,center]{\footnotesize{$\sigma^3_1$}}}
\pgfputat{\pgfxy(3,.5)}{\pgfbox[center,center]{\footnotesize{$\delta^2_2$}}}
\pgfputat{\pgfxy(3,0)}{\pgfbox[center,center]{\footnotesize{$\sigma^3_2$}}}
\pgfputat{\pgfxy(3,-.5)}{\pgfbox[center,center]{\footnotesize{$\delta^2_3$}}}
\pgfputat{\pgfxy(1.25,-.3)}{\pgfbox[center,center]{\footnotesize{$\tau^2$}}}
\pgfxyline(2,2.5)(2.8,2.5)
\pgfxyline(4,2)(3.2,2)
\pgfxyline(2,1.5)(2.8,1.5)
\pgfxyline(4,1)(3.2,1)
\pgfxyline(2,.5)(2.8,.5)
\pgfxyline(4,0)(3.2,0)
\pgfxyline(2,-.5)(2.8,-.5)
\pgfsetendarrow{\pgfarrowlargepointed{3pt}}
\pgfxyline(3.2,2.5)(4,2.5)
\pgfxyline(2.8,2)(2,2)
\pgfxyline(3.2,1.5)(4,1.5)
\pgfxyline(2.8,1)(2,1)
\pgfxyline(3.2,.5)(4,.5)
\pgfxyline(2.8,0)(2,0)
\pgfxyline(3.2,-.5)(4,-.5)
\pgfmoveto{\pgfxy(1,.6)}
\pgfcurveto{\pgfxy(.5,-.3)}{\pgfxy(2,-.3)}{\pgfxy(1.5,.6)}
\pgfstroke
\pgfclearendarrow
\end{pgftranslate}
\begin{pgftranslate}{\pgfpoint{10.75cm}{0cm}}
\pgfputat{\pgfxy(1.25,-.3)}{\pgfbox[center,center]{\footnotesize{$\tau^3$}}}
\pgfsetendarrow{\pgfarrowlargepointed{3pt}}
\pgfmoveto{\pgfxy(1,.6)}
\pgfcurveto{\pgfxy(.5,-.3)}{\pgfxy(2,-.3)}{\pgfxy(1.5,.6)}
\pgfstroke
\end{pgftranslate}
\pgfnodecircle{Node1}[fill]{\pgfxy(13.1,1)}{0.035cm}
\pgfnodecircle{Node1}[fill]{\pgfxy(13.45,1)}{0.035cm}
\pgfnodecircle{Node1}[fill]{\pgfxy(13.8,1)}{0.035cm}
\end{pgftranslate}
\end{pgfpicture}
\caption{The cyclic category.}
\label{cycliccategory}
\end{figure}

Every morphism $\beta:[n]\rightarrow[m]$ in the cyclic category has a unique factorization $\beta=(\tau^m)^l\alpha$, where $\alpha$ is a morphism in the simplicial category $\Delta$, and $(\tau^m)^l$ is an element of the cyclic group of order $m+1$ generated by $\tau^m$.   

Unlike the simplicial category, the cyclic category is isomorphic to its opposite category.

{\bf Cyclic Objects.}  Similar to a simplicial object, a cyclic object in a category $\mc{C}$ is a family of objects and morphisms in $\mc{C}$ that reverses the abstract structure of the cyclic category.  Cyclic objects really ought to be called cocyclic objects, but it is too late to change the terminology.  

Formally, a {\it cyclic object} in $\mathcal{C}$ is a contravariant functor from $\Delta C$ to $\mathcal{C}$.  More concretely, a cyclic object in $\mathcal{C}$ is a sequence $C_0,C_1,...$ of objects in $\mathcal{C}$, together with {\it face operators} $d_n^i:C_n\rightarrow C_{n-1}$ for $0\le i\le n$, {\it degeneracy operators} $s^i_n:C_{n}\rightarrow C_{n+1}$ for $0\le i\le n$, and {\it cyclic operators} $t_n:C_n\rightarrow C_n$, satisfying the following identites:
\begin{equation*}
d_{n+1}^id_n^j =d_{n+1}^{j-1}d_n^{i} \mbox{ if } i<j,
\end{equation*}
\begin{equation*}
s_{n-1}^is_n^j =s_{n-1}^{j+1}s_n^{i} \mbox{ if } i\le j,
\end{equation*}
\[d_{n+1}^is_n^j = \left\{ \begin{array}{ll}
         s_{n-1}^{j-1}d_n^i & \mbox{ if } i<j,\\
         \mbox{Id}_{C_n} & \mbox{ if } i=j \mbox{ or } i=j+1,\\
       s_{n-1}^{j}d_n^{i-1}& \mbox{ if } i>j+1,
        \end{array} \right. \] 
        
\[d^i_nt_n = \left\{ \begin{array}{ll}
        t_{n-1}d^{i-1}_n & \mbox{ if } i>0,\\
          d^n_n& \mbox{ if } i=0.
        \end{array} \right. \]    
\[s^i_nt_n = \left\{ \begin{array}{ll}
         t_{n+1}s^{i-1}_n & \mbox{ if } i>0,\\
          (t_{n+1})^2s_n^n& \mbox{ if } i=0.
        \end{array} \right. \]    
\[(t_n)^{n+1}=\mbox{Id}_{C_n}\]
        
Figure \ref{cyclicobject} shows a cyclic object in an arbitrary category $\mc{C}$.  Again, note that the face, degeneracy, and cyclic operators of a cyclic object go in the opposite directions from the face, degeneracy, and cyclic operators in the cyclic category $\Delta$. 
        
\begin{figure}[H]
\begin{pgfpicture}{0cm}{0cm}{17cm}{4.6cm}
\begin{pgftranslate}{\pgfpoint{0cm}{1.25cm}}
\pgfputat{\pgfxy(1.5,1)}{\pgfbox[center,center]{\large{$C_0$}}}
\pgfputat{\pgfxy(5,1)}{\pgfbox[center,center]{\large{$C_1$}}}
\pgfputat{\pgfxy(8.5,1)}{\pgfbox[center,center]{\large{$C_2$}}}
\pgfputat{\pgfxy(12,1)}{\pgfbox[center,center]{\large{$C_3$}}}
\begin{pgftranslate}{\pgfpoint{.25cm}{0cm}}
\pgfputat{\pgfxy(3,1.5)}{\pgfbox[center,center]{\footnotesize{$\delta_1^0$}}}
\pgfputat{\pgfxy(3,1)}{\pgfbox[center,center]{\footnotesize{$\sigma_0^0$}}}
\pgfputat{\pgfxy(3,.5)}{\pgfbox[center,center]{\footnotesize{$\delta_1^1$}}}
\pgfputat{\pgfxy(1.25,-.3)}{\pgfbox[center,center]{\footnotesize{$\tau_0$}}}
\pgfxyline(4,1.5)(3.2,1.5)
\pgfxyline(2,1)(2.8,1)
\pgfxyline(4,.5)(3.2,.5)
\pgfsetendarrow{\pgfarrowlargepointed{3pt}}
\pgfxyline(2.8,1.5)(2,1.5)
\pgfxyline(3.2,1)(4,1)
\pgfxyline(2.8,.5)(2,.5)
\pgfmoveto{\pgfxy(1.5,.6)}
\pgfcurveto{\pgfxy(2,-.3)}{\pgfxy(.5,-.3)}{\pgfxy(1,.6)}
\pgfstroke
\pgfclearendarrow
\end{pgftranslate}
\begin{pgftranslate}{\pgfpoint{3.75cm}{0cm}}
\pgfputat{\pgfxy(3,2)}{\pgfbox[center,center]{\footnotesize{$\delta_2^0$}}}
\pgfputat{\pgfxy(3,1.5)}{\pgfbox[center,center]{\footnotesize{$\sigma_1^0$}}}
\pgfputat{\pgfxy(3,1)}{\pgfbox[center,center]{\footnotesize{$\delta_2^1$}}}
\pgfputat{\pgfxy(3,.5)}{\pgfbox[center,center]{\footnotesize{$\sigma_1^1$}}}
\pgfputat{\pgfxy(3,0)}{\pgfbox[center,center]{\footnotesize{$\delta_2^2$}}}
\pgfputat{\pgfxy(1.25,-.3)}{\pgfbox[center,center]{\footnotesize{$\tau_1$}}}
\pgfxyline(4,2)(3.2,2)
\pgfxyline(2,1.5)(2.8,1.5)
\pgfxyline(4,1)(3.2,1)
\pgfxyline(2,.5)(2.8,.5)
\pgfxyline(4,0)(3.2,0)
\pgfsetendarrow{\pgfarrowlargepointed{3pt}}
\pgfxyline(2.8,2)(2,2)
\pgfxyline(3.2,1.5)(4,1.5)
\pgfxyline(2.8,1)(2,1)
\pgfxyline(3.2,.5)(4,.5)
\pgfxyline(2.8,0)(2,0)
\pgfmoveto{\pgfxy(1.5,.6)}
\pgfcurveto{\pgfxy(2,-.3)}{\pgfxy(.5,-.3)}{\pgfxy(1,.6)}
\pgfstroke
\pgfclearendarrow
\end{pgftranslate}
\begin{pgftranslate}{\pgfpoint{7.25cm}{0cm}}
\pgfputat{\pgfxy(3,2.5)}{\pgfbox[center,center]{\footnotesize{$\delta_3^0$}}}
\pgfputat{\pgfxy(3,2)}{\pgfbox[center,center]{\footnotesize{$\sigma_2^0$}}}
\pgfputat{\pgfxy(3,1.5)}{\pgfbox[center,center]{\footnotesize{$\delta_3^1$}}}
\pgfputat{\pgfxy(3,1)}{\pgfbox[center,center]{\footnotesize{$\sigma_2^1$}}}
\pgfputat{\pgfxy(3,.5)}{\pgfbox[center,center]{\footnotesize{$\delta_3^2$}}}
\pgfputat{\pgfxy(3,0)}{\pgfbox[center,center]{\footnotesize{$\sigma_2^2$}}}
\pgfputat{\pgfxy(3,-.5)}{\pgfbox[center,center]{\footnotesize{$\delta_3^3$}}}
\pgfputat{\pgfxy(1.25,-.3)}{\pgfbox[center,center]{\footnotesize{$\tau_2$}}}
\pgfxyline(4,2.5)(3.2,2.5)
\pgfxyline(2,2)(2.8,2)
\pgfxyline(4,1.5)(3.2,1.5)
\pgfxyline(2,1)(2.8,1)
\pgfxyline(4,.5)(3.2,.5)
\pgfxyline(2,0)(2.8,0)
\pgfxyline(4,-.5)(3.2,-.5)
\pgfsetendarrow{\pgfarrowlargepointed{3pt}}
\pgfxyline(2.8,2.5)(2,2.5)
\pgfxyline(3.2,2)(4,2)
\pgfxyline(2.8,1.5)(2,1.5)
\pgfxyline(3.2,1)(4,1)
\pgfxyline(2.8,.5)(2,.5)
\pgfxyline(3.2,0)(4,0)
\pgfxyline(2.8,-.5)(2,-.5)
\pgfmoveto{\pgfxy(1.5,.6)}
\pgfcurveto{\pgfxy(2,-.3)}{\pgfxy(.5,-.3)}{\pgfxy(1,.6)}
\pgfstroke
\pgfclearendarrow
\end{pgftranslate}
\begin{pgftranslate}{\pgfpoint{10.75cm}{0cm}}
\pgfputat{\pgfxy(1.25,-.3)}{\pgfbox[center,center]{\footnotesize{$\tau_3$}}}
\pgfsetendarrow{\pgfarrowlargepointed{3pt}}
\pgfmoveto{\pgfxy(1.5,.6)}
\pgfcurveto{\pgfxy(2,-.3)}{\pgfxy(.5,-.3)}{\pgfxy(1,.6)}
\pgfstroke
\end{pgftranslate}
\pgfnodecircle{Node1}[fill]{\pgfxy(13.1,1)}{0.035cm}
\pgfnodecircle{Node1}[fill]{\pgfxy(13.45,1)}{0.035cm}
\pgfnodecircle{Node1}[fill]{\pgfxy(13.8,1)}{0.035cm}
\end{pgftranslate}
\end{pgfpicture}
\caption{A cyclic object.}
\label{cyclicobject}
\end{figure}

\subsection{Cyclic Modules}\label{subsectioncyclicmodules}

A {\bf cyclic module} is a cyclic object in the category of $k$-modules.  A cyclic module induces a mixed complex, whose corresponding bicomplex may be used to compute cyclic homology and negative cyclic homology.   Not every mixed complex comes from a cyclic module, but many important examples of mixed complexes arise in this way.  In particular, the Hochschild complex of a unital $k$-algebra is a cyclic $k$-module.  


{\bf Cyclic Modules.} Let $k$ be a commutative ring.  A {\bf cyclic $k$-module} is a cyclic object in the category of $k$-modules.  It is represented by a diagram as shown in figure \ref{cyclicobject} above, where the objects $C_n$ are $k$-modules, and the face operators $\delta_n$, degeneracy operators $\sigma_n$, and cyclic operators $\tau_n$ are $k$-module morphisms.   A cyclic module $C$ defines a chain complex $(C,b)$, where for each $n$, $b_n:C_n\rightarrow C_{n-1}$ is given by the alternating sum of the face operators
\begin{equation}\label{equcycmodboundary}b=\sum_{i=0}^n(-1)^i \delta_n^i.\end{equation}
This formula is analogous to the boundary map for the Hochschild complex $\tn{C}_R$ of a unital $k$-algebra $R$, appearing in equation \hyperref[equhochschildsum]{\ref{equhochschildsum}} above.  


{\bf Cyclic Homology of a Cyclic Module.} A cyclic module defines a mixed complex $(C,b,B)$, where for each $n$, $B=B_n$ is defined by the formula 
\[B_n:=(1-\tau_{n+1})\sigma_{n}N_n,\]
where the maps $1$, $\tau_{n+1}$, $\sigma_{n}$, and $N_n$ are the abstract analogues of the maps $1$, $t_{n+1}$, $s_n$, and  $N_n$ appearing in the definition of Connes' boundary map in equation \hyperref[equconnesboundary]{\ref{equconnesboundary}} of section \hyperref[subsectioncychomalgcommring]{\ref{subsectioncychomalgcommring}} above.  In particular, the extra degeneracy $\sigma_{n}=\sigma_{n}^{n+1}$ is defined in terms of the last ``ordinary" degeneracy $\sigma_n^n$ and the cyclic operator $\tau_{n+1}$ as $\sigma_n^{n+1}=(-1)^{n+1}\tau_{n+1}\sigma_n^n$.

\begin{defi}\label{deficychomcycmod} Let $C$ be a cyclic module. The {\bf cyclic homology groups} $\tn{HN}_{n,C}$ of $C$ are the homology groups of the product total complex of the associated mixed complex.  A similar definition holds of negative cyclic homology. 
\end{defi}

\section{Weibel's Cyclic Homology for Schemes}\label{sectionweibelcychomschemes}

{\bf Weibel's Approach.}  Charles Weibel \cite{WeibelCyclicSchemes91} defines the cyclic homology of a scheme $X$ over a commutative ring $k$  by sheafifying the $\tn{B}$-bicomplex $\tn{B}_R$ illustrated in figure \hyperref[figBcomplex]{\ref{figBcomplex}} above, and taking {\it Cartan-Eilenberg hypercohomology.}   The bicomplex  $\tn{B}_R$ may either be sheafified directly, or the Hochschild complex $\tn{C}_R$ may be sheafified first, and the corresponding bicomplex of sheaves may be constructed afterwards using the resulting sheaf complex, in a manner analogous to the construction of $\tn{B}_R$ from $\tn{C}_R$.  Here, I begin by discussing the sheafification of the Hochschild complex, since this approach also allows for the construction of Hochschild homology of schemes along the way. 


\label{sheafificationhochschild}

{\bf Sheafification of the Hochschild Complex.} For any $n\ge 0$, let $\mathscr{C}_{n,X}$ be the sheaf associated to the presheaf
\[U\mapsto O_U^{\otimes n+1},\]
where as usual tensors are over $k$.  By elementary sheaf theory, the Hochschild boundary maps $b_{n,U}:O_U^{\otimes n+1}\rightarrow O_U^{\otimes n}$ for each open set $U\subset X$ assemble to give a complex of sheaves
\begin{equation}\label{equhochcomplexscheme}\xymatrix{...\ar[r]^b  &\mathscr{C}_{n,X}\ar[r]^b&\mathscr{C}_{n-1,X}\ar[r]^b&....}\end{equation}
called the {\it Hochschild complex} of $X$.   The Hochschild complex is a cyclic object in the category of sheaves on $X$. 


\label{weibelhochschildhomschemes}

{\bf Weibel's Hochschild Homology of Schemes.} The {\bf Hochschild homology} of $X$ is by definition the {\it Cartan-Eilenberg hypercohomology} of the unbounded cochain complex $\ms{C}_X^*$, where $\ms{C}_X^n:=\ms{C}_{-n,X}$.   Figure \ref{mirror} illustrates graphically how this definition converts $\ms{C}_X^*$ to be a cochain complex concentrated in negative degrees.  Section \hyperref[sectioncartaneilenberg]{\ref{sectioncartaneilenberg}} of the appendix provides a description of Cartan-Eilenberg hypercohomology.  Weibel \cite{WeibelCyclicSchemes91} also contains a good discussion. 

\begin{figure}[H]
\begin{pgfpicture}{0cm}{0cm}{17cm}{2cm}
\begin{pgftranslate}{\pgfpoint{0cm}{-.75cm}}
\pgfputat{\pgfxy(7.5,2.5)}{\pgfbox[center,center]{mirror}}
\pgfputat{\pgfxy(9,1.5)}{\pgfbox[center,center]{$\ms{C}_0$}}
\pgfputat{\pgfxy(10.5,1.5)}{\pgfbox[center,center]{$\ms{C}_1$}}
\pgfputat{\pgfxy(12,1.5)}{\pgfbox[center,center]{$\ms{C}_2$}}
\pgfputat{\pgfxy(13.5,1.5)}{\pgfbox[center,center]{$\ms{C}_3$}}
\pgfputat{\pgfxy(6,1.5)}{\pgfbox[center,center]{$\ms{C}^0$}}
\pgfputat{\pgfxy(4.5,1.5)}{\pgfbox[center,center]{$\ms{C}^{-1}$}}
\pgfputat{\pgfxy(3,1.5)}{\pgfbox[center,center]{$\ms{C}^{-2}$}}
\pgfputat{\pgfxy(1.5,1.5)}{\pgfbox[center,center]{$\ms{C}^{-3}$}}
\pgfputat{\pgfxy(7.5,1.5)}{\pgfbox[center,center]{$0$}}
\pgfsetendarrow{\pgfarrowlargepointed{3pt}}
\pgfxyline(14.3,1.5)(13.8,1.5)
\pgfxyline(13.1,1.5)(12.4,1.5)
\pgfxyline(11.6,1.5)(10.9,1.5)
\pgfxyline(10.1,1.5)(9.4,1.5)
\pgfxyline(8.6,1.5)(7.9,1.5)
\pgfxyline(.7,1.5)(1.1,1.5)
\pgfxyline(1.9,1.5)(2.6,1.5)
\pgfxyline(3.4,1.5)(4.1,1.5)
\pgfxyline(4.9,1.5)(5.6,1.5)
\pgfxyline(6.4,1.5)(7.1,1.5)
\pgfnodecircle{Node1}[fill]{\pgfxy(14.45,1.5)}{0.025cm}
\pgfnodecircle{Node1}[fill]{\pgfxy(14.6,1.5)}{0.025cm}
\pgfnodecircle{Node1}[fill]{\pgfxy(14.75,1.5)}{0.025cm}
\pgfnodecircle{Node1}[fill]{\pgfxy(.5,1.5)}{0.025cm}
\pgfnodecircle{Node1}[fill]{\pgfxy(.35,1.5)}{0.025cm}
\pgfnodecircle{Node1}[fill]{\pgfxy(.2,1.5)}{0.025cm}
\end{pgftranslate}
\end{pgfpicture}
\caption{Converting $\ms{C}_{X,*}$ to a cochain complex $\ms{C}_X^*$.}
\label{mirror}
\end{figure}

The {\bf Hochschild homology} $\mathbb{H}\tn{H}_{n,X}$ of an algebraic scheme $X$ is defined to be the Cartan-Eilenberg hypercohomology of the complex $\ms{C}^*(X)$:
\begin{equation}\label{equweibelhochhomscheme}\mathbb{H}\tn{H}_{n,X}:=\mathbb{H}^{-n}(\ms{C}_X^*),\end{equation}
where $\HH$ denotes Cartan-Eilenberg hypercohomology.   If $X=\tn{Spec} (R)$ for a ring $R$, then $\mathbb{H}\tn{H}_{n,X}=\tn{HH}_{n,R}$, the usual Hochschild homology of $R$, for $n\ge0$.\footnotemark\footnotetext{Weibel \cite{WeibelCyclicSchemes91} cites a paper by Weibel and Geller for this result.}  Note, however, that the Hochschild homology of a scheme $X$ is generally nonzero in negative degrees.   In particular, as noted by Weibel, $H^n(\mc{O}_X)$ is a summand of $\mathbb{H}\tn{H}_{-n,X}$.   However, if $X$ is quasi-compact and quasi-separated, only finitely many of the Hochschild homology modules in negative degrees can be nonzero.   For example, $\mathbb{H}\tn{H}_{n,X}=0$ for $n<-\tn{dim}(X)$ whenever $X$ is finite dimensional and noetherian over $k$.  See \cite{WeibelCyclicSchemes91}, page 2, for more details. 


\label{weibelcyclichomschemes}

{\bf Weibel's Cyclic Homology of Schemes.}  Let $X$ be an algebraic scheme over $k$ as above.  I will denote the sheafified $\tn{B}$ complex on $X$ by $\tn{B}_X$.  

\begin{defi}\label{defiweibelcyclicschemes}  Let $X$ be a scheme over a commutative ring $k$.  The {\bf cyclic homology groups} $\mathbb{H}\tn{C}_{n,X}$ of $X$ are the Cartan-Eilenberg hypercohomology groups $\mathbb{H}^{-n}(\tn{Tot } \ms{B}_X)$ of the cohomological version of the total complex $\tn{Tot } \ms{B}_X$ of the sheafified $\tn{B}$-bicomplex $\ms{B}_X$. 
\end{defi}

To recapitulate, the cyclic homology groups $\mathbb{H}\tn{C}_{n,X}$ may be constructed by means of the following steps:

\begin{enumerate}
\item Sheafify the $\tn{B}$-bicomplex $\tn{B}_R$ to obtain a bicomplex $\ms{B}_X$ of sheaves on $X$.  This may be done either by sheafifying $\tn{B}_R$ directly, or by first sheafifying the Hochschild complex $\tn{C}_R$ to obtain a cyclic object in the category of sheaves on $X$, then forming the corresponding bicomplex.  $\ms{B}_X$ is a first-quadrant bicomplex of sheaves on $X$. 
\item Take the total complex $\tn{Tot }\ms{B}_X$ of $\ms{B}_X$.  This is a chain complex of sheaves concentrated in non-negative degrees.  The number of factors in each total degree is finite, so there is no need to choose between sum and product total complexes.
\item Convert $\mbox{Tot }\ms{B}_X$ to a cochain complex of sheaves concentrated in non-positive degrees.
\item Take Cartan-Eilenberg hypercohomology. 
\end{enumerate}

{\bf Negative Cyclic Homology via Weibel's Approach.}  Although Weibel \cite{WeibelCyclicSchemes91} does not mention negative cyclic homology explicitly, it is ``obvious" how to modify the above construction to yield a negative variant: 


{\bf Weibel's Uniqueness Result.}  Weibel \cite{WeibelCyclicSchemes91} cites an important uniqueness result about cyclic homology theories for schemes over $k$.   Weibel defines a  {\bf cyclic homology theory} for schemes over $k$ to be a family of $\mathbb{Z}$-graded $k$ modules $\tn{HC}_{n,X}$, associated to every scheme $X$ over $k$ satisfying {\it naturality,} {\it contravariance,} {\it agreement} with ordinary cyclic homology in the case of affine schemes, and {\it existence of Mayer-Vietoris sequences.}\footnotemark\footnotetext{See Weibel \cite{WeibelCyclicSchemes91} page 1 for precise statements of these conditions and more details.}  The uniqueness result is then as follows: 

\begin{lem}\label{lemweibeluniquenesscyclic} Suppose that $\tn{HC}\rightarrow \tn{HC}'$ is a morphism of cyclic homology theories. If $X$ is a quasi-compact, quasi-separated scheme over a commutative ring $k$, then the corresponding cyclic homology groups of $X$ are isomorphic: $\tn{HC}_{n,X}\cong \tn{HC}_{n,X}'$.
\end{lem}
\begin{proof} See Weibel \cite{WeibelCyclicSchemes91}. 
\end{proof}

Like Hochschild homology, the cyclic homology of a scheme $X$ over a commutative ring $k$ is generally nonzero in negative degrees.  Weibel \cite{WeibelCyclicSchemes91} also mentions that it is possible to define lambda operations on the sheafified Hochschild complex $\ms{C}_X$ and the sheafified bicomplex $\ms{B}_X$, which descend to lambda operations on the Hochschild homology groups $\HH\tn{H}_{n,X}$ and the cyclic homology groups $\HH\tn{C}_{n,X}$.\footnotemark\footnotetext{See Weibel \cite{WeibelCyclicSchemes91}, page 3, for this remark, which cites the paper \cite{WeibelHodgeCyclic94}.}

\section{Cartan-Eilenberg Cohomology}\label{sectioncartaneilenberg}

The Cartan-Eilenberg hypercohomology is defined as follows.   Since the category of abelian sheaves on $X$ has enough injective objects, there exists an upper-half plane bicomplex $\ms{I}^{**}$, together with an augmentation map of complexes $\ms{C}_X^*\rightarrow \ms{I}^{*0}$, where each sheaf $\ms{I}^{ij}$ is injective, such that the (vertical) maps induced on coboundaries and cohomology fit together to give injective resolutions of the coboundaries and cohomology of $\ms{C}_X^*$.  
Note that in this context, cohomology means ordinary cohomology of the cochain complex $\ms{C}_X^*$, not derived functors of the global section functor.  The bicomplex $\ms{I}^{**}$ is called a (right) {\it Cartan-Eilenberg resolution} of See Weibel page 149 for more details.  

Since the complex $\ms{C}_X^*$ is concentrated in negative degrees in this case, the bicomplex $\ms{I}^{**}$ is a second-quadrant bicomplex.  Figure \ref{cartaneilenbergres} shows a Cartan-Eilenberg resolution of $\ms{C}_X^*$.  
 
\begin{figure}[H]
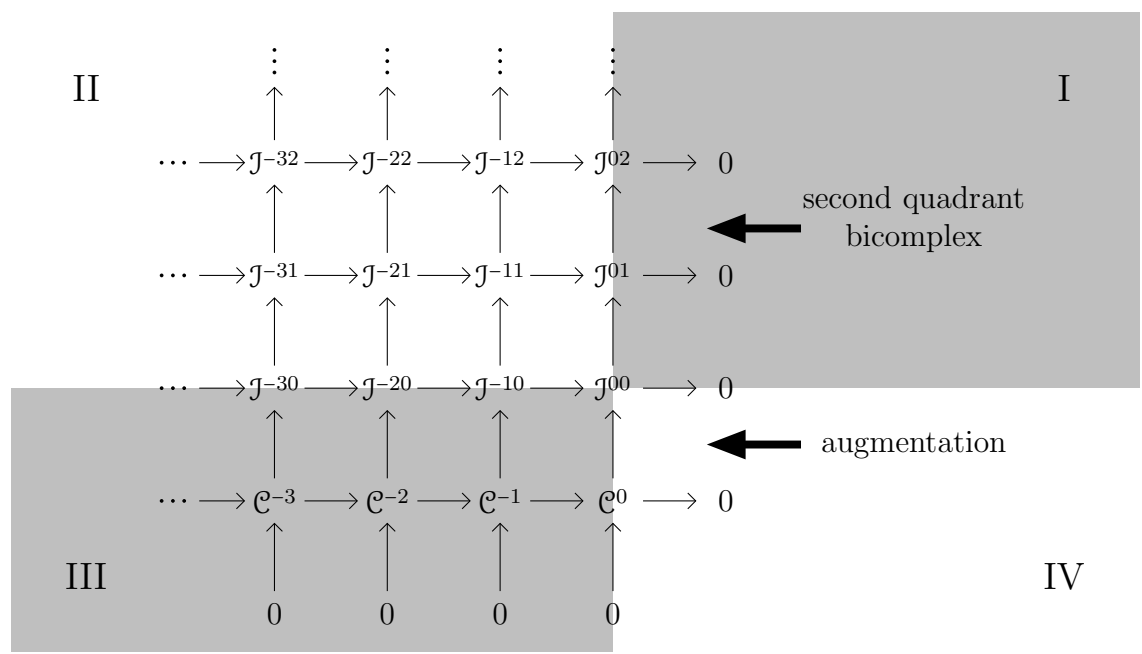

\begin{pgfpicture}{0cm}{0cm}{17cm}{8.6cm}
\begin{colormixin}{25!white}
\color{black}
\pgfmoveto{\pgfxy(8,3.5)}
\pgflineto{\pgfxy(8,8.5)}
\pgflineto{\pgfxy(15,8.5)}
\pgflineto{\pgfxy(15,3.5)}
\pgflineto{\pgfxy(8,3.5)}
\pgffill
\end{colormixin}
\begin{colormixin}{25!white}
\color{black}
\pgfmoveto{\pgfxy(8,3.5)}
\pgflineto{\pgfxy(8,0)}
\pgflineto{\pgfxy(0,0)}
\pgflineto{\pgfxy(0,3.5)}
\pgflineto{\pgfxy(8,3.5)}
\pgffill
\end{colormixin}
\pgfputat{\pgfxy(14,7.5)}{\pgfbox[center,center]{\large{I}}}
\pgfputat{\pgfxy(1,7.5)}{\pgfbox[center,center]{\large{II}}}
\pgfputat{\pgfxy(1,1)}{\pgfbox[center,center]{\large{III}}}
\pgfputat{\pgfxy(14,1)}{\pgfbox[center,center]{\large{IV}}}
\pgfputat{\pgfxy(12,2.75)}{\pgfbox[center,center]{augmentation}}
\pgfputat{\pgfxy(12,6)}{\pgfbox[center,center]{second quadrant}}
\pgfputat{\pgfxy(12,5.5)}{\pgfbox[center,center]{bicomplex}}
\begin{pgfscope}
\pgfsetlinewidth{3pt}
\pgfsetendarrow{\pgfarrowtriangle{6pt}}
\pgfxyline(10.5,2.75)(9.5,2.75)
\pgfxyline(10.5,5.625)(9.5,5.625)
\end{pgfscope}
\begin{pgftranslate}{\pgfpoint{2cm}{-1cm}}
\pgfputat{\pgfxy(6,1.5)}{\pgfbox[center,center]{$0$}}
\pgfputat{\pgfxy(4.5,1.5)}{\pgfbox[center,center]{$0$}}
\pgfputat{\pgfxy(3,1.5)}{\pgfbox[center,center]{$0$}}
\pgfputat{\pgfxy(1.5,1.5)}{\pgfbox[center,center]{$0$}}
\pgfsetendarrow{\pgfarrowlargepointed{3pt}}
\pgfxyline(6,1.8)(6,2.7)
\pgfxyline(4.5,1.8)(4.5,2.7)
\pgfxyline(3,1.8)(3,2.7)
\pgfxyline(1.5,1.8)(1.5,2.7)
\end{pgftranslate}
\begin{pgftranslate}{\pgfpoint{2cm}{.5cm}}
\pgfputat{\pgfxy(7.5,1.5)}{\pgfbox[center,center]{$0$}}
\pgfputat{\pgfxy(6,1.5)}{\pgfbox[center,center]{$\ms{C}^0$}}
\pgfputat{\pgfxy(4.5,1.5)}{\pgfbox[center,center]{$\ms{C}^{-1}$}}
\pgfputat{\pgfxy(3,1.5)}{\pgfbox[center,center]{$\ms{C}^{-2}$}}
\pgfputat{\pgfxy(1.5,1.5)}{\pgfbox[center,center]{$\ms{C}^{-3}$}}
\pgfsetendarrow{\pgfarrowlargepointed{3pt}}
\pgfxyline(.5,1.5)(1.1,1.5)
\pgfxyline(1.9,1.5)(2.6,1.5)
\pgfxyline(3.4,1.5)(4.1,1.5)
\pgfxyline(4.9,1.5)(5.6,1.5)
\pgfxyline(6.4,1.5)(7.1,1.5)
\pgfnodecircle{Node1}[fill]{\pgfxy(.3,1.5)}{0.025cm}
\pgfnodecircle{Node1}[fill]{\pgfxy(.15,1.5)}{0.025cm}
\pgfnodecircle{Node1}[fill]{\pgfxy(0,1.5)}{0.025cm}
\pgfxyline(6,1.8)(6,2.7)
\pgfxyline(4.5,1.8)(4.5,2.7)
\pgfxyline(3,1.8)(3,2.7)
\pgfxyline(1.5,1.8)(1.5,2.7)
\end{pgftranslate}
\begin{pgftranslate}{\pgfpoint{2cm}{2cm}}
\pgfputat{\pgfxy(7.5,1.5)}{\pgfbox[center,center]{$0$}}
\pgfputat{\pgfxy(6,1.5)}{\pgfbox[center,center]{$\ms{I}^{00}$}}
\pgfputat{\pgfxy(4.5,1.5)}{\pgfbox[center,center]{$\ms{I}^{-10}$}}
\pgfputat{\pgfxy(3,1.5)}{\pgfbox[center,center]{$\ms{I}^{-20}$}}
\pgfputat{\pgfxy(1.5,1.5)}{\pgfbox[center,center]{$\ms{I}^{-30}$}}
\pgfsetendarrow{\pgfarrowlargepointed{3pt}}
\pgfxyline(.5,1.5)(1.1,1.5)
\pgfxyline(1.9,1.5)(2.6,1.5)
\pgfxyline(3.4,1.5)(4.1,1.5)
\pgfxyline(4.9,1.5)(5.6,1.5)
\pgfxyline(6.4,1.5)(7.1,1.5)
\pgfnodecircle{Node1}[fill]{\pgfxy(.3,1.5)}{0.025cm}
\pgfnodecircle{Node1}[fill]{\pgfxy(.15,1.5)}{0.025cm}
\pgfnodecircle{Node1}[fill]{\pgfxy(0,1.5)}{0.025cm}
\pgfxyline(6,1.8)(6,2.7)
\pgfxyline(4.5,1.8)(4.5,2.7)
\pgfxyline(3,1.8)(3,2.7)
\pgfxyline(1.5,1.8)(1.5,2.7)
\end{pgftranslate}
\begin{pgftranslate}{\pgfpoint{2cm}{3.5cm}}
\pgfputat{\pgfxy(7.5,1.5)}{\pgfbox[center,center]{$0$}}
\pgfputat{\pgfxy(6,1.5)}{\pgfbox[center,center]{$\ms{I}^{01}$}}
\pgfputat{\pgfxy(4.5,1.5)}{\pgfbox[center,center]{$\ms{I}^{-11}$}}
\pgfputat{\pgfxy(3,1.5)}{\pgfbox[center,center]{$\ms{I}^{-21}$}}
\pgfputat{\pgfxy(1.5,1.5)}{\pgfbox[center,center]{$\ms{I}^{-31}$}}
\pgfsetendarrow{\pgfarrowlargepointed{3pt}}
\pgfxyline(.5,1.5)(1.1,1.5)
\pgfxyline(1.9,1.5)(2.6,1.5)
\pgfxyline(3.4,1.5)(4.1,1.5)
\pgfxyline(4.9,1.5)(5.6,1.5)
\pgfxyline(6.4,1.5)(7.1,1.5)
\pgfnodecircle{Node1}[fill]{\pgfxy(.3,1.5)}{0.025cm}
\pgfnodecircle{Node1}[fill]{\pgfxy(.15,1.5)}{0.025cm}
\pgfnodecircle{Node1}[fill]{\pgfxy(0,1.5)}{0.025cm}
\pgfxyline(6,1.8)(6,2.7)
\pgfxyline(4.5,1.8)(4.5,2.7)
\pgfxyline(3,1.8)(3,2.7)
\pgfxyline(1.5,1.8)(1.5,2.7)
\end{pgftranslate}
\begin{pgftranslate}{\pgfpoint{2cm}{5cm}}
\pgfputat{\pgfxy(7.5,1.5)}{\pgfbox[center,center]{$0$}}
\pgfputat{\pgfxy(6,1.5)}{\pgfbox[center,center]{$\ms{I}^{02}$}}
\pgfputat{\pgfxy(4.5,1.5)}{\pgfbox[center,center]{$\ms{I}^{-12}$}}
\pgfputat{\pgfxy(3,1.5)}{\pgfbox[center,center]{$\ms{I}^{-22}$}}
\pgfputat{\pgfxy(1.5,1.5)}{\pgfbox[center,center]{$\ms{I}^{-32}$}}
\pgfsetendarrow{\pgfarrowlargepointed{3pt}}
\pgfxyline(.5,1.5)(1.1,1.5)
\pgfxyline(1.9,1.5)(2.6,1.5)
\pgfxyline(3.4,1.5)(4.1,1.5)
\pgfxyline(4.9,1.5)(5.6,1.5)
\pgfxyline(6.4,1.5)(7.1,1.5)
\pgfnodecircle{Node1}[fill]{\pgfxy(.3,1.5)}{0.025cm}
\pgfnodecircle{Node1}[fill]{\pgfxy(.15,1.5)}{0.025cm}
\pgfnodecircle{Node1}[fill]{\pgfxy(0,1.5)}{0.025cm}
\pgfxyline(6,1.8)(6,2.5)
\pgfxyline(4.5,1.8)(4.5,2.5)
\pgfxyline(3,1.8)(3,2.5)
\pgfxyline(1.5,1.8)(1.5,2.5)
\pgfnodecircle{Node1}[fill]{\pgfxy(1.5,2.7)}{0.025cm}
\pgfnodecircle{Node1}[fill]{\pgfxy(1.5,2.85)}{0.025cm}
\pgfnodecircle{Node1}[fill]{\pgfxy(1.5,3)}{0.025cm}
\pgfnodecircle{Node1}[fill]{\pgfxy(3,2.7)}{0.025cm}
\pgfnodecircle{Node1}[fill]{\pgfxy(3,2.85)}{0.025cm}
\pgfnodecircle{Node1}[fill]{\pgfxy(3,3)}{0.025cm}
\pgfnodecircle{Node1}[fill]{\pgfxy(4.5,2.7)}{0.025cm}
\pgfnodecircle{Node1}[fill]{\pgfxy(4.5,2.85)}{0.025cm}
\pgfnodecircle{Node1}[fill]{\pgfxy(4.5,3)}{0.025cm}
\pgfnodecircle{Node1}[fill]{\pgfxy(6,2.7)}{0.025cm}
\pgfnodecircle{Node1}[fill]{\pgfxy(6,2.85)}{0.025cm}
\pgfnodecircle{Node1}[fill]{\pgfxy(6,3)}{0.025cm}
\end{pgftranslate}
\end{pgfpicture}
\caption{Cartan-Eilenberg resolution of the complex $\ms{C}_X^*$.}
\label{cartaneilenbergres}
\end{figure}

The {\it Cartan-Eilenberg hypercohomology} of $\ms{C}_X^*$ is defined to be the ordinary cohomology of the complex of global sections of the product total complex of $\ms{I}^{**}$:
\[\mathbb{H}^*(X, \ms{C}_X^*):=H^i(\Gamma(\mbox{ToT } \ms{I}^{**})).\]
For technical reasons (e.g. when a complex is not bounded below), the Cartan-Eilenberg hypercohomology of a complex is not necessarily equal to the corresponding hyper-derived cohomology.   The appendix of Weibel1991, beginning on page 5, is devoted to this subject.  Here Weibel explains the construction of the hyper-derived cohomology, and discusses when the Cartan-Eilenberg hypercohomology and hyper-derived cohomology are equal.   He proves that there are examples of complexes of sheaves on topological spaces for which the two definitions differ.  He does not make clear if this occurs for the Hochschild complex.

\section{Spectral Sequences}\label{spectralsequences}

{\bf Generalities.}  The material in this section is standard.  The main references are Weibel \cite{WeibelHomologicalAlgebra94}, Chapter 5, and McCleary \cite{McClearySpectralSequences01}, Chapter 2.   These references present modern, synthetic treatments of the original theory developed by Leray, Serre, Koszul, Massey, and others. The purpose of this material in the present context is to facilitate the construction of the coniveau spectral sequence for a cohomology theory with supports on a scheme over a field $k$, in section \hyperref[coniveauspectralsequence]{6.1.3}.  The coniveau spectral sequence is a spectral sequence of cohomological type, whose terms are $k$-vector spaces.  It stabilizes at a finite level, and converges with respect to a finite filtration.   In contrast with the specific nature its intended application, the material in this section is somewhat more general, with the goal of rendering transparent the terminology and results concerning the coniveau spectral sequence in the literature.  While I do restrict the discussion to spectral sequences and exact couples of cohomological type, I describe most of the constructions in terms of a general abelian category, which is taken to possess convenient limit properties whenever necessary.\footnotemark\footnotetext{A much more thorough discussion, in terms of Grothendieck's axioms for abelian categories, is included in Weibel \cite{WeibelHomologicalAlgebra94}.}  I do not assume that all spectral sequences stabilize at a finite level or converge with respect to a finite filtration, but this generality is intended only for context.  

A {\bf spectral sequence} of cohomological type in an abelian category $\mbf{A}$ is a family $\{E_r^{p,q}\}$ of objects of $\mbf{A}$ together with morphisms $d_r^{p,q}:E_r^{p,q}\rightarrow E_r^{p+r,q-r+1}$, called {\bf differentials}, such that each object $E_{r+1}^{p,q}$ with lower index $r+1$ is isomorphic to the cohomology object $\tn{Ker} (d_r^{p,q})/\tn{Im}( d_r^{p-r,q-r+1})$ of differentials with lower index $r$.  Here, $r$ ranges from some minimum integer $r_0$ to infinity, while $p$ and $q$ range over the integers.  A spectral sequence of cohomological type may be denoted by $\{E_r^{p,q},d_r^{p,q}\}$, or simply by $\{E_r\}$ for short.  The subfamily of $\{E_r\}$ consisting of objects with a particular lower index $r$ is a two-dimensional array called the {\bf $E_r$-level} of $\{E_r\}$.  The qualifier ``of cohomological type" means that the differentials ``go to the right," in the sense that they increase the first upper index $p$.   For the same reason, the objects $\tn{Ker} (d_r^{p,q})/\tn{Im}( d_r^{p-r,q-r+1})$ are called cohomology objects rather than homology objects.\footnotemark\footnotetext{In fact, many references {\it do} call them homology objects, but this can only cause trouble.}  A spectral sequence for which the differentials ``go to the left" is called a spectral sequence of homological type. All spectral sequences here are of cohomological type, so I will drop this qualifier going forward.  A spectral sequence may begin at any level, though most spectral sequences of interest begin at the $E_0$ or $E_1$-level.  The coniveau spectral sequence, developed in the next section, begins at the $E_1$-level.    Figure \hyperref[figspectralsequences]{\ref{figspectralsequences}} below illustrates spectral sequences. 

The $E_0$, $E_1$, and $E_2$-levels of a spectral sequence $\{E_r\}$ of cohomological type are illustrated in the following diagram.  The shading indicates the four quadrants of the planar lattice\footnotemark\footnotetext{Unfortunately, the ``counterclockwise" convention I use here for the numbering of the quadrants is not universal; some references interchange the third and fourth quadrants.}  whose horizontal axis corresponds to the first upper index $p$ and whose vertical axis corresponds to the second upper index $q$ of the terms $E_r^{p,q}$.   A spectral sequence is called a {\bf first-quadrant} spectral sequence if all its nontrivial terms have nonnegative upper indices; i.e., if $E_r^{p,q}=0$ whenever $p$ or $q$ is less than zero.  The {\bf boundary terms} terms of a first-quadrant spectral sequence are the terms $E_r^{p,q}$ with $p=0$ or $q=0$.  Analogous definitions apply to the other quadrants.  Similarly, a spectral sequence is called a {\bf right half-plane} spectral sequence if all its nontrivial terms have nonnegative first index $p$, and the {\bf boundary terms} are the terms $E_r^{p,q}$ with $p=0$.  Analogous definitions apply for other half-planes, and also for more general subsets of the planar lattice.\footnotemark\footnotetext{More generally, $E_r^{p,q}$ is an {\bf initial boundary term} if the term $E_r^{p-r,q+r-1}$ mapping into $E_r^{p,q}$ under $d_r^{p-r,q+r-1}$ is zero, and $E_r^{p,q}$ is a {\bf terminal boundary term} if the term $E_r^{p+r,q-r+1}$ that $E_r^{p,q}$ maps into under $d_r^{p,q}$ is zero.   If $E_r^{p,q}$ is an initial boundary term, then the corresponding term $E_{r+1}^{p,q}$ at the next level is subobject of $E_r^{p,q}$, and if $E_r^{p,q}$ is a terminal boundary term, then $E_{r+1}^{p,q}$ is a quotient of $E_r^{p,q}$.  The motivation for distinguishing boundary terms in the present context is to define edge morphisms below.} 

Many important spectral sequences {\bf stabilize} at some finite level $r_s$, meaning that for all $r\ge r_s$, the objects $E_r^{p,q}$ are independent of $r$ for all $p$ and $q$.  Whether or not a spectral sequence stabilizes, the terms $E_r^{p,q}$ may approach limits as $r$ increases.  Such stable or limiting terms are denoted by $E_{\infty}^{p,q}$.  Their existence depends both on the construction of the spectral sequence $\{E_r\}$ and on the axioms governing the ambient category $\mbf{A}$.\footnotemark\footnotetext{See \cite{WeibelHomologicalAlgebra94} page 125 for a preliminary discussion of this.}  Zeros in a spectral sequence are stable, meaning that whenever $E_r^{p,q}$ vanishes, so does $E_{r+1}^{p,q}$.  This is an immediate consequence of the fact that each level of a spectral sequence is given by taking the cohomology of the previous level.  This implies, in particular, that if the initial level of a spectral sequence belongs to a particular quadrant or half plane, then the entire spectral sequence belongs to the same quandrant or half-plane.  If any level of a spectral sequence has a finite number of nonzero rows or columns, then the spectral sequence stabilizes at a finite level.  This is true because zeros are stable, and because in this case the differentials into and out of each term are of the form $0\rightarrow E_{r}^{p,q}\rightarrow 0$ for sufficiently large $r$. 

The most common use of spectral sequences is to compute sequences of graded objects, such as cohomology groups, by successive approximations.  Let $\{H^n\}_{n\in\ZZ}$ be a sequence of graded objects in the abelian category $\mbf{A}$.   A {\bf decreasing filtration} $F^p$ on $\{H^n\}_{n\in\ZZ}$ is a family of decreasing filtrations on the objects $H^n$ for each $n$.  These filtrations consist of chains of subobjects $...\subset F^{p+1}H^n\subset F^{p}H^n\subset F^{p-1}H^n\subset...$ of $H^n$, usually defined by some canonical procedure for all values of $n$.  The filtration $F^p$ is called {\bf exhaustive} if $\cup_pF^pH^n=H^n$ for all $n$, and is called {\bf separated} if $\cap_pF^pH^n=0$ for all $n$.  The {\bf associated graded objects} of the filtration $F^p$ are the quotients $F^pH^n/F^{p+1}H^n$.    A spectral sequence $\{E_r\}$ in $\mbf{A}$ with stable or limiting terms $E_{\infty}^{p,q}$ is said to {\bf converge to $\{H^n\}$ with respect to the filtration $F^p$} if the objects of $\{E_r\}$ approach the associated graded objects of $F^p$ in an appropriate sense.\footnotemark\footnotetext{To clarify the terminology appearing in the literature, $\{E_r\}$ is said to {\bf converge weakly} to $\{H^n\}$ with respect to $F^p$ if $E_{\infty}^{p,q}\cong F^pH^{p+q}/F^{p+1}H^{p+q}$ for all $p$ and $q$.   If in addition, $F^p$ is exhaustive and separated, then $\{E_r\}$ is said to {\bf approach} $\{H^n\}$ or {\bf abut} to $\{H^n\}$ with respect to $F^p$.  Finally, if $\{E_r\}$ approaches $\{H^n\}$, then it is said to {\bf converge} to $\{H^n\}$ with respect to $F^p$ if $H^n$ is equal to the inverse limit $\displaystyle\lim_{\longleftarrow}H^n/F^pH^n$ for all $n$.}  In particular, if the filtration $F^p$ is finite, exhaustive, and separated, then $\{E_r\}$ is said to {\bf converge} to $\{H^n\}$ with respect to $F^p$ if $E_{\infty}^{p,q}\cong F^pH^{p+q}/F^{p+1}H^{p+q}$ for all $p$ and $q$.   If a spectral sequence beginning at the $E_{r_0}$-level converges to $\{H^n\}$, this is usually expressed via the concise notation $E_{r_0}^{p,q}\Rightarrow H^{p+q}$.

\begin{figure}[H]
\begin{pgfpicture}{0cm}{0cm}{17cm}{12.5cm}
\pgfputat{\pgfxy(3.5,12)}{\pgfbox[center,center]{$E_0$-level}}
\pgfputat{\pgfxy(13,12)}{\pgfbox[center,center]{$E_1$-level}}
\pgfputat{\pgfxy(3.5,6)}{\pgfbox[center,center]{$E_2$-level}}
\pgfputat{\pgfxy(13,6)}{\pgfbox[center,center]{Quadrants}}
\begin{pgftranslate}{\pgfpoint{1.5cm}{7cm}}
\begin{pgfmagnify}{.75}{.75}
\begin{colormixin}{20!white}
\color{black}
\pgfmoveto{\pgfxy(3,3)}
\pgflineto{\pgfxy(6,3)}
\pgflineto{\pgfxy(6,6)}
\pgflineto{\pgfxy(3,6)}
\pgflineto{\pgfxy(3,3)}
\pgffill
\pgfmoveto{\pgfxy(3,3)}
\pgflineto{\pgfxy(3,0)}
\pgflineto{\pgfxy(0,0)}
\pgflineto{\pgfxy(0,3)}
\pgflineto{\pgfxy(3,3)}
\pgffill
\end{colormixin}
\pgfputat{\pgfxy(1,1)}{\pgfbox[center,center]{\large{$E_0^{-1,-1}$}}}
\pgfputat{\pgfxy(3,1)}{\pgfbox[center,center]{\large{$E_0^{0,-1}$}}}
\pgfputat{\pgfxy(5,1)}{\pgfbox[center,center]{\large{$E_0^{1,-1}$}}}
\pgfputat{\pgfxy(1,3)}{\pgfbox[center,center]{\large{$E_0^{-1,0}$}}}
\pgfputat{\pgfxy(3,3)}{\pgfbox[center,center]{\large{$E_0^{0,0}$}}}
\pgfputat{\pgfxy(5,3)}{\pgfbox[center,center]{\large{$E_0^{1,0}$}}}
\pgfputat{\pgfxy(1,5)}{\pgfbox[center,center]{\large{$E_0^{-1,1}$}}}
\pgfputat{\pgfxy(3,5)}{\pgfbox[center,center]{\large{$E_0^{0,1}$}}}
\pgfputat{\pgfxy(5,5)}{\pgfbox[center,center]{\large{$E_0^{1,1}$}}}
\pgfputat{\pgfxy(.5,2)}{\pgfbox[center,center]{\small{$d_0^{-1,-1}$}}}
\pgfputat{\pgfxy(.5,4)}{\pgfbox[center,center]{\small{$d_0^{-1,0}$}}}
\pgfputat{\pgfxy(2.5,2)}{\pgfbox[center,center]{\small{$d_0^{0,-1}$}}}
\pgfputat{\pgfxy(2.5,4)}{\pgfbox[center,center]{\small{$d_0^{0,0}$}}}
\pgfputat{\pgfxy(4.5,2)}{\pgfbox[center,center]{\small{$d_0^{1,-1}$}}}
\pgfputat{\pgfxy(4.5,4)}{\pgfbox[center,center]{\small{$d_0^{1,0}$}}}
\pgfsetendarrow{\pgfarrowlargepointed{3pt}}
\pgfxyline(1,.1)(1,.5)
\pgfxyline(3,.1)(3,.5)
\pgfxyline(5,.1)(5,.5)
\pgfxyline(1,1.6)(1,2.4)
\pgfxyline(3,1.6)(3,2.4)
\pgfxyline(5,1.6)(5,2.4)
\pgfxyline(1,3.6)(1,4.4)
\pgfxyline(3,3.6)(3,4.4)
\pgfxyline(5,3.6)(5,4.4)
\pgfxyline(1,5.5)(1,5.9)
\pgfxyline(3,5.5)(3,5.9)
\pgfxyline(5,5.5)(5,5.9)
\end{pgfmagnify}
\end{pgftranslate}
\begin{pgftranslate}{\pgfpoint{11cm}{7cm}}
\begin{pgfmagnify}{.75}{.75}
\begin{colormixin}{20!white}
\color{black}
\pgfmoveto{\pgfxy(3,3)}
\pgflineto{\pgfxy(6,3)}
\pgflineto{\pgfxy(6,6)}
\pgflineto{\pgfxy(3,6)}
\pgflineto{\pgfxy(3,3)}
\pgffill
\pgfmoveto{\pgfxy(3,3)}
\pgflineto{\pgfxy(3,0)}
\pgflineto{\pgfxy(0,0)}
\pgflineto{\pgfxy(0,3)}
\pgflineto{\pgfxy(3,3)}
\pgffill
\end{colormixin}
\pgfputat{\pgfxy(1,1)}{\pgfbox[center,center]{\large{$E_1^{-1,-1}$}}}
\pgfputat{\pgfxy(3,1)}{\pgfbox[center,center]{\large{$E_1^{0,-1}$}}}
\pgfputat{\pgfxy(5,1)}{\pgfbox[center,center]{\large{$E_1^{1,-1}$}}}
\pgfputat{\pgfxy(1,3)}{\pgfbox[center,center]{\large{$E_1^{-1,0}$}}}
\pgfputat{\pgfxy(3,3)}{\pgfbox[center,center]{\large{$E_1^{0,0}$}}}
\pgfputat{\pgfxy(5,3)}{\pgfbox[center,center]{\large{$E_1^{1,0}$}}}
\pgfputat{\pgfxy(1,5)}{\pgfbox[center,center]{\large{$E_1^{-1,1}$}}}
\pgfputat{\pgfxy(3,5)}{\pgfbox[center,center]{\large{$E_1^{0,1}$}}}
\pgfputat{\pgfxy(5,5)}{\pgfbox[center,center]{\large{$E_1^{1,1}$}}}
\pgfputat{\pgfxy(2.2,1.3)}{\pgfbox[center,center]{\small{$d_1^{-1,-1}$}}}
\pgfputat{\pgfxy(4.2,1.3)}{\pgfbox[center,center]{\small{$d_1^{0,-1}$}}}
\pgfputat{\pgfxy(2.2,3.3)}{\pgfbox[center,center]{\small{$d_1^{-1,0}$}}}
\pgfputat{\pgfxy(4.2,3.3)}{\pgfbox[center,center]{\small{$d_1^{0,0}$}}}
\pgfputat{\pgfxy(2.2,5.3)}{\pgfbox[center,center]{\small{$d_1^{-1,1}$}}}
\pgfputat{\pgfxy(4.2,5.3)}{\pgfbox[center,center]{\small{$d_1^{1,0}$}}}
\pgfsetendarrow{\pgfarrowlargepointed{3pt}}
\pgfxyline(0,1)(.4,1)
\pgfxyline(0,3)(.4,3)
\pgfxyline(0,5)(.4,5)
\pgfxyline(1.6,1)(2.4,1)
\pgfxyline(3.6,1)(4.4,1)
\pgfxyline(1.6,3)(2.4,3)
\pgfxyline(3.6,3)(4.4,3)
\pgfxyline(1.6,5)(2.4,5)
\pgfxyline(3.6,5)(4.4,5)
\pgfxyline(5.5,1)(5.9,1)
\pgfxyline(5.5,3)(5.9,3)
\pgfxyline(5.5,5)(5.9,5)
\end{pgfmagnify}
\end{pgftranslate}
\begin{pgftranslate}{\pgfpoint{1.75cm}{1cm}}
\begin{pgfmagnify}{.75}{.75}
\begin{colormixin}{20!white}
\color{black}
\pgfmoveto{\pgfxy(3,3)}
\pgflineto{\pgfxy(9,3)}
\pgflineto{\pgfxy(9,6)}
\pgflineto{\pgfxy(3,6)}
\pgflineto{\pgfxy(3,3)}
\pgffill
\pgfmoveto{\pgfxy(3,3)}
\pgflineto{\pgfxy(3,-1)}
\pgflineto{\pgfxy(-1,-1)}
\pgflineto{\pgfxy(-1,3)}
\pgflineto{\pgfxy(3,3)}
\pgffill
\end{colormixin}
\pgfputat{\pgfxy(1,1)}{\pgfbox[center,center]{\large{$E_2^{-1,-1}$}}}
\pgfputat{\pgfxy(3,1)}{\pgfbox[center,center]{\large{$E_2^{0,-1}$}}}
\pgfputat{\pgfxy(5,1)}{\pgfbox[center,center]{\large{$E_2^{1,-1}$}}}
\pgfputat{\pgfxy(1,3)}{\pgfbox[center,center]{\large{$E_2^{-1,0}$}}}
\pgfputat{\pgfxy(3,3)}{\pgfbox[center,center]{\large{$E_2^{0,0}$}}}
\pgfputat{\pgfxy(5,3)}{\pgfbox[center,center]{\large{$E_2^{1,0}$}}}
\pgfputat{\pgfxy(1,5)}{\pgfbox[center,center]{\large{$E_2^{-1,1}$}}}
\pgfputat{\pgfxy(3,5)}{\pgfbox[center,center]{\large{$E_2^{0,1}$}}}
\pgfputat{\pgfxy(5,5)}{\pgfbox[center,center]{\large{$E_2^{1,1}$}}}
\pgfputat{\pgfxy(3.8,4)}{\pgfbox[center,center]{\small{$d_2^{-1,1}$}}}
\pgfputat{\pgfxy(3.8,2)}{\pgfbox[center,center]{\small{$d_2^{-1,0}$}}}
\pgfputat{\pgfxy(3.8,0)}{\pgfbox[center,center]{\small{$d_2^{-1,-1}$}}}
\pgfputat{\pgfxy(5.8,4)}{\pgfbox[center,center]{\small{$d_2^{0,1}$}}}
\pgfputat{\pgfxy(5.8,2)}{\pgfbox[center,center]{\small{$d_2^{0,0}$}}}
\pgfputat{\pgfxy(5.8,0)}{\pgfbox[center,center]{\small{$d_2^{0,-1}$}}}
\pgfputat{\pgfxy(7.8,4)}{\pgfbox[center,center]{\small{$d_2^{1,1}$}}}
\pgfputat{\pgfxy(7.8,2)}{\pgfbox[center,center]{\small{$d_2^{1,0}$}}}
\pgfputat{\pgfxy(7.8,0)}{\pgfbox[center,center]{\small{$d_2^{1,-1}$}}}
\pgfsetendarrow{\pgfarrowlargepointed{3pt}}
\pgfxyline(1.3,4.8)(4.5,3.2)
\pgfxyline(1.3,2.8)(4.5,1.2)
\pgfxyline(1.3,.8)(4.5,-.8)
\pgfxyline(3.3,4.8)(6.5,3.2)
\pgfxyline(3.3,2.8)(6.5,1.2)
\pgfxyline(3.3,.8)(6.5,-.8)
\pgfxyline(5.3,4.8)(8.5,3.2)
\pgfxyline(5.3,2.8)(8.5,1.2)
\pgfxyline(5.3,.8)(8.5,-.8)
\pgfxyline(-.5,5.7)(.5,5.2)
\pgfxyline(-.5,3.7)(.5,3.2)
\pgfxyline(-.5,1.7)(.5,1.2)
\pgfxyline(1.5,5.7)(2.5,5.2)
\pgfxyline(1.5,3.7)(2.5,3.2)
\pgfxyline(1.5,1.7)(2.5,1.2)
\pgfxyline(3.5,5.7)(4.5,5.2)
\end{pgfmagnify}
\end{pgftranslate}
\begin{pgftranslate}{\pgfpoint{11cm}{1cm}}
\begin{pgfmagnify}{.75}{.75}
\begin{colormixin}{20!white}
\color{black}
\pgfmoveto{\pgfxy(3,3)}
\pgflineto{\pgfxy(6,3)}
\pgflineto{\pgfxy(6,6)}
\pgflineto{\pgfxy(3,6)}
\pgflineto{\pgfxy(3,3)}
\pgffill
\pgfmoveto{\pgfxy(3,3)}
\pgflineto{\pgfxy(3,0)}
\pgflineto{\pgfxy(0,0)}
\pgflineto{\pgfxy(0,3)}
\pgflineto{\pgfxy(3,3)}
\pgffill
\end{colormixin}
\pgfputat{\pgfxy(4.5,4.5)}{\pgfbox[center,center]{\huge{1}}}
\pgfputat{\pgfxy(1.5,4.5)}{\pgfbox[center,center]{\huge{2}}}
\pgfputat{\pgfxy(1.5,1.5)}{\pgfbox[center,center]{\huge{3}}}
\pgfputat{\pgfxy(4.5,1.5)}{\pgfbox[center,center]{\huge{4}}}
\end{pgfmagnify}
\end{pgftranslate}
\end{pgfpicture}
\caption{Spectral sequences.}
\label{figspectralsequences}
\end{figure}

A spectral sequence $\{E_r\}$ is said to {\bf collapse} at the $E_r$-level if the lattice $E_r$ has exactly one nonzero row or column.  In this case, $\{E_r\}$ stabilizes at the $E_r$-level, since $\tn{Ker}(d_r^{p,q})=E_r^{p,q}$ and $\tn{Im}(d_r^{p-r,q+r-1})=0$.   Collapse implies convergence, and the family $\{H^n\}$ to which a collapsing sequence converges may be easily obtained in the following way.  Assuming convergence for the moment, stability at the $E_r$-level implies that $E_r^{p,q}\cong E_{\infty}^{p,q}\cong F^pH^{p+q}/F^{p+1}H^{p+q}$.  The diagonal in the lattice given by setting $p+q=n$ intersects the nonzero row or column in exactly one term, so there exists at most one choice of $p$ and $q$ summing to $n$ such that $F^pH^n/F^{p+1}H^n$ is nonzero.  This means that the filtration has only two distinct levels, $p$ and $p+1$.  Since the filtration is decreasing, it follows that for this particular choice of $p$, $F^pH^n=H^n$ and $F^{p+1}H^n=0$.  Hence, $H^n$ is the unique term $E_r^{p+q}$ in the nonzero row or column with $p+q=n$.  

Whenever a convergent spectral sequence belongs to a particular quadrant or half-plane, there exist special morphisms involving its boundary terms, called edge morphisms.  In particular, if $\{E_r\}$ is a right half-plane spectral sequence, then the source of every differential landing on the vertical axis is zero, so each term of the form $E_{r+1}^{0,q}$ is the subobject $\tn{Ker}(d_r^{p,q})$ of the term $E_{r}^{0,q}$ at the previous level.  Hence, there exists a sequence\footnotemark\footnotetext{Possibly transfinite, if the sequence does not stabilize at a finite level.} of inclusions
\[E_\infty^{0,q}\subset...\subset E_{r+1}^{0,q}\subset E_{r}^{0,q}\subset E_{r-1}^{0,q}\subset...\]
Since $E_\infty^{p,q}\cong F^pH^{p+q}/F^{p+1}H^{p+q}=0$ by the definition of convergence whenever $p<0$, it follows that the filtration levels $...F^{0},F^{-1},F^{-2}$ are all equal, so $F^0H^q=H^q$.  Hence, $E_\infty^{0,q}\cong F^0H^q/F^1H^q=H^q/F^1H^q$ is a quotient of $H^q$ by the definition of convergence.  The compositions of the corresponding quotient morphism with appropriate sequences of the above inclusion maps define {\bf edge morphisms} $H^q\rightarrow E_{r}^{0,q}$:
\[H^q\rightarrow H^q/F^1H^q\cong E_\infty^{p,q}\subset...\subset  E_{r+1}^{0,q}\subset E_{r}^{0,q}.\]
In the special case where $\{E_r\}$ collapses at the $E_r$-level with nonzero column given by $p=0$, the edge morphisms are injections.  Indeed, in this case $F^0H^q=H^q$ and $F^1H^q=0$, so $H^q\cong E_\infty^{p,q}$, which maps by inclusion into each $E_{r}^{0,q}$.


{\bf Spectral Sequence of an Exact Couple.}\label{exactcouples} An {\bf exact couple} is an algebraic structure which gives rise to a spectral sequence under appropriate conditions.  Abstractly, an exact couple $C$ consists of a pair of objects $(D,E)$ in an abelian category $\mbf{A}$, together with a triple of morphisms $i,j,k$, fitting into the {\bf exact diagram} below, where exactness means that $\tn{Im}(i)=\tn{Ker}(j)$, $\tn{Im}(j)=\tn{Ker}(k)$, and $\tn{Im}(k)=\tn{Ker}(i)$.  Hence, $C$ stands for the quintuple $\{D,E,i,j,k\}$.  The composition $j\circ k:E\rightarrow E$ is denoted by $d$.   It is a {\bf differential} in the sense that $d\circ d=0$.   

\begin{pgfpicture}{0cm}{0cm}{17cm}{2.25cm}
\begin{pgftranslate}{\pgfpoint{5cm}{-.3cm}}
\pgfputat{\pgfxy(1.5,2)}{\pgfbox[center,center]{$D$}}
\pgfputat{\pgfxy(5,2)}{\pgfbox[center,center]{$D$}}
\pgfputat{\pgfxy(3.25,.5)}{\pgfbox[center,center]{$E$}}
\pgfputat{\pgfxy(3.25,2.25)}{\pgfbox[center,center]{$i$}}
\pgfputat{\pgfxy(2.1,1.1)}{\pgfbox[center,center]{$k$}}
\pgfputat{\pgfxy(4.4,1.1)}{\pgfbox[center,center]{$j$}}
\pgfsetendarrow{\pgfarrowlargepointed{3pt}}
\pgfxyline(1.8,2)(4.7,2)
\pgfxyline(3,.7)(1.7,1.7)
\pgfxyline(4.7,1.7)(3.5,.7)
\end{pgftranslate}
\end{pgfpicture}

In practice, $D$ and $E$ are usually bigraded objects, often bigraded modules over a commutative unital ring $R$.  The bigraded pieces of $D$ and $E$ may be denoted by $D^{p,q}$ and $E^{p,q}$.  Restricting the morphisms $i,j,k$, and $d$ to each bigraded piece of $D$ or $E$ produces morphisms $i^{p,q},j^{p,q},k^{p,q},$ and $d^{p,q}$.\footnotemark\footnotetext{Note that the differential $d^{p,q}$ on $E^{p,q}$ is the composition $j^{p+1,q}\circ k^{p,q}$, since $k$ follows $j$ in the exact diagram.}  The morphisms $i,j,k,$ and $d$ are usually defined so as to respect the bigrading on $D$ and $E$ up to an appropriate shift of bidegrees.  For instance, $i$ might map each bigraded piece $D^{p,q}$ of $D$ into $D^{p+m,q+n}$ for some fixed $m$ and $n$.  In this case, $i$ is said to have {\bf bidegree} $(m,n)$. Similar statements apply to $j$, $k$, and $d$.  Often a large number of the bigraded pieces $D^{p,q}$ and $E^{p,q}$ vanish, and this is helpful in proving results about the associated spectral sequence.  In particular, an exact couple $\{D,E,i,j,k\}$ of bigraded objects in $\mbf{A}$ is called {\bf bounded above} if for each integer $n$ there exists an integer $N$ such that $D^{p,q}=0$ whenever $p>N$.  

The family of exact couples over $\mbf{A}$ becomes a category when equipped with an appropriate family of morphisms.  A {\bf morphism} of exact couples $C=\{D,E,i,j,k\}$ and $C'=\{D',E',i',j',k'\}$ is a pair of morphisms $\phi_D:D\rightarrow D'$ and $\phi_E:E\rightarrow E'$ in $\mbf{A}$, that fit together with the morphisms $i,j,k$ and $i',j',k'$ to yield a commutative diagram of the form shown in the diagram below.  Exact couples are often constructed in terms of some topological or geometric object, and it is often important to show that such a construction has appropriate functorial properties.   A construction of this type is well-behaved if it is a functor from the appropriate topological or geometric category into the category of exact couples over $\mbf{A}$.   The exact couples associated to chains of inclusions of closed subspaces of a topological space in the next section is an example of such a functorial construction. 

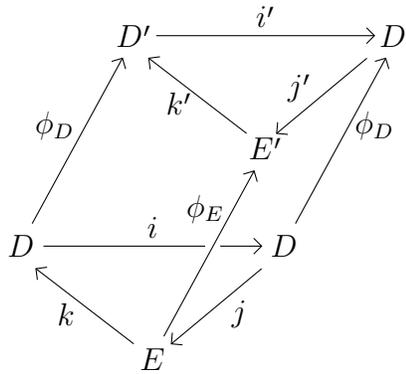
\begin{pgfpicture}{0cm}{0cm}{17cm}{5cm}
\begin{pgftranslate}{\pgfpoint{4cm}{-.3cm}}
\pgfputat{\pgfxy(1.5,2)}{\pgfbox[center,center]{$D$}}
\pgfputat{\pgfxy(5,2)}{\pgfbox[center,center]{$D$}}
\pgfputat{\pgfxy(3.25,.5)}{\pgfbox[center,center]{$E$}}
\pgfputat{\pgfxy(3.25,2.25)}{\pgfbox[center,center]{$i$}}
\pgfputat{\pgfxy(2.1,1.1)}{\pgfbox[center,center]{$k$}}
\pgfputat{\pgfxy(4.4,1.1)}{\pgfbox[center,center]{$j$}}
\pgfxyline(1.8,2)(3.95,2)
\pgfsetendarrow{\pgfarrowlargepointed{3pt}}
\pgfxyline(4.15,2)(4.7,2)
\pgfxyline(3,.7)(1.7,1.7)
\pgfxyline(4.7,1.7)(3.5,.7)
\end{pgftranslate}
\begin{pgftranslate}{\pgfpoint{5.5cm}{2.5cm}}
\pgfputat{\pgfxy(1.5,2)}{\pgfbox[center,center]{$D'$}}
\pgfputat{\pgfxy(5,2)}{\pgfbox[center,center]{$D'$}}
\pgfputat{\pgfxy(3.25,.5)}{\pgfbox[center,center]{$E'$}}
\pgfputat{\pgfxy(3.25,2.25)}{\pgfbox[center,center]{$i'$}}
\pgfputat{\pgfxy(2.1,1.1)}{\pgfbox[center,center]{$k'$}}
\pgfputat{\pgfxy(3.7,1.3)}{\pgfbox[center,center]{$j'$}}
\pgfsetendarrow{\pgfarrowlargepointed{3pt}}
\pgfxyline(1.8,2)(4.7,2)
\pgfxyline(3,.7)(1.7,1.7)
\pgfxyline(4.6,1.7)(3.4,.7)
\end{pgftranslate}
\begin{pgftranslate}{\pgfpoint{1cm}{0cm}}
\pgfsetendarrow{\pgfarrowlargepointed{3pt}}
\pgfxyline(6.4,.5)(7.6,2.7)
\pgfxyline(4.65,2)(5.85,4.2)
\pgfxyline(8.15,2)(9.35,4.2)
\pgfputat{\pgfxy(4.95,3.3)}{\pgfbox[center,center]{$\phi_D$}}
\pgfputat{\pgfxy(6.95,2.2)}{\pgfbox[center,center]{$\phi_E$}}
\pgfputat{\pgfxy(9.2,3.3)}{\pgfbox[center,center]{$\phi_D$}}
\end{pgftranslate}
\end{pgfpicture}

An exact couple $\{D,E,i,j,k\}$ may be extended to a sequence of derived exact couples.  In cases where $D$ and $E$ are bigraded, and the morphisms $i,j,$ and $k$ have appropriate bidegrees, this sequence leads to a spectral sequence.\footnotemark\footnotetext{In particular, the bidegrees of $i,j,$ and $k$ determine whether the exact couple defines a spectral sequence of homological type or of cohomological type.  The choice of bidegrees $(-1,1)$, $(0,0)$, and $(1,0)$ below gives a spectral sequence of cohomological type.} The {\bf derived exact couple} $C':=\{D',E',i',j',k'\}$ of an exact couple $C:=\{D,E,i,j,k\}$ is a new exact couple, where the {\bf derived object} $D'$ is the image of the endomorphism $i$ on $D$, and the {\bf derived object} $E'$ is the cohomology object\footnotemark\footnotetext{Calling this a cohomology object is justified by the assumption that the bidegrees of $i,j,$ and $k$  are chosen so as to define a spectral sequence of cohomological type.} $\tn{Ker}(d)/\tn{Im}(d)$ of the differential $d$ on $E$.  The morphisms $i', j'$ and $k'$ are induced by $i,j,$ and $k$ in a straightforward way: $i'$ is the restriction of $i$ to $D$, $j'$ is the composition of $j$ and the quotient morphism $\tn{Ker}(d)\rightarrow E'$, and $k'$ sends a class in $E'$ to the image under $k$ of any representative.  The composition $j'\circ k'$ is the derived differential $d'$.  The exactness of $C$ ensures that $C'$ is well-defined and exact.  The process of taking derived couples may be iterated, producing a sequence $C=C_1,C'=C_2,...,C_r,...$ of derived exact couples. 

Each derived object $E_r$ in the sequence $E=E_1,E'=E_2,...,E_r,...$ is a quotient of a subobject of the previous derived object. It is useful to analyze the subobjects involved in terms of the maps $i,j,$ and $k$.  Define the {\bf cocycle object} $A_1$ to be the kernel of $d$ and the {\bf coboundary object} $B_1$ to be the image of $d$.\footnotemark\footnotetext{Usually, these objects are denoted by $Z_1$ and $B_1$, but I use $A_1$ for the cocycles, since $Z$ is already overused.}  Then $E_2$ is equal to the cohomology object $A_1/B_1$.   In terms of the maps $i,j$ and $k$, $A_1$ is the inverse image $k^{-1}(\tn{Im}(i))$ of the image of $i$, and $B_1$ is the image $j(\tn{Ker}(i))$ of the kernel of $i$.  In a similar way, define higher cocycle objects $A_r:=k^{-1}(\tn{Im}(i^{r-1}))$ and higher coboundary objects $B_r:=j(\tn{Ker}(i^{r-1}))$, where the superscript $r-1$ denotes the $(r-1)$-fold iteration of $i$.  It is easy to check that the $r$th derived object $E_r$ is isomorphic to the $r$th cohomology object $A_r/B_r$.  If the abelian category $\mbf{A}$ includes the appropriate intersections and unions, then a limiting cocycle object $A_\infty=\cap_rA_r$ and coboundary object $B_\infty=\cup_rB_r$ may be defined.  In this case, $E_\infty$ is defined to be the quotient $A_\infty/B_\infty$.  

Now let $C=\{D^{p,q},E^{p,q},i^{p,q},j^{p,q},k^{p,q}\}$ be an exact couple of bigraded objects in $\mbf{A}$, such that the morphisms $i_1,j_1,k_1$ have bidegrees $(-1,1)$, $(0,0)$, and $(1,0)$.  Note that these choices endow the differential $d_1$ with bidegree $(1,0)$. For the purposes of the next section, note that the existence of such an exact couple is equivalent to the existence of long exact sequences of the form 
\begin{equation}\label{equlesexactcouple}...\longrightarrow D^{p+1,q-1}\overset{i^{p+1,q-1}}{\longrightarrow} D^{p,q}\overset{j^{p,q}}{\longrightarrow} E^{p,q}\overset{k^{p,q}}{\longrightarrow} D^{p+1,q}\longrightarrow...\end{equation}
for each $p$.  It is a standard result (\cite{McClearySpectralSequences01} Theorem 2.8, page 39) that the exact couple $C_1$ yields a spectral sequence $\{E_r^{p,q},d_r^{p,q}\}$, where $E_r^{p,q}$ is the $(p,q)$th bigraded piece of the object $E_r$ appearing in the $r$th derived exact couple $C_r:=\{D_r^{p,q},E_r^{p,q},i_r^{p,q},j_r^{p,q},k_r^{p,q}\}$ of $C$, and $d_r^{p,q}$ is the restriction of the differential $d_r$ to $E_r^{p,q}$.  In particular, the object $E^{p,q}$ of the original exact couple $C$ is $E_1^{p,q}$, so the spectral sequence $\{E_r\}$ begins at the $E_1$-level.   The cocycle objects $A_r$ and coboundary objects $B_r$ may be decomposed into bigraded pieces $A_r^{p,q}=k^{-1}\big(\tn{Im}\big(i^{r-1}|_{D^{p+r,q-r+1}}\big)\big)$ and $B_r^{p,q}=j\big(\tn{Ker}\big(i^{r-1}|_{D^{p+r,q-r+1}}\big)\big)$ such that $E_r^{p,q}$ is isomorphic to the quotient $A_r^{p,q}/B_r^{p,q}$.  The spectral sequence $\{E_r\}$ has limiting or stable terms $E_\infty^{p,q}$ if and only if the limiting objects $A_\infty$ and $B_\infty$ are defined in $\mbf{A}$, in which case $E_\infty^{p,q}$ is isomorphic to the quotient $A_\infty^{p,q}/B_\infty^{p,q}$.  

Now consider the convergence properties of the spectral sequence $\{E_r\}$ of the exact couple $C$.   For each integer $n$, consider the sequence of morphisms
\[...\overset{i^{p+1,n-p-1}}{\longrightarrow}D^{p,n-p}\overset{i^{p,n-p}}{\longrightarrow}D^{p-1,n-p+1}\overset{i_{}^{p-1,n-p+1}}{\longrightarrow}...\]
This sequence forms a directed system in $\mbf{A}$.  Denote by $H^n$ the directed limit $\displaystyle\lim_{\longrightarrow}D^{p,n-p}$, provided this limit exists in $\mbf{A}$.  Denote by $F^pH^n$ the image of $D^{p,n-p}$ in $H^n$.   The inclusions $...\subset F^{p+1}H^n\subset F^pH^n\subset F^{p-1}H^n\subset...$ form an exhaustive decreasing filtration of $H^n$.   If $\{E_r\}$ has stable or limiting terms $E_\infty^{p,q}$, then there exists a natural inclusion of the quotient $F^pH^n/F^{p+1}H^n$ into $E_\infty^{p,n-p}$.  Assuming appropriate axioms on the category $\mbf{A}$, it is a standard result that this inclusion is an isomorphism if and only if the limiting cocycle object $A_\infty=\cap_rk^{-1}(\tn{Im}\big(i^{r-1})\big)$ is equal to $\tn{Ker}(k)=\tn{Im}(j)$.\footnotemark\footnotetext{See Weibel \cite{WeibelHomologicalAlgebra94} proposition 5.9.6 page 156 for a homological version of this result.}  The sense of this statement is that the images of the iterates of $i$ become sufficiently small as $r$ increases.   Now suppose that the exact couple $C$ is bounded above.  Then for each $p$ and $q$, there exists an $r$ such that $i^{r-1}(D^{p+r,q-r+1})=i^r(0)=0$; i.e., $A_r^{p,q}=k^{-1}(0)$.  It follows that if $C$ is bounded above, then the corresponding spectral sequence $\{E_r\}$ is bounded above and converges to $\{H^n\}=\displaystyle\lim_{\longrightarrow}D^{p,n-p}$.

\section{Local Cohomology; Cousin Complexes}\label{subsectionlocalcohomsheaves}

The Cousin complex of an abelian sheaf on a topological space, first described by Hartshorne and Grothedieck \cite{HartshorneResiduesDuality66}, is a very general construction that provides a flasque resolution in many important cases.  Cousin complexes are defined in terms of local cohomology.

{\bf Local Cohomology.} The local cohomology of an abelian sheaf $\ms{F}$ plays a central role in the uniqueness properties, as well as in the actual construction, of its Cousin complex.  Local cohomology organizes information about sections of $\ms{F}$ whose support is contained in various closed or locally closed subsets of $X$.  Several different versions of local cohomology may be defined.  One may consider a single closed subset of $X$, a nested pair of closed subsets, an entire family of closed subsets, pairs of families, and so on.  One may also consider groups of sections of $\ms{F}$ whose supports lie in a closed subset, induced sheaves whose groups of sections are defined by restricting $\ms{F}$ to open sets and taking supports in relatively closed subsets, and various combinations of these ideas.  These different versions lead to local cohomology groups, relative local cohomology groups, local cohomology sheaves, relative local cohomology sheaves, and so on, via the machinery of derived functors.  [HartRD] analyzes the various possibilities in great detail.  

The version of local cohomology necessary in the present context involves the local cohomology sheaves $H_{\varphi}^i(\ms{F})$, where $\varphi$ is a {\it sheaf of families of supports} in the sense of [HartRD] variation 5, page 222.  Rather than going into the details of this, it suffices at present to note that the subsets $Z^p$ appearing in the coniveau filtration induce sheaves of families of supports in a canonical way.  I will denote the corresponding  local cohomology sheaves by $H_{Z^p}^i(\ms{F})$.   Note in particular that sheaves of families of supports assign a family of supports to each open set $U$ of $X$, so expressions like $Z^p(U)$ make sense in this context. 


{\bf Cousin Complexes.}  Before stating the appropriate existence and uniqueness result for Cousin complexes, there are two additional technical concepts to introduce.  The first is a cohomological condition called {\it depth}.  The $Z^p$-depth of $\ms{F}$ is defined to be the the largest integer $n$, or $\infty$, such that $H_{Z^p}^i(\ms{F})=0$ for all $i<n$.  The second is the concept of a {\it sheaf lying on the $Z^p/Z^{p+1}$-skeleton of $X$}.  This may be defined in several equivalent ways, two of which are the following: one says that $\ms{F}$ lies on the $Z^p/Z^{p+1}$-skeleton of $X$ if it has supports in $Z^p$, has $Z^{p+1}$-depth at least $1$, and is flasque, or if $\ms{F}$ has supports in $Z^p$, and $Z^{p+1}$-depth at least $2$.  

The appropriate existence and uniqueness result for Cousin complexes is the following:

{\bf Proposition 1.} Let $\ms{F}$ be an abelian sheaf on $X$, and let $X=Z^0\supset Z^1\supset...\supset Z^d\supset\oslash$ be the coniveau filtration.  There exists an augmented complex $C^\bullet$, called the Cousin complex of $\ms{F}$, of the form 
\[\xymatrixcolsep{2.5pc}\xymatrix{0\ar[r]&\ms{F}\ar[r]^{\gamma^0}&C^0\ar[r]^-{\gamma^1}&C^1\ar[r]&...,}\]
satisfying the following properties:
\begin{enumerate}
\item For each $p\ge 0$, $C^p$ lies on the $Z^p/Z^{p+1}$-skeleton of $X$.
\item For each $p> 0$, $H^p(C^\bullet)$ has supports in $Z^{p+2}$.
\item The map $\ms{F}\rightarrow H^0(C^\bullet)$ has kernel with supports in $Z^1$, and cokernel with supports in $Z^2$. 
\end{enumerate}
Furthermore, $C^\bullet$ is unique up to unique isomorphism of complexes and depends functorially on $\ms{F}$. 

Note that in the above proposition, the $H^p$ for $p\ge 0$ are sheaves given by taking kernels modulo images in the sheaf complex $C^{\bullet}$; they are not derived functors of the global sections functor. 


\setcounter{secnumdepth}{-2}
\chapter{Acknowledgements}\label{ChapterAcknowledgements}

I want to thank Jan Stienstra, Wilberd Van der Kallen, and Lars Hesselholt for their kind and thoughtful correspondence.  I also thank Jon Rosenberg, Charles Weibel, and Marco Schlichting for taking the initiative to make me aware of some interesting related work.  I thank my doctoral advisor, J. W. Hoffman, and his student Sen Yang, for working with me on related projects.  Special thanks go to Jimmie Lawson for agreeing to join my doctoral committee at the last moment, and for invaluable advice on a variety of mathematical subjects over the last few years, as well as James Madden, for support on algebraic foundations and the philosophy of mathematics.  I thank my family for their support and understanding.   Finally, I thank LSU's 460-pound mascot Mike the Tiger for his friendly greetings when stopping by his enclosure on my way to and from the math department. 
\chapter{References}

In the information age it is easy, and therefore not very helpful, to exhibit hundreds of papers on any given scientific subject not wholly obscure.  I have chosen instead to restrict attention to a minimal number of reliable and well-written references, subject to the requirement of conveying the proper priority of authorship. Only the names of the actual articles cited are in italics; journal names, books containing articles, etc., are in normal font.  Bold digits in typical journal citations indicate the {\it volume} of the journal cited; succeeding digits indicate the {\it number} or {\it issue,} the {\it page numbers} (preceded by ``pp.") or {\it article number,} and finally the {\it year published.} References are listed in the order in which they are {\it first cited} in the paper.  The chapter headings below indicate in which sections these initial citations appear.  I beg indulgence for any unintended slights due to my own ignorance. 

\normalsize

\begingroup
\renewcommand{\chapter}[2]{}%

\endgroup



\end{singlespace}

\end{document}